\documentclass[reqno, 12pt]{amsbook}
\usepackage{amsmath, latexsym, amsfonts, amssymb, amsthm, amscd, color}

\usepackage[toc,page]{appendix}
\usepackage[latin1,utf8]{inputenc}
\usepackage[T1]{fontenc}
\usepackage[english, french]{babel}
\usepackage{graphics,epsf,psfrag}
\usepackage[square]{natbib}	
\DeclareMathOperator{\var}{\mathrm{Var}}

\numberwithin{equation}{section}

\newtheorem{theorem}{Theorem}[section]
\newtheorem{lemma}[theorem]{Lemma}
\newtheorem{proposition}[theorem]{Proposition}
\newtheorem{cor}[theorem]{Corollary}
\newtheorem{rem}[theorem]{Remark}
\newtheorem{definition}[theorem]{Definition}
\newtheorem{theoreme}[theorem]{Théorème}

\newtheorem{rema}[theorem]{Remarque}

\newcommand{\ind}{\mathbf{1}}
\newcommand{\E}{\mathbb{E}}
\newcommand{\R}{\mathbb{R}}
\newcommand{\Z}{\mathbb{Z}}
\newcommand{\N}{\mathbb{N}}
\renewcommand{\tilde}{\widetilde}
\renewcommand{\hat}{\widehat}

\newcommand{\cN}{{\ensuremath{\mathcal N}} }
\newcommand{\cL}{{\ensuremath{\mathcal L}} }
\newcommand{\cT}{{\ensuremath{\mathcal T}} }
\newcommand{\cD}{{\ensuremath{\mathcal D}} }
\newcommand{\cI}{{\ensuremath{\mathcal I}} }
\newcommand{\cG}{{\ensuremath{\mathcal G}} }

\newcommand{\bP}{{\ensuremath{\mathbf P}} }
\newcommand{\bE}{{\ensuremath{\mathbf E}} }

\newcommand{\bL}{{\ensuremath{\mathbf L}} }

\DeclareMathSymbol{\leqslant}{\mathalpha}{AMSa}{"36} 
\DeclareMathSymbol{\geqslant}{\mathalpha}{AMSa}{"3E} 
\DeclareMathSymbol{\eset}{\mathalpha}{AMSb}{"3F}     
\renewcommand{\leq}{\;\leqslant\;}                   
\renewcommand{\geq}{\;\geqslant\;}                   
\newcommand{\dd}{\,\text{\rm d}}             
\newcommand{\med}[1]{\left\langle #1\right\rangle}
\newcommand{\sumtwo}[2]{\sum_{\substack{#1 \\ #2}}} 

\newcommand{\bbE}{{\ensuremath{\mathbb E}} }

\newcommand{\bbL}{{\ensuremath{\mathbb L}} }

\newcommand{\bbN}{{\ensuremath{\mathbb N}} }

\newcommand{\bbP}{{\ensuremath{\mathbb P}} }

\newcommand{\bbR}{{\ensuremath{\mathbb R}} }

\newcommand{\bbZ}{{\ensuremath{\mathbb Z}} }

\newcommand{\ga}{\alpha}
\newcommand{\gb}{\beta}
\newcommand{\gga}{\gamma}            
\newcommand{\gd}{\delta}
\newcommand{\gep}{\varepsilon}       

\newcommand{\gr}{\rho}

\newcommand{\gG}{\Gamma}

\newcommand{\gD}{\Delta}

\newcommand{\go}{\omega}

\newcommand{\gO}{\Omega}
\newcommand{\gl}{\lambda}

\newcommand{\gs}{\sigma}

\makeatletter
\def\captionfont@{\footnotesize}
\def\captionheadfont@{\scshape}

\long\def\@makecaption#1#2{%
  \vspace{2mm}
  \setbox\@tempboxa\vbox{\color@setgroup
    \advance\hsize-6pc\noindent
    \captionfont@\captionheadfont@#1\@xp\@ifnotempty\@xp
        {\@cdr#2\@nil}{.\captionfont@\upshape\enspace#2}%
    \unskip\kern-6pc\par
    \global\setbox\@ne\lastbox\color@endgroup}%
  \ifhbox\@ne 
    \setbox\@ne\hbox{\unhbox\@ne\unskip\unskip\unpenalty\unkern}%
  \fi
  \ifdim\wd\@tempboxa=\z@ 
    \setbox\@ne\hbox to\columnwidth{\hss\kern-6pc\box\@ne\hss}%
  \else 
    \setbox\@ne\vbox{\unvbox\@tempboxa\parskip\z@skip
        \noindent\unhbox\@ne\advance\hsize-6pc\par}%
\fi
  \ifnum\@tempcnta<64 
    \addvspace\abovecaptionskip
    \moveright 3pc\box\@ne
  \else 
    \moveright 3pc\box\@ne
    \nobreak
    \vskip\belowcaptionskip
  \fi
\relax
}
\makeatother
\def\writefig#1 #2 #3 {\rlap{\kern #1 truecm
\raise #2 truecm \hbox{#3}}}

\newcommand{\tf}{\textsc{f}}
\newcommand{\M}{\textsc{M}}

\newcommand{\tr}{\text{Trace}}
\newcommand{\Rav}[1]{\langle R _{\substack{#1}}\rangle}

\newcommand{\ui}{\underline{i}}
\newcommand{\uj}{\underline{j}}
\newcommand{\um}{\underline{m}}
\newcommand{\sort}{\mathtt{s}}
\newcommand{\mbeta}{{\mathtt m}_\gb}
\newcommand{\const}{{\mathtt c}}
\newcommand{\consta}{{\mathtt a}}
\newcommand{\constb}{{{\mathtt b}}}

\author{Hubert Lacoin}

\title{Désordre et phénomènes de localisation}

\begin{document}

\begin{center}
{\LARGE
{\bf \textbf{U}\sc{\textbf{niversité}} \textbf{P}\sc{\textbf{aris}} \textbf{D}\sc{\textbf{iderot}} - \textbf{P}\sc{\textbf{aris}} \textbf{7} \\
\textbf{U.F.R.\ }\sc{\textbf{de}} \textbf{M}\sc{\textbf{ath\'ematiques}}}\\
}

\vspace{1cm}

\begin{LARGE}
\bf{ \textbf{T}\sc{\textbf{h\`ese}}  \sc{\textbf{de}} \textbf{D}\sc{\textbf{octorat}} }\\

\vspace{0.5cm}

\end{LARGE}



{ Sp\'ecialit\'e: M\sc{ath\'ematiques} A\sc{ppliqu\'ees}}\\
\vspace{0.5cm}
{ Sous la direction de} { {\bf Giambattista GIACOMIN}}\\

\vspace{4cm}
\hbox{\raisebox{1em}{\vrule depth 0pt height 0.5pt width 15.2cm}}

{\Large {\bf \vspace{0.1cm} \textbf{D}\sc{\textbf{ésordre}} \sc{\textbf{et}} \textbf{P}\sc{\textbf{hénomènes}} \sc{\textbf{de}} 
\textbf{L}\sc{\textbf{ocalisation}} }}\\

\hbox{\raisebox{-0.8em}{\vrule depth 0pt height 0.5pt width
15.2cm}}

\vspace{1cm}
{{\bf Hubert LACOIN} }\\
\vspace{2cm}
\end{center}

\noindent{\large Soutenue publiquement le 19  octobre  2009, devant le jury compos\'e de :}

\vspace{0.5cm}

\begin{table}[htbp!]
\begin{center}
\begin{tabular}{lllll}
{\bf M.}   &{\bf Erwin} &{\bf BOLTHAUSEN}    & Université de Zürich            & rapporteur\\
{\bf M.}     &{\bf Philippe} &{\bf CARMONA}    &  Université de Nantes  & rapporteur\\
{\bf M.}     &{\bf Francis} &{\bf COMETS}       &   Université Paris-Diderot                    & \\
{\bf M.}     &{\bf Bernard } &{\bf DERRIDA}        &  \'Ecole Normale Supérieure   & \\
{\bf M.}     &{\bf Giambattista} &{\bf GIACOMIN}       &  Université Paris Diderot                      & \\
{\bf M.}     &{\bf  Yueyun} &{\bf HU}        &  Université Paris-Nord       & \\
{\bf M.}   &{\bf Fabio} &{\bf TONINELLI}    & E.N.S.\ Lyon           & \\
{\bf M.}     &{\bf Wendelin} &{\bf WERNER}  & Université Paris-Sud                       & \\
\end{tabular}
\end{center}
\end{table}

\thispagestyle{empty}

\vspace*{\fill}

\newpage
\null
\thispagestyle{empty}


\selectlanguage{french}
\newpage

\thispagestyle{empty}


\vspace{5cm}


\textit{Remerciements}
\bigskip

Je tiens à remercier en premier lieu mon directeur Giambattista Giacomin. Il m'a beaucoup apporté lors de ces deux années de thèse par son encadrement scientifique, son soutien logistique sans faille et sa participation à mon introduction dans la communauté scientifique. 
J'espère que notre collaboration se poursuivra dans les années à venir.
\\

Erwin Bolthausen et Philippe Carmona ont pris le temps d'examiner cette thèse. Je leur sais gré de m'avoir témoigné autant de bienveillance, c'est un grand honneur d'être soumis au jugement de ces chercheurs éminents. Je remercie également, pour avoir accepté de faire partie du Jury, Francis Comets, que j'ai eu la joie de cotoyer au L.P.M.A.\ et qui m'a fait profiter de son expertise scientifique; Bernard Derrida, les  conversations que nous avons eues m'ont beaucoup aidé et fait découvrir la richesse des échanges interdisciplinaires ; Yueyun Hu, dont j'admire les travaux ;
Fabio Toninelli avec qui j'ai eu le plaisir de travailler, la qualité de cette thèse lui doit beaucoup; 
 et Wendelin Werner, qui m'a guidé lorsque j'ai dû choisir un directeur pour mon stage de M2 puis pour ma thèse, je mesure aujourd'hui la valeur de ses conseils.
\\

Je remercie tous les gens qui m'ont permis de voyager lors de ces deux années: Laurence Vincent à L'E.N.S., Valérie Juve et Isabelle Mariage à Chevaleret qui m'ont guidé dans les méandres administratifs; Balint Toth, Eniko Sepsi, Philippe Gilles et Pierre Nolin, pour l'organisation d'un échange franco-hongrois; Olivier Glass et Kesavan qui ont eu la bonne idée de m'envoyer enseigner en Inde; Andrea Montanari qui m'a permis de justifier une escale californienne; Fabio Toninelli (encore lui), pour ses invitations répétées à l'E.N.S.\ Lyon; et Peter Mörters, grâce à qui j'ai pu découvrir le monde de la recherche lors de mon stage de M2, et qui m'a aimablement invité à Bath.
\\

Je suis très reconnaissant à Michelle Wasse, qui m'a permis de résoudre, toujours très rapidement et aimablement, toutes les tracasseries administratives que l'on peut rencontrer lors d'études doctorales.
\\

Je remercie  Gregorio Moreno pour m'avoir initié aux joies de la collaboration entre thésards; les autres collègues du 5C09: Christophe, François, Julien, Karim, Luca, Mohammed pour la bonne ambiance  au bureau; Paul et Denis pour la bonne ambiance à la maison; et Alice, parce que la syntaxe et la correction orthographique de cette thèse (pour la partie en français) lui doivent beaucoup.
\\

Enfin, je remercie ma famille: mes parents et mes  frères et soeurs, Claire, Laure, Thibaut et Guillaume, et tous ceux qui m'ont encouragé lors de la préparation de cette thèse.

\tableofcontents

\chapter*{Introduction}
\chaptermark{Introduction}

\section{Phénomènes de localisations, motivations physiques et modélisations}

\subsection{Motivation physiques}
Cette thèse est consacrée à l'étude mathématique des phénomènes de localisation et délocalisation pour différents modèles de polymères en environnement aléatoire. L'étude des modèles de polymères est un domaine de la physique mathématique qui a connu un développement particulièrement important ces dernières années (sur le sujet, on peut consulter l'article de survol \cite{cf:CSY_rev} et les monographies \cite{cf:LNMB, cf:Book, cf:denH}).
Les modèles étudiés dans cette thèse peuvent être utilisés pour décrire une grande variété de phénomènes physiques: interaction d'un polymère avec une interface entre solvants, transition de dénaturation thermique de l'ADN, accrochage et décrochage d'un polymère et d'un substrat solide, comportement trajectoriel d'une chaîne polymère dans une solution hétérogène... 
Ces phénomènes 
ont en commun l'existence d'une {\sl transition de phase}. \`A haute température, l'agitation thermique est forte ,l'entropie domine, les interactions chimiques ou physiques avec l'environnement peuvent être négligées et apparaît un phénomène de {\sl délocalisation}, c'est-à -dire que que le polymère se déploie librement, sans subir de contrainte de la part de l'environnement. \`A basse température, l'énergie d'interaction domine l'agitation thermique, et la trajectoire du polymère est très fortement conditionnée par l'environnement, nous dirons {\sl localisée}. Il existe une température dite critique qui sépare ces deux régions et qui marque une transition de la phase localisée vers la phase délocalisée. 
L'étude approfondie des modèles présentés dans cette thèse permet d'obtenir des informations: 
\begin{itemize}
\item [(1)] sur le comportement trajectoriel des polymères dans la phase localisée, dans la phase délocalisée, à la transition de phase;
\item [(2)] sur la valeur de la température critique. \end{itemize}
De plus, l'expertise développée dans le domaine des polymères en environnement aléatoire apporte une meilleure compréhension générale des systèmes désordonnés, branche importante de la mécanique statistique et de la physique mathématique.


\subsection{Modélisation mathématique}
On modélise le déploiement spatial de la chaîne polymère en adoptant un formalisme de Boltzmann-Gibbs.
\medskip

Dans un cadre discret, la chaîne polymère est modélisée par un chemin de taille finie (égale à un entier $N$) dans un ensemble de chemins donnés
$(S_n)_{n\in[0,N]}\in\gG_N$ (par exemple, $\gG_N$ peut être l'ensemble des chemins de taille $N$ dans $\Z^d$). \`A chaque trajectoire $S$ on associe une énergie modélisée par l'Hamiltonien $H_{N,\go}(S)$ (où $\go$ désigne l'environnement aléatoire) qui sera donné par la somme des énergies collectées sur les différents sites visités. La trajectoire d'une chaîne polymère de taille $N$, pour une température $T$, sera donnée par la mesure de probabilité $\mu_{N,\go,\gb}$ sur $\gG_N$ définie par
\begin{equation}
 \mu_{N,\go,\gb}(S)=\frac{\exp(\gb H_{N,\go}(S))}{|\gG_N| Z_{N,\go,\gb}},
\end{equation}
où
\begin{equation}
 Z_{N,\go,\gb}:=\frac{1}{|\gG_N|}\sum_{S\in \gG_N} \exp(\gb H_{N,\go}(S)),
\end{equation}
et $\gb$ est égal à $1/k_B T$ où $k_B$ désigne la constante de Boltzman.

\medskip
Ce formalisme peut engendrer une grande variété de modèles. Cette thèse se consacre uniquement aux cas où les trajectoires sont dirigées. Cela implique, en particulier, que les polymères étudiés ne possèdent pas d'auto--intersections, ce qui simplifie l'étude mathématique du modèle. Nous présentons maintenant deux exemples concrets de modèles de polymères basés sur la marche aléatoire simple dans $\Z^d$, desquels dérivent tous les modèles étudiés dans cette thèse:
\begin{itemize}
 \item {\bf Le modèle d'accrochage désordonné} (modèle avec désordre sur une ligne). Le polymère est modélisé par une marche aux plus proches voisins dans $\Z^d$ et reçoit des contributions énergétiques aléatoires lorsque la marche passe par l'origine. Selon le point de vue, on peut soit considérer  que le polymère se déploit dans $\Z^d$, et est constitué de maillons hétérogènes qui interagissent avec un potentiel placé à l'origine, soit que le polymère est une trajectoire dirigée $d+1$ dimensionnelle interagissant avec une ligne d'accrochage hétérogène. \\

Soient $S$ une marche aléatoire simple dans $\Z^d$ issue de l'origine, et $(\go_{n})_{n\in \N}$  la réalisation typique d'une suite de variables aléatoire centrées indépendantes de variance unitaire. \'Etant donné deux paramètres $\gb>0$ et $h\in \R$ et $N$ un entier pair, on pose (avec le formalisme précédent)
\begin{equation}
 H_{N,\go,\gb,h}=\sum_{n=1}^N [\go_n+(h/\gb)]\ind_{\{S_n=0\}}, 
\end{equation}
 et on définit donc la mesure de polymère $\mu_{\go,\gb,N}(S)$ par sa dérivé de Radon-Nicodym par rapport à la loi de la marche aléatoire simple $\bP$ (on notera $\bE$ l'espérance associée)
\begin{equation}\label{pinningdef}
 \frac{ \dd \mu_{N,\go,\gb,h}}{\dd \bP}(S):= \frac{\exp\left(\sum_{n=1}^N \gb [\go_n+h]\ind_{\{S_n=0\}}\right)\ind_{\{S_N=0\}}}{Z_{N,\go,\gb,h}},
\end{equation}
où
\begin{equation}
 Z_{N,\go,\gb,h}:= \bE\left[ \exp\left(\sum_{n=1}^N [\gb\go_n+h]\ind_{\{S_n=0\}}\ind_{\{S_N=0\}}\right)\right].
\end{equation}
On peut considérer, de manière équivalente, que le polymère se déploie en fait dans $\Z^{1+d}$ et est modélisé par le graphe de $S$ ($(n,S_n)_{n\in[0,N]}$) et que la marche reçoit des contributions énergétiques lorsqu'elle touche la {\sl ligne d'accrochage} sur laquelle la première coordonnée s'annule.

\medskip

\item {\bf Le polymère dirigé en milieu aléatoire} (modèle avec désordre partout). Le polymère est modélisé par une marche dirigée $((n,S_n)_{n\in \N}$  dans $\Z^{1+d}$, où $S$ est une marche au plus proche voisin dans $\Z^d$) et reçoit des contributions énergétiques aléatoire sur chaque site qu'il visite, ce qui peut être interprêté comme l'énergie d'interaction avec une solution hétérogène.\\

Soient $S$ une marche aléatoire simple dans $\Z^d$, et $(\go_{n,z})_{n\in \N, z\in \Z^d}$ la réalisation typique d'un champ de variables aléatoires indépendantes de variance unitaire. On définit l'Hamiltonien
\begin{equation}
 H_{N,\go}(S):= \sum_{n=1}^N \go_{n,S_n},
\end{equation}
 et la mesure de polymère $\mu_{N,\go,\gb}(S)$ pour $\gb>0$ (la température inverse) par sa dérivé de Radon-Nicodym par rapport à la loi de la marche aléatoire simple $\bP$
\begin{equation}
 \frac{\dd \mu_{N,\go,\gb}}{\dd \bP}:=\frac{\exp\left(\gb \sum_{n=1}^N \go_{n,S_n}\right)}{Z_{N,\go,\gb}},
\end{equation}
où
\begin{equation}
 Z_{N,\go,\gb}:= \bE\left[\exp\left(\gb \sum_{n=1}^N \go_{n,S_n}\right)\right].
\end{equation}
\end{itemize}
\medskip
Dans ces deux cas, on s'intéresse au comportement asymptotique des mesures de polymère quand $N$ tend vers l'infini.
Pour le modèle d'accrochage, on remarque que le modèle sans désordre ou {\sl homogène} (i.e.\ avec $\gb=0$) est non--trivial. Les questions étudiées sont: pour le modèle d'acrochage,
\begin{itemize}
 \item \`A quelles conditions sur les paramètres $h$ et $\gb$ la trajectoire $S$ du polymère reste-t-elle accrochée sur la ligne $\{0\}\times \Z^d$ (en un sens que l'on définira précisément par la suite)?
 \item Quel rôle le désordre joue-t-il dans ce phénomène d'accrochage? C'est-à-dire comment comparer qualitativement et quantitativement le modèle désordonné et le modèle homogène?
\end{itemize}
et pour le polymère dirigé en milieu aléatoire,
\begin{itemize}
\item \`A quelle condition sur le paramètre $\gb$, le comportement trajectoriel du polymère à grande échelle est-il ou n'est-il pas modifié par la présence de désordre ?
\item Lorsque le désordre a une influence sur les trajectoires, quelles caractéristiques ent sont changées? 
\end{itemize}
Les deux principales parties de la thèse seront dédiées à l'étude de chaque modèle.

\medskip

Avant d'être l'objet d'une étude mathématique approfondie, les modèles présentés ont été étudiés par les physiciens théoriciens. De nombreux résultats ont donc été prouvés ou conjecturés, la littérature physique constitue  une grande source d'inspiration et de motivation pour l'étude mathématique de ces modèles.

\medskip
Une des méthodes utilisées en mécanique statistique pour dériver des heuristiques est la méthode de (semi-){\sl groupe de renormalisation}. Cette méthode consiste à réécheloner le système en lui otant des degrés de liberté et en transformant l'Hamiltonien: cette procédure est répétée et le rééchelonage doit être calibré pour obtenir, dans la limite, un modèle effectif, invariant par l'action de la transformation de renormalisation. Cela engendre en général des calculs d'une très grande complexité. Les physiciens font ensuite des approximations simplifiant l'Hamiltonien obtenu après transformation pour pouvoir mener les calculs à leur terme. Cela permet souvent de faire des conjectures très fiables, mais qu'il est difficile de transformer en preuves rigoureuses.\\

Il existe cependant des modèles, introduits par les physiciens, pour lesquels la méthode donne des simplifications considérables: les modèles dits {\sl hiérarchiques}, construits sur des réseaux en diamants. Grâce à l'invariance d'échelle des réseaux en diamants, la méthode de groupe de renormalisation peut être appliquée de manière rigoureuse, sans avoir à opérer d'approximations dans les calculs.
Les modèles hiérarchiques ont été introduits pour étudier les modèles d'Ising/Potts \cite{cf:DG}, de percolation au dernier passage \cite{cf:R_al}, de polymère dirigés en milieu aléatoire \cite{cf:CD,cf:DGr}, ou les modèles d'accrochage \cite{cf:DHV}.

Contrairement aux modèles en champ moyen (polymères sur les arbres, modèles sur le graphe complet \dots), les modèles hiérarchiques conservent l'essence des modèles définis sur $\Z^d$. C'est pourquoi une fraction importante de cette thèse sera dédiée à l'étude des modèles hiérarchiques aussi bien pour les modèles d'accrochage, que pour les polymères dirigés en mileu aléatoire.


\section{Modèles d'accrochage}

\subsection{Modèle homogène}
Tous les modèles d'accrochage que nous allons étudier ont déjà un interêt dans leur version {\sl homogène}. \`A de nombreuses reprises, le modèle d'accrochage homogène sera d'ailleurs utilisé comme outil technique, aussi bien pour les  modèles d'accrochage désordonnés (par exemple pour des arguments de comparaisons) que pour  polymères dirigés en milieu aléatoire (arguments de second moment, méthode des répliques). Dans leur version non-hiérarchique, ils possèdent la propriété remarquable d'être {\sl exactement résoluble} (voir \cite{cf:Fisher}).
Nous proposons ici une étude sommaire du modèle homogène (pour plus de détails voir les premiers chapitres de \cite{cf:Book} dont cette introduction s'inspire).
Rappellons la définition du modèle en dimension $1$:
\medskip

 Considérons $S=(S_n)_{n\in \N}$ la marche aléatoire sur $\Z$ (sous la loi de probabilité $\bP$) définie par $S_n=\sum_{i=1}^n X_i$ où les $X_i$ sont des variables aléatoires indépendantes identiquement distribuées (i.i.d.) satisfaisant
\begin{equation}
 \bP(X_1=\pm 1):= 1/2.
\end{equation}
Pour $h\in \R$, on modifie la mesure $\bP$ en attribuant un bonus ou un malus d'énergie à une trajectoire $S$ à chaque fois qu'elle passe par zéro. On obtient ainsi la famille de {\sl mesures de polymère} $\bP_{N,h}$,  $N\in 2\N$ (on note $\bE_{N,h}$ l'espérance associée),  définie par
\begin{equation}\label{ouchouch}
 \frac{\dd \bP_{N,h}}{\dd \bP}:= \frac{1}{Z_{N,h}}\exp(h L_N)\ind_{\{S_N=0\}},
\end{equation}
où $L_N:=\sum_{n=1}^N \ind_{\{S_n=0\}}$ désigne le temps local en zéro et
\begin{equation}\label{ouchouch2}
 Z_{N,h}:=\bE\left[\exp(h L_N)\ind_{\{S_N=0\}}\right].
\end{equation}
La contrainte $S_N=0$ est une simple condition au bord, qui peut être modifiée sans changer les propriétés essentielles du système. Pour l'instant, elle oblige à considérer $N$ pair.
\medskip

On cherche à savoir si la marche aléatoire $(S_n)_{n\in[0,N]}$ sous la mesure $\bP_{N,h}$ reste accrochée à la ligne $S=0$ sous l'influence de la force d'accrochage $h$. On étudie donc la fraction de contact moyenne 
$f_N(h):=\frac{1}{N}\bE_{N,h}\left[L_N\right]$, ou plus précisement son comportement lorsque $N$ tend vers l'infini.
On définit à cet effet la quantité
\begin{equation}
\tf(h):=\lim_{N\to\infty} \frac{1}{N}\log Z_{N,h},
\end{equation}
que nommée {\sl énergie libre}. En tant que limite de fonctions convexes et croissantes de $h$, l'énergie libre est convexe et croissante.
La fraction de contact moyenne est liée à cette quantité par la relation
\begin{equation}
 \frac{\dd}{\dd h} \frac{1}{N}\log Z_{N,h}:=\frac{1}{N} \bE_{N,h}\left[L_N\right].
\end{equation}
Cette égalité passe à la limite par convexité. On a donc
\begin{equation}
 \frac{\dd}{\dd h}\tf(h)=\lim_{N\to\infty} f_N(h),
\end{equation}
lorsque le membre de gauche existe.
\medskip
L'existence de la limite $\tf(h)$ est due au caractère sur-additif de $\log Z_{N,h}$. De plus, l'inégalité
\begin{equation}
 Z_{N,h}\ge \exp(h)\bP(S_N=0 \text{ et } S_n\ne 0 , \forall n\in [1,N-1])\sim_{N\to\infty} \frac{\exp(h)}{\sqrt{\pi/2}N^{3/2}}
\end{equation}
assure la positivité de $\tf(h)$ pour tout $h$. On peut vérifier aussi que l'énergie libre est nulle lorsque $h$ est négatif ou nul.
Cela donne de nombreux renseignements sur la courbe de l'énergie libre (voir figure \ref{energielibre}).

\medskip
Pour une meilleure analyse du modèle il est préférable de l'étudier dans un cadre plus général. La mesure de polymère transforme la marche aléatoire simple en changeant la loi des temps de retour en zéro, mais ne modifie pas la loi des excursions hors de l'origine lorsque leur taille a été fixée. Pour cette raison, on se met à considérer seulement les temps de retour sur la ligne d'accrochage. Dans le cas de la marche aléatoire, l'espacement des temps de retour constitue une suite de variables aléatoires à valeurs entières, i.i.d., à queue de distribution polynomiale. C'est pourquoi on modélisera notre processus général de retours en zéro par une suite de temps aléatoires $(\tau_n)_{n\ge 0}$ dans $\N\cup\{\infty\}$ dont les accroissements sont gouvernés par une loi $\bP$ vérifiant:
\begin{itemize}
\item $\tau_0=0$ presque surement.
\item $(\tau_n-\tau_{n-1})_{n\ge 1}$ constitue une suite de variables aléatoires i.i.d.\ à valeur dans $\N\cup\{\infty\}$.
\item Il existe un réel $\alpha>0$ tel que 
\begin{equation}
K(n):=\bP(\tau_1=n)\sim_{n\to \infty} \frac{cste.}{n^{1+\alpha}}. 
\end{equation}
\end{itemize}
(dans le développement, on se placera dans le cadre un peu plus large où $K(n)$ est une fonction à variation régulière (voir \cite{cf:RegVar})).
Un tel processus s'appelle un processus de renouvellement.

En divisant les temps de retour de la marche aléatoire simple par deux, on se retrouve dans le cadre ci-dessus avec $\alpha=1/2$.
Le cas de la marche alétoire en dimension $d$, $d\ge 3$ correspond à $\alpha=d/2-1$, le cas $d=2$ corresponds à $\alpha=0$ avec une correction logarithmique. On considèrera  (de manière impropre) $\tau$ comme un sous ensemble de $\N\cup\{0\}$.
\medskip

Les définitions de \eqref{ouchouch} and  \eqref{ouchouch2} s'adaptent parfaitement en posant $L_N:=|\tau\cap [1,N]|$ (où $|\ . \ |$ représente la cardinalité d'un ensemble), en remplaçant $\ind_{\{S_N=0\}}$ par $\ind_{\{N\in \tau\}}$ et en considérant $\bP_{N,h}$ comme une modification de la mesure $\bP$ qui gouverne $\tau$.

Ce nouveau formalisme, bien qu'apparament plus complexe que le précédent, permet un traitement mathématique simple du modèle et 
le replace dans un contexte plus large.
En effet, on peut prouver sans trop d'efforts la formule explicite suivante pour l'énergie libre.
\begin{proposition}\label{th:libenerh}
 L'énergie libre est la solution de l'équation (d'inconnue $x$)
\begin{equation}\label{teteteopopop}
 \sum_{n=1}^{\infty} \exp(-n x)K(n)=e^{-h},
\end{equation}
si elle existe et zéro sinon (si $x<0$ le membre de gauche diverge, la solution de \eqref{teteteopopop} est donc toujours positive).
En particulier, 
\begin{equation}
 h_c=-\log \bP(\tau_1<\infty),
\end{equation}
et
\begin{equation}\label{critex}
\begin{split}
 \tf(h)&\sim_{h\to h_c^+} cste. (h-h_c)^{\max(1,\alpha^{-1})} \quad \text{ lorsque } \alpha\ne 1.\\
 \tf(h)&=h-\log K(1)+o(1)\quad \text{ quand $h$ tend vers l'infini.}
\end{split}
\end{equation}
\end{proposition}

Pour le cas $\alpha=1$, il faut ajouter une correction logarithmique à la première ligne de \eqref{critex}.
\medskip
\begin{proof}
 On réécrit la fonction de partition comme la somme des contributions de toutes les trajectoires accrochées en $N$,
\begin{equation}
 Z_N=\sum_{n\le N}\sum_{l_1+l_2+\dots+l_n=N} \exp(nh) \prod_{i=1}^n K(l_i).
\end{equation}
Soit $x_0$ la solution de l'équation \eqref{teteteopopop}, on pose $\tilde K(n)=\exp(h)\exp(-x_0 n)K(n)$. Une simple réécriture donne alors
\begin{equation}
 Z_N=\exp(Nx_0)\sum_{n\le N}\sum_{l_1+l_2+\dots+l_n=N} \prod_{i=1}^n \tilde K(l_i).
\end{equation}
Soit $\tilde \tau$ le processus de renouvellement défini par $\bP(\tilde \tau_1=n)=\tilde K(n)$. On a
\begin{equation}
 Z_N=\exp(Nx_0)\bP(N\in \tilde \tau).
\end{equation}
Ce renouvellement est récurrent, on a même $\bE[\tilde \tau_1]<\infty$ si $x_0>0$,  on peut donc appliquer le théorème du renouvellement (pour l'énoncé et une preuve, voir \cite[Chapitre I, Théorème 2.2]{cf:As}). On obtient
\begin{equation}
 Z_N\sim_{N\to\infty}\frac{1} {\bE[\tau_1]} \exp(Nx_0).
\end{equation}
Dans le cas où \eqref{teteteopopop} n'a pas de solution, ou dans le cas $x_0=0$, on peut poser $\tilde K(n)=\exp(h)K(n)$. Comme $\sum_{n\in \N} \tilde K(n)\le 1$, on peut définir le renouvellement $ \tilde \tau$ de la même manière. On a alors
\begin{equation}
 \exp(h)K(N)\le Z_N\le \bP(N\in \tilde \tau)\le 1.
\end{equation}
Ce qui donne le résultat.
\medskip

Lorsque $h$ tend vers l'infini, $\tf(h)$ tend vers l'infini. On a donc
\begin{equation}
 K(1)e^{-\tf(h)}(1+o(1))=e^{-h}
\end{equation}
Cela donne l'approximation au deuxième ordre de l'énergie libre pour $h$ tendant vers l'infini.
\medskip

Nous prouvons maintenant le résultat d'équivalence de l'énergie libre au voisinage du point critique (on se contentera du cas $\alpha\in(0,1)$). On se place dans le cas où $h_c=0$ (sinon une multiplication par $\exp(h_c)$ nous ramène à ce cas). On utilise un théorème Abélien qui se trouve dans \cite[Théorème 1.7.1, et Corrolaire 8.7.1]{cf:RegVar} qui assure que, lorsque $K(n)\sim C n^{-(1+\alpha)}$,
\begin{equation}
 1-\sum_{n=1}^{\infty} \exp(-x n)K(n)\sim C x^{\alpha}\gG(1-\alpha).
\end{equation}
Il est trivial de voir que $\tf(h)$ tend vers $0$ quand $h$ tend vers $h_c$, on en déduit 
\begin{equation}
 \tf(h)\stackrel{h\to 0_+}{\sim}\left(\frac{h}{C\gG(1-\alpha)}\right)^{\frac{1}{\alpha}}.
\end{equation}
Le cas $\alpha>1$ est plus facile, voir \cite[Theorem 2.1]{cf:Book} pour une preuve.\\
\end{proof}

La valeur critique $h_c$ sépare les phases {\sl localisée} et {\sl délocalisée} du modèle. Au voisinage de $h_c$, l'énergie libre a un comportement polynomial. On appelle l'exposant $\max(1,\alpha^{-1})$ {\sl exposant critique}, il contient beaucoup d'informations sur le système. Le résultat montre que quite à remplacer $K(n)$ par $\exp(h_c)K(n)$, on peut toujours considérer que $h_c(0)=0$, i.e.\ que le renouvellement $\tau$ est réccurent. C'est ce que nous ferons dans la suite de cette introduction.

\begin{figure}[hlt]
\begin{center}
\leavevmode
\epsfxsize =14 cm
\psfragscanon
\psfrag{f(h)}{\small $\tf(h)$}
\psfrag{h}{\small  $h$}
\psfrag{O}{\small O}
\psfrag{hc}{\small $h_c$}
\psfrag{f(h)sim.}{\small $\tf(h)\sim cste.(h-h_c)^{\frac{1}{\alpha}}$.}
\psfrag{f(h)sim..}{\small $\tf(h)=h-\log K(1)+o(1)$.}
\psfrag{loc}{\small \bf phase localisée}
\psfrag{deloc}{\small \bf phase délocalisée}
\epsfbox{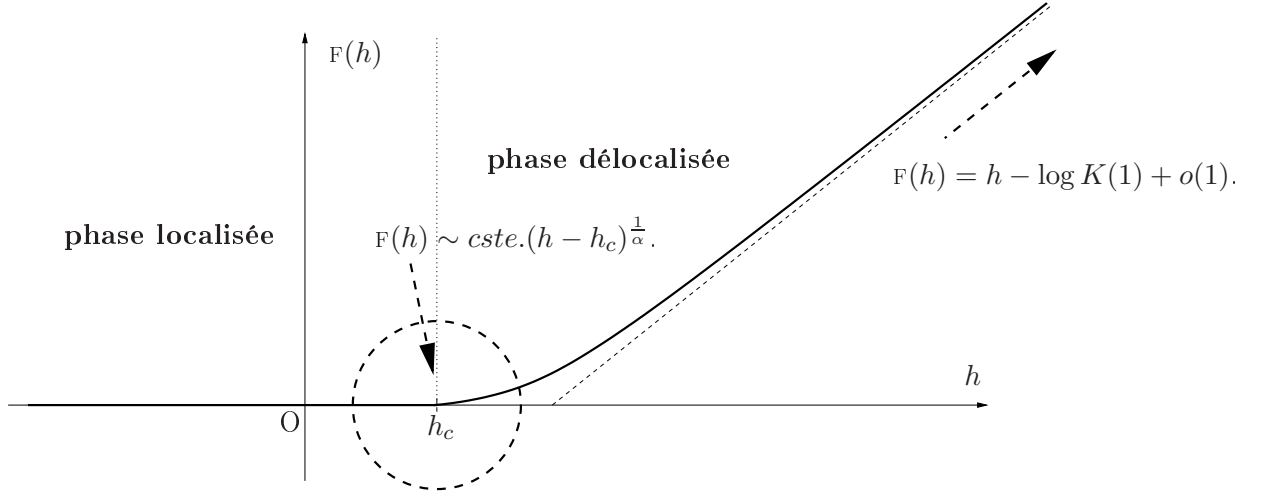}
\end{center}
\caption{L'allure de la courbe de l'énergie libre $\tf(h)$ en fonction $h$ pour $\alpha\in(0,1)$, tel qu'elle découle de la Proposition \ref{th:libenerh} et des observations faites sur la croissance et la convexité. }
\label{energielibre}
\end{figure}

\subsection{Modèle inhomogène}

Le formalisme avec processus de renouvellement peut facilement  être adapté dans le cas désordonné.
Soient un processus de renouvellement $\tau$ vérifiant les hypothèses de la section précédente pour un certain réel $\alpha>0$, de loi $\bP$, et $(\go_{n})_{n\in \N}$ la réalisation d'une suite de variables aléatoires i.i.d.\ centrées (de loi $\bbP$) de variance unitaire et vérifiant 
\begin{equation}\label{finitemomo}
 \gl(\gb):=\log \bbE\left[\exp(\gb\go_1)\right]<\infty, \quad \forall\ \gb>0.
\end{equation}
 On pose $\gd_n:=\ind_{n\in \tau}$, pour $n\in \N$ et on définit la mesure de polymère $\bP_{N,\go,\gb,h}$ comme modification de la loi de $\tau$ 
\begin{equation}
\frac{\dd \bP_{N,\go,\gb,h}}{\dd \bP}(\tau):=\frac{1}{Z_{N,\go,\gb,h}}\exp\left(\sum_{n=1}^N [\gb\go_n+h]\gd_n\right)\gd_N.
\end{equation}
avec
\begin{equation}
Z_{N,\go,\gb,h}:=\bE\left[\exp(\sum_{n=1}^N [\gb\go_n+h]\gd_n)\gd_N\right].
\end{equation}
Si l'on nomme $\theta$ l'opération de shift sur l'environnement, i.e. $\theta((\go_n)_{n_\in\N})=(\go_{n+1})_{n\in\N}$, alors on a pour tout $N, N'\in \N$
\begin{equation}
 Z_{N+N',\go,\gb,h}\ge \bE\left[\exp\left(\sum_{n=1}^{N+N'} [\gb\go_n+h]\gd_n\right)\gd_{N+N'}\gd_{N}\right]= Z_{N,\go,\gb,h}Z_{N',\theta^N \go,\gb,h},
\end{equation}
donc 
\begin{equation}
 \log  Z_{N+N',\go,\gb,h} \ge \log Z_{N,\go,\gb,h}+ \log Z_{N',\theta^N \go,\gb,h}.
\end{equation}

Cette propriété et le théorème ergodique sur-addiditif de Kingman (voir \cite{cf:King}) permettent de définir, comme pour le cas homogène, l'énergie libre:
\begin{proposition}
La limite
\begin{equation}
\tf(\gb,h):=\lim_{N\to\infty}\frac{1}{N}\log Z_{N,\go,\gb,h}.
\end{equation}
existe pour presque toute réalisation de l'environnement $\go$, et est égale à\\
$\lim_{N\to \infty} N^{-1} \bbE \left[\log Z_{N,\go,\gb,h}\right]$ . C'est une fonction convexe croissante de $h$, elle est nulle pour $h$ suffisament petit.
\end{proposition}
\medskip 
On pose,
\begin{equation}
 h_c(\gb):=\inf\{h\ : \tf(\gb,h)=0\}.
\end{equation}
\medskip

Comparons maintenant l'énergie libre du système désordonné à celle du système homogène. 
L'inégalité de Jensen donne
\begin{equation}
 \bbE \left[\log Z_{N,\go,\gb,h}\right]\le \log \bbE \left[ Z_{N,\go,\gb,h}\right]=\log \bE\left[\exp\left((\gl(\gb)+h)L_N\right)\gd_N\right],
\end{equation}
la dernière égalité découlant du théorème de Fubini. En divisant par $N$ et en passant à la limite il apparaît que
\begin{equation}
\tf(\gb,h)\le \tf(0,h+\gl(\gb)).
\end{equation}
La quantité $\bbE \left[ Z_{N,\go,\gb,h}\right]$ est la fonction de partition du système où le désordre a été moyenné, On parle de modèle {\sl recuit} ou {\sl annealed}. Le fait que le modèle annealed correspond au modèle homogène est une particularité de notre modèle.
Le système vraiment désordonné, où l'on regarde la mesure de polymère pour une réalisation figée  du désordre, est appelé modèle à {\sl désordre gelé} ou {\sl quenched}.
\medskip

Un autre inégalité de convexité donne $\tf(\gb,h)\ge \tf(0,h+\gl(\gb))$. 
Ces deux inégalités impliquent (dans le cas ou le renouvellement $\tau$ est réccurent),
\begin{equation}
-\gl(\gb)\le h_c(\gb)\le 0.
\end{equation}
Il a été montré sous des conditions assez générales (voir \cite{cf:AS}) que l'inégalité de droite n'est jamais vraie. En revanche, savoir si l'inégalité de gauche est satisfaite ou pas donne une information sur l'influence du désordre (voir figure \ref{energielibre2}):
\begin{itemize}
 \item Si $h_c(\gb)= -\gl(\gb)$ cela signifie (au moins heuristiquement) que le désordre n'a pas d'influence sur le comportement du système, i.e.\ que les modèles {\sl quenched} et {\sl annealed} coïncident. Dans ce cas que le désordre est dit {\sl non-pertinent}.
 \item Si $h_c(\gb)> -\gl(\gb)$, cela signifie que le désordre change le comportement du système, il est dit {\sl pertinent}.
\end{itemize}
Le critère de déplacement du point critique n'est pas le seul pour évaluer la pertinence du désordre. Pour les physiciens théoriciens il importe aussi (en fait, surtout) de savoir si l'exposant critique du modèle {\sl annealed} est conservé.\\
\medskip

\begin{figure}[hlt]
\begin{center}
\leavevmode
\epsfxsize =14 cm
\psfragscanon
\psfrag{glgb}{\tiny $-\gl(\gb)$}
\psfrag{h}{\tiny  $h$}
\psfrag{f}{\tiny $\tf$}
\psfrag{hc}{\tiny $h_c(\gb)$}
\psfrag{hcbetasim}{\tiny $h_c(\gb)=-\gl(\gb)$}
\psfrag{deuxlig} {\tiny et $\tf(\gb,h)\sim \tf(0,h+\gl(\gb))$}
\psfrag{O}{\tiny O}
\psfrag{fbetah}{\tiny $\tf(0,h+\gl(\gb))$}
\psfrag{fh}{\tiny $\tf(0,h)$}
\psfrag{irrelevant}{\small désordre non-pertinent}
\psfrag{relevant}{\small désordre pertinent}
\epsfbox{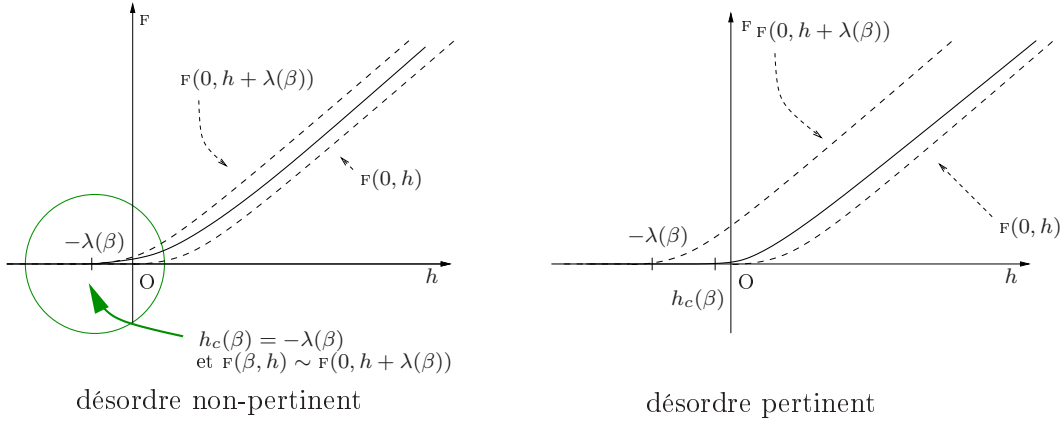}
\end{center}
\caption{Présentation de deux cas typiques de désordre, pertinent et non-pertinent: les courbes en trait plein représentent $\tf(\gb,h)$, les courbes en trait pointillé les deux bornes obtenues par convexité $\tf(0,h)$ et $\tf(0,\gl(\gb)$.}
\label{energielibre2}
\end{figure}

Le {\sl critère de Harris} (d'après le physicien A.B.\ Harris), donne une prédiction très générale concernant la pertinence du désordre à haute température \cite{cf:Harris}. La pertinence du désordre dans un modèle désordonné ne dépend que de l'exposant critique du système {\sl annealed}. Si cet exposant est strictement supérieur à $2$ le désordre doit être non-pertinent à haute température, alors que dans le cas où l'exposant est strictement inférieur à $2$ le désordre est pertinent à toute température. Dans le cadre du modèle d'accrochage, ces deux cas correspondent respectivement à $\alpha<1/2$ et $\alpha>1/2$ (cf.\ Proposition \ref{th:libenerh}).
Cette prédiction se trouve confirmé dans le cas des modèoles d'accrochage désordonnés par les physiciens \cite{cf:DHV, cf:FLNO}. En revanche le critère de Harris ne donne pas d'information sur le cas marginal où l'exposant critique est égal à $2$ (i.e. $\alpha=1/2$), et dans ce cas, il est conjecturé que la pertinence du désordre dépend du modèle considéré. Dans le cas des modèles d'accrochage, il n'y a pas de consensus parmis les physiciens quant à la pertinence du désordre, avec deux affirmations contradictoires dans \cite{cf:DHV} (pertinence) et \cite{cf:FLNO} (non-pertinence), chacune trouvant ensuite des partisans dans la communauté scientifique (\cite{cf:SC, cf:TangChate, cf:BM} soutenant la première affirmation, \cite{ cf:GS1, cf:GS2, cf:GN} la seconde).

\medskip

On peut donner une explication heuristique du critère de Harris dans le cadre des polymères. 
En fait, savoir si les énergies libres {\sl quenched} et {\sl annealed} ont le même comportement revient à savoir (très approximativement) si $Z_{N,\go,\gb,h}$ et son espérance ont le même comportement lorsque $N$ tend vers l'infini, ce qui arrive typiquement lorsque $\bbE\left[Z_{N,\go,\gb,h}^2\right]$ reste bornée quand $N$ tends vers l'infini. En fait, dès que $h>-\gl(\gb)$, $\bbE\left[Z_{N,\go,\gb,h}^2\right]$ diverge pour tout $\gb>0$; mais
comme on s'interesse à ce qui se passe au voisinage du point critique {\sl annealed}, on se contente d'étudier le second moment de $Z_{N,\go,\gb,h}$ pour $h=-\gl(\gb)$ (le point critique annealed). Le théorème de Fubini donne

\begin{equation}\begin{split}
\bbP\left[Z_N^2\right]&=\bP^{\otimes 2}\bbP\left[\exp\left(\sum_{n=0}^N (\gb\go_n-\gl(\gb))(\gd_n^{(1)}+\gd_n^{(2)})\right) \right]\\
                      &=\bP^{\otimes 2} \left[\exp\left(\sum_{n=0}^N \gd_n^{(1)}\gd_n^{(2)}(\gl(2\gb)-2\gl(\gb))\right)\right],
\end{split}\end{equation}
où $\bP^{\otimes 2}$ est la loi produit qui régit deux renouvellement indépendants $\tau^{(1)}$ et $\tau^{(2)}$ de même loi que $\tau$ et où $\gd^{(i)}_n=\ind_{n\in\tau^{(i)}}$.
Ce qu'on obtient est la fonction de partition d'un modèle d'accrochage homogène de paramètre $\gl(\gb)=\gl(2\gb)-2\gl(\gb)$, associé au processus de renouvellement $\tau^{(1)}\cap \tau^{(2)}$ (on peut vérifier que les sauts de $\tau^{(1)}\cap \tau^{(2)}$ sont i.i.d., il s'agit donc bien d'un processus de renouvellement). Savoir si le second moment diverge pour $\gb$ arbitrairement petit, revient donc (cf. Proposition \ref{th:libenerh}) à savoir si le renouvellement $\tau^{(1)}\cap \tau^{(2)}$ est récurrent. Sachant que $\bP(n\in \tau)\sim cste.n^{\alpha-1}$ (voir par exemple \cite{cf:Doney}), on peut vérifier que
\begin{equation}
 \sum_{n=1}^\infty \bP^{\otimes 2}(n\in\tau^{(1)}\cap \tau^{(2)})= \sum_{n=1}^{\infty}\bP(n\in\tau)^2=\infty  \Leftrightarrow   \alpha\ge 1/2.
\end{equation}
Cet argument de second moment peut être utilisé pour démontrer rigoureusement des bornes inférieure sur l'énergie libre (voir \cite{cf:Ken}).

Nous présentons maintenant les différents résultats mathématiques récements obtenus concernant la pertinence du désordre pour les modèles d'accrochage désordonnés:
\medskip

Un résultat de G. Giacomin et F. Toninelli \cite{cf:GT_cmp}, montre que sous certaines conditions sur la loi de l'environnement (gaussien convient, mais leur résultat est en fait plus général), la présence du désordre {\sl lisse} la courbe de l'énergie libre.
\medskip

\begin{theoreme}\label{th:smooothing}
Lorsque l'environnement est gaussien, pour tout $\gb>0$ il existe une constante $c(\gb)$ telle que pour tout
$\alpha\in [1,\infty)$ et pour tout $h\in \bbR$
\begin{equation}
 \tf(\gb,h)\le \alpha c(\gb)(h-h_c(\gb))_+^2.
\end{equation}
\end{theoreme}
\medskip

Ce résultat assure que l'exposant critique du système désordonné (s'il existe) est supérieur ou égal à deux, quel que soit la valeur de $\alpha$.
En particulier, cela prouve que l'exposant critique est modifié par le désordre si $\alpha>1/2$.
\medskip

Un résultat complémentaire concernant la {\sl non-pertinence} a été ensuite démontré par K. Alexander \cite{cf:Ken} (une preuve alternative a ensuite été proposée par F.\ Toninelli \cite{cf:T_cmp}).
\medskip

\begin{theoreme}
Lorsque l'environnement est gaussien, et que $\alpha<1/2$, il existe $\gb_0$ tel que, pour tout $\gb\in(0,\gb_0)$, on a
\begin{equation}\begin{split}
  h_c(\gb)=-\gl(\gb)=-\gb^2/2.\\
 \tf(\gb,h-\gb^2/2)\sim_{h\to 0+}\tf(0,h).
\end{split}\end{equation}
\end{theoreme}
\medskip
Donc pour $\alpha<1/2$, lorsque la température est suffisament élevée (i.e.\ lorsque $\gb$ est suffisament petit), il n'y a ni déplacement du point critique ni modification de l'exposant critique. F.\ Toninelli a  démontré que lorsque le désordre est constitué de variables non--bornées les points critiques {\sl quenched} et {\sl annealed} diffèrent à basse température  pour toute valeur de $\alpha>0$ \cite{cf:T_fractmom}.
\medskip

Plus récemment, il a  été démontré que pour $\alpha>1/2$ il y a un déplacement du point critique à toute température, et que l'on peut estimer quantitativement ce déplacement
\medskip

\begin{theoreme}
Pour $\alpha\in(1/2,1)$ ou $\alpha>1$, $h_c(\gb)<h_c(0)+\gl(\gb)$ pour toute valeur de $\gb$. De plus, pour tout $\gep>0$, il existe une constante $c$ (dépendante de la loi du renouvellement et de la loi de l'environnement) telle que pour tout $\gb<1$
\begin{equation}
 -c\gb^{\min\left(2,\frac{2\alpha}{2\alpha-1}\right)} \le h_c(\gb)+\gl(\gb)\le-\frac{1}{c} \gb^{\min\left(2,\frac{2\alpha}{2\alpha-1} \right)}.
\end{equation}
 \end{theoreme}
\medskip

La borne inférieure a été prouvée dans \cite{cf:Ken} en utilisant des méthodes identiques à celles utilisées pour le cas $\alpha<1/2$. Le déplacement du point critique a été démontré par l'auteur, en collaboration avec B. Derrida, G. Giacomin et 
F. Toninelli, avec une borne supérieure sur  $h_c(\gb)+\gl(\gb)$ légèrement moins précise (voir Chapitre \ref{CHAPDGLT}). Le résultat a donné lieu à un article publié dans la revue {\sl Communication in Mathematical Physics} \cite{cf:DGLT}.
La borne a depuis été améliorée par K.\ Alexander et N.\ Zygouras \cite{cf:AZ} pour coincider (en ordre de grandeur) avec la borne inférieure, en utilisant une méthode différente. La méthode utilisée dans \cite{cf:DGLT} peut elle aussi être adaptée pour obtenir la borne supérieure optimale.\\
\medskip

Enfin, la pertinence du désordre dans le cas marginal $\alpha=1/2$ a pu être démontrée, avec des bornes sur le déplacement du point critique,

\begin{theoreme}
Lorsque $\alpha=1/2$, $h_c(\gb)>-\gl(\gb)$ pour toute valeur de $\gb$. De plus pour tout $\gep>0$ il existe des constantes $c$ et $\gb_0$
(dépendantes de la loi du renouvellement et de la loi de l'environnement de $\gep$) telles que
\begin{equation}
-\exp(-1/c\gb^2) \le h_c(\gb)+\gl(\gb)\le -\exp(-c/\gb^{2+\gep}).
\end{equation}
\end{theoreme}

La borne inférieure est prouvée dans l'article de K.\ Alexander \cite{cf:Ken}. Le déplacement du point critique a été démontré par l'auteur, en collaboration avec G.\ Giacomin et F.\ Toninelli dans le cas gaussien avec une borne supérieure sur le déplacement du point critique égale à $\exp(-c/\gb^4)$. Ce travail a donné lieu à un article à paraître dans la revue {\sl Communication on Pure and Applied Mathematics} \cite{cf:GLT_marg}. Le résultat a ensuite été généralisé à tout type d'environnement, avec une amélioration de la borne, par les mêmes auteurs (i.e.\ le résultat mentionné ci-dessus), l'article correspondant est en cours d'examination pour publication \cite{cf:kbodies}. Ces travaux constituent les chapitres \ref{MARGREL} et \ref{DISRELCPS} de cette thèse.

\medskip
Les résultats présentés ci-dessus étaient déjà conjecturés dans la littérature physique en particulier dans un article de B.\ Derrida, V.\ Hakim et J.\ Vannimenus \cite{cf:DHV}, avec des heuristiques de preuves utilisant des techniques de groupes de renormalisation.
Cet article présente également un modèle d'accrochage sur des réseaux en diamants.
Pour ce modèle, ces idées de groupe de renormalisations peuvent être appliquées directement de manière rigoureuse. Pour cette raison nous avons d'abord étudié les questions considérées ci-dessus dans ce modèle hiérarchique.

\subsection{Modèle hiérarchique}\label{sectionhierar}

Le modèle d'accrochage hiérarchique correspond à l'accrochage d'une marche aléatoire dirigée sur une famille croissante de réseaux possédant une structure auto--similaire.
Plus précisément, ayant fixé $b,\ s\ge 2$ deux entiers, on définit la suite de réseaux $D_n$ comme suit (voir figure \ref{fig:diamond2}):
\begin{itemize}
 \item $D_0$ est une arête simple reliant deux points $A$ et $B$.
 \item $D_{n+1}$ est une réplique de $D_n$ où chaque arête est remplacée par $s$ arêtes en série et $b$ en parallèle. 
\end{itemize}
Sur ces réseaux, on fixe un chemin reliant $A$ et $B$ pour jouer le rôle de la ligne d'accrochage (tous les chemins étant équivalents le choix n'a pas d'importance) que l'on nomme $\sigma$. On considère un environnement aléatoire $(\go_e)_{e\in \sigma}$ fixé sur les arêtes de $\sigma$ composées de variables aléatoires i.i.d.\ centrées, de variance unitaire et satisfaisant \eqref{finitemomo} (la loi associée est $\bbP$). On considère $\gG_n$ l'ensemble des chemins dirigés de $D_n$ et $\bP_n$ la mesure uniforme sur $\gG_n$, et on définit la mesure de polymère sur $\gG_n$ associée aux paramètres $h\in \R$, $\gb>0$ par
\begin{equation}
 \frac{\dd \bP_{n,\go,\gb,h}}{\dd \bP_n}(\gamma):=\frac{1}{R_{n,\go,\gb,h}}\exp\left(\sum_{e\in \gamma}(\gb\go_{e}+h)\ind_{e\in \sigma}\right).
\end{equation}
où
\begin{equation}
R_{n,\go,\gb,h}:=\bP_n \left[\exp\left(\sum_{e\in \gamma}(\gb\go_{e}+h)\ind_{e\in \sigma}\right)\right]=R_n.
\end{equation}
La propriété remarquable de ce modèle est que la fonction de partition $R_n$ vérifie la relation de réccurence suivante
\begin{equation}\label{machinrecu}\begin{split}
 R_0&\stackrel{(\mathcal L)}{=}\exp(\gb \go+h),\\
 R_{n+1}&\stackrel{(\mathcal L)}{=}\frac{R^{(1)}_n\dots R^{(s)}_n+(b-1)}{b},
\end{split}\end{equation}
où les égalités sont en loi, et où $R^{(1)}_n, \dots, R_n^{(s)}$ sont des variables aléatoires i.i.d.\ de même loi que $R_n$, et $\go$ une variable aléatoire de même loi que les $\go_e$.
\medskip

Les équations \eqref{machinrecu} peuvent être définies pour n'importe quelle valeur de $b\ne 0$. Pour $b>1$, $R_n$ peut toujours être interprêtée comme la fonction de partition d'un modèle d'accrochage, et nous l'utiliserons comme définition pour $R_n$.\\

Un modèle voisin de celui présenté ci-dessus est celui où le désordre est situé non-plus sur les arêtes de $\sigma$, mais sur les sites (exceptés $A$ et $B$, mais on remarque que cette convention ne change pas la mesure de polymère, mais seulement la définition de $R_n$). Dans cas là, la récurrence permettant de construire la fonction de partition est légèrement modifiée et l'on a

\begin{equation}
\label{machinrecu2}\begin{split}
 R_0&\stackrel{(\mathcal L)}{=}1,\\
 R_{n+1}&\stackrel{(\mathcal L)}{=}\frac{R^{(1)}_n\dots R^{(s)}_nA^{(1)}\dots A^{(s-1)}+(b-1)}{b},
\end{split}
\end{equation}
où les égalités sont en loi,  où $R^{(1)}_n, \dots, R_n^{(s)}$ sont des variables aléatoires i.i.d.\ de même loi que $R_n$, et $A^{(1)},\dots, A^{(s-1)}$ des variables aléatoires i.i.d.\ indépendantes des $R_n^{(i)}$ qui ont pour loi la loi de $\exp(\gb\go+h)$.
Là encore on peut définir le modèle pour toute valeur de $b> 1$, $b\in \R$.

\begin{figure}[hlt]
\begin{center}
\leavevmode
\epsfxsize =14 cm
\psfragscanon
\psfrag{l0}{\small  D$0$}
\psfrag{l1}{\small  D$1$}
\psfrag{l2}{\small  D$2$}
\psfrag{d0}{$d_0$}
\psfrag{d1}{$d_1$}
\psfrag{d2}{$d_2$}
\psfrag{d3}{$d_3$}
\psfrag{d4}{$d_4$}
\psfrag{u1}{$u_1$}
\psfrag{u2}{$u_2$}
\psfrag{u3}{$u_3$}
\psfrag{u4}{$u_4$}
\psfrag{traject}{\small trajectoire $a$}
\psfrag{traject2}{\small trajectoire $b$}
\epsfbox{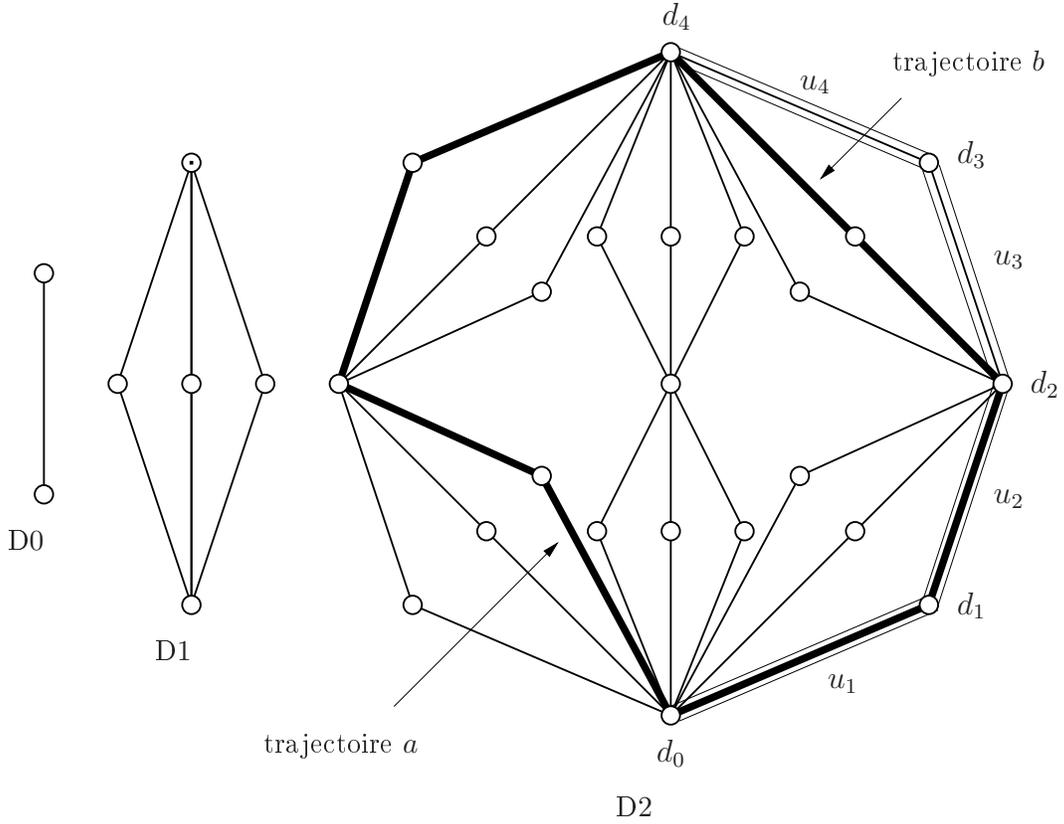}
\end{center}
\caption{Les trois premiers niveaux du réseau hiérarchique pour $s=2$, $b=3$. Sur le troisième schéma, on a représenté la ligne d'accrochage, composée des arêtes ($u_1,\ u_2,\ u_3,\ u_4$). Selon le modèle considéré, le désordre peut être situé sur les arêtes ($u_1,\ u_2,\ u_3,\ u_4$) ou sur les sites $d_1, d_2, d_3$.}
\label{fig:diamond2}
\end{figure}

Pour ces modèles on définit l'énergie libre comme précédemment
\begin{proposition}
 Dans les deux modèles (désordre par site et par arête), la limite
\begin{equation}
 \tf(\gb,h):=\lim_{n\to \infty}\frac{1}{s^n}\log R_n,
\end{equation}
existe presque surement et est égale à
\begin{equation}
 \lim_{n\to \infty}\frac{1}{s^n}\bbE\left[\log R_n \right].
\end{equation}
\end{proposition}
\medskip
Les inégalitées de convexité précédentes sont toujours valables et on a 
\begin{equation}\begin{split} \label{inqualit}
 \tf(0,h)  &\le  \tf(\gb,h)\le \tf(0,h+\gl(\gb)),\\
h_c(0)-\gl(\gb) &\le  h_c(\gb) \le h_c(0). \end{split}
\end{equation}

On se pose la même question que pour le modèle non-hiérarchique: Quand le désordre est il pertinent? Pour pouvoir appliquer le critère de Harris, il faut étudier tout d'abord le système homogène. Pour simplifer, les calculs faits dans cette partie concerneront uniquement le modèle avec désordre par arête. Etudions la récurence déterministe qui définit la fonction de partition lorsque $\gb=0$:
\begin{equation}\label{recdet}
\begin{split}
r_0&=\exp(h)\\
r_n&=\frac{r_n^s-(b-1)}{b}. 
\end{split}
\end{equation}
\`A un système de rang $n$ (correspondant au réseau $D_n$) de paramètre $h$ correspond un système de rang $(n-1)$ de paramètre $T(h)$, où $T$ est la transformation correspondant au groupe de renormalisation. Cette transformation est explicite:
\begin{equation}
T: h\mapsto \log \left(\frac{\exp(sh)+b-1}{b}\right).
\end{equation}
Cela signifie que l'on peut renormaliser le système en remplaçant chaque diamant élémentaire par une arête, et $h$ par $T(h)$ sans changer la fonction de partition.
La localisation ou délocalisation du système peut être déterminée en étudiant vers quel point fixe (de la transformation $T$) $T^n(h)$ converge. Si $T^n(h)$ tend vers l'infini lorsque $n$ tend vers l'infini, c'est que la force d'accrochage grandit en renormalisant et qu'on est donc dans la phase localisée. Sinon si $T^n(h)$ tend vers le point fixe stable (fini) de la transformation $T$, c'est que l'on se trouve  dans la phase localisée. La séparation des deux phases correspond au point fixe instable de $T$. 

La transformation étant très simple, on peut déterminer explicitement son unique point fixe instable de la transformation $T$. Dans le cas où $b<s$, on a $h_c=0$. De plus, la définition de l'énergie libre implique
\begin{equation}
 \tf(h)=\tf(T(h))/s.
\end{equation}
Au voisinage de zéro cela donne
\begin{equation}
 \tf(h):=\tf((s/b)h+o(h))/s
\end{equation}
Donc s'il existe un exposant critique pour l'énergie libre (i.e.\ $\alpha$ tel que $\tf(h)\approx h^{1/\alpha}$), alors il doit vérifier
\begin{equation}
 1=\frac{(s/b)^{1/\alpha}}{s}.
\end{equation}
En poussant plus loin ce raisonnement, on peut déterminer le comportement critique du système homogène:

\begin{proposition}
Pour $b\in(1,s)$ on a $h_c(0)=0$ pour le modèle homogène et il existe une constante $c$ (dépendant de $b$ et $s$) telle que pour tout $h\in[0,1]$ 
\begin{equation}\label{anneaboun}
 \frac{1}{c}h^{\frac{1}{\ga}}\le \tf(0,h)\le c h^{\frac{1}{\alpha}},
\end{equation}
où
\begin{equation}
 \alpha:=\frac{\log s -\log b}{\log s}.
\end{equation}
Dans le cas $s=2$, $b>2$, $h_c(0)=\log b-1$ (uniquement pour le désordre par arête) et \eqref{anneaboun} est vérifié avec
\begin{equation}
 \alpha:=\frac{\log 2(b-1)-\log b}{\log 2}.
\end{equation}
\end{proposition}

\begin{rema}\rm 
 Contrairement au cas homogène, on ne peut pas prouver $\tf(0,h)\sim cste. h^{\frac{1}{\alpha}}$. Il est conjecturé que cette équivalence n'est pas vraie, et que la constante doit être remplacée par une fonction $\log$ périodique ($f$ telle que $f(\exp (x))$ est périodique), avec des oscillations de très faible amplitude, celles-ci étant dues à la péridiodicité des réseaux (voir \cite{cf:DIL}).
\end{rema}

Le modèle avec désordre est plus difficile à étudier, car la transformation du groupe de renormalisation est aléatoire.
Cependant l'existence d'un exposant critique pour le modèle homogène permet d'utiliser le critère de Harris pour émettre une hypothèse quant à la pertinence du désordre.
Nous présentons les résultats obtenus pour les différents modèles hiérarchiques (désordre par site et désordre par arête), qui assurent la validité du critère de Harris.
D'abord dans le cas où le désordre est non-pertinent à haute température:

\medskip

\begin{theoreme}
 Dans tous les cas où $\alpha<1/2$ (i.e. $b\in(\sqrt s, s)$, et $s=2$, $b\in (2, 2+\sqrt{2})$) et le désordre est non pertinent à haute température dans le sens où il existe $\gb_0$ tel que pour tout $\gb<\gb_0$, $h_c(\gb)=h_c(0)+\gl(\gb)$ et 
\begin{equation}
 \tf(\gb,h)\sim_{h\to h_c(\gb)_+} \tf(0,h+\gl(\gb)).
\end{equation}
\end{theoreme}
\medskip
\noindent Et dans le cas où le désordre est pertinent pour tout valeur de $\gb$. 
\medskip
\begin{theoreme}\label{IIIIII}
 Dans tous les cas où $\alpha>1/2$, (i.e. $b\in (1,\sqrt{s})$, et $s=2$, $b\in (2+\sqrt{2},\infty)$ le désordre est pertinent dans le sens où
$h_c(\gb)>h_c(0)-\gl(\gb)$.
De plus, on peut estimer la différence entre les deux quantités. Il existe une constante $c$ (dépendante de $b$ et $s$) telle que
\begin{equation}
 -c\gb^\frac{2\alpha}{2\alpha-1}<h_c(\gb)-(h_c(0)-\gl(\gb))<-\frac{1}{c} \gb^{\frac{2\alpha}{2\alpha-1}}.
\end{equation}
\end{theoreme}
\medskip

Ces résultats sont analogues à ceux du modèle non-hiérarchique. La borne supérieure du Theorème \ref{IIIIII} a été démontrée avant le résultat correspondant pour le cadre non-hiérarchique, et c'est l'étude du modèle hiérarchique qui a donné l'intuition de la preuve.
\medskip

Ces deux théorèmes ainsi que le comportement de l'énergie libre du système homogène ont été prouvés (pour le modèle par arête dans  le cas $s=2$, mais les autres cas sont similaires) en collaboration avec G.\ Giacomin et F.\ Toninelli. Ces résultats ont donné lieu à un article à paraître dans la revue {\sl Probability Theory Related Fields} \cite{cf:GLT}, qui constitue le chapitre \ref{HPMQM} de cette thèse.
\medskip

Nous présentons maintenant les résultats obtenus pour le cas $\alpha=1/2$ pour lequel le critère de Harris ne donne pas de prédiction.
Pour le modèle avec désordre par site:

\begin{theoreme}
 Dans le modèle avec désordre par site, lorsque $b=\sqrt{s}$, on a $h_c(\gb)> h_c(0)-\gl(\gb)$ pour toute valeur de $\gb$. De plus on peut trouver des constantes $c$ et $\gb_0$ telles que
\begin{equation}
 -\exp\left(-\frac{1}{c\gb}\right)\le h_c(\gb)-(h_c(0)-\gl(\gb))\le -\exp\left(-\frac{c}{\gb^2}\right), \quad \forall \gb<\gb_0.
\end{equation}
\end{theoreme}

Ce résultat a été démontré en raffinant de la méthode utilisée dans \cite{cf:GLT}. La preuve utilise explicitement le caractère inhomogène de la fonction de Green pour ce modèle et n'est pas adaptable au cas du désordre par arête. Il a donné lieu à un article à paraître dans la revue {\sl Probability Theory Related Fields} \cite{cf:Hubert}, qui constitue le chapitre \ref{HPMDMR} de cette thèse.

\medskip

Enfin, dans l'article \cite{cf:GLT_marg}, nous avons  confirmé le caractère pertinent du désordre pour le cas marginal du modèle avec désordre par arête introduit dans \cite{cf:DHV}, dans le cas particulier du désordre gaussien.

\begin{theoreme}
Dans le modèle avec désordre par arête, pour les cas $s=2$, $b\in \{\sqrt{2},2+\sqrt{2}\}$, lorsque le désordre est gaussien, on a $h_c(\gb)< h_c(0)+\gl(\gb)$ pour toute valeur de $\gb$. De plus, on peut trouver une constante $c$ telle que pour tout $\gb\in[0,1]$

\begin{equation}
  -\exp\left(-\frac{1}{c\gb^2}\right)\le h_c(\gb)-h_c(0)-\gl(\gb)\le -\exp\left(-\frac{c}{\gb^4}\right).
\end{equation}
\end{theoreme}

Le resultat présenté se limite au cas gaussien avec $s=2$, mais des techniques ont depuis été développées pour améliorer les bornes et traiter le cas général (voir \cite{cf:kbodies}).
Pour le cas marginal, les bornes trouvées pour $h_c(\gb)$ dans les deux modèles diffèrent. Cela appuie l'hypothèse selon laquelle le comportement du système dans ce cas dépend du modèle considéré, alors qu'il est universel quand $\alpha\ne 1/2$ (cf.\ critère de Harris).
Enfin un résultat de {\sl lissage} de la courbe de l'énergie libre, similaire à la proposition \ref{th:smooothing} a été démontré en collaboration avec F. Toninelli \cite{cf:LT}. Les résultats présent dans cette thèse, ont été obtenus en combinant des méthodes de changement de mesure et d'estimation de moments non-entiers de la fonction de partition.

\section{Polymères dirigés en milieu aléatoire.}

\subsection{Généralités}
Ce modèle de polymère dirigé en environnement aléatoire a été introduit (en dimension $1+1$) par Henley et Huse pour rendre compte des effets de rugosité dans le modèle d'Ising $2$-dimensionnel perturbé par des impuretés aléatoires. Ce modèle et ses déclinaisons (modèles dans un cadre continu ou semi-continu, percolation dirigée au dernier passage) ont donné lieu à de nombreuses études mathématiques ces vingt dernières années, en particulier pour les phénomènes de localisation. Parmi les travaux sur le sujet on peut citer \cite{cf:B, cf:IS, cf:P, cf:M, cf:CH_ptrf, cf:CSY, cf:CY, cf:CY_cmp, cf:CV, cf:V, cf:BTV} et \cite{cf:CSY_rev} pour un article de survol. Nous présentons dans cette introduction le modèle dans sa définition la plus simple, et les principaux résultats et conjectures.
\medskip

Rappelons la définition du modèle en dimension $d$. Soit $(S_{n})_{n\in\N}$ la marche aléatoire simple dans $\Z^d$ de loi $\bP$, i.e.\ $S_0=0$ et la suite $(S_{n}-S_{n-1})_{n\in\N}$ est une suite i.i.d.\ de variables aléatoires à valeur dans $\Z^d$ et
\begin{equation}
\bP(S_1=x)=\begin{cases} \frac{1}{2d} \quad &\text{si } |x|=1.\\
          0 \quad &\text{sinon}.  
           \end{cases}
\end{equation}
\'Etant donnée la réalisation $(\go_{n,x})_{n\in \N, z\in \Z^d}$ d'un champs de variables aléatoires i.i.d.\ (de loi $\bbP$, on notera $\bbE$ l'espérance) centrées de variance unitaire, vérifiants \eqref{finitemomo}, on définit la mesure de polymère $\mu_{N,\go,\gb}$ comme la loi d'une marche aléatoire dont la dérivée de Radon-Nicodym par rapport à $\bP$ est

\begin{equation}
\frac{\dd \mu_{N,\go,\gb}}{\dd \bP}(S):=\frac{1}{Z_{N,\go,\gb}}\exp\left(\gb \sum_{n=1}^N \go_{n,S_n}\right),
\end{equation}
où
\begin{equation}
 Z_{N,\go,\gb}:=\bE\left[\exp\left(\gb \sum_{n=1}^N \go_{n,S_n}\right)\right].
\end{equation}

Le but est, comme pour les modèles d'accrochage, de déterminer certaines propriétés asymptotiques de la suite de mesure $\mu_{N,\go,\gb}$ lorsque $N$ tend vers l'infini. On peut intuiter que l'influence du désordre va croître avec $\gb$ (ou, de manière équivalente, décroître avec la température). En effet,
plus $\gb$ est grand, plus l'écart d'énergie entre les trajectoires va rendre la mesure $\mu_{N,\go,\gb}$ inhomogène.
Les principales questions à se poser sont:
\begin{itemize}
 \item \`A haute température, existe-t-il un régime où l'influence du désordre disparait asymptotiquement et où, après réechelonnage, la marche sous $\mu_{N,\go,\gb}$ converge vers un mouvement brownien?
 \item Quelles sont les propriétés typiques de la mesure $\mu_{N,\go,\gb}$ lorsque le désordre influe sur les trajectoires (distance typique à l'origine après $N$ pas, existence de couloirs préférenciels pour les trajectoires)?
 \item \`A quelle température s'effectue la transition entre les deux régimes, et quelle sont les propriétés du régime critique, de la transition de phase?
\end{itemize}

Comme dans le cas du modèle d'accrochage, la résolution de ces questions repose sur l'étude du facteur renormalisateur $Z_{N,\go,\gb}$ (la {\sl fonction de partition}). Plus précisément, on va s'intéresser à la quantité
\begin{equation}
 W_N:=\frac{Z_{N,\go,\gb}}{\bbE \left[ Z_{N,\go,\gb}\right]}=\exp(-N\gl(\gb))Z_{N,\go,\gb}.
\end{equation}
C'est E.\ Bolthausen qui a observé le premier que $W_N$, munie de la filtration $(\mathcal F_N)_{N\in\N}=\left( \sigma\{\go_{n,z},n\le N\}\right)$, est une martingale positive, et donc converge vers une limite $W_{\infty}$. De plus un argument standard de loi du zéro-un permet de montrer que
\begin{equation}
\bbP\left(W_\infty=0\right)\in\{0,1\}.
\end{equation}
Chacun des deux cas, $W_\infty>0$ p.s.\ et $W_\infty$ dégénérée, correspond à un régime de désordre différent que l'on nommera respectivement {\sl fort désordre} et {\sl faible désordre}. Pour expliquer cette terminologie nous citons des résultats illustrant l'influence du désordre sur les propriétés trajectorielles.

\medskip

Une série d'articles (\cite{cf:IS, cf:B, cf:AZZ, cf:SZ, cf:CY}) a permis d'arriver à la conclusion suivante: si la limite $W_{\infty}$ est non-dégénérée, alors le désordre n'a pas d'influence sur le comportement asymptotique des trajectoires. Le résultat ci-dessous ayant été prouvé par F.\ Comets et N.\ Yoshida \cite{cf:CY}.
\begin{theoreme}\label{invprinc}
Dans le régime de désordre faible, la trajectoire du polymère converge vers un mouvement brownien, dans le sens où la loi de 
\begin{equation}
S^{(N)}:=\left( S_{\lceil Nt \rceil}/\sqrt{N}\right)_{t\in[0,1]},
 \end{equation}
converge vers celle d'un mouvement brownien standard $(B_t)_{t\in[0,1]}$.
\end{theoreme}

De plus, il y a consensus pour dire dans le régime de désordre fort, le polymère est non-diffusif.
L'une des manières de caractériser le fort désordre est l'intersection de deux répliques (trajectoires indépendantes tirées selon la mesure 
$ \mu_{n-1}^{\otimes 2}(S_n^{(1)}=S_n^{(2)})$).
Carmona et Hu dans le cas Gaussien \cite{cf:CH_ptrf}, puis Comets, Shiga et Yoshida \cite{cf:CSY} dans le cas d'un désordre arbitraire ont prouvé le résultat suivant:

\begin{theoreme}
Dans le régime de  fort désordre, il existe une constante $c$ (qui dépend de $\gb$ de $d$ et de la loi de l'environnement) , telle que 
\begin{equation}
-c^{-1} \log W_N \le \sum_{n=1}^N \mu_{n-1}^{\otimes 2}\left(S_n^{(1)}=S_n^{(2)}\right)\le -c \log W_N.
\end{equation}
\end{theoreme}

Au vu du résultat précédent, il est raisonnable de regarder plus spécifiquement le cas où $W_N$ tend exponentiellement vers $0$. Dans ce cas, le résultat précédent suggère que que deux chemins choisi selon la mesure $\mu_N$ ont une fraction de recouvrement asymptotique positive (la proportion des pas que deux répliques $S^{(1)}$ et $S^{(2)}$ passent ensemble).
Pour cela on s'intéresse, à l'énergie libre du système

\begin{proposition}
La limite
\begin{equation}
 p(\gb):=\lim_{N\to \infty} \frac{1}{N}\log W_N,
\end{equation}
existe et est constante pour presque toute réalisation de l'environnement $\go$.
Elle est aussi égale à
\begin{equation}
 \lim_{N\to \infty} \frac{1}{N}\bbE \left[\log W_N\right]=: \lim_{N\to \infty}  p_N(\gb).
\end{equation}
C'est une fonction négative et décroissante de $\gb$.
\end{proposition}

Dans le cas où l'énergie libre est strictement négative, on dit que l'on est dans le régime de {\sl très fort désordre}. 
Une seconde justification (dans le cas gaussien) pour cette terminologie est la relation suivant liant l'énergie libre à la fraction de recouvrement de deux marches de loi $\mu_N$. Elle provient de P.\ Carmona et Y.\ Hu \cite{cf:CH_ptrf} et utilise l'intégration par partie Gaussienne, méthode utilisée dans plusieurs chapitres de cette thèse.
\medskip
\begin{proposition}
Aux points où $p(\gb)$ est dérivable on a
 \begin{equation}
  \frac{\dd }{\dd \gb}p(\gb)=\frac{1}{\gb}\lim_{N\to \infty} \frac{1}{N}\bbE\left[\sum_{n=1}^N\mu_N^{\otimes 2}( S_{n}^{(1)}=S_n^{(2)})\right].
 \end{equation}

\end{proposition}
\begin{proof}
On laisse au lecteur le soin de vérifier la propriété suivante: si $f$ est une fonction dérivable telle que $\lim_{|x|\to\infty} f(x)\exp(-x^2/2)=0$ et $\go$ une variable aléatoire Gaussienne, on a:
\begin{equation}
 \bbE\left[\go f(\go)\right]=\bbE\left[f'(\go)\right].
\end{equation}
On étudie la dérivée de la fonction $p_N(\sqrt{t})$

\begin{equation}\begin{split}
 \frac{\dd }{\dd t}p(\sqrt{t})&=\frac{1}{2N\sqrt{t}}\sum_{n=1}^N\sum_{z\in\Z^d}\bbE\left[\go_{n,z}\frac{\bE\left[\ind_{\{S_n=z\}}\exp\left(\sqrt{t} H_N(\go)\right)\right]}{\bE\left[\exp\left(\sqrt{t} H_N(\go)\right)\right]}\right]-1\\
&=-\frac{1}{2N}\sum_{n=1}^N\sum_{z\in\Z^d}\bbE\left[\frac{\bE\left[\ind_{\{S_n=z\}}\exp\left(\sqrt{t} H_N(\go)\right)\right]^2}{\bE\left[\exp\left(\sqrt{t} H_N(\go)\right)\right]^2}\right]\\
&=\frac{1}{2N}\bbE\left[\mu_N^{\otimes 2}\left[\sum_{n=1}^N \ind_{\{S_n^{(1)}=S_n^{(2)}\}}\right]\right].
\end{split}
\end{equation}
La formule d'intégration par partie gaussienne est utilisée pour chaque terme de la somme, pour le passage de la première à la seconde ligne.
La fonction $p_N(\gb)+\gb^2$ est convexe pour tout $N$, donc $p_N'(\gb)$ converge vers $p'(\gb)$ lorsque la dérivée existe.
\end{proof}

Intuitivement, l'influence du désordre devrait croître avec la température. Des résultats allant dans ce sens ont été démontrés par F.\ Comets et N.\ Yoshida dans \cite{cf:CY} et mettent en avant un phénomène de transition de phase. Il avait déja été démontré qu'en dimension $d\le 2$, il y a toujours fort désordre \cite{cf:CH_ptrf, cf:CSY}, et qu'il existait un régime de faible désordre en dimension $d\ge 3$ \cite{cf:IS, cf:B}. Nous rassemblons ces informations dans la proposition suivante:
\begin{theoreme}
En toute dimension, il existe deux constantes (dépendant de $d$) $\gb_c$ et $\bar \gb_c$ ($\gb_c\le \bar \gb_c$) dans $[0,\infty]$ telles que:
\begin{itemize}
 \item On est dans le régime de fort désordre pour tout $\gb> \gb_c$ et de faible désordre pour tout $\gb<\gb_c$.
 \item On est dans le régime de très fort désordre si et seulement si $\gb>\bar \gb_c$.
\end{itemize}
De plus
\begin{itemize}
 \item $\gb_c=0$ si et seulement si $d=1,2$. 
 \item $\bar \gb_c<\infty$ si le désordre est non borné.
\end{itemize}
\end{theoreme}
Ces informations sur le comportement de l'énergie libre sont illustrées par la figure \ref{phases}.
\medskip

D'un point de vue physique, il semblerait naturel que les deux définitions du désordre fort coïncident en dehors du point critique et que $\gb_c=\bar\gb_c$. Cela a été confirmé mathématiquement pour $d=1$ dans un travail de F.\ Comets et V.\ Vargas, qui ont montré que $\bar \gb_c=0$ dans ce cas \cite{cf:CV}.
\medskip

\begin{figure}[hlt]
\begin{center}
\leavevmode
\epsfxsize =14 cm
\psfragscanon
\psfrag{beta}{\tiny $\gb$}
\psfrag{p(beta)}{\tiny $p(\gb)$}
\psfrag{O}{\tiny O}
\psfrag{des}{\tiny \bf désordre}
\psfrag{fort}{\tiny  \bf fort}
\psfrag{tfort}{\tiny  \bf très fort}
\psfrag{faible}{\tiny \bf faible}
\psfrag{bc}{\tiny $\gb_c$}
\psfrag{bcc}{\tiny $\bar \gb_c$}
\epsfbox{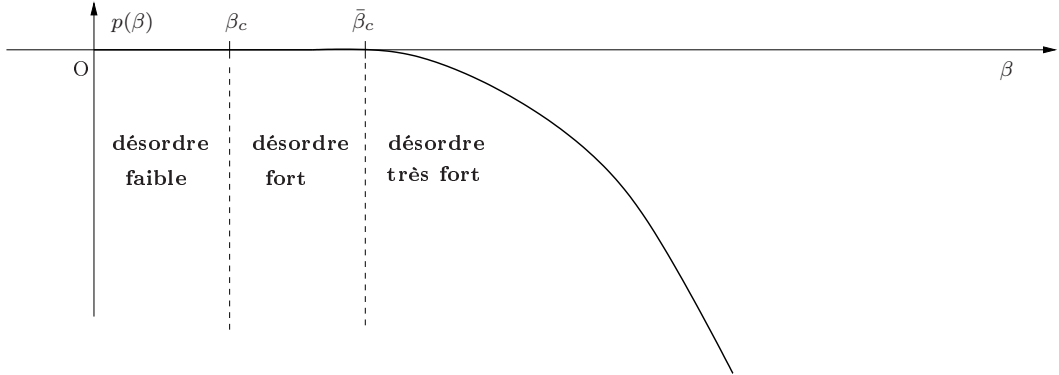}
\end{center}
\caption{Courbe représentative de $p(\gb)$ en fonction de $\gb$ et représentations des différentes phases du désordre. Il est conjecturé que $\bar \gb_c=\gb_c$. On ne sait pas dans quel phase on se trouve pour $\gb=\bar \gb_c$, et $\gb=\gb_c$. }\label{phases}

\end{figure}

Une autre thématique récurrente pour les polymères dirigés en milieu aléatoire (surtout en dimension $1$) est l'occurrence d'un phénomène de {\sl surdiffusivité}.
 On a vu que sous l'hypothèse de faible désordre (Théorème \ref{invprinc}), la distance typique de l'extrémité du polymère ($S_N$) à l'origine est d'ordre $N^{1/2}$, comme pour la marche aléatoire simple.
\medskip

Il est conjecturé que, dans la phase de fort désordre, ce comportement change et qu'il existe un exposant  $\xi>1/2$ (appelé exposant de volume) tel que, sous la mesure de polymère,  $\max_{n\in [0,N]} S_N \approx N^{\xi}$. Il est aussi conjecturé que dans la phase de désordre fort, l'exposant $\xi$ est universel et ne dépend pas de $\gb$ mais seulement de la dimension et est égal à l'exposant de volume du modèle de percolation orienté au premier passage associé (voir \cite{cf:Piza}). Cet exposant $\xi$ serait relié à l'{\sl exposant de fluctuation} $\chi$, défini par
\begin{equation}
 \var_{\bbP}\left[\log Z_N\right]\approx N^{2\chi},
\end{equation}
($\var_{\bbP}$ désigne la variance sous la loi $\bbP$).
La relation liant les exposants, obtenue par des méthode de rééchelonnage, est
\begin{equation}
 2\xi-1=\chi.
\end{equation}
En dimension $1$, il existe en plus une relation d'hyper-rééchelonnage $\xi=2\chi$ qui permet de conjecturer 
\begin{equation}\label{supersuper}
 \xi=2/3. 
\end{equation}
(voir Huse et Henley \cite{cf:HH}).
En dimension supérieure, il n'y a pas vraiment de consensus parmi les physiciens pour prédire la valeur de ces exposants.
\medskip

Mathématiquement, l'existence même de $\xi$ et $\chi$ ne peut pas être prouvée en général, et ces conjectures semblent très difficiles à confirmer (ou infirmer). 
Pour un certain modèle de percolation dirigée au dernier passage (cela correspond à notre modèle de polymère à température zéro), Johansson \cite{cf:J} a démontré l'égalité \eqref{supersuper} en utilisant des résultats de grandes matrices aléatoires. \`A cause des techniques utilisées, il semble très difficile d'étendre le résultat au cas où $\gb<\infty$ et à un type de désordre plus général. C'est (à notre connaissance) le seul cas non-diffusif où l'existence des exposants $\xi$ et $\chi$ et leurs valeurs a pu être démontrée rigoureusement. Dans la suite, on parlera (abusivement) de borne sur les exposants $\xi$ et $\chi$ pour parler des bornes obtenues sur $\var_{\bbP} [\log Z_n]$ ou $\max_{n\in[0,N]} S_N$.
\medskip

L'étude des exposants $\xi$ et $\chi$ a été menée d'abord dans d'autres modèles, proches des polymères dirigés en mileu aléatoire.
C.\ Newman et M.S.T.\ Piza ont prouvé (dans \cite{cf:NP}), sous certaines hypothèses,  que pour le modèle standard (non-dirigé) de percolation au premier passage
on a $\chi\ge 1/8$ et $\xi\le 3/4$ en toute dimension, puis en collaboration avec C.\ Licea \cite{cf:LNP} que $\xi\ge 3/5$ en dimension $1$ et que $\xi\ge 1/2$ en toute dimension. 
\medskip

Pour des raisons techniques essentiellement liées à l'utilisation d'outils de calcul stochastique, les résultats obtenus pour les polymères dirigés l'ont été dans des modèles partiellement ou totalement continus, où la marche aléatoire simple dans $\Z$ est remplacé par une marche aléatoire à accroissement gaussien (à valeur dans $\R$) ou un mouvement Brownien. M.\ Petermann a prouvé dans sa thèse (\cite{cf:P}, non publiée) que, sous certaines hypothèses, $\xi\ge \frac{3}{5}$. O.\ Méjane \cite{cf:M} a prouvé ensuite pour le même modèle un résultat donnant une borne supérieure pour $\xi$ indépendente de la dimension $\xi\le \frac{3}{4}$.
Enfin, pour un modèle de polymère brownien dans un environnement brownien, S.\ Bezerra, S.\ Tindel et F.\ Viens \cite{cf:BTV} ont redémontré, sous des hypothèses plus faibles, le résultat de Petermann.

\medskip

Les différentes thématiques abordées dans cette partie de la thèse sont:
\begin{itemize}
 \item l'étude l'énergie libre à haute-température en dimension $1$ et $2$ pour le modèle sur $\Z^d$.
 \item l'étude d'un modèle hiérarchique de polymère dirigé en milieu aléatoire, où l'application des méthodes de groupe de renormalisation est plus facile. On s'intéresse plus particulièrement aux conditions sous lesquelles il y a fort désordre, et aux propriétés de localisation dans le régime de fort désordre.
 \item l'étude de polymère brownien dans un environnement brownien avec corrélations spatiales à longue distance. On étudie comment la présence de telles corrélations  peut impliquer le fort désordre à toute température dans les grandes dimensions et des propriétés de surdiffusivités.
\end{itemize}

\subsection{Fort désordre en dimension $1$ et $2$}
L'objectif de l'étude des polymères dirigés en dimension $1$ et $2$ a pour but d'améliorer les résultats existants concernant le désordre fort sous plusieurs aspects. En effet, on a remarque que:

\begin{itemize}
 \item En dimension $1$, le résultat de Comets et Vargas \cite{cf:CV}, qui prouve le fort désordre à toute température ne donne pas de renseignements précis sur la valeur de $p(\gb)$.
 \item En dimension $2$, il a été prouvé \cite{cf:CH_al,cf:CSY} qu'il y avait fort désordre à toute température, sans pouvoir aller plus loin en prouvant le très fort désordre.
\end{itemize}
Nous étudions donc ces questions en essayant d'adapter les méthodes utilisées dans le cadre des modèles d'accrochage pour donner des bornes sur l'énergie libre.

\medskip

Parlons de l'heuristique amenant les résultats:
le second moment de la fonction de partition renormalisée $W_N$ est donné par un calcul simple
\begin{equation}\begin{split}
 \bbP[W_N^2]&= \bP^{\otimes 2}\bbP\left[\exp\left(\sum_{n=1}^N\gb[\go_{n,S_n^{(1)}}+\go_{n,S_n^{(2)}}]-2 N \gl(\gb)\right)\right]\\
            &= \bP^{\otimes 2} \bbP\left[\exp\left(\sum_{n=1}^N [\gl(2\gb)-2\gl(\gb)]\ind_{\{S_n^{(1)}=S_n^{(2)}\}}\right)\right].
\end{split}
\end{equation}
L'ensemble $\left\{n\ :\ S_n^{(1)}=S_n^{(2)}\right\}$ étant un processus de renouvellement, le second moment de $W_N$ est en fait la fonction de partition d'un modèle d'accrochage homogène. Lorsque $d\ge3$, le processus de renouvellement est transitoire (il n'y a presque surement qu'un nombre fini de point) et pour cette raison $W_N$ est borné dans $\bbL_2(\bbP)$ pour $\gb$ faible. Pour $d=1,\ 2$, le second moment explose avec $N$, et l'étude du modèle d'accrochage permet de controler la vitesse d'explosion. Une utilisation fine des méthodes classiques de second moment couplée à un argument de percolation nous permet d'obtenir une borne inférieure pour l'énergie libre.

\medskip

Les méthodes de changements de mesure couplée à l'estimations de moment non-entier développées dans le cadre des modèles d'accrochages (voir \cite{cf:GLT, cf:DGLT, cf:GLT_marg, cf:kbodies} ) permettent d'établir des bornes supérieures précises. Ces résultats ont donné lieu à un article à paraître dans la revue {\sl Communication in Mathematical Physics}  qui constitue le chapitre \ref{DirPOL12} de cette thèse.
Le premier résultats complête le résultat de Comets et Vargas \cite{cf:CV}.

\begin{theoreme}
 Lorsque $d=1$, il existe une constante $c$ telle que pour tout $\gb\le 1$ on ait
\begin{equation}
-c(1+|\log(\gb)|^2)\gb^4\le p(\gb)\le -c^{-1}\gb^4. 
\end{equation}
Dans le cas d'un environnement gaussien, on peut améliorer la borne inférieure pour obtenir
\begin{equation}
 -c\gb^4\le p(\gb)\le -c^{-1}\gb^4.
\end{equation}
 
\end{theoreme}

\medskip
Notre second résultat montre qu'il y a fort désordre à toute température en dimension $2$.

\begin{theoreme}
Lorsque $d=2$, il existe des constantes $c$ et $\gb_0$ telles que pour tout $\gb\le \gb_0$,
\begin{equation}
  -\exp\left(-\frac{1}{c\gb^2}\right)\le p(\gb)\le -\exp\left(-\frac{c}{\gb^4}\right).
\end{equation}
En particulier
\begin{equation}
 \bar \gb_c=0.
\end{equation}

\end{theoreme}

Ce résultat était conjecturé par les mathématiciens depuis la démonstration du fort désordre à toute température \cite{cf:CH_ptrf, cf:CSY}.
En physique, on peut, en revanche, il a été prédit, en partie sur la base de simulations numériques, qu'il existait une transition de phase pour $\gb>0$ en dimension $2$ (voir \cite{cf:DGO}). La difficulté à prédire le bon résultat à partir de simulations peut s'expliquer par la très faible valeur (en valeur absolue) de $p(\gb)$ au voisinage de $0$.

\medskip

Ces résultats doivent beaucoup à l'intuition developpée par l'étude du modèle de polymère hiérarchique correspondant et aux nouvelles techniques developpées pour les modèles d'accrochage.

\subsection{Modèle hiérarchique}
Comme pour le modèle d'accrochage, il existe un équivalent hiérachique du modèle polymère.
L'étude de ce modèle a été menée en collaboration avec G.\ Moreno. Elle a donné lieu à un article en cours d'examination pour publication \cite{cf:LM}.
\medskip

Définissons un modèle de polymère dirigé en milieu aléatoire sur la famille de réseaux $(D_n)_{n\in \N}$ décrite au début de la section \ref{sectionhierar}, dépendant de deux entiers $b$ et $s$ supérieurs à $2$. De manière naturelle, on peut voir le réseau $D_n$ comme l'ensemble des sites le composant, considérer que $D_n\subset D_{n+1}$ et poser $D=\bigcup_{n\in \N}D_n$. On considère le modèle de polymère dirigé $D_n$ suivant: soient $\gG_n$ l'ensemble des chemins auto-évitants (pour les sites) reliant $A$ à $B$ dans $D_n$, vu comme une suite de sites ($g\in\gG_n=(g_0=A,g_1,\dots,g_n=B)$ et $(\go_x)_{x\in D\setminus\{A,B\}}$ la réalisation d'une famille de variables aléatoires i.i.d. (sous la loi $\bbP$), vérifiant les conditions usuelles (variance unitaire, espérance nulle, moments exponentiels finis \eqref{finitemomo}). On écrit la mesure de polymère $\mu_{n,\go,\gb}$ (associé à la variable aléatoire $\gga$, pour $\gb>0$) comme une modification de la mesure uniforme sur $\gG_n$ 

\begin{equation}
 \mu_{n,\go,\gb}(\gga=g):= \frac{1}{|\gG_n|W_n}\exp\left(\sum_{i=1}^{s^n-1} [\gb\go_{g_i}-\gl(\gb)]\right)
\end{equation}
où $W_n$ est la fonction de partition renormalisée, égale à

\begin{equation}
 W_n:= \frac{1}{|\gG_n|}\sum_{g\in \gG_n} \exp\left(\sum_{i=1}^{s^n-1} [\gb\go_{g_i}-\gl(\gb)]\right).
\end{equation}

L'une des clefs de l'étude de ce modèle, est de remarquer que, comme pour le modèle d'accrochage hiérarchique, la fonction de partition obéit à une relation de récurrence en loi. Plus précisément on a:

\begin{equation}
 \begin{split}
W_0&=1\\
W_{n+1}&\stackrel{(\mathcal L)}{=}\frac{\sum_{i=1}^b W_n^{(i,1) }\dots W_n^{(i,s)}A^{(i,1)}\dots A_{(i,s-1)}}{b}.  
 \end{split}
\end{equation}
où la seconde inégalité a lieu en loi, les variables $W_n^{(i,j)}$ sont des copies inépendantes de $W_n$ et les variables $A^{(i,j)}$ sont des variables i.i.d\ indépendantes du reste, de même loi que $\exp(\gb \go_x-\gl(\gb))$.

\medskip
Comme pour le modèle non-hiérarchique, on définit l'énergie libre du système comme suit, son existence découle de certains résultats de concentration pour les martingales (voir \cite{cf:LV}).

\begin{proposition}
La limite 
\begin{equation}
 p(\gb):=\lim_{n\to\infty} \frac{1}{s^n}\log W_n
\end{equation}
 existe presque surement et est constante. C'est une fonction décroissante de $ \gb$.
\end{proposition}

Notre but a été de trouver les conditions nécessaires sur $b$ et $s$ pour qu'il y ait désordre à toute température.
Pour $\Z^d$, l'existence d'une phase de faible désordre est équivalente à $d\ge 3$, i.e.\ à la transitivité de la marche aléatoire. On peut  s'attendre à ce qu'il y ait une analogie. On étudie l'espérance du recouvrement de deux chemins sous la loi uniforme sur $\gG_n$ (qu'on appelle $\bP_n$).
Soit
\begin{equation}
 F_n:=\bP_n^{\otimes 2} \left[|\gamma^{(1)}\cap\gamma^{(2)}|\right],
\end{equation}
où $|\ . \ |$ désigne le nombre de sites et où les chemins sont considérés comme des ensembles de sites.
On peut constater que lorsque $n$ tend vers l'infini
\begin{equation}
F_n \begin{cases} =O(1) \quad &\text{ si }\ b>s,\\
 \sim cste.n \quad &\text{ si } b=s,\\
 \sim cste.(s/b)^n \quad &\text{ si } b<s. \end{cases}
\end{equation}
Le premier cas est analogue à $d>2$, le second (où $F_n$ croît comme le logarithme de la longueur du système, qui vaut $s^n$) est similaire au cas $d=2$, et le troisième (où $F_n$ croit comme une puissance de la longueur du système) correspond au cas $d<2$, et à plus forte raison à $d=1$ dans le cas $b=\sqrt{s}$ pour lequel $F_n$ croit comme $\sqrt{s^n}$.
\medskip

Les résultats obtenus lors de notre étude montrent que cette analogie n'est pas vaine.
Tout d'abord sur l'existence d'une transition de phase
\medskip
\begin{proposition}
L'énergie libre possède les propriétés suivantes:
\begin{itemize}
 \item [(i)] Lorsque $b\le s$, $p(\gb)<0$ pour tout $\gb>0$.
 \item [(ii)] Lorsque $b>s$ il existe $\gb_c\in(0,\infty]$ tel que $p(\gb)>0$ est équivalent à $\gb>\gb_c$, $\gb_c<\infty$ si le désordre n'est pas borné.
 \item [(iii)] On a les bornes suivantes pour $\gb_c$ dans le cas gaussien:
\begin{equation}
    \sqrt{\frac{2(b-s)\log b}{(b-1)(s-1)}}    <\gb_c\le \sqrt{\frac{s}{s-1}\log \frac{b}{s}-\log \frac{b-1}{s-1}}. 
\end{equation}
 \end{itemize}
\end{proposition}
\medskip

\begin{rema}\rm
 La valeur exacte de $\gb_c$ ne semble pas correspondre avec l'une de ces bornes.
Pour plusieurs autres modèles de polymère dirigé (voir \cite{cf:BPP} pour les polymères sur l'arbre, \cite{cf:Birk, cf:BS, cf:CC} pour les polymères dirigés dans $\Z^d$) il a été montré que la borne supérieure, obtenue par des méthodes de second moment, n'était pas optimale.
\end{rema}
\medskip

De plus, nous avons obtenu quand il y a fort désordre à toute température, une approximation précise de l'énergie libre à haute température.
\medskip
\begin{theoreme}
Lorsque $b<s$, il existe une constante $c$ telle que pour tout $\gb\le 1$ on ait
\begin{equation}
 -c\gb^{\frac{2}{\alpha}}\le p(\gb)\le -c^{-1}\gb^{\frac{2}{\alpha}}
\end{equation}
 où
\begin{equation}
 \alpha=\frac{\log s -\log b}{\log s}.
\end{equation}
\end{theoreme}
\medskip
\begin{theoreme}
Lorsque $b=s$, il existe des constantes $c$ et $\gb_0$ telles que pour tout $\gb\le \gb_0$ on ait
\begin{equation}
     -\exp\left(-\frac{1}{c\gb}\right) \le p(\gb)\le -\exp\left(-\frac{c}{\gb^2}\right).
\end{equation}
\end{theoreme}
\medskip

La possibilité de faire varier les paramètres permet de mettre en valeur une relation entre l'exposant de croissance de $F_n$, et l'exposant $\alpha$ qui gouverne l'énergie libre.
\medskip

Enfin, la géométrie particulière des réseaux en diamant a permis d'obtenir un résultat fort de localisation.
Ce résultat découle de l'obtention d'une borne inférieure sur l'exposant  de fluctuation. 
Des raisons techniques contraignent à se restreindre au cas gaussien.

\begin{theoreme}
On se place dans le cas d'un environnement gaussien.

Lorsque $b<s$, il existe une constante $c$ (dépendant de $s,\ b$ et $\gb$) telle que pour tout $n$
\begin{equation}
 \var_{\bbP} Z_n\ge c(s/b)^n.
\end{equation}
Lorsque $s=b$, il existe une constante $c$ (dépendant de $s,\ b$ et $\gb$) telle que pour tout $n$
\begin{equation}
  \var_{\bbP} Z_n\ge (s/b)^n\ge c\sqrt{n}.
\end{equation}
\end{theoreme}

Pour un chemin $g\in \gG_m$, $m\ge n$, on définit $g|_n\in \gG_n$ sa restriction à $D_m$.
Le résultat suivant montre que, si l'on regarde à une échelle fixée, asymptotiquement, la mesure se concentre sur un seul chemin.

\begin{proposition}
 Lorsque $b\le s$ et pour tout $n\in \N$ fixé on a
\begin{equation}
 \lim_{m\to \infty} \sup_{g\in \gG_n} \mu_m(\gga|_n=g)=1,
\end{equation}
où la convergence a lieu en probabilité.
\end{proposition}

\subsection{Surdiffusivité et fort désordre, un cas de polymère continu}

Un dernier modèle évoqué dans cette thèse est celui des polymères dirigés en environnement brownien.
Son étude constitue le chapitre \ref{polBB} de cette thèse. Ce modèle à été introduit par C.\ Rovira et S.\ Tindel \cite{cf:RT} et ses propriétés de superdiffusivité ont  été étudiées ensuite dans un article de S.\ Bezzera, S.\ Tindel et F.\ Viens \cite{cf:BTV}. On peut citer un autre modèle de polymères brownien (dans un environnement Poissonien), introduit par F.\ Comets et N.\ Yoshida \cite{cf:CY_cmp}.
Se placer dans un cadre continu donne certains avantages:
\begin{itemize}
 \item cela permet d'utiliser les outils du calcul stochastique qui se révèlent très utiles pour étudier la superdiffusivité;
 \item c'est un cadre naturel pour introduire des corrélations dans l'environnement et étudier l'effet de la mémoire spatiale de l'environnement.
\end{itemize}

Contrairement aux autres modèles étudiés dans cette thèse, on n'étudie pas une modification de la mesure de la marche aléatoire simple, mais une modification de la mesure de Wiener $\bP$ qui décrit la loi d'une trajectoire brownienne. L'environnement est déterminé par la réalisation d'un champ brownien centré $(\go(t,x))_{t\in \R_+, x\in \R^d}$ (de loi $\bbP$, on note $\bbE$ l'espérance associée)
sur $\R_+\times \R^d$ déterminée par sa fonction de covariance
\begin{equation}
 \bbE\left[\go(t,x)\go(t',x')\right]:=(t\wedge t') Q(x-x'),
\end{equation}
où $Q$ est une fonction positive de $x\in \R^d$ qui tend vers $0$ à l'infini et telle que $Q(0)=1$. On suppose de plus qu'il existe un réel $\theta>0$ tel que
\begin{equation}
 Q(x)\asymp_{x\to \infty} \|x\|^{-\theta}.
\end{equation}
L'énergie d'un chemin ($B=(B_s)_{s\in[0,y]}$) est définie par l'intégrale des accroissements du champ $\go$ le long du chemin
\begin{equation}
 H_{t,\go}(B):=\int_{0}^t \go(\dd s, B_s),
\end{equation}
intégrale à laquelle on peut donner une signification rigoureuse (voir \cite{cf:RT, cf:BTV}).
\medskip

On définit la mesure de polymère $ \mu_{t,\go,\gb}$ pour le système de taille $t$ et l'énergie libre $p(\gb)$ de la même manière que pour le modèle discret:
\begin{equation}\begin{split}
 \frac{\dd \mu_{t,\go,\gb}}{\dd \bP}(B)&:= \frac{1}{Z_t}\exp\left(\gb H_{t,\go}(B)\right),\\
Z_t&:=\bE\left[\exp\left(\gb H_{t,\go}(B)\right)\right],
\end{split}
\end{equation}
et
\begin{equation}
 p(\gb):= \lim_{t\to\infty}\frac{1}{t}\log Z_t-\gb^2/2=\lim_{t\to\infty}\frac{1}{t}\log W_t.
\end{equation}
où $W_t:=Z_t/ \bbE\left[Z_t\right]$.
L'existence de l'énergie libre est prouvé dans \cite{cf:RT}.
On peut, de même, définir les notions de désordre, faible, fort, très fort que dans le cas discret.
Il est raisonnable de penser que: lorsque la corrélation spatiale de l'environnement (quantifiée par la fonction $Q$) décroît suffisamment vite à l'infini, ce modèle doit avoir un comportement identique à celui du modèle discret, i.e. $p(\gb)<0$ pour tout $\gb$ en dimension  $1$ et $2$, et $p(\gb)=0$ en dimension $3$ et supérieure ;  lorsque la corrélation est plus forte, le modèle change de comportement.
\medskip

Nous avons décidé d'étudier ce phénomène sous deux aspects:
\begin{itemize}
 \item l'évaluation de $p(\gb)$ à haute température;
 \item le phénomènes de surdiffusivité.
\end{itemize}

Pour l'étude de l'énergie libre, il a fallu adapter les méthodes développées dans \cite{cf:Lac} à ce nouveau problème, et étudier un modèle d'accrochage homogène dans un potentiel polynomial, qui apparaît naturellement lors de l'utilisation de méthodes de second moment. Ce modèle d'accrochage brownien possède en lui même un certain intérêt (voir \cite{cf:CKMV}) mais nous n'en parlerons pas dans cette introduction.
Nous avons obtenu des résultats qui permettent de déterminer s'il y a très fort désordre à toute température dans quasiment tout les cas, couplés à des résultats précis sur le comportement de l'énergie libre:

\medskip

\begin{theoreme}
Pour $d=1$ et $\theta<1$ ou $d\ge 2$ et $\theta<2$, il y a fort désordre à toute température et il
existe une constante $c$ (qui dépend de $d$ et $Q$) telle que
\begin{equation}
 -c\gb^{\frac{4}{2-\theta}}\le p(\gb)\le c^{-1}\gb^{\frac{4}{2-\theta}}.
\end{equation}
Lorsque $d=1$, $\theta>1$, il y a fort désordre à toute température et il
existe une constante $c$ (qui dépends de  $Q$) telle que
\begin{equation}
-c \gb^4 \le p(\gb)\le c^{-1} \gb^4.
\end{equation}
De plus, lorsque $d\ge 3$, $\theta>2$, il y a faible désordre pour $\gb$ suffisamment faible.
\end{theoreme}

\medskip

Concernant la surdiffusivité, nous avons significativement amélioré le résultat de Bezerra, Tidel et Viens \cite{cf:BTV}, avec une preuve beaucoup plus courte et intuitive, qui peut se généraliser dans des cas où le désordre n'est pas gaussien, (cas du désordre poissonnien abordé dans \cite{cf:CY_cmp}, par exemple). En particulier, nous avons pu montrer que sous certaines conditions il y a surdiffusivité en toute dimension avec la borne suivante:

\medskip

\begin{theoreme}
 Pour $d=1$, $\theta<1$, $d\ge 2$, $\theta<2$, on a
\begin{equation}
 \lim_{\gep\to 0}\liminf_{N\to \infty}\bbE \left[\mu_{t,\go\gb} \left(\sup_{s\in [0, t]}|B_s|\ge \gep t^{\frac{3}{4+\theta}}\right)\right]=1.
\end{equation}
\end{theoreme}
 
\medskip

Informellement donc, ce résultat implique une inégalité sur l'exposant de fluctuation, $\xi>3/(4+\theta)$.
Dans le cas $d=1$, $\theta>1$ on retrouve le même résultat que \cite{cf:BTV} ou \cite{cf:P}.

Dans cette introduction nous donnons une heurisque de preuve pour la dimension $1$ qui explique comment est obtenue la borne $\xi>3/(4+\theta)$:\\

\'Etudions sous $\mu_t$ le poids des trajectoires $(B_{t})_{t\in[1,N]}$ qui restent coincées dans une boite de largeur $t^{\alpha}$ centrée en zéro lors de la seconde partie du parcours, et comparons le au poids des trajectoires qui passent la seconde partie du parcours dans la boîte $B2$ (voir figure \ref{superdiff}, trajectoires $A$ et $B$).
Le coût entropique pour qu'une trajectoire atteigne la boîte $B2$ et passe le seconde partie du parcours dans cette boîte (trajectoire $B$ sur la figure \ref{superdiff}) est égal à $\log \bP[S \text{ reste dans } B_2 ]\sim -N^{2\alpha-1}$ (résultat classique de grande déviation).
\medskip

Considérons la variable aléatoire somme de tous les acroissements de $\go$ dans la boite $B_2$,
\begin{equation}
\gO \int_{B_2} \go(\dd s, x) \dd x.
\end{equation}
On peut vérifier que dans le cas $\theta>1$ la variable aléatoire $\gO$ à une variance $\approx t^{\alpha(2-\theta)+1}$.
Donc la moyenne empirique dans la boîte $B2$  de l'accroissement par unité de temps  est une gaussienne ($\gO/|B_2|$) dont l'écart-type est approximativement $t^{\frac{-1-\alpha\theta}{2}}$. 
\medskip

Multiplié par le temps passé dans la boîte ($t/2$), en supposant que les environnements sont à peu près indépendants dans les boites $B1$ et $B2$, cela indique que la différence énergétique typique entre les deux chemins représentés sur la figure \ref{superdiff} est d'ordre $t^{\frac{-1-\alpha\theta}{2}}$.
\medskip

Il est favorable sous la mesure $\mu_N$ de sortir de la boite $B1$ dès que le gain énergétique surpasse la perte entropique, i.e.\ dès que
\begin{equation}
t^{(1-\theta\alpha)/2}>>t^{2\alpha-1},
\end{equation}
soit dès que $\alpha<3/(4-\theta)$.
\medskip

La nouveauté par rapport aux travaux de Petermann ou Bezzerra \textit{et al.} est changer l'approche, en utilisant des techniques de changement de mesure au lieu de regarder un problème d'inversion de matrice de covariance. Cela raccourcit considérablement la preuve et la rend plus intuitive. Cela permet également de généraliser le résultat à des modèles où l'environnement n'est pas Gaussien (par exemple au modèle à environnement poissonien introduit dans 
\cite{cf:CY_cmp}).

\medskip


\begin{figure}[hlt]
\begin{center}
\leavevmode
\epsfxsize =14 cm
\psfragscanon
\psfrag{trajA}{trajectoire A}
\psfrag{trajB}{trajectoire B}
\psfrag{B1}{Boîte B1}
\psfrag{B2}{Boîte B2}
\psfrag{0}{$0$}
\psfrag{N/2}{$N/2$}
\psfrag{N}{$N$}
\epsfbox{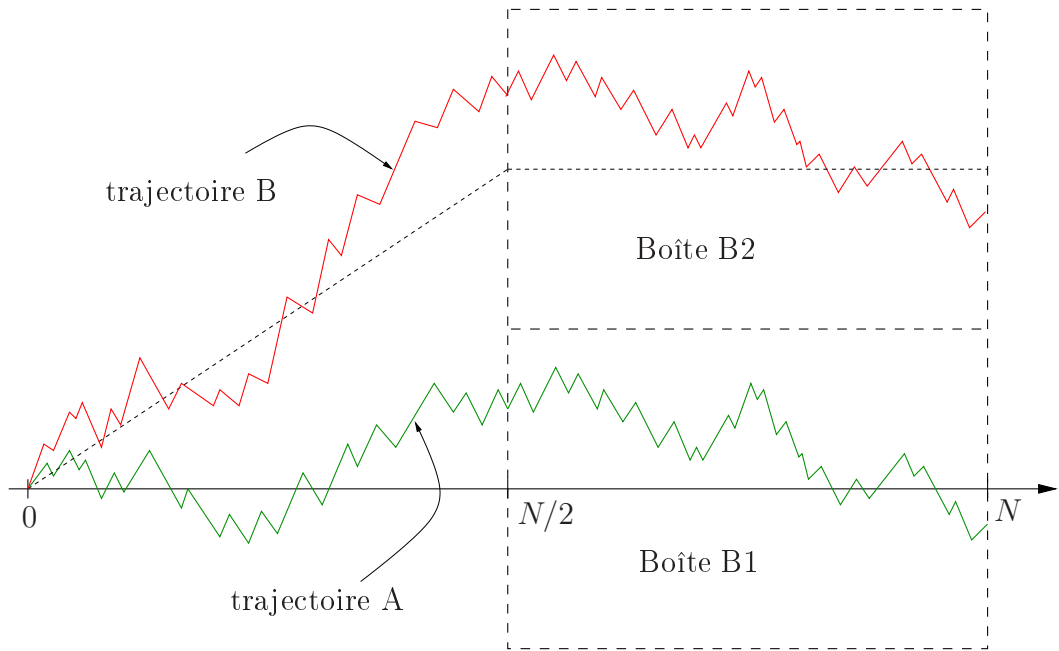}
\end{center}
\caption{Cette figure illustre l'heuristique de la preuve pour la superdiffusivité. Les trajectoires $A$ et $B$ sont deux trajectoires typiques des deux événements mentionnés.}
\label{superdiff}
\end{figure}

\section*{Elargissement}

Les méthodes utilisées pour prouver les résultats présentés de cette thèse ont pu être exploitées pour l'étude d'autres modèles de polymères.
\begin{itemize}
 \item Pour le modèle de copolymère, qui décrit le comportement d'un polymère hétérogène avec une interface de solvant (voir par exemple \cite{cf:BdH, cf:Sinai}, les méthodes développées pour les modèles présentés dans cette thèse ont été utilisées pour prouver que l'énergie libre quenched était bornée par l'énergie libre annealed à toute température pour ce modèle (voir \cite{cf:BGLT}) le résultat a été ensuite amélioré par Toninelli, pour montrer que la pente de l'énergie libre à l'origine était strictement inférieure à $1$ \cite{cf:T_cg}.
 \item Birkner et Sun \cite{cf:BS} ont proposé un modèle d'accrochage homogène alternatif, où la ligne d'accrochage n'est plus la ligne où la première coordonnée s'annule mais le graphe de la réalisation (fixée) d'une marche aléatoire dans $\Z^d$. Dans ce cas, il a été montré en adaptant les techniques developpés dans cette thèse, que les points critique quenched et annealed coincident pour $d=1, 2$, et diffèrent pour $d\ge4$. Le cas $d=3$, qui correspond au cas $\alpha=1/2$ de notre modèle d'accrochage désordonné, ne semble pas pouvoir être résolu en adaptant telle qu'elle les méthodes utilisées ici pour le cas marginal.
\end{itemize}

\selectlanguage{english}

\part{Modèles d'accrochage désordonnés}

\chapter[Hiercarchical pinning model]{Hierarchical pinning model, quadratic maps and quenched disorder}\label{HPMQM}

\section{Introduction}
\label{sec:intro}

\subsection{The model}
Consider the dynamical system 
defined by 
the initial condition $R_0^{(i)}>0$, $i\in \N :=\{1,2, \ldots\}$
and  the array of recurrence equations
\begin{eqnarray}
\label{eq:R}
  R_{n+1}^{(i)}\, =\, \frac{R_n^{(2i-1)}R_n^{(2i)}+(B-1)}B,  \ \ \  \, i \in \N,
\end{eqnarray}
for $n=0,1,\ldots$ and a given $B>2$. 
Of course if $R_0^{(i)}=r_0$ for every $i$, then the problem
 reduces to studying the quadratic recurrence equation
\begin{equation}
\label{eq:r}
r_{n+1}= \frac{r_n^2 + (B-1)}B,
\end{equation}
a particular case of a very classical problem, the {\sl logistic map},
as it is clear from the fact 
that $z_n:=1/2- r_n/(2(B-1))$ satisfies
the recursion 
\begin{equation}
\label{eq:logistic}
z_{n+1}=\frac{2(B-1)}B\,z_n(1-z_n).
\end{equation}
We are instead interested in non-constant initial data and, more
precisely, in initial data that are typical realizations of a sequence
of independent identically distributed (IID) random variables. In its
random version, the model was first considered in \cite{cf:DHV} (see
\S~\ref{sec:qmpm} and \S~\ref{sec:pinning} below for motivations in terms of
pinning/wetting models and for an informal discussion of what 
the
interesting questions are 
and what is expected to be true). We will
consider rather general distributions, but we will assume that all the
moments of $R^{(i)}_0$ are finite. As it will be clear later,
for our purposes it is actually useful to write
\begin{equation}
\label{eq:R0}
  R_0^{(i)} \, =\, \exp(\gb \go_i -\log \M (\gb)+h),
\end{equation}
with $\gb \ge 0$, $h \in \R$, 
 $\{\go_i\}_{i\in\N}$ a sequence of exponentially integrable IID
centered random variables normalized to  $\bbE\,\go_1^2=1$ and 
for every $\gb$
\begin{equation}
  \label{eq:M}
\M (\gb)\, := \, \bbE \exp(\gb \go_1)\, < \, \infty.
\end{equation}
The law of $\{\go_i\}_{i\in\N}$ is denoted by $\bbP$ and 
we will often alternatively denote the average $\bbE(\cdot)$
by brackets $\langle\cdot\rangle$.

Note that, for every $n$, $\{R_n^{(i)}\}_{i\in\N}$ are IID random
variables and therefore this dynamical system is naturally
re-interpreted as the evolution of the probability law $\cL _n$ (the
law of $R_n^{(1)}$): given $\cL_n$, the law $\cL_{n+1}$ is obtained by
constructing two IID variables distributed according to $\cL_n$ and
applying
\begin{equation}
\label{eq:basic}
R_{n+1}\, =\, \frac{R_n^{(1)}R_n^{(2)}+(B-1)}{B}.
\end{equation}

Of course, the iteration \eqref{eq:basic} is well defined for every
$B\ne0$. In particular, as detailed in Appendix
\ref{sec:Bsmallerthan2}, the case $B\in(1,2)$ can be mapped exactly
into the case $B>2$ we explicitly consider here, while for $B<1$ one
loses the direct statistical mechanics interpretation of the model
discussed in Section \ref{sec:pinning}.

\subsection{Quadratic maps and pinning models}
\label{sec:qmpm}

The model we are considering may be viewed as a hierarchical version
of a class of statistical mechanics models that goes under the name of
(disordered) {\sl pinning} or {\sl wetting } models \cite{cf:Fisher,cf:Book}, that 
are going to be described in some detail in  \S~\ref{sec:pinning}.  It has been
introduced in \cite[Section 4.2]{cf:DHV}, where the partition function
$R_n=R_n^{(1)}$ is defined for $B=2,3, \ldots$ as
\begin{equation}
\label{eq:DHV}
R_n\, =\, \bE_n^B\left[\exp\left(\sum_{i=1}^{2^n} 
\left(\gb\go_i-\log \M (\gb)+h\right) \ind_{\left\{(S_{i-1},S_i)=(d_{i-1},d_i)\right\}}
  \right)\right],
\end{equation}
with $\{S_i\}_{i=0, \ldots , 2^n}$ a simple random walk (of law
$\bP_n^B$) on a hierarchical {\sl diamond}  lattice with growth parameter $B$
and $d_0, \ldots, d_{2^n}$ are the labels for the vertices of a particular path
that has been singled out and dubbed {\sl defect line}.
The construction of diamond lattices and a graphical description of the model are
 detailed in Figure~\ref{fig:diamond}
and its caption.

\begin{figure}[!h]
\begin{center}
\leavevmode
\epsfxsize =14 cm
\psfragscanon
\psfrag{l0}{\small level $0$}
\psfrag{l1}{\small level $1$}
\psfrag{l2}{\small level $2$}
\psfrag{d0}{$d_0$}
\psfrag{d1}{$d_1$}
\psfrag{d2}{$d_2$}
\psfrag{d3}{$d_3$}
\psfrag{d4}{$d_4$}
\psfrag{u1}{$u_1$}
\psfrag{u2}{$u_2$}
\psfrag{u3}{$u_3$}
\psfrag{u4}{$u_4$}
\psfrag{traject}{\small trajectory $a$}
\psfrag{traject2}{\small trajectory $b$}
\epsfbox{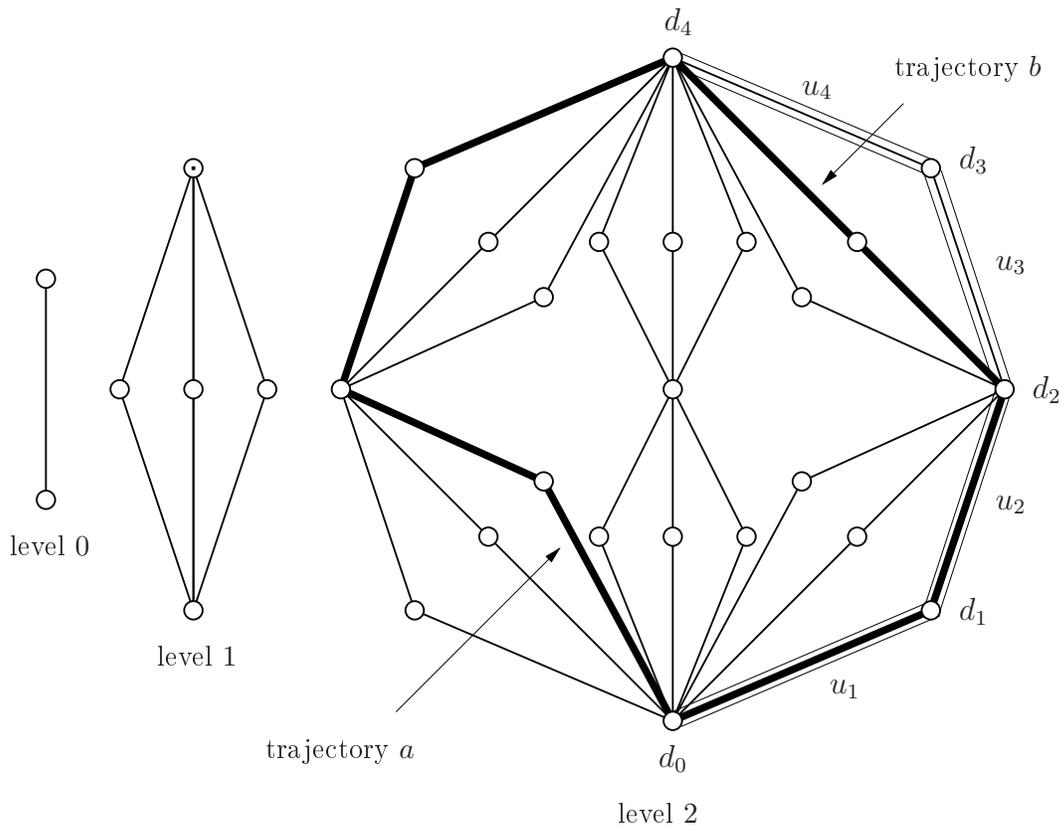}
\end{center}
\caption{Given $B=2,3, \ldots$ ($B=3$ in the drawing)
we build a diamond lattice by iterative steps
(left to right):
  at each step one replaces every bond by $B$
branches consisting of two bonds each. A trajectory of our process 
in a diamond lattice at {\sl level} $n$ is a path connecting the two {\sl poles}
$d_0$ and $d_{2^n}$: two trajectories, $a$ and $b$, are singled out by thick lines. 
Note that at level $n$, each trajectory is made of
$2^n$ bonds and there are $N_n$ trajectories, $N_0:=1$
and $N_{n+1}= B N_n^2$. A simple random walk at level $n$ is  the uniform measure
over the $N_n$ trajectories. 
A special trajectory, with vertices labeled 
$d_0, d_1, \ldots, d_4$, is  chosen (and marked by a triple line:
the right-most trajectory in the
drawing, but any other trajectory would lead to an equivalent model), we may call it defect line or wall boundary,
and rewards $u_j:= \gb \go_j - \log \M(\gb) +h$ (negative or positive) are assigned to the bonds of this trajectory.
The energy of a trajectory depends on how many and which bonds it shares
with the defect line: trajectory $a$ carries no energy, while trajectory $b$ carries energy $u_1+u_2$. 
The pinning model is then built by rewarding or penalizing 
the trajectories according to their energy  in the standard statistical mechanics fashion
and the partition function of such a model is therefore given by
$R_n$ in  \eqref{eq:DHV}. It is rather elementary, and fully detailed in   
 \cite{cf:DHV},
 to extract from \eqref{eq:DHV} the recursion
\eqref{eq:basic}. But the recursion itself is  well defined for
arbitrary real value 
$B\ne0$ and one may forget the definition of the
hierarchical lattice, as we do here. The definition of 
$\bP^B_n$ can also be easily generalized to $B >1$, see Appendix~
\ref{sec:Bsmallerthan2}.}
\label{fig:diamond}
\end{figure}

The phenomenon that one is trying to capture is the {\sl
  (de)localization at (or away from) the defect line}, that is one
would like to understand whether the rewards (that could be negative,
hence penalizations) {\sl force} the trajectories to stick close to
the defect line, or the trajectories {\sl avoid} the defect line. A
priori it is not clear that there is necessarily a sharp distinction
between these two qualitative behaviors, but it turns out that it is
the case and which of the two scenarios prevail may be read from the
asymptotic behavior of $R_n$.  The Laplace asymptotics carries already
a substantial amount of information, so we define the {\sl quenched
  free energy}
\begin{equation}
  \label{eq:F}
\tf (\beta,h) \, := \, \lim_{n \to \infty}\frac1{2^n} \log R^{(1)}_n,
\end{equation}
where the limit is in the almost sure sense: the existence of such a
limit and the fact that it is non-random may be found in Theorem~\ref{th:F}.
Note in fact that $\partial_h \tf(\gb, h)$ coincides with
the $N \to \infty$ limit of 
$\bE^B_{n, \go} [ 2^{-n } \sum_{i}  \ind_{\left\{(S_{i-1},S_i)=(d_{i-1},d_i)\right\}}]$,
where $\bP^B_{n, \go}$ is the probability measure associated to the
partition function $R_n$, when $\partial_h \tf(\gb, h)$ exists
(that is for all $h$ except at most a countable number of points,
by convexity of $\tf (\gb, \cdot)$, see below). 
Therefore $\partial_h \tf(\gb, h)$ measures the density of contacts
between the walk and the defect line and below we will
see that $\partial_h \tf(\gb, h)$ is zero up to a critical value $h_c(\gb)$,
and positive for $h>h_c(\gb)$: this is a clear signature of a {\sl localization transition}.

\subsection{A first look at the role of disorder}
Of course if $\gb=0$ the {\sl disorder} $\go$ plays no role
and the model reduces to
  the one-dimensional map \eqref{eq:r}
 (in our language $\gb>0$
corresponds to the model in which  disorder is present).  This map
has two fixed points: $1$, which is stable, and $B-1$, which is
unstable.  More precisely, if $r_0<B-1$ then $r_n$ converges
monotonically (and exponentially fast) to $1$. If $r_0>B-1$, $r_n$
increases to infinity in a super-exponential fashion, namely
$2^{-n}\log r_n$ converges to a positive number which is of course
function of $r_0$.  The question is whether, and how,  introducing
disorder in the initial condition ($\gb>0$) modifies this behavior.

There is also an alternative way to link \eqref{eq:R} and \eqref{eq:r}. 
In fact, by  taking the average we obtain
\begin{equation}
\label{eq:Rav}
\Rav{n+1}\, =\, \frac{\Rav{n}^2 + (B-1)}{B},
\end{equation}
where we have dropped the superscript in $\langle R_n^{(i)}\rangle$.
Therefore the behavior of the sequence $\{ \Rav{n}\}_n$ is (rather)
explicit, in particular such a sequence tends (monotonically) to $1$
if $\Rav{0}<B-1$, while $\Rav{n}=B-1$ for $n=1,2, \ldots$ if
$\Rav{0}=B-1$.  This is already a strong piece of information on
$R_n^{(1)}$ (the sequence $\{ \cL_n\}_n$ is tight). Less informative
is instead the fact that $\Rav{n}$ diverges if $\Rav{0}>B-1$, even if
we know precisely the speed of divergence: in fact the sequence of random
variables can still be tight!  Of course such an issue may be tackled by looking at higher moments, but while
$\med{R_n}$ satisfies a closed recursion, the same is not true for
higher moments, in the sense that the recursions they satisfy depend
on the behavior of the lower-order moments.  For instance, if we set
$\gD_n := \text{var}\left( R_n \right)$, we have
\begin{equation}
\label{eq:Delta}
\gD_{n+1}\, =\, \frac{\gD_n \left( 2 \Rav{n}^2 + \gD_n\right)}{B^2}.
\end{equation}
In principle such an approach can be pushed further, but most important for 
understanding the behavior of the system is capturing the asymptotic behavior
of $\log R_n^{(i)}$, {\sl i.e.} \eqref{eq:F}.

\subsection{Quenched and annealed free energies}
Our first result says, in particular, that the quenched free energy
\eqref{eq:F} is well defined:

\medskip

\begin{theorem}
\label{th:F}
The limit in \eqref{eq:F} exists $\bbP(\dd \go)$-almost surely and in
$\mathbb L^1(\dd \mathbb P)$, it is almost-surely constant and it is
non-negative. 
The function $(\gb,h) \mapsto \tf(\gb , h+\log \M (\gb))$ is convex and 
$\tf(\beta,\cdot)$
   is non-decreasing (and convex). These properties are inherited from 
 $\tf_N(\cdot,\cdot)$, defined by
\begin{equation}
\tf_N(\gb,h)\,=\, \frac{1}{2^N}\med{\log R_N}.
\end{equation}
Moreover $\tf_N (\gb, h)$ converges to $\tf(\gb, h)$ with exponential speed, more precisely for
all $N\ge 1$
\begin{align}
  \tf_N(\gb,h)-2^{-N}\log B \le \tf(\gb,h)
  \le \tf_N(\gb,h)+2^{-N}\log \left(\frac{B^2+B-1}{B(B-1)}\right).
\label{eq:encadre}
\end{align}
\end{theorem}

\medskip

Let us also point out that $\tf(\beta,h)\ge 0 $ is immediate in view of the
fact that $R_n^{(i)}\ge (B-1)/B$ for $n\ge 1$, {\sl cf.} \eqref{eq:R}.
The lower bound $\tf(\beta,h)\ge 0 $ implies that we can split the
parameter space (or {\sl phase diagram}) of the system according to
$\tf(\gb,h)=0$ and $\tf (\gb, h)>0$ and this clearly corresponds to
sharply different asymptotic behaviors of $R_n$.  In conformity with
related literature, see \S~\ref{sec:pinning}, we define localized
and delocalized phases as $\cL:=\{(\gb, h):\, \tf(\gb,h)>0\}$ and
$\cD:=\{ (\gb, h):\, \tf(\gb,h)=0\}$ respectively.  It is therefore
natural to define, for given $\beta\ge0$, the {\sl critical} value
$h_c(\beta)$ as
\begin{equation}
h_c(\beta)\, =\, \sup\{h\in\R:\tf(\beta,h)=0\},
\end{equation}
and
Theorem \ref{th:F} says in particular that 
\begin{equation}
h_c(\gb)\, =\, \inf\{h\in\R:\tf(\gb,h)>0\},
\end{equation}
and that $\tf(\gb,\cdot)$ is (strictly) increasing on $(h_c(\beta),\infty)$.
Note that, thanks to the properties we just mentioned, the contact fraction, defined in the end of \S~\ref{sec:qmpm},  is zero 
$h<h_c(\gb)$ and  it is instead positive if
$h>h_c(\gb)$ (define the contact fraction by taking the inferior limit
for the values of $h$ at which  $\tf (\gb, \cdot)$ is not differentiable).

\smallskip

Another important observation on Theorem \ref{th:F} is that 
it yields also the existence of $\lim_{n \to \infty} 2^{-n} \log \langle R_n\rangle$
and this limit is simply $\tf(0,h)$, in  fact
$\tf_n(0, h)= 2^{-n }\log \Rav{n}$ for every $n$.
In statistical mechanics language $\Rav{n}$ is an {\sl annealed} quantity and
$\lim_{n \to \infty} 2^{-n} \log \langle R_n\rangle$
is the {\sl annealed free energy}: by Jensen
inequality it follows that $\tf(\beta,h)\le\tf(0,h)$
and $h_c(\beta)\ge h_c(0)$.
It is also a consequence of Jensen inequality (see Remark~\ref{rem:5}) 
the fact that $\tf(\gb, h+ \log \M( \gb))\ge \tf (0,  h)$,
so that $h_c(\beta)\le h_c(0)+ \log \M (\gb) $.
Summing up:
\begin{equation}
\label{eq:bb}
h_c(0) \, \le \, 
h_c(\beta)\, \le\, h_c(0)+ \log \M (\gb) .
\end{equation}
Therefore, by the convexity properties of $\tf(\cdot, \cdot)$ (Theorem~\ref{th:F})
and by \eqref{eq:bb}, we see that $h_c(\cdot) - \log \M(\cdot)$ is concave
and may diverge only at infinity, so that $h_c(\cdot)$ is a continuous function.

The following result on the {\sl annealed system}, {\sl i.e.} just the non-disordered system,
is going to play an important role:
 
\medskip

\begin{theorem}
  \label{th:pure} {\it (Annealed system estimates)}.  The function
  $h\mapsto \tf(0, h)$ is real analytic 
  except at $h=h_c:=h_c(0)$.  Moreover $h_c=\log
  (B-1)$ and there exists $c=c(B)>0$ such that for all
  $h\in(h_c,h_c+1)$
\begin{eqnarray}
\label{eq:alpha0}
  c(B)^{-1}(h-h_c)^{1/\alpha}\, \le\, \tf(0,h)\, \le
  \,  c(B)(h-h_c)^{1/\alpha},
\end{eqnarray}
where 
\begin{equation}
\label{eq:alpha}
  \alpha\, :=\, \frac{\log(2(B-1)/B)}{\log 2}.
\end{equation}
\end{theorem}

\medskip

Bounds on the annealed free energy can be extracted directly from
\eqref{eq:encadre}, namely
that for every $n\ge 1$
\begin{equation}
\label{eq:annbounds}
\frac{B(B-1)}{B^2+B-1}\exp\left(2^n \tf(0,h)\right) \, \le \, \Rav{n} \,\le\,
B \exp\left(2^n \tf(0,h)\right).
\end{equation}
Moreover let us note from now that $\alpha\in(0,1)$ and that
$1/\alpha>2$ if and only if $B<B_c:=2+\sqrt 2$, and $1/\alpha=2$ for
$B=B_c$.  It follows that $\tf(0, h)=o((h-h_c)^2)$ for $B<B_c$ ($\ga <1/2$),
while  this is not true for $B>B_c$ ($\ga>1/2$). 
\medskip

\begin{rem}\rm
\label{rem:oscillations}
  For  models defined on hierarchical lattices, in
  general
  one does not expect 
   the (singular part of the) free energy to have a pure
  power-law behavior close to the critical point $h_c$, but rather to
  behave like $H(\log (h-h_c))(h-h_c)^\nu$, with $\nu$ the critical
  exponent and $H(\cdot)$ a periodic function, see in particular \cite{cf:DIL}.  
    Note that, unless
  $H(\cdot)$ is trivial ({\sl i.e.} constant), the oscillations it produces become more and
  more rapid for $h\searrow h_c$. We have observed numerically such oscillations
  in our case and therefore we expect that
  estimate \eqref{eq:alpha0} cannot be improved at a qualitative level as $h$ approaches
  $h_c$
  (the problem of estimating sharply the size of the oscillations appears to be a
   non-trivial one, but this is not particularly important for our analysis).
\end{rem}

\subsection{Results for the disordered system}
\label{sec:results}

\smallskip

The first result we present gives  information
on the phase diagram: we use the definition
\begin{equation}
  \label{eq:D}
  \gD\,  = \, \gD(\gb)\, :=\, (B-1)^2\left(\frac{M(2\gb)}{M(\gb)^2}-1\right)\, \left(\ge\, 0\right),
\end{equation}
so that $  \text{Var}(R_0)\stackrel{h=h_c}= \gD$. The quantity $\gD$
should be though of  as the size of the disorder at a given $\gb$. 

\medskip

\begin{theorem}
\label{th:eps_c}
Recall that the critical value for the annealed system is $h_c= \log(B-1)$.
We have the following estimates on the quenched critical line:
\begin{enumerate}
\item
Choose $B \in (2, B_c)$. If  $\gD (\gb) \le  B^2 -2(B-1)^2$ then
$h_c(\gb)=h_c$. 
\item Choose $B>B_c$. Then $h_c(\gb)>h_c$ for every $\gb>0$.
  Moreover for $\gb$ small (say, $\gb \le 1$) one can find
  $c\in(0,1)$ such that
\begin{equation}
\label{eq:eps_c}
c \, \gb^{2\ga/(2\ga -1)} \, \le \, 
h_c(\gb ) -h_c \, \le \, c^{-1}  \gb^{2\ga/(2\ga -1)}.
\end{equation} 
\item If $B=B_c$ then one can find $C>0$ 
such that, for $\gb \le 1$,
\begin{equation}
\label{eq:shiftBc}
0 \le h_c(\gb)-h_c \le \exp(-C/\gb^2).  
\end{equation}
\end{enumerate}
Moreover if $\go_1$ is such that $\bbP(\go_1 >t)>0$ for every $t>0$,
then for every $B>2$ we have $h_c(\gb)-h_c>0$ for $\gb$ sufficiently
large, in fact $\lim_{\gb \to \infty} h_c(\gb )=\infty$.
\end{theorem}

Concerning point (1) of the above theorem, let us mention for
completeness that if the disorder variables are not bounded, more
precisely if $\bbP(\go_1>t)$ is non-zero for every $t$, then for every
value of $B>2$ we can prove that $h_c(\beta)>h_c$ for $\beta$
sufficiently large (see Corollary \ref{cor:unbound} below).  \medskip

Of course  \eqref{eq:shiftBc} leaves open an evident question  for
$B=B_c$, that will be discussed  in \S~\ref{sec:pinning}.  We point out
that the constant $C$ is explicit (see Proposition~\ref{th:lwbdc}) but
it does not have any particular meaning.  It is possible to show that
$C$ can be chosen arbitrarly close to the constant given in
\cite{cf:DHV}, but here, for the sake of simplicity, we have decided
to prove a weaker result ({\sl i.e.}, with a smaller constant).  This
is not a crucial issue, since the upper bound on $h_c(\gb)$ is not
comforted by a suitable lower bound.

\smallskip

\begin{figure}[!h]
\begin{center}
\leavevmode
\epsfysize =5.7 cm
\psfragscanon
\psfrag{0}[c]{$0$}
\psfrag{D}[c]{$\gb$}
\psfrag{eps}[c]{$h$}
\psfrag{aa}[c]{ $\hat \gb$}
\psfrag{a}[c]{$\gb_c$}
\psfrag{Del}[c]{\large \cD}
\psfrag{Loc}[c]{\large \cL}
\psfrag{hc}[c]{$h_c$}
\psfrag{B>Bc}[c]{\large $B>B_c$}
\psfrag{B<Bc}[c]{\large $B<B_c$}
\epsfbox{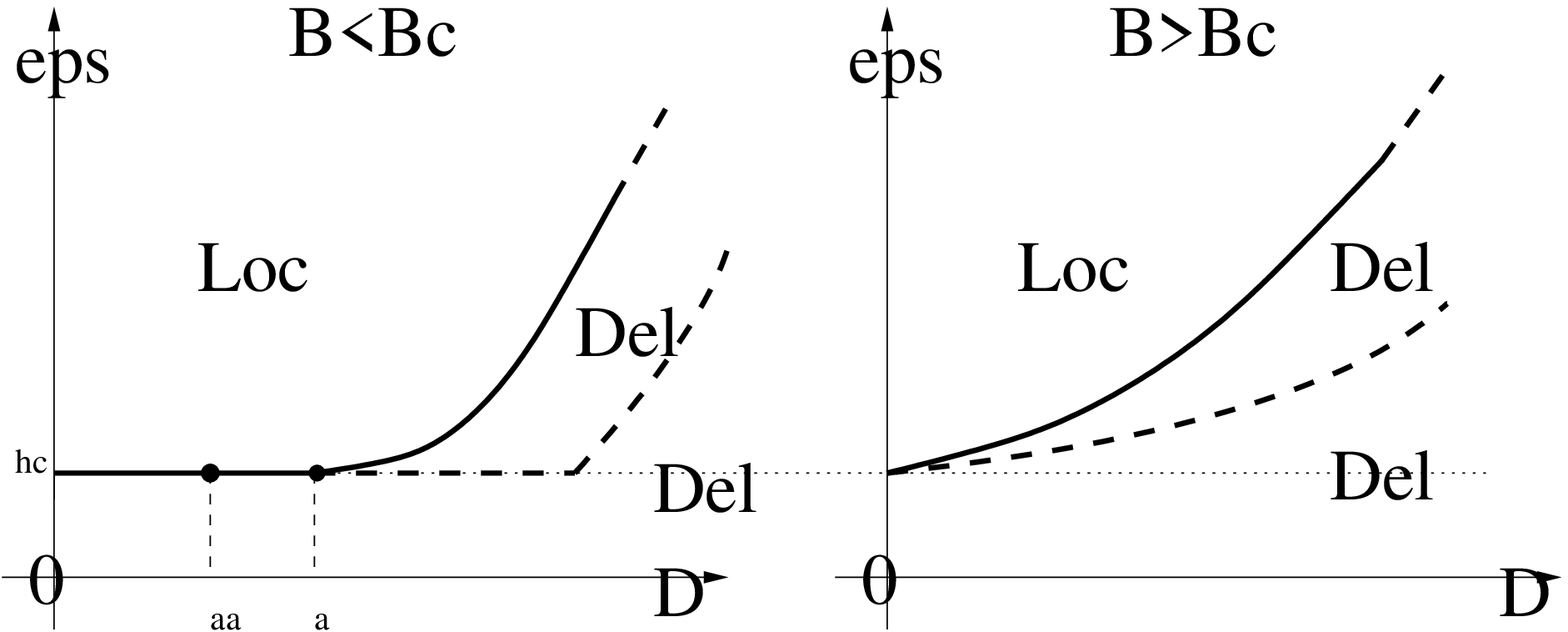}
\end{center}
\caption{\label{fig:hbeta} This is a sketch of the phase diagram 
 and a graphical view of
  Theorem~\ref{th:eps_c} and Theorem~\ref{th:paths}.  The thick line
  in both graphs is $h_c(\cdot)$. The dashed line is instead the lower bound
on $h_c(\cdot)$ which we obtain with our methods.
Below the dashed line we can establish the
  a.s. convergence of $R_n^{(i)}$ to $1$.  We have also used $\gb_c
  :=\sup\{\gb:\, h_c(\gb)=h_c\}$ and 
  $\widehat \gb := \sup\{\gb:\, \gD (\gb) < B^2-2(B-1)^2\}$. We do not prove  
  the (strict) inequality 
$\gb_c> \widehat \gb$.}
\end{figure}

The next result is about  the free energy.

\medskip

\begin{theorem}
\label{th:fe}
We have the following
\begin{enumerate}
\item Choose $B\in (2,B_c)$ and $\gb$ such that
$\gD (\gb) < B^2-2(B-1)^2$. Then for every
$\eta\in (0,1)$ one can find $\epsilon>0$ such that
\begin{equation}
\label{eq:lbfe}
 \tf ( \gb, h)\, \ge\, (1-\eta)  \tf(0,h),
\end{equation}
for $h\in (h_c, h_c+\epsilon)$.
\item Choose $B>B_c$. Then for every
$\eta\in (0,1)$ one can find $c>0$ and $\gb_0>0$ such that
\eqref{eq:lbfe} holds for $\gb< \gb_0$ and 
$h-h_c \in (c \gb^{2\ga /(2\ga -1)}, 1)$.  
\end{enumerate}
\end{theorem}
\medskip

While the relevance of the analysis of the free energy will
be discussed in depth in the next subsection, it is natural to address 
the following issue: in a {\sl sharp} sense, how does the random 
array $R_n^{(1)}$ behave as $n$ tends to infinity?
We recall that the non-disordered system displays only three possible
asymptotic behaviors: $r_n\to 1$, $r_n=B-1$ for all $n$ and 
$r_n\nearrow \infty$ in a super-exponentially fast fashion.

What can be extracted directly from the free energy is quite
satisfactory if the free energy is positive: $R^{(1)}_n$ diverges at a
super-exponential speed that is determined to leading order.  However,
the information readily available from the fact that the free energy
is zero is rather poor; this can be considerably improved, starting
with the fact that, by the lower bound in \eqref{eq:encadre}, if the
free energy is zero then $\sup_n \langle \log R_n \rangle \le \log B$,
which implies the tightness of the sequence.

\medskip

\begin{theorem}
  \label{th:paths} If $\tf (\gb, h)=0$ then the sequence $\{R_n\}_n$
  is tight. Moreover 
if $h<h_c(\beta)$ then
\begin{equation}
\label{eq:paths}
\lim_{n \to \infty} R_n^{(1)}\, =\,1 \mbox{\; in probability}.
\end{equation}
\end{theorem}

\medskip

Let us mention that we also establish almost sure convergence of $R_n$
toward $1$ when we are able to find $\gamma\in (0,1)$ and $n\in \N$
such that $\bbE \left[ ([R_n-1]^+)^\gamma\right]$ is smaller than an
explicit constant (see Section~\ref{sec:rel}, in particular
Remark~\ref{rem:as1}). It is interesting to compare such results
with the estimates on the size of the partition
function $Z_{N,\go}$ of non-hierarchical pinning/wetting models, which
are proven in \cite[end of Sec.~3.1]{cf:T_fractmom} in the delocalized
phase, again via estimation of fractional moments of $Z_{N,\go}$
(which plays the role of our $R_n$).

\medskip 

What one should expect at criticality is rather unclear to us (see
however \cite{cf:GM_h} for a number of predictions and numerical
results on hierarchical pinning and
also \cite{cf:CEGM-JSP,cf:CEGM-CMP} for
some theoretical considerations on a different class of hierarchical models).

\subsection{Pinning models: the role of
disorder}
\label{sec:pinning}

Hierarchical models on diamond lattices, homogeneous or disordered
\cite{cf:Bleher,cf:BO,cf:CEGM-JSP,cf:CEGM-CMP,cf:DG}, are a powerful
tool in the study of the critical behavior of statistical mechanics
models, especially because real-space renormalization group
transformations {\sl \`a la } Migdal-Kadanoff are exact in this case.
In most of the cases, hierarchical models are introduced in association
with a more realistic non-hierarchical one. It should however be
pointed out that hierarchical models on diamond lattices are not rough
simplifications of non-hierarchical ones. They are in fact meant to
retain the essential features of the associated non-hierarchical
models (notably: the critical properties!). 
In particular, it would be definitely misleading to think of the
hierarchical model as a mean field approximation of the real one.

Non-hierarchical pinning models have an extended
literature ({\sl e.g.}  \cite{cf:Fisher,cf:Book}).  They may be
defined like in \eqref{eq:DHV}, with $S$ a symmetric random walk with
increment steps in $\{-1, 0, +1\}$, energetically rewarded or
penalized when the bond $(S_{n-1},S_n)$ lies on the horizontal axis
(that is $d_j=0$ for every $j$ in  \eqref{eq:DHV}),
but they can be restated in much greater generality by considering
arbitrary homogeneous Markov chains that visit a given site (say, the
origin) with positive probability and that are then rewarded or
penalized when passing by this site.  In their non-disordered version
\cite{cf:Fisher}, this general class of models has the remarkable
property of being {\sl exactly solvable}, while displaying a phase
transition --  a localization-delocalization transition -- and
the order of such a transition depends on a parameter of the model
(the tail decay exponent of the distribution of the first return of
the Markov chain to the origin: we call $\ga$ such an exponent and it
is the analog of the quantity $\ga$ in our hierarchical context, {\sl
  cf.} \eqref{eq:alpha}; one should however note that for
non-hierarchical models values $\ga\ge1$ can also be considered, in
contrast with the model we are studying here). As a matter of fact,
transitions of all order, from first order to infinite order, can be
observed in such models.  They therefore constitute an ideal set-up in
which to address the natural question: how does the disorder affect
the transition?

Such an issue has often been considered in the physical literature and
a criterion, proposed by A.~B.~Harris in a somewhat different context,
adapted to pinning models \cite{cf:FLNO,cf:DHV},  yields that the
disorder is irrelevant if $\gb$ is small and $\ga <1/2$, meaning by
this that quenched and annealed critical  points coincide and
the critical behavior of the free energy is the same for annealed and
quenched system (note that the annealed system is a homogeneous
pinning system, and therefore exactly solvable).  The disorder instead
becomes relevant when $\ga>1/2$, with a shift in the critical point
(quenched is different from annealed) and different critical behaviors
(possibly expecting a smoother transition, but the Harris criterion
does not really address such an issue). In the marginal case,
$\alpha=1/2$, disorder could be {\sl marginally} relevant or {\sl
  marginally} irrelevant, but this is an open issue in the physical
literature, see \cite{cf:FLNO,cf:DHV} and \cite{cf:Book} for further
literature.

\smallskip

Much progress has been made very recently in the mathematical literature 
on non-hierarchical pinning models, in particular:
\begin{enumerate}
\item The irrelevant disorder regime is under control
  \cite{cf:Ken,cf:T_cmp} and even more detailed results on the
  closeness between quenched and annealed models can be established
  \cite{cf:GT_irrel}.
\item 
Concerning 
 the  relevant disorder regime, in \cite{cf:GT_cmp} it has been shown that the quenched
  free energy is smoother than the annealed free energy if $\ga>1/2$.
The non-coincidence of quenched and annealed critical points for 
{\sl large} disorder (and for every $\ga$) has been  proven in 
\cite{cf:T_fractmom} via an estimation of non-integer moments of the
partition function. The idea of considering non-integer moments (this time,
of $R_n-1$)  plays an important role also in the present paper.
\item A number of results on the behavior of the paths of the model
  have been proven addressing the question of what can be said about
  the trajectories of the system once we know that the free energy is
  zero (or positive) \cite{cf:GTdeloc,cf:GT_alea}.  One can in fact
  prove that if $\tf(\gb, h)>0$ then the process sticks close to the
  origin (in a strong  sense) and it is therefore in a localized
  ($\cL$) regime.  When  $\tf(\gb, h)=0$, and leaving aside the
  critical case, one expects that the process {\sl essentially never}
  visits the origin, and we say that we are in a delocalized regime
  ($\cD$).  We refer to \cite{cf:Book} for further discussion and
  literature on this point.
\end{enumerate}
\medskip

In this work we rigorously establish the full Harris criterion picture
for the hierarchical version of the model. In particular we wish to
emphasize that we do show that there is a shift in the critical point
of the system {\sl for arbitrarily small disorder} if $\ga>1/2$ and we
locate such a point in a window that has a precise scaling behavior,
{\sl cf.} \eqref{eq:eps_c} (a behavior which coincides with that
predicted in \cite{cf:DHV}).

As a side remark, one can also generalize the smoothing inequality
proven in \cite{cf:GT_cmp} to the hierarchical context and show that
for every $B>2$ there exists $c(B)<\infty$ such that, if $\go_1\sim
\cN (0,1)$, for every $\beta>0$ and $\delta>0$ one has
\begin{equation}
\label{eq:cB}
  \tf(\beta,h_c(\beta)+\delta)\, \le\,  {\delta^2}c(B)/{\beta^2},
\end{equation}
which implies that annealed and quenched free energy critical 
behaviors are different for $\ga >1/2$, {\sl cf.}  \eqref{eq:alpha0} (as
in \cite{cf:GT_cmp}, such inequality can be generalized well beyond
Gaussian $\go_1$, but we are not able to establish it only assuming
the finiteness of the exponential moments of $\go_1$). 
The proof of \eqref{eq:cB} is detailed in \cite{cf:LT}

\medskip

Various intriguing issues remain open:
\begin{enumerate}
\item Is there a shift in the critical point at small disorder if
  $B=B_c$ (that is $\ga=1/2$)? We stress that in \cite{cf:DHV}
is predicted that $h_c(\gb)-h_c(0)\simeq \exp(-\log 2/\gb^2)$ for 
$\beta$ small.
\item Can one go beyond \eqref{eq:cB}? That is, can one find sharp estimates
  on the critical behavior when the disorder is relevant?
\item With reference to 
the caption of
  Figure~\ref{fig:hbeta}, can one prove $\gb_c > \widehat \gb$ (for 
  $B<B_c$)? 
\item Does the law of $R_n$ converge to a non-trivial limit for $n\to\infty$,
when $h=h_c(\gb)$?
 \end{enumerate}
\medskip

Of course, all these issues are open also in the non-hierarchical context
and, even if not every question becomes easier for the hierarchical model,
it may be the right context in which to attack them first.

\subsection{Some recurrent notation and organization of the subsequent sections}

Aside for standard notation like $\lceil x\rceil := \min\{n\in \Z: \,
n\ge x\}$  and $\lfloor x\rfloor:= \lceil x\rceil -1$, or
$[\cdot]^+:=\max(0,\cdot)$, we will repeatedly use $\gD_n$ for the
variance of $R_n^{(1)}$, see \eqref{eq:Delta}, and $Q_n := \gD_n /
\Rav{n}^2$ so that from \eqref{eq:Rav} and \eqref{eq:Delta}, one sees
that
\begin{equation}
\label{eq:Q}
Q_{n+1}\, =\,
2 \left( \frac{B-1}B\right)^2
\left(\frac{\Rav{n}^4}{\Rav{n+1}^2 (B-1)^2}\right) 
\left( Q_n + \frac 12 Q_n^2\right)
\end{equation}
and we observe that
\begin{equation}
  \label{eq:Q0}
  Q_0\, =\,\left(\frac{M(2\gb)}{M(\gb)^2}-1\right)\stackrel{\beta\searrow0}
\sim\beta^2.
\end{equation}
Note that  $2(B-1)^2/B^2$ is smaller than $1$ if and
only if $B<B_c$ and 
\begin{equation}
\label{eq:Qbound}
\left(\frac{\Rav{n}^4}{\Rav{n+1}^2 (B-1)^2}\right) \, \le \, \left( \frac B{B-1}\right)^2.
\end{equation}

We will also frequently use  $P_n:=\langle R_n\rangle-(B-1)$, which satisfies
\begin{equation}
\label{eq:P}
P_{n+1}\, =\,
2\frac{(B-1)}B P_n+ \frac 1B P_n^2,
\end{equation}
and of course $P_0=\gep$ in our notations.
With some effort, one can explicitly verify that for every $n$ 
\begin{equation}
\label{eq:QboundP}
\left(\frac{\Rav{n}^4}{\Rav{n+1}^2 (B-1)^2}\right) \, \le \, 1+ \frac{4P_n}{B(B-1)}.
\end{equation}

Finally, there is some notational convenience at times in making the change of variables 
\begin{equation}
  \label{eq:gep}
  \gep\, :=\, \Rav{0}-(B-1)\, =\,  e^h-(B-1)\ge -(B-1),
\end{equation}
and 
\begin{equation}
\label{eq:hatF}
\hat
\tf(\gb,\gep)\,:=\,\tf(\beta,h(\gep)),
\end{equation}
and when we write $ h(\gep)$ we refer to the invertible map 
defined by \eqref{eq:gep}.

\medskip

The work is organized as follows. Part (1) of Theorem \ref{th:eps_c}
and of Theorem \ref{th:fe} are proven in Section \ref{sec:irrel}. In
Section \ref{sec:LBrel} we prove part (2) of Theorem \ref{th:fe} and,
as a consequence, part (2) of Theorem \ref{th:eps_c}, except the lower
bound in \eqref{eq:eps_c}. Part (3) of Theorem \ref{th:eps_c} is
proven in Section \ref{sec:LBmarg} and the lower bound of
\eqref{eq:eps_c} in Section \ref{sec:rel} (after a brief sketch of our
method). The proof of Theorem \ref{th:paths} is given in Section
\ref{sec:deloc}. Finally, the proofs of Theorems \ref{th:F} and
\ref{th:pure} are based on more standard techniques and can be found
in Appendix \ref{sec:stand}.

\section{Free energy lower bounds: $B<B_c= 2+\sqrt{2}$}

\label{sec:irrel}

We want to give a proof of part (1) of Theorem \ref{th:fe}, which in particular
implies part (1) of Theorem \ref{th:eps_c}. 

The strategy goes roughly as follows: since $h>h_c$ is close to $h_c$,
that is $\gep (=P_0)>0$ is close to $0$, $P_n$ keeps close to zero for
many values of $n$ and $P_{n+1} \approx (2(B-1)/B)P_n$ (recall that
$2(B-1)/B>1$ for $B>2$).  This is going to be true up to $n$ much
smaller than $\log (1/\gep) / \log (2(B-1)/B)$.  At the same time for
the normalized variance $Q_n$ we have the approximated recursion
$Q_{n+1} \approx 2((B-1)/B)^2 (Q_n+(1/2) Q_n^2)$, which one derives
from \eqref{eq:Q} by using $P_n\approx 0$. Since $2((B-1)/B)^2 <1$ is
equivalent to $B<B_c$, we easily see that (if $Q_0$ is not too large)
$Q_n$ shrinks at an exponential rate. This {\sl scenario} actually
breaks down when $P_n$ is no longer small, but at that stage $Q_n$ is
already extremely small (such a value of $n$ is precisely defined and
called $n_0$ below). From that point onward $Q_n$ starts growing
exponentially and eventually it diverges, but after $(1+\gamma) n_0$
steps, for some $\gamma>0$, $Q_n$ is still small while $P_n$ is large,
so that a second moment argument, combined with \eqref{eq:encadre}
which yields a control on $\tf( \gb, h)$ via $\tf_n(\gb, h)$, allows
to conclude.




\medskip

Before starting the proof we give an upper bound on the size 
of  $Q_n(= \gD_n/ \Rav{n}^2)$ in the regime in which the recursion for
$\Rav{n}$ can be linearized (for what follows, recall  
\eqref{eq:Q},
\eqref{eq:Q0} and \eqref{eq:Qbound}).

\smallskip

\begin{lemma}
\label{th:linear}
Let $B\in(2,B_c)$ and $\gb$ such that $\gD (\gb)<B^2-2(B-1)^2$.  There exist
$c:=c(B,\Delta)>0$, $c_1:=c_1(B,\Delta)>0$ and
$\gd_0:=\delta_0(B,\Delta)>0$ with
\begin{eqnarray}
\label{eq:d0c1}
  2(1+\delta_0)\left(\frac{B-1}B\right)^2<1
\end{eqnarray}
such that for every $\gep$ satisfying $0< (B-1) \gep < 
( (B^2-2(B-1)^2)/\gD)^{1/2}-1
$ (recall the definition \eqref{eq:gep} of $\gep$)
and
\begin{equation}
\label{eq:x0-1}
n \, \le \,n_0:=  \left\lfloor \log\left(c\,\gd_0 /\gep \right) / 
\log \left( \frac{2(B-1)}B\right)\right\rfloor,
\end{equation}
one has
\begin{equation}
\label{eq:linear}
Q_n \, \le \, c_1\,Q_0 \left( 2(1+\gd_0)\left(\frac{B-1}B\right)^2 \right)^n .
\end{equation} 
\end{lemma}
\smallskip

Note that the condition 
on $\gep $ simply guarantees  $\gD_0 =
(1+\gep/(B-1))^2 \gD$
is smaller than
$ B^2-2(B-1)^2$.
\smallskip 
 
\noindent
{\it Proof of Lemma \ref{th:linear}.}
Recall that  $P_n=\langle R_n\rangle-(B-1)$ and that
it satisfies the recursion \eqref{eq:P}
(and that $P_0=\gep$).

 For $G_n := (P_n/P_0) (2(B-1)/B)^{-n}$ we have from \eqref{eq:P} and
\eqref{eq:gep}
\begin{equation}
G_{n+1} \, =\, G_n  + \frac {\gep}B \,\left( 2\frac{(B-1)}B\right)^{n-1} G_n^2,
\end{equation}
and $G_0=1$. If  $G_m\le 2$ for $m\le n$, then
\begin{equation}
\frac{G_{n+1}}{G_n} \, \le \, 1+ 2\frac {\gep}B \,\left( 2\frac{(B-1)}B\right)^{n-1},
\end{equation}
which entails

\begin{equation} 
\label{eq:recineqGn}
G_{n+1}  \, \le \, 
\exp\left( 2\frac {\gep}B \sum_{j=0}^{n}\left( 2\frac{(B-1)}B\right)^{j-1}
\right)
\, \le \, 
1+\gep C(B) \left(\frac{2(B-1)}B\right)^{n+1},
\end{equation}
for a suitable constant $C(B)<\infty$. 

As we have already remarked, our assumption on $\gep$ yields  $\gD_0 < B^2-2(B-1)^2$,
so  
\begin{equation}
  \label{eq:cQ0}
Q_0 \, <\,  \left(\frac B{B-1}\right)^2-2. 
\end{equation}
Choose $\delta_0>0$ sufficiently small so that \eqref{eq:d0c1} is
satisfied and moreover
\begin{eqnarray}
  \label{eq:d0c2}
  2\left(\frac{B-1}B\right)^2(1+\delta_0)(Q_0+\frac12Q_0^2)<Q_0
\end{eqnarray}
(the latter can be satisfied in view of \eqref{eq:cQ0}).  It is
immediate to deduce from \eqref{eq:recineqGn} that if $c$ in 
\eqref{eq:x0-1} is chosen sufficiently small (in particular, $c\le B(B-1)/8$), 
then $G_n\le 2$ for
$n\le n_0$ and, as an immediate consequence,
\begin{equation}
\label{eq:recineqPn}
0<P_{n} \, \le \, 2 \gep  \left(\frac{2(B-1)}B\right)^{n}\, \le \,
2c\delta_0\,\le\, 
\delta_0 \frac {B(B-1)}4,
\end{equation}
where the first inequality is immediate from \eqref{eq:P} and $P_0=\gep>0$.
Now we apply \eqref{eq:QboundP}
\begin{eqnarray}
  \label{eq:Qnnuova}
  Q_{n+1}\le 2\left(\frac{B-1}B\right)^2(1+\delta_0)(Q_n+\frac12Q_n^2).
\end{eqnarray}
Notice also that $Q_1<Q_0$ thanks to \eqref{eq:d0c2}. From this it is
easy to deduce that, as long as $n\le n_0$, $Q_n$ is decreasing and
satisfies \eqref{eq:linear} for a suitable $c_1$.  In particular, $c_1(B, \gD_0)$ can be chosen such that
$\lim_{\gD_0 \searrow 0} c_1(B, \gD_0)=1$.  \qed

\bigskip

\noindent
{\it Proof of Theorem \ref{th:fe}, part (1).}
We use the bound \eqref{eq:Qbound} to get
\begin{equation}
\label{eq:superlinear}
Q_{n+1} \, \le \, 2 \left( Q_n + \frac 12 Q_n^2 \right) \, \le \, 3Q_n,
\end{equation}
where the last inequality holds as long as $Q_n \le 1$.  
Then we apply Lemma~\ref{th:linear} (recall in particular $\gd_0$ and $n_0$ in there).
Combining
\eqref{eq:linear} and \eqref{eq:superlinear} we get
\begin{equation}
\label{eq:QnQn0}
Q_n \, \le \, Q_{n_0} 3^{n-n_0} \, \le\, 
c_1\,Q_0 \left( 2(1+\gd_0)\left(\frac{B-1}B\right)^2 \right)^{n_0} 3^{n-n_0},
\end{equation}
for every $n\ge n_0$ 
satisfying $Q_n \le 1$ (which implies $Q_n'\le 1$ for all $n'\le n$ as $Q_n$ is increasing). Of course this boils
down to requiring that the right-most term in \eqref{eq:QnQn0} does not
get larger than $1$. Since $n_0$ diverges as $\gep \searrow 0$, if we choose 
$\gamma >0$ such that $3^\gamma\, 2(1+\gd_0) (B-1)^2/B^2 <1$, then 
the right-most term in \eqref{eq:QnQn0} is bounded above for every
$n \le (1+ \gamma) n_0$ by a quantity $o_\gep(1)$ which 
vanishes for $\gep\to0$.
Summing all up:
\begin{equation}
Q_{\lfloor(1+\gamma)n_0\rfloor }\, =\, o_\gep(1).
\end{equation}
 Next,
note that 
\begin{equation}
\label{eq:decomp}
\langle \log R_{\lfloor(1+\gamma)n_0\rfloor}\rangle  \ge
   \log\left(\frac 12
  \med{R_{\lfloor(1+\gamma)n_0\rfloor}}\right)
\bbP\left(R_{\lfloor(1+\gamma)n_0\rfloor}\ge
  \frac 12 \med{R_{\lfloor(1+\gamma)n_0\rfloor}}\right)+\log
  \left(\frac{B-1}B\right),
\end{equation}
where we have used the fact that $R_n\ge
(B-1)/B$ for $n \ge 1$. Applying the Chebyshev inequality one has
\begin{eqnarray}
\label{eq:PZ}
 \bbP\left(R_{\lfloor(1+\gamma)n_0\rfloor}\ge (1/2)
\med{R_{\lfloor(1+\gamma)n_0\rfloor}}\right)&\ge&
1-4Q_{\lfloor(1+\gamma)n_0\rfloor}=1+o_\gep(1).
\end{eqnarray}
Therefore, from  \eqref{eq:annbounds}, \eqref{eq:decomp} and
\eqref{eq:PZ}  one has
\begin{eqnarray}
  \tf_{\lfloor(1+\gamma)n_0\rfloor}(\beta,h)\ge
(1+o_\gep(1)) \,\hat \tf(0,\gep)-
2^{-\lfloor(1+\gamma)n_0\rfloor}c(B)
\end{eqnarray}
for some $c(B)<\infty$ and, from \eqref{eq:encadre} (or, equivalently, \eqref{eq:nonincrease}),
\begin{eqnarray}
  \tf(\beta,h)\ge (1+o_\gep(1)) \,\hat\tf(0,\gep)-
2^{-\lfloor(1+\gamma)n_0\rfloor}c_1(B).
\end{eqnarray}
Since $\hat\tf(0,\gep)2^{\lfloor(1+\gamma)n_0\rfloor}$ diverges for
$\gep\to0$ if $\gamma>0$, as one may immediately check from
\eqref{eq:x0-1} and \eqref{eq:alpha0}, one directly extracts that for
every $\eta>0$ there exists $\gep_0>0$ such that
\begin{equation}
\label{eq:R1alpha}
\hat \tf(\Delta_0,\gep)=\tf(\beta,h)
\ge \, (1-\eta) \hat\tf(0,\gep)
\end{equation}
for $\gep \le \gep_0$, and we are done.
\qed 

\section{Free energy lower bounds: $B\ge B_c =2+\sqrt{2}$}

\label{sec:LBrel}

The arguments in this section are close in spirit to the ones of the
previous section. However, since $B>B_c$, the constant $2((B-1)/B)^2$ in
the linear term of the recursion equation \eqref{eq:Q} is larger than
one, so the normalized variance $Q_n$ grows from the very beginning.
Nonetheless, if $Q_0$ is small, it will keep small for a while. The
point is to show that, if $P_0$ is not too small (this concept is of
course related to the size of $Q_0$), when $Q_n$ becomes of order one
$P_n$ is sufficiently large. Therefore, once gain, a second moment
argument and \eqref{eq:encadre} yield the result we are after, that
is:

\medskip

\begin{proposition}
\label{th:LB_B>B_c}
Let $B>B_c$.  For every $\eta \in (0,1)$ there exist $c>0$
and $ \gb_0>0$ such that
\begin{equation}
 \label{eq:LBBc}
\tf \left( \beta ,h\right) \, \ge \, 
(1-\eta)  \tf \left( 0, h\right),
\end{equation}
for $\gb \le \gb_0$ and $c \gb^{2\ga/ (2\ga -1)}\le h-h_c(0)\le 1$.
This implies in particular that
$h_c(\gb) < h_c(0)+ c \gb^{2\ga/ (2\ga -1)}$, for every $\gb \le \gb_0$.
\end{proposition}
Of course this proves part (2) of Theorem \ref{th:fe} and the upper bound
in \eqref{eq:eps_c}.

\medskip

In this section $q:=2(B-1)^2/B^2$ and $\bar q := 2(B-1)/B$: note that
in full generality $ q < \bar q <2$ and $\bar q>1$, while $q>1$
because we assume $B>B_c$. One can easily check that
\begin{eqnarray}
  \label{eq:aa-1}
\frac{\ga}{2\ga-1}=\frac{\log\bar q}{\log q}.
\end{eqnarray}
Moreover in what follows
some expressions are in the form $\max A$, $A\subset \N \cup\{0\}$:
also when we do not state it explicitly, we do assume that $A$ is not
empty (in all cases this boils down to choosing $\gb$ sufficiently
small).

\smallskip

We start with an upper bound on the growth of 
$\Rav{n}= (B-1)+ P_n$ (recall \eqref{eq:P}) 
for $n$ {\sl not too large}.
\medskip

\begin{lemma}
\label{th:lem1-rel}
If $P_0=c_1 \gb^{2\ga/(2\ga -1)}$, $c_1>0$, then 
\begin{equation}
\label{eq:lem1-rel}
P_n \,\le\, 2c_1 \gb^{2\ga/(2\ga -1)} \bar q ^n\, \le \, 1
\end{equation}
for $n \le N_1 := \max\{n:\, C_1(B) c_1 \gb^{2\ga/(2\ga -1)} \bar q ^n
\le 1 \}$, where 
\begin{equation}
C_1(B) \,:= \,
2 \max\left(\frac1{(\bar q -1)B \log 2},1\right).
\end{equation}
\end{lemma}

\medskip

The next result controls the growth of the variance of $R_n$ in the
regime when $\Rav{n}$ is close to $(B-1)$, {\sl i.e.} $P_n$ is small.
Let us set $N_2 := \max\{n:\, (2 c_1 /(\bar q -1))\gb ^{2\ga
  /(2\ga -1)} \bar q^n \le (\log 2)/2\}$. Observe that $N_2\le N_1$
and recall that $Q_0 \stackrel{\gb \searrow 0} \sim \gb^2$, cf.
\eqref{eq:Q0}.

\medskip
\begin{lemma}
\label{th:lem2-rel}
Under the same assumptions as in  Lemma~\ref{th:lem1-rel}, for
$Q_0 \le 2\gb^2$ and assuming $c_1 \ge 20^{\log \bar q / \log q}$
we have 
\begin{equation}
Q_n \, \le\, 2 Q_0 q^n,
\end{equation}
for $n \le N_2$.
\end{lemma}

\medskip

\noindent
{\it Proof of Proposition~\ref{th:LB_B>B_c}}.
Let us choose $c_1$ as in  Lemma~\ref{th:lem2-rel}. Let us observe also that,
thanks to \eqref{eq:aa-1},
 $N_2 = \lfloor \log (1/\gb^2) /\log q - \log (C c_1)/ \log \bar q\rfloor$ 
 for a suitable choice of the constant $C=C(B)$. 
Therefore   Lemma~\ref{th:lem2-rel} ensures that
\begin{equation}
Q_{N_2} \, \le\, 4 (C c_1)^{-\log q /\log \bar q}.
\end{equation}
From the definition of $Q_n$ we directly see that
$Q_{n+1 }\le 3 Q_n$ if $Q_n \le 1$, as in \eqref{eq:superlinear}. 
Therefore for any fixed $\gd\in(0,1/16)$
\begin{equation}
\label{eq:Qbound1}
Q_{N_2+n } \,\le \,3^n 4(C c_1)^{-\log q /\log \bar q} \, \le \, 4\gd,
\end{equation}
if 
\begin{equation}
\label{eq:Qbound2}
n \, \le \, N_3 \, :=\, \left \lfloor \frac{\log q \log (C c_1)}{\log
  \bar q \log 3} - \frac{\log(1/\gd)}{\log 3}\right\rfloor.
\end{equation}
Since $Q_{N_2+N_3} \le
4\gd$ (by definition of $N_3$), we have then
 \begin{equation}
   \bbP\left( R_{N_2+N_3} \, \le \, \frac 12 \Rav{N_2+N_3} \right) \, 
\le \, 16 \gd.
 \end{equation}
As a consequence, applying \eqref{eq:encadre} and \eqref{eq:annbounds}
with $N=N_2+N_3$ one finds
\begin{eqnarray}
  \label{eq:c3}
  \tf(\beta,h)\ge (1-16\delta)\tf(0,h)-2^{-(N_2+N_3)}c_3(B),
\end{eqnarray}
of course with $h$ such that $P_0=c_1\beta^{2\ga/(2\ga-1)}$, {\sl i.e.},
\begin{eqnarray}
  h=\log\left((B-1)+c_1\beta^{2\ga/(2\ga-1)}\right).
\end{eqnarray}
  The last
step consists in showing that the last term in the right-hand side of
\eqref{eq:c3} is negligible with respect to the first one.  A look at
\eqref{eq:Qbound2} shows that $N_3$ can be made arbitrarily large by
choosing $c_1$ large; moreover, by definition of $N_2$ we have
\begin{equation}
2^{N_2} c_1^{1/\ga}\gb ^{2/(2\ga -1)} \, \ge \, \frac 12 C^{-1/\ga},
\end{equation}
for $\gb $ sufficiently small.  From these two facts and from the
critical behavior of $\tf(0,\cdot)$ (cf. \eqref{eq:alpha0}) one
deduces that for any given $\delta$ one may take $c_1$ sufficiently
large so that
\begin{equation}
2^{-(N_2+N_3)}/\tf(0,h)\le \delta,
\end{equation}
provided that $h\le h_c(0)+1$.
For a given $\eta\in(0,1)$ this  proves \eqref{eq:LBBc} whenever 
$\beta$ is sufficiently small and $c \gb^{2\ga/(2\ga-1)}\le h-h_c(0)\le1$,
with $c$ sufficiently large (when $\eta$ is small) but independent of $\beta$.


\qed

\medskip

\noindent
{\it Proof of Lemma~\ref{th:lem1-rel}.}
Call $N_0$ the largest value of $n$ for which $P_n \le 2c_1 \gb^{2\ga/(2\ga -1)}
\bar q^n$ (for $c_1$ and $\gb$ such that $P_0 \le  1$).
Recalling \eqref{eq:P}, for $n \le N_0$ we have
\begin{equation}
\frac{P_{n+1}}{P_n}\, \le \, \bar q \left( 1+ \frac{2c_1}{B \bar q}
\gb ^{2\ga/(2\ga -1)} \bar q ^n\right),
\end{equation}
so that for $N\le N_0$, using the properties  of $\exp(\cdot)$ and
the elementary bound $\sum_{n=0}^{N-1} a^n\le a^N/(a-1)$ ($a>1$),
 we obtain
\begin{equation}
P_N \, \le \, P_0 \,\bar q^N\exp \left( \frac{2c_1}{ (\bar q-1)B}
\gb ^{2\ga/(2\ga -1)} \bar q ^N \right).
\end{equation}
The latter estimate yields a lower bound on $N_0$:
\begin{equation}
\label{eq:ifrhs}
N_0 \, \ge \, \, \max\left\{
n:\, \frac{2c_1}{ (\bar q-1)B}
\gb ^{2\ga/(2\ga -1)} \bar q ^n \le \log 2
\right\}.
\end{equation}
$N_1$ is found by choosing it as the minimum between the right-hand side 
in \eqref{eq:ifrhs} and the maximal value of $n$ for which 
the second inequality in \eqref{eq:lem1-rel} holds.
  \qed

\medskip

\noindent
{\it Proof of Lemma~\ref{th:lem2-rel}.}
Let us call $N_0^{\prime}$ the largest $n$ such that $Q_n \le 2Q_0 q^n$ 
($N_0^{\prime}$ is introduced to control the nonlinearity in \eqref{eq:Q}) 
and let us work with $n \le \min(N_0^{\prime}, N_2)$.
Since $N_2 \le N_1$, 
 ($N_1$ given in Lemma~\ref{th:lem1-rel}), the bound \eqref{eq:lem1-rel} holds and $P_n \le 1$.
Therefore, by using first \eqref{eq:Q} and \eqref{eq:QboundP}, and then   
 \eqref{eq:lem1-rel}, we have 
\begin{equation}
\label{eq:satQ}
\frac{Q_{n+1}}{Q_n} \, \le \, q (1+P_n) \left( 1+2\gb^2 q^n\right)\, \le \, 
q \left(1 + 2 c_1  \gb ^{2\ga /(2\ga -1)} \bar q^n + 
4\gb^2 q^n
\right), 
\end{equation}
which implies
\begin{equation}
\label{eq:satQ-2}
Q_n \, \le \, 
Q_0 q^n \exp \left( 
  \frac{2c_1}{\bar q-1} \gb ^{2\ga /(2\ga -1)}\bar q^n  + \frac{4 }{q-1} \gb^2 q^n
\right) 
\end{equation}
By definition of $N_2$ the first term in the exponent is at most $(\log 2)/2$.
Moreover $n \le N_2$ implies, via \eqref{eq:aa-1},
\begin{equation}
\label{eq:satQ-3}
n \, \le \, \frac{\log (1/\gb^2)}{\log q}- \frac{\log\left((4/\log 2) c_1
/ (\bar q -1)\right)}{ \log \bar q},
\end{equation}
and one directly sees that for such values of $n$ we have $\gb^2 q^n
\le \left(4 c_1 / (\bar q -1)\log 2\right)^{-\log q / \log \bar
q}$. Therefore also the second term in the exponent ({\sl cf.}
\eqref{eq:satQ-2}) can be made smaller than $(\log 2)/2$ by choosing
$c_1$ larger than a number that depends only on $B$, see the statement
for an explicit expression.

Summing all up, for $c_1$ chosen suitably large, $Q_n\le 2Q_0 q^n$ for
$n \le \min(N_0^{\prime}, N_2)$. But, by definition of $N_0^{\prime}$,
this just means $n \le N_2$ and the proof is complete.  \qed

\subsection{The $B=B_c$ case}

\label{sec:LBmarg}
\begin{proposition}\label{th:lwbdc}
Set $B=B_c$. There exists $\gb_0$ such that for all $\gb\le \gb_0$
\begin{equation}
h_c(\gb)-h_c(0)\, < \,  \exp\left(-\frac{(\log2)^2}{2 \gb^2}\right).
\end{equation}
\end{proposition}

\medskip

\begin{rem} \rm The constant $(\log 2)^2/2$ that appears in the exponential is
  certainly not the best possible. In fact, one can get arbitrarily close 
   to the optimal constant $\log 2$ given in \cite{cf:DHV},
  but we made the choice to keep the proof as simple as possible.
\end{rem} 

\medskip

\noindent{\it Proof of Proposition \ref{th:lwbdc}.}
Choose $$h= e^{-(\log 2)^2/ (2\gb^2)} + \log(B_c-1),$$ so that $$P_0=
\exp(h)-(B_c-1) \stackrel{\gb \searrow 0}\sim (B_c-1) \exp(-(\log 2)^2/(2
\gb^2)).$$ Given $\gd>0$ small (for example, $\gd=1/70$), we let
$n_\gd$ be the integer uniquely identified (because of the strict
monotonicity of $\{P_n\}_n$) by
\begin{equation}
P_{n_\gd}\, < \, \gd \, \le \, P_{n_\gd +1}
\end{equation} 
(we assume that $P_0<\gd$, which just means that we take $\gb$ small
enough).  We observe that \eqref{eq:P} implies $P_{n+1}/P_n \ge
\sqrt{2}$ for every $n$, from which follows immediately that (say, for
$\gb$ sufficiently small)
\begin{eqnarray}
n_\gd\le  \left \lceil \frac{\log 2}{\beta^2}\right\rceil.
\end{eqnarray}
  We want to show first of
all that $Q_{n_\gd}$ is of the same order of magnitude as
$Q_0$, and therefore much smaller than $P_{n_\gd}$ (for $\gb$ small)
in view of $Q_0 \stackrel{\gb \searrow 0}\sim \gb^2$.

From \eqref{eq:Q}, recalling the definition of $P_n$ ({\sl cf.}
\eqref{eq:P}) and the bound \eqref{eq:QboundP}, we derive
\begin{equation}
 \label{eq:QQ}
Q_{n+1}\, =\,
\left(\frac{\Rav{n}^4}{\Rav{n+1}^2 (B-1)^2}\right) 
\left( Q_n + \frac 12 Q_n^2\right)  
\, \le\,Q_n\left (1+P_n\right)\left(1+\frac{Q_n}{2}\right).
\end{equation}
If we define $c(\gd)$ through
\begin{eqnarray}
  \label{eq:cd}
  c(\gd)=\prod_{k=0}^{\infty}\left(1+\gd 2^{-k/2}\right)\le 
\exp(\delta (2+\sqrt 2)) \, \le \, \frac{21}{20},
\end{eqnarray}
from \eqref{eq:QQ} we directly obtain that,
as long as $Q_n\le 3 Q_0$ and $n\le n_\gd$,
\begin{eqnarray}
  Q_n\le Q_0 \left(1+(3/2)Q_0\right)^n\,\prod_{k=0}^{n-1}(1+P_n)\le
 c(\gd)\,Q_0\, e^{(3/2)Q_0\,n}.
\end{eqnarray}
It is then immediate to check, using \eqref{eq:Q0}, that $Q_{n_\gd}\le 3 Q_0$ for $\gb $ small.

But, as already exploited in \eqref{eq:superlinear},
$Q_{n+1}/Q_n \le 3$ for every $n$ such that $Q_n \le 1$,
so that  $Q_{n_\gd +n } \le  4\gb^2 3^{n} \le 1$ for 
$n\le n_1:=  \log_3 (1/(4\gb^2))-1 $.
But for such values of $n$
\begin{equation}
P_{n_\gd+n} \, \ge \,  \gd 2^{(n-1)/2},
\end{equation}
so that we directly see that $P_{n_\gd+n_1}$ diverges as $\gb$
tends to zero, and therefore $\Rav{n_\gd+n_1}$, can be made large for
$\gb$ small, while $Q_{n_\gd+n_1}$, that is the ratio between the
variance of $R_{n_\gd+n_1}$ and $\Rav{n_\gd+n_1}^2$ is bounded by $1$.  By
exploiting $R_n \ge (B-1)/B$ for $n \ge 1$ and using Chebyshev
inequality it is now straightforward to see that $\langle \log
(R_{n_\gd+n_1}/B)\rangle >0$ and by \eqref{eq:encadre} (or,
equivalently, \eqref{eq:nonincrease}) we have $\tf (\gb, h)>0$.

\qed

\section{Free energy upper bounds beyond annealing}

\label{sec:rel}

In this section we introduce our main new idea, which we briefly
sketch here.  In order to show that the free energy vanishes for $h$
larger but 
close to $h_c(0)$, we take the system at the
$n$-th step of the iteration, for some $n=n(\gb)$ that scales suitably
with $\gb$ (in particular, $n(\gb)$ diverges for $\gb\to0$) and we
modify (via a tilting) the distribution $\bbP$ of the disorder.  If
$\alpha>1/2$, it turns out that one can perform such
 tilting so to guarantee on one hand that, under the new law, $R_{n(\gb)}$
is concentrated around $1$, and, on the other hand,  that the two laws are
very close (they have a mutual density close to $1$).  This in turn
implies that $R_{n(\gb)}$ is concentrated around $1$ also under the
original law $\bbP$, and the conclusion that $\tf(\gb,h)=0$ follows
then via the fact that if  some non-integer moment (of order
smaller than $1$) of $R_{n_0}-1$ is sufficiently small for some
integer $n_0$, then it remains so for every $n\ge n_0$ (cf.
Proposition \ref{th:fractmom}).

\subsection{Fractional moment bounds}
The following result says that if $R_{n_0}$ is sufficiently
concentrated around $1$ for some $n_0\ge0$, then it remains
concentrated for every $n>n_0$ and the free energy vanishes. In other
words, we establish a {\sl finite-volume condition for delocalization}.

\medskip

\begin{proposition}
  \label{th:fractmom} Let $B>2$ and $(\beta,h)$ be given. 
Assume that there exists $n_0\ge0$ and $(\log 2/\log B)<\gamma<1$ such that 
$\langle ([R_{n_0}-1]^+)^\gamma\rangle<B^\gamma-2$. Then, $\tf(\beta,h)=0$.
\end{proposition}

{\sl Proof of Proposition \ref{th:fractmom}.} We rewrite \eqref{eq:basic} as
\begin{equation}
\label{eq:basic2}
R_{n+1}-1 \, =\, \frac1B \left[\left( R_n^{(1)}-1 \right)
  \left( R_n^{(2)}-1 \right) +\left( R_n^{(1)}-1 \right)+
  \left( R_n^{(2)}-1 \right)\right],
\end{equation}
and we use the inequalities $[rs+r+s]^+ \le [r]^+[s]^+ +[r]^+ +
[s]^+$, that holds for $r, s \ge -1$, and $(a+b)^\gamma\le
a^\gamma+b^\gamma$, that holds for $\gamma \in (0,1]$ and $a, b \ge
0$.  If we set $A_n := \langle \left([ R_n -1]^+\right)^\gamma
\rangle$ we have
\begin{equation}
\label{eq:A}
A_{n+1}\, \le \, \frac1{B^\gamma} \left[A_n^2 + 2A_n\right]
\end{equation}
and therefore $A_n\searrow 0$ for $n\to\infty$ under the assumptions of the
Proposition.  Deducing $\tf(\beta,h)=0$ (and actually more than
that) is then immediate:
\begin{eqnarray}
  \langle \log R_n\rangle=\frac1\gamma  \langle \log (R_n)^\gamma\rangle
\le \frac1\gamma\langle \log \left[([R_n-1]^+)^\gamma+1\right]\rangle\le
\frac1\gamma\log (A_n+1)\stackrel{n\to\infty}\searrow 0.
\end{eqnarray}
\qed
\medskip

Proposition \ref{th:fractmom} will be essential in Section
\ref{sec:rel} to prove that, for $B>B_c$, an arbitrarily small
amount of disorder shifts the critical point. Let us also point out
that it implies that, if $\go_1$ is an unbounded random variable, then for any $B>2$ and
$\beta$ sufficiently large quenched and annealed critical points
differ (the analogous result for non-hierarchical pinning models was
proven in \cite[Corollary 3.2]{cf:T_fractmom}):
\medskip

\begin{cor} \label{cor:unbound}
Assume that $\bbP(\go_1>t)>0$ for every $t>0$. Then, for
every $h\in\R$ and $B>2$ there exists $\bar \beta_0<\infty$ such that
$\tf(\beta,h)=0$ for $\beta\ge\bar\beta_0$.
\end{cor}

\medskip
\noindent
{\it Proof of Corollary \ref{cor:unbound}}
Choose some $\gamma\in(\log 2/\log B,1)$.  One has $\lim_{\gb \to \infty}
  R_0=0$ $\bbP(\dd \go)$-a.s. (see \eqref{eq:R0} and note that
$\log M(\beta)/\beta\to\infty$ for $\beta\to+\infty$ under our assumption on
$\go_1$),
  while $\langle \left(([ R_0-1]^+)^\gamma \right)^{1/\gamma} \rangle
  \le 1+ \langle R_0 \rangle=1+\exp(h)$, so $\lim_{\gb \to \infty }A_0
  =0$.  \qed

 \medskip

 \begin{rem}\rm
 Note moreover that if $\exp(h)=B-1$ and we set $X= \exp(\gb \go_1 -\log
 \M (\gb))$ we have (without requiring $\go_1$ unbounded) that $\langle
 ([(B-1) X -1]^+)^\gamma\rangle \stackrel{B\to \infty} \sim B^\gamma
 \langle X^\gamma \rangle$. The right-hand side is smaller than
 $B^\gamma -2$ for $X$ non-degenerate and $B$ large, so that if we
 choose $\gd>0$ such that $\exp( \gd \gamma) \langle X^\gamma \rangle
 <1$ we have
 \begin{equation}
 \langle ([(B-1)\exp(\gd) X -1]^+)^\gamma\rangle \, < \, B^\gamma-2, 
 \end{equation}
 for $B$ sufficiently large. Therefore, by applying
 Proposition~\ref{th:fractmom}, we see that for every $\gb>0$ there
 exists $\gd>0$ such that $\tf\left(\gb, h_c(0)+\delta\right)=0$ for
 $B$ sufficiently large. This observation actually follows also from
 the much more refined Proposition \ref{th:rel} below, which by the way
 says precisely how large $B$ has to be taken: $B>B_c$.
 \end{rem}  

 \begin{rem}\rm
\label{rem:as1}
  It follows from inequality \eqref{eq:A} that, if the assumptions of 
Proposition \ref{th:fractmom} are verified, then $A_n$ actually
vanishes exponentially fast for $n\to\infty$. Therefore, for $\gep>0$ one
has
\begin{eqnarray}
\label{eq:BorCan}
  \bbP(R_n\ge 1+\gep)=\bbP([R_n-1]^+\ge\gep)\le \frac {A_n}{\gep^\gamma}
\end{eqnarray}
and from the Borel-Cantelli lemma follows the almost sure convergence
of $R_n$ to $1$ when we recall that $R_n^{(i)}\ge r_n$ with $r_0=0$
($r_n$ is the solution of the iteration scheme \eqref{eq:r} and
converges to $1$).  
 \end{rem}

\subsection{Upper bounds on the free energy for $B >B_c$}

Here we want to prove the lower bound in \eqref{eq:eps_c}, plus the
fact that $\gep_c(\Delta_0)>0$ whenever $\Delta_0>0$ and $B>B_c$. This
follows from

\begin{proposition}
  \label{th:rel}
Let $B>B_c$. For every $\beta>0$ one has $h_c(\beta)>\log(B-1)$. 
Moreover,
there exists a positive constant $c$ (possibly depending on $B$)
such that for every $0\le \beta\le 1$
\begin{eqnarray}
  \label{eq:rel}
h_c(\beta)-\log(B-1)\ge
c \beta^{2\alpha/(2\alpha-1)}.
\end{eqnarray}
\end{proposition}
Proposition \ref{th:rel} is proven in section \ref{sec:proofrel}, but
first we need to state a couple of technical facts.

\subsection{Auxiliary definitions and lemmas}
\label{sec:aux}
For $\gl \in \R$ and $N\in\N$ let $\bbP_{N, \gl}$ be defined by
\begin{equation}
\frac{
\dd \bbP_{N, \gl}}{\dd \bbP} (x_1,x_2, \ldots)\, =\,
\frac1 {\M(-\gl)^N} 
\exp\left(-\gl \sum_{i=1}^N x_i \right).
\end{equation} 

\smallskip

\begin{lemma}
\label{th:lemP}
There exists $1<C<\infty$ such that for $a\in(0,1)$, $\gd \in (0,
a/C)$ and $N\in\N$  we have
\begin{equation}
\label{eq:lemP}
\bbP_{N, \frac{\gd}{\sqrt{N}}}\left( \frac{\dd \bbP} {\dd \bbP_{N,
\frac{\gd}{\sqrt{N}}}} (\go)\, <\, \exp(-a) \right)\, \le \, C \left(
\frac{\gd}a\right)^2 .
\end{equation}
\end{lemma}
\smallskip

\noindent
{\it Proof of Lemma \ref{th:lemP}.}
We write
\begin{equation}
\label{eq:prtoest}
\bbP_{N, \frac{\gd}{\sqrt{N}}}\left( \frac{\dd \bbP} {\dd \bbP_{N,
\frac{\gd}{\sqrt{N}}}} (\go)\, <\, \exp(-a) \right)\, =\, \bbP_{N,
\frac\gd{\sqrt{N}}}\left( \gd \frac{\sum_{i=1}^{N} \go_i}{\sqrt{N}} +N \log
\M \left(-\frac{\gd }{\sqrt{N}}\right) \, < \, -a \right).
\end{equation}
Since all exponential moments of $\go_1$ are assumed to be finite, one has
\begin{eqnarray}
 0\ge \log \M(-\gl)-\gl\frac{\dd}{\dd \lambda}\left[\log \M(-\lambda)\right]\ge
 -\frac C2\gl^2
\end{eqnarray}
for some $1<C<\infty$ and $0\le\gl\le 1$ (the first inequality is due
 to convexity of $\gl\mapsto\log\M(-\gl)$). Note also that
$$
\bbE_{N, \gl}(\go_1)=
-\frac{\dd}{\dd \lambda}\left[\log \M(-\lambda)\right].
$$ 
Therefore, the right-hand side of \eqref{eq:prtoest}
is bounded above by
\begin{equation}
\bbP_{N, \gd/\sqrt{N}}\left(
 \frac{\sum_{i=1}^{N} \go_i}{\sqrt{N}} - \bbE_{N, \gd/\sqrt{N}}
\left[\frac{\sum_{i=1}^{N} \go_i}{\sqrt{N}}
\right] \, \, <- \frac a{2\gd}
\right) \, \le \,\frac{4\delta^2}{a^2}\bbE_{N,\delta/\sqrt N}\left(
\go_1-\bbE_{N,\delta/\sqrt N}(\go_1)\right)^2
\end{equation}
where we have used Chebyshev's inequality and the fact that, under the
assumptions we made, $(a/\delta)-(C/2)\delta >a/(2\delta)$.  The proof
of \eqref{eq:lemP} is then concluded by observing that the variance of
$\go_1$ under $\bbP_{N,\gl}$ is $\dd^2/\dd \gl^2 \log \M(-\gl)$, which
is bounded uniformly for $0\le\gl\le 1$.  \qed
\bigskip

We define the sequence  $\{a_n\}_{n=0, 1, \ldots}$ by setting 
$a_0=a>0$ and $a_{n+1}= f(a_n)$ with
\begin{equation}
\label{eq:rec-inv}
f(x)\, :=\, \sqrt{B x+(B-1)^2} -(B-1). 
\end{equation}
We define also the sequence 
$\{b_n\}_{n=0,1, \ldots}$ by setting $b_0=b \in (-(B-2), 0)$ 
and $b_{n+1}= f(b_n)$.
Note that $a_n= g(a_{n+1})$ and $b_n= g(b_{n+1})$
for $g(x)=(2(B-1)x+x^2)/B$.

\smallskip

\begin{lemma} 
\label{th:ptlem}
There exist two constants $G_a>0$ et $H_b>0$ such that   for
$n \to \infty$
\begin{equation}
\label{eq:ptlem}
a_n\sim G_a \left(\frac{B}{2(B-1)}\right)^n\, =\, G_a 2^{-\ga n} 
\  \ \text{ and }\ \  b_n\sim -H_b \left(\frac{B}{2(B-1)}\right)^n
\, =\, -H_b 2^{-\ga n}
.
\end{equation}
Moreover, $G_a\stackrel{a\to0}\sim a$ and $H_b\stackrel{b\to0}\sim |b|$.
\end{lemma}
\smallskip

\noindent
{\it Proof of Lemma \ref{th:ptlem}.}
In order to lighten the proof we put $s:=B/(2(B-1))$ and we observe that
$0<s<1$ since $B>2$. The function
 $f(\cdot)$ is concave
  and  $f^\prime (0)=s$, so  we have
\begin{equation}
 \label{eq:pop}
a_n\le a\, s^n:
\end{equation}
$a_n$ vanishes exponentially fast.  Moreover, 
\begin{eqnarray}
\label{eq:monoton}
  \frac{a_n}{s^n}=\frac{a_{n-1}}{s^{n-1}}\frac1{1+a_n/(2(B-1))}\ge
\frac{a_{n-1}}{s^{n-1}}\frac1{1+a s^n/(2(B-1))}
\end{eqnarray}
so that for every $n>0$
\begin{eqnarray}
\label{eq:Ca2}
\frac{a_n}{s^n}\ge a\, \prod_{\ell=1}^\infty\frac1{1+a s^\ell/(2(B-1))}>0.
\end{eqnarray}
From \eqref{eq:monoton} we see that $a_n\,s^{-n}$ is monotone increasing in 
$n$, so that the first statement in \eqref{eq:ptlem} holds with
$G_a\in (0,a)$ from  \eqref{eq:pop} and \eqref{eq:Ca2}.
The fact that $G_a\sim a$ for $a\to0$ follows from the fact that
the product in  \eqref{eq:Ca2} converges to $1$ in this limit.
\smallskip

The second relation is proven in a similar way. Since $b_n<0$ for every
$n$, one has first of all
\begin{eqnarray}
\label{eq:Cb}
  \frac{b_n}{s^n}=\frac{b_{n-1}}{s^{n-1}}\frac1{1+b_n/(2(B-1))}<
\frac{b_{n-1}}{s^{n-1}}.
\end{eqnarray}
Moreover, since $\vert b_n \vert $ decreases to zero and $f(x)\ge
c_1(b) x$ for $b\le x\le0$ for some $c_1(b)<1$ if $b>-(B-2)$, one
sees that $|b_n|$ actually vanishes exponentially fast.  Therefore,
from \eqref{eq:Cb}
\begin{eqnarray}
    \frac{b_n}{s^n}\ge \frac{b_{n-1}}{s^{n-1}}\frac1{1-c_2(b)\,c_1(b)^n}
\ge b\prod_{\ell=1}^\infty\frac1{1-c_2(b)\,c_1(b)^\ell}.
\end{eqnarray} 
One has then the second statement of \eqref{eq:ptlem} with
$H_b\in(|b|,\infty)$.

\qed

\subsection{Proof of Proposition \ref{th:rel}}
\label{sec:proofrel}
In this proof $C_i$ denotes a constant depending only on $\gb_0$ and (possibly)
on $B$.
Recall that the exponent $\alpha$ defined in \eqref{eq:alpha}
satisfies $1/2<\ga<1$ for $B>B_c$. Fix $\beta_0>0$, let
$0<\beta<\beta_0$ and choose $h=h(\beta)$ such that
\begin{eqnarray}
  \med{R_0}=(B-1)+\eta\gb^{\frac{2\ga}{2\ga-1}},
\end{eqnarray}
where $\eta>0$ will be chosen sufficiently small and independent of
$\beta$ later.  Call $n_0:=n_0(\eta,\gb)$ the integer such that
\begin{eqnarray}
\label{eq:prima}
  \med{R_{n_0}}\le B\le \med{R_{n_0+1}},
\end{eqnarray}
{\sl i.e.}, $P_{n_0}\le 1\le P_{n_0+1}$.  
Note that
$n_0(\eta, \gb)$ becomes larger and larger
as $\gb\searrow 0$: this can be quantified since 
from  \eqref{eq:P} one sees that
$a_n:=P_{n_0-n}$ satisfies for $0\le n< n_0$ the iteration
$a_{n+1}=f(a_n)$ introduced in \S~\ref{sec:aux},
 and therefore it
follows from Lemma \ref{th:ptlem} that 
\begin{equation}
\label{eq:n0}
  \left|n_0(\eta,\gb)-\log\left(\eta^{-1}\gb^{-\frac{2\ga}{2\ga-1}}\right)
/(\alpha\log 2)\right|\, \le\,  C_1
\end{equation}
for every $0<\eta<1/C_1$ and $\beta\in[0,\beta_0]$.  With the 
notations of Section  \ref{sec:aux}, let $\tilde
\bbP:=\bbP_{2^{n_0},\delta 2^{-n_0/2}}$, where $\delta:=\delta(\eta)$ will
be chosen suitably small later.
Note that, with $\gl:=\delta 2^{-n_0/2}$, one has from \eqref{eq:n0}
\begin{eqnarray}
\label{eq:lpetit}
  \frac1{C_2}\delta\,\eta^{1/(2\alpha)}\gb^{1/(2\ga-1)} \le \gl\le
C_2\delta\,\eta^{1/(2\alpha)}  \gb^{1/(2\ga-1)}.
\end{eqnarray}
In particular, since $\alpha<1$,  if $\eta$ is small enough
then $\gl\le \beta$ uniformly for $\beta\le \beta_0$. Observe also that
\begin{eqnarray}
\label{eq:R0'}
  \tilde \bbE(R_0)=\med{R_0}\frac{\M(\gb-\gl)}{\M(\gb)\M(-\gl)},
\end{eqnarray}
and call $\phi(\cdot):=\log \M(\cdot)$. Since $\phi(\cdot)$ is strictly convex,
one has
\begin{eqnarray}
  \phi(\gb-\gl)-\phi(\gb)-\phi(-\gl)=-\int_{-\gl}^0 \dd x\int_0^\gb\dd y
\,\phi''(x+y)\in \left(- \frac{\gl\gb}{C_3},-C_3\gl\gb\right),
\end{eqnarray}
for some $C_3>0$, uniformly in $\beta\le \beta_0$ and $0\le \lambda\le
\beta$  and, thanks to \eqref{eq:lpetit}, if $\eta$ is chosen sufficiently
small,
\begin{eqnarray}
1-\frac{\beta\gl}{C_4} \le \frac{\M(\gb-\gl)}{\M(\gb)\M(-\gl)}\le 1- C_4\beta\gl.
\end{eqnarray}
Therefore, from  \eqref{eq:R0'} and \eqref{eq:lpetit} and choosing
\begin{eqnarray}
  \label{eq:deltall}
\eta^{1-1/(2\ga)}\ll \delta(\eta)\ll 1,
\end{eqnarray}
(which is possible with $\eta$ small since $\alpha>1/2$) one has
\begin{eqnarray}
-C_5^{-1}\,\delta(\eta)\,
    \eta^{1/(2\alpha)}\beta^{\frac{2\alpha}{2\alpha-1}}< \tilde
    \bbE(R_0)-(B-1)\le-C_5\delta(\eta)\,
    \eta^{1/(2\alpha)}\beta^{\frac{2\alpha}{2\alpha-1}},
\end{eqnarray}
always uniformly in $\beta\le \beta_0$.

Since $b_n:=\tilde\bbE (R_{n_0-n})-(B-1)$ satisfies the recursion
$b_{n+1}=f(b_n)$, from the second statement of  \eqref{eq:ptlem} if
follows that
\begin{equation}
  \tilde \bbE R_{n_1}\, \le\,  \frac {B}2,
\end{equation}
for some integer $n_1:=n_1(\eta,\beta)$ satisfying
\begin{equation}
  \label{eq:n1} 
n_1\, \le\, 
\log\left(
\delta(\eta)^{-1}\eta^{-1/(2\alpha)}\beta^{-2\ga/(2\ga-1)}
\right)
\Big /
\left(
\ga \log 2\right)
+C_6.
\end{equation}
It is immediate to see that $n_0(\eta,\beta)-n_1(\eta,\beta)$ gets
 large (uniformly in $\beta$) for $\eta$ small, if condition
\eqref{eq:deltall} is satisfied. Therefore, since the fixed point $1$
of the iteration for $\tilde \bbE R_{n}$
is attractive, one has that
\begin{eqnarray}
   \tilde \bbE R_{n_0}\le 1+r_1(\eta),
\end{eqnarray}
(here and in the following, $r_i(\eta)$ with $i\in\N$ denotes a positive
quantity which vanishes for $\eta\searrow0$, uniformly in $\gb\le \gb_0$.)
On the other hand, 
one has deterministically 
\begin{eqnarray}
\lim_{n\to\infty}  [1-R_n]^+=0,
\end{eqnarray}
as one sees immediately comparing the evolution of $R_n$ with that
obtained setting $R_0^{(i)}=0$ for every $i$. In particular,
$R_{n_0}\ge 1-r_2(\eta)$.  An
application of Markov's inequality gives
\begin{eqnarray}
  \tilde \bbP(R_{n_0}\ge 1+r_3(\eta))\le r_3(\eta).
\end{eqnarray}
It is immediate to prove that, given a random variable $X$ and two 
mutually absolutely continuous laws $\bbP$ and $\tilde \bbP$, one has
for every $x,y>0$
\begin{eqnarray}
  \bbP(X\le 1+x)\ge e^{-y}\left[\tilde\bbP(X\le 1+x)-\tilde \bbP\left(
\frac{\dd\bbP}{\dd\tilde\bbP}\le e^{-y}\right)\right].
\end{eqnarray}
Applying this to the case $X=R_{n_0}$ and using Lemma \ref{th:lemP} with 
$r_4(\eta)>C\delta(\eta)$ gives
\begin{eqnarray}
  \bbP(R_{n_0}\le 1+r_3(\eta))
\ge e^{-r_4(\eta)}\left[1-r_3(\eta)-C\left(\frac{\delta(\eta)}{r_4(\eta)}
\right)^2\right].
\end{eqnarray}
In particular,
choosing 
\begin{equation}
\delta(\eta)\ll r_4(\eta)\ll1,
\end{equation}
one has 
\begin{eqnarray}
\label{eq:seconda}
  \bbP(R_{n_0}\le 1+r_3(\eta))\ge 1-r_5(\eta),
\end{eqnarray}
and we emphasize that this inequality holds uniformly in $\gb\le
\gb_0$.

At this point  \eqref{eq:rel} is essentially proven: choose some
$\gamma\in(\log 2/\log B,1)$ and observe that
\begin{equation}
\begin{split}
 \med{\left([R_{n_0}-1]^+\right)^\gamma} \,&\le\, 
r_3(\eta)^\gamma+\left(\bbE[R_{n_0}-1]^+\right)^\gamma
\left(\bbP(R_{n_0}\ge 1+r_3(\eta))\right)^{1-\gamma}\\
\label{eq:uff}
&\le\, r_3(\eta)^\gamma+B^\gamma r_5(\eta)^{1-\gamma},
\end{split}
\end{equation}
where in the first inequality we have used H\"older inequality and in the
second one we have used  \eqref{eq:prima} and \eqref{eq:seconda}.  Finally, we
remark that the quantity in \eqref{eq:uff} can be made smaller than
$B^\gamma-2$ choosing $\eta$ small enough.  At this point, we can
apply Proposition \ref{th:fractmom} to deduce that $\tf(\beta,h)=0$
for $h=\log (B-1)+\eta \beta^{2\ga/(2\ga-1)}$ with $\eta$ small but
finite, which proves  \eqref{eq:rel}.

We complete the proof by observing 
 that $h_c(\beta)>\log (B-1)$ for every $\beta>0$ follows from
the arbitrariness of $\beta_0$.  \qed

\section{The delocalized phase}
\label{sec:deloc}
Here we prove Theorem \ref{th:paths} using the
representation \eqref{eq:EBgen}, given in Appendix \ref{sec:stand}, for $R_n$.
With reference to \eqref{eq:EBgen}, let us observe that
\begin{equation}
\label{eq:fewcontacts}
\lim_{n\to\infty} p(n,\emptyset)\,=\,1,
 \end{equation}
which is just a way to interpret 
\begin{equation}
 \lim_{n\to\infty}r_n\,=\,1
\end{equation}
when $r_0=0$, that follows directly from \eqref{eq:r}.



\medskip

Le $\gep$ be fixed and consider $h<h_c(\gb)$.  Let $\bar R_n$ be the
partition function which corresponds to $h_c(\gb)$ and $R_n$ the one
that corresponds to $h$.  We can find $K$ large enough such that
\begin{equation}\label{eq:barR}
 \bbP \left(\bar R_n \ge K \right)\le \gep/2 \quad \text{ for all } n\ge 1.
\end{equation}
This follows from the fact that 
$\bar R_n\ge (B-1)/B$, and from \eqref{eq:nondecrease}.
We define $C:=(\log(2K/\gep))/{(h_c-h)}$ and we write, using \eqref{eq:EBgen},
\begin{multline}
  R_n\, =\,p(n,\emptyset)+
  \sumtwo{\mathcal I \subset\{1,\dots,2^n\}}{1\le |\mathcal I|\le C}
p(n,\mathcal I) \exp\left(\sum_{i\in \mathcal I}(\gb\go_i-\log M(\gb)+h)\right)
\\
  +\sumtwo{\mathcal I \subset\{1,\dots,2^n\}}{|\mathcal I|
    >C}p(n,\mathcal I) \exp\left(\sum_{i\in \mathcal I}(\gb\go_i-\log
    M(\gb)+h)\right)\, =:\, T_1+T_2+T_3.
\end{multline}
$T_1$ is smaller than $1$ and 
\begin{equation}
T_3\, \le \, 
 \exp\left(-C(h_c-h)\right)\bar R_n,
\end{equation}
so that $T_3 \le \gep/2$ with probability greater than $(1-\gep/2)$
({\sl cf.} \eqref{eq:barR}) for all $n$.  As for $T_2$, its easy to
compute and bound its expectation
\begin{equation}
  \left\langle \sumtwo{\mathcal I \subset\{1,\dots,2^n\}}
    {1\le |\mathcal I|\le C}p(n,\mathcal I) \exp\left(\sum_{i\in \mathcal I}(\gb\go_i-\log M(\gb)+h)\right)\right\rangle \, \le\,  \exp(Ch)[1-p(n,\emptyset)],
\end{equation}
and \eqref{eq:fewcontacts} tells us that the right-hand side tends to
zero when $n$ goes to infinity.  In particular we can find $N$
(depending on $C$) such that for all $n\ge N$ we have
\begin{equation}
  \left\langle\sumtwo{\mathcal I \subset\{1,\dots,2^n\}}
{1\le |\mathcal I|\le C}p(n,\mathcal I) \exp\left(\sum_{i\in \mathcal I}(\gb\go_i-\log M(\gb)+h)\right)\right\rangle\, \le\,  \gep^2/4.
\end{equation}
Then for $n\ge N$ we have
$
 \bbP(T_2\ge \gep/2)\le \gep/2$.
Altogether we have
\begin{equation}
 \bbP(R_n\ge 1+\gep)\,\le\,  \gep,
\end{equation}
and since $R_n$ is bounded from below by $p(n,\emptyset)$ which tends to $1$,
the proof is complete.
\qed

\begin{subappendices}
\section{Existence of the free energy and annealed system estimates}

\label{sec:stand}

\subsection{Proof of Theorem \ref{th:F}}


Since the basic induction \eqref{eq:basic} gives $ R_n\ge (B-1)/B $
for every $n\ge 1$, one has
\begin{align}
\frac{R_{n+1}}{B}&\ge \frac{R_n^{(1)}}{B}\frac{R_n^{(2)}}{B} \label{eq:ge}
\end{align}
and
\begin{align}
R_{n+1} &\le \frac{R_n^{(1)}R_n^{(2)}}{B}+ \frac{B}{B-1}R_n^{(1)}R_n^{(2)}
\notag
\end{align}
so  that
\begin{align}
  (K_B R_{n+1}) \le (K_BR_n^{(1)})(K_BR_n^{(2)}) \quad \text{with}
  \quad K_B=\frac{B^2+B-1}{B(B-1)} \label{eq:le}.
\end{align}
Taking the logarithm of \eqref{eq:ge} and \eqref{eq:le}, we get that
\begin{equation}
\label{eq:nondecrease}
\left\{2^{-n}\E\big[\log(R_n/B)\big]\right\}_{n=1, 2, \ldots}
\ \text{
 is non-decreasing,} 
 \end{equation}
 while
 \begin{equation}
\label{eq:nonincrease}
\left\{2^{-n}\E\big[\log(K_B R_n)\big]\right\}_{n=1, 2, \ldots} \
\text{ is non-increasing},
\end{equation}
so that both sequences are converging to the same limit
\begin{equation}
\tf(\beta,h)=\lim_{n\rightarrow\infty}2^{-n}\med{\log R_n}
\end{equation}
and \eqref{eq:encadre} immediately follows.  It remains to be proven that
the limit of $2^{-n}\log R_n$ exists $\bbP(\dd \go)$--almost surely and  $\bbL^1(\dd \bbP)$ convergence.
Fixing some $k\ge1$ and iterating \eqref{eq:ge} one obtains for $n>k$
\begin{equation}
  2^{-n}\log (R_{n}/B)\ge 2^{-k}\big(2^{k-n}\sum_{i=1}^{2^{n-k}}\log
  (R_k^{(i)}/B)\big).
\end{equation}
Using the strong law of large numbers in the right-hand side, we get
\begin{equation}
  \liminf_{n\to\infty} 2^{-n} \log (R_n/B) \ge 2^{-k}\med{\log
  (R_k/B)} \quad \bbP(\dd \go)- \text{a.s.}.
\end{equation}
Hence taking the limit for $k\rightarrow\infty$ in the right-hand side
again we obtain
\begin{equation}
  \liminf_{n\to\infty} 2^{-n} \log R_n=\liminf_{n\to\infty} 2^{-n}
  \log (R_n/B) \ge \tf(\beta,h) \quad \bbP(\dd \go)- \text{a.s.}
\end{equation}
Doing the same computations with \eqref{eq:le} we obtain
\begin{equation}
  \limsup_{n\to\infty} 2^{-n} \log R_n=\limsup_{n\to\infty} 2^{-n} \log (K_B R_n)
  \le \tf(\beta,h) \quad \bbP(\dd \go)- \text{a.s.}.
\end{equation}
This ends the proof for the almost sure convergence. The proof of the
$\bbL^1(\dd \bbP)$ convergence is also fairly standard, and we leave it to
the reader.

The fact that $\tf(\beta,\cdot)$ is non-decreasing follows from the
fact that the same holds for $R_n(\beta,\cdot)$, and this is easily
proved by induction on $n$. Convexity of $(\gb, h) \mapsto \tf (\gb, h
+ \log \M(\gb))$ is immediate from \eqref{eq:DHV} (hence for $B=2,3,
\ldots$). But \eqref{eq:DHV} can be easily generalized to every $B>1$:
this follows by observing that from \eqref{eq:basic} and \eqref{eq:R0}
one has that
\begin{equation}
\label{eq:EBgen}
 R_n\, =\, \sum_{\mathcal{I} \subset\{1,\dots,2^n\}}p(n,\mathcal I) \exp\left(\sum_{i\in \mathcal I}(\gb\go_i-\log M(\gb)+h)\right).
\end{equation}
for suitable positive values $p(n, \mathcal I)$, which depend on $B$:
by setting $\gb=h=0$ we see that $ \sum_{\mathcal{I}}p(n,\mathcal
I)=1$ and hence $R_n$ can be cast in the form of the expectation of a
Boltzmann factor, like \eqref{eq:DHV}. This yields the desired
convexity.  \qed

\smallskip

\begin{rem}
\label{rem:5} 
\rm
Another consequence of \eqref{eq:EBgen} is that
 $\tf(\gb, h+ \log \M( \gb))\ge \tf (0,  h)$ \cite[Ch.~5, Prop.~5.1]{cf:Book}.
\end{rem}


\medskip

\subsection{Proof of Theorem \ref{th:pure}} 
When $\gb=0$  the iteration \eqref{eq:basic} reads
\begin{equation}
R_{n+1}=\frac{R_n^2+(B-1)}{B}.
\end{equation}
A quick study of the function $x\mapsto [x^2+(B-1)]/B$, gives that
$R_n\stackrel{n\rightarrow\infty}{\rightarrow}\infty$ if and only if
$R_0>(B-1)$. Initial conditions $R_0<B-1$ are attracted by the stable
fixed point $1$, while the fixed point $(B-1)$ is unstable.  The
inequality \eqref{eq:ge} guaranties that $\tf(0,h)>0$ when $R_N>B$ for
some $N$. This immediately shows that that $h_c(0)=\log(B-1)$.


Next we prove \eqref{eq:alpha0}, {\sl i.e.}, that there exists a constant $C$
such that 
\begin{equation}
\label{eq:modif}
\frac{1}{C}\gep^{1/\ga}\, \le\,  \hat{\tf}(0,\gep) \, \le\,  C \gep^{1/\ga} 
\end{equation}
for all $\gep\in(0,1)$.  To that purpose take $a:=a_0$ such that
$\hat{\tf}(0,a)=1$ 
(this is possible because of the  convexity  of
$\tf(\gb, \cdot+ \log \M(\gb))$ we obtain both continuity and 
$\lim_{a \to \infty}\hat{\tf}(0,a)=\infty$) 
and note that  the sequence $\{a_n\}_{n\ge0}$ defined
just before Lemma \ref{th:ptlem} 
is such that $2\,\hat{\tf}(0,a_{n+1})=\hat{\tf}(0,a_{n})$, so that
$\hat \tf(0,a_{n+1})=2^{-n}$.
Thanks to Lemma \ref{th:ptlem} we have that along this sequence
\begin{equation}
\hat{\tf}(0,a_n)\sim  2G_a^{-1/\ga} a_n^{1/\ga}. \label{eq:sim}
\end{equation}
Let $K_a$ be such that $a_n\le K_a a_{n+1}$ for all $n$, and $c_a$
such that $c_a^{-1}a_n^{1/\ga}\le\hat{\tf}(0,a_n)\le
c_a\,a_n^{1/\ga}$.  Then, for all $n$ and all
$\gep\in[a_{n+1},a_{n}]$, since $\hat{\tf}(0,\cdot)$ is increasing we
have
\begin{equation}
\begin{split}
  \hat{\tf}(0,\gep)&\ge \hat{\tf}(0,a_{n+1}) 
\ge c_a^{-1}a_{n+1}^{1/\ga}\ge c_a^{-1}K_a^{-1/\ga}\gep^{1/\ga},\\
  \hat{\tf}(0,\gep)&\le \hat{\tf}(0,a_{n})\le c_a a_n^{1/\ga}\le
  c_aK_a^{1/\ga}\gep^{1/\ga}.
  \end{split}
\end{equation}
Finally, the analyticity of $\tf(0, \cdot)$ on $(h_c, \infty)$  follows 
for example from \cite[Lemma~4.1]{cf:OvKr}.

\qed

\subsection{About   models with $B\le 2$}
\label{sec:Bsmallerthan2}

We have chosen to work with the model \eqref{eq:R}, with positive
initial data and $B>2$, because this is the case that is directly
related to pinning models and because in this framework we had the
precise aim of proving the physical conjectures formulated in
\cite{cf:DHV}. But of course the model is well defined for all $B\neq
0$ and in view of the direct link with the logistic map  $z \mapsto A
z(1-z)$, {\sl cf.}  \eqref{eq:logistic}, also the case $B\le 2$ appears to be
intriguing. Recall that $A=2(B-1)/B$ and note that $A \in (1,2)$ if
$B\in (2, \infty)$.  What we want to point out here is mainly that the
case of \eqref{eq:R} with positive initial data and $B \in (1,2)$,
{\sl i.e.} $A\in (0,1)$, is already contained in our analysis.  This
is simply the fact that there is a duality transformation relating
this new framework to the one we have considered. 
Namely, if we let $B\in(1,2)$ and we set $\hat R_n := R_n/(B-1)$, then
$\hat R_n$ satisfies \eqref{eq:R} with $B$ replaced by $\hat B:=
B/(B-1) >2$. Of course the fixed points of $x\mapsto (x^2 +\hat B -1)/
\hat B$ are again $1$ (stable) and $\hat B -1$ (unstable). This
transformation allows us to generalize immediately all the theorems we
have proven in the obvious way, in particular the marginal case
corresponds to $\hat B=\hat B_c:=2+\sqrt{2}$, {\sl i.e.}, $B= \sqrt{2}$
and in the irrelevant case the condition $\gD_0< B^2-2(B-1)^2$ now
reads $\gD_0< \hat B^2-2(\hat B-1)^2$.
  
 This discussion leaves open the cases $B=1$ and $B=2$ to
 which we cannot apply directly our theorems, but:
 \begin{enumerate}
 \item If $ B=1$ the model is exactly solvable and $R_n$ is equal to
   the product of $2^n$ positive IID random variables distributed like
   $R_0$, so $\tf (\gb, h)= h-\log \M (\gb)$. The model in this case
   is a bit anomalous, since the stable fixed point is $0$ and
   therefore the free energy can be negative and no phase transition
   is present (this appears to be the analogue of the non-hierarchical
   case with inter-arrival probabilities that decay exponentially fast
   \cite[Ch.~1, Sec.~9]{cf:Book}).
\item If $B=2$ then, with reference to \eqref{eq:r}, $r_n \nearrow
  \infty$ if $r_0>1$ and $r_n \nearrow 1$ if $r_0<1$. The basic
  results like Thorem~\ref{th:F} are quickly generalized to cover this
  case. Only slightly more involved is the generalization of the other
  results, notably Theorem~\ref{th:eps_c}$(1)$.  In fact we cannot
  apply directly our results because the iteration for $P_n$, that is
  $\Rav{n} -1$, reads $P_{n+1}=P_n +(P_n^2 /2)$ ({\sl cf.}
  \eqref{eq:P}) so that the growth of $P_n$, for $P_0>0$, is just due
  to the nonlinear term and it is therefore slow as long as $P_n$ is
  small.  However the technique still applies (note in particular that, by 
   \eqref{eq:Q} and \eqref{eq:QboundP}, the variance of $R_n$
  decreases exponentially if $\gD_0<2$ as long as $P_n$ is sufficiently
  small)  and along this line one shows that the disorder is
  irrelevant, at least as long as $\gD_0 <2$.
 \end{enumerate} 

\medskip

If we now let $B $ run from $1$ to infinity, we simply conclude that
the disorder is irrelevant if $B \in (\sqrt{2}, 2+\sqrt{2})$, 
and it is instead relevant in $B \in (1,
\sqrt{2}) \cup (2+\sqrt{2}, \infty)$.  In the case $B=1$ (and, by
duality, $B=\infty$) there is no phase transition.

\medskip

Finally, a word about the models with $B<1$.  Various cases should be
distinguished: going back to the logistic map, we easily see that
playing on the values of $B$ one can obtain values of $\vert A\vert $
larger than $2$ and the very rich behavior of the logistic map sets
in \cite{cf:Beardon}: non monotone convergence to the fixed point, oscillations in a
finite set of points, chaotic behavior, unbounded trajectories for any
initial value.  It appears that it is still  possible to generalize
our approach to deal with some of these cases, but this would lead us
far from our original aim.
Moreover, for $B<1$ the property of positivity of $R_n$, and
therefore its statistical mechanics interpretation as a partition
function, is lost.
\end{subappendices}

\bigskip

\noindent
{\bf Note added in proof.}
After this work was completed, a number of results have been proven by
developing further the idea set forth here, solving some of questions
raised at the end of Section~\ref{sec:pinning}. First of all we were
able (in collaboration with B. Derrida \cite{cf:DGLT}) to extend the
main idea of this work to the non hierarchical set-up and we have
shown that the quenched critical point (of the non-hierarchical model)
is shifted with respect to the annealed value for arbitrarily small
disorder, if $\ga>1/2$ (this result has been sharpened in
\cite{cf:AZ}, taking a different approach).  Then one of us
\cite{cf:Hubert} has been able to show the shift of the critical point
for arbitrarily small disorder for $\ga =1/2$ in a hierarchical with
site disorder (the case considered here is  bond disorder, {\sl
  cf.}  Figure~\ref{fig:hbeta}) by using a location-dependent shift of
the disorder variables in the change of measure argument (in our case,
the shift is the same for each variable). Finally, very recently
\cite{cf:GLT_marg} we have also been able to treat the case $\ga =1/2$
($B=B_c$), both for the hierarchical and non-hierarchical model, by
introducing long range correlations in the auxiliary measure 
$\tilde \bbP$.

\chapter[Hierarchical pinning with site disorder]{Hierarchical pinning model with site disorder: disorder is marginally relevant}\label{HPMDMR}
 \section{The model and the results}

\subsection{A quick survey and some motivations}

A lot of progress have been made recently in the understanding of pinning models (see \cite{cf:Book} for a  survey and particularly Chapter $1$ for a definition of the simplest random walk based model) in particular in comparing quenched and annealed critical point (see \cite{cf:Ken, cf:AZ,cf:DGLT, cf:T_cmp,  cf:T_fractmom}).
But these works do not settle the question of whether the quenched and annealed critical points
coincide or not for the simple random walk based model. Moreover, physicists predictions on such an issue do not agree (see for example \cite{cf:DHV, cf:FLNO}). The reason why the question is not settled 
 for the random walk based model lies in the exponent ($3/2$) of the 
 law of the first return to zero: for smaller (respectively larger) values of the exponent
 a heuristic argument ({\sl Harris criterion}, \cite{cf:DHV, cf:FLNO}), 
 made rigorous in \cite{cf:Ken, cf:AZ, cf:DGLT, cf:T_cmp}, tells 
 us that annealed and quenched critical points coincide at high
 temperature (respectively, they differ at all temperatures). The first scenario
 goes under the name of irrelevant disorder regime and the second as 
 relevant disorder regime. The irrelevant disorder regime is also characterized
 by the fact that quenched and annealed critical exponents coincide
 \cite{cf:Ken, cf:T_cmp}, while for the relevant disorder they are different
 \cite{cf:GT_cmp}. The terminology {\sl relevant/irrelevant} comes from
 the renormalization group arguments \cite{cf:DHV, cf:FLNO} leading to the Harris criterion
 and in such a context the {\sl undecided} case is called {\sl marginal}
 and, in general, it poses a very challenging problem even at a heuristic level 
 (see {\sl e.g.} \cite{cf:DG} and references therein).  

Much work has been done on the statistical mechanics on 
a particular class of hierarchical lattices, the {\sl diamond lattices}, 
because of the explicit form of the renormalization group
transformations on such lattices, while often retaining a clear link
with the corresponding non hierarchical lattices \cite{cf:DG}.
A hierarchical model for disordered pinning has been explicitly 
considered in  \cite{cf:DHV} and a rigorous analysis of this model 
has been taken up in  \cite{cf:GLT}, but such a rigorous analysis cannot confirm the
prediction in  \cite{cf:DHV}   that the disorder is relevant also in the marginal 
regime (more precisely, in \cite{cf:DHV} it is claimed that annealed and quenched
critical points differ at marginality).
It should however be pointed out that the model in \cite{cf:DHV,cf:GLT}
is a {\sl bond} disorder model and there is a natural companion to such a model,
that is the one in which the disorder is on the sites.  A priori there is no
particular reason to choose either of the two cases, but, if we take a closer look, the site disorder case is somewhat closer to the non hierarchical
case. The reason is that the Green function (see below) of the
bond model is constant through the lattice, while the Green function of the site model is not
 (this analogy can be pushed further, see Remark \ref{rem:Green}). 
In this paper we analyze the hierarchical pinning model with site disorder  
  and  we  establish disorder relevance in the marginal regime. It seems unlikely that our method can be adapted in a straightforward way to settle the question
for the bond model or for the random walk based model, but one can use it to improve
the result in \cite{cf:DGLT} (we will come back to this point in Remark \ref{remrev} below). This result is a confirmation (although the setup we consider here is slightly different) to the claim made in \cite{cf:DHV} that disorder is relevant at any temperature when the specific heat exponent vanishes.

\subsection{The model}

Let $(D_n)_{n\in\N}$ be the sequence of lattices defined as follow
\begin{itemize}
	\item $D_0$ is made of one single edge linking two points $A$ and $B$.
	\item $D_{n+1}$ is obtained for $D_n$ by replacing each edge by $b$ branches of $s$ edges (with $b$ and $s$ in $\left\{2,3,4,\dots,\right\}$).
\end{itemize}
On $D_n$ we fix one directed path $\sigma $ linking $A$ and $B$ ({\sl the wall}).

Given $\gb>0$, $h\in \R$ and $\{\go_i\}_{i\in\N}$ a sequence of i.i.d.\ random variables (with law $\bbP$ and expectation denoted by $\bbE$) with zero mean, unit variance, and satisfying
\begin{equation}
M(\gb):=\bbE\left[\exp(\gb\go_1)\right]<\infty\quad \text{ for every } \gb>0,
\end{equation} 
one defines the partition function of the system of rank $n$ by
\begin{equation}
\label{eq:part}
R_n\, =\, R_n(\gb,h)\, :=\, \bE_n\left[\exp(H_{n,\go,\gb,h}(S)\right], 
\end{equation}
 where  $\bP_n$ is the uniform probability on all directed path $S=(S_i)_{0 \le i\le s^n}$ on $D_n$ linking $A$ to $B$, and $\bE_n$ the related expectation, and 
 \begin{equation}
 H_{n,\go,\gb,h}(S)\, =\, \sum_{i=1}^{s^n-1} \left[\gb\go_i+h-\log M(\gb)\right]\ind_{\{S_i=\sigma_i\}} .
\end{equation}

\begin{figure}[h]
\begin{center}
\leavevmode
\epsfysize =6.5 cm
\psfragscanon
\psfrag{o1}[c]{{\small$\go_1$}}
\psfrag{o2}[c]{{\small$\go_2$}}
\psfrag{o3}[c]{{\small$\go_3$}}
\psfrag{a}[c]{\normalsize A}
\psfrag{b}[c]{B}
\psfrag{l0}[c]{$D_0$}
\psfrag{l1}[c]{$D_1$}
\psfrag{l2}[c]{$D_2$}
\psfrag{traject}[l]{$S$ {\tiny (a directed path on $D_2$)}}

\epsfbox{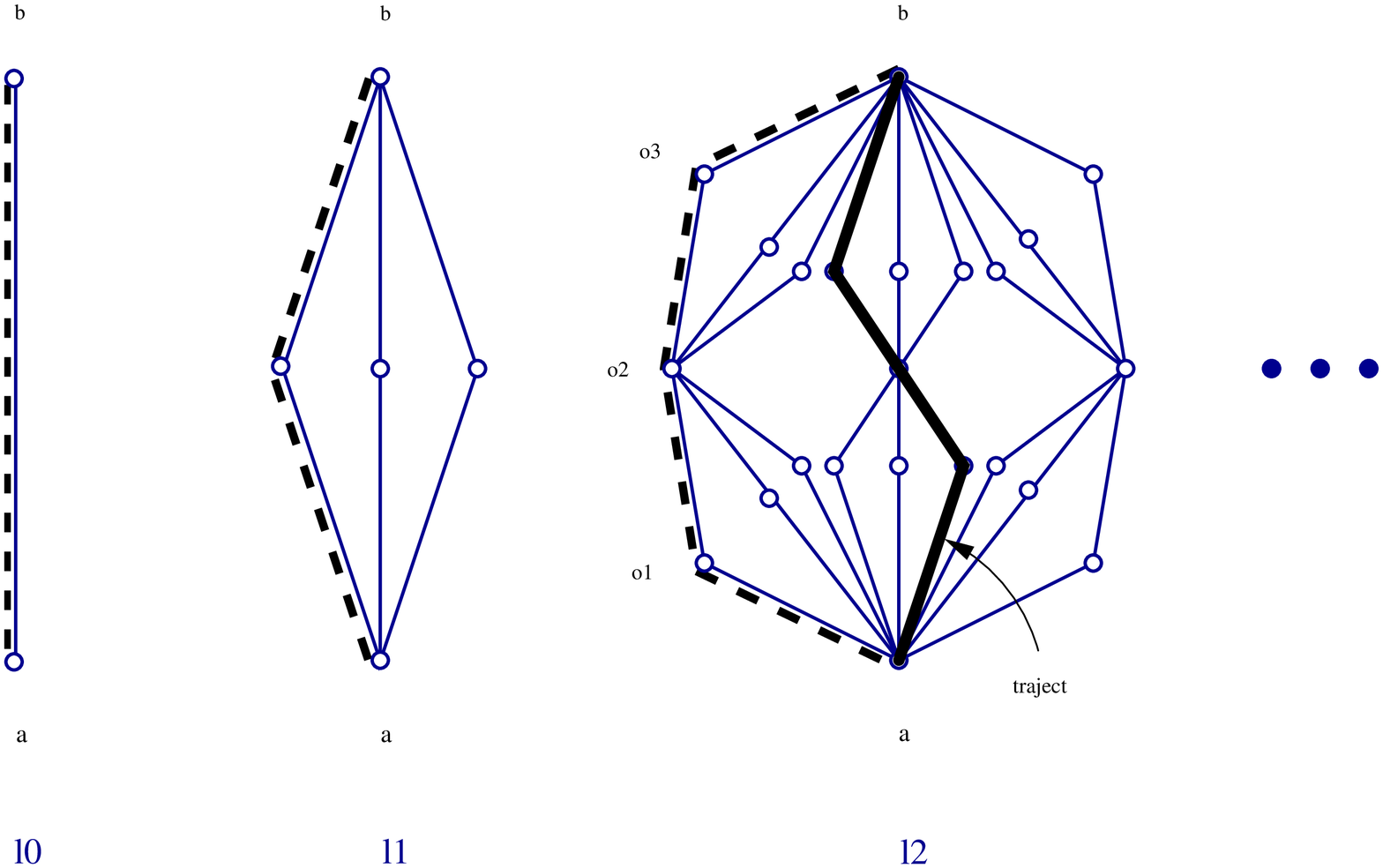}
\end{center}
\caption{\label{fig:Dn} We present here the recursive construction of the first three levels of the hierarchical lattice  $D_n$, for $b=3$, $s=2$. The geometric position of the disordered environment is specified for $D_2$. The law $\bP_n$ is the uniform law over all directed path and   the path $\sigma$
({\sl the wall}) is marked by a dashed line. In the bond disorder case \cite{cf:DHV,cf:GLT} the
hierarchy of lattices is the same, but, with reference to $D_2$, there would be four variables of disorder.}
\end{figure}

\medskip
\begin{rem}
\label{rem:Green}
\rm 
One can directly check that for this model, the site Green function of the directed random walk, that is the probability that a given site is visited by the path, is inhomogeneous (taking value $1$ for the graph extremities $A$ and $B$ and equal to $b^{-i}$ $1\le i\le n$ on the other sites of $D_n$, with $i$ corresponding to the level at which that site has appeared in the hierarchical construction of the lattice, see Figure~\ref{fig:Dn}), and this makes this model similar to the random walk based model
where the Green function decays with a power law with the length of the system
(in the hierarchical context the length of the system is $s^n$). This is not true for the bond model Green function (every bond is visited with probability $b^{-n}$). This inhomogeneity will play a crucial role in the proof, as it will allow us to improve the  method introduced in \cite{cf:GLT}. See Remark~\ref{rem:Green2} for more on this. 
\end{rem}
\medskip

The relatively cumbersome hierarchical construction actually 
boils down to a very simple
 recursion giving the law of the random variable $R_{n+1}$ in terms of the law of $R_n$, for every $n$.
For $s=2$ the recursion is particularly compact:  $R_0:=1$ and  
\begin{equation}
R_n=\frac{R_n^{(1)}A R_n^{(2)}+b-1}{b},
\end{equation}
where $R_n^{(1)}$, $R_n^{(2)}$ and $A$ are
independent random variables with $R_n^{(1)}\stackrel{\cL}=R_n^{(2)}$ and 
\begin{equation}
A \stackrel{\cL}= \exp(\gb\go_1-\log M(\gb)+h).
\end{equation}
For arbitrary $s$ the recursion gets slightly more involved:
\begin{equation}
R_{n+1}=\frac{\prod_{1\le j\le s}R_n^{(j)}\prod_{1\le j\le s-1}A_j+b-1}{b},\label{eq:rec}
\end{equation}
where, once again, all the variables appearing in the right-hand side 
are independent,  $(R_n^{(i)})_{i=1, \ldots, s}$ are also identically
distributed and  $A_j$ has the same law as $A$ for every $j$. 

The expression \eqref{eq:rec} is not only very important on a technical
level, but it
allows an important generalization of the model: there is no reason 
of choosing $b$ integer valued. We will therefore choose it real valued
($b>1$), while $s$ will always be an integer larger or equal to $2$.

The quenched disorder $\go$ therefore enters in each step 
of the recursion: the annealed (or pure) model is given by  $r_n=\bbE[R_n]$
and it solves the recursion
\begin{equation}
r_{n+1}=\frac{\exp((s-1)h)r_n^s+(b-1)}{b}, \label{recr}
\end{equation}
with $r_0=1$. It is important to stress that $r_n=R_n^{(1)}$ if $\gb=0$
so, with our choice of the parameters, the annealed case coincides
with the infinite temperature case.
The annealed \textsl{critical point} $h_c$ is the infimum in the set of $h$ such that $r_n$ tends to infinity.

The class of systems obtained by choosing $b\in(1,s)$ in \eqref{eq:rec} is of a particular interest. With this setup, the pure system undergoes a phase transition in which the free energy critical exponent can take any value in $(1, \infty)$.
We explain below that the case $b \in [s, \infty)$ can be tackled by our methods too,
but it is less interesting for the viewpoint of the question we are addressing (i.e.\ behavior at marginality).




\subsection{Definition and existence of the free energy}

We are interested in the study of the free energy of this pinning model. The following result ensures that it exists and states some useful technical estimates.

\medskip

\begin{proposition}
 The limit
\begin{equation}
 \lim_{n\to\infty} s^{-n}\log R_n = \tf(\gb,h),
\end{equation}
exists almost surely and it is non-random. Moreover the convergence holds also in $\bbL_1(\dd \bbP)$.
The function $\tf(\gb,\cdot)$ is non-negative, non-decreasing. 
Moreover 
there exists a constant $c$ (depending on $\gb$ and $h$, $b$ and $s$
) such that
\begin{equation}
 |s^{-n}\bbE \log R_n-\tf(\gb,h)|\le cs^{-n}.
\end{equation}
\end{proposition}

\medskip

This result is the exact equivalent of \cite[Theorem 1.1]{cf:GLT}. Though certain modifications are needed to take into account the difference between the two models, it is straightforward to adapt the method of proof.

We set
\begin{equation}
 h_c(\gb)\, =\inf\{h \text{ such that } \tf(\gb,h)>0)\}\ge 0.
\end{equation}
The value $h_c(\gb)$ is called the \textsl{critical point of the system}, basic properties of the free energy ensure that $\tf(\gb,h)>0$ if and only if $h>h_c(\gb)$. Of course
$h_c(\gb)$ is a non analyticity point of $\tf(\gb, \cdot)$. By {\sl critical behavior of the system} we will refer to how the
free energy vanishes as $h \searrow h_c(\gb)$.

It is not hard to check that with $b\in(1,s)$, we have $h_c(0)=0$.
Indeed if $h<0$, $r_n$ converges to $r_\infty\in(0,1)$ the stable fixed point of the map $x\mapsto (\exp((s-1)h)x^s+(b-1)r_0)/{b}$;
if $h=0$, $r_n=1$ for every $n$;
if $h>0$, $r_n$ diverges to infinity in such a way that  the free energy is  positive,
 therefore $h_c(0)=0$.

Moreover 
there is an immediate comparison between annealed and quenched systems:  by Jensen inequality we have that for any $\gb$ and $h$
\begin{equation}
\bbE \log R_n \, \le\,  \log r_n, 
\end{equation}
so that $\tf(\gb,h)\le \tf(0,h)$ for any $\gb>0$ and $h_c(\gb)\ge h_c(0)=0$.

\subsection{Some results on the free energy}

We state here all the results on properties of the critical system that have been proved for the bond-disorder model in
\cite{cf:GLT} and that are still true in our framework. 
Our first result concerns the shape of the free energy curve around zero in the pure system.

\medskip
\begin{theorem}\label{th:freeen}
For every $b\in (1,s)$ 
there exists a constant $c_b$ such that for all $0\le h\le 1$,
\begin{equation}
 \frac{1}{c_b}h^{1/\alpha}\le \tf(0,h)\le c_b h^{1/\alpha}, \label{eq:anneal}
\end{equation}
where $\alpha:=(\log s-\log b)/{\log s}$.
\end{theorem}

\medskip

The second result describes the influence of the disorder for $b\neq\sqrt{s}$, i.e.\ whether or not the quenched annealed critical point coincide, and it gives an estimate for their difference. Recall that if quenched and annealed critical points differ, we say that the disorder is relevant (and irrelevant if they do not).

\medskip
\begin{theorem}\label{th:relirel}
 When $b\in(\sqrt{s},s)$, for $\gb\le \gb_0$ (depending on $b$ and $s$), we have
 $h_c(\gb)=h_c(0)=0$ and moreover, for any $\gep>0$ the exists $h_\gep$ such that for any $h\le h_\gep$
\begin{equation}
(1-\gep)\tf(0,h)\le \tf(\gb,h)\le \tf(0,h). \label{eq:simsim}
\end{equation}
When $b< \sqrt{s}$ there exists $c_{b,s}$ such that for every $\gb\le1$
\begin{equation}
 \frac{1}{c_{b,s}}\gb^{\frac{2\ga}{2\ga-1}}\le h_c(\gb)-h_c(0)\le c_{b,s}\gb^{\frac{2\ga}{2\ga-1}}.
\end{equation}
\end{theorem}

As a matter of fact, it can be shown that when $b\ge s$, the annealed free energy grows slower than any power of $(h-h_c(0))$ when $h\to (h_c)_+$. In a sense, $\eqref{eq:anneal}$ holds with $\alpha=0$, and Harris Criterion predicts that $\eqref{eq:simsim}$ holds in that case. One should be able to prove this by using a second moment method approach. For a recent work concerning the case $\alpha=0$ in the non-hierarchical setup, see \cite{cf:AZ_new}.

\subsection{Main result: the marginal case $b=\sqrt{s}$}

The main novelty of this paper is the result we  present now: 
disorder is relevant for the marginal case $b=\sqrt{s}$.
To our knowledge this is the first example in which one can 
establish the character of the disorder in the marginal case.
Moreover, we believe that the method of proof
is sufficiently flexible to be adapted to other contexts.
 
 \medskip
 
\begin{theorem}\label{th:mainres}
When $b=\sqrt{s}$ there exists positive constants $c_1$, $c_2$ and $\gb_0$ (depending on $s$) such that for every $\gb\le \gb_0$ 

\begin{equation}\label{eq:tbd}
 \exp\left(-\frac{c_1}{\gb^2}\right) \le h_c(\gb)\le \exp\left(-\frac{c_2}{\gb}\right).
\end{equation}
 \end{theorem}
 
 \medskip

The two bounds are obtained by very different methods, and they will be proven in the two sections that follow. The upper bound for $h_c(\gb)$ is proven by controlling the variance for a finite volume, and using a finite volume estimate for the energy. This is essentially the same method as the one developed in \cite{cf:GLT} and it is analagous to what is done in \cite{cf:Ken} for the model based on a renewal process. The lower bound is obtained using fractional moments, and a change of measure on the environment, in fact a shift of the values taken by $\go$ (in the Gaussian case). The argument  uses strongly the specific property of the site disorder model, in fact the value of the shift is site dependent, in order to take advantage of  the fact that some sites are more likely to be visited than others.
\medskip
\begin{rem}\rm \label{rem6}
Contrary to non marginal cases, the two bounds we obtain for $h_c(\gb)$ for $b=\sqrt{s}$ do not match, showing that the understanding of the marginal regime is still incomplete. However:
\begin{itemize}
 \item[(1)] The fact that the lower bound given in Theorem \ref{th:mainres} corresponds to the upper bound found in \cite{cf:GLT} is purely accidental and it should not lead to misleading conclusions.
 \item[(2)] In \cite{cf:DHV}, it is predicted that for the bond model, the second moment method gives the right bound for $h_c(\gb)$. In view of this prediction, the upper bound should give the right order for $h_c(\gb)$, although we do not have any mathematics to support this prediction.
\end{itemize}
\end{rem}

\begin{rem}\label{rem:Green2}\rm 
We can now make Remark \ref{rem:Green} more precise. The case $b=\sqrt{s}$ on which we focus is really similar the pinning model defined with the simple random walk. 
Indeed, for integer $b$ (recall definition \eqref{eq:part}) The expected number of contacts with the interface $\sigma$ is
\begin{equation}
\bE_n\left(\sum_{i=1}^{s^n-1} \ind_{\{S_i=\sigma_i\}}\right)=\sum_{i=1}^n b^{-i}(s-1)s^{i-1}\,
=\, \frac{s-1}{s-b}(s/b)^n.
\end{equation}
When $b=\sqrt{s}$, it is proportional to $s^{n/2}$ which is the square root of the length of the system. This is also the case for the random walk in dimension $1$ where
$\bE\left(\sum_{i=1}^n \ind_{\{S_i=0\}}\right)$  behaves like $ n^{1/2}$ and we have thus a clear analogy between site hierarchical model and random walk model. On the other hand Remark \ref{rem6}(2) possibly suggests that $h_c(\gb)-h_c(0)$ behaves like $\exp(-c/\gb^2)$ both in the bond hierarchical and random walk model (and not like $\exp(-c/\gb)$). 
\end{rem}

\begin{rem}\rm   \label{remrev}
Since the first version of the present paper, some substantial progresses have been made in the understanding of pinning model at marginality.
Using a different method, in \cite{cf:GLT_marg}, Giacomin and Toninelli together with the author proved marginal relevance of disorder for both the hierarchical-lattice with bond disorder model and the random-walk based model. While the method we present here turns out to be less performing than the method in \cite{cf:GLT_marg}
on bond disorder and random walk based models, it should be pointed out
that it does improve on the results in \cite{cf:DGLT} (in the direction of the statements in \cite{cf:AZ}). Moreover
our inhomogeneous shifting procedure, with respect to the more complex change of measure in \cite{cf:GLT_marg},
has the advantage of being very flexible and easier to adapt to more general contexts.
On the other hand, it can be shown that adapting the procedure in \cite{cf:GLT_marg} to the
site disorder model would lead to replacing the exponent $2$ in the left-most side of \eqref{eq:tbd} with $4/3$,
but the proof is substantially heavier than the one that we present, for a result that is still comparable on a qualitative level
(the upper bound is not matched).
\end{rem}

\medskip

In the sequel we focus on the proof of the case $b=\sqrt{s}$, but the arguments can be adapted (and they get simpler) to prove also the inequalities of 
Theorem \ref{th:relirel}. We stress once again that the case $b\neq\sqrt{s}$ 
is  detailed in \cite{cf:GLT} 
for the bond model.

\section{The upper bound: control of the variance}
The main result of this section is:

\begin{proposition}\label{th:lvbds}
For any fixed $s$, and $b=\sqrt{s}$ one can find constants $c_s$ and $\gb_0$, such that for all $\gb\le \gb_0$,
\begin{equation}
h_c(\gb)\le \exp\left(-\frac{c_s}{\gb}\right).
\end{equation}
\end{proposition}
Such a result is achieved by a second moment computation and for this we have to get some bounds on the first two moments of $R_n$.

\subsection{A lower bound on the growth of $r_n$}
We prove a technical result on the growth of $r_n$ when $h>0$.
For convenience we write $p_n:=(r_n-1)$. We have $p_0=0$ and \eqref{recr} becomes
\begin{equation}
p_{n+1}=\frac{1}{b}\left((1+p_n)^s\exp((s-1)h)-1\right). \label{recrr}
\end{equation}
Note that if $h>0$, $p_n>0$ for all $n \ge 1$, so that \eqref{recrr} implies
\begin{equation}
p_{n+1}\ge \frac{s}{b} p_n+\frac{h(s-1)}{b}\ge \frac{s}{b}p_n+\frac{h}{b}, \label{rect}
\end{equation}
and therefore 
\begin{equation}\label{pnge}
p_n\ge (s/b)^{n-1} (h/b).
\end{equation} 

 








\subsection{An upper bound for the growth the variance}
We prove now a technical result concerning the variance of $R_n$ which is crucial for the proof of Proposition \ref{th:lvbds}.
Before stating the result we introduce some notation and write the induction equation for the variance.
The variance $\Delta_n$ of the random variable $R_n$ is given by the following recursion

\begin{equation}
 \gD_{n+1}=\frac{1}{b^2}\left(\left(\Delta_n+r_n^{2}\right)^s\exp\left((s-1)(\gga(\gb)+2h)\right)-r_n^{2s}\exp(2(s-1)h)\right),
\end{equation}
where $\gga(\gb)=\log M(2\gb)-2\log M(\gb)$ (recall that $M(\gb)=\bbE[\exp(\gb\go_1)])$. Because $\go$ has unit variance, we have $\gga(\gb)\stackrel{\gb\searrow 0}{\sim} \gb^2$.\\
Let $v_n$ denote the relative variance $\gD_n/(r_n)^2$.
We have
\begin{equation} \label{eq:relv}
v_{n+1}=\frac{b^2r_n^{2s}\exp\left((s-1)2h\right)}{\left(r_n^s\exp\left((s-1)h\right)+(b-1)\right)^2}\frac{\exp((s-1)\gga(\gb))(v_n+1)^s-1}{b^2}.
\end{equation}
Let $n_1$ be the smallest integer such that $p_n\ge 1$. 

\begin{lemma}\label{th:dddd}
We can find constants $c_5$ and $\gb_0$ such that for all $\gb<\gb_0$,
for $h=\exp(-c_5/\gb)$,
\begin{equation}
v_{n_1}\le \gb
\end{equation}
\end{lemma}


\begin{proof}
We make a Taylor expansion of \eqref{eq:relv} around $r_n=1$, $h=0$, $\gb=0$, $v_n=0$,

\begin{equation}
v_{n+1}=\left(1+O(h+p_n)\right)\left(\frac{(s-1)}{b^2}\gb^2+\frac{s}{b^2}v_n+O(v_n^2)+o(\gb^2)\right).
\end{equation}
From the previous line, we can find a constant $c_4$ such that if $0<p_n\le 1$, $h\le1$, $\gb\le1$ and $v_n\le 1$ we have (recall that $b=\sqrt{s}$)
\begin{equation}
v_{n+1}\le c_4\gb^2+v_n(1+c_4v_n)(1+c_4(h+p_n)).
\end{equation}
By induction, we get that as long if $v_{n-1}\le 1$ and $p_n\le 1$, we have
\begin{equation}
v_n\le n c_4\gb^2\prod_{i=0}^{n-1}(1+c_4v_i)(1+c_4(h+p_i)). 
\end{equation}
By \eqref{pnge} we have $(h+p_i)\le (1+b)p_i$ for all $i\ge 1$. Changing the constant $c_4$ if necessary we get the nicer formula
\begin{equation}
v_n\le n c_4\gb^2\prod_{i=1}^{n-1}(1+c_4v_i)(1+c_4p_i). \label{nnn}
\end{equation}
Let $n_0$ be the smallest integer such that $v_{n_0}\ge \gb$. We have to show that we cannot have $n_0\le n_1$.
If $n_0\le n_1$, \eqref{nnn} implies
\begin{equation}
v_{n_0}\le n_0 c_4 \gb^2 \prod_{i=1}^{n_0-1}(1+c_4v_i)(1+c_4 p_i). \label{eq:pff}
\end{equation}
As $p_{n+1}\ge (s/b)\, p_n$ for all $n\ge 0$ (cf.\ \eqref{rect}), and $p_{n_1-1}\le 1$, we have $p_{n_1-2}\le (b/s)$,\\
 $p_{n_1-3}\le (s/b)^{-2}$ and by induction for all $i\le n_1 -1$
\begin{equation}
 p_i\le (s/b)^{i-(n_1-1)}.
\end{equation}
Therefore
\begin{equation}
\prod_{i=1}^{n_0-1} \left(1+c_4 p_i\right)\le\prod_{i=1}^{n_1-1} \left(1+c_4 p_i\right)\le \prod_{i=1}^{n_1-1} \left(1+c_4 (s/b)^{i-(n_1-1)}\right)\le \prod_{k=0}^{\infty}\left(1+c_4(s/b)^{-k}\right).
\end{equation}
The last term is finite, and is clearly not dependent on $\gb$ or $h$.
Moreover, because $p_n\ge (s/b)^n (h/b)$, it is necessary that 
\begin{equation}
n_1-2\le \frac{\log (b/h)}{\log (s/b)}.
\end{equation}
Replacing $b$ and $h$ with $\sqrt{s}$ and $\exp\left(-\frac{c_5}{\gb}\right)$, we get
\begin{equation}
n_0\le n_1\le \frac{2c_5}{\gb\log s}+1\le \frac{3c_5}{\gb\log s}
\end{equation}
for $\gb$ small enough.
Replacing $n_0$ by this upper bound in \eqref{eq:pff} gives us that $n_0\le n_1$ implies
\begin{equation}
\gb\le v_{n_0}\le \gb \left[\frac{3c_5c_4}{\log s}(1+c_4\gb)^{\frac{3c_5}{\gb\log s}}\prod_{k=0}^{\infty}(1+c_4 (s/b)^{-k})\right].
\end{equation}
If $c_5$ is chosen small enough the right-hand side is smaller than the left--hand side.
\end{proof}

\subsection{Proof of Proposition \ref{th:lvbds}}

Let us choose $c_5$ as in Lemma \ref{th:dddd}, $\gb$ small enough, and $h=\exp(-c_5/\gb)$. We fix some small $\gep>0$.
From Lemma \ref{th:dddd}, we have $v_{n_1}\le \gb$ and $r_{n_1}\ge 2$. The idea of the proof is to consider some $n$ a bit larger that $n_1$ such that $r_n$ is big and $v_n$ is small, in order to get a good bound on $\bbE \log R_n$.

We use \eqref{eq:relv} to get a rough bound on the growth of $v_n$ when $n\ge n_1$,
\begin{equation}
v_{n+1}\le (1+v_n)^s\exp\left[\gga(\gb)(s-1)\right]-1.
\end{equation}
Hence, one can find a constant $c_6$ such that as long as $v_n\le 1$ and $\gb$ small enough, we have
\begin{equation}
v_{n+1}\le c_6(v_n+\gb^2).
\end{equation}
If we choose $c_6>1$, this implies that for any integer $k\ge 0$
\begin{equation}
v_{n_1+k}\le c_6^k(\gb+k\gb^2),
\end{equation}
provided the right--hand side is less than $1$.
\medskip

We fix $k$ large enough, and $\gb_0$ such that  $c_6^k(\gb+k\gb^2)\le \gep$ for all $\gb\le \gb_0$.
From \eqref{rect} and the definition of $n_1$, we have $r_{n_1+k}\ge 1+(s/b)^k$. Let $k$ be a fixed (large) integer, Chebycheff inequality implies that
\begin{equation}
\bbP\left(R_{n_1+k}\le (1/2) r_{n_1+k}\right)\le 4v_{n_1+k}\le 4\gep.
\end{equation}
We write $n_2=n_1+k$. Using the fact that $R_n\ge (b-1)/b$ we have
\begin{multline}
 \bbE\left[\log R_{n_2}\right]\ge  \left[\log r_{n_2}-\log 2\right]\bbP\left(R_{n_2}\ge (1/2) r_{n_2}\right)+\log\frac{b-1}{b}\bbP\left(R_{n_2}\le (1/2)r_{n_2}\right)\\
										\ge  (1-4\gep)\left[\log (1+(s/b)^k)-\log 2\right]-4\gep \log \frac{b}{b-1}.
\end{multline}
By choosing a suitable $k$, this can be made arbitrarily large.
Taking the $\log$ in \eqref{eq:rec} and forgetting the $(b-1)$ term gives
\begin{equation}
\bbE \log R_{n+1}\ge s \bbE\left[\log R_{n}\right]+(s-1)\left(h-\log M(\gb)+\bbE[\go_1]\right)-\log b.
\end{equation} 
Therefore, the sequence $s^{-n}\left[\bbE\left[\log R_n\right] -\frac{\log{b}}{s-1}+h-\log M(\gb)\right]$ is increasing (recall $\bbE[\go_1]=0$).
With our settings we have
\begin{equation}
\bbE\left[\log R_{n_2}\right]> \frac{\log{b}}{s-1}+\log M(\gb)-h,
\end{equation}
therefore 
\begin{equation}
\tf(\gb,h)=\lim_{n\rightarrow\infty}s^{-n}\left[\bbE\left[\log R_n\right] -\frac{\log{b}}{s-1}+h-\log M(\gb)\right]>0.
\end{equation}
\qed

\section[The lower bound]{The lower bound: Fractional moment and improved shifting method}

In this section, we prove the lower bound by improving the method of measure-shifting used in \cite{cf:GLT} and \cite{cf:DGLT}. Instead of considering an homogeneous shift on the environment, we chose to shift more the sites that are more likely to be visited.

\begin{proposition}\label{th:upbd}
When $b=\sqrt{s}$, there exists a constant $c_s$ such that for all $\gb\le 1$ we have
\begin{equation}
h_c\ge \exp(-c_s/\gb^2).
\end{equation}
\end{proposition}
\subsection{Fractional moment}

\begin{lemma}\label{th:fracmomlem}
Fix $\theta\in(0,1)$ and set
\begin{equation}
x_{\theta}=\max\left\{ x \ \Big| \ \frac{x^s+(b-1)^{\theta}}{b^{\theta}}\le x\right\}.
\end{equation}
(Note that $x_{\theta}$ is defined whenever $\theta$ is close enough to $1$. When it is defined we have $x_{\theta}<1$ as the inequality cannot be fulfilled for $x\ge 1$.)\\
 If $\bbE\left[\exp\left(\theta(\gb\go_1-\log M(\gb)+h)\right)\right]\le 1$, and if there exists $n$ such that 
 $\bbE[R_n^{\theta}]\le x_\theta$, then $\tf(\gb,h)=0$.
\end{lemma}

\begin{proof}

Let $0<\theta<1$ be fixed, an $u_n=\bbE[ R_n^{\theta}]$ denotes the fractional moment of $R_n$.
We write $a_{\theta}=\bbE[ A_1^{\theta}]=\bbE\left[\exp\left(\theta(\gb\go_1-\log M(\gb)+h)\right)\right]$.
Using the basic inequality $\left(\sum x_i\right)^{\theta}\le \sum x_i^{\theta}$ and averaging with respect to $\bbP$,
we get from \eqref{eq:rec}

\begin{equation}
 u_{i+1}\le \frac{u_i^s a_{\theta}^{s-1}+(b-1)^{\theta}}{b^{\theta}}.
\end{equation}
With our assumption, $a_{\theta}\le 1$, so that
\begin{equation}
 u_{i+1}\le \frac{u_i^s+(b-1)^{\theta}}{b^{\theta}}.
\end{equation}
The map 
\begin{equation}
g_\theta:\ x\mapsto \frac{x^s+(b-1)^{\theta}}{b^{\theta}}\quad \text{ for } x\ge 0 \label{eq:gg},
\end{equation}
 is non-decreasing so that if $u_n\le x_\theta$ for some $n$, then $u_i\le x_\theta$ for every $i\ge n$.
In this case the free energy is $0$ as
\begin{equation}
\tf(\gb,h)=\lim_{n\to\infty}s^{-n}\bbE \log R_n\le \liminf_{n\to\infty}\frac{1}{\theta s^n}\log u_n=0,
\end{equation}
where the last inequality is just Jensen inequality.
\end{proof}

We add a second result that guaranties that the previous lemma is useful.

\begin{lemma}
\begin{equation}
\lim_{\theta\rightarrow 1^{-}} x_{\theta}=1.
\end{equation}
\end{lemma}

\begin{proof}
One just has to check that 
\begin{equation}
\lim_{\theta\rightarrow 1^{-}}g_{\theta}(1-\gep)=\frac{(1-\gep)^s+(b-1)}{b}<1-\gep,
\end{equation}
if $\gep$ is small enough.
\end{proof}

\subsection{Improved shifting method}

In this section, we prove Proposition \ref{th:upbd}, by estimating the fractional moment of $R_n$ by making a measure change on the environnement, making $\go$ lower. We simply use H\"older inequality to estimate the cost of the change of measure.
For simplicity we first write the proof for gaussian environment.

\begin{proof}[Proof of Proposition \ref{th:upbd}, Gaussian case]

Let $\gep>0$ be small (we will fix conditions on it later). We chose $\theta<1$ (close to 1) such that $x_{\theta}\ge 1-\gep$,
and some small $\eta>0$ whose value will depend on $\gep$. We consider (for $\gb\le 1$), the system of size $n=\frac{1}{\eta^2\gb^2}$ (without loss of generality, we can suppose it to be an integer) and $h=s^{-n}=\exp\left(-\frac{\ln s}{\eta^2\gb^2}\right)$. One can check that the condition $\bbE\left[\exp\left(\theta(\gb\go_1-\log M(\gb)+h)\right)\right]\le 1$ is fulfilled for $\gb$ small enough.

We define sets $V_i$ for $0\le i< n$ by
\begin{equation}
V_i=\left\{ j\in\{1,\dots,s^n-1\} \text{ such that } s^i \text{ divides } j \text{ and } s^{j+1} \text{does not divide } j \right\}  \label{eq:vi}
\end{equation}
Note that $|V_i|=(s-1)s^{n-1-i}$.

We slightly modify the measure $\bbP$ of the environment by shifting the value of $\go_j$ for $j\in V_i$, $i< n$ by $\frac{\eta s^{i/2}}{s^{n/2}\sqrt{n}}=\gd_i$ and we call $\tilde \bbP$ the modified measure. Notice that $h$ is small compared to any of the $\gd_i$.
The density of this measure is

\begin{equation}
 \frac{\dd \tilde\bbP}{\dd \bbP}(\go)=\exp\left(-\sum_{0\le i \le n-1}\sum_{j\in V_i}\left(\gd_i\go_j+\frac{\gd_i^2}{2}\right)\right).
\end{equation}
In order to estimate $u_n$ we use H\"older inequality
\begin{equation}
\bbE [ R_n^{\theta}] = \tilde\bbE\left[\frac{\dd\bbP}{\dd \tilde \bbP}R_n^{\theta} \right]\le\left( \tilde\bbE \left[\left(\frac{\dd\bbP}{\dd \tilde \bbP}\right)^{1/(1-\theta)}\right]\right)^{1-\theta} \left(\tilde \bbE\left[ R_n \right]\right)^{\theta}. \label{eq:hld}
\end{equation}
The first term is computed explicitly with the expression of the density and it is equal to
\begin{multline}
  \left( \tilde\bbE \left[\left(\frac{\dd\bbP}{\dd\tilde \bbP}\right)^ 
{1/(1-\theta)}\right]\right)^{1-\theta}\, =\,
\exp\left(\frac{\theta}{2 (1-\theta)}\sum_{i=0}^{n-1} |V_i|\gd_i^2 
\right)\, \\=
\exp\left(\frac{\theta}{2(1-\theta)}\sum_{i=0}^{n-1}(s-1)s^{n-i-1} 
\frac{\eta^2 s^i}{s^n n}\right)\, =\, \exp\left(\frac{\eta^2\theta  
(s-1)}{2(1-\theta)s}\right), \label{eq:deg}
  \end{multline}
  and we can choose $\eta$ to be such that the right-hand side is less than $1+\gep/2$.
 
Therefore in order to get an upper bound for $u_n$ we have to estimate $\tilde\bbE [R_n]$.
To do this, we write down the recursion giving $\tilde{r}_i=\tilde\bbE [R_i^{(1)}]$, for $i\le n-1$.
We have  $\tilde r_0=1$ and
\begin{equation}
 \tilde r_{i+1}\, =\, \frac{\tilde r_{i}^s\exp\left[(s-1)(-\gb\gd_{i}+h)\right]+(b-1)}{b}\label{eq:shiftrec}.
\end{equation}

Let us look at the evolution of $g_i=1-\tilde r_i\ge 0$ for $i\le n$.
We have
\begin{equation}\label{okok}\begin{split}
g_{i+1}&=\frac{1}{b}\left[1-(1-g_i)^{s}\exp\left((s-1)(-\gb\gd_{i}+h)\right)\right]\\
       &=\frac{1}{b}\left[1-(1-g_i)^{s}+(1-g_i)^{s}\left(1-\exp\left((s-1)(-\gb\gd_{i}+h)\right)\right)\right].
\end{split}\end{equation}
We can find $c_1$ such that $1-(1-g)^s\ge s(g- c_1 g^2)$ for all $0\le g\le 1$.
By choosing $\beta$ sufficiently small, the term in the exponential is small enough and $h$ is negligible compared to $\gb\gd_i$ so that when $g_i\le 2\gep$
\begin{equation} 
(1-g_i)^{s}\left(1-\exp\left((s-1)(-\gb\gd_{i}+h)\right)\right)\, \ge\, \frac{\gb\gd_i}{2}.
\end{equation}
Therefore, as long as $g_i\le 2\gep$, we have
\begin{equation} \label{eq:pii}
g_{i+1}\ge \frac{s}{b}(g_i-c_1g_i^2)+\frac{\gb\gd_i}{2b}.
\end{equation}

\begin{lemma}
If $\gep$ is chosen small enough (not depending on $\gb$), $g_{n}\ge 2\gep$.
\end{lemma}
\begin{proof}
Suppose that $g_n\le 2\gep$. Equation \eqref{okok} implies that
\begin{equation}
 g_i\ge \frac{1}{b}[1-(1-g_{i-1})^s], \quad \forall i\le n.
\end{equation}
This is equivalent to
\begin{equation}
 g_{i-1}\le 1-(1-b g_i)^{1/s},
\end{equation}
so that if $g_i\le 2\gep$ and $\gep$ is small enough,
\begin{equation}
 g_{i-1}\le s^{-1/4}g_{i}.
\end{equation}
Using the above equation inductively from $i=n$ to $k+1$, one gets 

\begin{equation}\label{s1/4}
g_k\le 2\gep s^{(k-n)/4}, \quad \forall k\le n.
\end{equation}

We write $q_i=s^{(n-i)/2}g_i=\frac{s^{n-i}}{b^{n-i}}g_i$, note that $q_0=0$.
When $g_i\le 2\gep$, from \eqref{eq:pii} and the definition of $q_i$, we have
\begin{equation}
q_{i+1}\ge q_i(1-c_1 g_i)+\frac{\gb\eta}{2b\sqrt{n}}.
\end{equation}
Using the fact that $g_k\le 2\gep$ for all $k\le n$, and the above inequality for $1\le k\le n-1$ we get that
\begin{equation}
 q_n\ge \frac{n\gb\eta}{2b\sqrt{n}}\prod_{k=0}^{n-1} (1-c_1 g_k).
\end{equation}
Now to we use \eqref{s1/4} to get that
  \begin{equation}
  \prod_{k=0}^{n-1} \left(1-c_1 g_k\right)\ge \prod_{k=0}^{n-1} \left(1-c_1 2\gep s^{(k-n)/4}\right)\ge \prod_{i=1}^{\infty}\left(1-c_1 2\gep s^{-i/4}\right) \ge 1/2,
  \end{equation}
where the last inequality holds if $\gep$ is chosen small enough.
Hence
 
\begin{equation}
 g_n=q_n \ge \frac{n\gb\eta}{2b\sqrt{n}}\prod_{k=0}^{n-1} (1-c_1 g_k)\ge \frac{\gb\eta\sqrt{n}}{4b}=\frac{1}{4b},
\end{equation}
and from that we infer that $g_n\ge 2\gep$ if $\gep<1/(8b)$.
\end{proof}

We just proved that $\tilde r_n \le 1-2\gep$. Using \eqref{eq:hld}, and the bound we have on \eqref{eq:deg} we get
\begin{equation}
u_n\le (1-2\gep)^{\theta}(1+\gep/2)\le 1-\gep\le x_\theta.
\end{equation}
The last inequality is just comes from our choice for $\theta$, the second inequality is true if $\gep$ is small, and $\theta$ sufficiently close to one. The result follows from Lemma \ref{th:fracmomlem}
\end{proof}

\begin{proof}[Proof of Proposition \ref{th:upbd}, general non-gaussian case]

When the environment is\\ non--gaussian,
one can generalize the preceding proof by making a tilt on the measure instead of a shift. This means that the change of measure becomes
\begin{equation}
\frac{\dd \tilde \bbP}{\dd \bbP}(\go)=\exp\left(-\sum_{0\le i \le n-1}\sum_{j\in V_i}\left(\gd_i\go_j+\log M(-\gd_i)\right)\right).
\end{equation}

Therefore the term giving the cost of the change of measure (cf. \eqref{eq:hld}) is

\begin{equation} \begin{split}
 \left( \tilde\bbE \left[\left(\frac{\dd\bbP}{\dd\tilde \bbP}\right)^{1/(1-\theta)}\right]\right)^{1-\theta}&=
\exp\left((1-\theta)\sum_{i=0}^{n-1} |V_i|\left[\log M\left(\frac{\theta}{1-\theta}\gd_i\right)+\frac{\theta}{1-\theta}\log M(-\gd_i)\right]\right)\\
&\le \exp\left(\frac{\theta}{(1-\theta)}\sum_{i=0}^{n-1}(s-1)s^{n-i-1}\frac{\eta^2 s^i}{s^n n}\right)=\exp\left(\frac{\eta^2\theta (s-1)}{(1-\theta)s}\right).
 \end{split}\end{equation}

Where the inequality is obtained by using that $\log M(\gb)\sim \gb^2/2$, so that for $x$ small enough,
$\log M (x)\le x^2$.

The second point where we have to look is the estimate of $\tilde r_n$. In fact, the term $\exp((s-1)(-\gb\gd_i+h))$ should be replaced by
\begin{equation}
\exp\left[(s-1)\left(\log M(\gb-\gd_i)-\log M(\gb)-\log M(-\gd_i)+h\right)\right].
\end{equation}
Knowing the behavior of the convex function $\log M(\cdot)$ near zero, it is not difficult to see that this is less than $\exp(-c_6\gb\gd_i+(s-1)h)$ for some constant $c_6$, which is enough for the rest of the computations.
\end{proof}

\chapter[Disorder relevance for pinning models]{Fractional moment bounds  and disorder relevance for pinning models}\label{CHAPDGLT}
\section{Introduction}

Pinning/wetting models with quenched disorder describe the random
interaction between a directed polymer and a one-dimensional {\sl
defect line}. In absence of interaction, a typical polymer
configuration is given by $\{(n,S_n)\}_{n\ge0}$, where
$\{S_n\}_{n\ge0}$ is a Markov Chain on some state space $\Sigma$ (for
instance, $\Sigma=\Z^d$ for $(1+d)$-dimensional directed polymers),
and the initial condition $S_0$ is some fixed element of $\Sigma$
which by convention we call $0$. The defect line, on the other hand,
is just $\{(n,0)\}_{n\ge0}$. The polymer-line interaction is
introduced as follows: each time $S_n=0$ ({\sl i.e.}, the polymer
touches the line at step $n$) the polymer gets an energy
reward/penalty $\epsilon_n$, which can be either positive or
negative. In the situation we consider here, the $\epsilon_n$'s are
independent and identically distributed (IID) random variables, with
positive or negative mean $h$ and variance $\beta^2\ge0$.

Up to now, we have made no assumption on the Markov Chain. The
physically most interesting case is the one where the distribution
$K(\cdot)$ of the first return time, call it $\tau_1$, of $S_n$ to $0$
has a power-law tail: $K(n):=\bP(\tau_1=n)\approx n^{-\alpha-1}$, with
$\alpha\ge0$. This framework allows to cover various situations
motivated by (bio)-physics: for instance, $(1+1)$-dimensional wetting
models \cite{cf:FLNO,cf:DHV} ($\ga=1/2$; in this case $S_n\ge0$, and
the line represents an impenetrable wall), pinning of
$(1+d)$-dimensional directed polymers on a columnar defect ($\ga=1/2$
if $d=1$ and $\ga=d/2-1$ if $d\ge2$), and the Poland-Scheraga model of
DNA denaturation (here, $\ga\simeq 1.15$ \cite{cf:dna}).  This is a
very active field of research, and not only from the point of view of
mathematical physics, see {\sl. e.g.}
\cite{cf:coluzzi} and references therein.  We refer to \cite[Ch. 1]{cf:Book} and references
therein for further discussion.

The model undergoes a localization/delocalization phase transition:
for any given value $\gb$ of the disorder strength, if the average
pinning intensity $h$ exceeds some critical value $h_c(\gb)$ then the
polymer typically stays tightly close to the defect line and the free
energy is positive.  On the contrary, for $h< h_c(\gb)$ the free
energy vanishes and the polymer has only few contacts with the defect:
entropic effects prevail. The annealed model, obtained by averaging the
Boltzmann weight with respect to disorder, is exactly solvable, and
near its critical point $h_c^{ann}(\gb)$ one finds that the annealed
free energy vanishes like $(h-h_c^{ann}(\gb))^{\max(1,1/\ga)}$
\cite{cf:Fisher}. In particular, the annealed phase transition is
first order for $\ga>1$ and second order for $\ga<1$, and it gets
smoother and smoother as $\ga$ approaches $0$.

A very natural and intriguing question is whether and how randomness
affects critical properties. The scenario suggested by the so-called
{\sl Harris criterion} \cite{cf:Harris} is the following: disorder
should be irrelevant for $\ga<1/2$, meaning that quenched critical
point and critical exponents should coincide with the annealed ones if
$\beta$ is small enough, and relevant for $\ga>1/2$: they should
differ for every $\gb>0$.  In the marginal case $\ga=1/2$, the Harris
criterion gives no prediction and there is no general consensus on
what to expect: renormalization-group considerations led Forgacs {\sl
  et al.}  \cite{cf:FLNO} to predict that disorder is irrelevant (see also the recent 
   \cite{cf:GN}),
while Derrida {\sl et al.} \cite{cf:DHV} concluded for marginal
relevance: quenched and annealed critical points should differ for
every $\gb>0$, even if the difference is zero at every perturbative
order in $\gb$.

\medskip

The mathematical understanding of these questions witnessed
remarkable progress recently, and we summarize here the state of the
art (prior to the the present contribution). 
\begin{enumerate}
\item A lot is now known on the  {\sl irrelevant-disorder regime}. 
In particular, it was proven in \cite{cf:Ken} (see
  \cite{cf:T_cmp} for an alternative proof) that quenched and annealed
  critical points and critical exponents coincide for $\beta$ small
  enough. Moreover, in \cite{cf:GT_irrel} a small-disorder expansion
  of the free energy, worked out in \cite{cf:FLNO}, was rigorously
  justified.

\item In the {\sl strong-disorder regime}, for which the Harris
  criterion makes no prediction, a few results were obtained recently.
  In particular, in \cite{cf:T_fractmom} it was proven that for any given
  $\ga>0$ and, say, for Gaussian randomness, $h_c(\gb)\ne
  h_c^{ann}(\gb)$ for $\beta$ large enough, and the asymptotic
  behavior of $h_c(\gb)$ for $\gb\to\infty$ was computed. These
  results were obtained through upper bounds on fractional moments of
  the partition function.  Let us mention by the way that the
  fractional moment method allowed also to compute exactly
  \cite{cf:T_fractmom} the quenched critical point of a {\sl reduced
    wetting model} (which somehow has a built-in strong-disorder
  limit); the same result was obtained in \cite{cf:BCT} via a rigorous
  implementation of renormalization-group ideas. Fractional moment
  methods  have proven to be useful also for other classes of disordered models
\cite{cf:AENSS,cf:AM,cf:ASFH,cf:BPP,cf:ED}. 

\item The {\sl relevant-disorder regime} is only partly understood. In
\cite{cf:GT_cmp} it was proven that the free-energy critical exponent
differs from the quenched one whenever $\gb>0$ and 
$\alpha>1/2$. 
However, the arguments in \cite{cf:GT_cmp} do not imply the critical point shift.
Nonetheless, 
the critical point shift issue  has been recently solved for a {\sl hierarchical
version} of the model, introduced in \cite{cf:DHV}. The hierarchical
model also depends on the parameter $\alpha$, and in \cite{cf:GLT} it
was shown that $h_c(\gb)-h_c^{ann}(\gb)\approx \beta^{2\ga/(2\ga-1)}$
for $\gb$ small (upper and lower bounds of the same order are
proven).

\item In the {\sl marginal case $\ga=1/2$} it was proven in
  \cite{cf:Ken,cf:T_cmp} that the difference
  $h_c(\gb)-h_c^{ann}(\gb)$ vanishes faster than any power of $\gb$,
  for $\gb\to0$.  Before discussing lower bounds on this difference,
  one has to be more precise on the tail behavior of $K(n)$, the
  probability that the first return to zero of the Markov Chain
  $\{S_n\}_n$ occurs at $n$: if $K(n)=n^{-(1+1/2)}L(n)$ with
  $L(\cdot)$ slowly varying (say, a logarithm raised to a positive or negative power), then the two critical points coincide for $\beta$ small
  if $L(\cdot)$ diverges sufficiently fast at infinity so that
  \begin{equation}
    \label{eq:condL}
\sum_{n=1}^\infty\frac1{n L(n)^2}\, <\,\infty.    
  \end{equation}
  The case of
  the $(1+1)$-dimensional wetting model \cite{cf:DHV} corresponds
  however to the case where $L(\cdot)$ behaves like a constant at
  infinity, and the result just mentioned does not apply.

The case $\ga=1/2$ is open also for the hierarchical model mentioned
above.

\end{enumerate}

In the present work we prove that if $\ga\in(1/2,1)$ or $\ga>1$ then
quenched and annealed critical points differ for every $\beta>0$, and
$h_c(\gb)-h_c^{ann}(\gb)\approx \gb^{2\ga/(2\ga-1)}$ for
$\gb\searrow0$ ({\sl cf.} Theorem \ref{th:a121} for a more precise
statement). For the case $\ga=1$, see the {\sl note added in proof} at the end
of this paper. 
In the case $\ga=1/2$, while we do not
prove that $h_c(\gb)\ne h_c^{ann}(\gb)$ in all cases in which condition
\eqref{eq:condL} fails, we do prove such a result if the function
$L(\cdot)$ vanishes sufficiently fast at infinity.  Of course,
$h_c(\gb)- h_c^{ann}(\gb)$ turns out to be exponentially small for
$\gb\searrow0$.

\medskip

Starting from next section, we will forget the full Markov structure 
of the polymer, and retain only the fact that the set of points of contact
with the defect line, $\tau:=\{n\ge 0:S_n=0\}$, is a renewal process
under the law $\bP$ of the Markov Chain.

\section{Model and main results}

Let $\tau:=\{\tau_0,\tau_1,\ldots\}$ be a renewal sequence started
from $\tau_0=0$ and with inter-arrival law $K(\cdot)$, {\sl i.e.,}
$\{\tau_i-\tau_{i-1}\}_{i\in \N:=\{1, 2, \ldots\}}$ are IID integer-valued random
variables with law $\bP(\tau_1=n)=K(n)$ for every $n\in\N$.   We assume that
$\sum_{n\in\N}K(n)=1$ (the renewal is recurrent) and that there exists
$\ga>0$ such that
\begin{equation}
  \label{eq:K}
  K(n)= \frac{L(n)}{n^{1+\alpha}}
\end{equation}
with $L(\cdot)$ a function that varies slowly at infinity, {\sl i.e.}, $L: (0, \infty) \to (0, \infty)$ is 
 measurable  and
such that $L(rx)/L(x)\to1$ when $x\to\infty$, for every $r>0$.  We
refer to \cite{cf:RegVar} for an extended treatment of slowly varying
functions, recalling just that examples of $L(x)$ include $(\log (1+x))^b$,
any $b \in \R$, and any (positive, measurable) function admitting
a positive limit at infinity (in this case we say that $L(\cdot)$ is {\sl trivial}).
Dwelling a bit more on nomenclature, $x \mapsto x^\gr L(x)$
is a {\sl regularly varying function of exponent } $\gr$,
so $K(\cdot)$ is just the restriction to the natural numbers of
a regularly varying function of exponent $-(1+\ga)$.

\medskip

We let $\gb\ge0$, $h\in\R$ and
$\go:=\{\go_n\}_{n\ge1}$ be a sequence of IID centered random variables
with unit variance and finite exponential moments. The law of $\go$ is 
denoted by $\bbP$ and the corresponding expectation  by $\bbE$.

For $a,b\in\{0,1,\ldots\}$ with $a\le b$ we let $Z_{a,b,\omega}$ be the
partition function for the system on the interval
$\{a,a+1,\ldots,b\}$, with zero boundary conditions at both endpoints:
\begin{equation}
\label{eq:Z}
Z_{a,b,\omega}=\bE\left(\left.e^{\sum_{n=a+1}^b(\beta\go_n+h)
\ind_{\{n\in\tau\}}}\ind_{\{b\in\tau\}}\right|a\in\tau\right),
\end{equation}
where $\bE$ denotes expectation with respect to the law $\bP$ of the
renewal.  One may rewrite $Z_{a,b,\omega}$ more explicitly as
\begin{equation}
  \label{eq:Z+}
  Z_{a,b,\omega}=\sum_{\ell=1}^{b-a}\sum_{i_0=a<i_1<\ldots<i_\ell=b}
\prod_{j=1}^\ell K(i_j-i_{j-1})e^{h\ell+\beta\sum_{j=1}^\ell \omega_{i_j}},
\end{equation}
with the convention that $Z_{a,a,\omega}=1$. 
Notice that, when writing  $n \in \tau$, we are interpreting $\tau$ as a subset
of $\N \cup \{0\}$ rather than as a sequence of random variables.
 We will write for
simplicity $Z_{N,\go}$ for $Z_{0,N,\go}$ (and in that case the
conditioning on $0\in\tau$ in \eqref{eq:Z} is superfluous since
$\tau_0=0$).  In absence of disorder ($\gb=0$), it is convenient to
use the notation
\begin{equation}
\label{eq:Zpure}
  Z_N(h):=\bE\left(e^{h\sum_{n=1}^N\ind_{\{n\in\tau\}}}
\ind_{\{N\in\tau\}}\right)=\bE\left(e^{h|\tau\cap\{1,\ldots,N\}|}\ind_{\{N\in\tau\}}
\right), 
\end{equation}
for the partition function.

We mention that the recurrence assumption $\sum_{n\in\N}K(n)=1$
entails no loss of generality, since one can always reduce to this
situation via a redefinition of $h$ ({\sl cf.} \cite[Ch.~1]{cf:Book}).

\medskip
As usual the {\sl quenched free energy} is defined as
\begin{equation}
\label{eq:FFF}
  \tf(\beta,h)=\lim_{N\to\infty}\frac1N \log Z_{N,\go}.
\end{equation}
It is well known ({\sl cf.} for instance \cite[Ch. 4]{cf:Book}) that the limit
\eqref{eq:FFF} exists $\bbP(\dd\go)$-almost surely and in $\mathbb
L^1(\bbP)$, and that it is almost-surely independent of $\go$.
Another well-established fact is that $\tf(\beta,h)\ge0$, which
immediately follows from $Z_{N,\go}\ge K(N)\exp(\gb \go_N+h)$.  This
allows to define, for a given $\gb\ge0$, the critical point $h_c(\gb)$
as
\begin{equation}
  \label{eq:hc}
  h_c(\gb):=\sup\{h\in\R:\;\tf(\beta,h)=0\}.
\end{equation}
It is well known that $h>h_c(\gb)$ corresponds to the {\sl localized
  phase} where typically $\tau$ occupies a non-zero fraction of
$\{1,\ldots,N\}$ while, for $h<h_c(\gb)$, $\tau\cap\{1,\ldots,N\}$
contains with large probability at most $O(\log N)$ points
\cite{cf:GTdeloc}. We refer to \cite[Ch.s 7 and 8]{cf:Book} for further
literature and discussion on this point.

\medskip

In analogy with the quenched free energy, the {\sl annealed free
  energy} is defined by
\begin{equation}
\label{eq:Fann}
  \tf^{ann}(\gb,h)\, :=\, \lim_{N\to\infty}\frac1N\log \bbE Z_{N,\go}=
\tf(0,h+\log\M(\gb)),
\end{equation}
with 
\begin{equation}
  \label{eq:M23}
  \M(\gb):= \bbE(e^{\beta\go_1}).
\end{equation}
We see therefore that the annealed free energy is just the free energy
of the pure model ($\beta=0$) with a different value of $h$.  The pure
model is exactly solvable \cite{cf:Fisher}, and we collect here a few
facts we will need in the course of the paper.

\begin{theorem}\cite[Th. 2.1]{cf:Book} 
\label{th:pure23} For the pure model $h_c(0)=0$.
  Moreover, there exists a slowly varying function $\hat L(\cdot)$
  such that for $h>0$ one has
  \begin{equation}
    \label{eq:Fpure}
    \tf(0,h)=h^{1/\min(1,\ga)}\hat L(1/h).
  \end{equation}
In particular,
\begin{enumerate}
\item if $\bE(\tau_1)=\sum_{n\in\N}n\,K(n)<\infty$ (for instance, if
  $\ga>1$) then $\hat L(1/h)\stackrel{h\searrow0}\sim 1/\bE(\tau_1)$.
\item if $\ga\in(0,1)$, then $\hat L(1/h)=C_\ga h^{-1/\ga}R_\ga(h)$ where
$C_\ga$ is an explicit constant and $R_\ga(\cdot)$ is the 
function, unique up to asymptotic equivalence,  that satisfies
$R_\ga ( b^\ga L(1/b)) \stackrel{b \searrow 0}\sim b$.
\end{enumerate}
\end{theorem}
As a consequence of Theorem \ref{th:pure23} and \eqref{eq:Fann}, the
annealed critical point is simply given by
\begin{equation}
\label{eq:hann}
  h^{ann}_c(\beta):=\sup\{h:\tf^{ann}(\gb,h)=0\}=-\log \M(\beta).
\end{equation}

Via Jensen's inequality one has immediately that $\tf(\gb,h)\le
\tf^{ann}(\gb, h)$ and as a consequence $h_c(\gb)\ge h_c^{ann}(\gb)$,
and the point of the present paper is to understand when this last
inequality is strict.  In this respect, let us recall that the
following is known: if $\ga\in(0,1/2)$, then $h_c(\gb)=
h_c^{ann}(\gb)$ for $\beta$ small enough \cite{cf:Ken,cf:T_cmp}. Also
for $\ga =1/2$ it has been shown that $h_c(\gb)= h_c^{ann}(\gb)$ if
$L(\cdot) $ diverges sufficiently fast (see below).  Moreover,
assuming that $\bbP(\go_1>t)>0$ for every $t>0$, one has that for
every $\ga>0$ and $L(\cdot)$ there exists $\beta_0<\infty$ such that
$h_c(\gb)\ne h_c^{ann}(\gb)$ for $\gb>\beta_0$ \cite{cf:T_fractmom}:
quenched and annealed critical points differ for {\sl strong
  disorder}. The strategy we develop here addresses the complementary
situations: $\ga>1/2$ and {\sl small disorder} (and also the case
$\ga=1/2$ as we shall see below).

\bigskip

Our first result concerns the case $\ga>1$:
\medskip

\begin{theorem}
\label{th:a>1}
  Let $\alpha>1$. There exists $a>0$ such that for every $\beta\le1$
  \begin{equation}
    \label{eq:a>1}
h_c(\beta)-h_c^{ann}(\beta)\ge a \beta^2.
  \end{equation}
Moreover, $h_c(\gb)>h_c^{ann}(\beta)$ for every $\gb>0$.
\end{theorem}
\medskip

Since $h_c(\gb)\le h_c(0)$ 
and $h_c^{ann}(\gb)\stackrel{\gb\searrow0}\sim-\gb^2/2$, we conclude
that the inequality \eqref{eq:a>1} is, in a sense, of the optimal
order in $\gb$. 
Note that $h_c(\gb)\le h_c(0)$ is just a consequence of Jensen's inequality:
\begin{eqnarray}
Z_{N,\go}&= &Z_N(h)\frac{\bE \left(e^{\sum_{n=1}^N(\gb\go_n+h)\ind_{\{n\in\tau\}}}\ind_{\{N\in\tau\}}\right)}
{\bE\left(e^{h\sum_{n=1}^N\ind_{\{n\in\tau\}}}\ind_{\{N\in\tau\}}\right)}\\\nonumber
&\ge& Z_N(h)\exp\left[
\gb\sum_{n=1}^N\go_n \frac{\bE\left(\ind_{\{n\in\tau\}}e^{h|\tau\cap\{1,\ldots,N\}|}\ind_{\{N\in\tau\}}\right)}
{\bE\left(e^{h|\tau\cap\{1,\ldots,N\}|}\ind_{\{N\in\tau\}}\right)}
\right],
\end{eqnarray}
from which $\tf(\gb,h)\ge\tf(0,h)$ and therefore $h_c(\gb)\le h_c(0)$ immediately follows from $\bbE(\go_n)=0$.
This can be made sharper in the sense that from the
explicit bound in \cite[Th.~5.2(1)]{cf:Book}  one directly extract also
that $h_c(\gb) \le -b \gb^2$ for a suitable $b\in (0,1/2)$ and
every $\gb \le 1$, so that $-h_c(\gb)/\gb^2 \in (b, 1/2-a)$.
We recall also that the (strict) inequality $h_c(\gb)<  h_c(0)$ has been
established in great generality in \cite{cf:AS}.

\medskip

In the case $\ga\in(1/2,1)$ we have the following:
\medskip

\begin{theorem}
  \label{th:a121} Let $\alpha\in(1/2,1)$. For every $\gep>0$ there
exists $a(\gep)>0$ such that 
\begin{equation}
\label{eq:witheps}
  h_c(\gb)-h_c^{ann}(\gb)\ge a(\gep)\,\gb^{(2\alpha/(2\alpha-1))+\gep},
\end{equation}
for $\beta\le1$. Moreover, $h_c(\gb)>h_c^{ann}(\gb)$ for every 
$\gb>0$.
\end{theorem}
\medskip

To appreciate this result, recall that in \cite{cf:Ken,cf:T_cmp} it
was proven that
\begin{equation}
\label{eq:fromKA-FT}
    h_c(\gb)-h_c^{ann}(\gb)\le \tilde L(1/\gb)\gb^{2\alpha/(2\alpha-1)},
\end{equation}
for some (rather explicit, {\sl cf.} in particular \cite{cf:Ken}) slowly
varying function $\tilde L(\cdot)$. Notably,  $\tilde L(\cdot)$ is
trivial if $L(\cdot)$ is.  The conclusion of Theorem~\ref{th:a121} can
actually be strengthened and we are able to replace the right-hand
side of \eqref{eq:witheps} with $\bar
L(1/\gb)\gb^{2\alpha/(2\alpha-1)}$ with $\bar L(\cdot)$ another slowly
varying function, but on one hand $\bar L(\cdot)$ does not match the
bound in \eqref{eq:fromKA-FT} and on the other hand it is rather clear
that it reflects more a limit of our technique than the actual
behavior of the model; therefore, we  decided to present the simpler
argument leading to the slightly weaker result \eqref{eq:witheps}.

\medskip
The case $\ga=1/2$ is the most delicate, and whether quenched and
annealed critical points coincide or not crucially depends on the
slowly varying function $L(\cdot)$. In \cite{cf:Ken,cf:T_cmp} it was
proven that, whenever
\begin{equation}
\label{eq:sommaconverge}
\sum_{n\ge1}\frac{1}{n\,L(n)^2}<\infty,  
\end{equation}
there exists $\beta_0>0$ such that $h_c(\beta)=h_c^{ann}(\gb)$ for $\gb\le
\gb_0$, and that when the same sum diverges then 
$h_c(\beta)-h_c^{ann}(\gb)$ is bounded {\sl above} by some function of $\gb$
which vanishes faster than any power for $\gb\searrow0$. For instance, if
$L(\cdot)$ is asymptotically constant then
\begin{equation}
  h_c(\beta)-h_c^{ann}(\gb)\le c_1\,e^{-c_2/\gb^2},
\end{equation}
for $\gb\le1$. While we are not able to prove that quenched and
annealed critical points differ as soon as condition
\eqref{eq:sommaconverge} fails  (in particular not when $L(\cdot)$
is asymptotically constant), our method can be pushed further to prove this if
$L(\cdot)$ vanishes sufficiently fast at infinity:

\medskip

\begin{theorem}
  \label{th:a12} Assume that for every $n\in\N$
  \begin{equation}
\label{eq:Lpiccola}
    K(n)\le c\frac{n^{-3/2}}{(\log n)^\eta},
  \end{equation}
for some $c>0$ and $\eta>1/2$. Then 
for every $0<\gep<\eta-1/2$ there exists $a(\gep)>0$ such that
\begin{equation}
  \label{eq:a12}
  h_c(\beta)-h_c^{ann}(\beta)\ge a(\gep) 
\exp\left(-\frac1{\gb^{\frac{1}{\eta-1/2-\gep}}}\right).
\end{equation}
Moreover, $h_c(\gb)>h_c^{ann}(\beta)$ for every $\gb>0$.
\end{theorem}

\medskip

\subsection{Fractional moment method}
In order to introduce our basic idea and, effectively, start the proof, we need some additional notation.
We fix some $k\in\N$ and we set for $n\in\N$
\begin{equation}
  \label{eq:z}
  z_n:=e^{h+\beta\omega_n}.
\end{equation} 
Then, the following identity holds for $N\ge k$:
\begin{equation}
  \label{eq:rec2}
  Z_{N,\omega}\, =\, \sum_{n=k}^N Z_{N-n,\omega}\sum_{j=0}^{k-1}K(n-j)\,
  {z_{N-j}} Z_{N-j,N,\omega}.
\end{equation}
This is simply obtained by decomposing the partition function \eqref{eq:Z}
according to the value $N-n$ of the last point of $\tau$ which does
not exceed $N-k$ (whence the condition $0\le N-n\le N-k$ in the sum), and to the value $N-j$ of the first point of $\tau$
to the right of $N-k$ (so that $N-k<N-j\le N$).  It is important to notice that
$Z_{N-j,N,\omega}$
has the same law as $Z_{j,\go}$ and that the three random variables 
$Z_{N-n,\go}$, $z_{N-j}$ and 
$ Z_{N-j,N,\omega}$
 are independent, provided that $n\ge k$ and $j<k$.

\begin{figure}[h]
\begin{center}
\leavevmode
\epsfysize =2.3 cm
\epsfxsize =14.5 cm
\psfragscanon
\psfrag{0}[c]{$0$}
\psfrag{N}[c]{$N$}
\psfrag{N-k}[c]{\small $N-k$}
\psfrag{N-j}[c]{\small $N-j$}
\psfrag{N-n}[c]{\small $N-n$}
\psfrag{Z1}[c]{$Z_{N-n, \go}$}
\psfrag{Z2}[c]{$K(n-j) z_{N-j}$}
\psfrag{Z3}[c]{$Z_{ N-j,N, \go}$}
\epsfbox{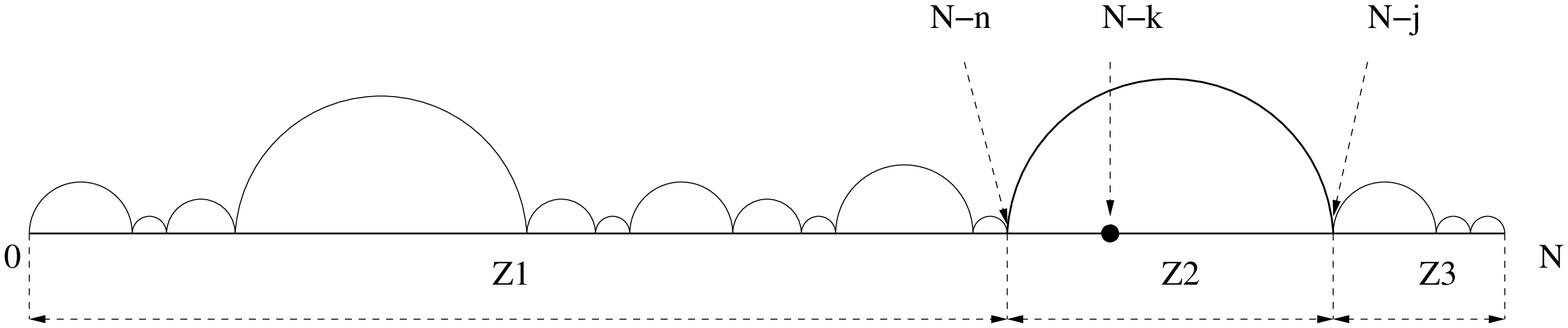}
\end{center}
\caption{\label{fig:chop} 
The decomposition of the partition function
is simply obtained by fixing a value of $k$ and 
summing over the values of the last contact (or renewal epoch) before $N-k$ and the first
after $N-k$. In the drawing the two contacts are respectively $N-n$ and $N-j$ and arcs of course  identify steps between successive contacts.}
\end{figure}

\medskip

Let $0<\gamma<1$ and $A_N:=\bbE[(Z_{N,\omega})^\gamma]$, with $A_0:=1$.
Then, from \eqref{eq:rec2} and using the elementary inequality 
\begin{equation}
  \label{eq:ineqgamma}
(a_1+\ldots+a_n)^\gamma\le
a_1^\gamma+\ldots+a_n^\gamma , 
\end{equation}
which holds for $a_i\ge0$, one deduces
\begin{equation}
  \label{eq:rec3}
  A_N\le 
\bbE[ z_1^\gamma] \sum_{n=k}^N A_{N-n}
\sum_{j=0}^{k-1}K(n-j)^\gamma A_j.
\end{equation}

The basic principle is the following:

\medskip

\begin{proposition}
\label{th:deloc}
  Fix $\beta$ and $h$. If there exists $k\in\N$ and $\gamma<1$ such that 
\begin{equation}
  \label{eq:if}
\rho\,:=\,  \bbE[ z_1^\gamma]  \sum_{n=k}^\infty 
\sum_{j=0}^{k-1}K(n-j)^\gamma A_j\le1,
\end{equation}
then $\tf(\beta,h)=0$. Moreover if
$\rho<1$ there exists $C=C(\rho, \gamma, k, K(\cdot))>0$ such that
\begin{equation}
  \label{eq:if2}
  A_N \, \le \, C \left( K(N)\right)^\gamma,
  \end{equation}
  for every $N$.
\end{proposition} 
\medskip

Of course, in view of the results we want to prove, 
the main result of Proposition~\ref{th:deloc} is the first.
The second one, namely \eqref{eq:if2}, is however of independent
interest and may be used to obtain path estimates on the process
(using for example the techniques in \cite{cf:GTdeloc} and
\cite[Ch.~8]{cf:Book}).

\medskip

{\sl Proof of Proposition \ref{th:deloc}}.  Let
$\bar A:=\max\{A_0,A_1,\ldots,A_{k-1}\}$. From \eqref{eq:rec3} it
follows that for every $N\ge k$
\begin{equation}
  A_N\, \le \, \rho \max\{A_0,\ldots,A_{N-k}\},
\end{equation}
from which one sees by induction that, since $\rho\le 1$, for every
$n$ one has $A_n\le \bar A$.  The statement $\tf(\beta,h)=0$ follows
then from Jensen's inequality:
\begin{equation}
  \tf(\beta,h)=\lim_{N\to\infty}\frac1{N\gamma}\bbE \log (Z_{N,\go})^\gamma\le 
\lim_{N\to\infty}\frac1{N\gamma}\log A_N=0.
\end{equation}


In order to prove \eqref{eq:if2} we introduce
\begin{equation}
Q_k(n)\, :=\, \begin{cases}
\bbE[ z_1^\gamma]
\sum_{j=0}^{k-1}K(n-j)^\gamma A_j,
&\text{ if } n \ge k,
\\
0 &\text{ if } n=1, \ldots k-1.
\end{cases}
\end{equation}
Since $\rho= \sum_n Q_k(n)$, the assumption $\rho<1$
tells us that $Q_k(\cdot)$ is a sub-probability distribution
and it becomes a probability distribution if we set, as we do, $Q_k(\infty):=
1-\rho$. Therefore the renewal process $\tilde \tau$ with 
inter-arrival law $Q_k(\cdot)$ is {\sl terminating}, that is 
$\tilde \tau$ contains, almost surely, only a finite number of points.
A particularity of terminating renewals with regularly varying
inter-arrival  distribution is  the asymptotic equivalence, up
to a multiplicative factor, of inter-arrival distribution and mass renewal function
(\cite[Th.~A.4]{cf:Book}), namely
\begin{equation}
\label{eq:3stars}
{u_N}\, \stackrel{N\to \infty} \sim\, \frac 1{(1-\rho)^2} {Q_k(N)},
\end{equation}
where $u_N:= \bP (N \in \tilde \tau)$ and it satisfies the 
renewal equation $u_N= \sum_{n=1}^N u_{N-n} Q_k(n)$
for $N\ge 1$ (and $u_0=1$). Since $Q_k(n)=0$ for 
$n=1, \ldots, k-1$, for the same values of $n$ we have 
$u_n=0$ too. Therefore the renewal equation may be rewritten,
for $N \ge k$,
as
\begin{equation}
\label{eq:1star}
u_N\, =\, \sum_{n=1}^{N-(k-1)} u_{N-n} Q_k(n) \, +\, u_0 Q_k(N).
\end{equation}

Let us observe now that if we set $\tilde A_N := A_N \ind _{N \ge k}$
then \eqref{eq:rec3} implies that for $N \ge k$
\begin{equation}
\tilde A_N \, \le \,  \sum_{n=1}^{N-(k-1)}\tilde A_{N-n} Q_k(n) \, +\,
P_k(N), \ \ \ \text{ with } \ \ 
P_k(N)\, :=\, \sum_{n=0}^{k-1} A(n) Q_k(N-n),
\end{equation}
and observe  that $P_k(N) \le c \, Q_k(N)$, with
$c$ that depends on $\rho$, $\gamma$, $k$ and $K(\cdot)$
(and on $h$ and $\gb$, but these variables  are kept fixed).
Therefore
\begin{equation}
\label{eq:2stars}
\tilde A_N \, \le \,  \sum_{n=1}^{N-(k-1)}\tilde A_{N-n} Q_k(n) \, +\,
c \, Q_k(N),
\end{equation}
for $N \ge k$. 
By comparing \eqref{eq:1star} and \eqref{eq:2stars}, and
by using \eqref{eq:3stars} and 
$Q_k(N) \stackrel{N \to \infty} \sim K(N)^\gamma
\bbE[ z_1^\gamma]
\sum_{j=0}^{k-1} A_j$, 
 one directly obtains  \eqref{eq:if2}.
\qed

\subsection{Disorder relevance: sketch of the proof}
Let us consider for instance the case $\ga>1$, which is technically
less involved than the others, but still fully representative of
our strategy.  Take $(\gb,h)$ such that $\gb$ is small and
$h=h_c^{ann}(\gb)+\Delta$, with $\Delta=a\gb^2$. We are therefore
considering the system inside the annealed localized phase, but close
to the annealed critical point (at a distance $\Delta$ from it), and
we want to show that $\tf(\gb,h)=0$. In view of Proposition
\ref{th:deloc}, it is sufficient to show that $\rho$ in
\eqref{eq:if} is sufficiently small, and we have the freedom to choose
a suitable $k$.  Specifically, we choose $k$ to be of the order of the
correlation length of the annealed system:
$k=1/\tf^{ann}(\gb,h)=1/\tf(0,\Delta)\approx \textrm{const.}/(a\gb^2)$, where the
last estimate holds since the phase transition of the annealed system
is first order for $\ga>1$. Note that $k$ diverges for $\gb$ small.

For the purpose of this informal discussion, assume that
$K(n)=c\,n^{-(1+\alpha)}$, {\sl i.e.}, the slowly varying function 
$L(\cdot)$ is constant. The sum over $n$ is then immediately performed
and (up to a  multiplicative constant) one is left with estimating
\begin{equation}
\label{eq:schecci}
  \sum_{j=0}^{k-1}\frac {A_j}{(k-j)^{(1+\ga)\gga-1}}.
\end{equation}

One can choose $\gamma<1$ such $(1+\ga)\gga-1>1$ and it is actually
not difficult to show that $\sup_{j<k}A_j$ is bounded by a constant
uniformly in $k$. On one hand in fact $A_j\le [\bbE
Z_{j,\go}]^\gamma=[Z_j(\Delta)]^\gamma$, where the first step follows
from Jensen's inequality and the second one from the definition of the
model (recall \eqref{eq:Zpure}). On the other hand for $j< k$, {\sl
  i.e.}, for $j$ smaller than the correlation length of the annealed
model, one has that the annealed partition function $Z_{j}(\Delta)$ is
bounded above by a constant, {\sl independently of how small $\Delta$
  is}, {\sl i.e.}, of how large the correlation length is.  This
just establishes that the quantity in \eqref{eq:schecci} is bounded,
so we need to go beyond and show that $A_j$ is small: this of course
is not true unless $j$ is large, but if we restrict the sum in
\eqref{eq:schecci} to $j\ll k$ what we obtain is small, since the
denominator is approximately $k^{(1+\ga)\gamma -1}$, that is $k$ to a power
larger than $1$.

In order to control the terms for which $k-j$ is of order $1$ a new
ingredient is clearly needed, and we really have to estimate the
fractional moment of the partition function without resorting to
Jensen's inequality. To this purpose, we apply an idea which was
introduced in \cite{cf:GLT}. Specifically, we change the law $\bbP$ of
the disorder in such a way that under the new law, $\tilde \bbP$, the
system is delocalized and $\tilde\bbE(Z_{j,\go})^\gamma$ is small.
The change of measure corresponds to tilting negatively the law of
$\go_i, i\le j$, {\sl cf.}  \eqref{eq:tildeP}, so that the system is more
delocalized than under $\bbP$.  The non-trivial fact is that with our
choice $\Delta=a\gb^2$ and $j\le 1/\tf(0,\Delta)$, one can guarantee
on one hand that $Z_{j,\go}$ is typically small under $\tilde\bbP$,
and on the other that $\bbP$ and $\tilde \bbP$ are close (their mutual
density is bounded, in a suitable sense), so that the same statement
about $Z_{j,\go}$ holds also under the original measure $\bbP$. At
this point, we have that all terms in \eqref{eq:schecci} are small:
actually, as we will see, the whole sum is as small as we wish if we
choose $a$ small.  The fact that $\tf(\gb,h)=0$ then follows from
Proposition \ref{th:deloc}.

As we have mentioned above, the case $\ga \in [1/2, 1)$ is not much 
harder, at least on a conceptual level, but this time it is not
sufficient to establish bounds on  $A_j$ that do not depend
on $j$: the exponent in the denominator of the summand in  \eqref{eq:schecci}
is in any case smaller than $1$ and one has to exploit
the decay in $j$ of $A_j$: with respect to the 
 $\ga>1$ case, here one can exploit the decay of  $\bP (j \in \tau)$
 as $j$ grows, while such a quantity converges to a positive constant if
 $\ga >1$. 
Once again the case of $j \ll k$ can be dealt with by direct 
annealed estimates, while when one gets close to $k$ 
a finer argument, direct generalization of the one used for
the $\ga>1$ case, is needed.



\section{The case $\alpha>1$}

In order to avoid repetitions let us establish that, 
in this and next sections, $R_i, i=1,2,\ldots$  denote (large)
constants,  $L_i(\cdot)$ are slowly varying functions and $C_i$ 
positive constants (not necessarily large).

\medskip 

{\sl Proof of Theorem \ref{th:a>1}}. Fix $\beta_0>0$ 
 and let $\beta\le \beta_0$, 
$h=h_c^{ann}(\beta)+ a \beta^2$ and $\gamma<1$ sufficiently close to
$1$ so that
\begin{equation}
\label{eq:alphans}
  (1+\alpha)\gamma>2.
\end{equation}
It is sufficient to show that the sum in \eqref{eq:if} can be made
arbitrarily small (for some suitable choice of $k$) by choosing $a$
small, since $\bbE [z_1^\gamma]$ can be bounded above by a constant
independent of $a$ (for $a$ small).

We choose $k=k(\beta)=1/(a\beta^2)$, so that $\beta=1/\sqrt{a
  k(\beta)}$.  In order to avoid a plethora of $\lfloor\cdot\rfloor$,
we will assume that $k(\beta)$ is integer.  Note that $k(\beta)$ is
large if $\beta$ or $a$ are small.

First of all note that, thanks to Eqs. \eqref{eq:sommasv} and 
\eqref{eq:maxsv}, the sum in the r.h.s. of \eqref{eq:if} is
bounded above by
\begin{equation}
\label{eq:S}
\sum_{j=0}^{k(\gb)-1}
\frac{L_1(k(\gb)-j)\,A_j}{(k(\gb)-j)^{(1+\ga)\gamma-1}}.
\end{equation}

We split this sum as
\begin{equation}
\label{eq:split1}
  S_1+S_2:=  \sum_{j=0}^{k(\gb)-1-R_1}
\frac{L_1(k(\gb)-j)\,A_j}{(k(\gb)-j)^{(1+\ga)\gamma-1}}
+ \sum_{j=k(\gb)-R_1}^{k(\gb)-1}
\frac{L_1(k(\gb)-j)\,A_j}{(k(\gb)-j)^{(1+\ga)\gamma-1}}.
\end{equation}
To estimate $S_1$, note that by Jensen's inequality $A_j\le (\bbE
Z_{j,\go})^\gamma\le C_1$ with $C_1$ a constant independent of $j$ as
long as $j< k(\gb)$.
Indeed, from \eqref{eq:Z} and the definition of the annealed critical
point one sees that (recall \eqref{eq:Zpure})
\begin{equation}
\label{eq:C0}
  \bbE{Z_{j,\go}}=Z_j(a\beta^2)=\bE\left(e^{a\beta^2|\tau\cap\{1,\ldots,j\}|}
\ind_{\{j\in\tau\}}\right),
\end{equation}
and the last term is clearly smaller than $e$.
Therefore, using again \eqref{eq:sommasv}
\begin{equation}
  S_1\le 
\frac{L_2(R_1)}{R_1^{(1+\ga)\gga-2}},
\end{equation}
which can be made small with $R_1$ large in view of the 
choice \eqref{eq:alphans}.  As for $S_2$, one has
\begin{equation}
\label{eq:B2}
  S_2\, \le\, 
C_2\,\max_{k(\gb)-R_1\le j<k(\gb)}A_j.
\end{equation}
We apply now Lemma~\ref{th:bernard} (note also the definition in
 \eqref{eq:tildeP}) with $N=j$ and
$\gl= 1/\sqrt{j}$ 
so that we have 
\begin{equation}
A_j \, \le \, \left[\bbE_{j, 1/\sqrt{j}} \left( Z_{j, \go}\right)\right]^\gamma
\exp\left(c \gamma/(1-\gamma)\right), 
\end{equation}
for $1/\sqrt{j} \le \min (1, (1-\gamma)/\gamma)$, that is for $a$
sufficiently small, since we are in any case assuming $j\ge
k(\gb)-R_1$.

We are therefore left with showing that $ \bbE_{j, 1/\sqrt{j}} \left[
  Z_{j, \go}\right]$ is small for the range of $j$'s we are
  considering.  For such an estimate it is convenient to recall
  \eqref{eq:hann} and to observe that for any given values of $\beta$,
  $ h$ and $\gl$ and for any $j$
\begin{equation}
\label{eq:MMM}
\bbE_{j,\gl}[Z_{j,\go}]=\bE
\left[\left(\exp\left({h-h_c^{ann}(\gb)}\right)
\frac{\M(\gb-\gl)}{\M(\gb)\M(-\gl)}
  \right)^{|\tau\cap\{1,\ldots,j\}|}\;\ind_{\{j\in\tau\}}\right].
\end{equation}
In order to exploit such a formula let us observe that
\begin{equation}
\label{eq:MM}
   \frac{\M(\gb-\gl)}{\M(\gb)\M(-\gl)}=
\exp\left[-\int_0^\gb \dd x\int_{-\gl}^0\dd y \left.\frac{\dd^2}{\dd t^2}
\log \M(t)\right|_{t=x+y}\right]
\, \le\,  e^{-C_3\beta\gl}, 
\end{equation}
which holds for $0 < \gl \le  \gb \le \gb_0$  and $C_3:=
\min_{t \in [-\gb_0, \gb_0]} \dd^2(\log \M (t))/\dd t^2>0$.
  If $a$ is sufficiently small, for
$j \le k(\gb)=1/(a\gb^2)$ we have
\begin{equation}
a\beta^2-\frac{C_3\gb}{\sqrt{j}}\le\frac1{k(\gb)}\left[1-\frac{C_3}{\sqrt
    a}\right] \le -\frac{C_3}{2k(\gb)\sqrt a}.
\end{equation}
As a consequence,
\begin{equation}
\label{eq:law}
\max_{k(\gb)-R_1\le j<k(\gb)}
\bbE_{j, 1/\sqrt{j}}(Z_{j,\go})\,\le\, 
e^{C_3\sqrt a \gb^2 R_1/2}\,
\bE\left[\exp\left(-\frac{C_3}{2\sqrt{a}k(\gb)}
\left|\tau\cap\{1,\ldots,k(\gb)\}\right|
  \right)
\right].
\end{equation}
The right-hand side in  \eqref{eq:law} can be made small
by choosing $a$ small (and this is uniform on $\beta\le \beta_0$) because of
\begin{equation}
  \label{eq:isi}
  \lim_{c\to+\infty}\limsup_{N\to\infty} 
\bE\left(e^{-(c/N)|\tau\cap\{1,\ldots,N\}|}\right)=0,
\end{equation}
that we are going to prove just below.
Putting everything together, we have shown that both $S_1$ and $S_2$
can be made small via a suitable choice of $R_1$ and $a$, and the
theorem is proven.

To prove \eqref{eq:isi}, since the function under
expectation is bounded by $1$ it is sufficient to observe that
\begin{equation}
\frac1N
\sum_{n=1}^N\ind_{\{n\in\tau\}}\stackrel{N\to\infty}\longrightarrow
\frac1{\sum_{n\in\N}n K(n)}=\frac1{\bE(\tau_1)}\, >\,0,
\end{equation}
almost surely (with respect to $\bP$) by the classical Renewal
Theorem (or by the strong law of large numbers).

The claim $h_c(\gb)>h_c^{ann}(\gb)$ for every $\gb$ follows from the 
arbitrariness of $\beta_0$.
\qed

\section{The case $1/2<\alpha<1$}

{\sl Proof of Theorem \ref{th:a121}}.  To make things clear, we fix
now $\gep>0$ small 
and $0<\gga<1$ such that
\begin{equation}
\label{eq:gga1}
\gga\left\{(1+\ga)+(1-\gep^2)\left[1-\alpha+(\gep/2)(\ga-1/2)\right]\right\}
\, > \, 2,
  \end{equation}
  and
  \begin{equation}
   \label{eq:gga2}
 \gga\left[(1+\ga)+(1-\gep^2)(1-\ga)\right]\, >\, 2-\gep^2.
 \end{equation}
Moreover we  take $\gb\le \beta_0$ and
\begin{equation}
\label{eq:h}
h=h_c^{ann}(\gb)+\Delta:=
h_c^{ann}(\gb)+a\gb^{\frac{2\ga}{2\ga-1}(1+\gep)}.
\end{equation}
We notice that it is crucial that $(\ga-1/2)>0$ for \eqref{eq:gga1} to
be satisfied.  We will take $\gep$ sufficiently small (so that
\eqref{eq:gga1} and \eqref{eq:gga2} can occur) {\sl and then, once
  $\gep$ and $\gga$ are fixed,} $a$ also small.  We set moreover
\begin{equation}
  \label{eq:kb}
  k(\gb):= \frac1{\tf(0,\Delta)}
\end{equation}
and we notice that $k(\gb)$ can be made large by choosing $a$
small, uniformly for $\gb\le\beta_0$. As in the previous section, we
assume for ease of notation that $k(\gb)\in\N$ (and we  write just
$k$ for $k(\gb)$). 

Our aim is to show that $\tf(\gb,h)=0$ if $a$ is chosen
sufficiently small in \eqref{eq:h}.  We recall that, thanks to
Proposition \ref{th:deloc}, the result is proven if we show that
\eqref{eq:S} is $o(1)$ for $k$ large.  
In order to estimate this sum, we need a couple of technical
estimates which are proven at the end of this section (Lemma \ref{th:lemmaA})
and in Appendix \ref{sec:Arenew} (Lemma \ref{th:lemmaU2}).
\begin{lemma}
\label{th:lemmaU2} Let $\ga\in(0,1)$.
There exists a constant $C_4$ such that for every
$0<h<1$ and every 
$j\le 1/\tf(0,h)$
  \begin{equation}
     Z_{j}(h)\, \le\, \frac{ C_4}{j^{1-\ga}L(j)}.
  \end{equation}
\end{lemma}
In view of $Z_j(h_c(0))=Z_j(0)=\bP(j\in\tau)$ and \eqref{eq:DoneyB},
this means that as long as $j\le 1/\tf(0,h)$ the partition function of
the homogeneous model behaves essentially like in the (homogeneous)
critical case.  \medskip

\begin{lemma}
\label{th:lemmaA}
  There exists $\gep_0>0$ such that, if $\gep\le \gep_0$ ($\gep$ 
being the same one which appears in \eqref{eq:h}),
  \begin{equation}
    \bbE_{j,1/\sqrt{j}}[Z_{j,\go}]\le
    \frac{C_5}{j^{1-\ga+(\gep/2)(\ga-1/2)}}
  \end{equation}
  for some constant $C_5$ (depending on $\gep$ but not on $\gb$ or $a$),
  uniformly in $0\le \gb\le\gb_0$ and in $k^{(1-\gep^2)}\le j<k$.
\end{lemma}

\medskip

In order to bound above \eqref{eq:S}, we split it as
\begin{equation}
  S_3+S_4:=\sum_{j=0}^{\lfloor k^{(1-\gep^2)}\rfloor}
  \frac{L_1(k-j)\,A_j}{(k-j)^{(1+\ga)\gamma-1}}+
\sum_{j=\lfloor k^{(1-\gep^2)}\rfloor+1}^{k-1}
  \frac{L_1(k-j)\,A_j}{(k-j)^{(1+\ga)\gamma-1}}.
\end{equation}
For $S_3$ we use simply $A_j\le (\bbE
Z_{j,\go})^\gamma=[Z_j(\Delta)]^\gga$ and Lemma \ref{th:lemmaU2}, 
together with \eqref{eq:sommasv} and \eqref{eq:maxsv}:
\begin{equation}
\label{eq:S3}
  S_3\le \frac{L_3(k)}{k^{[(1+\ga)\gamma-1]}}\frac1{
k^{(1-\gep^2)((1-\ga)\gamma-1)}},
\end{equation}
where $L_3(\cdot)$ can depend on $\gep$ but not on $a$.  The second
condition \eqref{eq:gga2} imposed on $\gga$ guarantees that $S_3$ is
arbitrarily small for $k$ large, {\sl i.e.}, for $a$ small.

As for $S_4$, we use Lemma \ref{th:bernard} with $N=j$ and
$\gl=1/\sqrt{j}$ to estimate $A_j$ (recall the definition in
\eqref{eq:tildeP}). We get
\begin{align}
  A_j\le\left[\bbE_{j,1/\sqrt{j}}(Z_{j,\go})\right]^{\gga}\exp(c\gga/(1-\gga)),
  \label{eq:esti}
\end{align}
provided that 
$1/\sqrt j\le \min(1,(1-\gga)/\gga)$, which is true for all
$j\ge k^{1-\gep^2}$ if $a$ is small.  Then, provided we have chosen
$\gep\le \gep_0$, Lemma \ref{th:lemmaA} gives for every
$k^{(1-\gep^2)}< j<k$,
\begin{equation}
   A_j\le  \frac{C_{6}}{j^{[1-\ga+(\gep/2)(\ga-1/2)]\gamma}}.
\end{equation}
Note that $C_6$ is large for $\gep$ small (since from
\eqref{eq:gga1}-\eqref{eq:gga2} it is clear that $\gga$ must be close to 
$1$ for $\gep$ small) but it is independent of $a$. As a consequence,
using \eqref{eq:sommasv2},
\begin{equation}
\begin{split}
  S_4&\, \le\, \max_{k^{(1-\gep^2)}\le j<k}A_j\times
\sum_{r=1}^{k}\frac{L_1(r)}{r^{(1+\ga)\gga-1}}\le
\max_{k^{(1-\gep^2)}\le j<k}A_j\times
\frac{L_4(k)}{k^{(1+\ga)\gga-2}}\\
&\, \le\, C_{6}\,L_4(k)\,k^{2-(1+\ga)\gga-
(1-\gep^2)[1-\ga+(\gep/2)(\ga-1/2)]\gga}.
\end{split}
\end{equation}
Then, the first condition \eqref{eq:gga1} imposed on $\gga$  guarantees
that $S_4$ tends to zero when $k$ tends to infinity. 
\qed

\bigskip

{\sl Proof of Lemma \ref{th:lemmaA}.}  Using \eqref{eq:MMM} together
with the observation \eqref{eq:MM}, the definition of $\Delta$ and
of $k=k(\gb)$ in terms of $\tf(0,\Delta)$ (plus the 
behavior of $\tf(0,\Delta)$ for $\Delta$ small described
in Theorem \ref{th:pure23} (2)) one sees that for $j\le k(\gb)$
\begin{equation}
\label{eq:sole11}
  \bbE_{j,1/\sqrt{j}}[Z_{j,\go}]\le \bE\left(
e^{-C_{7}\frac{\gb}{\sqrt j}|\tau\cap
\{1,\ldots,j\}|}\,\ind_{\{j\in\tau\}}\right),
\end{equation}
uniformly for $0\le \gb\le\beta_0$. If moreover $j\ge k^{(1-\gep^2)}$
one has
\begin{equation}
\frac{\gb}{\sqrt j}\ge \frac{C_{8}}{j^{1/2+(\ga-1/2)(1+2
\gep^2)/(1+\gep)}}\ge  \frac{C_{8}}{j^{\ga-(\gep/2)(\ga-1/2)}},
\end{equation}
with $C_{8}$ independent of $a$ for $a$ small.
The condition that $\gep$ is small has been used, say, to neglect 
$\gep^2$ with respect to $\gep$.
Going back to \eqref{eq:sole11} and using Proposition \ref{th:homog}
one has then
\begin{equation}
  \bbE_{j,1/\sqrt{j}}[Z_{j,\go}]\le 
  \frac{C_{9}}{j^{1-\ga+(\gep/2)(\ga-1/2)}}.
\end{equation}
with $C_{9}$ depending on $\gep$ but not on $a$.
\qed

\section{The case $\ga=1/2$}

{\sl Proof of Theorem \ref{th:a12}.}  The proof is not conceptually
different from that of Theorem \ref{th:a121}, but here we have to
carefully keep track of the slowly varying functions, and we have to
choose $\gamma(<1)$ as a function of $k$.  Under our assumption
\eqref{eq:Lpiccola} on $L(\cdot)$, it is easy to deduce from Theorem
\ref{th:pure23} (2) that (say, for $0<\Delta<1$)
 \begin{equation}
 \label{eq:F1/2}
   \tf(0,\Delta)=\Delta^2\hat L(1/\Delta)\ge C(c,\eta)\Delta^2\,|\log 
   \Delta|^{2\eta}.
 \end{equation}
We take $\gb\le\gb_0$ and 
\begin{equation}
\label{eq:D12}
h\, =\, h_c^{ann}(\gb)+\Delta\, :=\,  h_c^{ann}(\gb)+a
\exp\left(-\gb^{-1/(\eta-1/2-\gep)}\right),
\end{equation}
 and, as in last section, $k=1/\tf(0,\Delta)=\Delta^{-2}/\hat L(1/\Delta)$.
We note also that (for $a<1$) 
\begin{equation}
\label{eq:gbge}
  \gb\ge |\log \Delta|^{-\eta+1/2+\gep}.
\end{equation}
We set $\gga=\gga(k)=1-1/(\log k)$. As $\gga$ is $k$--dependent one
cannot use \eqref{eq:sommasv} and \eqref{eq:maxsv} without care to pass 
from \eqref{eq:if} to \eqref{eq:S}, since one could in principle
have  $\gamma$-dependent (and therefore $k$-dependent) constants
in front.
Therefore, our first aim will be to (partly) get rid of $\gga$ in
\eqref{eq:if}. We notice that for any $j\le k-1$, for $k$ such that
$\gga(k)\ge 5/6$,
\begin{align}
\label{eq:tispiezzoindue}
  \sum_{n=k}^{\infty}K(n-j)^{\gamma}&\le
  \sum_{n=k-j}^{k^6}K(n)\exp\left[(3/2\log n-\log L(n))/\log
    k\right]+\sum_{n=k^6+1}^{\infty}[K(n)]^{5/6}.
\end{align}
Now, properties of slowly varying functions guarantee that the quantity
in the exponential in the first sum is bounded (uniformly in $j$ and
$k$). As for the second sum, \eqref{eq:sommasv} guarantees it is
smaller than $k^{-6/5}$ for $k$ large. Since by Lemma \ref{th:lemmaU2}
the $A_j$ are bounded by a constant in the regime we are considering,
when we reinsert this term in \eqref{eq:if} and we sum over $j<k$ we
obtain a contribution which vanishes at least like $k^{-1/5}$ for
$k\to\infty$ . We will therefore forget from now on the second sum in
\eqref{eq:tispiezzoindue}.

Therefore one has
\begin{align}
\rho  \le C_{10}\sum_{n=k}^{\infty}\sum_{j=0}^{k-1}K(n-j)A_j
      \le C_{11}\sum_{j=0}^{k-1}\frac{L(k-j)A_j}{(k-j)^{1/2}},
\end{align}     
where we have safely used \eqref{eq:sommasv} to get the second expression
and now $\gamma$ appears only (implicitly) in the fractional moment
$A_j$ but not in the constants $C_i$.

Once again, it is convenient to split this sum into
\begin{equation} 
  S_5+S_6:=\sum_{j=0}^{k/R_2}\frac{A_j\, L(k-j)}{ (k-j)^{1/2}}
+\sum_{j=(k/R_2)+1}^{k-1}\frac{A_j\,L(k-j)}{(k-j)^{1/2}},
\end{equation}
with $R_2$ a large constant.  To bound $S_5$ we simply use Jensen
inequality to estimate $A_j$. Lemma \ref{th:lemmaU2} gives that for
all $j\le k$,
\begin{align}
A_j\le \frac{C_{12}}{j^{\gga/2}L(j)^\gga}\le \frac{C_{13}}{\sqrt j L(j)},
\end{align}
where the second inequality comes from our choice $\gamma=1-1/(\log
k)$. Knowing this, we can use \eqref{eq:sommasv} to compute $S_5$ and
get
 \begin{equation}
\label{eq:regvar}
S_5\le \frac{C_{14}}{\sqrt{R_2}}\frac{L(k(1-1/R_2))}
{L(k/R_2)}.
\end{equation}
We see that $S_5$ can be made small
choosing $R_2$ large. It is important for the following to note that
it is sufficient to choose $R_2$ large but independent of $k$; in
particular, for $k$ large at $R_2$ fixed the last factor in
\eqref{eq:regvar} approaches $1$ by the property of slow variation of
$L(\cdot)$. As for $S_6$,
\begin{equation}
  S_6\le C_{15}\max_{k/R_2< j<k}A_j\times \sqrt k\,L(k).
\end{equation}
In order to estimate this maximum, we need to refine Lemma \ref{th:lemmaA}:
\begin{lemma}
  \label{th:lemmaa12} 
There exists a constant
$C_{16}:=C_{16}(R_2)$ such that for  $\gamma=1-1/(\log k)$ and $k/R_2< j<k$
  \begin{equation}
    A_j\le C_{16}\left(L(j)\sqrt j\,(\log j)^{2\gep}\right)^{-1}.
  \end{equation}
\end{lemma}
Given this, we obtain immediately
\begin{equation}
  S_6 \, \le\,  C_{17}(R_2)\left[\log
\left(\frac k{R_2}\right)\right]^{
-2\gep}.
\end{equation}
It is then clear that $S_6$ can be made
arbitrarily small with $k$ large, {\sl i.e.}, with $a$ small.
\qed

\medskip

{\sl Proof of Lemma \ref{th:lemmaa12}.}
Once again, we use Lemma \ref{th:bernard} with $N=j$ but this time $\gl=(j\log
j)^{-1/2}$. Recalling that $\gga=1-1/(\log k)$  we obtain
\begin{align}
  A_j\le \left[\bbE_{j,(j\log
      j)^{-1/2}}(Z_{j,\go})\right]^\gga\exp\left(c\frac{\log k}{\log
      j}\right), \label{eq:undemi}
\end{align}
for all $j$ such that $(j\log j)^{1/2} \ge \log k$. The latter
condition is satisfied for all $k/R_2<j<k$ if $k$ is large enough.
Note that, since $j>k/R_2$, the exponential factor in \eqref{eq:undemi}, is
bounded by a constant 
$C_{18}:=C_{18}(R_2)$.

Furthermore, for $j\le k$, Eqs. \eqref{eq:MMM}, \eqref{eq:MM} combined give
\begin{align}\label{eq:ss}
\bbE_{j,(j\log j)^{-1/2}}[Z_{j,\go}]\le Z_j\left(-C_{19}\gb(j\log j)^{-1/2}\right),
\end{align}
for some positive constant $C_{19}$, provided $a$ is small (here we
have used \eqref{eq:F1/2} and the definition $k=1/\tf(0,\gD)$).

In view of $j\ge k/R_2$, the definition of $k$ in terms of $\gb$ and 
assumption \eqref{eq:Lpiccola},
we see that 
\begin{equation}
\gb\ge C_{20}(\log j)^{(-\eta+1/2+\gep)}
\ge \frac {C_{21}}c L(j)(\log j)^{1/2+\gep},
\end{equation} 
so that the r.h.s. of \eqref{eq:ss} is bounded above by
\begin{equation}
  Z_j\left(-C_{21}\frac{L(j)}{c\sqrt{j}}(\log j)^{\gep}\right)
\le C_{22}\frac{(\log j)^{-2\gep}}{L(j)\,\sqrt j} ,
\end{equation}
where in the last inequality we used Lemma \ref{th:homog}. The result
is obtained by re-injecting this in \eqref{eq:undemi}, and using the
value of $\gamma(k)$.

\qed

\begin{subappendices}

\section{Frequently used bounds}
\label{sec:bounds}

\subsection{Bounding the partition function via tilting}
For $\gl \in \R$ and $N \in \N$ consider the probability measure
$\bbP_{N, \gl}$ defined by 
\begin{equation}
\label{eq:tildeP}
\frac{
\dd \bbP_{N, \gl}}{\dd \bbP} (\go)\, =\,
\frac1 {\M(-\gl)^N} 
\exp\left(-\gl \sum_{i=1}^N \go_i \right),
\end{equation} 
where $\M(\cdot)$ was defined in \eqref{eq:M23}.
Note that under $\bbP_{N,\gl}$ the random variables $\go_i$ are
still independent but no more identically distributed: the law of
$\go_i, i\le N$ is tilted while $\go_i, i>N$ are distributed exactly
like under $\bbP$.

\medskip

\begin{lemma}
\label{th:bernard}
There exists $c>0$ such that, for every $N \in \N $ and $\gamma \in
(0,1)$,
\begin{equation}
\label{eq:bernard}
\bbE \left[ \left(Z_{N, \go}\right)^\gamma 
\right]\, \le \,
\left[\bbE_{N , \gl} \left(Z_{N, \go}\right)
\right]^\gamma
\, \exp \left( c \left( \frac{\gga}{1-\gamma} \right)\gl ^2 N
\right),
\end{equation}
for $\vert \gl \vert \le \min(1,(1-\gamma)/\gamma)$. 
\end{lemma}
\medskip

\noindent
{\it Proof.}
 We have
\begin{equation}
\begin{split}
\bbE \left[ \left(Z_{N, \go}\right)^\gamma 
\right]\, &=\, 
\bbE_{N , \gl} \left[ \left(Z_{N, \go}\right)^\gamma 
\frac{\dd \bbP}{\dd \bbP_{N , \gl}}(\go)
\right]
\\
&\le \,
 \left[\bbE_{N , \gl} \left(Z_{N, \go}\right)
\right]^\gamma
\left(
\bbE_{N,\gl}\left[ \left(\frac{\dd \bbP}{\dd \bbP_{N , \gl}}
(\go)\right)^{1/(1-\gamma)}\right]\right)^{1-\gamma}
\\
&=\, \left[\bbE_{N , \gl} \left(Z_{N, \go}\right)
\right]^\gamma
\left( \M (-\gl)^\gga \M\left( \gl \gga  /(1-\gamma)\right)^{1-\gamma}\right)^N, 
\end{split}
\end{equation}
where in the second step we have used 
H\"older inequality and the last step is a direct computation.
The proof is complete once we observe that
$0 \le \log \M (x) \le  c x^2$ for $\vert x\vert \le 1$
if $c$ is the maximum of the   second derivative of $(1/2)\log \M (\cdot)$
over $[-1, 1]$.

\qed

\subsection{Estimates on the renewal process}
\label{sec:Arenew}

With the notation \eqref{eq:Zpure} one has
\medskip

\begin{proposition}\label{th:homog} Let $\alpha\in(0,1)$ and 
 $r(\cdot)$ be a function diverging at infinity and such that
\begin{equation}
\label{eq:condr}
\lim_{N\to\infty}\frac{r(N)L(N)}{N^{\ga}}=0.
\end{equation}
For the homogeneous pinning model,
\begin{equation}
Z_{N}(-N^{-\ga}L(N)r(N))\stackrel{N\to\infty}\sim
\frac{N^{\ga-1}}{L(N)\,r(N)^2}.
\end{equation}
\end{proposition}
\medskip

To prove this result we use:
\medskip

\begin{proposition}(\cite[Theorems A \& B]{cf:Doney})\label{th:Don}
  Let $\ga\in(0,1)$.  There exists a function $\gs(\cdot)$ satisfying
\begin{equation}
\lim_{x\rightarrow+\infty} \gs(x)\, =\, 0 ,
\end{equation}
and such that 
for all $n,N\in\N$
\begin{equation}
\label{eq:fD}
\left|\frac{\bP(\tau_n=N)}{n K(N)}-1\right|\le \gs\left(\frac{N}{a(n)}\right),
\end{equation}
where $a(\cdot)$ is an asymptotic inverse of $x \mapsto x^\ga /L(x)$.

Moreover, 
\begin{equation}
  \label{eq:DoneyB}
  \bP(N\in\tau)\stackrel{N\to\infty} \sim 
  \left(
  \frac{\ga \sin(\pi \ga)}{\pi}\right)
   \frac{N^{\ga-1}}{L(N)}.
\end{equation}
\end{proposition}
\medskip

We observe that by \cite[Th.~1.5.12]{cf:RegVar}
we have that $a(\cdot)$ is regularly varying of exponent 
$1/\ga$, in particular $\lim_{n \to \infty}a(n)/n^b=0$ if $b >1/\ga$. 
We point out also that \eqref{eq:DoneyB}
has been first established for $\ga \in (1/2,1)$
in \cite{cf:GarsiaLamperti}.

\smallskip

\begin{proof}[Proof of Proposition \ref{th:homog}]
    We put for simplicity of notation $v(N):=N^{\ga}/L(N)$.
  Decomposing $Z_{N }$ with respect to the cardinality of $\tau\cap
  \{1,\ldots,N\}$,
\begin{equation}
\label{eq:lastline}
\begin{split}
  Z_{N}(-r(N)/v(N))\, &=\,\sum_{n=1}^{N}
  \bP\left(|\tau\cap\{1,\ldots,N\}|=n,N\in
  \tau\right)e^{-n\,r(N)/v(N)}\\
  &=\, \sum_{n=1}^{N}\bP(\tau_n=N)e^{-n\,r(N)/v(N)}\\
&=\, 
  \sum_{n=1}^{\frac{v(N)}{\sqrt{r(N)}}}\bP(\tau_n=N)e^{-n\,\frac{r(N)}{v(N)}}+
  \sum_{n=\frac{v(N)}{\sqrt{r(N)}}+1}^N\bP(\tau_n=N)e^{-n\,\frac{r(N)}{v(N)}}.
  \end{split}
\end{equation}
Observe now that one can rewrite  the first term in  the last line of \eqref{eq:lastline}  as
\begin{equation}
  (1+o(1))K(N) \sum_{n=1}^{v(N)/\sqrt{r(N)}}n \,e^{-n\,r(N)/v(N)},
\end{equation}
and $o(1)$ is a quantity which vanishes for $N\to\infty$
(this follows from Proposition
\ref{th:Don}, which applies uniformly over all terms of the sum in view
of $\lim_N r(N)=\infty$).  Thanks to
condition \eqref{eq:condr}, one can estimate this sum by an integral:
\begin{align*}
  \sum_{n=1}^{v(N)/\sqrt{r(N)}}n\,
  e^{-n\,r(N)/v(N)}=\frac{v(N)^2}{r(N)^2} (1+o(1))\int_{0}^\infty\dd x\,
  x\,e^{-x}= \frac{v(N)^2}{r(N)^2}(1+o(1)).
\end{align*}
As for the second sum in \eqref{eq:lastline}, observing that
$\sum_{n\in\N}\bP(\tau_n=N)=\bP(N\in\tau)$, we can bound it above by
\begin{equation}
\bP(N\in\tau)e^{-\sqrt{r(N)}}.
\end{equation}
In view of  \eqref{eq:DoneyB}, the last term is
negligible with respect to $N^{\ga-1}/(L(N)\,r(N)^2)$
and our result is proved.
\end{proof}

\noindent
{\it Proof of Lemma \ref{th:lemmaU2}.}
Recalling the notation
\eqref{eq:Zpure},  point (2) of Theorem \ref{th:pure23} (see in 
particular the definition of $\hat L(\cdot)$) and  
\eqref{eq:DoneyB},
we see that the result we are looking for follows if we can show that
for every $c>0$ there exists $C_{23}=C_{23}(c)>0$ such that
\begin{equation}
\bE \left[e^{c|\tau\cap\{1,\ldots,N\}|L(N)/N^\ga}\Big \vert \, N\in\tau
\right]\, \le C_{23},
\end{equation}
uniformly in $N$.
Let us assume that $N/4\in \N$; by Cauchy-Schwarz inequality the
result follows if we can show that
\begin{equation}
\label{eq:withc}
\bE \left[e^{2c|\tau\cap\{1,\ldots,N/2\}|L(N)/N^\ga}\Big \vert \, N\in\tau
\right]\, \le C_{24}.
\end{equation}
Let us define $X_N :=\max\{n=0,1, \ldots, N/2:\, n \in \tau\}$
(last renewal epoch up to $N/2$).
By the renewal property  we have 
\begin{multline}
\bE \left[e^{2c|\tau\cap\{1,\ldots,N/2\}|L(N)/N^\ga}\Big \vert \, N\in\tau
\right]\\ 
\, =\, \sum_{n=0}^{N/2} 
\bE \left[e^{2c|\tau\cap\{1,\ldots,N/2\}|L(N)/N^\ga}\Big \vert \,X_N=n
\right] \bP\left( X_N=n \big \vert \, N \in \tau\right).
\end{multline}
If we can show that for every $n =0, 1, \ldots, N/2$
\begin{equation}
\bP\left( X_N=n \big \vert \, N \in \tau\right) \le C_{25} 
\bP\left( X_N=n\right),
\end{equation}
then we are reduced to proving \eqref{eq:withc} with
$\bE[\cdot \vert N \in \tau]$ replaced by $\bE[\cdot]$.

Let us then observe that
\begin{equation}
\label{eq:justbefore}
\begin{split}
\bP\left( X_N=n ,  \, N \in \tau\right)\, &=\,
\bP(n \in \tau) \bP\left( \tau_1> (N/2) -n, \, N-n \in \tau\right)
\\ 
&=\,\bP(n \in \tau)
\sum_{j= (N/2)-n+1}^{N-n} \bP( \tau_1=j)  \bP \left( N-n-j \in \tau\right).
\end{split}
\end{equation}
We are done if we can show that
\begin{equation}
\label{eq:tobesplit}
\sum_{j= (N/2)-n+1}^{N-n} \bP( \tau_1=j)  \bP \left( N-n-j \in \tau\right)
\,  \le \, 
C_{26}
\bP \left( N \in \tau\right)
\sum_{j= (N/2)-n+1}^{\infty} \bP( \tau_1=j),
\end{equation}
because the mass renewal function $\bP(N\in\tau)$ cancels when we consider the
conditioned probability and, recovering $\bP(n \in \tau)$ from
\eqref{eq:justbefore} we rebuild $\bP(X_N=n)$.  We split the sum in
the left-hand side of \eqref{eq:tobesplit} in two terms.  By using
\eqref{eq:DoneyB} (but just as upper bound) and the fact that the
inter-arrival distribution is regularly varying we obtain
\begin{multline}
\label{eq:term2}
\sum_{j= (3N/4)-n+1}^{N-n} \bP( \tau_1=j)  \bP \left( N-n-j \in \tau\right)
\\
\,  \le \, C_{27} \frac{L(N)}{N ^{1+\ga }} 
\sum_{j= (3N/4)-n+1}^{N-n} \frac1{(N-n-j+1)^{1-\ga }L(N-n-j+1)}
\\ =\, 
C_{27} \frac{L(N)}{N ^{1+\ga }} 
\sum_{j= 1}^{N/4}
\frac1{j^{1-\ga }L(j)}\, \le \, \frac{C_{28}}N. 
\end{multline}
Since the right-hand side of \eqref{eq:tobesplit} is bounded below by
$1/N$ times a suitable constant (of course if $n$ is 
close to $N/2$ this quantity is sensibly larger) this first term of the
splitting is under control.  Now the other term: since the renewal
function is regularly varying
\begin{equation}
\label{eq:term1}
\sum_{j= (N/2)-n+1}^{(3N/4)-n} \bP( \tau_1=j) \bP \left( N-n-j \in
\tau\right) \, \le \, C_{29} \bP \left( N \in \tau\right)\sum_{j=
(N/2)-n+1}^{(3N/4)-n} \bP( \tau_1=j),
\end{equation}
that gives what we wanted.

It remains to show that \eqref{eq:withc} holds without conditioning.
For this we use the asymptotic estimate $-\log \bE[\exp(- \gl \tau_1)]
\stackrel{\gl \searrow 0}\sim c_\ga \gl ^\ga L(1/\gl)$, with $c_\ga =
\int_0^\infty r^{-1-\ga} (1-\exp(-r)) \dd r=\Gamma(1-\ga)/\ga$, and
the Markov inequality to get that if $x>0$
\begin{equation}
\label{eq:withoutc}
\bP\left( 
|\tau\cap\{1,\ldots,N\}|L(N)/N^\ga >x
\right)
\, =\, 
\bP\left( \tau_n <N \right) \, \le \, 
\exp \left( - \frac 12 c_\ga \gl^\ga L(1/\gl) n  + \gl N\right),
\end{equation}
with $n$  the integer part of $x N^\ga /L(N)$ and 
$\gl \in (0, \gl_0)$ for some $\gl_0>0$.  
If one chooses  $\gl = y/N$, $y$ a positive number, then for 
$x\ge 1$ and $N$ sufficiently large (depending on $\gl_0$ and $y$)
we have that the quantity at the exponent in the right-most term in
\eqref{eq:withoutc} is bounded above by
$ -(c_\ga/3) y^\ga x  +y$. 
The proof is then complete if we select $y$ such that
$(c_\ga/3) y^\ga> 2c$ ($c $ appears in \eqref{eq:withc}) since
if $X$ is a non-negative random variable and $q$ is a real number
$\bE[\exp(qX)] =1+q\int_0^\infty e^{qx} \bP(X>x) \dd x$.

 \subsection{Some basic facts about slowly varying functions}

\label{sec:Asv}

We recall here some of the elementary properties 
of slowly varying functions which
we repeatedly use, and we refer to \cite{cf:RegVar} for a complete
treatment of slow variation.

The first two well-known facts are that, if $U(\cdot)$ is slowly
varying at infinity,
\begin{equation}
\label{eq:sommasv}
  \sum_{n\ge N}\frac{U(n)}{n^m}\stackrel{N\to\infty}\sim U(N)\frac{N^{1-m}}
{m-1},
\end{equation}
if $m>1$ and 
\begin{equation}
  \label{eq:sommasv2}
  \sum_{n=1}^N \frac{U(n)}{n^m}\stackrel{N\to\infty}\sim U(N)\frac{N^{1-m}}{1-m},
\end{equation}
if $m<1$ ({\sl cf.} for instance \cite[Sec. A.4]{cf:Book}). 
 The second two facts are that
({\sl cf.} \cite[Th. 1.5.3]{cf:RegVar})
\begin{equation}
  \label{eq:minsv} 
\inf_{n\ge N} U(n) n^m\stackrel{N\to\infty}\sim U(N)\,N^m,
\end{equation}
if $m>0$, and
\begin{equation}
  \label{eq:maxsv}
\sup_{n\ge N} U(n) n^m\stackrel{N\to\infty}\sim U(N)\,N^m, 
\end{equation}
if $m<0$.

\end{subappendices}

\qed

\bigskip


\noindent
{\bf Note added in proof.}
After  this work has appeared in preprint form
(arXiv:0712.2515 [math.PR]), the results have been improved
\cite{cf:AZ}. It has been shown in particular that 
when $L(\cdot)$ is trivial, then $\gep$ in Thereom~\ref{th:a121} can be chosen
equal to zero, with $a(0)>0$. The case
$\ga=1$ is also treated in \cite{cf:AZ}. 
The method we have developed here may be adapted
to deal with the $\ga=1$ case too: this has been 
done in \cite{cf:BS}, where a related model is treated.

\chapter{Marginal relevance of disorder for pinning models}\label{MARGREL}

\section{Introduction}

\subsection{Wetting and pinning on a defect line in $(1+1)$-dimensions}
\label{sec:intro1}
The intense activity aiming at understanding phenomena like wetting in
two dimensions \cite{cf:Abraham} 
and pinning of polymers by a defect line \cite{cf:FLN} has
led several people to focus on a class of simplified models based on
random walks. In order to describe more realistic, spatially
inhomogeneous situations, these models include disordered
interactions.  While a very substantial amount of work has been done,
it is quite remarkable that some crucial issues are not only
mathematically open (which is not surprising given the presence of
disorder), but also controversial in the  physics literature.

Let us start by introducing the most basic, and most studied, model in
the class we consider (it is the case considered 
in \cite{cf:FLNO,cf:DHV}, but also in \cite{cf:BM,cf:GN,cf:GS1,cf:GS2,cf:SC,cf:TangChate},  up to some inessential details, although 
the notations used by the various authors are quite different).
Let $S=\{S_0,S_1, \ldots\}$ be a simple
symmetric random walk on $\bbZ$, {\sl i.e.}, $S_0=0$ and $\{ S_n-S_{n-1}\}_{n \in
  \bbN}$ is an IID sequence (with law $\bP$) of random
variables taking values $\pm 1$ with probability $1/2$.  
It  is  better to take a directed walk viewpoint, that is to consider
the process $\{ (n, S_n)\}_{n=0,1, \ldots}$.
This random
walk is the {\sl free model} and we want to understand
the situation where the walk interacts with a substrate or with a defect
line that provides {\sl disordered} ({\sl e.g.}  random)
rewards/penalties each time  the walk hits it (see Fig.~\ref{fig:SRWwetting}).  The walk
may or may not be allowed to take negative values: we call {\sl
  pinning on a defect line} the first case and {\sl wetting of a
  substrate} the second one. 
  It is by now well understood that these two
cases are equivalent and we briefly discuss the wetting case only in
the caption of Figure~\ref{fig:SRWwetting}: the general model we will
consider covers both wetting and pinning cases.  The interaction is
introduced via the Hamiltonian
\begin{equation}
\label{eq:Haus}
  H_{N,\go}(S):=-\sum_{n=1}^{N}\left(\gb\go_n+h-\log \bbE(\exp(\gb \go_1))\right)\ind_{\{S_n=0\}},
\end{equation}
where $N\in 2\N$ is the system size, $h$ (homogeneous pinning
potential) is a real number, $\go:=\{\go_1,\go_2,\ldots\}$ is a
sequence of IID centered random variables with finite exponential
moments (in this work, we will restrict
to the Gaussian case), $\beta\ge0$ is the disorder strength and $\bbE$
denotes the average with respect to $\go$. It will be soon clear what
is the notational convenience in introducing the non-random term $\log
\bbE(\exp(\gb \go_1))$ (which could be absorbed into $h$ anyway).

\begin{figure}[ht]
\begin{center}
\leavevmode
\epsfxsize =10.4 cm
\psfragscanon
\psfrag{0}{$0$}
\psfrag{N}{$N$}
\psfrag{Sn}{$S_n$}
\psfrag{n}{$n$}
\psfrag{t0}{$\!\!\!\!\!\tau_0(=0)$}
\psfrag{t1}{$\tau_1$}
\psfrag{t2}{$\tau_2$}
\psfrag{t3}{$\tau_3$}
\psfrag{t4}{$\tau_4$}
\psfrag{t5}{$\tau_5(=N/2)$}
\psfrag{om2}{$\tilde\go_2$}
\psfrag{om8}{$\tilde\go_8$}
\psfrag{o12}{$\tilde\go_{12}$}
\psfrag{o14}{$\tilde\go_{14}$}
\psfrag{o16}{$\tilde\go_{16}$}
\psfrag{hlabel}{$\tilde\go_{n}\,:=\, \gb \go_n +h -\log \bbE \exp(\gb \go_1)$}
\psfrag{pinning}{\small trajectory of the pinning model}
\psfrag{wetting}{\small trajectory of the wetting model}
\epsfbox{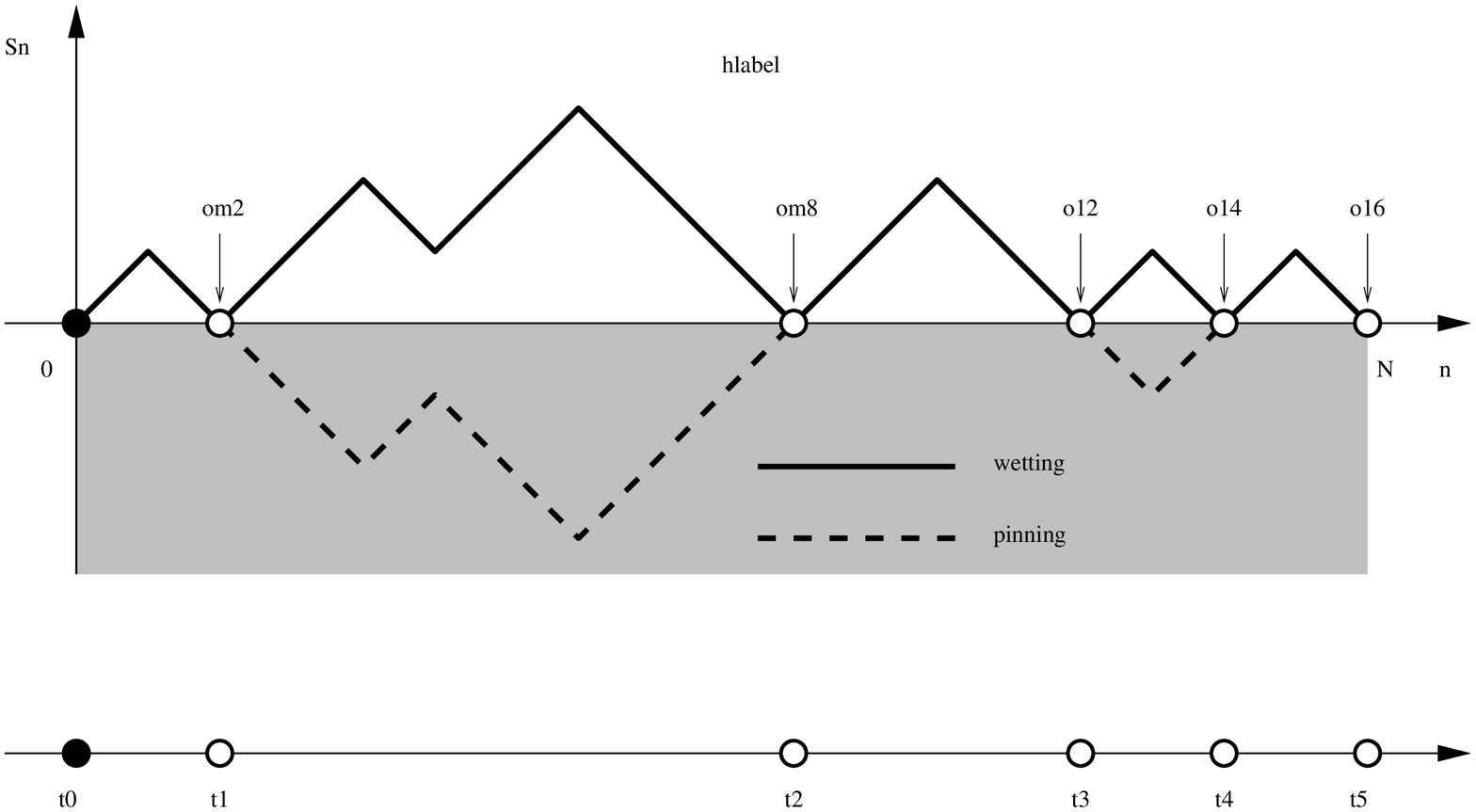}
\end{center}
\caption{In the top a random walk trajectory, pinned at $N$, which
 is not allowed to enter the lower half-plane (the shadowed region
should be regarded as a wall). The trajectory collects the {\sl
charges} $\tilde \go_n$ when it hits the wall.  The question is
whether the rewards/penalties collected pin the walk to the wall or
not. The precise definition of the wetting model is obtained by multiplying
the numerator in the right-hand side of \eqref{eq:Paus} by the indicator
function of the event $\{S_j \ge 0, \, j=1, \ldots, N\}$ (and consequently
modifying the partition function $Z_{N,\go}$).  This model
is actually equivalent to the model  \eqref{eq:Paus} without a wall, 
 whose
trajectories  (dashed line) can visit the lower half plane, provided that
$h $ is replaced by $h-\log 2$ (see \cite[Ch.~1]{cf:Book}). The bottom
part of the figure illustrates the simple but crucial point that the
energy of the model depends only on the location of the points of
contact between walk and wall (or defect line); such points
form a renewal process, giving thus a natural generalized framework in
which to tackle the problem. In order to circumvent the annoying
periodicity two of the simple random walk we set $\tau_0=0$ and
$\tau_{j+1}:= \inf\{ n/2\ge \tau_j:\, S_n=0\}$. From the renewal 
process standpoint, introducing a wall just leads to a {\sl terminating} renewal
(see text).}
\label{fig:SRWwetting}
\end{figure}

The Gibbs measure $\bP_{N,\go}$ for the pinning model is then defined as
\begin{equation}
\label{eq:Paus}
\frac{\dd \bP_{N,\go}}{\dd \bP}(S)=\frac{e^{-H_{N,\go}(S)}\ind_{\{S_N=0\}}}{Z_{N,\go}}
\end{equation}
and of course $Z_{N,\go}:=\bE[\exp(-H_{N,\go}(S))\ind_{\{S_N=0\}}]$,
where $\bE$ denotes expectation with respect to the simple
random walk  measure  $\bP$.
Note that we imposed the boundary condition $S_N=0=S_0$ (just to be
consistent with the rest of the paper).  
It is well known that the model undergoes a
localization/delocalization transition as $h$ varies: if $h$ is
larger than a certain threshold value $h_c(\beta)$ ({\sl quenched
  critical point}) then, under the Gibbs measure, the system is {\sl
  localized}: the contact fraction, defined as
\begin{equation}
\frac1N\bE_{N,\go}\left[\sum_{n=1}^N\ind_{\{S_n=0\}}\right],
\end{equation}
tends to a positive limit for $N\to\infty$. 
On the other hand, for $h<h_c(\gb)$ the system is {\sl delocalized}, {\sl i.e.}, the limit is zero.

The result we just stated is true also in absence of disorder
($\gb=0$) and a remarkable fact for the homogeneous ({\sl i.e.}
non-disordered) model is that it is exactly solvable (\cite{cf:Fisher,cf:Book} and 
references therein).  In particular, 
we know that $h_c(0)=0$, {\sl i.e.}, an arbitrarily small
reward is necessary and sufficient for pinning, and that the free
energy behaves quadratically close to criticality.  If now we consider
the {\sl annealed measure} corresponding to \eqref{eq:Paus}, that is
the model in which one replaces both $\exp(-H_{N,\go}(S))$ and
$Z_{N,\go}$ by their averages with respect to $\go$, one readily
realizes that the annealed model is a homogeneous model, and
precisely the one  we obtain by setting $\gb=0$ in
\eqref{eq:Paus}.  Therefore one finds that the {\sl annealed critical
  point} $h_c^a(\gb)$ equals $0$ for every $\beta$, and that the {\sl annealed free
  energy} $\tf^a(\gb,h)$ behaves, for $h\searrow 0$, like
$\tf^a(\gb,h)\sim const\times h^2$, while it is zero for $h\le 0$.

Very natural questions are: does $h_c(\gb)$ differ
from $h_c^a(\gb)$? Are quenched and annealed critical exponents
different?  As we are going to explain, the first question finds contradictory answers in the
literature, while no clear-cut statement can really be found about the
second. Below we are going to argue that 
these two questions are intimately related, but first we make
 a short detour in order to define a more
general class of models. It is in this more general context that the
role of the disorder and the specificity of the simple random walk
case can be best appreciated.

\subsection{Reduction to renewal-based models}
As argued in the caption of Figure~\ref{fig:SRWwetting}, the basic
underlying process is the {\sl point process} $\tau:=\{\tau_0, \tau_1,
\ldots\}$, which is a renewal process (that is $\{
\tau_{n}-\tau_{n-1}\}_{n \in \N}$ is an IID sequence of 
integer-valued random variables). We set
$K(n):=\bP(\tau_1=n)$. It is well known that, for the simple random
walk case, $\sum_{n\in\N} K(n)=1 $ (the walk is recurrent) and
$K(n)\stackrel{n \to \infty} \sim 1/(\sqrt{4\pi}n^{3/2})$. This
suggests the natural generalized framework of  models based
on discrete renewal processes such that
\begin{equation}
\label{eq:Kintro}
\sum_{n\in \N} K(n) \le 1 \ \text { and } \ K(n)\stackrel{n \to \infty} \sim \frac{C_K}{n^{1+\ga}},
\end{equation}
with $C_K>0$ and $\ga>0$. 
We are of course employing
the standard notation $a_n \sim b_n$ for 
$\lim a_n/b_n =1$.
The case $\sum_{n\in \N} K(n) < 1$ 
 refers to transient (or {\sl terminating}) renewals (of which the wetting 
 case is an example), see also 
Remark \ref{rem:trans} below. This
framework includes for example the simple random walk in $d \ge 3$, for which 
$\sum_{n\in \N} K(n) < 1$ and $\ga =(d/2)-1$, but it is of course
much more general.
We will come back with more details on this model, but let us
just say now that the definition of the Gibbs measure
is given in this case by \eqref{eq:Haus}-\eqref{eq:Paus}, with
$S$ replaced by $\tau$ in the left-hand side and with
the event $\{S_n=0\}$ replaced by the event 
$\{\text{there is } j \text{ such that } \tau_j=n\}$.

\subsection{Harris criterion and disorder relevance: the state of the art}
The questions mentioned at the end of Section \ref{sec:intro1} are
typical questions of disorder relevance, {\sl i.e.}, of stability of
critical properties with respect to (weak) disorder. In
renormalization group language, one is asking whether or not disorder
drives the system towards a new fixed point.  A heuristic tool which
was devised to give an answer to such questions is the  {\sl
Harris criterion} \cite{cf:Harris}, originally proposed for random
ferromagnetic Ising models. The Harris criterion states that disorder
is relevant if the specific heat exponent of the pure system is
positive, and irrelevant if it is negative. In case such critical
exponent is zero (this is called a {\sl marginal case}), the Harris
criterion gives no prediction and a case-by-case delicate analysis is
needed.  Now, it turns out that the random pinning model described
above is a marginal case, and from this point of view it is not
surprising that the question of disorder relevance is not solved yet,
even on heuristic grounds: in particular, the authors of
\cite{cf:FLNO} (and then also \cite{cf:GS1,cf:GS2} and, very recently,
\cite{cf:GN}) claimed that for small $\beta$ the quenched critical
point coincides with the annealed one (with our conventions, this
means that both are zero), while in \cite{cf:DHV} it was concluded
that they differ for every $\beta>0$, and that their difference is of
order $\exp(-const/\beta^2)$ for $\beta$ small (we mention
\cite{cf:BM,cf:SC,cf:TangChate} which  support this 
second possibility). Note that such a quantity is smaller than any
power of $\beta$, and therefore vanishes at all orders in
weak-disorder perturbation theory (this is also typical of marginal
cases).

In an effort to reduce the problem to its core, beyond the
difficulties connected to the random walk or renewal structure, a {\sl
  hierarchical pinning model}, defined on a diamond lattice,
   was introduced in \cite{cf:DHV}. In this
case, the laws of the partition functions for the systems of size $N$
and $2N$ are linked by a simple recursion. The role of $\alpha$ is
played here by a real parameter $B\in(1,2)$, which is related to the
geometry of the hierarchical lattice.  Also in this case, the Harris
criterion predicts that disorder is relevant in a certain regime
(here, $B<B_c:=\sqrt 2$) and irrelevant in another ($B>B_c$), while
$B=B_c$ is the marginal case where the specific heat critical exponent
of the pure model vanishes.  Again, the authors of \cite{cf:DHV}
predicted that disorder is marginally relevant for $B=B_c$, and that
the difference between annealed and quenched critical point behaves
like $\exp(-const/\beta^2)$ for $\beta$ small (they gave also
numerical evidence that the critical exponent is modified by
disorder).

Let us mention that hierarchical models based on diamond lattices have played an important role in 
elucidating the effect of disorder on various statistical mechanics 
models: we mention for instance \cite{cf:DG}.

The mathematical comprehension of the question of disorder relevance
in pinning models has witnessed remarkable progress lately.  First of
all, it was proven in \cite{cf:GT_cmp} that an arbitrarily weak (but
extensive) disorder changes the critical exponent if $\ga>1/2$ (the
analogous result for the hierarchical model was proven in
\cite{cf:LT}). Results concerning the critical points came later: in
\cite{cf:Ken,cf:T_cmp} it was proven that if $\ga<1/2$ then
$h_c(\gb)=0$ (and the quenched critical exponent coincides with the
annealed one) for $\gb$ sufficiently small (the analogous result for
the hierarchical model was given in \cite{cf:GLT}). Finally, the fact
that $h_c(\gb)>0$ for every $\gb>0$ (together with the correct
small-$\gb$ behavior) in the regime where the Harris criterion
predicts disorder relevance was proven in \cite{cf:GLT} in the
hierarchical set-up, and then in \cite{cf:DGLT,cf:AZ} in the
non-hierarchical one. One can therefore safely say that the
comprehension of the relevance question is by now rather solid, {\sl
  except in the marginal case} (of course some problems remain
open, for instance the determination of the value of the quenched
critical exponent in the relevant disorder regime, beyond the bounds
proved in \cite{cf:GT_cmp}).

 \subsection{Marginal relevance of disorder}
 In this work, we solve the question of disorder relevance for the
 marginal case $\alpha=1/2$ (or $B=B_c$ in the hierarchical
 situation), showing that {\sl quenched and annealed critical points
   differ for every disorder strength $\gb>0$}. We also give a
 quantitative bound, $h_c(\gb)\ge \exp(-const/\beta^4)$ for $\gb$
 small, which is however presumably not optimal.  The method we use is
 a non-trivial extension of the {\sl fractional moment -- change of
   measure method} which already allowed to prove disorder relevance
 for $B<B_c$ in \cite{cf:GLT} or for $\ga>1/2$ in \cite{cf:DGLT}. A
 few words about the evolution of this method may be useful to the
 reader. The idea of estimating non-integer moments of the partition
 function of disordered systems is not new: consider for instance
 \cite{cf:BPP} in the context of directed polymers in random
 environment, or \cite{cf:AM} in the context of Anderson
 localization (in the latter case one deals with non-integer moments
 of the propagator). However, the power of non-integer moments in
 pinning/wetting models was not appreciated until \cite{cf:T_fractmom},
 where it was employed to prove, among other facts, that quenched and
 annealed critical points differ for large $\gb$, irrespective of the
 value of $\ga\in(0,\infty)$. The new idea which was needed to treat
 the case of weak disorder (small $\gb$) was instead introduced in
 \cite{cf:GLT,cf:DGLT}, and it is a change-of-measure idea, coupled
 with an {\sl iteration procedure}: one changes the law of the
 disorder $\go$ in such a way that the new and the old laws are very
 close in a certain sense, but under the new one it is easier to prove
 that the fractional moments of the partition function are small. In
 the relevant disorder regime, $\ga>1/2$ or $B<B_c$, it turns out that
 it is possible to choose the new law so that the $\go_n$'s are still
 IID random variables, whose law is simply tilted with respect to the
 original one.  This tilting procedure is bound to fail if applied for
 arbitrarily large volumes, but having such bounds for sufficiently
 large, but finite, system sizes is actually sufficient because of an
 iteration argument (which appears very cleanly in the hierarchical
 set-up).

In order to deal with the marginal case we will instead introduce 
 a long-range anti-correlation structure
 for the $\go$-variables.  Such correlations are carefully chosen in
 order to reflect the structure of the two-point function of the
 annealed model and, in the non-hierarchical case, they are
 restricted, via a coarse-graining procedure 
inspired by \cite{cf:T_cg}, only to suitable {\sl
 disorder pockets}.

We mention also that one of us \cite{cf:Hubert} proved recently that
disorder is marginally relevant in a different version of the
hierarchical pinning model. What simplifies the task in that case is
that the Green function of the model is spatially inhomogeneous and
one can take advantage of that by tilting the $\go$-distributions in a
inhomogeneous way (keeping the $\go$'s independent). The Green
function of the hierarchical model proposed in \cite{cf:DHV} is
instead constant throughout the system and inhomogeneous tilting does
not seem to be of help (as it does not seem to be of help in the
non-hierarchical case, since it does not match with the coarse
graining procedure).

\smallskip 
 The paper is organized as follows: the hierarchical
(resp. non-hierarchical) pinning model is precisely defined in
Section \ref{sec:Hmodel} (resp. in Section \ref{sec:nHmodel}), where
we also state our result concerning marginal relevance of disorder.
Such result is proven in Section \ref{sec:Hproof} in the hierarchical
case, and in Section \ref{sec:nHproof} in the non-hierarchical one.

In order not to hide the novelty of the
idea with technicalities, we restrict ourselves to Gaussian disorder
and, in the case of the non-hierarchical model, we do not treat the
natural generalization where $K(\cdot)$ is of the form
$K(n)={L(n)}/{n^{3/2}}$ with $L(\cdot)$ a slowly varying function
\cite[VIII.8]{cf:Feller2}.  We plan to come back to both issues in a
forthcoming paper \cite{cf:kbodies}.

\section{The hierarchical model}

\label{sec:Hmodel}

Let $1<B<2$.
We study the following iteration which transforms a vector $\{R_n^{(i)}\}_{i\in\N}\in (\R^+)^\N$ into
a new vector $\{R_{n+1}^{(i)}\}_{i\in\N}\in(\R^+)^\N$:
\begin{equation}
  \label{eq:Rn+1}
  R^{(i)}_{n+1}\, =\, \frac{R_n^{(2i-1)} R_n^{(2i)}+(B-1)}B, 
\end{equation}
for $n\in\N\cup\{0\}$ and $i\in\N$.  

In particular, we are interested in
the case in which the initial condition is random and  given by
$R_0^{(i)}=e^{\beta\go_i-\beta^2/2+h}$, with $\go:=\{\go_i\}_{i\in\N}$
a sequence of IID standard Gaussian random variables and $h\in\R, \gb\ge0$.  We
denote by $\bbP$ the law of $\go$ and by $\bbE$
the corresponding average.  In this case, it is immediate to realize
that for every given $n$ the random variables $\{R_n^{(i)}\}_{i\in\N}$ are IID. We
will study the behavior for large $n$ of $X_n:=R_n^{(1)}$.

It is easy to see that the average of $X_n$  
 satisfies the  iteration
\begin{equation}
\label{eq:homomap}
  \bbE({X_{n+1}})\, =\, \frac{(\bbE{X_n})^2+(B-1)}B, 
\end{equation}
with initial condition $\bbE({X_0})=e^h$. The map \eqref{eq:homomap}
has two fixed points: a stable one, $\bbE X_n=(B-1)$, and an unstable
one, $\bbE X_n=1$. This means that if $0\le\bbE X_0<1$ then $\bbE X_n$
tends to $(B-1)$ when $n\to\infty$, while if $\bbE X_0>1$ then $\bbE
X_n$ tends to $+\infty$.

\begin{rem}\rm
  In \cite{cf:DHV} and \cite{cf:GLT}, the model with $B>2$ was
considered. However, the cases $B\in(1,2)$
and $B\in(2,\infty)$ are equivalent. Indeed, if $R_n^{(i)}$ satisfies
\eqref{eq:Rn+1} with $B>2$, it is immediate to see that $\hat
R_n^{(i)}:= R_n^{(i)}/(B-1)$ satisfies the same iteration but with $B$
replaced by $\hat B:=B/(B-1)\in(1,2)$. In this work, we prefer to work
with $B\in(1,2)$ because things turn out to be notationally simpler
({\sl e.g.}, the annealed critical point (defined in the next section) turns
out to be $0$ rather than $\log (B-1)$).  In the following, whenever
we refer to results from \cite{cf:GLT} we give them for $B\in(1,2)$.
\end{rem}

\subsection{Quenched and annealed free energy and critical point}
The random variable $X_n$ is interpreted as the partition function of
the hierarchical random pinning model on a diamond lattice of
generation $n$ (we refer to \cite{cf:DHV} for a clear discussion of
this connection).  The {\sl quenched free energy} is then defined as
\begin{equation}
  \label{eq:F27}
  \tf(\beta,h):=\lim_{n\to\infty}\frac1{2^n}\bbE \log X_n.
\end{equation}
In \cite[Th. 1.1]{cf:GLT} it was proven, among other facts, that for every $\beta\ge 0,h\in\R$ the limit
\eqref{eq:F27} exists and  it is non-negative. Moreover, $\tf(\beta,\cdot)$ is convex and non-decreasing.
On the other hand, the {\sl annealed free energy} is by definition
\begin{equation}
  \label{eq:Fa}
  \tf^a(\beta,h):=\lim_{n\to\infty}\frac1{2^n}\log \bbE X_n.
\end{equation}
Since the initial condition of \eqref{eq:Rn+1} was normalized so that $\bbE X_0=e^h$, it is
easy to see that the annealed free energy is nothing but the free energy of the non-disordered model:
\begin{equation}
  \label{eq:qz}
\tf^a(\beta,h)=\tf(0,h).  
\end{equation}
Non-negativity of the free energy allows to define the {\sl quenched critical} point in a natural way, as
\begin{equation}
  \label{eq:hc222}
  h_c(\beta):=\inf\{h\in\R: \tf(\beta,h)>0\},
\end{equation}
and analogously one defines the {\sl annealed critical point}
$h_c^a(\beta)$. In view of observation \eqref{eq:qz}, one sees that
$h_c^a(\beta)=h_c(0)$.  Monotonicity and convexity of $\tf(\beta,\cdot)$ imply that $\tf(\beta,h)=0$ for
$h\le h_c(\beta)$.

For the annealed system, the critical point
and the critical behavior of the free energy around it are known (see
\cite{cf:DHV} or \cite[Th. 1.2]{cf:GLT}). What one finds is that for
every $B\in(1,2)$ one has $h_c(0)=0$, and there exists $c:=c(B)>0$
such that for all $0\le h\le 1$
\begin{equation}
\label{eq:alpha-1}
  c(B)^{-1}h^{1/\alpha}\le \tf(0,h)\le c(B) h^{1/\alpha},
\end{equation}
where 
\begin{equation}
  \label{eq:alphaia}
  \alpha:=\frac{\log(2/B)}{\log 2}\in (0,1).
\end{equation}
Observe that $\alpha$ is decreasing as a function of $B$, and equals $1/2$ for $B=B_c:=\sqrt2$.

\subsection{Disorder relevance or irrelevance}

The main question we are interested in is whether quenched and annealed critical points differ, and 
if yes how does their difference behave for small disorder.
Jensen's inequality, $\bbE\log X_n\le \log \bbE X_n$, implies in particular that $\tf(\beta,h)\le \tf(0,h)$ so that
$h_c(\beta)\ge h_c(0)=0$. Is this inequality strict?

In \cite{cf:GLT} a quite complete picture was given, except in the marginal case $B=B_c$ which was left open:

\begin{theorem}\cite[Th. 1.4]{cf:GLT}
  If $1<B<B_c$, $h_c(\gb)>0$ for every $\gb>0$ and there exists $c_1>0$ such that for $0\le \beta\le 1$
  \begin{equation}
\label{eq:relh}
    c_1 \beta^{2\alpha/(2\alpha-1)}\le h_c(\beta)\le c_1^{-1} \beta^{2\alpha/(2\alpha-1)}.
  \end{equation}

If $B=B_c$ there exists $c_2>0$ such that for $0\le \beta\le 1$
\begin{equation}
\label{eq:margh}
  h_c(\beta)\le \exp(-c_2/\beta^2).
\end{equation}

If $B_c<B<2$ there exists $\beta_0>0$ such that $h_c(\beta)=0$ for every $0<\beta\le \beta_0$.
\end{theorem}

\medskip

The main result of the present work is that in the marginal case, the
two critical points {\sl do differ} for every disorder strength:
\begin{theorem}
\label{th:main}
Let $B=B_c$. For every $0<\beta_0<\infty$ there exists a constant
$0<c_3:=c_3(\beta_0)<\infty$ such that for every $0<\beta\le\beta_0$
  \begin{equation}
\label{eq:main}
    h_c(\beta)\ge \exp(-c_3/\beta^4).
  \end{equation}
\end{theorem}

\section{The non-hierarchical model}\label{sec:nHmodel}

We let $\tau:=\{\tau_0,\tau_1,\ldots\}$ be a renewal process of law
$\bP$, with inter-arrival law $K(\cdot)$, {\sl i.e.}, $\tau_0=0$ and
$\{\tau_i-\tau_{i-1}\}_{i\in\N}$ is a sequence of IID integer-valued
random variables such that
\begin{equation}
  \label{eq:K222}
\bP(\tau_1=n)= : K(n)\stackrel{n\to\infty}\sim \frac {C_K}{n^{1+\alpha}},
\end{equation}
with $C_K>0$ and $\alpha>0$. We require that $K(\cdot)$ is a
probability on $\N$, which amounts to assuming that the renewal
process is recurrent. 
 We require also that $K(n)>0$ for every $n\in\N$, but this is inessential 
 and it is just meant to avoid making a certain number of remarks 
 and small detours in the proofs to take care of this point.

As in Section \ref{sec:Hmodel}, $\go:=\{\go_1,\go_2,\ldots\}$ denotes a sequence of
IID standard Gaussian random variables.
For a given system size $N\in\N$, coupling parameters
$h\in \R$, $\beta\ge 0$ and a given disorder realization $\go$ the
partition function of the model is defined by
\begin{equation}
  \label{eq:Znh}
  Z_{N,\go}:=\bE\left[e^{\sum_{n=1}^N(\beta\go_n+h-\beta^2/2)\delta_n}\delta_N\right],
\end{equation}
where $\delta_n:=\ind_{\{n\in\tau\}}$, while the quenched free energy is
\begin{equation}
  \label{eq:F_nh}
  \tf(\beta,h):=\lim_{N\to\infty}\frac1N\bbE\log Z_{N,\go},
\end{equation}
(we use the same notation as for the hierarchical model, since there
is no risk of confusion).  Like for the hierarchical model,
the limit exists and is non-negative \cite[Ch. 4]{cf:Book}, and one
defines the critical point $h_c(\beta)$ for a given $\beta\ge0$
exactly as in \eqref{eq:hc222}. Again, one notices that the annealed free
energy, {\sl i.e.}, the limit of $(1/N)\log\bbE Z_{N,\go}$, is nothing but
$\tf(0,h)$, so that the annealed critical point is just $h_c(0)$.

\begin{rem}\rm
  With respect to most of the literature, our definition of the
  model is different (but of course completely equivalent) in that
  usually the partition function is defined as in \eqref{eq:Znh} with
  $h-\beta^2/2$ replaced simply by $h$. 
\end{rem}
The annealed (or pure) model can be exactly solved and in particular
it is well known \cite[Th. 2.1]{cf:Book} that, if $\alpha\ne 1$, there
exists a positive constant $c_K$ (which depends on $K(\cdot)$) such that
\begin{equation}
  \label{eq:annF}
  \tf(0,h)\stackrel{h\searrow0}\sim c_K h^{\max(1,1/\alpha)}.
\end{equation}
In the case $\alpha=1$, \eqref{eq:annF} has to be modified in that the
right-hand side becomes $\phi(1/h)h$ for some slowly-varying function
$\phi(\cdot)$ which vanishes at infinity \cite[Th. 2.1]{cf:Book}. In
particular, note that $h_c(0)=0$ so that $h_c(\beta)\ge0$ by Jensen's
inequality, exactly like for the hierarchical model.

\begin{rem}\rm
\label{rem:trans}
  The assumption of
  recurrence for $\tau$, {\sl i.e.},  $\sum_{n\in\N}K(n)=1$, is
  by no means a restriction. In fact, as it has been observed several
  times in the literature, if $\Sigma_K:=\sum_{n\in\N}K(n)<1$ one can
  define $\tilde K(n):=K(n)/\Sigma_K$, and of course the renewal
  $\tau$ with law $\tilde\bP(\tau_1=n)=\tilde K(n)$ is recurrent. Then,
it is immediate to realize from definition \eqref{eq:F_nh} that
\begin{equation}
\tf(\beta,h)=\tilde \tf(\beta,h+\log \Sigma_K),
\end{equation}
$\tilde \tf$ being the free energy of the model defined as in
\eqref{eq:Znh}-\eqref{eq:F_nh} but with $\bP$ replaced by $\tilde\bP$.
In particular, $h^a_c(\beta)=-\log \Sigma_K$.
This observation allows to apply Theorem \ref{th:main2} below, for
instance, to the case where $\tau$ is the set of returns to the origin
of a symmetric, finite-variance random walk on $\Z^3$ (pinning of a
directed polymer in dimension $(3+1)$): indeed, in this case 
\eqref{eq:K222} holds with $\alpha=1/2$. For more details on this issue
we refer to \cite[Ch.~1]{cf:Book}.
\end{rem}

\subsection{Relevance or irrelevance of disorder}
Like for the hierarchical model, the 
question whether $h_c(\beta)$ coincides
or not with $h_c(0)$ for $\beta$ small has been recently solved, {\sl except in the marginal case}
$\ga=1/2$:

\begin{theorem}
If $0<\alpha<1/2$, there exists $\beta_0>0$ such that $h_c(\beta)=0$ for every $0\le \beta\le \beta_0$.
If $\ga=1/2$, there exists a constant $c_4>0$ such that for $\beta\le 1$
\begin{equation}
  h_c(\beta)\le \exp(-c_4/\beta^2).
\end{equation}
If $\ga>1/2$, $h_c(\gb)>0$ for every $\gb>0$ and, if in addition $\ga\ne1$,  there exists a constant $c_5>0$ such that if $\beta\le 1$
\begin{equation}
  c_5\beta^{\max(2\ga/(2\ga-1),2)}\le h_c(\beta)\le  c_5^{-1}\beta^{\max(2\ga/(2\ga-1),2)}.
\end{equation}
If $\ga=1$ there exist a constant $c_6>0$ and a slowly varying function $\psi(\cdot)$ vanishing at infinity such that 
for $\gb\le 1$
\begin{equation}
   c_6\beta^{2}\psi(1/\beta)\le h_c(\beta)\le   c_6^{-1}\beta^{2}\psi(1/\beta).
\end{equation}
\end{theorem}
The results for $\ga\le 1/2$, together with the critical point upper
bounds for $\ga>1/2$, have been proven in \cite{cf:Ken}, and then in
\cite{cf:T_cmp}; the lower bounds on the critical point for $\ga>1/2$
have been proven in \cite{cf:DGLT} 
(the result in \cite{cf:DGLT} is slightly weaker than what we state
here and the case $\ga=1$ was not treated) and then in \cite{cf:AZ}
(with the full result cited here). 

The case $\alpha=0$ has also been
considered, but in that case \eqref{eq:K222} has to be replaced by 
$K(n)=L(n)/n$, with
$L(\cdot)$ a function varying slowly at infinity and such that
$\sum_{n\in \N}K(n)=1$. For instance, this corresponds to the case
where $\tau$ is the set of returns to the origin of a symmetric
random walk on $\Z^2$. In this case, it has been shown in \cite{cf:AZ_new} that
quenched and annealed critical points coincide for every value of
$\beta\ge0$.

\begin{rem}\rm
Let us recall also that it is proven in \cite{cf:GT_cmp} that, for every $\ga>0$, we have 
\begin{equation}
  \label{eq:smooth}
\tf(\beta,h)\, \le \, \frac{1+\ga}{2\gb^2}\,(h-h_c(\gb))^2, 
\end{equation}
for all $\gb>0,h>h_c(\gb)$: this means that when $\ga>1/2$
disorder is relevant also in the sense that it changes the free-energy critical exponent ({\sl cf.} \eqref{eq:annF}).
The analogous result for the hierarchical model, with $(1+\ga)$ replaced by some constant $c(B)$ in \eqref{eq:smooth},
is proven in \cite{cf:LT}.
\end{rem}

\bigskip

In the present work we prove the following: 
\begin{theorem}
\label{th:main2}
  Assume that \eqref{eq:K222} holds with $\alpha=1/2$. 
  For every $\gb_0>0$
  there exists a 
constant $0<c_7:=c_7(\gb_0)<\infty$  such that for $\beta\le \beta_0$
\begin{equation}
  h_c(\beta)\, \ge\,  e^{-c_7/\beta^4}.
\end{equation}
\end{theorem}

\section{Marginal relevance of disorder: the hierarchical case}\label{sec:Hproof}
\subsection{Preliminaries: a Galton-Watson representation
 for $X_n$}
\label{sec:trees}

One can give an expression for $X_n$ which is analogous to that of the
partition function \eqref{eq:Znh} of the non-hierarchical model, and
which is more practical for our purposes. This involves a Galton-Watson tree 
\cite{cf:T-Harris} describing the successive offsprings  of one individual.
The offspring distribution concentrates on $0$ (with probability $(B-1)/B$) and on $2$ (with probability $1/B$). So, at a given 
 generation,
each individual  that is  present has either no descendant or two
descendants, and this independently of any other individual of the
generation. This branching procedure directly maps to a random
tree (see Figure~\ref{fig:treemarg}): the law of such a branching process up
to generation $n$ (the first individual is at generation
$0$) or, analogously, the law of the random tree 
 from the root (level $n$) up to the leaves (level $0$), is denoted by $\bP_n$.
The individuals that are present at the $n^{\textrm{th}}$ generation
are a random subset ${\mathcal R}_n $ of 
$\{ 1, \ldots, 2^n\}$. We set $\gd_j:= \ind _{j \in {\mathcal R}_n}$.
Note that the mean offspring size is $2/B>1$, so that the Galton-Watson process
is supercritical. 

The following procedure on the standard binary graph $\cT^{(n)}$ of
depth $n+1$ (again, the root is at level $n$ and the leaves, numbered
from $1$ to $2^n$, at level $0$) is going to be of help too.  Given
$\mathcal I\subset \{1,\ldots,2^n\}$, let $\mathcal T^{(n)}_{\mathcal
  I}$ be the subtree obtained from $\mathcal T^{(n)}$ by deleting all
edges except those which lead from leaves $j\in\mathcal I$ to the
root. Note that, with the offspring distribution we consider, in general
$\mathcal T^{(n)}_{\mathcal I}$ is not a realization of 
the $n$-generation Galton-Watson tree (some individuals may have
just one descendant in $\mathcal T^{(n)}_{\mathcal I}$, 
see Figure~\ref{fig:treemarg}).

 Let $v(n,\mathcal I)$ be the number of nodes
in $\mathcal T^{(n)}_{\mathcal I}$, with the convention that leaves are not counted as nodes, while the root is.

\begin{figure}[ht]
\begin{center}
\leavevmode
\epsfxsize =10.7 cm
\psfragscanon
\psfrag{i}{$4$}
\psfrag{j}{$6$}
\psfrag{k}{$13$}
\psfrag{0}[r]{\small \sl level $0$}
\psfrag{L}[r]{\small \sl (the leaves)}
\psfrag{R}[r]{\small \sl (\ \ \ \ \ \ the root)}
\psfrag{1}[r]{\small  \sl level $1$}
\psfrag{2}[r]{\small  \sl level $2$}
\psfrag{r}[r]{\small  \sl level $4$}
\psfrag{3}[r]{$\ldots\ \ $ }
\epsfbox{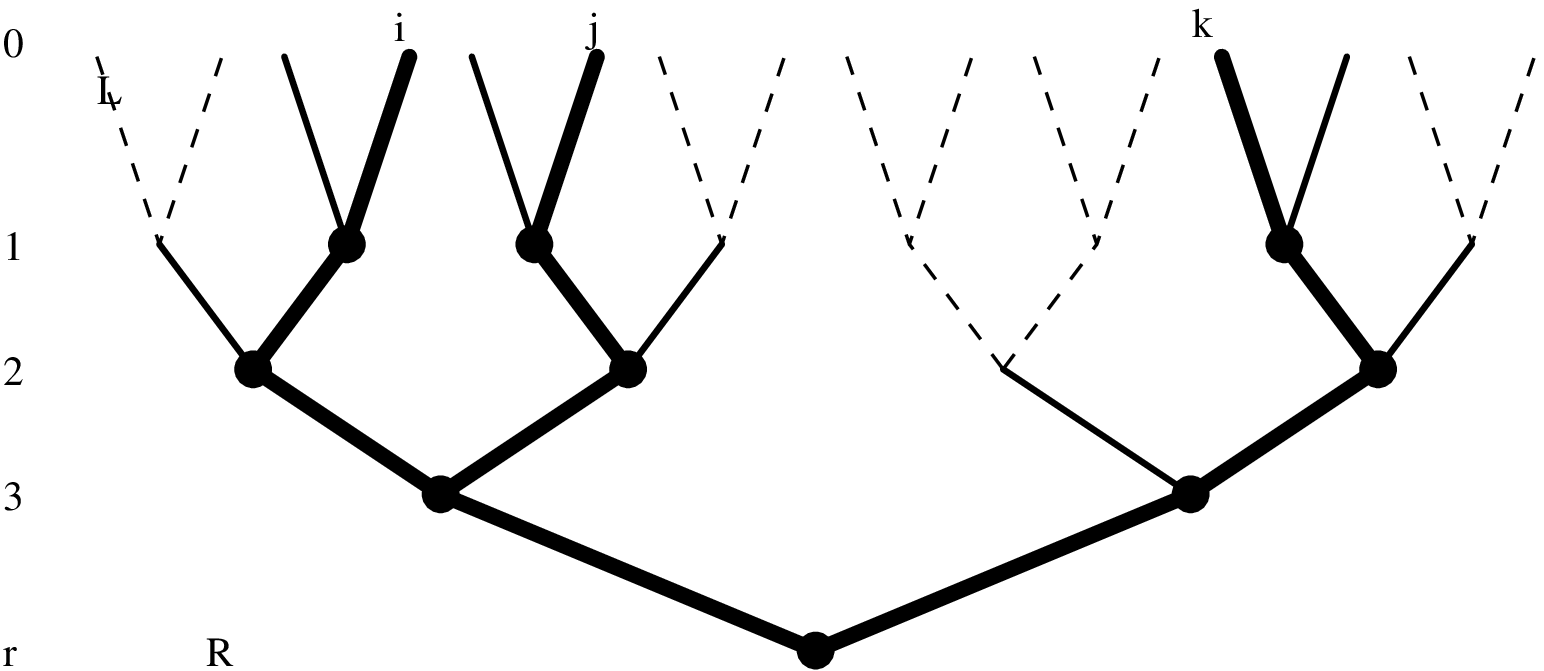}
\end{center}
\caption{
The thick solid lines in the figure form the
tree $\mathcal T^{(4)}_{\{4,6,13\}}$, which is a subtree of the
 binary tree $\mathcal T^{(n)}$ ($n=4$). Note that
$\mathcal T^{(4)}_{\{4,6,13\}}$ is {\sl not} a possible 
realization of the Galton-Watson tree, while it becomes so if we complete it
by adding the thin solid lines.
At level $0$ there
are the leaves; the nodes of $\mathcal T^{(4)}_{\{4,6,13\}}$
are marked by dots. $\mathcal T^{(4)}_{\{4,6,13\}}$ contains
$v(4,\{4,6,13\})=9$ nodes. In terms of Galton-Watson offsprings,
for the ({\sl completed}) trajectory above
$\mathcal R _4=\{3,4,5,6,13,14\}$. Moreover,
computing the average in  
\eqref{eq:ffv1} means computing the probability that the 
realization of 
the Galton-Watson 
tree  contains $\mathcal T^{(n)}_{\mathcal I}$
as a subset: but this simply means  requiring that 
the individuals at the nodes of $\mathcal T^{(n)}_{\mathcal I}$
have two children and the expression \eqref{eq:ffv1} becomes clear.}
\label{fig:treemarg}
\end{figure}

\medskip

\begin{proposition}
\label{th:deltas}
  For every $n\ge0$ we have
\begin{equation}
  \label{eq:Rn}
  X_n\, = \,\bE_n\left[
  e^{\sum_{i=1}^{2^n}(\beta \go_i+h-\beta^2/2)\delta_i}\right].
\end{equation}
For every $n\ge 0$ and $\mathcal I\subset \{1,\ldots,2^n\}$, one has
\begin{equation}
\label{eq:ffv1}
  \bE_n\bigg[\prod_{i\in\mathcal I}\delta_i\bigg]\, =\, B^{-v(n,\mathcal I)}.
\end{equation}
In particular, $\bE_n[\delta_i]=B^{-n}$ for every $i=1,\ldots,2^n$,
{\sl i.e.}, the Green function is constant throughout the system. 
\end{proposition}
\medskip



\begin{proof}[Proof of Proposition \ref{th:deltas}]
The right-hand side in  \eqref{eq:Rn} 
for $n=0$ is equal to $\exp(h-\gb^2/2+ \gb \go_1)$.
Moreover, at the $(n+1)^{\textrm{th}}$ generation the branching process
either contains only the initial individual (with probability $(B-1)/B$)
or the initial individual  has two children, which we may look at as 
initial  individuals
of two independent 
Galton-Watson trees containing $n$ new generations.
We therefore have  that the basic recursion  \eqref{eq:Rn+1}
is satisfied. 

The second fact, \eqref{eq:ffv1}, is a direct consequence of the definitions
(see also the caption of Figure~\ref{fig:treemarg}).
\end{proof}

\begin{rem}\rm
The representation we have introduced in this section shows in particular
that $\bbE X_n$ is just the generating function of $\vert \mathcal R_n\vert$
and the free energy $\tf(0, h)$ is therefore a natural quantity
for the 
Galton-Watson process: and in fact  $1/\ga$ ($\ga$ given in 
\eqref{eq:alphaia}) appears in the original works on branching processes by
 T.~E.~Harris
(of course not to be confused with  A.~B.~Harris, who proposed the 
disorder relevance criterion on which we are focusing in this work).
\end{rem}

\subsection{The proof of Theorem~\ref{th:main}}
While the discussion of the previous section is valid for every
$B\in(1,2)$,  
now we have to assume $B=B_c=\sqrt 2$. However some of the steps are still
valid in general and we are going to replace $B$ with $B_c$ only 
when it is really needed.  
 The
proof is split into three subsections: the first  introduces
the fractional moment method and reduces the statement we want to
prove, which is a statement on the limit $n \to \infty$ behavior of
$X_n$, to finite-$n$ estimates. The estimates are provided in the
second and third subsection.

\subsubsection{The fractional moment method}
Let $U_n^{(i)}$ denote the quantity $[R_n^{(i)}-(B-1)]_+$ where
$[x]_+=\max(x,0)$. 
Using the inequality
\begin{equation}
  [rs+r+s]_+\le [r]_+[s]_++[r]_++[s]_+,
\end{equation}
which holds whenever $r,s\ge-1$,
it is easy to check that \eqref{eq:Rn+1} implies

\begin{equation}
\label{eq:recU}
 U^{(i)}_{n+1}\le \frac{U_n^{(2i-1)}U_n^{(2i)}+(U_n^{(2i-1)}+U_n^{(2i)})(B-1)}{B}.
\end{equation}
Given $0<\gamma<1$, we define $A_n:=\bbE([X_n-(B-1)]_+^\gamma)$. From \eqref{eq:recU} above and by using the
fractional inequality 
\begin{equation}
  \label{eq:fracineq}
  \left(\sum a_i\right)^{\gamma}\, \le\,  \sum a_i^{\gamma},
\end{equation}
which holds whenever $a_i\ge0$, we derive
\begin{equation}
 A_{n+1}\le \frac{A_n^2+2(B-1)^{\gamma}A_n}{B^{\gamma}}.
\end{equation}
One readily sees now 
 that,  if there exists some integer $k$ such that 
\begin{equation}
\label{eq:k}
A_k\, < \, B^{\gamma}-2(B-1)^{\gamma},
\end{equation}
then $A_n$ tends to zero as $n$ tends to infinity (this statement is
easily obtain by studying the fixed points of the function $x\mapsto
{(x^2+2(B-1)^{\gamma}x)}/{B^{\gamma}}$). 
On the other hand,
\begin{equation}
\label{eq:LB0411}
\bbE\left[ X_n ^\gamma \right] \, \le \, \bbE\left([X_n-(B-1)]_++(B-1)\right)^\gamma\le 
(B-1)^\gamma + A_n,
\end{equation}
and therefore \eqref{eq:k}  implies that $\tf (\gb, h)=0$
since, by Jensen inequality, we have
\begin{equation}
\label{eq:JensAn}
\frac 1{2^n} \bbE \log X_n  \, \le \, \frac 1{2^n\gamma} \log  \bbE\left[ X_n ^\gamma \right].
\end{equation}
Note that, to establish $\tf (\gb, h)=0$, 
it suffices to prove that 
  $\limsup_n 2^{-n}\log A_n \le 0$,
  hence our approach yields a substantially stronger piece of
  information, {\sl i.e.} that the fractional moment $A_n$ does go to
  zero.

\medskip 

In order to find a $k$ such that \eqref{eq:k} holds we introduce a new
probability measure $\tilde \bbP$ (which is going to depend on $k$) such
that $\tilde \bbP$ and $\bbP$ are equivalent, that is mutually
absolutely continuous.  By H\"older's inequality 
applied for $p=1/\gamma$ and $q=1/(1-\gamma)$ we have
\begin{equation}
\begin{split}
\label{eq:hold}
  A_k\, =\, \tilde \bbE\left[ \frac{\dd \bbP}{\dd \tilde
      \bbP}\;[X_k-(B-1)]_+^{\gamma}\right] \le
  \left(\bbE\left[\left(\frac{\dd \bbP}{\dd \tilde
          \bbP}\right)^{\frac{\gamma}{1-\gamma}}\right]\right)^{1-\gamma}
  \left(\tilde\bbE\left[ [X_k-(B-1)]_+\right]\right)^{\gamma},
\end{split}
\end{equation}
and a sufficient condition for \eqref{eq:k} is  therefore that
\begin{equation}
\label{eq:cond1}
 \tilde\bbE\left[ [X_k-(B-1)]_+\right]\, \le\,  \left(\bbE\left[\left(\frac{\dd \bbP}{\dd \tilde \bbP}\right)^{\frac{\gamma}{1-\gamma}}\right]\right)^{1-\frac{1}{\gamma}}\left(B^{\gamma}-2(B-1)^{\gamma}\right)^{\frac{1}{\gamma}}.
\end{equation}

Let $x^{(0)}_n$ be obtained applying $n$ times the annealed iteration
$x\mapsto (x^2+(B-1))/B$ to the initial condition $x^{(0)}_0=0$. One
has that $x_n^{(0)}$ approaches monotonically the stable fixed point
$(B-1)$. Since the coefficients in the iteration \eqref{eq:Rn+1} are
positive, one has for every $h,\gb,\go$ that $X_n
\ge x_n^{(0)} \stackrel{n\to\infty}\nearrow B-1$ (this is a deterministic bound) and therefore, for any given $\zeta>0$,
one can find an integer $n_\zeta$ such that if $n\ge n_\zeta$ we have
\begin{equation}
\tilde\bbE\left[ [X_n-(B-1)]_+\right]\, \le \, \tilde\bbE\left[ X_n-(B-1)\right]+\frac \zeta 4.
\end{equation}
Moreover, since
$\left(B^{\gamma}-2(B-1)^{\gamma}\right)^{\frac{1}{\gamma}}- (2-B) \stackrel{\gamma
\nearrow 1}\sim 
-c_B (1-\gamma)$ for some $c_B>0$, one can find  
$\gamma=\gamma_\zeta$ such that $\left(B^{\gamma}-2(B-1)^{\gamma}\right)^{\frac{1}{\gamma}}\ge 2-B-\zeta/4$.
At this point,
if $\gamma=\gamma_\zeta$, $k\ge n_\zeta$ and if $\tilde\bbP$ is such that
\begin{equation}
\left(\bbE\left[\left(\frac{\dd \bbP}{\dd \tilde \bbP}\right)^{\frac{\gamma}{1-\gamma}}\right]\right)^{1-\frac{1}{\gamma}}\ge 1-\frac {\zeta}4.
 \end{equation}
 (recall that $\tilde \bbP$ depends on $k$)
and $\tilde\bbE [X_k]\le 1-\zeta$ then \eqref{eq:cond1} is satisfied
and $\tf(\gb, h)=0$.

We sum up what we have obtained:
\medskip

\begin{lemma}\label{th:moment}Let $\zeta>0$ and
choose  $\gamma(=\gamma_\zeta)$ and   $n_\zeta$
as above. If there exists $k\ge n_\zeta$ and 
a probability measure $\tilde \bbP$ (such that $\bbP$
and $\tilde \bbP$ are equivalent probabilities) such that
\begin{equation}\label{eq:close}
\left(\bbE\left[\left(\frac{\dd \bbP}{\dd \tilde \bbP}\right)^{\frac{\gamma}{1-\gamma}}\right]\right)^{1-\frac{1}{\gamma}}\, \ge \, 1-\frac{\zeta}4 , 
 \end{equation}
and 
\begin{equation}
 \tilde \bbE [X_k]\le 1-\zeta,
\end{equation}
then the free energy is equal to zero.
\end{lemma}

\subsubsection{The change of measure}

In order to use wisely the result of the previous section, we have to
find a measure $\tilde \bbE:=\tilde \bbE_n$ on the environment which
is, in a sense, close to $\bbE$ ({\sl cf.} \eqref{eq:close}), and that
lowers significantly the expectation of $X_n$. In \cite{cf:GLT} we
introduced the idea of changing the mean of the $\go$-variables,
while keeping their IID character. This strategy was enough to prove
disorder relevance for $B<B_c$, but it is not effective in the
marginal case $B=B_c$ we are considering here. Here, instead, we
choose to introduce {\sl weak, long range} negative correlations
between the different $\go_i$ without changing the laws of the
1-dimensional marginals. As it will be clear, the covariance structure we
choose reflects the hierarchical structure of the model we are
considering.

 In the
sequel we take $h\ge h_c(0)=0$.

We define $\tilde \bbP_n$ by stipulating that the variables $\go_i,i>2^n,$ are still 
IID standard Gaussian independent of $\go_1, \ldots, \go_{2^n}$, while $\go_1, \ldots, \go_{2^n}$
are Gaussian, centered,  and with covariance matrix
\begin{equation}
\label{eq:covar}
  C:=I-\gep V,
\end{equation}
where $I$ is the $2^n\times 2^n$ identity matrix, $\gep>0$ and $V$ is
a  symmetric $2^n\times 2^n$ matrix with zero diagonal terms 
and with positive off-diagonal terms
($\gep$ and $V$
will be specified in a moment). 

The choice $V_{ii}=0$ implies of course $\text{Trace}(V)=0$, and we
are also going to impose that the Hilbert-Schmidt norm of $V$ verifies 
$\Vert V \Vert ^2 := \sum_{i,j } V_{i,j}^2 =\text{Trace}(V^2)=1$. This in particular
implies that $C$ is positive definite (so that $\tilde \bbP_n$
exists!) as soon as $\gep<1$: this is because $\Vert V \Vert$, being a
matrix norm, dominates the spectral radius of $V$.

Now, still without choosing $V$ explicitly,  we compute a lower bound for the left--hand side of \eqref{eq:close}. 
The mutual density of $\tilde \bbP_n$ and $\bbP$  is
\begin{equation}
  \label{eq:densit}
  \frac{\dd \tilde \bbP_n}{\dd \bbP}(\go)\, =
  \, 
  \frac{e^{-1/2 ((C^{-1}-I)\go,\go)}}{\sqrt{\det C}},
\end{equation}
with the notation $(Av,v):=\sum_{1\le i,j\le 2^n}A_{ij}v_i v_j$,
and therefore a straightforward Gaussian computation gives
\begin{equation}
  \label{eq:dets}
  \left( \bbE \left[ \left(\frac{\dd\bbP}{\dd\tilde \bbP_n}\right)^{\gamma/(1-\gamma)}\right]\right)^{1-1/\gamma}\, =\, 
  \frac{(\det[I-(\gep/(1-\gamma)) V])^{(1-\gamma)/(2\gamma)}}{(\det C)^{1/(2\gamma)}}.
\end{equation}
If we want to prove a lower bound of the type \eqref{eq:close}, a
necessary condition is of course that the numerator in \eqref{eq:dets}
is  positive: this is ensured by requiring $\gep < 1-\gamma$. For the
next computation we are going to require also that $ \gep/(1-\gamma)
\le 1/2$: we are going in fact to use that $\log (1+x) \ge x -x^2$ if
$x \ge -1/2$, and $\text{Trace}(V)=0$ to obtain that
\begin{multline}
\label{eq:dano}
\det\left[I-(\gep/(1-\gamma)) V\right] \, =\,
\exp \left( \text{Trace}( \log (I-(\gep/(1-\gamma))V)) \right)\\
 \ge \, 
\exp\left( -\frac{\gep^2}{(1-\gamma)^2} \Vert V \Vert ^2 \right),
\end{multline}
while
$\log (1+x) \le x$ and the traceless character of $V$ directly imply $\det C \le 1$ so that finally
\begin{equation}
  \label{eq:dets-est}
  \left(  \bbE\left[\left(\frac{\dd\bbP}{\dd\tilde \bbP_n}\right)^{\gamma/(1-\gamma)}\right]\right)^{1-1/\gamma}\, \ge\, 
  \exp \left( -\frac{\gep^2}{2\gamma(1-\gamma)} \right).
\end{equation}

\smallskip

Next, we estimate the expected value of $X_n$ under the modified measure: from \eqref{eq:Rn} we see that
\begin{equation}\begin{split}
\label{eq:EZN}
\tilde \bbE_n X_n&\, =\, \bE_n\left[e^{(h-(\beta^2/2))\sum_{i=1}^{2^n}
  \delta_i}\,\tilde\bbE_n e^{\sum_{i=1}^{2^n}
  \gb\go_i\delta_i}\right]\\
&\, =\, \bE_n\left[e^{-\gep(\beta^2/2) (V\delta,\delta)+\sum_{i=1}^{2^n}
  h\delta_i}\right]
  \, \le\, e^{2^n h}\,\bE_n\left[e^{-\gep(\beta^2/2)
  (V\delta,\delta)}\right].
\end{split}
\end{equation}

Finally we choose $V$. From \eqref{eq:EZN}, it is not hard to guess that the most convenient choice, subject to the 
constraint $\Vert V\Vert^2=1$, is
\begin{equation}
  \label{eq:V}
  V_{ij}\, =\, \bE_n[\delta_i\delta_j]\bigg /\sqrt{\sum_{1\le i\ne j\le 2^n}(\bE_n[\delta_i\delta_j])^2},
\end{equation}
for $i\ne j$, while we recall that $V_{ii}=0$. 
The normalization in \eqref{eq:V} can be computed with the help of Proposition 
\ref{th:deltas}:
\begin{equation}
\label{eq:2n}
  \sum_{1\le i\ne j\le 2^n}\left(\bE_n[\delta_i\delta_j]\right)^2=
2^n\sum_{1<j\le 2^n}\left(\bE_n[\delta_1\delta_j]\right)^2=
2^n\sum_{1\le a\le n}\frac{2^{a-1}}{B_c^{2(n+(a-1))}}=n.
\end{equation}
In the second equality, we used the fact that there are $2^{a-1}$
values of $1<j\le 2^n$ such that the two branches of the tree
$\mathcal T^{(n)}_{\{1,j\}}$ join at level $a$ ({\sl cf.} the notations of
Section \ref{sec:trees}), and such tree contains $n+a-1$ nodes.

As a side remark, note that if $B_c<B<2$ (irrelevant disorder regime)
the left-hand side of \eqref{eq:2n} instead goes to zero with $n$,
while for $1<B<B_c$ (relevant disorder regime) it diverges
exponentially with $n$.

So, in the end,  our choice for $V$ is: 
\begin{equation}
  \label{eq:VV}
  V_{ij}=\left\{
    \begin{array}{lll}
      {\bE_n[\delta_i\delta_j]}/{\sqrt{n}} &\mbox{if}& i\ne j\\
0&\mbox{if}& i=j.
    \end{array}
\right.
\end{equation}

\subsubsection{Checking the conditions of Lemma \ref{th:moment}}

To conclude the proof of Theorem \ref{th:main} we have to show that if
$\beta\le \beta_0$ and $h\le \exp(-c_3/\beta^4)$ (and provided that
$c_3= c_3(\beta_0)$ is chosen large enough) the conditions of Lemma
\ref{th:moment} are satisfied. The main point is therefore to
estimate the expectation of $X_n$ under $\tilde\bbP_n$.

Recalling that ({\sl cf.} \eqref{eq:EZN})
\begin{equation}
\label{eq:q2}
  \tilde \bbE_n X_n\le \bE_n \left[e^{-(\beta^2/2)\gep\sum_{1\le i\ne j\le 2^n}\delta_i\delta_j
    \frac{\bE_n[\delta_i\delta_j]}{\sqrt{n}}}\right]e^{2^n h}
,
\end{equation}
we define
\begin{equation}
Y_n:=\sum_{1\le i\ne j\le 2^n}\delta_i\delta_j
    \frac{\bE_n[\delta_i\delta_j]}{n}.
\end{equation}
Thanks
to \eqref{eq:2n}, we know that $\bE_n (Y_n)=1$,
so that the Paley-Zygmund inequality gives
\begin{equation}
  \label{eq:PZ222}
  \bP_n \left(Y_n\ge 1/2\right) \,=\, 
  \bP_n \left(Y_n\ge (1/2)\bE_n(Y_n)\right)\, \ge 
  \, \frac{(\bE_n(Y_n))^2}{4\,\bE_n(Y_n^2)}=
  \frac{1}{4\,\bE_n(Y_n^2)}.
\end{equation}
We need therefore the following estimate, which will be proved at the end of the section:
\begin{lemma} \label{th:secmom}
We have: 
\begin{equation}
(1\le )\;\mathcal K:=\sup_{n}  \bE_n[Y_n^2]<\infty.
\end{equation}
\end{lemma}

Together with \eqref{eq:PZ222} this implies
\begin{equation}
 \bP_n[Y_n\ge 1/2]\,\ge  \, \frac{1}{4\mathcal K}, 
\end{equation}
so that, for all $n\ge 0$,
\begin{equation}
\label{eq:bd}
\begin{split}
\bE_n \left[e^{-(\beta^2/2)\gep\sum_{1\le i\ne j\le 2^n}\delta_i\delta_j
  \frac{\bE_n[\delta_i\delta_j]}{\sqrt{n}}}\right]\, &=\, 
\bE_n\left[e^{-\frac{\sqrt{n}\gb^2\gep}2 Y_n}\right]\\
&\le \, 1-\frac{1}{4\mathcal K}\left(1-4
  \mathcal K\exp\left(-\frac{\sqrt{n}\gb^2\gep}{4}\right)\right).
  \end{split}
\end{equation}
We fix $\zeta:=1/(40\mathcal K)$
and we choose $\gamma=\gamma_\zeta$ ({\sl cf.} Lemma \ref{th:moment}) and 
$\gep$ in \eqref{eq:covar} small enough so that ({\sl cf.} \eqref{eq:dets-est})
\begin{equation}\label{eq:upb2}
  \left[  \bbE\left(\frac{\dd\bbP}{\dd\tilde \bbP_n}\right)^{\gamma/(1-\gamma)}\right]^{1-1/\gamma}\, \ge\, 
\exp \left( -\frac{\gep^2}{2\gamma(1-\gamma)} \right)\ge 1-\frac{\zeta}{4}.
\end{equation}
Then one can check with the help of \eqref{eq:bd} that for $n\ge {50\mathcal K}/{(\gb^4\gep^2)}$,
\begin{equation}\label{eq:upb3}
\bE_n \left[e^{-(\beta^2/2)\gep\sum_{1\le i\ne j\le 2^n}\delta_i\delta_j
    \frac{\bE_n[\delta_i\delta_j]}{\sqrt{n}}}\right]\, \le \,  1-3\zeta.
\end{equation}
We choose $n=n_\gb$ in
$\left[\frac{50\mathcal K}{\gb^4\gep^2},\frac{50\mathcal K}{\gb^4\gep^2}+1\right)$ and
$h=\zeta 2^{-n}$. If $\gep$ has been chosen  small enough above (how small,
 depending only 
on $\beta_0$), this guarantees that $n\ge n_\zeta$, where $n_\zeta$
was defined just before Lemma \ref{th:moment}. Injecting
\eqref{eq:upb3} in \eqref{eq:q2} finally gives
\begin{equation}
 \tilde \bbE[X_n]\le (1-3\zeta)e^{\zeta}\le 1-\zeta.
\end{equation}

The two conditions of Lemma \ref{th:moment} are therefore verified,
which ensures that the free energy is zero for this value of $h$. In
conclusion, for every $\gb\le \beta_0$ we have proven that
\begin{equation}
 h_c(\gb)\,\ge
 \,  \zeta\, 2^{-n_{\gb}}\, \ge\,  \frac 1{80\mathcal K }\exp\left(-\frac{50  \mathcal K \log 2}{\gb^4\gep^2} \right),
\end{equation}
for some $\gep=\gep(\beta_0)$ sufficiently small but independent of $\gb$.
\qed

\begin{proof}[Proof of Lemma \ref{th:secmom}]
We have
\begin{equation}
  \label{eq:Y2}
  \bE_n(Y_n^2)=\frac 1{n^2}\sum_{1\le i\ne j\le 2^n}\sum_{1\le k\ne l\le 2^n}\bE_n[\delta_i\delta_j]\bE_n[\delta_k\delta_l]
  \bE_n[\delta_i\delta_j\delta_k\delta_l].
\end{equation}
We will consider only the contribution coming from the terms such that $i\ne k,l$ and $j\ne k,l$. The 
remaining terms can be treated similarly and their global contribution is easily seen to be exponentially small 
in $n$. (For instance,  when $i=k$ and $j=l$ one gets
\begin{equation}
  \frac 1{n^2}\sum_{1\le i\ne j\le 2^n}\bE_n[\delta_i\delta_j]^3
  \, \le\,
   \frac1n\bE_n(Y_n)\max_{1\le i<j\le 2^n}\bE_n[\delta_i\delta_j],
\end{equation}
which is exponentially small in $n$, in view of Theorem \ref{th:deltas}.)

\begin{figure}[ht]
\begin{center}
\leavevmode
\epsfxsize =12 cm
\psfragscanon
\psfrag{a}{(a)}
\psfrag{b}{(b)}
\psfrag{k}{$k$}
\psfrag{0}[r]{level $0$: the leaves}
\psfrag{v}[r]{ $v$}
\psfrag{2}[r]{level $2$}
\psfrag{4}[c]{level $n=4$: the root}
\psfrag{3}[c]{$\ldots$}
\psfrag{r}[c]{level $N$: the root}
\epsfbox{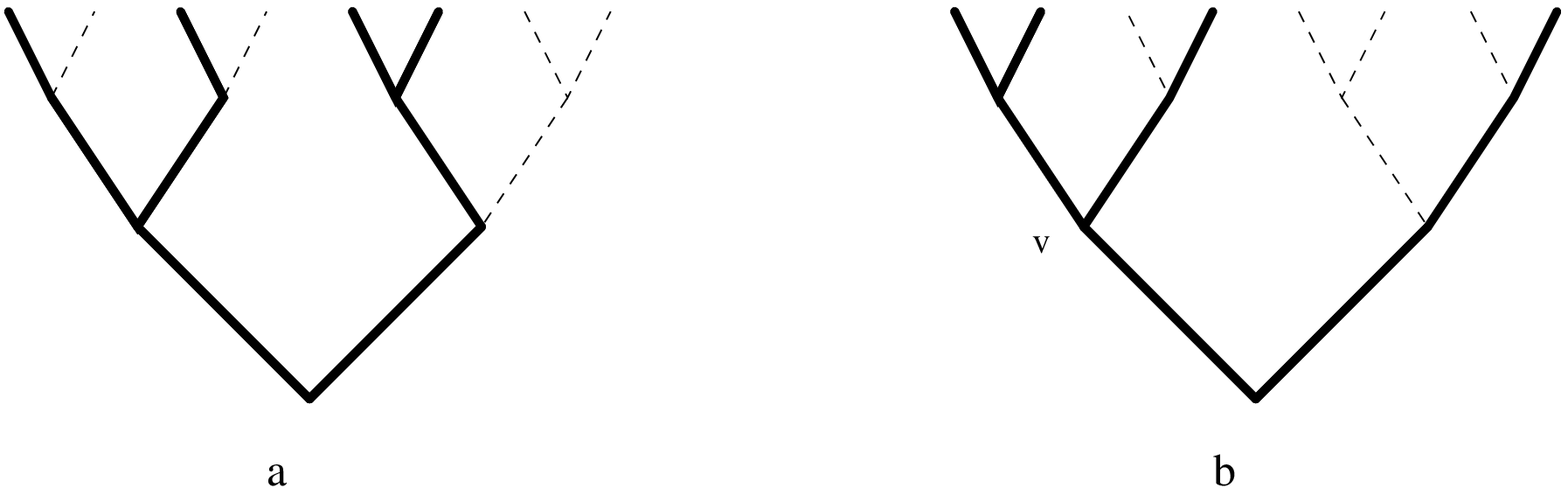}
\end{center}
\caption{The two different possible topologies of the tree $\mathcal
  T^{(n)}_{\{i,j,k,l\}}$.  Case (b) is understood to include also the trees
  where the branch which does not bifurcate is the one on the left, or
  where the sub-branch which bifurcates is the right descendent of the
  node $v$.  We consider only trees where the four leaves are
  distinct, since the remaining ones give a contribution to $\bE_n
  (Y_n^2)$ which vanishes for $n\to\infty$.}
\label{fig:2t}
\end{figure}
From now on, therefore, we assume that $i,j,k,l$ are all distinct. Two cases can occur:
\begin{enumerate}
\item the tree $\mathcal T^{(n)}_{\{i,j,k,l\}}$   (it is better to view it here
has the backbone tree, not as the Galton-Watson tree, see Figure~\ref{fig:treemarg})
has two branches, which
  themselves bifurcate into two sub-branches, {\sl cf.} Fig. \ref{fig:2t}(a)
  for an example.  We call $c$ the level at which the first
  bifurcation occurs ($c=n$ in the example of Fig. \ref{fig:2t}(a)), and
  $a,b$ the levels at which the two branches bifurcate. One has
  clearly $1\le a<c\le n$ and $1\le b<c\le n$. All trees of this form
  can be obtained as follows: first choose a leaf $f_1$,
  between $1$ and $ 2^n$. Then choose $f_2$ among the $2^{a-1}$
  possible ones which join with $f_1$ at level $a$, $f_3$ among the
  $2^{c-1}$ which join with $f_1$ at level $c$ and finally $f_4$ among
  the $2^{b-1}$ which join with $f_3$ at level $b$. 
Clearly we are over-counting the trees (note for example that already 
in the choice of $f_1$ and $f_2$ we are over-counting by a factor $2$), but 
we are only after an {\sl upper bound} for $\bE_n(Y_n^2)$ (the
same remark applies to case (2) below).
We still have
  to specify how to identify $(f_1,f_2,f_3,f_4)$ with a permutation
  of $(i,j,k,l)$. When $(f_1,f_2,f_3,f_4)=(i,j,k,l)$ we get the
  following contribution to \eqref{eq:Y2}:
\begin{equation}
\label{eq:caso11}
\frac1{n^2}  \sum_{1\le a<c\le n}\sum_{1\le b<c\le n}\frac{2^{n+a+b+c-3}}{B_c^{n+a+b+c-3}B_c^{n+a-1} B_c^{n+b-1}},
\end{equation}
where we used Theorem \ref{th:deltas} to write, {\sl e.g.},
$\bE_n[\delta_i\delta_j]=B_c^{-n-a+1}$. Since $B_c=\sqrt 2$ we can
rewrite \eqref{eq:caso11} as
\begin{equation}
  \frac1{\sqrt 2n^2}\sum_{1<c\le n}(c-1)^2 2^{-(n-c)/2},
\end{equation}
which  is clearly bounded as $n$ grows.

If instead $(f_1,f_2,f_3,f_4)=(i,k,j,l)$ or $(f_1,f_2,f_3,f_4)=(i,k,l,j)$,
 one gets
\begin{equation}
\frac1{n^2}  \sum_{1\le a<c\le n}\sum_{1\le b<c\le n}\frac{2^{n+a+b+c-3}}{B_c^{n+a+b+c-3}B_c^{n+c-1} B_c^{n+c-1}},
\end{equation}
which is easily seen to be $O(1/n^2)$.

All the other permutations of $(i,j,k,l)$ give a contribution which
equals, by symmetry, one of the three we just considered.

\item the tree $\mathcal T^{(n)}_{\{i,j,k,l\}}$ has two branches: one of
  them does not bifurcate, the other one bifurcates into two sub-branches, one of which bifurcates into two
sub-sub-branches, {\sl cf.} Figure
  \ref{fig:2t}(b). Let $a_1,a_2,a_3$ be the levels where the three bifurcations occur, ordered so that $1\le a_1<a_2<a_3\le n$. This time,
we choose $f_1$ between $1$ and $2^n$ and then, for $i=1,2,3$, $f_{i+1}$ among the $2^{a_i-1}$ leaves which join with $f_1$ at level $a_i$.
If $(f_1,f_2,f_3,f_4)=(i,j,k,l)$ one has in this case
\begin{multline}
\phantom{move}
\frac1{n^2}  \sum_{1\le a_1<a_2<a_3\le n}\frac{2^{n+a_1+a_2+a_3-3}}{B_c^{n+a_1+a_2+a_3-3}B_c^{n+a_1-1} B_c^{n+a_3-1}}\,
=
\\
\frac1{\sqrt 2 n^2} \sum_{1\le a_1<a_2<a_3\le n}2^{-(n-a_2)/2},
\end{multline}
which is $O(1/n)$. 
Finally, when $(f_1,f_2,f_3,f_4)$ is equal to $(i,k,j,l)$ or 
to $(i,k,l,j)$ one gets
  \begin{multline}
  \phantom{move}
\frac1{n^2}  \sum_{1\le a_1<a_2<a_3\le n}\frac{2^{n+a_1+a_2+a_3-3}}{B_c^{n+a_1+a_2+a_3-3}B_c^{n+a_2-1} B_c^{n+a_3-1}}\,=
\\
\frac1{\sqrt 2 n^2} \sum_{1\le a_1<a_2<a_3\le n}2^{-(n-a_1)/2},
\end{multline}
which is $O(1/n^2)$.
\end{enumerate}
\end{proof}

\section{Marginal relevance of disorder: the non-hierarchical case}\label{sec:nHproof}

Here we prove Theorem \ref{th:main2} and therefore we assume that
\eqref{eq:K222} holds with $\alpha=1/2$.

We choose and fix once and for all a $\gamma \in (2/3,1)$ and set for $h>0$
\begin{equation}
\label{eq:kh1}
  k:=k(h):=\left\lfloor \frac 1h\right\rfloor.
\end{equation}

\smallskip

\begin{rem}\rm
\label{rem:k}
In \cite{cf:DGLT}
 the choice $k(h)= \lfloor 1/\tf(0,h)\rfloor $ was made and 
 it corresponds to choosing $k(h)$ equal to the correlation length
 of the annealed system. In our case $1/\tf(0,h) \stackrel{h\searrow 0}
 \sim 1/(c_K h^2)$ ({\sl cf.} \eqref{eq:annF}) and therefore  \eqref{eq:kh1}
 may look surprising. However, there is nothing particularly deep behind: 
 for  $\ga=1/2$, due to the fact that we have to prove delocalization
 for $h\le  \exp(-c_7/\gb^4)$, choosing $k(h)$ that diverges for small
 $h$ like $1/h$ instead of $1/h^2$ just leads to choosing $c_7$
 different by a factor $2$ (and we do not track the precise value of constants). 
 We take this occasion to stress that it is practical to work always
 with sufficiently large values of $k(h)$, and this can be achieved
 by choosing $c_7$ sufficiently large.
\end{rem}

\smallskip

We divide $\N$ into blocks
\begin{equation}
  \label{eq:blocks}
  B_i:=\{(i-1)k+1,(i-1)k+2,\ldots,ik\}\;\mbox{with\;\;}i=1,2,\ldots .
\end{equation}
From now on we assume that $(N/k)$ is integer, and of course it is also the
number of blocks contained in the interval $\{1,\ldots,N\}$.

We define, in analogy with the hierarchical case, 
\begin{equation}
  A_N\, :=\, \bbE\left(Z_{N,\go}^\gamma\right),
\end{equation}
and we note that, as in \eqref{eq:JensAn}, Jensen's inequality implies that a
sufficient condition for $\tf(\beta,h)=0$ is that $A_N$ does not
diverge when $N\to\infty$.  Therefore, our task is to show that 
for every $\gb_0>0$ we can find $c_7>0$ such that
for every $\beta\le \beta_0$ and $h$ such that
\begin{equation}
  \label{eq:assh}
0<h\le \exp(-c_7/\beta^4), 
\end{equation}
one has that
$\sup_N A_N<\infty$.

\subsection{Decomposition of $Z_{N,\go}$ and change of measure}
The first step is a decomposition of the partition function similar to
that used in \cite{cf:T_cg}, which is a refinement of the strategy employed 
in \cite{cf:DGLT}.  For $0<i\le j$ we let
$Z_{i,j}:=Z_{(j-i),\theta^i\go}$, with $(\theta^i\go)_a:=\go_{i+a},
a\in \N$, {\sl i.e.}, $\theta^i\go$ is the result of the application to
$\go$ of a shift by $i$ units to the left.  We decompose $Z_{N,\go}$
according to the value of the first point ($n_1$) of $\tau$ after $0$,
the last point ($j_1$) of $\tau$ not exceeding $n_1+k-1$, then the
first point ($n_2$) of $\tau$ after $j_1$, and so on. We call $i_r$
the index of the block in which $n_r$ falls, 
and $\ell:=\max\{r:n_r\le N\}$,
see Figure \ref{fig:decomp}. Due to the constraint $N\in\tau$, one has always $i_\ell=(N/k)$.
\begin{figure}[ht]
\begin{center}
\leavevmode
\epsfxsize =12.4 cm
\psfragscanon
\psfrag{0}[c]{\small $0$}
\psfrag{ui}[c]{\tiny $i_0=j_0=0$}
\psfrag{N}[c]{\small N} 
\psfrag{B1}[c]{\small $B_1$}
\psfrag{B2}[c]{\small $B_3$}
\psfrag{B3}[c]{\small $B_4$}
\psfrag{B8}[c]{\small $B_9$}
\psfrag{B9}[c]{\small $B_{10}$}
\psfrag{B10}[c]{\small $B_{11}$}
\psfrag{B13}[c]{\small $B_{14}$}
\psfrag{Z1}[c]{\tiny $Z_{n_1,j_1}$}
\psfrag{Z2}[c]{\tiny $Z_{n_2,j_2}$}
\psfrag{Z3}[c]{\tiny $Z_{n_3,j_3}$}
\psfrag{Z4}[c]{\tiny $Z_{n_4,N}$}
\psfrag{n1}[c]{ \small $n_1$}
\psfrag{j1}[c]{ \small $j_1$}
\psfrag{n2}[c]{ \small $n_2$}
\psfrag{j2}[c]{ \small $j_2$}
\psfrag{n3}[c]{ \small $n_3$}
\psfrag{j3}[c]{ \small $j_3$}
\psfrag{n4}[c]{ \small $n_4$}
\psfrag{j1}[c]{ \small $j_1$}
\psfrag{k}[c]{\small $k$}
\psfrag{2k}[c]{\small $2k$}
\psfrag{n1+k}[c]{ \tiny $n_1+k$}
\psfrag{n2+k}[c]{ \tiny $n_2+k$}
\psfrag{n3+k}[c]{ \tiny $n_3+k$}
\epsfbox{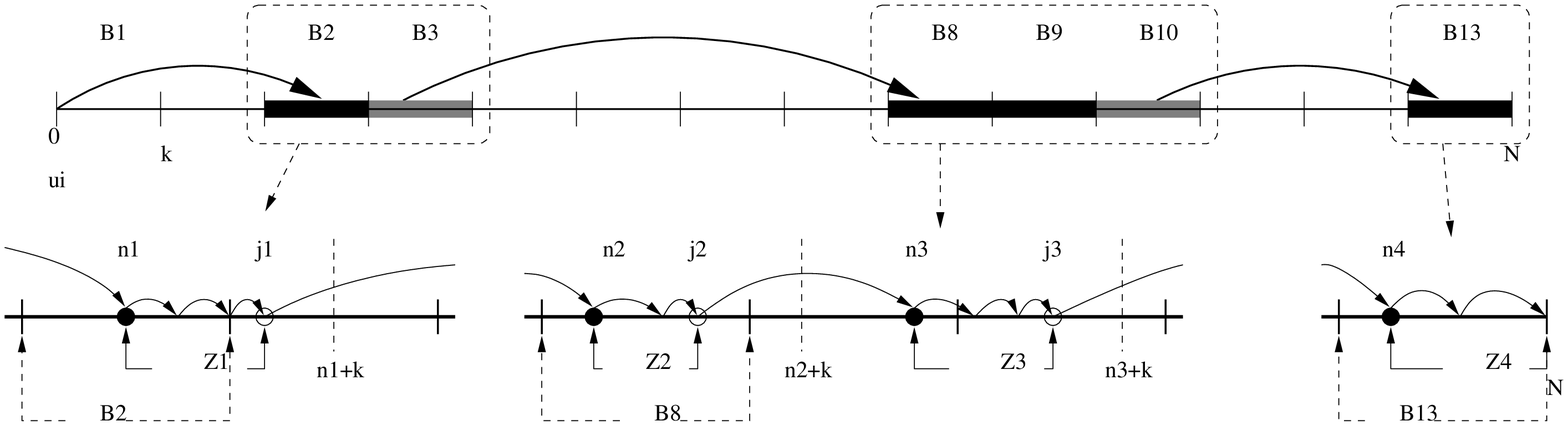}
\end{center}
\caption{\label{fig:decompos} A typical configuration which
  contributes to $\hat Z_\go^{(i_1,\ldots,i_\ell)}$.  In this example
  we have $N/k=14$, $\ell=4$, $i_1=3$, $i_2=9$, $i_3=10$ and
  $i_4=N/k=14$ (by definition $i_\ell =N/k$, {\sl cf.}
  \eqref{eq:dec1}).  Contact points are only in black and grey blocks:
  the blocks $B_{i_j}$, $j=1,\ldots, \ell$ are black and they contain
  one (and only one) point $n_i$. To the
  right of a black block there is either another black block or a grey
  block (except for the last black block, $B_{i_\ell}$, that contains
  the end-point $N$ of the system).  The bottom part of the figure
  zooms on black and grey blocks.  We see that 
  to the right of $n_{i}$ (big black dots) there are
  renewal points before $n_{i}+k$; for $i<\ell$,  $j_i$ is the rightmost one and
  it is marked by a big empty dot (even if it is not the case in the
  figure, it may happen that there is none: in that case $j_i=n_i$).
  Therefore, between empty dots and black dots there is no contact
  point (the origin should be considered an empty dot too).  Note that
  $j_i$ can be in $B_i$, as it is the case for $j_2$, or in $B_{i+1}$,
  as it is the case for $j_1$ and $j_3$.  Going back to the figure on
  top, we observe that the set $M$ of \eqref{eq:M} is
  $\{3,4,9,10,11,14\}$, that is the collection of black and grey
  blocks. We point out that it may happen that a grey block contains
  no point, but it is convenient for us to treat grey blocks as if
  they always contained contact points. It is only to the charges
  $\go$ in black and grey blocks that we apply the change-of-measure
  argument that is crucial for our proof. }
\label{fig:decomp}
\end{figure}

In formulas:
\begin{equation}
\label{eq:dec1}
  Z_{N,\go}=\sum_{\ell=1}^{N/k}\sum_{i_0:=0<i_1<\ldots<i_\ell=N/k}
\hat Z_\go^{(i_1,\ldots,i_\ell)},
\end{equation}
where
\begin{multline}
  \label{eq:Zhat}
  \hat Z_\go^{(i_1,\ldots,i_\ell)}\, :=\, 
  \sum_{n_1\in B_{i_1}}\sum_{j_1=n_1}^{n_1+k-1}\sumtwo{n_2\in B_{i_2}:}{n_2\ge
    n_1+k}\sum_{j_2=n_2}^{n_2+k-1}\ldots \sumtwo{n_{\ell-1}\in B_{i_{\ell-1}}:}
{n_{\ell-1}\ge
    n_{\ell-2}+k}\sum_{j_{\ell-1}=n_{\ell-1}}^{n_{\ell-1}+k-1}\sumtwo
{n_\ell\in B_{N/k}:}{n_\ell\ge n_{\ell-1}+k}\\
 z_{n_1}K(n_1)Z_{n_1,j_1}z_{n_2}K(n_2-j_1)
Z_{n_2,j_2}\, \ldots \, z_{n_\ell}K(n_\ell-j_{\ell-1})Z_{n_\ell,N},
\end{multline}
and $z_n:=e^{\beta\go_n+h-\beta^2/2}$.

Then, from inequality \eqref{eq:fracineq}, we have
\begin{equation}
\label{eq:A_nh}
  A_N\,\le\, \sum_{\ell=1}^{N/k}\sum_{i_0:=0<i_1<\ldots<i_\ell=N/k}
\bbE\left[(\hat Z_\go^{(i_1,\ldots,i_\ell)})^\gamma\right],
\end{equation}
and, 
as in \eqref{eq:hold}, we apply H\"older's inequality to get
\begin{multline}
  \label{eq:Hold}
  \bbE \left[\left(\hat Z^{(i_1,\ldots,i_\ell)}_\go\right)^\gamma\right]\,
=\\
\tilde \bbE\left[\left(\hat Z^{(i_1,\ldots,i_\ell)}_\go\right)^\gamma
\frac{\dd \bbP}{\dd \tilde \bbP}(\go)\right]
\le \left(
\tilde \bbE \hat Z^{(i_1,\ldots,i_\ell)}_\go\right)^\gamma
\left(\bbE \left[\left(\frac{\dd \bbP}{\dd \tilde \bbP}\right)^{\gamma/(1-\gamma)}\right]
\right)^{1-\gamma}.
\end{multline}
The new law $\tilde\bbP:= \tilde\bbP^{(i_1,\ldots,i_\ell)}$ will be
taken to depend on the set $(i_1,\ldots,i_\ell)$. In order to define
it, let first of all
\begin{equation}
  \label{eq:M24}
  M:=M(i_1,\ldots,i_\ell):=\{i_1,i_2,\ldots,i_\ell\}\cup \{i_1+1,i_2+1,\ldots,
i_{\ell-1}+1\}.
\end{equation}
Then, we say that under $\tilde\bbP $ the random vector $\go$  is Gaussian, centered 
and with covariance matrix
\begin{equation}
  \label{eq:covarianze}
  \tilde \bbE(\go_i\go_j)\, =\, \ind_{i=j}-\mathcal C_{ij}\,:=\,
  \begin{cases}
   \ind_{i=j}-H_{ij} & \text{ if there exists } u\in M \text{ such that } i,j \in B_u,
   \\
    \ind_{i=j} & \text{ otherwise,}
  \end{cases}
\end{equation}
and 
\begin{equation}
  H_{ij}\,:=\, 
    \begin{cases}
(1-\gamma)/\sqrt{9\,  k(\log k) \;|i-j|}&\text{ if } i\ne j,
\\
0& \text{ if } i=j.      
\label{eq:H}
    \end{cases}
\end{equation}
Note that all the $\mathcal C_{ij}$'s are non-negative.
It is immediate to check that the $k\times k$ symmetric matrix
$\hat H:=\{H_{ij}\}_{i,j=1}^k$ satisfies
\begin{equation}
  \label{eq:HS_H}
\Vert\hat H\Vert\, :=\,
\sqrt{\sum_{i,j=1}^k H_{ij}^2}\, \le \frac{1-\gamma}2, 
\end{equation}
for $k$ sufficiently large. 
In words: $\omega_n$'s in different
blocks are independent; in blocks $B_u$ with $u\notin M$ they are just
IID standard Gaussian random variables, while if $u\in M$
then the random vector $\{\go_n\}_{n\in B_u}$ has
covariance matrix $I-\hat H$, where $I$ is the $k\times k$ identity matrix.
Note that, since $\Vert\hat H\Vert$ dominates the spectral
radius of $\hat H$, \eqref{eq:HS_H} guarantees 
  that $I -\hat H$ is positive definite (and also that 
  $I -(1-\gamma)^{-1}\hat H$ is positive definite, that will be needed just below).

The last factor in the right-hand side of \eqref{eq:Hold} is easily
obtained recalling \eqref{eq:dets} and independence of the $\go_n$'s
in different blocks, and one gets
\begin{equation}
\label{eq:RN_nh}
\left(\tilde \bbE\left[ \left(\frac{\dd \bbP}{\dd \tilde \bbP}\right)^{\gamma/(1-\gamma)}\right]
\right)^{1-\gamma}=\left(\frac{\det(I-\hat H)}
{(\det(I-1/(1-\gamma)\hat H))^{1-\gamma}}
\right)^{|M|/2}.
\end{equation}
Since $\hat H$ has trace zero and its (Hilbert-Schmidt) norm satisfies
\eqref{eq:HS_H}, one can apply $\det(I-\hat H)\le
\exp(-\mbox{Trace}(\hat H))=1$ and \eqref{eq:dano} (with $V$ replaced by
$\hat H$ and $\gep$ by $1$) to get that the right-hand side of
\eqref{eq:RN_nh} is bounded above by $ \exp(|M|/2)$, which in turn is 
bounded by $\exp(\ell)$.  Together with \eqref{eq:Hold} and
\eqref{eq:A_nh}, we conclude that
\begin{equation}
  A_N\le\sum_{\ell=1}^{N/k}\sum_{i_0:=0<i_1<\ldots<i_\ell=N/k}
e^{{\ell}} \left[\tilde \bbE \hat Z_\go^{(i_1,\ldots,i_\ell)}
\right]^\gamma.
\label{eq:sqbr}
\end{equation}

\subsection{Reduction to a non-disordered model  }

We wish to  bound  the right-hand side of \eqref{eq:sqbr} 
with the
partition function of a non-disordered pinning model in the
delocalized phase, which goes to zero for large $N$.  We start by
claiming that 
\begin{multline}
\label{eq:claimU}
  \tilde \bbE \hat Z_\go^{(i_1,\ldots,i_\ell)}\, \le\,  \sum_{n_1\in B_{i_1}}\ldots \sumtwo{n_\ell\in B_{N/k}:}{n_\ell\ge 
n_{\ell-1}+k} K(n_1)K(n_2-j_1)\ldots K(n_\ell-j_{\ell-1})\\
\times U(j_1-n_1)U(j_2-n_2)\ldots U(N-n_\ell),
\end{multline}
where
\begin{equation}
\label{eq:U}
  U(n)\, =\, c_{8}\bP(n\in\tau)\bE\left[e^{-\beta^2\sum_{1\le i<j\le {n/2}}H_{ij}\delta_i
\delta_j}
\right],
\end{equation}
and $c_8$ is a positive constant depending only on $K(\cdot)$. This is
proven in Appendix \ref{sec:appprova}. We are also going to make use
of:

\medskip

\begin{lemma}
\label{th:condlemma}
\label{th:lemma_cg}
   There exists $C_2=C_2(K(\cdot))<\infty$ such that if, for some $\eta>0$,
  \begin{equation}
    \label{eq:condlemma1}
    \sum_{j=0}^{k-1}U(j)\le \eta  {\sqrt k}
  \end{equation}
and 
\begin{equation}
  \label{eq:condlemma2}
  \sum_{j=0}^{k-1}\sum_{n\ge k}U(j)K(n-j)\le \eta,
\end{equation}
then there exists $C_1=C_1(\eta, k,K(\cdot))$ 
such that
the right-hand side of \eqref{eq:claimU} is  bounded above by 
\begin{equation}
  C_1 \eta^\ell C_2^{\ell}\prod_{r=1}^\ell\frac1{(i_r-i_{r-1})^{3/2}}.
\end{equation}
\end{lemma}
\medskip
It is important to note that $C_2$ does not depend on $\eta$.

Lemma~\ref{th:condlemma} is a small variation on 
 \cite[Lemma 3.1]{cf:T_cg}, but,
 both because the model we are considering is somewhat different
 and for sake of completeness, we give
 the details
of the proof in Appendix \ref{sec:appprova}.
\smallskip

Now assume that conditions \eqref{eq:condlemma1}-\eqref{eq:condlemma2}
are verified for some $\eta$. Collecting \eqref{eq:sqbr},
\eqref{eq:claimU} and Lemma \ref{th:lemma_cg}, we have then
\begin{equation}
  A_N\le C_1^\gamma\sum_{\ell=1}^{N/k}\sum_{i_0:=0<i_1<\ldots<i_\ell=N/k}
\left(\eta^\gamma\,C_2^{\gamma}e 
\right)^\ell\prod_{r=1}^\ell\frac1{(i_r-i_{r-1})^{(3/2)\gamma}}.
\label{eq:puremod}
\end{equation}
In the right-hand side we recognize, apart from the irrelevant
multiplicative constant $C_1^\gamma$, the partition function of a
non-random $(\beta=0)$ pinning model with $N$ replaced by $N/k$,
$K(\cdot)$ replaced by
\begin{equation}
  \hat K(n)\, =\, \frac1{n^{(3/2)\gamma}}\frac1{\sum_{i\ge1}i^{-(3/2)\gamma}},
\end{equation}
and $h$ replaced by
\begin{equation}
  \label{eq:hbar}
  \hat h:=\log \left(\eta^\gamma\,C_2^{\gamma}e
  \sum_{n\in\N}\frac1{n^{(3/2)\gamma}}\right).
\end{equation}
Note that $\hat K(\cdot)$ is normalized to be a probability measure on $\N$,
which is possible since (by assumption) $\gamma>2/3$, and that it has a
power-law tail with exponent $(3/2)\gamma>1$. Thanks to Lemma
\ref{th:puro} below, one has that the right-hand side of \eqref{eq:puremod} tends to zero
for $N\to\infty$ whenever
\begin{equation}
\label{eq:condeta}
\hat h<0.
\end{equation}
Therefore, if  $\eta$ is so small that \eqref{eq:condeta} holds,
we can conclude that $A_N$ tends to zero for $N\to\infty$ and therefore $\tf(\beta,h)=0$.

The proof of Theorem \ref{th:main2} is therefore concluded once we prove

\medskip
\begin{proposition}
\label{th:stimeU}
Fix $\eta>0$ such that \eqref{eq:condeta} holds. For every $\gb_0>0$
there exists $0<c_7<\infty$ such that if $\beta\le \gb_0$ and $0<h\le
\exp(-c_7/\beta^4)$, conditions
\eqref{eq:condlemma1}-\eqref{eq:condlemma2} are verified.
\end{proposition}
\medskip

\noindent
\begin{proof}[Proof of Proposition \ref{th:stimeU}]
We need to show that the two hypotheses of Lemma~\ref{th:condlemma} hold
and for this 
we are going to use the following result:
\medskip

\begin{lemma}
\label{th:CE}
Under the law $\bP$, the  random variable 
\begin{equation}
W_L\, :=\,  (\sqrt{L} \log L)^{-1}
\sum_{1\le i<j\le L} \delta_i
        \delta_j /\sqrt{j-i},
\end{equation}        
        converges in distribution, as $L$ tends to $\infty$, to $c \vert Z \vert$
        ($Z\sim N(0,1)$ and $c$ a positive constant).
\end{lemma}
\medskip

This lemma, 
the proof of which may be found just below (together with the explicit value of $c$), directly implies that,
if we set $S(a,L):=  \bE\left[\exp\left(- a W_L\right)\right]$, we have
 $\lim_{a\to\infty} \lim_{L\to\infty}S(a,L)\, = \, 0$ and, by the monotonicity
 of $S (\cdot, L)$, we get
\begin{equation}
  \label{eq:small}
  \lim_{a, L\to\infty}   S(a,L)\, = \, 0.
\end{equation}

Let us  verify \eqref{eq:condlemma1}.
Note first of all that ({\sl cf.} \eqref{eq:U} and \eqref{eq:H})
\begin{multline}
  U(n)\, =\, c_8 \bP(n\in\tau)
     S\left(\beta^2(1-\gamma)\sqrt{\log k}\,\sqrt{\frac {n/2}{9\,  k}}\,
        \frac{\log (n/2)}{
          {\log k}}
        , \frac n2\right)\\
         =:\, c_8 \bP(n\in\tau) s_\gb(k,n).
\end{multline}
We recall also that
(\cite[Th. B]{cf:Doney})
\begin{equation}
\label{eq:doney}
  \bP(n\in\tau)\stackrel{n\to\infty}\sim \frac1{2\pi C_K \sqrt n},
\end{equation}
and therefore there exists $c_9>0$ such that for every $n \in \N$
\begin{equation}
\label{eq:doney-bound}
  \bP(n\in\tau)\, \le \,  \frac{c_9}{\sqrt {n}}.
\end{equation}
Split the sum in \eqref{eq:condlemma1} according to whether
$j \le \gd k$ or not ($\gd=\gd(\eta) \in (0,1)$
is going to be chosen below).
By using $S(a,L)\le 1$ (in the case $j \le \gd k$) and  \eqref{eq:doney-bound}
we obtain
\begin{equation}
\label{eq:interm4}
\sum_{j=0}^{k-1} U(j) \, \le \, c_8+c_8 c_9 \sum_{j=1}^{\gd k}
\frac1{\sqrt{j}} \, +\, c_8 c_9 \sum_{j= \gd k +1}^{k-1}
\frac1{\sqrt{j}}\, s_\gb(k,j).
\end{equation}
Since if $c_7$ is chosen sufficiently large
\begin{equation}
  \beta^2\sqrt{\log k}\ge\sqrt{c_7-\beta^4\log 2}\ge \sqrt{c_7}/2, 
\end{equation}
and since $k$ may be made large by increasing $c_7$, 
we directly see that \eqref{eq:small} implies
that $s_\gb(k,j)$ may be made  smaller than (say) 
$\gd$ for every $\delta k <j<k$ by choosing $c_7$ sufficiently large. Therefore \eqref{eq:interm4} implies
 \begin{equation}
\label{eq:interm5}
\sum_{j=0}^{k-1} U(j) \, \le \, 4c_8 c_9 (\sqrt{\gd} + \gd) \sqrt{k}.
\end{equation}       
By choosing $\gd=\gd(\eta) $ such that         $4c_8 c_9 (\sqrt{\gd} + \gd) \le \eta$,
we have \eqref{eq:condlemma1}.
The proof of \eqref{eq:condlemma2} is absolutely analogous to the proof
of \eqref{eq:condlemma1} and it is therefore omitted. 
\end{proof}

\subsection{Proof of Lemma~\ref{th:CE}}
We introduce the notation
\begin{equation}
Y_L^{(i)}\, :=\, \sum_{j=i+1}^{L} \frac{\gd_j}{\sqrt{j-i}}
,\text{\  so that \  }
W_L \, =\, \frac 1{\sqrt{L}\log L} \sum_{i=1}^{L-1}\gd_i Y_L^{(i)}.
\end{equation}
Let us observe that, thanks to the renewal property of $\tau$, under
$\bP (\cdot \vert \gd_i=1)$, $Y_L^{(i)}$ is distributed like
$Y_{L-i}:= Y_{L-i}^{(0)}$ (under $\bP$).  
 The first step in the proof is  observing that, in view of \eqref{eq:doney-bound},
\begin{multline}
  \bE \left[\frac 1{\sqrt{L}\log L} \sum_{i=(1-\gep)L }^{L-1}\gd_i Y_L^{(i)}\right]\,
=\\
\frac1{\log L \sqrt L}\sum_{i=(1-\gep)L}^{L-1}\sum_{j=i+1}^L\frac{
\bP(i\in\tau)\bP(j-i\in\tau)}{\sqrt{j-i}}
\, 
=
\, O(\gep),
\end{multline}
uniformly in $L$,
so we can focus on studying $W_{L, \gep}$, defined as $W_L$, but stopping the 
sum over $i$ at $(1-\gep) L$. At this point we use that
\begin{equation}
\label{eq:Chung-Erdos}
\lim_{L \to \infty} \frac{Y_L}{\log L} \, =\, \frac 1{2\pi C_K} \, =: \hat c_K,
\end{equation}
in $L^2(\bP)$ (and hence in $L^1(\bP)$). We postpone the proof of \eqref{eq:Chung-Erdos}
and observe that, thanks to the properties of the logarithm,
it implies that for every $\gep>0$
\begin{equation}
\label{eq:use}
\lim_{L \to \infty}
\sup_{q \in [\gep, 1]} \bE \left[ \left \vert 
\frac 1{\log L} \sum_{j=1}^{qL} \frac {\gd _j}{\sqrt{j}} \, - \hat c_K \right \vert \right]\, =\, 0.
\end{equation}

Let us write
\begin{equation}
R_L \, := W_{L, \gep} \, -\, \frac{\hat c_K} {\sqrt{L} } \sum_{i=1}^{(1-\gep) L} \gd _i
\end{equation}
and note that $ L^{-1/2} \sum_{i=1}^{(1-\gep) L} \gd _i$ converges in law toward
 $\sqrt{(1-\gep)/(2\pi C_K^2)} \, \vert Z\vert$.
 This follows 
  directly by using that the event $ \sum_{i=1}^{L} \gd _i \ge  m$
 is the event $\tau_m \le L$ ($\tau_m$ is of course the 
$m$-th point in $\tau$ after $0$) and by using the fact that $\tau_1$ is in
 the domain of attraction of the positive stable law of index $1/2$
 \cite[VI.2 and XI.5]{cf:Feller2}.
It suffices therefore to show that $\bE [\vert R_L\vert] $ tends to zero. 
We have
\begin{multline}
  \bE \left[\vert R_L\vert\right]\, \le\, 
  \frac {1}{\sqrt{L}}   \sum_{i=1}^{(1-\gep) L}
  \bE [ \gd_i ] \bE\left [ \left.\bigg\vert \frac{Y^{(i)}_L}{\log L}-\hat c_K \bigg \vert \;\right\vert
  \delta_i=1 \right] \, =\, 
  \\
  \frac {1}{\sqrt{L}}   \sum_{i=1}^{(1-\gep) L}
  \bE [ \gd_i ] \bE \left[ \left\vert \frac{Y_{L-i}}{\log L}-\hat c_K \right\vert 
  \right] \, =\, o(1), 
\end{multline}
where in the last step we have used \eqref{eq:use} and \eqref{eq:doney-bound}.

Note that we have also proven that $c=(2\pi)^{-3/2} C_K^{-2}$ in the statement of Lemma \ref{th:CE}.

\medskip

We are therefore left with the task of proving \eqref{eq:Chung-Erdos}.
This result has been already proven \cite[Th.~6]{cf:chungerdos} when
$\tau$ is given by the successive returns to zero of a centered,
aperiodic and irreducible random walk on $\bbZ$ with bounded variance
of the increment variable. Note that, by well established local limit
theorems, for such a class of random walks we have \eqref{eq:doney}.
Actually in \cite{cf:chungerdos} it is proven that
\eqref{eq:Chung-Erdos} holds almost surely as a consequence of
$\text{var}_\bP(Y_L)=O(\log L)$.  What we are going to do is
simply to re-obtain such a bound, by repeating the steps in
\cite{cf:chungerdos} and using \eqref{eq:doney}-\eqref{eq:doney-bound}, for the general
renewal processes that we consider (as a side remark: also in our
generalized set-up, almost sure convergence holds).

The proof goes as follows:
by using \eqref{eq:doney} it is straightforward to see that
$\lim_{L\to \infty}\bE [Y_L/\log L]=\hat c_K$, so that we are done
if we show that $\text{var}_\bP(Y_L/\log L)$ vanishes as $L\to \infty$.
So we start by observing that
\begin{equation}
\label{eq:CEst1}
\text{var}_\bP(Y_L) \, =\, \sum_{i, j} \frac{\bE[ \gd_i \gd_j]-
\bE[ \gd_i] \bE[ \gd_j]
 }{\sqrt{ij}} \, =\, 2 \sum_{i=1}^{L-1}\sum_{j=i+1}^L
  \frac{\bE[ \gd_i \gd_j]-
\bE[ \gd_i] \bE[ \gd_j]
 }{\sqrt{ij}}
 + O(1),
\end{equation}
by \eqref{eq:doney-bound}. Now we compute 
\begin{equation}
\begin{split}
\sum_{i=1}^{L-1}\sum_{j=i+1}^L
  \frac{\bE[ \gd_i \gd_j]-
\bE[ \gd_i] \bE[ \gd_j]
 }{\sqrt{ij}} 
 \, &=\,
 \sum_{i=1}^{L-1} \frac{\bE[\gd_i]}{\sqrt{i}}
 \left[\sum_{j=1}^{L-i} \frac{\bE[\gd_j]}{\sqrt{j+i}}-
 \sum_{j=i+1}^{L} \frac{\bE[\gd_j]}{\sqrt{j}}
 \right]
 \\
 &\le \, 
 \sum_{i=1}^{L-1} \frac{\bE[\gd_i]}{\sqrt{i}}
 \left[\sum_{j=1}^{L-i} \frac{\bE[\gd_j]}{\sqrt{j+i}}-
 \sum_{j=i+1}^{L} \frac{\bE[\gd_j]}{\sqrt{j+i}}
 \right]
 \\
 &\le \, 
 \sum_{i=1}^{L-1} \frac{\bE[\gd_i]}{\sqrt{i}}
 \sum_{j=1}^{i} \frac{\bE[\gd_j]}{\sqrt{j+i}}
 \\
 \le \, \sum_{i=1}^{L-1} &\frac{\bE[\gd_i]}{i} 
  \sum_{j=1}^{i} {\bE[\gd_j]} \le  c_9^2  
  \sum_{i=1}^{L-1} \frac{1}{i^{3/2}} 
  \sum_{j=1}^{i} \frac 1{j^{1/2}} = O(\log L),
 \end{split}
\end{equation}
where, in the last line, we have used \eqref{eq:doney-bound}.
In view of \eqref{eq:CEst1}, we have obtained 
$\text{var}_\bP(Y_L)=O(\log L)$ and
the proof  \eqref{eq:Chung-Erdos}, and therefore of Lemma~\ref{th:CE} is complete.
\qed

        

\begin{subappendices}
\section{Some technical results and useful estimates}

\subsection{Two results on renewal processes}
The first result concerns the non-disordered pinning model and is well known:
\begin{lemma}
\label{th:puro}
  Let $ K(\cdot)$ be a probability on $\N$ which satisfies \eqref{eq:K222} for
some $\alpha>0$. If $ h<0$, we have
that 
\begin{equation}
\label{eq:lemma1}
\sum_{\ell=1}^N \sum_{i_0:=0<i_1<\ldots<i_\ell=N}
e^{h\ell}\prod_{r=1}^\ell K(i_r-i_{r-1})\stackrel {N\to\infty}\longrightarrow0.  
\end{equation}
\end{lemma}
\medskip

This is implied by \cite[Th. 2.2]{cf:Book}, since the left-hand side of 
\eqref{eq:lemma1} is nothing but the partition function of the homogeneous
pinning model of length $N$, whose critical point is $h_c=0$ ({\sl cf.} also 
\eqref{eq:annF}).

\bigskip

The second fact we need is
\medskip

\begin{lemma}
\label{th:condiz}
There exists a positive constant $c$, which depends only on
$K(\cdot)$, such that for every positive function $f_N(\tau)$ which depends
only on $\tau\cap\{1,\ldots,N\}$ one has
  \begin{equation}
  \label{eq:condiz}
\sup_{N>0}    \frac{\bE[f_N(\tau)|2N\in\tau]}{\bE[f_N(\tau)]}\le c.
  \end{equation}
\end{lemma}

\medskip

\noindent
\begin{proof} 
The statement follows by writing $f_N(\tau)$
as $f_N(\tau)\sum_{n=0}^N \ind_{\{X_N=n\}}$, where $X_N$ is the last renewal epoch up
to (and including) $N$, and using the bound
$$
\sup_N \max_{n=0,\ldots,N}\frac{
\bP(X_N=n \vert 2N \in \tau)}{\bP(X_N=n )} =:c<\infty ,$$
which  is equation (A.15) in \cite{cf:DGLT} 
(this has been proven also in  \cite{cf:T_cg}, where the proof is repeated to show 
that $c$ can be chosen as a function of $\ga$ only).
\end{proof}

\smallskip

\subsection{Proof of \eqref{eq:claimU}}
\label{sec:appprova}
Defining the event
\begin{equation}
  \label{eq:Omega}
  \Omega_{\underline n,\underline j}:=\{N\in\tau\;\;\mbox{and}\;\;
\{j_{r-1},\ldots,n_r\}\cap \tau=\{j_{r-1},n_r\}\;\;\mbox{for all}\;\;
  r=1,\ldots,\ell\}, 
\end{equation}
with the convention that $j_0:=0$, 
we have
\begin{equation}
  \hat Z_\go^{(i_1,\ldots,i_\ell)}=\sum_{n_1\in B_{i_1}}\ldots \sumtwo{n_\ell\in B_{N/k}:}{n_\ell\ge 
n_{\ell-1}+k}\bE\left[e^{\sum_{n=1}^N(\beta\go_n+h-\beta^2/2)\delta_n};\Omega_{\underline n,
\underline j}\right].
\end{equation}
Since $\tilde\bbP$ is a Gaussian measure and $\delta_i^2=\delta_i$ for 
every $i$,
the computation of $\tilde\bbE \hat Z_\go^{(i_1,\ldots,i_\ell)}$ is
immediate:
\begin{equation}
  \tilde\bbE \hat Z_\go^{(i_1,\ldots,i_\ell)}=\sum_{n_1\in B_{i_1}}\ldots \sumtwo{n_\ell\in B_{N/k}:}{n_\ell\ge 
n_{\ell-1}+k}\bE\left[
e^{h\sum_{n=1}^N\delta_n-\beta^2/2\sum_{i,j=1}^N \mathcal C_{ij}\delta_i\delta_j}
    ;\Omega_{\underline n,
      \underline j}\right].
\end{equation}
In view of $\mathcal C_{ij}\ge 0$, we obtain an upper bound by
neglecting in the exponent the terms such that $n_r\le i\le j_{r}$ and
$n_{r'}\le j\le j_{r'}$ with $r\ne r'$. At that point, the $\bE$
average may be factorized, by using the renewal property,
 and we obtain (recall that $\mathcal C_{ii}=0$)
\begin{equation}
\begin{split}
  \tilde\bbE\hat Z_\go^{(i_1,\ldots,i_\ell)}\, \le\, & \sum_{n_1\in B_{i_1}}\ldots \sumtwo{n_\ell\in B_{N/k}:}{n_\ell\ge 
    n_{\ell-1}+k}K(n_1)\ldots K(n_\ell-j_{\ell-1})\\
  &\times \prod_{r=1}^\ell\bE\left[\left.e^{h\sum_{i=n_r}^{j_r}\delta_i-\beta^2\sum_{n_r\le i< j\le j_r}\mathcal C_{ij}
        \delta_i\delta_j}\ind_{\{j_{r}\in\tau\}}\right|n_r\in \tau\right], 
\end{split}
\end{equation}
with the convention that $j_\ell:=N$.
We are left with the task of proving that
\begin{equation}
\label{eq:auxU}
  \bE\left[\left.e^{h\sum_{i=n_r}^{j_r}\delta_i-\beta^2\sum_{n_r\le i<j\le j_r}\mathcal C_{ij}
\delta_i\delta_j}\ind_{\{j_{r}\in\tau\}}\right|n_r\in \tau\right]\le U(j_r-n_r),
\end{equation}
with $U(\cdot)$ satisfying \eqref{eq:U}.
We remark first of all that the left-hand side of \eqref{eq:auxU} equals
\begin{equation}
  \bP(j_r-n_r\in\tau)\bE\left[\left.e^{h\sum_{i=n_r}^{j_r}\delta_i-\beta^2\sum_{n_r\le i<j\le j_r}\mathcal C_{ij}
\delta_i\delta_j}\right|n_r\in \tau,j_r\in\tau\right].
\end{equation}
Since by construction $j_r-n_r<k(h)=\lfloor 1/h\rfloor$, 
one has
\begin{equation}
e^{h\sum_{i=n_r}^{j_r}\delta_i}\, \le\,  e.
\end{equation}
As for the remaining average, assume without loss
of generality that $|\{n_r,n_r+1,\ldots,j_r\}\cap B_{i_r}|\ge (j_r-n_r)/2$
(if this is not the case, the inequality clearly holds with $B_{i_r}$
replaced by $B_{i_r+1}$ and the  arguments which follow are trivially
modified).  Then,
\begin{multline}
\label{eq:uffa}
  \bE\left[\left.e^{-\beta^2\sum_{n_r\le i< j\le j_r}\mathcal C_{ij}
\delta_i\delta_j
}\right|n_r\in \tau,j_r\in\tau
\right]\, \le
\\
 \bE\left[\left.
\exp\left(-\beta^2\sum_{0<i<j\le (j_r-n_r)/2}\delta_i\delta_j H_{ij}\right)
\right| j_r-n_r\in\tau \right].
\end{multline}
Finally, 
 the conditioning in \eqref{eq:uffa} can be eliminated using Lemma \ref{th:condiz}, and
\eqref{eq:claimU} is proved.
\qed

\subsection{Proof of Lemma \ref{th:condlemma} }

In this proof (and in the statement) two positive numbers $C_1$ and $C_2$ appear. 
$C_1$ is going to change along with the steps
of the proof: it depends on $\eta$, $k$ and on $K(\cdot)$. 
$C_2$ instead is chosen once and for all below and it 
depends only on $K(\cdot)$. 
We start by giving a name to  the right-hand side of \eqref{eq:claimU}:
\begin{multline}
  \label{eq:restart}
  Q\, :=\, 
  \sum_{n_1\in B_{i_1}}\sum_{j_1=n_1}^{n_1+k-1}\sumtwo{n_2\in B_{i_2}:}{n_2\ge
    n_1+k}\sum_{j_2=n_2}^{n_2+k-1}\ldots \sumtwo{n_{\ell-1}\in B_{i_{\ell-1}}:}
{n_{\ell-1}\ge
    n_{\ell-2}+k}\sum_{j_{\ell-1}=n_{\ell-1}}^{n_{\ell-1}+k-1}\sumtwo
{n_\ell\in B_{N/k}:}{n_\ell\ge n_{\ell-1}+k}\\
 K(n_1)\ldots K(n_\ell-j_{\ell-1})
U(j_1-n_1)\ldots   U(j_{\ell-1}-n_{\ell-1}) U(N-n_\ell)
 .
\end{multline}
Since $N -n _\ell <k$, we can get rid of $U(N-n_\ell)\, (\le c_8
 \bP(N-n_\ell \in \tau) )$ and of the right-most sum (on $n_\ell$),
 replacing $n_\ell$ by $N$, by paying a price that depends on $k$ and
 $K(\cdot)$ (this price goes into $C_1$).  Therefore we have
 \begin{equation}
  \label{eq:lems1}
 Q\, \le \, 
 C_1\,  \sum_{n_1\in B_{i_1}} \ldots
 \sum_{j_{\ell-1}=n_{\ell-1}}^{n_{\ell-1}+k-1}
 K(n_1)\ldots K(n_\ell-j_{\ell-1})
U(j_1-n_1)\ldots U(j_{\ell-1}-n_{\ell-1})
 ,
\end{equation}
where by convention from now on $n_\ell:=N$.
Now we single out the long jumps. The set of long jump arrival points is
defined as
\begin{equation}
J\, =\, J( i_1,i_2, \ldots, i_\ell)\, :=\, \left\{r:\, 1 \le r \le \ell, \, i_r>i_{r-1}+2\right\},
\end{equation}
and the definition guarantees that a long jump $\{j_{r-1},\ldots,n_r\}$ contains at least one whole
block with no renewal point inside. For $r\in J$ 
we use the bound
\begin{equation}
K(n_r-j_{r-1}) \, \le \, \frac{C_2}{(i_r-i_{r-1})^{3/2} k^{3/2}}, 
\end{equation}
and we stress that we may and do choose $C_2$ depending only on $K(\cdot)$.
For later use, we choose $C_2\ge 2^{3/2}$.
This leads to
\begin{multline}
  \label{eq:lems3}
  Q\, \le \, 
 C_1\,  {k^{-3|J| /2}}
 \prod_{r \in J} \frac{C_2}{(i_r-i_{r-1})^{3/2}}
 \\
\times \sum_{n_1\in B_{i_1}} \ldots
 \sum_{j_{\ell-1}=n_{\ell-1}}^{n_{\ell-1}+k-1}
 \left( 
 \prod_{r\in \{1,\ldots, \ell\}\setminus J} 
 K(n_r- j_{r-1})
 \right)
 \,
U(j_1-n_1)\ldots U(j_{\ell -1}-n_{\ell-1}).
\end{multline}
Now we perform the sums in \eqref{eq:lems3} and bound the outcome by
using the assumptions \eqref{eq:condlemma1} and \eqref{eq:condlemma2}.

We first sum over $j_{r-1}$, $r \in J$, keeping of course
into account the constraint $0\le j_{r-1}-n_{r-1}<k$. By using
 \eqref{eq:condlemma1} such sum yields at most $(\eta \sqrt{k})^{\vert J\vert}$
 if $1 \notin J$. If $1 \in J$, for $r=1$ then $j_0=0$ and there is no summation: 
 we can still bound the sum by $(\eta \sqrt{k})^{\vert J\vert}$,
 provided that we change the constant $C_1$.

Second, we sum over $j_{r-1},n_r$ for 
$r \in  \{1,\ldots, \ell\}\setminus J$ and use 
\eqref{eq:condlemma2}. Once again we have to treat separately the case
$r=1$, as above. But if $1 \notin \{1,\ldots, \ell\}\setminus J$ we directly see that
the summation is bounded by $\eta^{\ell -\vert J \vert}$.

Finally, we have to sum over $n_r$, for $r\in J$. The summand does not depend
on these variables anymore, so this gives at most $k^{\vert J \vert }$.

Putting these estimates together we obtain
\begin{equation}
\label{eq:lems4}
Q\, \le \, 
 C_1\,  
 \frac{
 (\eta \sqrt{k})^{\vert J\vert }
 \eta^{\ell -\vert J \vert  } k^{\vert J \vert}  
   }
 {k^{3|J| /2}}
 \prod_{r \in J} \frac{C_2}{(i_r-i_{r-1})^{3/2}}\, \le\, 
 C_1 \eta^\ell C_2^\ell 
 \prod_{r=1}^{\ell} \frac{1}{(i_r-i_{r-1})^{3/2}},
\end{equation}
where, in the last step, we have used $C_2\ge 2^{3/2}$. 
The proof of Lemma~\ref{th:condlemma} is therefore 
complete.
\qed
\end{subappendices}

\chapter{Disorder relevance at marginality and critical point shift}\label{DISRELCPS}

\section{introduction}

\subsection{Relevant, irrelevant and marginal disorder}
\label{sec:RIM}
The renormalization group approach to disordered statistical mechanics
systems introduces a very interesting viewpoint on the role of
disorder and on whether or not the critical behavior of a quenched
system coincides with the critical behavior of the corresponding
{\sl pure} system.  The Harris criterion \cite{cf:Harris} is based on such an
approach and it may be summarized in the following way: if the
specific heat exponent of the pure system is negative, then a {\sl
  small} amount of disorder does not modify the critical properties of
the pure system ({\sl irrelevant disorder regime}), but if the
specific heat exponent of the pure system is positive then even an
arbitrarily small amount of disorder may lead to a quenched critical
behavior different from the critical behavior of the pure system.
 
A class of disordered models on which such ideas have been applied by
several authors is the one of pinning models (see {\sl e.g.}
\cite{cf:FLNO,cf:DHV} and the extensive bibliography in
\cite{cf:Book,cf:GLT_marg}).  The reason is in part due to the
remarkable fact that pure pinning models are 
exactly solvable models for which, by tuning a parameter, one can
explore all possible values of the specific heat exponent
\cite{cf:Fisher}. As a matter of fact, the validity of Harris
criterion for pinning models in the physical literature finds a rather
general agreement. Moreover, for the pinning models the
renormalization group approach goes beyond the critical properties and
yields a prediction also on the location of the critical point.
 
Recently, the Harris criterion predictions for pinning models have
been put on firm grounds in a series of papers \cite{cf:Ken,cf:T_cmp,cf:DGLT,cf:AZ_new} and some of these
rigorous results go even beyond the predictions. Notably in
\cite{cf:GT_cmp} it has been shown that disorder has a  smoothing effect in 
this class of models (a fact that is not a consequence of the Harris
criterion and that does not find unanimous agreement in the physical literature).
 
However, a substantial amount of the literature on disordered pinning
and Harris criterion revolves around a specific issue: what happens if
the specific heat exponent is zero ({\sl i.e.} at {\sl marginality})?
This is really a controversial issue in the physical literature,
started by the disagreement in the conclusions of \cite{cf:FLNO} and
\cite{cf:DHV}.  In a nutshell, the disagreement lies on the fact that
the authors of \cite{cf:FLNO} predict that disorder is irrelevant at
marginality and, notably, that quenched and annealed critical points
coincide at small disorder, while the authors of \cite{cf:DHV} claim
that disorder is relevant for arbitrarily small disorder, leading to a
critical point shift of the order of $\exp( -c \gb^{-2})$ ($c>0$) for
$\gb \searrow 0$ ($\gb^2$ is the disorder variance).

Recently we have been able to prove that, at marginality, there is a
shift of the critical point induced by the presence of disorder
\cite{cf:GLT_marg}, at least for
Gaussian disorder. We have actually proven that the shift is at least
 $\exp( -c \gb^{-4})$. The purpose of the present work is to go
beyond \cite{cf:GLT_marg} in three aspects:
\begin{enumerate}
\item We want to deal with rather general disorder variables: we are going to
assume only that the exponential moments are finite.
\item We are going to improve the bound $\exp( -c \gb^{-b})$, $b=4$,
  on the critical point shift, to $b=2+\epsilon$ ($\epsilon>0$
  arbitrarily small, and $c=c(b)$).
\item We will prove our results for a generalized class of pinning
  models. Pinning models are based on discrete renewal processes,
  characterized by an inter-arrival distribution which has power-law
  decay (the exponent in the power law parametrizes the model and
  varying such parameter one explores the different types of critical
  behaviors we mentioned before). The generalized pinning model is
  obtained by relaxing the power law decay to  regularly varying decay, that is (in
  particular) we allow {\sl logarithmic correction} to power-law
  decay.  This, in a sense, allows zooming into the marginal case and
  makes clearer the interplay between the underlying renewal and the
  disorder variables.
\end{enumerate}

\subsection{The framework and some basic facts}
In mathematical terms, disordered pinning models are one-dimensional
Gibbs measures with random one-body potentials and reference measure
given by the law of a renewal process.  Namely, pinning models are
built starting from a (non-delayed, discrete) renewal process $\tau=
\{ \tau_n \}_{n=0,1, \ldots}$, that is a sequence of random variables
such that $\tau_0=0$ and $\{ \tau_{j+1}-\tau_j\}_{j=0,1, \ldots}$ are
independent and identically distributed with common law (called {\sl
  inter-arrival distribution}) concentrated on $\N:=\{1,2, \ldots\}$
(the law of $\tau$ is denoted by $\bP$): we will actually assume that
such a distribution is regularly varying of exponent $1+\ga$, {\sl
  i.e.}
\begin{equation}
\label{eq:K333}
K(n) \, :=\, \bP(\tau_1=n)=\frac{L(n)}{n^{1+\ga}}, \ \ \text{ for } \ n=1,2, \ldots,
\end{equation}
where $\ga \ge 0$ and $L(\cdot)$ is a slowly varying function, that is $L: (0 , \infty) \to (0, \infty)$
is measurable and it satisfies $\lim_{x \to \infty}L(c x)/L(x)=1$ for every $c>0$. 
There is actually no loss of generality in assuming $L(\cdot)$ smooth
and we will do so  (we refer to 
 \cite{cf:RegVar} for properties of slowly varying functions).

\smallskip
\begin{rem}
\label{rem:SVF}\rm
Examples of slowly varying functions include {\sl logarithmic slowly
  varying functions} (this is probably not a standard terminology, but
it will come handy), that is the positive measurable functions that
behave like $\consta (\log(x))^{\constb}$ as $x\to \infty$, with
$\consta>0$ and $\constb \in \R$. These functions are just a
particular class of slowly varying functions, but it is already rich
enough to appreciate the results we are going to present.  Moreover we
will say that $L(\cdot)$ is trivial if $\lim_{x \to \infty} L(x)=c \in
(0, \infty) $.  The general statements about slowly varying function
 that we are going to use can be verified in an elementary
way for logarithmic slowly varying functions; readers who feel
uneasy with the general theory may safely focus on this restricted
class.
\end{rem}
\smallskip

Without loss of generality we assume that $\sum_{n \in \N} K(n)=1$
(actually, we have implicitly done so when we have introduced $\tau$).
This does not look at all like an innocuous assumption at first,
because it means that $\tau$ is {\sl persistent}, namely $\tau_j<
\infty$ for every $j$, while if $\sum_n K(n) < 1$ then $\tau$ is {\sl
  terminating}, that is $\vert \{j : \, \tau_j < \infty\}\vert <
\infty$ a.s.. It is however really a harmless assumption, as explained
in detail in \cite[Ch.~1]{cf:Book} and recalled in the caption of
Figure~\ref{fig:RW}.

The disordered potentials  are introduced by means
of the IID sequence $\{ \go_n\}_{n=1,2, \ldots}$ of random variables
(the {\sl charges}) such that $\M( t) := \bbE[ \exp(t\go_1)]< \infty$
for every $t$. Without loss of generality we may and do assume 
that $\bbE[\go_1]=0$ and $\text{var}_\bbP(\go_1)=1$.

The model we are going to focus on 
is defined by the sequence of  probability measures
$\bP_{N,\go, \gb, h}= \bP_{N,\go}$, indexed by $N \in \N$, defined by
\begin{equation}
\label{eq:Gibbs}
\frac{\dd \bP_{N, \go}}{\dd \bP} (\tau) \, :=\, 
\frac1{Z_{N, \go}} \exp \left(\sum_{n=1}^N 
\left( \gb \go_n +h - \log \M(\gb) \right) \gd_n 
\right)  \gd_N \, ,
\end{equation}
where $\gb \ge 0$, $h \in \R$, $\gd_n $ is the indicator 
function that $n= \tau_j$ for some $j$ and 
$Z_{N , \go}$ is the partition function, that is the normalization constant.
It is practical to look at $\tau$ as a random subset of $\{0\} \cup \N$,
so that, for example, $\gd_n = \ind_{n \in \tau}$.

\smallskip

\begin{rem}
\rm
We have chosen $\M(t)< \infty$ for every $t$
only for ease of exposition. The results we present directly
generalize to the case in which $\M(t_0)+\M(-t_0)< \infty$ for
a $t_0>0$. In this case it suffices to look at the system
only for $\gb \in [0, t_0)$.  
\end{rem}
\smallskip

Three comments on \eqref{eq:Gibbs} are in order:
\smallskip

\begin{enumerate}
\item we have introduced the model in a very general set-up which is, possibly,
 not too  intuitive, but it  allows
a unified approach to a large class of models \cite{cf:Fisher,cf:Book}. It may
be useful at this stage to look at Figure~\ref{fig:RW} that
illustrates the random walk pinning model; 
\item the presence of $-\log \M (\gb)$ in the exponent is just 
a parametrization of the problem that comes particularly handy and it
can be absorbed by redefining $h$;
\item the presence of $\gd_N$ in the right-hand side means that we are looking
only at trajectories that are {\sl pinned} at the endpoint of the system. 
This is just a boundary condition and we may as well remove 
$\gd_N$ for the purpose of the results that we are going to state,
since it is well known for example that the free energy of this system
is independent of the boundary condition ({\sl e.g.} \cite[Ch.~4]{cf:Book}). Nonetheless, at a technical level
it is more practical to work with the system pinned at the endpoint.
\end{enumerate}

\smallskip

\begin{figure}[hlt]
\begin{center}
\leavevmode
\epsfxsize =14.5 cm
\psfragscanon
\psfrag{0}[c][l]{\small $0$}
\psfrag{Sn}[c][l]{\small $S_n$}
\psfrag{n}[c][l]{\small $n$}
\psfrag{t0}[c][l]{\small $\tau_0$}
\psfrag{t1}[c][l]{\small $\tau_1$}
\psfrag{t2}[c][l]{\small $\tau_2$}
\psfrag{t3}[c][l]{\small $\tau_3$}
\psfrag{t4}[c][l]{\small $\tau_4$}
\psfrag{o1}[c][l]{\small $\go_1$}
\psfrag{o2}[c][l]{\small $\go_2$}
\psfrag{o3}[c][l]{\small $\go_3$}
\psfrag{o4}[c][l]{\small $\go_4$}
\psfrag{o5}[c][l]{\small $\go_5$}
\psfrag{o6}[c][l]{\small $\go_6$}
\psfrag{o7}[c][l]{\small $\go_7$}
\psfrag{o8}[c][l]{\small $\go_8$}
\psfrag{o9}[c][l]{\small $\go_9$}
\psfrag{oa}[c][l]{\small $\go_{10}$}
\psfrag{ob}[c][l]{\small $\go_{11}$}
\psfrag{oc}[c][l]{\small $\go_{12}$}
\psfrag{od}[c][l]{\small $\go_{13}$}
\psfrag{oe}[c][l]{\small $\go_{14}$}
\psfrag{of}[c][l]{\small $\go_{15}$}
\epsfbox{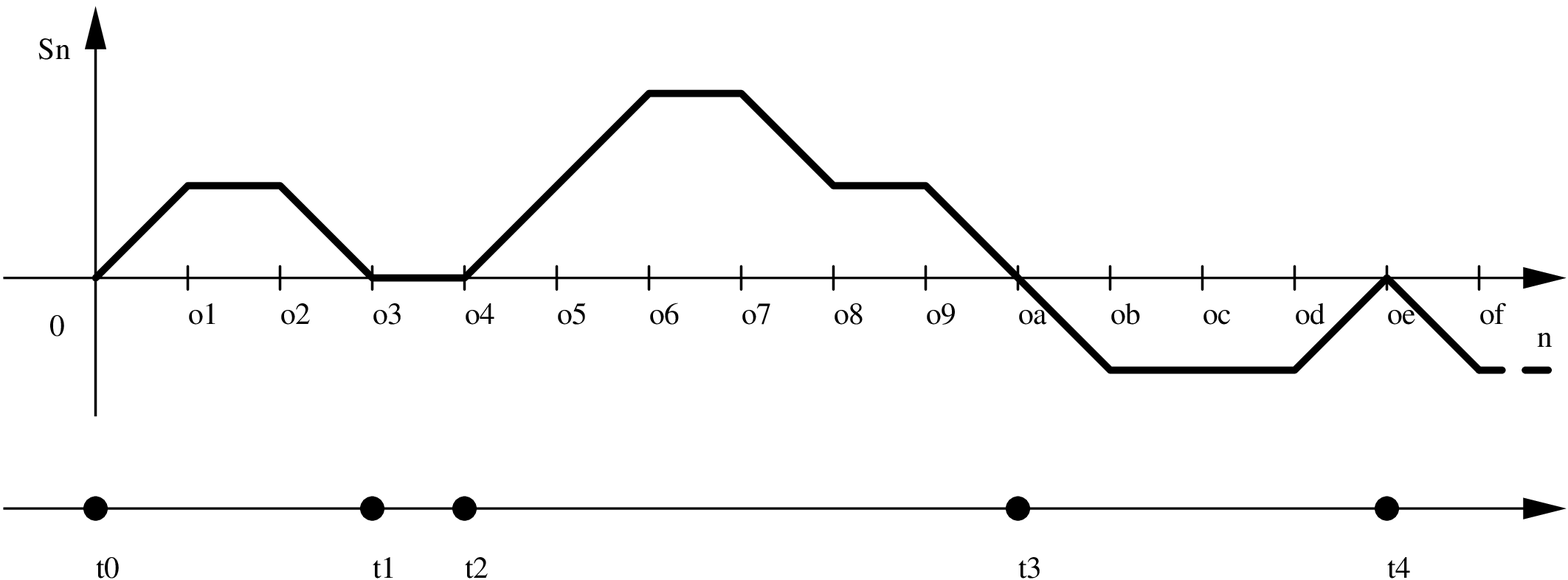}
\caption{\label{fig:RW} A symmetric random walk trajectory with
  increments taking values in $\{-1,0,+1\}$ is represented as a
  directed random walk. On the $x$-axis, the {\sl defect line}, there
  are quenched charges $\go$ that are collected by the walk when it
  hits the charge location.  The energy of a trajectory just depends
  on the underlying renewal process $\tau$. For the case in the
  figure, $K(n):=\bP(\tau_1=n)\sim const. n^{-3/2}$ for $n \to \infty$
  ({\sl e.g.} \cite[App.~A.6]{cf:Book}).  Moreover the walk is
  recurrent, so $\sum_n K(n)=1$.  There is however another
  interpretation of the model: the charges may be thought of as sticking
  to $S$, not viewed this time as a directed walk. If the walk hits
  the origin at time $n$, the energy is incremented by $(\beta\go_n+h-
\log \M(\gb))$.  This
  interpretation is particularly interesting for a three-dimensional
  symmetric walk in $\Z^3$: the walk may be interpreted as a polymer
  in $d=3$, carrying charges on each monomer, and the monomers
  interact with a point in space (the origin) via a charge-dependent
  potential.  Also in this case $K(n)\sim const. n^{-3/2}$, but the
  walk is transient so that $\sum_n K(n) <1$ ({\sl e.g.} 
  \cite[App.~A.6]{cf:Book}).  It is rather easy to see that any model based on a
  terminating renewal with inter-arrival distribution $K(\cdot)$ can
  be mapped to a model based on the persistent renewal with
  inter-arrival distribution $K(\cdot)/\sum_n K(n)$ at the expense of
  changing $h$ to $h+\log \sum_n K(n)$.  For much more detailed
  accounts on the (very many!) models that can be directly mapped to
  pinning models we refer to \cite{cf:Fisher,cf:Book}.  }
\end{center}
\end{figure}

The (Laplace) asymptotic behavior of $Z_{N , \go}$ shows a phase transition.
In fact, if we define the free energy as
\begin{equation}
\label{eq:fe}
\tf (\gb, h) \, :=\, \lim_{N \to \infty} \frac 1N \bbE \log Z_{N, \go},
\end{equation}
where the limit exists since the sequence $\{ \bbE \log Z_{N, \go} \}_N$
is super-additive (see {\sl e.g.} \cite[Ch.~4]{cf:Book}, where it is also
proven that $\tf (\gb, h)$ coincides with the $\bbP(\dd \go)$-almost sure
limit of $(1/N)\log Z_{N, \go}$, so that $\tf(\gb, h)$ is effectively
the {\sl quenched} free energy), then it is easy to see that 
$\tf(\gb, h)\ge 0$: in fact,
\begin{multline}
\tf(\gb, h)\, \ge \, \limsup_{N \to \infty}
\frac 1N \bbE \log 
\bE\left[
\exp \left(\sum_{n=1}^N 
\left( \gb \go_n +h - \log \M(\gb) \right) \gd_n 
\right)   \ind_{\tau_1=N}\right] 
\\ =\, \lim_{N \to \infty} \frac 1N \left( (h -\log \M(\gb))+ \log \bP(\tau_1=N)
\right)\, =\, 0.
\end{multline}
The transition we are after is captured by
setting
\begin{equation}
h_c(\gb) \, :=\, \sup \{ h:\, \tf(\gb,h)=0\} \, =\, \inf \{ h:\, \tf(\gb,h)>0\} ,
\end{equation}
where the equality is a direct consequence of the fact that
$\tf(\gb, \cdot)$ is non-decreasing (let us point out also that 
the free energy is  a continuous function of both arguments, as it
follows from standard convexity arguments). 
We have the bounds (see point (2) just below for the proof)
\begin{equation}
\label{eq:febounds}
\tf(0,h-\log \M (\gb)) \, \le \, \tf(\gb, h) \, \le \, \tf(0, h)\, , 
\end{equation}
which directly imply
\begin{equation}
\label{eq:hcbounds}
h_c(0) \, \le \, h_c (\gb) \, \le \, h_c(0)+\log \M(\gb)\, .
\end{equation}
Two important observations are:
\begin{enumerate}
\item the bounds in \eqref{eq:febounds} are given in terms of
$\tf(0, \cdot)$, that is the free energy of the non-disordered system, which 
can be solved analytically ({\sl e.g.} \cite{cf:Fisher,cf:Book}).
In particular $h_c(0)=0$
for every $\ga$ and every choice of $L(\cdot)$ (in fact
$h_c(0)= -\log \sum_n K(n)$ and we are assuming that $\tau$ is 
persistent).
We will keep in our formulae $h_c(0)$ both because
we think that it makes them more readable and because
they happen to be true also if $\tau$ were a terminating renewal).
\item The upper bound in \eqref{eq:febounds}, that entails the lower bound
in \eqref{eq:hcbounds},
follows directly from the standard {\sl annealed bound}, that is 
$\bbE \log Z_{N ,\go}\le  \log \bbE Z_{N ,\go}$, and by observing that
 the {\sl annealed partition function} $\bbE Z_{N ,\go}$ coincides with
 the partition function of the quenched model with $\gb=0$, that is simply the
 non-disordered case
 (of course, the presence of the term
 $-\log \M (\gb)$ in \eqref{eq:Gibbs} finds here its motivation).
 The lower bound in \eqref{eq:febounds}, entailing the upper bound
in \eqref{eq:hcbounds}, follows by a convexity argument too
(see \cite[Ch.~5]{cf:Book}).
\end{enumerate}
\smallskip

\begin{rem}\rm
It is rather easy (just take the derivative of the free energy
with respect to $h$) to realize that the phase transition we have
outlined in this model is a localization transition: when $h< h_c(\gb)$,
for $N$ large, the random set $\tau$ is {\sl almost empty},
while when $h> h_c(\gb)$ it is of size $const. N$ (in fact $const.= \partial_h\tf (\gb ,h)$).
Very sharp  results have been obtained on this issue: we refer to 
\cite[Ch.s~7 and 8]{cf:Book} and references therein. 
\end{rem}

\subsection{The Harris criterion}
\label{sec:Harris}
We can now make precise the Harris criterion predictions
mentioned in \S~\ref{sec:RIM}. As we have seen, in our case
 the pure (or {\sl annealed}) model is just
the non-disordered model, and the latter is exactly solvable, so that
the critical behavior is fully understood, notably \cite[Ch.~2]{cf:Book}
\begin{equation}
\label{eq:feexp}
\lim_{a \searrow 0}\frac{\log \tf(0, h_c(0)+a)}{\log a}\,=\,
\max \left( 1, \frac 1{\ga} \right)\, =:\, \nu_{\text{pure}}. 
\end{equation}  
The specific heat exponent of the pure model (that is the critical
exponent associated to 
 $1/\partial_h^2 \tf (0, h)$) is
computed analogously and it is equal to $2-\nu_{\text{pure}}$.
Therefore the Harris criterion predicts {\sl disorder relevance} for
$\ga>1/2$ ($2-\nu_{\text{pure}}>0$) and {\sl disorder irrelevance} for
$\ga<1/2$ ($2-\nu_{\text{pure}}<0$) at least for $\gb$ below a
threshold, with $\ga=1/2$ as marginal case.  So, what one expects is
that $\nu_{\text{pure}}=\nu_{\text{quenched}}$ (with obvious
definition of the latter) if $\ga<1/2$ for $\gb$ {\sl not too large}
and $\nu_{\text{pure}}\neq\nu_{\text{quenched}}$ if $\ga>1/2$ (for
every $\gb>0$).

While a priori the Harris criterion attacks the issue of critical behavior,
it turns out that  a Harris-like approach
in the pinning context  \cite{cf:FLNO,cf:DHV} yields information also 
on $h_c(\gb)$, namely that $h_c(\gb)=h_c(0)$
if $\ga<1/2$ and $\gb$ again not too large, while 
$h_c(\gb)>h_c(0)$ as soon as $\gb>0$. For the sequel it is 
 important to recall some aspects of the approaches in  \cite{cf:FLNO,cf:DHV}.

\smallskip

The main focus of \cite{cf:FLNO,cf:DHV},
 is on the case $\ga=1/2$ and trivial $L(\cdot)$.
In fact they focus on
 the interface wetting problem in two dimensions, that boils down to
 directed random walk pinning in $(1+1)$-dimensions. In this framework
 the conclusions of the two papers differ:  \cite{cf:FLNO}
 stands for $h_c(\gb)=h_c(0)$ for $\gb $ small, while  in \cite{cf:DHV}
one finds an argument in favor of 
\begin{equation}
\label{eq:DHV2}
h_c(\gb)-h_c(0) \approx \exp( -c \gb^{-2}),
\end{equation}
as $\gb \searrow 0$ (with $c>0$ an explicit constant).

We will not go into the details 
of these arguments, but we wish to point out why, in these arguments,  $\ga=1/2$
plays such a singular role. 
\smallskip

\begin{enumerate}
\item In the approach of \cite{cf:FLNO} an expansion of the free
  energy to all orders in the variance of $\exp(\gb\go_1 -\log
  \M(\gb))$, that is $(\M(2\gb)/\M^2(\gb))-1 \stackrel{\gb \searrow
    0}\sim \gb^2$, is performed. In particular (in the Gaussian case)
\begin{equation}
\label{eq:FLNO}
\tf(\gb, h_c(0) + a)\, =\ \tf(0, h_c(0) + a) - \frac 12 \left( \exp(\gb^2)-1
\right) \left( \partial_a \tf(0, h_c(0) + a) \right)^2 + \ldots
\end{equation}
and, when $L(\cdot)$ is trivial, 
$\partial_a \tf(0, h_c(0) + a) $ behaves like (a constant times)
$ a^{(1-\ga)/\ga}$ for $\ga \in (0,1)$ (this is detailed
for example in \cite{cf:GBreview}) 
and like a constant 
for $\ga\ge1$.  This suggests that the
expansion \eqref{eq:FLNO} cannot work for $\ga>1/2$, because the
second-order term, for $a \searrow 0$, becomes larger than the first
order term ($a^{\max(1/\ga, 1)}$).  The borderline case is $\ga=1/2$,
and trust in such an expansion for $\ga =1/2$ may follow from the fact
that $\gb$ can be chosen small. 
In conclusion, an argument along the lines of \cite{cf:FLNO} predicts disorder
relevance if and only if $\ga>1/2$ (if $L(\cdot)$ is trivial).
\item The approach of \cite{cf:DHV} instead is based on the analysis
  of $\text{var}_{\bbP} (Z_{N, \go})$ at the pure critical point
  $h_c(0)$.  This directly leads to studying the random set $\tilde
  \tau:= \tau \cap \tau^\prime$ (it appears in the computation in a
  very natural way, we call it {\sl intersection renewal}), with
  $\tau^\prime$ an independent copy of $\tau$ (note that $ \tilde \tau
  $ is still a renewal process): in physical terms, one is looking at
  the {\sl two-replica system}.  It turns out that, even if we have
  assumed $\tau$ persistent, $ \tilde \tau$ may not be: in fact, if
  $L(\cdot)$ is trivial, then $\tilde \tau$ is persistent if and only
  if $\ga \ge 1/2$ (see just below for a proof of this fact). And
  \cite{cf:DHV} predicts disorder relevance if and only if $\ga\ge
  1/2$.
\end{enumerate}
\smallskip
Some aspects of these two approaches were made rigorous mathematically: The expansion of the free energy \eqref{eq:FLNO} was proved to hold for $\alpha<1/2$ in \cite{cf:GT_irrel}, and the second moment analysis of \cite{cf:DHV} was used to prove {\sl disorder irrelevance} in \cite{cf:Ken,cf:T_cmp}, making it difficult to choose between the predictions.
\smallskip

We can actually find in the physical literature a number of authors standing for one or the other
of the two predictions in the marginal case $\ga=1/2$ (the reader can
find a detailed review of the literature in  \cite{cf:GLT_marg}).
 But we would like to go a step 
farther and we point out that,
by generalizing naively the  approach in  \cite{cf:DHV}, one is  tempted 
to conjecture disorder relevance (at arbitrarily small $\gb$)
if and only if the intersection renewal is recurrent. Let us 
 make this condition explicit: while one does not have direct access
to the inter-arrival distribution of $\tilde \tau$, it is straightforward, by independence, 
to write the renewal function of $\tilde \tau$:
\begin{equation}
\bP(n \in \tilde \tau) \, =\, \bP( n \in \tau)^2.
\end{equation}
It is then sufficient to use the basic (and general) renewal process
formula $\sum_n \bP( n \in \tilde \tau) = (1- \sum_n \bP( \tilde
\tau_1=n))^{-1}$ to realize that $\tilde \tau$ is persistent if and
only if $\sum_n \bP( n \in \tilde \tau)=\infty$.  Since under our
assumptions for 
$\ga\in(0,1)$ \cite[Th.~B]{cf:Doney}
\begin{equation}
\label{eq:Doney}
\bP(n \in \tau) \stackrel{n \to \infty} \sim \frac{\ga \sin (\pi \ga)}{\pi} \frac1{n^{1-\ga}L(n)},
\end{equation} 
we easily see that the intersection renewal $\tilde \tau$ is persistent
for $\ga >1/2$ and terminating if $\ga<1/2$ (the case $\ga=0$ can be treated too \cite{cf:RegVar},
and $\tilde \tau$ is terminating).
In  the $\ga=1/2$ case the argument we have
just outlined yields
\begin{equation}
\label{eq:persist}
\tau\cap \tau^\prime \ \text{ is persistent } \ \Longleftrightarrow  \ 
\sum_n \frac 1{n \, L(n)^2} \, = \, \infty.
\end{equation}
Roughly, this is telling us that the intersection renewal $\tilde
\tau$ is persistent up to a slowly varying function $L(x)$ diverging
{\sl slightly less} than $(\log x)^{1/2}$.  In particular, as we have
already pointed out, if $L(\cdot)$ is trivial, $\tilde \tau$ is
persistent.

Let us remark that
the expansion
\eqref{eq:FLNO} has been actually made rigorous in \cite{cf:GT_irrel}, but 
only under the assumption that the intersection renewal $\tilde \tau$
is terminating (that is,  $\constb>1/2$ for logarithmic slowly varying functions).

\smallskip

\begin{rem}
\label{rem:Ltilde}
\rm
In view of the argument we have just outlined, we introduce 
the increasing function  $\tilde L : (0, \infty) \to (0, \infty)$ defined as
\begin{equation}
\label{eq:Ltilde}
\tilde L (x) \, :=\, \int_0^x \frac1{(1+y)L(y)^2}\dd y,
\end{equation}
that is going to play a central role from now on.
Let us point out that, by \cite[Th.~1.5.9a]{cf:RegVar}, $\tilde L(\cdot)$ is a slowly varying function 
which has the property
\begin{equation}
\label{eq:Ltilde-prop}
\lim_{x \to \infty}\tilde L(x) L(x)^2\, =\, +\infty,
\end{equation}
which is a non-trivial statement when $L(\cdot)$ does not diverge at infinity. 
Of course we are most interested in the fact that, 
when $\ga=1/2$,  $\tilde L(x)$ diverges as $x \to \infty$
 if and only if the intersection renewal $\tilde \tau$ is recurrent ({\sl cf.}
 \eqref{eq:persist}).
 For completeness we point out  that 
$\tilde L(\cdot)$ is  a special type of slowly varying function
(a {\sl den Haan function} \cite[Ch.~3]{cf:RegVar}), but we will not exploit
the further regularity properties stemming out of this observation. 
\end{rem}

\smallskip


\subsection{Review of the rigorous  results}
\label{sec:review-rs}
Much mathematical work has been done on disordered pinning models
recently. Let us start with a quick review of the $\ga \neq 1/2$ case:
\begin{itemize}
\item If $\ga>1/2$ disorder relevance is established. The positivity of 
$h_c(\gb)-h_c(0)$ (with precise asymptotic estimates as $\gb \searrow 0$)
 is proven \cite{cf:DGLT,cf:AZ}. It has been also shown that disorder has a smoothing
 effect on the transition and the quenched free energy critical exponent differs
 from the annealed one
  \cite{cf:GT_cmp}.
\item If $\ga<1/2$ disorder irrelevance is established, along with 
a number of sharp results saying in particular that, if $\gb$ is not too large, 
$h_c(\gb)=h_c(0)$
and that the free energy critical behavior coincides in the quenched and annealed
framework \cite{cf:Ken,cf:T_cmp,cf:GT_irrel,cf:AZ_new}.
\end{itemize}
\medskip

In the case $\ga=1/2$ results are less complete. 
Particularly relevant for the sequel are the next two results that we state as theorems.
The first one is taken from \cite{cf:Ken} (see also \cite{cf:GT_cmp})
and uses the auxiliary function $a_0(\cdot)$ defined by
\begin{equation}
a_0(\gb)\, :=\, C_1 L\left( \tilde L^{-1}\left(C_2/\gb^2\right) \right) 
\Big / \left( \tilde L^{-1}\left(C_2/\gb^2\right) \right) ^{1/2}\,  \ \ \text{ with }C_1>0 \text{ and }
C_2>0,
\end{equation}
if $\lim_{x \to \infty}\tilde L(x)=\infty$, and 
$a_0(\cdot)\equiv 0$ otherwise.

\medskip

\begin{theorem}
\label{th:1/2UB}
Fix $\go_1 \sim \cN(0,1)$, $\ga=1/2$ and choose a slowly varying
function $L(\cdot)$.  Then there exists $\gb_0>0$ and $a_1>0$ such
that for every $\epsilon>0$ there exist $C_1$ and $C_2>0$ such that
\begin{equation}
1-\epsilon \, \le \, 
\frac{\tf(\gb, a) }{\tf(0, a)}\, \le \, 1 \ \ \text { for } \ a> a_0(\gb), \, a \le a_1
\text { and } \gb\le \gb_0. 
\end{equation}
This implies for $\gb \le \gb_0$
\begin{equation}
\label{eq:hcUB}
h_c(\gb)- h_c(0) \, \le \, a_0(\gb).
\end{equation}
\end{theorem}
\medskip

It is worth pointing out that Theorem~\ref{th:1/2UB} yields
an upper bound matching  \eqref{eq:DHV2} when $L(\cdot)$
is trivial. 

The next result addresses instead the lower bound on 
$h_c(\gb)-h_c(0)$ and it is taken from \cite{cf:GLT_marg}:

\medskip

\begin{theorem}
\label{th:1/2LB}
Fix $\go_1 \sim \cN(0,1)$ and  $\ga=1/2$.
If $L(\cdot) $ is trivial, then 
$h_c(\gb)-h_c(0)>0$ for every $\gb>0$ and 
there exists $C>0$ such that
\begin{equation}
\label{eq:beta4}
h_c(\gb)-h_c(0)\, \ge \, \exp\left(-C/ \gb^4 \right),
\end{equation}
for $\gb \le 1$.
\end{theorem}
\medskip

It should be pointed out that \cite{cf:GLT_marg} has been worked out
for trivial $L(\cdot)$, addressing thus precisely the controversial
issue in the physical literature. 
The case of $\lim_{x \to
  \infty}L(x)=0$ has been treated \cite{cf:AZ} (see \cite{cf:DGLT} for
a weaker result) where $h_c(\gb)-h_c(0)>0$ has been established with
an explicit but not optimal bound.
We point out also that a result analogous to Theorem~\ref{th:1/2LB}
has been proven  for a hierarchical version
of the pinning model (see
\cite{cf:GLT_marg} for the case of the
  hierarchical model proposed in \cite{cf:DHV}). 
\medskip

The understanding of the marginal case is therefore 
still partial and the following problems are clearly open:
\smallskip

\begin{enumerate}
\item What is really the behavior of $h_c(\gb)-h_c(0)$ in the marginal case? In particular,
for $L(\cdot)$ trivial, is \eqref{eq:DHV2} correct?
\item Going beyond the case of $L(\cdot)$ trivial: is the 
 two-replica condition \eqref{eq:persist}
 equivalent to disorder relevance for small $\gb$?
\item What about non-Gaussian disorder? It should be pointed out that
  a part of the literature focuses on Gaussian disorder, notably
  Theorem~\ref{th:1/2UB}, but this choice appears to have been made in
  order to have more concise proofs (for example, the results in
  \cite{cf:DGLT} are given for very general disorder distribution).
  Theorem~\ref{th:1/2LB} instead exploits a technique that is more
  inherently Gaussian and generalizing the approach in
  \cite{cf:GLT_marg} to non-Gaussian disorder is not straightforward.
\end{enumerate}

\smallskip

As we explain in the next subsection, in this paper 
we will give {\sl almost} complete answers to 
questions (1), (2) and (3). In addition we will prove a monotonicity result for the phase diagram of pinning model which holds in great generality.

\subsection{The main result}
Our main result requires the existence of $\epsilon\in (0, 1/2]$ such that
\begin{equation}
\label{eq:Lassumption}
L(x) \, =\, o\left( (\log (x))^{(1/2)-\epsilon}\right) \ \ \text{ as } \ x \to \infty,
\end{equation}
that is $ \lim_{x \to \infty} L(x) (\log (x))^{-(1/2)+\epsilon}=0$. Of
course, if $L(\cdot)$ vanishes at infinity, \eqref{eq:Lassumption}
holds with $\epsilon=1/2$.  Going back to the slowly varying function
$\tilde L(\cdot)$, {\sl cf.} Remark~\ref{rem:Ltilde}, we note that,
under assumption \eqref{eq:Lassumption}, we have
\begin{equation}
\label{eq:Lmore}
\tilde L(x) \stackrel{x \to \infty} {\gg} \int_2^x \frac 1{y ( \log y)^{1-2\epsilon}} \dd y
\, =\, \frac 1{2 \epsilon } (\log x)^{2 \epsilon} - \frac 1{2 \epsilon } (\log 2)^{2 \epsilon}.
\end{equation} 
Therefore, under assumption \eqref{eq:Lassumption},
we have that if $q>(2\epsilon)^{-1}$ then
\begin{equation}
\lim_{x \to \infty} \frac{\tilde L(x)}{L(x)^{2/(q-1)}}\, =\, \infty,
\end{equation}
which guarantees that given $q>(2\epsilon)^{-1}$ (actually, in the sequel
$q \in \N$) and $A>0$,
\begin{equation}
\label{eq:Delta23}
\gD (\gb;q,A)\, :=\, \left( \inf\left\{ n\in \N :\, {\tilde L(n)}/{L(n)^{2/(q-1)}}\ge A \gb^{-2q/(q-1)}
  \right\} \right)^{-1}\,  
\end{equation}
is greater than $0$ for every $\gb>0$.
\smallskip

Our main result is

\medskip

\begin{theorem}
\label{th:main333}
Let us assume that $\ga=1/2$ and that \eqref{eq:Lassumption}
holds for some $\epsilon\in (0, 1/2]$. 
For every $\gb_0$ and every integer $q>(2\epsilon)^{-1}$
there exists $A>0$ such that
\begin{equation}
h_c(\gb) -h_c(0) \, \ge \, \gD (\gb;q,A)\, >\, 0, 
\end{equation}
for every $\gb\le \gb_0$.
\end{theorem}

\medskip

The result may be more directly appreciated in the particular case of
$L(\cdot)$ of logarithmic type, {\sl cf.} Remark~\ref{rem:SVF}, with
$\constb<1/2$, so that \eqref{eq:Lassumption} holds with
$\epsilon<\min((1/2)-\constb,1/2)$.  By explicit integration we see
that $\tilde L(x) \sim (\consta^2(1-2\constb))^{-1} (\log
(x))^{1-2\constb}$ so that
\begin{equation}
 \frac{\tilde L(x)}{L(x)^{2/(q-1)}} \, \sim \, \frac {\consta^{-2q/(q-1)}}{(1-2\constb)}
 (\log (x))^{1-2\constb q(q-1)^{-1}}
\end{equation}
and in this case
\begin{equation}
\gD (\gb; q, A) \stackrel{\gb\searrow 0}\sim
\exp\left(-c(\constb,A,q) \gb^{-b}\right),
\end{equation}
where 
$c(\constb,A,q):=((1-2\constb)\consta^{2q/(q-1)}A)^{1/C}$
and $b:= 2q/((q-1)C)$ with 
$C:=1-2\constb q(q-1)^{-1}$.
In short, by choosing $q$ large the exponent
$b>2/(1-2\constb)$ becomes arbitrarily close to $2/(1-2\constb)$,
at the expense of course of a large constant
$c(\constb,A,q)$, 
since $A$ will have to be chosen sufficiently large.

We sum up these steps into the following simplified version
of Theorem~\ref{th:main333}
\medskip

\begin{cor}
If $\ga=1/2$ and $L(\cdot) $ is of logarithmic type with $\constb\in (-\infty, 1/2)$
({\sl cf.} Remark~\ref{rem:SVF}) 
then $h_c(\gb)>h_c(0)$ for every $\gb>0$ and for every
$b>2/(1-2\constb)$
there exists
$c>0$ such that, for $\gb$ sufficiently small
\begin{equation}
h_c(\gb)-h_c(0)\, \ge \, \exp\left(-c \gb^{-b}\right).
\end{equation}
\end{cor}

\medskip

This result of course has to be compared with the 
 upper bound in Theorem~\ref{th:1/2UB} that for
 $L(\cdot)$ of logarithmic type yields for $\constb< 1/2$
\begin{equation}
h_c(\gb)-h_c(0)\, \le \, 
\tilde C_1 \gb^{-2 \constb/(1-2\constb)}\exp\left(-\tilde C_2 \gb^{-2 /(1-2\constb)}\right),
\end{equation}
where $\tilde C_1$ and $\tilde C_2$ are positive constants that depend
(explicitly) on $\consta$, $\constb$ and on the two constants $C_1$
and $C_2$ of Theorem~\ref{th:1/2UB} (we stress that $\tilde C_1>0$ and
$\tilde C_2>0$ for every $\consta>0$ and $\constb< 1/2$).  \smallskip

\smallskip

The main body of the proof of Theorem~\ref{th:main333} is given in the
next section.  In the subsequent sections a number of technical
results are proven.  In the last section
(Section~\ref{sec:monotonicity}) we prove a general result (Proposition~\ref{th:monotonicity})
for the
models we are considering: the monotonicity of the free energy with
respect to $\gb$. This result, proven for other disordered models,
appears not to have been pointed out up to now for the pinning model.
We stress that Proposition~\ref{th:monotonicity}  
 is  not used  in the rest of the
paper, but, as  discussed in Section~\ref{sec:monotonicity},
one can find a link of some interest with our main results.  

\section[Coarse graining]{Coarse graining, fractional moment and measure change arguments}
\label{sec:CGfmmc}

The purpose of this section is to reduce  the proof to 
a number of technical statements, that are going to be proven
in the next sections. In doing so, we are going to introduce the quantities and 
notations used in the technical statements and, at the same time,
we will stress the main ideas and the novelties with respect to earlier
approaches (notably, with respect to \cite{cf:GLT_marg}).

We anticipate that the main ingredients of the proof are (like in
\cite{cf:GLT_marg}) a coarse graining procedure and a fractional moment
estimate on the partition function combined with a change of measure.
However:
\begin{enumerate}
\item In \cite{cf:GLT_marg} we have exploited the Gaussian character of
  the disorder to introduce {\sl weak, long-range correlations} while
  keeping the Gaussian character of the random variables. In fact, the
  change of measure is given by a density that is just the exponential
  of a quadratic functional of $\go$, that is a measure change via a
  {\sl 2-body potential}.  In order to lower the exponent $4$ in the
  right-hand side of \eqref{eq:beta4} we will use $q$-body potentials
  $q=3,4,  \ldots$ (this is the $q$ appearing in
  Theorem~\ref{th:main333}). Such potentials carry with themselves a
  number of difficulties: for example, when the law of the disorder is Gaussian, the
  modified measure is not. As a matter of fact, there are even
  problems in defining the modified disorder variables if one modifies
  in a straightforward way the procedure in \cite{cf:GLT_marg} to use
  $q$-body potentials, due to integrability issues: such problems may
  look absent if one deals with bounded $\go$ variables, but they
  actually reappear when taking limits. The change-of-measure
  procedure is therefore performed by introducing $q$-body potentials
  {\sl and} suitable cut-offs.  Estimating the effect of
  such {\sl $q$-body potential with cut-off} change of measure is at
  the heart of our technical estimates.
\item The coarse-graining procedure is different from the one used in
  \cite{cf:T_cg,cf:GLT_marg}, since we have to adapt it to the new
  change of measure procedure. However, unlike point (1), 
  the difference between the previous coarse graining procedure and the
  one we are employing now is more technical than conceptual. 
\end{enumerate}

\subsection{The coarse graining length}
Recall the definition \eqref{eq:Ltilde} of $\tilde L(\cdot)$.
We are assuming  \eqref{eq:Lassumption}, therefore $\lim_{x \to \infty}\tilde
L(x)=+\infty$.
Chosen a value of  $q \in\{ 2,3, \ldots\}$ ($q$ is kept fixed throughout the proof) and a positive constant $A$ (that is going to be chosen large) we define
\begin{equation}
\label{eq:k333}
k\, =\, k(\gb; q,A)\, :=\, 
\inf \left\{ n \in \N :\, \tilde L(n)/ L(n)^{2/(q-1)} \ge A \gb^{-2q/(q-1)}\right\}.
\end{equation}
Since we are interested also in cases in which  $L(\cdot)$ diverges (and possibly faster than
$\tilde L(\cdot)$) it is in general false  that $k< \infty$. However, the 
assumption \eqref{eq:Lassumption} guarantees that, for $q>(2\epsilon)^{-1}$,
 $L(x)/ L(x)^{2/(q-1)}\to \infty$ for $x\to \infty$ and therefore $k< \infty$.
 
Moreover, if $L(\cdot)$ is of logarithmic type (Remark~\ref{rem:SVF}) with
$\constb<1/2$,
then   for $q>1/(1-2\constb)$ the function $\tilde L(\cdot)/ L(\cdot)^{2/(q-1)}$
is (eventually) increasing.

\smallskip

Of course $k(\gb; q,A)$ is just $1/\gD (\gb;q,A)$,
{\sl cf.}
\eqref{eq:Delta23}, and the reason for such a link 
is explained in Remark~\ref{rem:kgD}.
Note by now  that $k$ is monotonic in both $\gb$ and $A$.
Since $\gb$ is chosen smaller than an arbitrary fixed quantity
$\gb_0$, in order to guarantee that $k$ is large we will rather play on 
choosing $A$ large.

\smallskip

\begin{rem}
\label{rem:L}
\rm
For the proof  certain monotonicity properties will be important. Notably, 
we know \cite[\S~1.5.2]{cf:RegVar} that $1/(\sqrt{x} L(x))$ is asymptotic 
to a monotonic (decreasing) function and this directly implies that we can
find a slowly varying function $\bL(\cdot)$ and a constant $\const _L \in (0,1]$ such that
\begin{equation}
\label{eq:Lb}
x \mapsto \frac1{\sqrt{x}\bL(x)} \text{ is decreasing  and } 
 \const _L  \bL(x) \, \le \, L(x) \, \le \, \bL(x) \text{ for every } x\in (0, \infty).
\end{equation}
Given the asymptotic behavior of the renewal function of $\tau$ (a special case
of \eqref{eq:Doney})
\begin{equation}
\label{eq:Doney1/2}
\bP\left(n \in \tau\right)
\stackrel{n \to \infty}\sim \frac 1{2\pi \sqrt{n} L(n)},
\end{equation}
and the fact that $\bP\left(n \in \tau\right)>0$ for every $n\in \N$, 
 we can choose
$\bL (\cdot)$ and 
$\const _L $ such that we have also
\begin{equation}
\label{eq:Doney-bound}
\frac{1}{\sqrt{n+1}\, \bL(n+1)}\, \le \,
\bP(n \in \tau) \, \le \, \frac{\const _L ^{-1}}{\sqrt{n+1}\, \bL(n+1)}, \ \ \ n=0,1,2, \ldots .
\end{equation}
It is natural to choose $\bL(\cdot)$ such that $\lim_{x \to \infty}\bL(x)/L(x)
\in [1, 1/\const_L)$ exists, and we will do so.
For later convenience
we set
\begin{equation}
\label{eq:R12}
R_{\frac12} (x) \, :=\,  \frac1{\sqrt{x+1}\, \bL(x+1)}.
\end{equation}
\end{rem}

\subsection{The coarse graining procedure and the fractional moment bound}
Let us start by introducing for $0\le M < N$ the notation(s)
\begin{equation}
  \label{eq:Znh23}
  Z_{M, N}\, =\,
  Z_{M, N,\go}:=\bE\left[e^{\sum_{n=M+1}^N(\beta\go_n+h-\log \M(\gb))\delta_n}\delta_N
  \big \vert \gd_M=1\right]\, ,
\end{equation}
and $Z_{M,M}:=1$ (of course $Z_{N, \go}= Z_{0, N}$).
We consider without loss of generality 
a system of size proportional to $k$, that is $N=km$
with $m \in \N$. For $\mathcal I\subset
\left\{1,\dots,m\right\}$ we define
\begin{equation}
  \hat Z_{\go}^{\mathcal I}:=\bE \left[e^{\sum_{n=1}^N(\beta\go_n+h-\log \M(\gb))\delta_n} \gd_N \ind_{E_\cI}(\tau)\right],
\end{equation}
where
 $E_\cI:=\{ \tau \cap (\cup_{i \in \cI} B_i) = \tau \setminus \{0\}\}$, 
and
\begin{equation}
 B_i:=\left\{(i-1)k+1,\dots,ik\right\},
\end{equation}
that is $E_\cI$ is the event that the renewal $\tau$ intersects the
blocks $(B_i)_{i\in \mathcal I}$ and only these blocks over $\{1,
\ldots, N\}$.  It follows from this definition that
\begin{equation}\label{eq:decomp333}
 Z_{N,\go}=\sum_{\mathcal I\subset \left\{1,\dots,m\right\}} \hat Z_{\go}^{\mathcal I}.
\end{equation}
Note that $\hat Z_{\go}^{\mathcal I}=0$ if $m\notin \mathcal I$. Therefore in the following we will always assume $m\in \mathcal I$.
For $\mathcal I=\{i_1,\dots, i_l\}$, ($i_1<\dots<i_l$, $i_l=m$), one can express $\hat Z_{\go}^{\mathcal I}$ in the following way:
\begin{multline}
\label{eq:decomp1}
\hat Z_{\go}^{\mathcal I}=
\sumtwo{d_1, f_1 \in B_{i_1}}{d_1\le f_1}\sumtwo{d_2, f_2\in B_{i_2}}{d_2\le f_2}\ldots \sum_{d_l\in B_{i_l}}\\
K(d_1)z_{d_1}Z_{d_1,f_1}K(d_2-f_1)Z_{d_2,f_2}\ldots K(d_l-f_{l-1})z_{d_l}
Z_{d_l,N},
 \end{multline}
 with $z_n := \exp( \gb \go_n +h - \log \M(\gb))$.  Let us fix a value of $\gamma\in (0,1)$
  (we  actually choose $\gamma=6/7$, but we will keep
  writing it as $\gamma$). Using the inequality $\left(\sum a_i\right)^{\gamma}\le
 \sum a_i^{\gamma}$ (which is valid for $a_i\ge 0$ and an arbitrary
 collection of indexes) we get
\begin{equation}
\label{eq:cg+fm}
  \bbE\left[ Z_{N,\go}^{\gamma}\right]\, \le\, \sum_{\mathcal I\subset \left\{1,\dots,m\right\}} \bbE \left[\left(\hat Z_{\go}^{\mathcal I}\right)^{\gamma}\right].
\end{equation}
An elementary, but crucial, observation is that
\begin{equation}
\label{eq:fmm}
\tf(\gb, h)\, =\, \lim_{N \to \infty} \frac1{\gamma N} \bbE
\log Z_{N, \go }^\gamma\, \le \, 
 \liminf_{N \to \infty} \frac1{\gamma N} 
\log \bbE Z_{N, \go }^\gamma,
\end{equation}
so that if we can prove that $\limsup_N\bbE Z_{N, \go }^\gamma < \infty$
for $h=h_c(0)+\Delta(\beta;q,A)$
we are done.

\begin{figure}[hlt]
\begin{center}
\leavevmode
\epsfxsize =14.5 cm
\psfragscanon
\psfrag{O}[c][l]{$0$}
\psfrag{1K}[c][l]{$k$}
\psfrag{2K}[c][l]{$2k$}
\psfrag{3K}[c][l]{$3k$}
\psfrag{4K}[c][l]{$4k$}
\psfrag{5K}[c][l]{$5k$}
\psfrag{6K}[c][l]{$6k$}
\psfrag{7K}[c][l]{$7k$}
\psfrag{8K}[c][l]{$8k=N$}
\psfrag{d1}[c][l]{\Small{$d_1$}}
\psfrag{d2}[c][l]{\Small{$d_2$}}
\psfrag{d3}[c][l]{\Small{$d_3$}}
\psfrag{d4}[c][l]{\Small{$d_4$}}
\psfrag{f1}[c][l]{\Small{$f_1$}}
\psfrag{f2}[c][l]{\Small{$f_2$}}
\psfrag{f3}[c][l]{\Small{$f_3$}}
\psfrag{f4N}[c][l]{\Small{$f_4=N$}}
\epsfbox{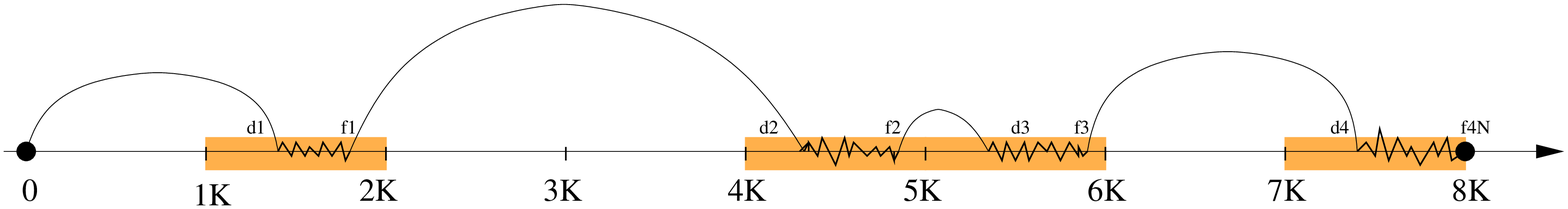}
\caption{\label{fig:cg} The figure above explains our coarse graining
  procedure. Here $N=8k$, $\mathcal I=\{2,5,6,8\}$. The drawn
  trajectory is a typical trajectory contributing to $\hat Z^{\mathcal
    I}_{N,\go}$; $d_i$ and $f_i$, $1\le i\le 4$, correspond to the
  indexes of \eqref{eq:decomp1}. The shadowed regions represent the sites
  on which the change of measure 
  procedure (presented in \S~\ref{sec:com}) acts.}
\end{center}
\end{figure}

\subsection{The change of measure}
\label{sec:com}
We introduce 
\begin{equation}
\label{eq:X}
 X_j\, :=\, \sum_{\ui \in B_j^q} V_k({\ui})\go_{\ui}, 
\end{equation}
where $ B_j^q$ is the Cartesian product of $B_j$ with itself $q$ times and 
$\go_{\ui}= \prod_{a=1}^q \go_{i_a}$. The {\sl potential}
 $V_k(\cdot)$ plays a crucial role for the sequel: we define it and discuss some
 of its
 properties  in the next remark.
 \medskip
 
 \begin{rem}
\label{rem:V}\rm
The potential $V$ is best introduced if we define the sorting operator $\sort(\cdot)$:
if $\ui  \in \R^q$ ($q =2,3, \ldots$), $\sort (\ui) \in \R^q$ is the non-decreasing rearrangement 
of the entries of $\ui$. We introduce then
\begin{equation}
\label{eq:U23}
U({\ui}) \, :=\, \prod_{a=2}^{ q}
 R_{\frac12} \left( \sort(\ui)_a- \sort(\ui)_{a-1}\right),
\end{equation}
The {\sl potential} $V$ is defined  by renormalizing  $U$ and by setting to zero
the {\sl diagonal} terms: 
\begin{equation}
\label{eq:useR2}
V_k({\ui}) \, := \, \frac{1}{(q!)^{1/2}  k^{1/2}\tilde \bL(k)^{(q-1)/2}}U( \ui)
 \ind_{\{i_a\neq i_b \text{ for every } a, b\}}
,
\end{equation}
where $\tilde \bL(\cdot)$ is defined as in \eqref{eq:Ltilde}, with $L(\cdot)$
replaced by $\bL(\cdot)$. By exploiting the fact that for every $c>0$ 
we have $\sum_{i\le cN} R_{1/2}(i)^2 \stackrel{N \to \infty} \sim 
\tilde \bL (N)$
one  sees that 
\begin{equation}
\label{eq:cV-1}
\sum_{\ui \in B_1^q} V_k({\ui})^2 \, = \, \frac{1}{k\,  ( \tilde \bL(k))^{q-1}}
\sum_{0<i_1<\ldots < i_q \le k} \prod_{a=2}^q \left(R_{1/2}(i_a-i_{a-1})\right)^2
 \stackrel{k \to \infty}\sim  1.
\end{equation}
Therefore 
\begin{equation}
\label{eq:cV}
\sum_{\ui \in B_1^q} V_k({\ui})^2 \, \le \, 2,
\end{equation}
for $k$ sufficiently large.
\end{rem}

\medskip

Let us introduce, for $K>0$, also
\begin{equation}
\begin{split}
 f_K(x)\, &:=\, -K\ind_{\{x\ge \exp(K^2)\}},\\
 g_{\mathcal I}(\go)\, &:= \, \exp\left(\sum_{j\in \mathcal I}f_K(X_j)\right),\\
 \bar g(\go)\, &:= \, \exp(f_K(X_1)).
\end{split}
\end{equation}
We are now going to replace, for fixed $\cI$, the measure
$\bbP(\dd \go)$ with $ g_{\mathcal I}(\go) \bbP(\dd \go)$. The latter is not a probability measure:
we could normalize it, but this is inessential because we are directly exploiting 
 H\"older inequality to get
\begin{equation}
\bbE \left[\left(\hat Z_{\go}^{\mathcal I}\right)^{\gamma}\right]
\, \le\,  
\left(\bbE\left[g_{\mathcal I}(\go)^{-\frac{\gamma}{1-\gamma}}\right]\right)^{1-\gamma}\left(\bbE\left[g_{\mathcal I}(\go)
\hat Z_{\go}^{\mathcal I}\right]\right)^{\gamma}.
\end{equation}
The first factor in the right-hand side is easily controlled, in fact
\begin{equation}
\label{eq:factor1}
\bbE\left[g_{\mathcal I}(\go)^{-\frac{\gamma}{1-\gamma}}\right]\,=
\, \bbE \left[ \bar g (\go)^{-\frac{\gamma}{1-\gamma}}\right]^{\vert \cI \vert} 
\, =\, \left[ \left(\exp\left(\frac{K\gamma}{1-\gamma}\right)-1\right)
\bbP\left(X_1\ge \exp\left(K^2\right)\right)+1 \right]^{\vert \cI \vert} ,
\end{equation}
and since $X_1$ is centered and its variance coincides 
with the left-hand side of \eqref{eq:cV}, by Chebyshev inequality
the term $ \exp\left({K\gamma}/{(1-\gamma)}\right)\bbP\left(X_1\ge \exp\left(K^2\right)\right)$
can be made arbitrarily small by choosing $K$ large. Therefore for $K$ sufficiently large
(depending only on $\gamma(=6/7)$)
\begin{equation}
\label{eq:fHolder}
\bbE \left[\left(\hat Z_{\go}^{\mathcal I}\right)^{\gamma}\right]
\, \le\,  2^{\gamma \vert \cI\vert}
\left(\bbE\left[g_{\mathcal I}(\go)
\hat Z_{\go}^{\mathcal I}\right]\right)^{\gamma}.
\end{equation}

Estimating the remaining factor is a more involved matter. We will actually use the following two statements,
that we prove in the next section. Set
$P_\cI \, :=\, 
 \bP\left( E_\cI ; \, \gd_N=1\right)$.

\medskip

\begin{proposition}
\label{th:eta}
Assume that $\alpha=1/2$ and that \eqref{eq:Lassumption} holds
for some $\epsilon\in(0,1/2]$. For every $\eta>0$ and every $q>(2\epsilon)^{-1}$
we can choose $A>0$ such that if $\gb \le\gb_0$ and $h\le \Delta(\beta;q,A)$,
 for every $\mathcal I\subset\left\{1,\dots,m\right\}$ with $m\in\mathcal I$ we have
\begin{equation}
\label{eq:eta}
 \bbE\left[g_{\mathcal I}(\go)\hat Z_{\go}^{\mathcal I}\right]
 \,\le\, 
  \eta^{\vert\mathcal I\vert }P_{\mathcal I}.
\end{equation}
\end{proposition}

\medskip

 The following technical estimate controls $P_{\mathcal I}$ 
 (recall that $\cI=\{i_1, \ldots, i_{\vert \cI\vert}\}$).
\medskip

\begin{lemma} 
\label{th:P_cI}
Assume $\ga=1/2$. 
There exist $C_1=C_1(L(\cdot), k)$,
 $C_2=C_2(L(\cdot))$  and $k_0=k_0(L(\cdot))$ such that (with $i_0:=0$) 
\begin{equation}
\label{eq:P_cI}
P_\cI \, 
\le \,  C_1C_2^{\vert\cI\vert}\prod_{j=1}^{\vert\cI\vert}\frac{1}{(i_j-i_{j-1})^{7/5}},
\end{equation}
for $k \ge k_0$. 
\end{lemma}

\medskip

Note that in this statement $k$ is just a natural number, but we will apply
it with $k$ as in \eqref{eq:k333} so that $k \ge k_0$
is just a requirement on $A$. Note also that the choice of $7/5$ is arbitrary
(any number in $(1, 3/2)$ would do: the constants $C_1$ and $C_2$
depend on such a number).

\smallskip

Let us now go back to \eqref{eq:fHolder}
and let us plug it into \eqref{eq:cg+fm} and use
Proposition~\ref{th:eta} and Lemma~\ref{th:P_cI} to get:
\begin{equation}
\bbE\left[ Z_{N , \go}^\gamma\right]
\, \le\, C_1^\gamma \sumtwo{\mathcal I\subset \left\{1,\dots,m\right\}}{m\in\mathcal I}
\prod_{j=1}^{\vert\cI\vert}
\left(
\frac{(2C_2 \eta)^\gamma}{(i_j-i_{j-1})^{7\gamma/5}}\right).
\end{equation}
But $7\gamma/5=6/5 >1$, so
we can choose
\begin{equation}
 \label{eq:fmm2-1}
\eta \, :=\, \frac1{3C_2 
\left(\sum_{i=1}^\infty i^{-6/5}
\right)^{7/6}
}, 
\end{equation}
and this implies that $n \mapsto (2C_2 \eta)^\gamma n^{-7\gamma/5}$
is a sub-probability, which directly entails that
 \begin{equation}
 \label{eq:fmm2}
\bbE\left[ Z_{N , \go}^\gamma\right]
\, \le\, C_1^\gamma,  \quad \quad \quad (\gamma=6/7)
\end{equation}
for every $N$, which implies, via \eqref{eq:fmm},
that $\tf(\gb, h)=0$ and we are done.

It is important to stress that $C_1$ may depend
 on $k$ (we need \eqref{eq:fmm2}  uniform in
 $N$, not in $k$), but $C_2$ does not ($C_2$ is just a function
 of $L(\cdot)$, that is a function of the chosen renewal), so that $\eta$ may actually be chosen
 a priori as in \eqref{eq:fmm2-1}: it is a small but fixed constant that depends only on
 the underlying renewal $\tau$. 

\medskip 

\begin{rem}\rm
\label{rem:kgD}
In this section we have actually hidden the role of 
$\gD(\gb; q,A)$ in the hypotheses of Proposition~\ref{th:eta},
which are the hypotheses of Theorem~\ref{th:main333}. 
Let us therefore explain informally why we can prove a  critical point shift of  $\gD=1/k$.

The coarse graining procedure reduces proving delocalization to 
Proposition~\ref{th:eta}. As it is quite intuitive from \eqref{eq:decomp333}--\eqref{eq:fHolder} and 
Figure~\ref{fig:cg},
one has to estimate  the expectation, with respect to
the $\bar g (\go)$-modified measure, of the partition function 
$Z_{d_j, f_j}$ (or, equivalently, $Z_{d_j, f_j}/\bP(f_j-d_j\in \tau)$)
in each visited block
(let us assume that $f_j-d_j$ is {\sl of the order of $k$},
because if it is much smaller than $k$ one can bound this
contribution in a much more elementary way). The Boltzmann factor in $Z_{d_j, f_j}$
is $\exp( \sum_{n=d_j+1}^{f_j} (\gb \go_n -\log \M (\gb) +h) \gd_n)$
which can be bounded (in an apparently very rough way)
by $\exp( \sum_{n=d_j+1}^{f_j} (\gb \go_n -\log \M (\gb) ) \gd_n)
\exp(hk)$, since $f_j-d_j \le k$. Therefore, if $h \le \Delta(\beta;q,A)\sim 1/k$ we can drop
the dependence on $h$ at the expense of the multiplicative factor $e$
that is innocuous because we can show that
the expectation (with respect to
the $\bar g (\go)$-modified measure) of 
$Z_{d_j, f_j}/\bP(f_j-d_j\in \tau)$ when $h=0$ can be made 
arbitrarily small by choosing $A$ sufficiently large. 
\end{rem}

\medskip

\begin{rem}\rm
\label{rem:lazy}
A last observation on the proof is about $\gb_0$.
It can be chosen arbitrarily, but for the sake of simplifying
the constants appearing in the proofs we choose 
 $\gb_0\in (0, \infty)$
such that
\begin{equation}
\label{eq:gb0}
\frac 12 \, \le \, 
\frac {\dd^2}{\dd \gb^2} \log \M( \gb)\, \le \, 2 \,  ,
\end{equation}
for $\gb \in [0, \gb_0]$.
Choosing $\gb_0$ arbitrarily just boils down to changing
the constants in the right-most and left-most terms in \eqref{eq:gb0}.
\end{rem}

\section{Coarse graining estimates}
\label{sec:cg-est}

We start by proving Lemma~\ref{th:P_cI}, namely
\eqref{eq:P_cI}. The proof is however more clear if instead
of working with the exponent $7/5$ we work with
$3/2-\xi$ ($\xi \in (0,1/2)$, in the end, plug in $\xi=1/10$).
\medskip

\noindent {\it Proof of 
Lemma~\ref{th:P_cI}}.
First of all, 
  in the product on the right--hand side of \eqref{eq:P_cI} one can clearly ignore the terms
  such that
  $i_j-i_{j-1}=1$. 
We then express $\cI$ in a more practical way by observing that
we can define, in a unique way, an integer $p\le l:=
\vert \cI\vert$ and increasing sequences of integers $\{a_j\}_{ j=1, \ldots, p} $, 
$\{b_j\}_{ j=1, \ldots,  p}$ with $b_p=m$, $a_j\ge b_{j-1}+2$ 
(for $j>1$) and $b_{j}\ge a_j$ such that
\begin{equation}
 \mathcal I= \bigcup_{j=1}^p [a_j,b_j]\cap \N.
\end{equation}
With this definition, it is sufficient  to show
\begin{equation}
 P_{\mathcal I}\le  C_1 C_2^l \frac1{a_1^{3/2-\xi}}\prod_{j=1}^{p-1}\frac{1}{(a_{j+1}-b_{j})^{3/2-\xi}}.
\end{equation}
We start then by writing
\begin{multline}\label{eq:upbound}
 P_{\mathcal I}
\le
\sumtwo{d_1\in B_{a_1}}{f_1\in B_{b_1}}\ldots\sumtwo{d_{p-1}\in B_{a_{p-1}}}{f_p\in B_{b_{p-1}}}\sum_{d_p\in B_{a_p}}K(d_1)\bP(f_1-d_1\in \tau)\ldots K(d_p-f_{p-1})\bP(N-d_p\in \tau),
\end{multline}
where the inequality comes from neglecting the constraint that
$\tau$ has to intersect $B_{a_j+1},\ldots$  $B_{b_j-1}$.
Note that the meaning of the $d$ and $f$ indexes is somewhat different
with respect to \eqref{eq:decomp1} and that in the above sum we always
have
\begin{equation}\begin{split}
 d_1&\in B_{a_1},\\
 (a_j-b_{j-1}-1)k&\le d_j-f_{j-1}\le (a_j-b_{j-1}+1)k,\\
 (b_j-a_{j}-1)k \vee 0&\le f_j-d_j\le (b_j-a_j+1)k.
 \end{split}\end{equation}
In particular, $f_j\ge d_j$ is guaranteed by the fact that 
$\bP(f_j-d_j\in\tau)=0$ otherwise.

 Observe now that for $k $ sufficiently large
\begin{equation}
 \sum_{x\in B_{a_1}}K(x)\, \le\,  
\begin{cases} 
1& \text{ if } a_1=1, \\
 3 \frac{L((a_1-1) k)}{k^{1/2}(a_1-1)^{3/2}}& \text{ if } a_1=2,3, \ldots, 
 \end{cases}
 \, \le \, c_1(k) \frac{ L(a_1k)}{k^{1/2} a_1^{3/2}},
\end{equation}
where $c_1(k):= \max ( 10, k^{1/2}/L(k))$.
 Moreover
there exists a constant $c_2$ depending on $L(\cdot)$ such that 
for $j>1$
\begin{equation}\begin{split}
 \sum_{x=(a_j-b_{j-1}-1)k}^{(a_j-b_{j-1}+1)k}K(x)&\le c_2 \frac{L(k(a_j-b_{j-1}))}
 {k^{1/2}(a_j-b_{j-1})^{3/2}},\\
 \sum_{x=(b_j-a_{j}-1)k\vee 0}^{(b_j-a_j+1)k}\bP(x\in \tau)&\le c_2 \frac{k^{1/2}}{(b_j-a_j+1)^{1/2}L\left(k(b_j-a_j+1)\right)}.
\end{split}\end{equation}
The first inequality is obtained by making use of $a_j\ge b_{j-1}+2$.
Neglecting the last term which is smaller than one, we can bound the right--hand side of \eqref{eq:upbound} and get
\begin{equation}
\label{eq:truc}
 P_{\mathcal I}\le c_1(k)c_2^{2p}  \frac{ L(a_1k)}{k^{1/2} a_1^{3/2}}\prod_{j=1}^{p-1}
 \left( \frac{L(k(a_{j+1}-b_{j}))}{(a_{j+1}-b_{j})^{3/2}} \right)
 \left(  \frac{1}{(b_j-a_j+1)^{1/2}L\left(k(b_j-a_j+1)\right)}\right)
.
\end{equation}
Notice now that since $L(\cdot)$ grows slower than any power,
 $ \sup_{a_1}{ L(a_1k)}/({k^{1/2} a_1^{\xi}})$ is $o(1)$ for $k$ large.
 To control the other terms 
 we use the {\sl Potter bound} \cite[Th.~1.5.6]{cf:RegVar}: given
 a slowly varying function $L(\cdot)$ which is locally bounded away
 from zero and infinity (which we may assume in our set up without loss of generality),
 for every $a>0$ there exists $c_a>0$ such that for every $x,y >0$
 \begin{equation}
\label{eq:Potter}
\frac{L(x)}{L(y)} \, \le \,  c_a \max \left(\frac{x}{y}, \frac{y}{x} \right)^a.
 \end{equation}
This bound implies 
that for large enough $k$
\begin{equation}
\label{eq:2ineqs}
 \sup_{x\ge 1}\frac{L(k)}{\sqrt{x}L(kx)}\, \le\,  2 \ \text{ and } \
 \sup_{x\ge 1} \frac{L(kx)}{L(k)x^{\xi}}\,\le \,2.
\end{equation}
In fact consider the second bound (the argument for the first one is identical): by choosing $a=\xi/2$ we have
 $L(kx)/(L(k)x^{\xi})\le c_{\xi/2} x^{-\xi/2} \le 2$ and the second
 inequality holds for $x$ larger than a suitable constant $C_\xi$.
 For $x(\ge 1)$ smaller than $C_\xi$ instead it suffices
 to choose $k$ sufficiently large so that $L(kx)/L(k) \le 2$
 for every $x \in [1, C_\xi]$. 
Using the two bounds \eqref{eq:2ineqs} 
in \eqref{eq:truc} we complete the proof.
\qed

\medskip
The proof of 
Proposition~\ref{th:eta} depends on the following lemma that will be proven
in the next section. 
\medskip

\begin{lemma}
\label{TH:DF}
Set $h=0$, fix $q\in \N$, $q> (2\epsilon)^{-1}$ as in Theorem 
\ref{th:main333},  and 
recall the definition of $k=k(\gb; q,A)$ \eqref{eq:k333}.
For every $\gep$ and $\gd>0$ there exists $A_0>0$ such that
for $A\ge A_0$
\begin{equation}
 \bbE\left[\bar g(\go) z_d Z_{d,f}\right]\, \le\,
  \delta\,  \bP(f-d\in \tau),
 \label{eq:df}
\end{equation}
for every  $d$ and $f$ such that $0\le d\le d+ \gep k\le f\le k$
and $\gb \le \gb_0$. 
\end{lemma}
\medskip

\begin{proof}[Proof of Proposition~\ref{th:eta}]
Recalling \eqref{eq:decomp1} and the notations for the set $\cI$ in there, we have
\begin{multline}
\label{eq:longlong}
  \bbE\left[g_{\mathcal I}(\go)\hat Z_{\go}^{\mathcal I}\right]\,=\\
  \sumtwo{d_1, f_1 \in B_{i_1}}{d_1\le f_1}
  \sumtwo{d_2, f_2\in B_{i_2}}{d_2\le f_2}\ldots \sum_{d_l\in B_{i_l}}
  K(d_1)\bbE\left[\bar g(\go) z_{d_1-k(i_1-1)}Z_{d_1-k(i_1-1),f_1-k(i_1-1)}\right] K(d_2-f_1)
  \ldots \\ \quad \quad \quad \quad \quad \quad \quad \quad \quad \quad \quad \quad 
  K(d_l-f_{l-1})\bbE\left[\bar g(\go) z_{d_l-k(m-1)} Z_{d_l-k(m-1),k}\right]\\
  \le e^l \sumtwo{d_1, f_1 \in B_{i_1}}{d_1\le f_1}
\sumtwo{d_2, f_2\in B_{i_2}}{d_2\le f_2}\ldots \sum_{d_l\in B_{i_l}}
  K(d_1)(\gd+\ind_{\{f_1-d_1\le \gep k\}})\bP(f_1-d_1\in\tau)
  K(d_2-f_1)\ldots \\
  K(d_l-f_{l-1})(\gd+\ind_{\{N-d_l\le \gep
    k\}})\bP(N-d_l\in \tau),
\end{multline}
where the factor $e^l$ in the last expression 
comes from bounding the contribution due to $h$ (recall that $h k \le 1$).
We now consider $B_{i_j}$ as the union of two sub-blocks
\begin{equation}\begin{split}
 B_{i_j}^{(1)}&:=\left\{(i_j-1)k,\dots,(i_j-1)k+\lfloor k/2 \rfloor \right\},\\
 B_{i_j}^{(2)}&:=\left\{(i_j-1)k+\lceil k/2 \rceil,\dots, i_jk \right\}.
\end{split}\end{equation}
If $d_j\in B_{i_j}^{(1)}$ then if $\gep$ is sufficiently small
($\gep \le 1/10$ suffices) we have that for $k$ sufficiently large
({\sl i.e.} $k \ge k_0(L(\cdot), \gep)$)
\begin{equation}
\sum_{f=d_j}^{d_j+\gep k}
\bP(f-d_j\in \tau )K(d_{j+1}-f)
\le 4 \left(\sum_{x=1}^{k\gep} \bP(x\in \tau)\right)
K(k(i_{j+1}-i_j)).
\end{equation}
This can be compared to
\begin{equation}
\sum_{f=d_j}^{ki_j}\bP(f-d_j\in \tau )K(d_{j+1}-f)
\, \ge\,  \frac{1}{3} \left(\sum_{x=1}^{\lfloor k/4\rfloor}\bP(x\in \tau )\right) 
K(k(i_{j+1}-i_j)),
\end{equation}
that holds once again for $k$ large. By using that
 $\sum_{x=1}^n \bP(x \in \tau)$
behaves for $n $ large like $\sqrt{n}$ times a slowly varying function 
({\sl cf.} \eqref{eq:Doney1/2}) we therefore see that
 given $\delta>0$ we can  find $\gep$ such that for any $d_j\in B_{i_j}^{(1)}$ we have
\begin{equation}
\sum_{f=d_j}^{d_j+\gep k}\bP(f-d_j\in \tau )K(d_{j+1}-f)\le \gd \sum_{f=d_j}^{ki_j}\bP(f-d_j\in \tau ) K(d_{j+1}-f).
\end{equation}
Using the same argument in the opposite way one finds that if $f_j\in B_{i_j}^{(2)}$
\begin{equation}
 \sum_{d=f_j-\gep k}^{f_j}K(d-f_{j-1})\bP(f_j-d\in \tau )\, \le\,
  \gd \sum_{d=k(i_j-1)}^{f_j}K(d-f_{j-1})\bP(f_j-d\in \tau ).
\end{equation}
Since either $d_j\in B_{i_j}^{(1)}$ or $f_j\in B_{i_j}^{(2)}$, we conclude that
\begin{multline}
 \sumtwo{d_j, f_j \in B_{i_j}}{d_j\le f_j}\ind_{\{f_j-d_j\le k\gep \}} K(d_j-f_{j-1})\bP(f_j-d_j\in \tau)K(d_{j+1}-f_j)\\
\le \gd \sumtwo{d_j, f_j \in B_{i_k}}{d_j\le f_j}  K(d_j-f_{j-1})\bP(f_j-d_j\in \tau)K(d_{j+1}-f_j).
\end{multline}
The analog estimate can be obtained for the sum over $d_l$
in \eqref{eq:longlong} (rather, it is  easier). Using this inequality $j=1\dots l$ we get our result for $\eta=2e \gd$. 

\end{proof}

\section{The $q$-body potential estimates (proof of Lemma~\ref{TH:DF})}

In what follows $X=X_1$ and we fix $\gd\in (0,1)$. The positive (small) number
$\gep$ is fixed too, as well as $q>(2\epsilon)^{-1}$, where 
$\epsilon$ is the same which appears in the statement 
of Theorem \ref{th:main333}.


\medskip

\noindent
{\it Proof of Lemma~\ref{TH:DF}.}
We start by observing that, since $h=0$,
\begin{equation}
\bbE\left[ \bar g(\go) z_d Z_{d,f} \right]\, =\, 
\bE_{d,f}\left[
\bbE \left[ \bar g(\go) \exp \left( \sum_{n=d}^f
(\gb \go_n - \log \M (\gb) ) \gd_n \right) \right] \right] \bP(f-d \in \tau),
\end{equation}
where $\bP_{d,f}$ is the law of  $\tau \cap [d,f]$, conditioned to
$f,d \in \tau $. Given the random set (or renewal trajectory) $\tau$ we  introduce the probability measure
\begin{equation}\label{eq:ptau}
\hat \bbP_{\tau}(\dd\go)\, :=\,  \exp \left( \sum_{n=d}^f
(\gb \go_n - \log \M (\gb) ) \gd_n \right) \bbP(\dd \go).
\end{equation}
Note that $\go$, under $\hat \bbP_\tau$, is still a sequence
of independent random variables, but they are no longer identically distributed.
We will use that, for $d\le n\le f$,
\begin{equation}
\label{eq:2beused}
\hat \bbE_\tau \go_n \, =\, \mbeta  \gd_n \stackrel{\gb \searrow 0}\sim \gb \gd_n
\ \text{(so that }\gb/2 \le \mbeta \le 2 \gb \, \text{)} 
\ \  \text{ and } \  \
\text{var}_{\hat \bbP_\tau} \left( \go_n \right) \, \le \, 2, 
\end{equation}
where the inequalities hold for $\gb\le \gb_0$ (recall \eqref{eq:gb0}) and all relations hold uniformly in the 
renewal trajectory
$\tau$. 
On the other hand, for $n\notin \{d,\ldots,f\}$ the $\go_n$'s are IID exactly as under
$\bbP$.
We have:
\begin{multline}
\label{eq:3lhat}
\frac{
\bbE\left[ \bar g(\go) z_d Z_{d,f} \right]}
{\bP(f-d \in \tau)}
\, =\, \bE_{d,f} \hat \bbE_{\tau} \left[ \bar g(\go) \right]\, =
\\ 
\exp(-K) 
 \bE_{d,f} \hat \bbP_{\tau} \left[X \ge \exp(K^2) \right]
 +  \bE_{d,f} \hat \bbP_{\tau} \left[ X< \exp(K^2) \right]
 \, \le \\
  \exp(-K) + \bE_{d,f} \hat \bbP_{\tau} \left[ X< \exp(K^2) \right]\, \le \, 
 \frac {\gd }3 +  
  \bE_{d,f} \hat \bbP_{\tau} \left[ X< \exp(K^2) \right],
\end{multline}
where in the last step we have chosen $K $ such that 
$\exp(-K ) \le \gd /3$. 
We are now going to use the following lemma:

\medskip

\begin{lemma}
\label{TH:FROMCE}
If $d$ and $f$ are chosen such that 
$f-d \ge \gep k$ and 
 $X(=X_1)$ is defined as in \eqref{eq:X}, that is
$X= \sum_{{\ui}\in B_1^q} V_k({\ui})\go_{\ui}$, 
 we have  that for every $\zeta>0$ we can find $a>0$ and $A_0$ such that
\begin{equation}
\label{eq:fromCE}
 \bP_{d,f} \left( \hat \bbE_\tau X \, >\, aA^{(q-1)/2}\right)\, \ge \, 1-\zeta,
\end{equation}
for $\gb\le \gb_0$ and $A\ge A_0$.
\end{lemma}
\medskip

We apply this lemma by setting $\zeta=\gd/3$ (so $a$ is fixed once
$\gd$ is chosen)
so that,
if we choose $K$ such that $2\exp(K^2)=aA^{(q-1)/2}$
(note that, by choosing $A$ large we make $K$ large
and we automatically satisfy the previous requirements on $K$),
we have 
$\bP_{d,f} \left( \hat \bbE_\tau X < 2
\exp(K^2)\right) \le \gd/3  $, so that, in view of \eqref{eq:3lhat},
we obtain
\begin{equation}
\label{eq:2/3and}
\begin{split}
\frac{
\bbE\left[ \bar g(\go) z_d Z_{d,f} \right]}
{\bP(f-d \in \tau)}
\, &\le \, \frac {2\gd}3\, +\,  \bE_{d,f} \hat \bbP_{\tau} \left[ X-\hat \bbE_\tau X \le  -\exp(K^2) \right]
\\
&\le \, \frac {2\gd}3\, +\,  \frac{4}{a^2 A^{q-1}}
\bE_{d,f} \hat \bbE_{\tau} \left[ \left(X-\hat \bbE_\tau X\right)^2 \right]\, .
\end{split}
\end{equation}

The conclusion now follows 
as soon as we can show that the second moment appearing  in the last term of \eqref{eq:2/3and}
is $o(A^{q-1})$ for $A$ large. But this is precisely 
what is granted by the 
next lemma:

\medskip

\begin{lemma}
\label{th:square}
There exist $A_0>0$  such that
\begin{equation}
\label{eq:square}
 \bE_{d,f} \hat \bbE_{\tau} \left[ \left(
X-\hat \bbE_\tau X \right)^2 \right] \, \le \, A^{(q-1)^2/q} ,
\end{equation}
 for every $\gb \le \gb_0$ and every $A\ge A_0$. 
\end{lemma}

\medskip

\noindent
{\it Proof of Lemma~\ref{th:square}.}
We start by introducing the notation 
$\hat \go_{n} :=
\go_{n} - \mbeta \gd_{n}{\bf 1}_{\{d\le n\le f\}}$ and by observing that
\begin{equation}
\begin{split}
\hat \bbE_{\tau} \left[ \left(
X-\hat \bbE_\tau X \right)^2 \right]
\, &= \, 
\hat \bbE_{\tau} \left[ \left( \sum_{\ui \in B_1^q} V_k({\ui}) 
\prod_{a=1}^q \left( \hat \go_{i_a} + \mbeta \gd_{i_a} 
{\bf 1}_{\{d\le i_a\le f\}}
\right)
- \mbeta ^q \sum_{\ui \in \{d,\ldots,f\}^q} V_k({\ui}) \gd_{\ui}
\right) ^2\right]
\\
&\le C(q)\, \hat \bbE_{\tau} \left[\left(
\sum_{\ell=0}^{q-1} \mbeta ^{\ell} \sum_{\ui \in B_1^{q-\ell}}
\sum_{\uj \in \{d,\ldots,f\}^\ell}  V_k({\ui \,  \uj}) \hat \go_{\ui} \gd_{\uj} \right)^2\right]
\\
& \, \le \, C(q)
\sum_{\ell=0}^{q-1} \mbeta ^{2\ell}
\sum_{\ui \in B_1^{q-\ell}}
\sum_{\uj,\um \in \{d,\ldots,f\}^\ell} 
 V_k({\ui \,  \uj}) V_k({\ui \,  \um}) \gd_{\uj} \gd_{\um} 
\, ,
\end{split}
\end{equation}
where $\ui \, \uj \in B_1^q$ is the concatenation of $\ui$ and $ \uj$
and in the last step we have first used the Cauchy-Schwarz inequality,
the fact that the $\hat \go$ variables are independent and centered and
\eqref{eq:2beused}.  
\begin{rem}
  \rm
\label{rem:nonumero}
Here and in the following, we adopt the convention that
$C(a,b,\ldots)$ is a positive constant (which depends on the
parameters $a,b,\ldots$), whose numerical value may change from line
to line.
\end{rem}
Therefore
\begin{equation}
\label{eq:bfl1}
\bE_{d,f}
\hat \bbE_{\tau} \left[ \left(
X-\hat \bbE_\tau X \right)^2 \right]
\, \le \, C(q)
\sum_{\ell=0}^{q-1} \mbeta ^{2\ell}
\sum_{\ui \in B_1^{q-\ell}}
\sum_{\uj, \um \in \{d, \ldots , f\}^\ell}
 V_k({\ui \,  \uj}) V_k({\ui \,  \um}) \bE_{d,f} \left[\gd_{\uj} \gd_{\um} \right].
\end{equation}
Let us point out immediately that we know how to deal with the $\ell=0$ case:
it is simply $ C(q) \sum_{\ui \in B_1^q} V_k({\ui})^2$  and it is therefore bounded by 
$ 2C(q)$ ({\sl cf.} \eqref{eq:cV}).
By using the notation and the bounds in Remarks \ref{rem:L} and \ref{rem:V},
together with the renewal property, we readily see that 
\begin{equation}
\label{eq:useR}
\bE_{d,f} \left[\gd_{\uj} \gd_{\um} \right] \, \le \, 
\frac{\const _L ^{-(2\ell+1)}}{\bP( f-d \in \tau)}
\prod_{a=1}^{ 2\ell+1}
R_{\frac12} \left( r_a-r_{a-1}\right)\, \le \, 
\frac{\const _L ^{-(2\ell+1)}}{R_{\frac12} \left(f-d \right)}
\prod_{a=1}^{ 2\ell+1}
R_{\frac12} \left( r_a-r_{a-1}\right),
\end{equation}
for $\uj , \um\in \{d, \ldots, f\}^\ell$, $r= \sort(\uj \, \um)$,
$r_0:=d$ and $r_{2\ell+1}:=f$.  A notational simplification may be
therefore achieved by exploiting further Remark~\ref{rem:V}, namely by
using \eqref{eq:U23}, so that \eqref{eq:useR} becomes
\begin{equation}
\label{eq:useR.1}
\begin{split}
\bE_{d,f} \left[\gd_{\uj} \gd_{\um} \right] \, &\le \, 
\const _L ^{-(2\ell+1)}
R_{\frac12} \left(f-d \right)^{-1} R_{\frac12}(\min(\uj\, \um) -d) 
U(\uj\, \um) R_{\frac12}(f-\max(\uj\, \um) )
\\
&= \, 
\const _L ^{-(2\ell+1)}
R_{\frac12} \left(f-d \right)^{-1} 
U(d\, \uj\, \um \, f) \, 
.
\end{split}
\end{equation}
By inserting \eqref{eq:useR.1} and \eqref{eq:useR2} into 
\eqref{eq:bfl1} we get to 
\begin{equation}
\label{eq:bfl2}
\begin{split}
&\bE_{d,f} 
\hat \bbE_{\tau} \left[ \left(
X-\hat \bbE_\tau X \right)^2 \right]
\\ 
& \le  C\left(1+\frac1
{k \tilde L(k)^{q-1}R_{\frac 12}(f-d) } 
\sum_{\ell=1}^{q-1} \mbeta ^{2\ell}
\sum_{\ui \in B_1^{q-\ell}}
 \sum_{\uj, \um \in \{d, \ldots , f\}^\ell}
U(\ui\, \uj) U(\ui\, \um) U(d\, \uj\, \um \, f)\right)
\\
& \le C \left(1+
 \frac1
{k \tilde L(k)^{q-1} R_{\frac 12}(f-d)} 
\sum_{\ell=1}^{q-1} \mbeta ^{2\ell} 
\sum_{\ui \in \sort(B_1^{q-\ell})}
\sum_{\uj, \um \in \sort(\{d, \ldots , f\}^\ell)}
U(\ui\, \uj) U(\ui\, \um) U(d\, \uj\, \um\, f)\right),
\end{split}
\end{equation}
where of course $\sort(\{1, \ldots, a\}^n)= \{ \ui \in \{1, \ldots,
a\}^n: \, i_1 \le i_2 \le \ldots \le i_n\}$ and $C=C(q,L(\cdot))$,
with the convention of Remark \ref{rem:nonumero}.

\smallskip

The rest of the proof is devoted to bounding
\begin{equation}
\label{eq:Tell}
T_{q,\ell} \, := \, 
\sum_{\ui \in \sort(B_1^{q-\ell})}
\sum_{\uj, \um \in \sort(\{d, \ldots , f\}^\ell)}
U(\ui\, \uj) U(\ui\, \um) U(d\, \uj\, \um \,f).
\end{equation}
This is relatively heavy, because, while $\ui$, $\uj$ and $\um$ are
ordered, $\ui\, \uj$, $\ui\, \um$ and $\uj\, \um$ are not.  We have
therefore to estimate the contributions given by every mutual
arrangement of $\ui$, $\uj$ and $\um$.  This will be done in a
systematic way with the help of a {\sl diagram representation} (the
diagrams will correspond to groups of configurations $\ui$, $\uj$ and
$\um$ that have the same {\sl mutual order}).  \smallskip

Fix $q$ and $\ell$ and choose $\ui\in \sort(\{1, \ldots, k\}^{q-\ell})$ and  $\uj, \um
\in \sort(\in \{d, \ldots, f\}^{\ell})$.
The construction of the diagram of  $\ui$, $\uj$ and $\um$ is done in steps:

\smallskip

\begin{enumerate}
\item
Mark with $\Box$'s
on the horizontal axis (the dotted line in Figure~\ref{fig:diagram-ex})
  the positions $i_1 \le i_2\le \dots\le i_{q-\ell}$.
Do the  same for $\uj$ (using $\circ$) and $\um$ (using $\bullet$).
As explained in Remark~\ref{rem:superpos} below, we may and do assume 
that symbols do not
 sit on the same position (this amounts to assuming
 strict inequality between all  indexes).
\item 
Consider the set of  $\Box$'s and $\circ$'s, and connect
all nearest neighbors  with a line (the line may be
straight or curved for the sake of visual clarity).   
\item Do the same for the set of  $\Box$'s and $\bullet$'s.
\item Do the same for the set of  $\circ$'s and $\bullet$'s.
\item Consider   the set of $\circ$'s and $\bullet$'s and connect the element that
is closest to $d$ with $d$. Do the analogous action  with the element which
is closest to $f$. The point $d$ is always to the left of $\circ$'s and $\bullet$'s
and the point $f$ is always to the right.
\end{enumerate}

\smallskip

We have now a graph with vertex set $\{d,f, \ui, \uj, \um\}$.
Vertices have a type ($\Box$, $\circ$ and $\bullet$): $d$ and $f$ have
their own type too, graphically this type is $\vert$.  We actually
consider the {\sl richer} graph with vertex set given by the points
and the type of the point.  The edges are the ones built with the
above procedure; note that there may be double edges: we keep them and
call them {\sl twin} edges.  Two indexes configurations are equivalent
if they can be transformed into each other by translating the indexes
without allowing them cross (and, of course, keeping their type; the
vertices $d$ and $f$ are fixed).  This leads to equivalence classes
and a class is denoted by $\cG$: we split the sum in \eqref{eq:Tell}
according to these classes, that is $T_{q,\ell}= \sum_{\cG} T_{q,\ell,
  \cG}$. The bound we are going to find is rather rough: we are going
in fact to bound $\max _{\cG} T_{q,\ell, \cG}$.

\smallskip

\begin{rem}
\label{rem:superpos}
\rm
We have built equivalent classes of non-superposing points only.
However in estimating $T_{q,\ell, \cG}$ we will allow the index summations to include
coinciding indexes so  in the end we include (and over-estimate) the contributions
of all the configurations of indexes. 
\end{rem}

\smallskip

\begin{figure}[hlt]
\begin{center}
\leavevmode
\epsfxsize =14.5 cm
\psfragscanon
\psfrag{0}[c][l]{\small $0$}
\psfrag{k}[c][l]{\small $k$}
\psfrag{d}[c][l]{\small $d$}
\psfrag{f}[c][l]{\small $f$}
\psfrag{1}[c][l]{\small Start}
\psfrag{2}[c][l]{\small Trim step $1$}
\psfrag{3}[c][l]{\small Trim step $2$}
\psfrag{4}[c][l]{\small Trim step $3$}
\psfrag{5}[c][l]{\small Trim step $4$}
\psfrag{i1}[c][l]{\small $i_1$}
\psfrag{i2}[c][l]{\small $i_2$}
\psfrag{i3}[c][l]{\small $i_3$}
\psfrag{i4}[c][l]{\small $i_4$}
\psfrag{i5}[c][l]{\small $i_5$}
\psfrag{j1}[c][l]{\small $j_1$}
\psfrag{m1}[c][l]{\small $m_1$}
\psfrag{j2}[c][l]{\small $j_2$}
\psfrag{m2}[c][l]{\small $m_2$}
\epsfbox{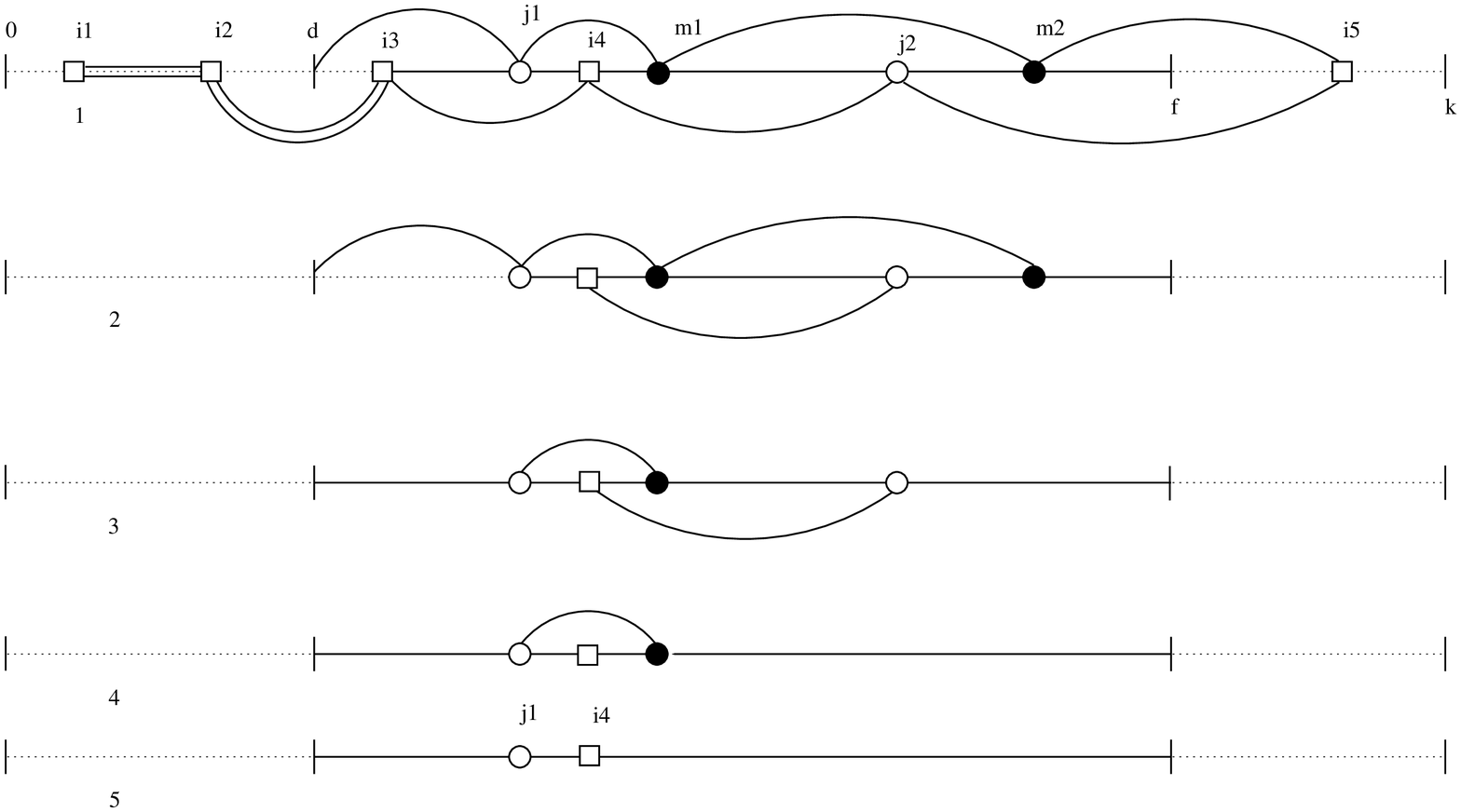}
\caption{\label{fig:diagram-ex}A diagram arising for $q=7$ and $\ell=2$ and the successive trimming procedure explained in the text}
\end{center}
\end{figure}
 
 \medskip
 
 In order to estimate $T_{q,\ell, \cG}$ we proceed to a graph trimming
 procedure that will be then matched to successive estimates on 
  $T_{q,\ell, \cG}$.
 
 The trimming procedure is the following:
 \smallskip

\begin{enumerate}
\item If there are $\Box$ vertices that are left of leftmost element of the 
set of $\circ$ and $\bullet$ vertices (we may call these $\Box$
vertices {\sl external vertices}), we erase them and we trim
the edges linking them. Note that if we do this procedure left to right,
we  erase one vertex and two edges at a time: at each step 
we trim a couple of twin edges, except at the last step in
which the edges are not twin.
We do the same with the $\Box$ vertices that are right 
of the rightmost element of the 
set of $\circ$ and $\bullet$ vertices (if any, of course). The trimming procedure goes 
this time right to left. We call {\sl internal} the vertices that are left. 
\item Now we start (say) right and we erase the rightmost internal
  vertex (in this first step is necessarily a $\circ$ or a $\bullet$,
  later it may be a $\Box$; we do not touch $d$ and $f$). Note that it
  has two edges (linking to vertices on the left) and one edge linking
  it with $f$: we trim these three edges and we add an edge linking
  the rightmost vertex (it can have any type among $\Box$, $\circ$ and
  $\bullet$) that is still present to $f$ with an edge.
\item We repeat step (2) till one is left with only four  vertices (among them, 
only  one may be a $\Box$) and three edges. {\sl Trim step 4} in Figure~\ref{fig:diagram-ex}
is a possible fully trimmed configuration.
\end{enumerate}

\medskip

Let us now explain  the link between the trimming procedure and quantitative estimates
on $T_{q, \ell, \cG}$.
Also this is done by steps corresponding precisely to the three steps
of the trimming procedure:

\medskip

\begin{enumerate}
\item Consider the external $\Box$ vertices connected to the rest of the graph by twin edges, if any.
 We start by the leftmost (if there is at least one on the left: the procedure from the right is absolutely analogous)
and notice that we can sum over the index, that is $i_1$, and use that, thanks to \eqref{eq:Lb} (recall
\eqref{eq:Ltilde} and \eqref{eq:R12}), there exists $C_L$ such that for $0<n \le k$
\begin{equation}
\label{eq:basicbound}
\sum_{i=0}^n (R_{\frac 12}(n-i))^2 \, \le \, C_L \tilde L(k).
\end{equation}
We are of course over-estimating the real sums that are, in most cases,
restricted to small portions of $B_1$. 
This estimate allows {\sl trimming} $T_{q, \ell, \cG}$ in the sense that
it gives the bound $T_{q, \ell, \cG} \le C_L^r \tilde L(k)^r T_{q-r, \ell, \cG^\prime}$, 
where $r$ is the number of twin edges and $\cG^\prime$ is the graph, with $q-r+\ell$ vertices
that is left after this procedure. This step can be repeated also for the last 
external $\Box$'s (there are at most two, one on the left and one on the right). In these cases
we simply use that $R_{\frac 12}(\cdot)$ is decreasing so that if 
$0\le n \le n^\prime $
\begin{equation}
\label{eq:basicbound2}
\sum_{i=0}^n R_{\frac 12}(n-i)R_{\frac 12}(n^\prime -i)  \, \le \, \sum_{i=0}^n (R_{\frac 12}(n-i))^2,
\end{equation}
and then \eqref{eq:basicbound} applies. So this extra trimming yields again $C_L \tilde L(k)$
to the power of half the number of edges trimmed, that is, to the power of the number of the external vertices.
\item We are left with the internal vertices and we start erasing 
the vertex (it is necessarily $\circ$ or $\bullet$ at this stage)
which is most on the right. So we sum over its index and use the bound: there exists a constant $C_L$ such that
for $(0\le )d\le n^\prime \le n \le f(\le k)$
we have 
\begin{multline}
\label{eq:ffR}
\sum_{j=n}^f R_{\frac 12}(j-n)R_{\frac 12}(j-n^\prime)R_{\frac 12}(f-j)\, 
\le \, \sum_{j=0}^{f-n} R_{\frac 12}(j)^2 R_{\frac 12}((f-n)-j)
\\ 
\le \, C_L \tilde L (f-n) R_{\frac 12}(f-n) \, \le \, C_L \tilde L (k) R_{\frac 12}(f-n),
\end{multline}
where in the first inequality we have used the monotonicity of $R_{\frac 12}(\cdot)$,
in the second we have explicitly estimated the sum by using standard results
on regularly varying function and \eqref{eq:Ltilde-prop}. The last inequality is just the monotonicity
of $\tilde L(\cdot)$. This means that this trimming step brings once  again a multiplicative factor  
$C_L \tilde L (k)$: of course this time we have trimmed three edges, but we have also the extra 
factor $R_{\frac 12}(f-n)$ which is precisely the contribution of a longer edge that we rebuild
(see Figure~\ref{fig:diagram-gen}).
\item Keep repeating the previous step (the type of the vertices is not really important), 
trimming each time three edges, but rebuilding one too (so, in total, minus two edges), till
the graph with four vertices and three edges.
\end{enumerate}

\medskip

\begin{figure}[hlt]
\begin{center}
\leavevmode
\epsfxsize =12.5 cm
\psfragscanon
\psfrag{f}[c][l]{\small $f$}
\psfrag{a1}[c][l]{\small $a_1$}
\psfrag{a2}[c][l]{\small $a_2$}
\psfrag{a3}[c][l]{\small $a_3$}
\epsfbox{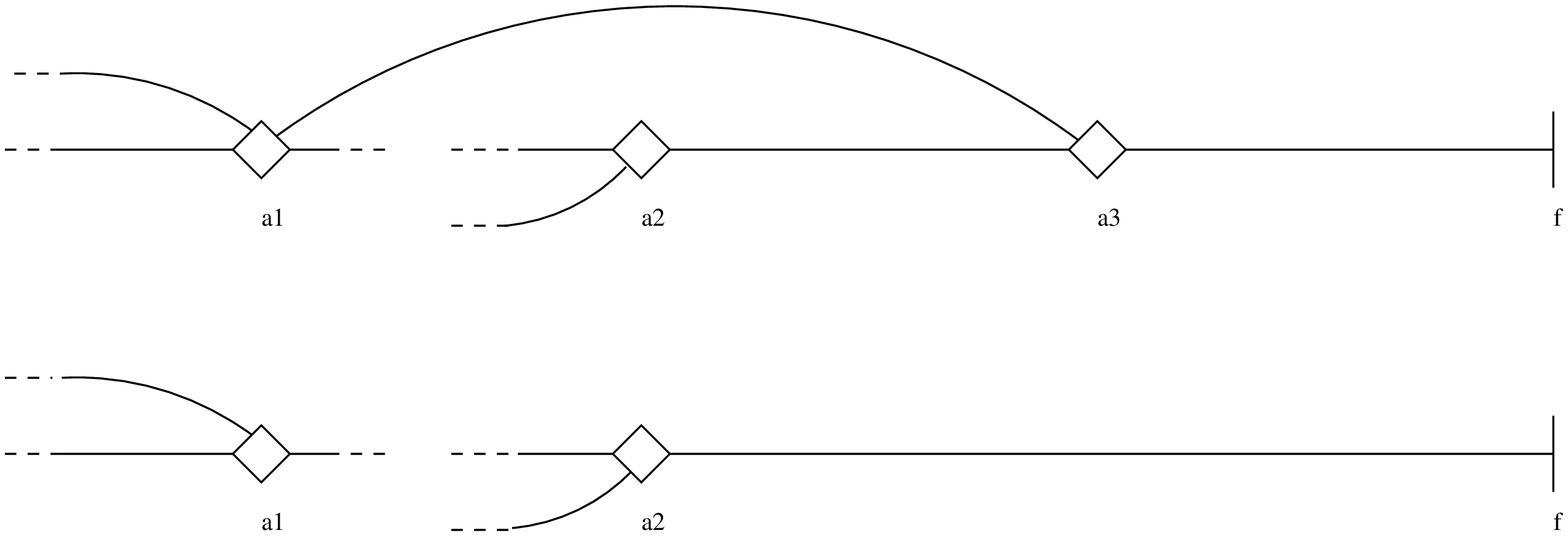}
\caption{\label{fig:diagram-gen} The second step of the trimming
  procedure corresponding to the estimate \eqref{eq:ffR}.  The symbol
  $\diamond$ may represent $\Box$, $\circ$ and $\bullet$: the choice
  is not fully arbitrary, in the sense that for example before
  starting the trimming procedure there is no edge between $f$ (or
  $d$) and a $\Box$.  However the estimate is independent of the type
  of symbols.}
\end{center}
\end{figure}

In order to evaluate the contribution of all the trimming procedure we
just need to count the number of vertices that we have erased: $q+\ell
-2$. We are now left with the contribution given by the last diagram
(four points, three edges: see for example
{\sl trim step 4} in Figure~\ref{fig:diagram-ex}), times of course $(C_L \tilde L(k))^{q+\ell
  -2}$: we bound  the last diagram using
\begin{equation}
\sum_{i=d}^f\sum_{j=i}^f R_{\frac 12}(i-d) R_{\frac 12}(j-i) R_{\frac 12}(f-j)
\, \le \, C_L \frac{\sqrt{f-d}}{L(f-d)^3}
\, \le \, C_{L, \gep} \frac{k\, R_{\frac 12}(f-d )}{L(k)^2}.
\end{equation}
where $C_L$ is once again a constant that depends only on $L(\cdot)$,
while in the last step we have used $k\ge f-d \ge \gep k$ and
\eqref{eq:Doney-bound}.  Going back to \eqref{eq:bfl2} we see that
there exists $C=C(\gep,q,L(\cdot))$ such that (with the convention of
Remark \ref{rem:nonumero})
\begin{multline}
\label{eq:varbfi}
\bE_{d,f} \hat \bbE_{\tau} \left[ \left( X-\hat \bbE_\tau X \right)^2
\right]\, \le \\ C\left(1 +\, \, \max_{\ell=1,2, \ldots, q-1}
  \frac{1}{k \tilde L(k)^{q-1} R_{\frac 12}(f-d)} \frac{\tilde
    L(k)^{q+\ell-2}k\, R_{\frac 12}(f-d
    )}{L(k)^2}\mbeta^{2\ell}\right)
\\
=\, C\left(1 + \, \max_{\ell=1,2, \ldots, q-1} \frac{\tilde L
    (k)^{\ell-1}}{L(k)^2} \mbeta^{2\ell}\right)\, \le \, C\left(1 \, +
  \, \max_{\ell=1,2, \ldots, q-1} \frac{\tilde L
    (k)^{\ell-1}}{L(k)^2} \gb^{2\ell}\,\right),
\end{multline}
where in the last line we have used $\mbeta \le 2 \gb$, for $\gb \le \gb_0$
({\sl cf.} \eqref{eq:2beused}). We now recall \eqref{eq:k333} that guarantees that
\begin{equation}
\label{eq:LtL}
\frac{\tilde L (k-1)}{L(k-1) ^{2/(q-1)}} \gb ^{2q/(q-1)} \, < \, A
\ \ \text{ so that } \ \ 
\frac{\tilde L (k)}{L(k) ^{2/(q-1)}} \gb ^{2q/(q-1)} \, \le  \, 2A,
\end{equation}
where the second inequality is a consequence of the slowly varying character
of $L(\cdot) $ and $\tilde L(\cdot)$ and it holds for $k$ sufficiently large.
But this implies
\begin{equation}
\label{eq:llb}
 \frac{\tilde L (k)^{\ell-1}}{L(k)^2}
 \gb^{2\ell}\, \le \, (2A)^{(q-1)\ell/q} \, \left( \tilde L(k) L(k)^2\right)^{-1+(\ell/q)}
 \, , 
\end{equation}  
so that, by \eqref{eq:Ltilde-prop}, 
by choosing $A$ large we can make  
the quantity in \eqref{eq:llb} arbitrarily small (recall that $\ell=1, \ldots, q-1$),
so that going back to \eqref{eq:varbfi}, we see that
\begin{multline}
\bE_{d,f} 
\hat \bbE_{\tau} \left[ \left(
X-\hat \bbE_\tau X \right)^2 \right]\, \le \\
C(\gep,q, L(\cdot))\left(1 + A^{(q-1)^2/q}  
 \max_{\ell=1, \ldots, q-1}\left( \tilde L(k) L(k)^2\right)^{-1+(\ell/q)}\right) \,
\le \, A^{(q-1)^2/q} \, , 
\end{multline}
where in the last step we have used that, by \eqref{eq:Ltilde-prop},
the maximum in the intermediate term can be made arbitrarily small,
by choosing $k$ large (that is, $A$ larger than a constant depending 
on $\gep$, $q$ and $L(\cdot)$).
 This completes the proof of Lemma~\ref{th:square}.
\qed

\section{Some probability estimates (Proof of Lemma~\ref{TH:FROMCE})}

The proof is done in four steps.

\medskip

\noindent
{\it Step 1: reduction to an asymptotic estimate on a constrained renewal.}
In this step we show  that it is sufficient to establish
 that for every $\zeta>0$
there exists $\varrho>0$ and $N_\zeta\in\N$ such that
\begin{equation}
\label{eq:suff1}
\bP \left(\frac{\bL(N)}{\tilde \bL (N)^{(q-1)/2}}
\sum_{\ui \in \{0, \ldots, N \}^q}   V_N(\ui) \gd_{\ui} 
\, \ge \, \varrho \bigg \vert \, N \in \tau \right)\, \ge \, 1-\zeta,
\end{equation}
for $N \ge N_\zeta$.

Notice in fact  that 
$\hat \bbE_\tau X= \mbeta^q \sum_{\ui} V_k(\ui) \gd_{\ui}$,
where $\ui \in \{d, \ldots, f\}^q$. 
Since $V_k(\ui)$ is invariant under the transformation $\ui=(i_1, \ldots, 
i_q) \mapsto (i_1+n, \ldots, 
i_q+n)$ (any $n \in \bbZ$), we may very well work on $\{0, \ldots, f-d\}$,
that is on an interval $\{0, \ldots, N\}$ ($\gep k \le N \le k$) and $\tau$ is a renewal
with $\tau_0=0$ and conditioned to $N \in \tau$. With this change of
variables, \eqref{eq:fromCE} reads
\begin{equation}
\label{eq:forCE1}
\bP \left(
\mbeta^q \sum_{\ui \in \{0, \ldots, N\}^q}
 V_k(\ui) \gd_{\ui} \, \ge \, a A^{(q-1)/2} \, \bigg \vert N \in \tau
\right) \,
\ge \, 1- \zeta. 
\end{equation}
Now two observations are in order:
\begin{itemize}
\item $V_k(\ui)/V_N(\ui)= (N/k)^{1/2} [\tilde \bL (N)/ \tilde \bL(k)]^{(q-1)/2}$ so that for
$k$ sufficiently large (that is for $A$ larger than a constant depending on
$\gep $ and $L(\cdot)$) we have
\begin{equation}
\frac{V_k(\ui)}{V_N(\ui)} \, \ge \, \frac{\gep^{1/2}}2.
\end{equation} 
\item By \eqref{eq:2beused},
 \eqref{eq:k333} and \eqref{eq:Lb} we see that 
 \begin{equation}
\mbeta^q \, \ge \, 2^{-q} A^{(q-1)/2} \, \frac{L(k-1)}{\tilde L(k-1)^{(q-1)/2}}
\, \ge \,  2^{-q} \const _L A^{(q-1)/2} \, \frac{\bL(N)}{\tilde \bL(N)^{(q-1)/2}}
\, .
\end{equation}
\end{itemize}

These two observations show that for $A$ sufficiently large 
\eqref{eq:forCE1} is implied by
\begin{equation}
\label{eq:for CE2}
\bP \left(
 \frac{\bL(N)}{\tilde \bL(N)^{(q-1)/2}} 
 \sum_{\ui \in \{0, \ldots, N\}^q}
 V_N(\ui) \gd_{\ui} \, \ge\, \frac{2^{q}}{\const_L \gep^{1/2}} a
\bigg\vert \, N \in \tau 
\right) \, \ge \, 1-\zeta\, .
\end{equation}
Therefore, at least if 
$A$
is  larger than a suitable constant depending on  $\gep$ and $L(\cdot)$,
it is  sufficient to
prove \eqref{eq:suff1}.

\medskip

\noindent
{\it Step 2: removing the constraint.}
In this step we claim that there exists a positive constant $c$,
that depends only on $L(\cdot)$, such that
if 
\begin{equation}
\label{eq:suff2}
\bP \left(\frac{\bL(N)}{\tilde \bL (N)^{(q-1)/2}}
\sum_{\ui \in \{1, \ldots, \lfloor N/2\rfloor \}^q}   V_N(\ui) \gd_{\ui} 
\, \ge \, \varrho \, \right)\, \ge \, 1- c \zeta,
\end{equation}
then \eqref{eq:suff1} holds. Note first of all that the 
random variable that we are estimating is smaller (since 
$V_N(\cdot) \ge 0$)
than the random variable  in \eqref{eq:suff1}, for every given $\tau$-trajectory.
It is therefore sufficient to bound the 
Radon-Nykodym derivative of the law of $\tau\cap [0, \lfloor N/2
\rfloor]$ without constraint
$N \in \tau$ with respect to the law of the same random set with the constraint.
Such an estimate can be found for example in 
\cite[Lemma A.2]{cf:GLT_marg}. 

\medskip

\noindent{\it Step 3: reduction to a convergence in law statement.}
For $\rho:= 1/(2(q-1))$ we define the subset $S_\rho (N)$ of 
$\sort(\{ 0,1, \ldots, N \}^q)$ (recall that the latter is the set of 
increasingly rearranged $\ui$ vectors) such that 
$i_j \le N ((j-1)\rho +(1/2))$ for $j=1, 2, \ldots, q$. 

The claim of this step is that \eqref{eq:suff2} follows if 
\begin{equation}
\label{eq:suff3}
\eta_N \, :=\, \frac{\bL (N)}{\tilde \bL (N)^{(q-1)/2}}
\sum_{\ui \in S_\rho (N)} V_N(\ui) \gd_{\ui} \stackrel{N \to \infty}\Longrightarrow 
\eta _\infty \ \ \ \text{ with } \ \eta_\infty\, >\, 0 \text{ a.s.}\, ,
\end{equation}
where $\Longrightarrow$ denotes convergence in law. 
 
In order to see why \eqref{eq:suff3} implies \eqref{eq:suff2} it
suffices to observe that replacing $N$ with $\lfloor N/2\rfloor$ in
\eqref{eq:suff2} (except when it already appears as $\lfloor
N/2\rfloor$) introduces an error that can be bounded by a
multiplicative constant (say, 2) for $N$ sufficiently large, so that
it suffices to show that $\bP( \eta_N \ge 2 \varrho) \ge 1-c \zeta$.
But \eqref{eq:suff3} yields $\lim_N\bP( \eta_N \ge 2 \varrho) \ge \bP(
\eta_\infty \ge 3 \varrho)$.  At this point if we choose
$\varrho:=\varrho(\zeta)$ such that $ \bP( \eta_\infty \ge 3 \varrho)=
1-(c\zeta/2)$, we are assured that for $N$ sufficiently large (how
large depends on $\zeta$) $\bP( \eta_N \ge 2 \varrho) \ge 1-c \zeta$
and we are reduced to proving \eqref{eq:suff3}.

\medskip

\noindent
{\it Step 4: proof of the convergence in law statement \eqref{eq:suff3}}.
This step depends on the following  lemma, that we prove just below:

\medskip

\begin{lemma}
\label{th:CE23}
For every $\theta_0 \in (0,1)$ we have
\begin{equation}
  \lim_{N \to \infty} \sup_{\theta \in [\theta_0,1]}
  \bE\left[ \left (
      \frac 1{\tilde \bL (N)} \sum_{j=1}^{\lfloor \theta N\rfloor} 
R_{1/2}(j) {\gd_j} \, - \, \frac{\const}{2\pi}
      \right )^2 \right]\, =\, 0\, ,
\end{equation}
with $\const := \lim_{x \to \infty} \bL(x)/ L(x)(\in [1,\const _L^{-1} ])$.
\end{lemma}
\medskip

For $p=1,2, \ldots, q$ we introduce the random variables
\begin{multline}
\eta_{N, p}\, :=\\
\left(\frac{2\pi}{\const}\right)^{p-q} \frac{\bL(N)}{N^{1/2}\tilde \bL (N)^{p-1}}
\sum_{i_1=0}^{ [N/2 ]}\sum_{i_2 =i_1+1}^{[(\rho+(1/2))N] }\ldots 
\sum_{i_p=i_{p-1}+1}^{[((p-1)\rho+(1/2))N] } \gd_{i_1}
\prod_{r=2}^p R_{1/2}\left(i_r-i_{r-1}\right) \gd_{i_r}\, ,
\end{multline}
where the product in the right-hand side has to be read as $1$ if $p=1$ and, in this case, there is only the sum over $i_1$. 
First of all remark that
$\eta_{N, q}= \sqrt{q!}\eta_N$ (recall \eqref{eq:useR2})
 and that $\eta_{N, p-1}$ is obtained 
from $\eta_{N,p}$ by removing the last term in the product, the corresponding
sum and  by multiplying by  $ 2\pi \tilde \bL(N)/\const$.
We now claim that Lemma~\ref{th:CE23} implies that for $p=2, 3, \ldots, q$
\begin{equation}
\label{eq:L1andind}
\lim_{N \to \infty} \bE \left[ \left \vert \eta_{N,p}- \eta_{N, p-1}\right\vert \right] \, =\, 0\, , 
\end{equation}
which clearly reduces the problem of proving
$\eta_N \Longrightarrow \eta_\infty$ to
proving $\eta_{N,1} \Longrightarrow \eta_\infty$, and  $\eta_\infty$ has to be 
a positive random variable.
But in fact we have
\begin{equation}
\label{eq:convlaw}
(2\pi /\const)^{q-1} \frac{L(N)}{\bL(N)}\, \eta_{N,1}\, =\, 
\frac{L(N)}{\sqrt{N}}
\sum_{i=0}^{\lfloor N/2\rfloor} \gd_i 
\stackrel{N \to \infty}{\Longrightarrow}
\frac{1}{2\sqrt{\pi}} \vert Z \vert  \ \ \ \ \ \ \ \ ( Z \sim \cN(0,1)).
\end{equation}
The convergence in \eqref{eq:convlaw} is a standard result that
we outline briefly. First of all for every choice of $n, m \in \N$ we
have
\begin{equation}
\label{eq:deltatau}
\left\{ \sum_{i=1}^n \gd_i \, < \, m \right\} \, =\, \left\{ \tau_m > n \right\} ,
\end{equation} 
so that the asymptotic law of the {\sl normalized local time}
of $\tau$ up to $n$, {\sl i.e.} $L(n)n^{-1/2}\sum_{i=1}^n \gd_i$, is directly linked
to the domain of attraction of the random variable $\tau_1$. 
Explicitly, one directly verifies that for $\gl>0$
\begin{equation}
\bE \left[ \left( 1- \exp(-\gl \tau_1)\right)\right] \stackrel{\gl \searrow 0} \sim
2\sqrt{\pi}L(1/\gl) \sqrt{\gl},
\end{equation}
so that, if $a(\cdot)$ is the asymptotic inverse of the regularly varying function $r(\cdot)$, defined by
$r(x):=\sqrt{x}/L(x)$ for  $x>0$, that is $a(r(x))\sim r(a(x))\sim x$
for $x \to \infty$, we have 
\begin{equation}
\lim_{N \to \infty}
\bE \left[ \exp \left( -\gl \tau_N /a(N) \right) \right] \, =\, \exp\left(-2 \sqrt{\pi \gl}\right) \, =\, 
\bE \left[ \exp(-\gl Y)\right],
\end{equation}
where $Y$ is a positive random variable with density $f_Y(y)$ equal to
$y^{-3/2} \exp(-\pi/y)$ (for $y>0$).  On the other hand for $t >0$
by \eqref{eq:deltatau} we have
\begin{equation}
\bP\left( \frac{L(n)}{\sqrt{n}}
\sum_{j=1}^n \gd_j <t\right) \stackrel{n \to \infty}\sim
\bP\left( \tau_{\lfloor t \sqrt{n}/L(n)\rfloor} >n \right).
\end{equation}
Therefore
if we observe that $a(t \sqrt{n}/L(n)) \sim t^2 a( \sqrt{n}/L(n))
\sim t^2 n$, for $n \to \infty$, we directly obtain that
\begin{equation}
\lim_{n \to \infty}
\bP\left( \frac{L(n)}{\sqrt{n}}
\sum_{j=1}^n \gd_j <t\right) \stackrel{n \to \infty}\sim
\bP\left(Y\, >\, \frac 1{t^2} \right).
\end{equation}
By using the (explicit) density  of $Y$, 
one directly verifies that $\bP (Y >  1/t^2 )$ coincides with
$\bP ( \vert Z\vert/\sqrt{2\pi} < t)$ for every $t>0$, that is  
 \eqref{eq:convlaw} is established (recall that in \eqref{eq:convlaw}
 the summation is up to $N/2$).
\medskip

We are therefore left with
proving   \eqref{eq:L1andind}. This  follows by
observing that for $p =3,4, \ldots,q$
\begin{multline}
\label{eq:forL1andind}
\bE \left[ \left \vert \eta_{N,p}- \eta_{N, p-1}\right\vert \right] \, \le \,
\left(\frac{2\pi}{\const}\right)^{p-q} \frac{\bL(N)}{N^{1/2}\tilde \bL
  (N)^{p-2}} \times \\
   \sum_{i_1=0}^{ \lfloor N/2 \rfloor}\sum_{i_2
  =i_1+1}^{\lfloor(\rho+(1/2))N\rfloor }\ldots
\sum_{i_{p-1}=i_{p-2}+1}^{\lfloor ((p-2)\rho+(1/2))N\rfloor } \bE \left[ \gd_{i_1}
  \prod_{r=2}^{p-1} R_{1/2}\left(i_r-i_{r-1}\right) \gd_{i_r}\right]\,
\times
\\
\bE \left[ \left\vert \frac 1{ \tilde\bL (N)}
    \sum_{i_{p}=i_{p-1}+1}^{\lfloor((p-1)\rho+(1/2))N\rfloor }
    R_{1/2}\left(i_p-i_{p-1}\right) \gd_{i_p} \, - \frac
    {\const}{2\pi} \right \vert \, \bigg \vert
  \gd_{i_{p-1}}=1\right]\, ,
\end{multline}
and the same expression holds if $p=2$ but in this case
the external summation is only over $i_1$
and $\prod_{r=2}^{p-1} R_{1/2}\left(i_r-i_{r-1}\right) \gd_{i_r}$ is replaced by $1$.
The bound \eqref{eq:forL1andind} follows from the triangular inequality
and from the renewal property of $\tau$.
Next, note that
\begin{multline}
\label{eq:forL1andind2}
\bE \left[ \left\vert \frac 1{ \tilde \bL (N)}
    \sum_{i_{p}=i_{p-1}+1}^{\lfloor((p-1)\rho+(1/2))N\rfloor }
    R_{1/2}\left(i_p-i_{p-1}\right) \gd_{i_p} \, - \frac
    {\const}{2\pi} \right \vert \, \bigg \vert
  \gd_{i_{p-1}}=1\right]\, =
\\
\bE \left[ \left\vert \frac 1{ \tilde \bL (N)}
    \sum_{i=1}^{\lfloor((p-1)\rho+(1/2))N\rfloor -i_{p-1}} R_{1/2}\left(i\right)
    \gd_{i} \, - \frac {\const}{2\pi} \right \vert \right] \stackrel{N
  \to \infty} \longrightarrow 0\, ,
\end{multline}
uniformly in the choice of $i_{p-1}\in \{ i_{p-2}+1, \ldots, \lfloor
((p-1)\rho+(1/2)N\rfloor \}$.  This is because the summation in
\eqref{eq:forL1andind2} contains at least $[\rho N]$ terms (and no
more than $N$) so that we can apply Lemma~\ref{th:CE23}.  The fact that
$\bE \left[ \left \vert \eta_{N,p}- \eta_{N, p-1}\right\vert
\right]=o(1)$ as $N \to \infty$ is therefore a consequence of the
following explicit estimate:
\begin{multline}
  \frac{\bL(N)}{N^{1/2}\tilde \bL (N)^{p-2}} \sum_{i_1=0}^{ \lfloor
    N/2 \rfloor}\sum_{i_2 =i_1+1}^{\lfloor(\rho+(1/2))N\rfloor }\ldots
  \sum_{i_{p-1}=i_{p-2}+1}^{\lfloor ((p-2)\rho+(1/2))N\rfloor } \bE
  \left[ \gd_{i_1} \prod_{r=2}^{p-1} R_{1/2}\left(i_r-i_{r-1}\right)
    \gd_{i_r}\right]\, \le
  \\
  \frac{\bL(N) \const _L ^{-(p-1)}}{N^{1/2}\tilde \bL (N)^{p-2}}
  \sum_{i_1=0}^{ \lfloor N/2 \rfloor}\sum_{i_2
    =i_1+1}^{\lfloor(\rho+(1/2))N\rfloor }\ldots
  \sum_{i_{p-1}=i_{p-2}+1}^{\lfloor((p-2)\rho+(1/2))N\rfloor } R_{1/2}(i_1)
  \prod_{r=2}^{p-1} \left(R_{1/2}\left(i_r-i_{r-1}\right)\right)^2 \\
 \stackrel{N \to \infty}\sim \sqrt2 \const _L ^{-(p-1)}\, ,
\end{multline}
where we have used the definition \eqref{eq:Ltilde} of the slowly varying
function $\tilde \bL(\cdot)$ and the fact that $\int_0^x (y^{1/2}\bL(y))^{-1}
\dd y \stackrel{x \to \infty}\sim 2 x^{1/2}/\bL (x)$.
This completes the proof of Lemma~\ref{TH:FROMCE}.
\qed

\medskip

\noindent
{\it Proof of Lemma~\ref{th:CE23}.}
This is very similar to the proof of Lemma~5.4 in \cite{cf:GLT_marg}
(that, in turn generalizes a result of K.~L.~Chung and P.~Erd\"os
\cite{cf:chungerdos}). We give it in detail in order to clarify the role of the
slowly varying function.

First of all let us remark that 
\begin{equation}
\frac 1{\tilde \bL (N)} \sum_{j=1}^{[\theta N]} R_{1/2}(j) \bE\left[\gd_j\right] \stackrel{N\to \infty}\sim
 \frac{\const \tilde \bL (\theta N)}{2\pi \tilde \bL (N)}  \stackrel{N\to \infty}\sim 
 \frac{\const}{2\pi }\, , 
\end{equation}
where the last asymptotic relation holds uniformly in $\theta$, when $\theta$ lies  in a compact subinterval of $(0, \infty)$. 
The statement is therefore reduced to showing that
the variance of
\begin{equation}
Y_n\, :=\, \sum_{j=1}^{n} R_{1/2}(j) {\gd_j},
\end{equation}
is $o(\tilde L(n)^2)$. 

Let us compute and start by observing that
\begin{equation}
\begin{split}
\text{var}_\bP \left( Y_n \right) \, &=\, \sum_{i,j=1}^n R_{1/2}(i)R_{1/2}(j)\left[
\bE\left[ \gd_i \gd_j\right] - \bE\left[ \gd_i \right]   \bE\left[\gd_j\right]
\right]
\\ 
&=\, 
2\sum_{i=1}^{n-1} \sum_{j=i+1}^n 
R_{1/2}(i)R_{1/2}(j)\left[
\bE\left[ \gd_i \gd_j\right] - \bE\left[ \gd_i \right]   \bE\left[\gd_j\right]
\right]\, 
+\, O(\tilde L(n)) \, \\
&=: \, 2 T_n + O(\tilde L(n)),
\end{split}
\end{equation}
and 
\begin{equation}
\begin{split}
T_n \, &=\, \sum_{i=1}^{n-1} R_{1/2}(i){ \bE\left[ \gd_i \right] }
\left[
\sum_{j=1}^{n-i} R_{1/2}(i+j){ \bE\left[ \gd_j \right] } -
\sum_{j=i+1}^{n} R_{1/2}(j){ \bE\left[ \gd_j \right] } 
\right]
\\
& \le \, 
 \sum_{i=1}^{n-1} R_{1/2}(i){ \bE\left[ \gd_i \right] }
\left[
\sum_{j=1}^{n-i} R_{1/2}(i+j){ \bE\left[ \gd_j \right] } -
\sum_{j=i+1}^{n} R_{1/2}(i+j){ \bE\left[ \gd_j \right] } 
\right]
\\
& \le \, 
 \sum_{i=1}^{n-1} R_{1/2}(i){ \bE\left[ \gd_i \right] }
\sum_{j=1}^{i} R_{1/2}(i+j){ \bE\left[ \gd_j \right] } 
\, \le \, 
 \sum_{i=1}^{n-1} \left(R_{1/2}(i)\right)^2{ \bE\left[ \gd_i \right] }
\sum_{j=1}^{i} { \bE\left[ \gd_j \right] } 
\\ 
& \phantom{movemovemove} \le \, \const _L ^{-2}
 \sum_{i=1}^{n-1} \left(R_{1/2}(i)\right)^3
 \sum_{j=1}^{i} R_{1/2}(j)\stackrel{n \to \infty}\sim
 2\const _L ^{-2} \int_{0}^n \frac1{(1+x) (\bL (x))^4} \dd x\, , 
\end{split}
\end{equation}
where the first three inequalities follow since $R_{1/2}(\cdot)$ is non increasing and the fourth follows from \eqref{eq:Doney-bound}.
The conclusion of the proof follows now from Remark~\ref{rem:auxL}.
\qed

\medskip

\begin{rem}
\label{rem:auxL}
\rm
For $x \to \infty$
\begin{equation}
\int_0^x \frac1{(1+y) (\bL(y))^4} \dd y \, \ll \, \left( \tilde \bL (x) \right)^2,
\end{equation}
with $ \tilde \bL (x)$ defined as in \eqref{eq:Ltilde} with $L(\cdot)$
replaced by $\bL(\cdot)$.
This is a consequence of \eqref{eq:Ltilde-prop} (which of course holds also for $\bL(\cdot)$):
\begin{equation}
\int_0^x \frac1{(1+y) (\bL(y))^4} \dd y \, \ll \, 
\int_0^x \frac1{(1+y) (\bL(y))^2} \tilde \bL (y) \dd y\, \le \, \tilde \bL (x)
\int_0^x \frac1{(1+y) (\bL(y))^2} \dd y \, ,
\end{equation}
and the rightmost term is $\left(  \tilde \bL (x)\right)^2$.
\end{rem}

\section{A general monotonicity result}
\label{sec:monotonicity}

We present now a very general result: we give it in our context but a look at the proof suffices
to see 
that it   holds also under substantially  milder assumptions 
 on the process $\tau$.
\medskip

\begin{proposition}\label{th:monotonicity}
The free energy $\tf(\gb,h)$ is a non-increasing function of $\gb$ on $[0,\infty)$.
Therefore
\begin{itemize}
 \item[(i)] $\gb\mapsto h_c(\gb)$ is a non-decreasing function of $\gb$.
 \item[(ii)] There exists a critical value $\gb_c\in[0,\infty]$ such that $h_c(0)=h_c(\gb)$ if and only if $\gb\le \gb_c$.
\end{itemize}
\end{proposition}
\medskip

This result is of particular relevance when $\sum_n 1/(n L(n)^2) <\infty$,
that is when for small $\gb$ we have $h_c(\gb)= h_c(0)$
({\sl cf.} \S~\ref{sec:review-rs}): in this case $\gb_c$
is the transition point from the  irrelevant disorder regime to the  relevant one. 
But also in our set-up, in which 
$\sum_n 1/(n L(n)^2) =\infty$, it is  of some use since it implies 
that it is sufficient to prove Theorem~\ref{th:main333} for one value of $\gb_0>0$
and the statement holds also for any other value of $\gb_0$
(by accepting, of course,   a worse estimate on the shift of the critical point if one follows the estimates quantitatively, see Remark~\ref{rem:lazy}). 

\begin{proof}
We just need to prove that $\gb\mapsto\tf(\gb,h)$ is a non-increasing function on $[0,\infty)$ as the other points are trivial consequence of this result.
To do so, we prove that $\gb\mapsto \bbE[\log Z_{N,\go}]$ is a non-increasing function of $\gb$, and pass to the limit. 
The proof is the adaptation of an argument  used  in \cite{cf:CY} for directed polymers with bulk disorder  to prove a similar result.

What we will show is 
\begin{equation}
\label{eq:fCY}
 \frac{\partial}{\partial \gb} \bbE\left[  \log Z_{N,\go} \right] \,=\, \bbE\left[ \frac{\partial}{\partial \gb} \log Z_{N,\go} \right]\,\le\, 0.
\end{equation}
The proof of the equality
in \eqref{eq:fCY} is standard and can be easily adapted from \cite[Lemma 3.3]{cf:CY}.
Recall now that $\mbeta:=M'(\gb)/M(\gb)$. We have
\begin{equation}\begin{split}
\bbE\left[ \frac{\partial}{\partial \gb}\log Z_{N,\go}\right]&=\bE \left[\bbE\left[ \frac{1}{Z_{N,\go}}\sum_{n=1}^N(\go_n-\mbeta)\gd_n\exp\left(\sum_{n=1}^N  [\gb\go_n+h-\log M(\gb)]\gd_n\right)\gd_N\right]\right]\\
&=\bE\left[\exp\left(\sum_{n=1}^N h \gd_n\right)\gd_N \hat\bbE_{\tau}\left[Z_{N,\go}^{-1}\sum_{n=1}^N(\go_n-\mbeta)\gd_n\right]\right].
\end{split}\end{equation}
For a fixed trajectory of the renewal, the  probability measure $\hat \bbP_{\tau}$ (recall definition \eqref{eq:ptau}), is a product measure, so that, since
 $Z_{N,\go}^{-1}$ is a decreasing function of $\go$ and 
 $\sum_{n=1}^N(\go_n-\mbeta)\gd_n$ is 
a non-decreasing  function of $\go$, by the Harris--FKG inequality we have
\begin{equation}
\hat\bbE_{\tau}\left[Z_{N,\go}^{-1}\sum_{n=1}^N(\go_n-\mbeta)\gd_n\right]\le \hat\bbE_{\tau}\left[Z_{N,\go}^{-1}\right]\hat\bbE_{\tau}\left[\sum_{n=1}^N(\go_n-\mbeta)\gd_n\right]=0.
\end{equation}

\end{proof}

\part{Polymères dirigés en milieu aléatoire}

\chapter{Directed polymer on hierarchical lattices with site disorder}\label{DPOLHLSD}
\section{Introduction and presentation of the model}

The model of directed polymers in random environment appeared first in the physics literature as
an attempt to modelize roughening in domain wall in the $2D$-Ising model due to impurities \cite{cf:HH}.
It then reached the mathematical community in \cite{cf:IS}, and in \cite{cf:B}, where the author applied
martingale techniques that have became the major technical tools in the study of this model since then.
A lot of progress has been made recently in the mathematical understanding of directed polymer model (see for example \cite{cf:J,cf:CY, cf:CSY, cf:CY_cmp, cf:CH_al, cf:CH_ptrf, cf:CV} and \cite{cf:CSY_rev} for a recent review). 
It is known that there is a phase transition from a delocalized phase at high temperature, where the behavior
of the polymer is diffusive, to a localized phase, where it is expected that the influence of the media is relevant in order to produce nontrivial phenomenons, such as super-diffusivity. These two different situations are
usually referred to as weak and strong disorder, respectively. A simple characterization of this dichotomy is
given in terms of the limit of a certain positive martingale related to the partition function
of this model.

It is known that in low dimensions ($d=1$ or $2$), the polymer is essentially
in the strong disorder phase (see \cite{cf:Lac}, for more precise results), but for $d\geq 3$, there
is a nontrivial region of temperatures where weak disorder holds. A weak form of invariance principle is
proved in \cite{cf:CY}.

However, the exact value of the critical temperature which separates the two regions (when it is finite) remains an open question.
It is known exactly in the case of directed polymers on the tree, where a complete analysis is available
(see \cite{cf:BPP, cf:Fr, cf:KP}). In the case of $\Z^d$, for $d\ge 3$, an $L^2$ computation yields an upper bound on the critical temperature, which is however known not to coincide with this bound (see \cite{cf:BS, cf:Birk} and \cite{cf:CC}). 

We choose to study the same model of directed polymers on diamond hierarchical lattices. These lattices present
a very simple structure allowing to perform a lot of computations together with a richer geometry
than the tree (see Remark \ref{convexity-tree} for more details).   They have been introduced in physics in order to perform exact renormalization group computations for spin systems (\cite{cf:Migd, cf:Kad}). A detailed treatment of more general hierarchical lattices can be found
in \cite{cf:KG1} and \cite{cf:KG2}. For an overview of the extensive literature on Ising and Potts models on hierarchical lattices, we refer the reader to \cite{cf:BZ, cf:DF} and references therein. Whereas statistical mechanics model on trees have to be considered as mean-field versions of the original models, the hierarchical lattice models are in many sense very close to the models on $\Z^d$; they are a very powerful tool to get an intuition for results and proofs on the more complex $\Z^d$ models (for instance, the work on hierarchical pinning model in \cite{cf:GLT} lead to a solution of the original model in \cite{cf:DGLT}. In the same manner, the present work has been a great source of inspiration for \cite{cf:Lac}).

Directed polymers on hierarchical lattices (with bond disorder) appeared in \cite{cf:CD, cf:DG,cf:DGr, cf:DHV}
(see also \cite{cf:R_al} for directed first-passage percolation).
More recently, these lattice models raised the interest of mathematicians in the study of random resistor  networks (\cite{cf:W}), pinning/wetting transitions (\cite{cf:GLT,cf:Hubert}) and diffusion on a percolation cluster (\cite{cf:HK}).

We can also mention \cite{cf:Ham} where the authors consider a random analogue of the hierarchical lattice, where
at each step, each bond transforms either into a series of two bonds or into two bonds in parallel, with probability
$p$ and $p-1$ respectively. 

\vspace{2ex}

Our aim in this paper is to describe the properties of the quenched free energy of directed polymers on
hierarchical lattices with site disorder at high temperature:
\begin{itemize}
	\item First, to be able to decide, in all cases, if the quenched and annealed free energy differ at low temperature.
	\item If they do, we want to be able to describe the phase transition and to compute the critical exponent.
\end{itemize}

We choose to focus on the model with site disorder, whereas \cite{cf:GM_h, cf:CD} focus on the model with \textsl{bond disorder} where computations are simpler.
We do so because we believe that this model is closer to the model of directed polymer in $\Z^{d}$ (in particular, because of the inhomogeneity of the Green Function), and because there exists a nice recursive construction of the partition functions in our case, that leads to a martingale property. Apart from that, both models are very similar, and we will shortly talk about the bound disorder model in section \ref{bddis}.

The diamond hierarchical lattice $D_n$ can be constructed recursively:

\begin{itemize}
	\item $D_0$ is one single edge linking two vertices $A$ and $B$.
	\item $D_{n+1}$ is obtained from $D_n$ by replacing each edges by $b$ branches of $s-1$ edges.
\end{itemize}

\begin{figure}[h]
\begin{center}
\leavevmode
\epsfysize =6.5 cm
\psfragscanon
\psfrag{a}[c]{\normalsize A}
\psfrag{b}[c]{B}
\psfrag{l0}[c]{$D_0$}
\psfrag{l1}[c]{$D_1$}
\psfrag{l2}[c]{$D_2$}
\psfrag{traject}[l]{{\tiny (a directed path on $D_2$)}}
\epsfbox{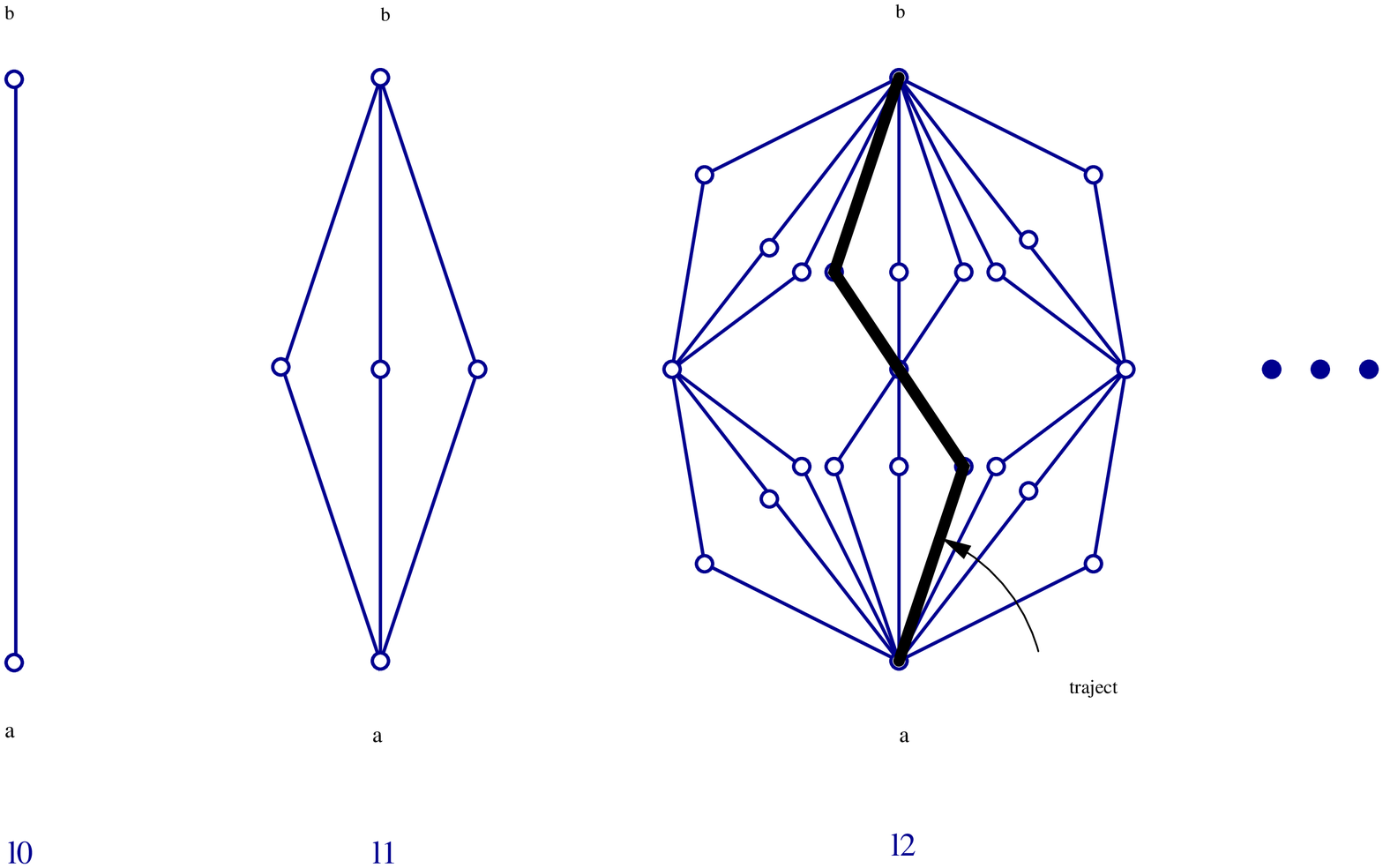}
\end{center}
\caption{\label{fig:Dn2} We present here the recursive construction of the first three levels of the hierarchical lattice  $D_n$, for $b=3$, $s=2$.}
\end{figure}
\noindent We can, improperly, consider $D_n$ as a set of vertices, and, with the above construction, we have $D_n\subset D_{n+1}$. We set $D=\bigcup_{n\ge 0} D_n$.
The vertices introduced at the $n$-th iteration are said to belong to the $n$-th generation $V_n=D_n\setminus D_{n-1}$.
We easily see that $|V_n|= (bs)^{n-1} b (s-1)$.

\noindent We restrict to $b\geq 2$ and $s\geq 2$. The case $b=1$ (resp. $s=1$) is not interesting as it just corresponds to a familly of edges
in serie (resp. in parallel)

 We introduce disorder in the system as a set of real numbers associated to vertices $\go=(\go_z)_{z\in D\setminus{\{A,B\}}}$.
Consider $\Gamma_n$ the space of directed paths in $D_n$ linking $A$ to $B$. For each
$g \in \Gamma_n$ (to be understood as a sequence of connected vertices in $D_n$, $(g_0=A, g_1, \dots, g_{s^n}=B)$), 
we define the Hamiltonian

\begin{eqnarray}
	H^{\omega}_n(g) := \sum^{s^{n}-1}_{t=1} \omega(g_t).
	\label{eq: Hn}
\end{eqnarray}

\noindent For $\beta > 0$, $n\geq 1$, we define the (quenched) polymer measure on $\gG_n$ which chooses a path $\gga$ at random  with law

\begin{eqnarray}
	\mu^{\omega}_{\beta, n}(\gamma=g) := \frac{1}{Z_{n}(\beta)} \exp( \beta H^{\omega}_n(g)),
	\label{eq:polymer}
\end{eqnarray}

\noindent where

\begin{eqnarray}
	Z_{n}(\beta)= Z_{n}(\beta, \omega):=\sum_{g \in \Gamma_n} \exp (\beta H^{\omega}_n(g)),
	\label{eq:Zn}
\end{eqnarray}

\noindent is the partition function, and $\gb$ is the inverse temperature parameter.

 In the sequel, we will focus on the case where $\go=( \omega_z,\, z\in D\setminus\{A,B\} )$ is a collection of i.i.d.\ random variables and denote the product measure by $Q$. Let $\go_0$ denote a one dimensional marginal of $Q$, we assume that $\go_0$ has expectation zero, unit variance, and that

\begin{eqnarray}
	\lambda(\beta):= \log Q e^{\beta \go_0}<\infty \quad \forall \gb>0.
	\label{eq:lambda}
\end{eqnarray}
As usual, we define the quenched free energy (see Theorem \ref{thermo}) by

\begin{eqnarray}
	p(\beta) := \lim_{n\to +\infty} \frac{1}{s^{n}} Q \log Z_n(\beta),
	\label{eq:quenched}
\end{eqnarray}

\noindent and its annealed counterpart by

\begin{eqnarray}
	f(\beta) := \lim_{n\to +\infty} \frac{1}{s^{n}} \log Q Z_n(\beta).
	\label{eq:annealed}
\end{eqnarray}

\noindent This annealed free energy can be exactly computed. We will prove

\begin{eqnarray}
	f(\beta):= \lambda(\beta)+ \frac{\log b}{s-1}.
	\label{eq:annealedexact}
\end{eqnarray}

\vspace{3ex}

This model can also be stated as a random dynamical system: given two integer parameters $b$ and $s$ larger than $2$, $\gb>0$, consider the following recursion:

\begin{align}
W_0&\stackrel{\mathcal L}{=}1\notag\\
W_{n+1}&\stackrel{\mathcal L}{=}\frac{1}{b}\sum_{i=1}^b \prod_{j=1}^{s}W_n^{(i,j)}\prod_{i=1}^{s-1}A^{(i,j)}_n\label{eq:recab},
\end{align}

\noindent where equalities hold in distribution, $W_n^{(i,j)}$ are independent copies of $W_n$, and $A^{(i,j)}_n$ are i.i.d.\ random variables, independent of the $W_n^{(i,j)}$ with law

\begin{align*}
A\stackrel{\mathcal L}{=}\exp(\gb\go-\gl(\gb)).
\end{align*}

\noindent In the directed polymer setting, $W_n$ can be interpretative as the normalized partition function

\begin{eqnarray}
	W_n(\beta)=W_n(\beta,\omega) = \frac{Z_n(\beta,\omega)}{Q Z_n(\beta, \omega)}.
	\label{eq:W}
\end{eqnarray}

\noindent Then, (\ref{eq:recab}) turns out to be an almost sure equality if we interpret $W_n^{(i,j)}$ as the partition function of the 
$j$-th edge of the $i$-th branch of $D_1$.

\noindent The sequence $(W_n)_{n\ge 0}$ is a martingale with respect to $\mathcal{F}_n = \sigma(\omega_z:\, z\in \cup^n_{i=1}V_i)$ and
as $W_n>0$ for all $n$, we can define the almost sure limit
$W_{\infty}= \lim_{n\to +\infty} W_n$. Taking limits in both sides of (\ref{eq:recab}), we obtain a functional equation for $W_{\infty}$.




\section{Results}

Our first result is about the existence of the free energy.

\begin{theorem}\label{thermo}
For all $\beta$, the limit

\begin{eqnarray}
	\lim_{n\to +\infty} \frac{1}{s^n} \log Z_n(\beta),
	\label{eq:free}
\end{eqnarray}

\noindent exists a.s. and is a.s. equal to the quenched free
energy $p(\beta)$. In fact for any $\gep>0$, one can find $n_0(\gep,\gb)$ such that
\begin{eqnarray}\label{eq:concentrate}
  Q\left(\left|Z_n-Q\log Z_n\right|> s^n \gep\right)\le \exp\left(-\frac{\gep^{2/3}s^{n/3}}{4}\right), \quad \text{for all } n\ge n_0
\end{eqnarray}

Moreover, $p(\cdot)$ is a strictly convex function of $\beta$.
\end{theorem}

\vspace{4ex}

\begin{rem}\rm
The inequality (\ref{eq:concentrate}) is the exact equivalent of \cite[Proposition 2.5]{cf:CSY}, 
and the proof given there can easily be adapted to our case. It applies concentration results for
martingales from \cite{cf:LV}.
It can be improved in order to obtain the same bound as for Gaussian
environments stated in \cite{cf:CH_ptrf} (see \cite{cf:Cnotes} for details). However, it is believed that it is
not of the optimal order, similar to the case of directed polymers on $\Z^d$.
\end{rem}

\begin{rem}\rm \label{convexity-tree}
The strict convexity of the free energy is an interesting property. It is known that it holds also for the directed polymer on $\Z^d$ but not on the tree. In the 
later case, the free energy is strictly convex only for values of $\beta$ smaller than the critical value $\beta_c$ 
(to be defined latter) and it is linear on $[\beta_c,+\infty)$. This fact is related to the particular structure of
the tree that leads to major simplifications in the 'correlation' structure of the model (see \cite{cf:BPP}). The strict convexity, in
our setting, arises essentially from the property that two path on the hierarchical lattice can re-interesect after being separated at some step. This underlines once more, that $\Z^d$ and the hierarchical lattice have a lot of features in common, which they do not share with the tree.
\end{rem}

We next establish the martingale property for $W_n$ and the zero-one law for its limit.

\begin{lemma}\label{martingale}
$(W_n)_n$ is a positive $\mathcal{F}_n$-martingale. It converges
$Q$-almost surely to a non-negative limit $W_{\infty}$ that satisfies
the following zero-one law:

\begin{eqnarray}
 Q \left( W_{\infty} > 0 \right)\,\in\, \{0,1\}.
 \label{zeroone}
\end{eqnarray}
\end{lemma}

\vspace{2ex}
\noindent Recall that martingales appear when the disorder is displayed on sites, in contrast with disorder on bonds 
as in \cite{cf:CD,cf:DG}.
\vspace{3ex}

Observe that

\begin{eqnarray*}
p(\beta)-f(\beta)=\lim_{n\to +\infty} \frac{1}{s^n} \log W_n(\beta),
\end{eqnarray*}

\noindent so, if we are in the situation $Q(W_{\infty}>0)=1$, we have
that $p(\beta)=f(\beta)$. This motivates the following definition:

\begin{definition}\label{disorder}
 If $Q(W_{\infty}>0)=1$, we say that weak disorder holds. In the opposite situation,
 we say that strong disorder holds.
\end{definition}

\begin{rem}\rm
Later, we will give a statement(Proposition \ref{th:fracmom}) that guarantees that strong disorder is equivalent to $p(\gb)\neq f(\gb)$, a situation that is sometimes called very strong disorder. This is believed to be true for polymer models on $\Z^d$ or $\R^d$ but it remains an unproved and challenging conjecture in dimension $d\ge 3$ (see \cite{cf:CH_al}).
\end{rem}

The next proposition lists a series of partial results that in some sense clarify
the phase diagram of our model. 

\begin{proposition}\label{misc}

  $(i)$ There exists $\beta_0\in[0, +\infty]$ such that strong disorder holds
  for $\beta > \beta_0$ and weak disorder holds for $\beta\le  \beta_0$.
  \vspace{2ex}

  $(ii)$ If $b>s$, $\beta_0>0$. Indeed, there exists $\beta_2\in(0,\infty]$ such that for all $\beta<\beta_2$,
         $\sup_n Q(W^2_{n}(\beta))<+\infty$, and therefore weak disorder holds.
  \vspace{2ex}

  $(iii)$ If $\beta \lambda'(\beta)-\lambda(\beta)> \frac{2\log b}{s-1}$, then strong
          disorder holds.
  \vspace{2ex}

  $(iv)$ 
  In the case where $\go_z$ are gaussian random variables, $(iii)$ can be improved for $b>s$: strong disorder holds as soon as $\gb>\sqrt{\frac{2(b-s)\log b}{(b-1)(s-1)}}$.
  \vspace{2ex}

  $(v)$ If $b\leq s$, then strong disorder holds for all $\beta$.
\end{proposition}

\begin{rem}\rm
On can check that the formula in $(iii)$ ensures that $\gb_0<\infty$ whenever the distribution of $\go_z$ is unbounded.
\end{rem}
\begin{rem}\rm
An implicit formula is given for $\gb_2$ in the proof and this gives a lower bound for $\gb_0$. However, when $\gb_2<\infty$, it never coincides with the upper bound given by $(iii)$ and $(iv)$, and therefore knowing the exact value of the critical temperature when $b>s$ remains an open problem.
\end{rem}
\vspace{4ex}

We now provide more quantitative information for the regime considered in $(v)$:

\begin{theorem} \label{th:bs}
When $s>b$, there exists a constant $c_{s,b}=c$ such that for any $\gb\le 1$ we have
\begin{align*}
\frac{1}{c}\gb^{\frac{2}{\alpha}}\le \lambda(\beta)-p(\beta) \le c\gb^{\frac{2}{\alpha}}
\end{align*}
where $\alpha=\frac{\log s-\log b}{\log s}$.
\end{theorem}

\begin{theorem} \label{th:ss}
When $s=b$, there exists a constant $c_s=c$ such that for any $\gb\le 1$ we have
\begin{align*}
\exp\left(-\frac{c}{\gb^2}\right)\le  \lambda(\beta) - p(\beta) \le c\exp\left(-\frac{1}{c\gb}\right)
\end{align*}
\end{theorem}

\vspace{4ex}
In the theory of directed polymer in random environment, it is believed that, in low dimension, the quantity $\log Z_n$ undergoes large fluctuations around its average (as opposed to what happens in the weak disorder regime where the fluctuation are of order $1$). More precisely: it is believed that there exists  exponents $\xi>0$ and $\chi\ge 0$ such that
\begin{equation}
 \log Z_n-Q\log Z_n \asymp N^{\xi} \text{ and } \var_Q  \log Z_n\asymp N^{2\chi},
\end{equation}
where $N$ is the length of the system ($=n$ on $\Z^d$ and $s^n$ one our hierarchical lattice).
In the non-hierarchical model this exponent is of major importance as it is closely related to the {\sl volume exponent} $\xi$ that gives
the spatial fluctuation of the polymer chain (see e.g.\ \cite{cf:J} for a discussion on fluctuation exponents). Indeed it is conjectured for the $ \Z^d$ models that
\begin{equation}
\chi=2\xi-1.
\end{equation}
This implies that the polymer trajectories are superdiffusive as soon as $\chi>0$.
  In our hierarchical setup, there is no such geometric interpretation but having a lower bound on the fluctuation allows to get a significant localization result.

\begin{proposition}\label{fluctuations}
When $b<s$, there exists a constant $c$ such that for all $n\ge 0$ we have
\begin{equation}
\var_Q\left(\log Z_n\right)\ge \frac{c (s/b)^{n}}{\gb^2}.
\end{equation}
Moreover, for any $\gep>0$, $n\ge 0$, and $a\in \bbR$,
\begin{equation}
Q\left\{\log Z_n \in [a, a+\gep(s/b)^{n/2}]\right\}\le \frac{8\gep}{\gb}.
\end{equation}
\end{proposition}

\noindent This implies that if the fluctuation exponent $\chi$ exists, $\chi\ge \frac{\log s -\log b}{2\log s}$.
We also have the corresponding result for the case $b=s$

\begin{proposition}\label{fluc2}
When $b=s$, there exists a constant $c$ such that for all $n\ge 0$ we have
\begin{equation}
\var_Q \left(\log Z_n\right)\ge \frac{c n}{\gb^2}.
\end{equation}
Moreover for any $\gep>0$, $n\ge 0$, and $a\in \bbR$,
\begin{equation}
Q\left\{\log Z_n \in [a, a+\gep\sqrt{n}]\right\}\le \frac{8\gep}{\gb}.
\end{equation}
\end{proposition}

From the fluctuations of the free energy we can prove the following:
For $g \in \gG_n$ and $m<n$, we define $g|_m$ to be the restriction of $g$ to $D_m$.

\begin{cor}\label{locloc}
If $b\le s$, and $n$ is fixed we have
\begin{equation}
 \lim_{n\to\infty} \sup_{g\in \gG_m} \mu_n(\gga|_m=g)=1,
\end{equation}
\noindent where the convergence holds in probability.
\end{cor}
Intuitively this result means that if one look on a large scale, the law of $\mu_n$ is concentrated in the neighborhood of a single path. Equipping $\gG_n$ with a natural metric (two path $g$ and $g'$ in $\gG_n$ are at distance $2^{-m}$ if and only if $g|_m\ne g'|_m$ and $g|_{m-1}=g|_{m-1}$) makes this statement rigorous.

\begin{rem}\rm
Proposition \ref{misc}$(v)$ brings the idea that $b\le s$ for this hierarchical model is equivalent to the $d\le 2$ case for the model in $\Z^{d}$ (and that $b>s$ is equivalent to $d>2$). Let us push further the analogy: let $\gga^{(1)}$ , $\gga^{(2)}$ be two paths chosen uniformly at random in $\gG_n$ (denote the uniform-product law by $P^{\otimes 2}$), their expected site overlap is of order $(s/b)^n$ if $b<s$, of order $n$ if $b=s$, and of order $1$ if $b>s$.
If one denotes by $N= s^n$ the length of the system, one has
\begin{equation}
 P^{\otimes 2}\left[\sum_{t=0}^{N} \ind_{\{\gamma^{(1)}_t= \gamma^{(2)}_t\}}\right]\, \asymp\, \begin{cases}N^{\alpha} \quad \text{ if } b < s,\\
                                                          \log N \quad \text{ if } b=s,\\
								1 \text{ if } b> s,
                                                          \end{cases}
 \end{equation}
(where $\alpha=(\log s-\log b)/\log s$).
Comparing this to the case of random walk on $\Z^d$, we can infer that the case $b=s$ is just like $d=2$ and that the case $d=1$ is similar to $b=\sqrt{s}$ ($\ga=1/2$). One can check in comparing \cite[Theorem 1.4, 1.5, 1.6]{cf:Lac} with Theorem \ref{th:bs} and \ref{th:ss}, that this analogy is relevant.
\end{rem}

The paper is organised as follow
\begin{itemize}
 \item In section \ref{mtricks} we prove 
some basic statements about the free energy, Lemma \ref{martingale} and
the first part of Proposition \ref{misc}.
 \item Item $(ii)$ from Proposition \ref{misc} is proved in Section $5.1$.
Item $(v)$ is a consequence of Theorems \ref{th:bs} and \ref{th:ss}.
 \item Items $(iii)$ and $(iv)$ are proved in Section $6.3$.
Theorems \ref{th:bs} and \ref{th:ss} are proved in Section $6.1$ and $6.3$
respectively.
 \item In section \ref{flucloc} we prove Propositions \ref{fluctuations} and \ref{fluc2} and Corrolary \ref{locloc}.
 \item  In section \ref{weakpol} we define and investigate the properties of the infinite volume polymer measure in the weak disorder phase.
 \item In section \ref{bddis} we shortly discuss about the bond disorder model.
\end{itemize}




\section{Martingale tricks and free energy}\label{mtricks}

We first look at to the existence of the quenched free energy

\begin{eqnarray*}
  p(\beta)= \lim_{n\to +\infty} \frac{1}{n} Q \left( \log Z_n(\beta) \right),
\end{eqnarray*}

\noindent and its relation with the annealed free energy. The case $\beta=0$ is
somehow instructive. It gives the number of paths in $\Gamma_n$ and is
handled by the simple recursion:

\begin{eqnarray*}
Z_n(0)=b \left( Z_{n-1}(0) \right).
\end{eqnarray*}

\noindent This easily yields

\begin{eqnarray}
\label{paths}
 |\Gamma_n| = Z_n(\beta=0)= b^{\frac{s^{n}-1}{s-1}}.
\end{eqnarray}

Much in the same spirit than (\ref{eq:recab}), we can find a recursion for
$Z_n$:

\begin{eqnarray}
	Z_{n+1}= \sum^b_{i=1} Z^{(i,1)}_n \cdots Z^{(i,s)}_n \times e^{\beta \omega_{i,1}}\cdots e^{\beta \omega_{i,s-1}}.
	\label{eq:recz}
\end{eqnarray}
The existence of the quenched free energy follows by monotonicity: we have

$$Z_{n+1} \geq Z^{(1,1)}_n Z^{(1,2)}_n \cdots Z^{(1,s)}_n \times e^{\beta \omega_{1,1}}\cdots e^{\beta \omega_{1,s-1}},$$

\noindent so that (recall the $\omega$'s are centered random variables)

$$\frac{1}{s^{n+1}} Q \log Z_{n+1} \geq \frac{1}{s^n} Q \log Z_n.$$

\noindent The annealed free energy provides an upper bound:

\begin{eqnarray*}
\frac{1}{s^n} Q \log Z_n &\leq& \frac{1}{s^n} \log Q Z_n\\
                         &=& \frac{1}{s^n} \log e^{\lambda(\beta)(s^n-1)}Z_n(\beta=0)\\
                         &=& \left( 1 - \frac{1}{s^n} \right) \left( \lambda(\beta) + \frac{\log b}{s-1} \right)\\
                         &=& \left( 1 - \frac{1}{s^n} \right) f(\beta).
\end{eqnarray*}

\vspace{3ex}


We now prove the strict convexity of the free energy. The proof
is essentially borrowed from \cite{cf:CPV}, but it is remarkably simpler
in our case.

\begin{proof}[Proof of the strict convexity of the free energy]
We will consider a Bernoulli environment ($\go_z=\pm 1$ with probability $p$, $1-p$; note that our assumptions on the variance and expectation for $\go$ are violated but centering and rescaling $\go$ does not change the argument).
We refer to \cite{cf:CPV} for generalization to more general environment.

\noindent An easy computation yields

\begin{eqnarray*}
\frac{d^2}{d\beta^2} Q \log Z_n = Q {\rm Var}_{\mu_n} H_n (\gamma).
\end{eqnarray*}

\noindent We will prove that for each $K>0$, there exists a constant
$C$ such that, for all $\beta \in[0,K]$ and $n\geq 1$,

\begin{eqnarray}\label{lowerboundenergy}
{\rm Var}_{\mu_n} H_n (\gamma) \geq C s^n
\end{eqnarray}

\vspace{2ex}
For $g \in \gG_n$ and $m<n$, we define $g|_m$ to be the restriction of $g$ to $D_m$.
By the conditional variance formula,

\begin{eqnarray}
\nonumber
{\rm Var}_{\mu_n} H_n &=& \mu_n \left( {\rm Var}_{\mu_n} (H_n(\gamma)\, |\, \gamma_{|_{n-1}}) \right) +
                   {\rm Var}_{\mu_n} \left( \mu_n(H_n(\gamma)\, |\,  \gamma_{|_{n-1}})\right)\\
                   \label{condvar}
                  &\geq& \mu_n \left( {\rm Var}_{\mu_n} (H_n(\gamma) \,|\, \gamma_{|_{n-1}}) \right)
\end{eqnarray}

\noindent Now, for $l=0,...,s^{n-1}-1$, $g \in \Gamma_n$, define

\begin{eqnarray*}
H^{(l)}_n (g) = \sum^{(l+1)s-1}_{t=ls+1} \omega(g_t),
\end{eqnarray*}

\noindent so  (\ref{condvar}) is equal to

\begin{eqnarray*}
\mu_n  {\rm Var}_{\mu_n} \left( \sum^{s^{n-1}-1}_{l=0} H^{(l)}_{n}(\gamma) | \gamma_{|_{n-1}} \right)
= \sum^{s^{n-1}-1}_{l=0} \mu_n  {\rm Var}_{\mu_n} \left(  H^{(l)}_{n}(\gamma) | \gamma_{|_{n-1}} \right),
\end{eqnarray*}

\noindent by independence. Summarizing,

\begin{equation}\label{varcond}
{\rm Var}_{\mu_n} H_n\geq \sum^{s^{n-1}}_{l=1} \mu_n  {\rm Var}_{\mu_n} \left(  H^{(l)}_{n}(\gamma) | \gamma_{|_{n-1}} \right).
\end{equation}

\noindent The rest of the proof consists in showing that each term of the sum is bounded from below by a positive constant,
uniformly in $l$ and $n$.
For any $x\in D_{n-1}$ such that the graph distance between $x$ and $A$ is $ls$ in $D_n$ (i.e.\ $x\in D_{n-1}$), we define the set of environment

\begin{eqnarray*}
M(n,l,x)= \left\lbrace \omega: \left|\{H^{(l)}_n(g,\omega)\,:\, g\in \gG_n, g_{ls}=x\} \right| \geq 2 \right\rbrace.
\end{eqnarray*}

\noindent These environments provide the fluctuations in the energy needed for the
uniform lower bound we are searching for.
One second suffices to convince oneself
that $Q(M(n,l,x)>0$, and does not depend on the parameters $n,\, l$ or $x$. Let $Q(M)$ denote improperly the common value of $Q( M(n,l,x))$. 
Now, it is easy to see (from \eqref{varcond}) that there exists a constant $C$ such that for all $\gb<K$,

\begin{eqnarray*}
Q \left[{\rm Var}_{\mu_n} H_n\right] &\ge& C Q \left[\sum^{s^{n-1}-1}_{l=1} \sum_{x\in D_{n-1}}  {\bf 1}_{M(n, l, x)}\mu_n ({\gamma_{ls} = x})\right].
    \end{eqnarray*}
    
\noindent Define now $\mu^{(l)}_n$ as the polymer measure in the environment obtained from
$\omega$ by setting $\omega(y)=0$ for all sites $y$ which distance to $0$ is between $ls$ and $(l+1)s$.
One can check that for all $n$, and all path $g$,

\begin{eqnarray*}
\exp(-2\gb(s-1)) \mu^{(l)}_n(\gamma=g) \leq \mu_n(\gamma=g) \leq \exp(2\gb (s-1)) \mu^{(l)}_n(\gamma).
\end{eqnarray*}

\noindent We note that under $Q$, $\mu_n^{(l)}(\gamma_{ls} = x)$ and $\ind_{M(n, l, x)}$ are random variables, so that 

\begin{eqnarray*}
Q\left[ {\rm Var}_{\mu_n} H_n \right]&\geq& C\exp(-2\gb(s-1)) Q \left[\sum^{s^{n-1}-1}_{l=0} \sum_x  {\bf 1}_{M(n, l, x)}\mu^{(l)}_n ({\gamma_{ls} = x})\right]\\
   &=& C\exp(-2\gb(s-1))\sum^{s^{n-1}}_{l=1} \sum_{x\in D_{n-1}} Q(M(n,l,x))   Q\left[ \mu^{(l)}_n ({\gamma_l = x})\right]\\
   &=& C\exp(-2\gb(s-1)) Q(M) s^{n-1}.
\end{eqnarray*}
\end{proof}

\vspace{3ex}


\noindent We now establish the martingale property for the normalized free energy.

\begin{proof}[Proof of Lemma \ref{martingale}]

Set $z_n=Z_n(\beta=0)$. We have already remarked that this is just the number of (directed) paths in $D_n$, 
and its value is given by (\ref{paths}). 
Observe that
$g\in \gG_n$ visits $s^n(s-1)$ sites of $n+1$-th generation. The restriction of paths in $D_{n+1}$
to $D_n$ is obviously not one-to-one as for each path $g'\in \gG_n$, there are $b^{s^n}$ paths
in $\gG_{n+1}$ such that $g|_n=g'$. Now,

\begin{eqnarray*}
Q\left(Z_{n+1}(\beta)| \mathcal{F}_n\right)
&=& \sum_{g \in D_{n+1}} Q\left( e^{\beta H_{n+1}(g)}|\mathcal{F}_n\right)\\
&=&\sum_{g' \in D_n} \sum_{g \in D_{n+1}}Q\left( e^{\beta H_{n+1}(g)}|\mathcal{F}_n\right){\bf 1}_{g|_n=g'}\\
&=&\sum_{g' \in D_n}\sum_{g \in D_{n+1}} e^{\beta H_{n}(g')}e^{s^n(s-1)\lambda(\beta)}{\bf 1}_{g|_n=g'}\\
&=&\sum_{g' \in D_n} e^{\beta H_{n}(g')}e^{s^n(s-1)\lambda(\beta)}\sum_{g \in D_{n+1}}{\bf 1}_{g|_n=g'}\\
&=&e^{s^n(s-1)\lambda(\beta)}b^{s^n}\sum_{g' \in D_n} e^{\beta H_{n}(g')}\\
&=&Z_n(\beta) \frac{z_{n+1}e^{s^{n+1}\lambda(\beta)}}{z_n e^{s^n \lambda(\beta)}}.
\end{eqnarray*}

\noindent This proves the martingale property. For (\ref{zeroone}), let's generalize a little
the preceding restriction procedure. As before, for a path $g \in D_{n+k}$, denote by
$g |_n$ its restriction to $D_n$. Denote by $I_{n,n+k}$ the set of time indexes that have been
removed in order to perform this restriction and by $N_{n,n+k}$ its cardinality. Then

\begin{eqnarray*}
Z_{n+k}= \sum_{g \in D_n} e^{\beta H_n(g)} \sum_{g' \in D_{n+k}, g'|_n=g}
\exp \left\lbrace \beta \sum_{t\in I_{n,n+k}} \omega(g'_t) \right\rbrace.
\end{eqnarray*}

\noindent Consider the following notation, for $g\in \gG_n$,
\begin{eqnarray*}
\tilde{W}_{n,n+k}(g)= c^{-1}_{n,n+k}\sum_{g' \in D_{n+k}, g'|_n=g}\exp \left\lbrace \beta \sum_{t\in I_{n,n+k}} \omega(g'_t)  - N_{n,n+k} \lambda(\beta)\right\rbrace,
\end{eqnarray*}

\noindent where $c_{n,n+k}$ stands for the number paths in the sum. With this notations, we have, 

\begin{eqnarray}
 W_{n+k} = \frac{1}{z_n} \sum_{g \in D_n} e^{\beta H_n(g)-(s^{n}-1)\lambda(\beta)}\tilde{W}_{n,n+k}(g),
\end{eqnarray}

\noindent and, for all $n$,

\begin{eqnarray}
\left\lbrace W_{\infty}=0 \right\rbrace =
\left\lbrace \tilde{W}_{n,n+k}(g) \to 0,\, {\rm as}\, k\to +\infty,\, \forall \, g \in D_n \right\rbrace.
\label{eq:WW}
\end{eqnarray}

\noindent The event in the right hand side is measurable with respect to the disorder of generation
not earlier than $n$. As $n$ is arbitrary, the right hand side of (\ref{eq:WW}) is in the tail $\sigma$-algebra
and its probability is either $0$ or $1$.
\end{proof}

This, combined with FKG-type arguments (see \cite[Theorem 3.2]{cf:CY} for details), proves part $(i)$ of
Proposition \ref{misc}. Roughly speaking, the FKG inequality is used to insure that there is no reentrance
phase.







\section{Second moment method and lower bounds}

This section contains all the proofs concerning coincidence of annealed and quenched free--energy for $s>b$ and lower bounds on the free--energy for $b\le s$ (i.e. half of the results from Proposition \ref{misc} to Theorem \ref{th:ss}.)
First, we discuss briefly the condition on $\gb$ that one has to fulfill to to have $W_{n}$ bounded in $\bbL_2(Q)$. Then for the cases when strong disorder holds at all temperature ($b\le s$), we present a method that combines control of the second moment up to some scale $n$ and a percolation argument to get a lower bound on the free energy.
\\
First we investigate how to get the variance of $W_n$ (under $Q$).
From \eqref{eq:recab} we get the induction for the variance $v_n=Q\left[(W_n-1)^2\right]$:

\begin{eqnarray}
v_{n+1}&=&\frac{1}{b}\left(e^{(s-1)\gga(\gb)}(v_n+1)^s-1\right), \label{eq:var}\\
v_0&=&0.\label{eq:v00}
\end{eqnarray}
where $\gga(\gb):=\gl(2\gb)-2\gl(\gb)$.

\subsection{The $L^2$ domain: $s<b$}

If $b>s$, and $\gga(\gb)$ is small, the map
\begin{align*}
g:\ x\mapsto\frac{1}{b}\left(e^{(s-1)\gga(\gb)}(x+1)^s-1\right)
\end{align*}
possesses a fixed point. In this case, \eqref{eq:var} guaranties that $v_n$ converges to some finite limit. Therefore, in this case, $W_n$ is a positive martingale bounded in $\bbL^2$, and therefore converges almost surely to $W_\infty \in \bbL^2(Q)$ with $Q W_{\infty}=1$, so that
\begin{align*}
p(\beta)-\lambda(\beta)=\lim_{n\rightarrow\infty}\frac{1}{s^n}\log W_n=0,
\end{align*}

\noindent and weak disorder holds.
One can check that $g$ has a fixed point if and only if
\begin{align*}
\gga(\gb)\le \frac{s}{s-1}\log \frac{s}{b}-\log \frac{b-1}{s-1}
\end{align*}




\subsection{Control of the variance: $s>b$}

For $\epsilon>0$, let $n_0$ be the smallest integer such that $v_n\ge \gep$.

\vspace{2ex}

\begin{lemma}\label{th:bss}
For any $\gep>0$, there exists a constant $c_{\gep}$ such that for any $\gb\le 1$
\begin{align*}
n_0\ge \frac{2 |\log \gb|}{\log s - \log b}-c_{\gep}.
\end{align*}
\end{lemma}

\begin{proof}

Expanding \eqref{eq:var} around $\gb=0$, $v_n=0$, we find a constant $c_1$ such that, whenever $v_n\le 1$ and $\gb\le 1$,
\begin{align}
v_{n+1}\le \frac{s}{b}(v_n+c_1\gb^2)(1+c_1v_n) \label{eq:varmod}.
\end{align}
 Using \eqref{eq:varmod}, we obtain by induction
\begin{align*}
v_{n_0}\le \prod_{i=0}^{n_0-1}(1+c_1v_i)\left[c_1\gb^2\left(\sum_{i=0}^{n_0-1}(s/b)^{i}\right)\right].
\end{align*}
From \eqref{eq:var}, we see that $v_{i+1}\ge (s/b)v_i$. By definition of $n_0$, $v_{n_0-1}<\epsilon,$ so that $v_{i} <\gep (s/b)^{i-n_0+1}$. Then
\begin{align*}
\prod_{i=0}^{n_0-1}(1+c_1v_i)\le \prod_{i=0}^{n_0-1}(1+c_1\gep(s/b)^{i-n_0+1})\le \prod_{k=0}^{\infty}(1+c_1\gep(s/b)^{-k})\le 2,
\end{align*}
where the last inequality holds for $\gep$ small enough.
In that case we have
\begin{align*}
\gep\le v_{n_0}\le 2 c_1 \gb^2 (s/b)^{n_0},
\end{align*}
so that \begin{equation*}
         n_0\ge \frac{\log(\gep/2c_1\gb^2)}{\log(s/b)}.
        \end{equation*}
\end{proof}




\subsection{Control of the variance: $s=b$}

\begin{lemma}\label{th:sss}
There exists a constant $c_2$ such that, for every $\gb\le 1$, 
\begin{align*}
v_n\le \gb, \quad \forall \, n\ \le  \ \frac{c_2}{\gb}.
\end{align*}
\end{lemma}
\begin{proof}

By (\ref{eq:varmod}) and induction we have, for any $n$ such that $v_{n-1}\le 1$ and $\beta\leq 1$,
\begin{align*}
v_{n}\le n\gb^2 \prod_{i=0}^{n-1} (1+c_1v_i).
\end{align*}
Let $n_0$ be the smallest integer such that $v_{n_0}>\gb$. By the above formula, we have
\begin{align*}
v_{n_0}\le n_0\gb^2(1+c_1\gb)^{n_0}
\end{align*}
Suppose that $n_0\le (c_2/\gb)$, then
\begin{align*}
\gb\le v_{n_0}\le c_2c_1\gb(1+c_1\gb)^{c_2/\gb}.
\end{align*}
If $c_4$ is chosen small enough, this is impossible.
\end{proof}




\subsection{Directed percolation on $D_n$} For technical reasons, we need to get some understanding on directed independent bond percolation on $D_n$.
Let $p$ be the probability that an edge is open (more detailed considerations about edge disorder
are given in the last section). 
The probability of having an open path from $A$ to $B$ in $D_n$ follows the recursion
\begin{align*}
p_0&=p,\\
p_n&=1-(1-p_{n-1}^s)^b.
\end{align*}
On can check that the map $x\mapsto 1-(1-x^s)^b$ has a unique unstable fixed point on $(0,1)$; we call it $p_c$.
Therefore if $p>p_c$, with a probability tending to $1$, there will be an open path linking $A$ and $B$ in $D_n$. If $p<p_c$, $A$ and $B$ will be disconnected in $D_n$ with probability tending to $1$. If $p=p_c$, the probability that $A$ and $B$ are linked in $D_n$ by an open path is stationary. See \cite{cf:HK} for a deep investigation of percolation on hierarchical lattices.

\subsection{From control of the variance to lower bounds on the free energy}

Given $b$ and $s$, let $p_c=p_c(b,s)$ be the critical parameter for directed bond percolation.

\begin{proposition}\label{th:perco}
Let $n$ be an integer such that $v_n=Q (W_n-1)^2 <\frac{1-p_c}{4}$ and $\gb$ such that $p(\gb)\le (1-\log2)$.
Then
\begin{align*}
\lambda(\beta)-p(\beta)\geq s^{-n}
\end{align*}
\end{proposition}
\begin{proof}

If $n$ is such that $Q\left[ (W_n-1)^2\right] <\frac{1-p_c}{4}$, we apply Chebycheff inequality to see that
\begin{align*}
Q(W_n<1/2)\le 4 v_n< 1-p_c.
\end{align*}

Now let be $m\ge n$. $D_m$ can be seen as the graph $D_{m-n}$ where the edges have been replaced by i.i.d.\ copies of $D_n$ with its environment (see fig. \ref{perco}).
To each copy of $D_n$ we associate its renormalized partition function; therefore, to each edge $e$ of $D_{m-n}$ corresponds an independent copy of $W_n$, $W_n^{(e)}$. By percolation (see fig. \ref{perco2}), we will have, with a positive probability not depending on $n$, a path in $D_{m-n}$ linking $A$ to $B$, going only through edges which associated $W_n^{(e)}$ is larger than $1/2$.

\begin{figure}[h]
\begin{center}
\leavevmode
\epsfysize =6.5 cm
\psfragscanon
\psfrag{dn}[c]{\tiny{Dn}}
\psfrag{a}[c]{A}
\psfrag{b}[c]{B}
\psfrag{copies}[c]{\tiny{Independent copies of system of rank $n$.}}
\epsfbox{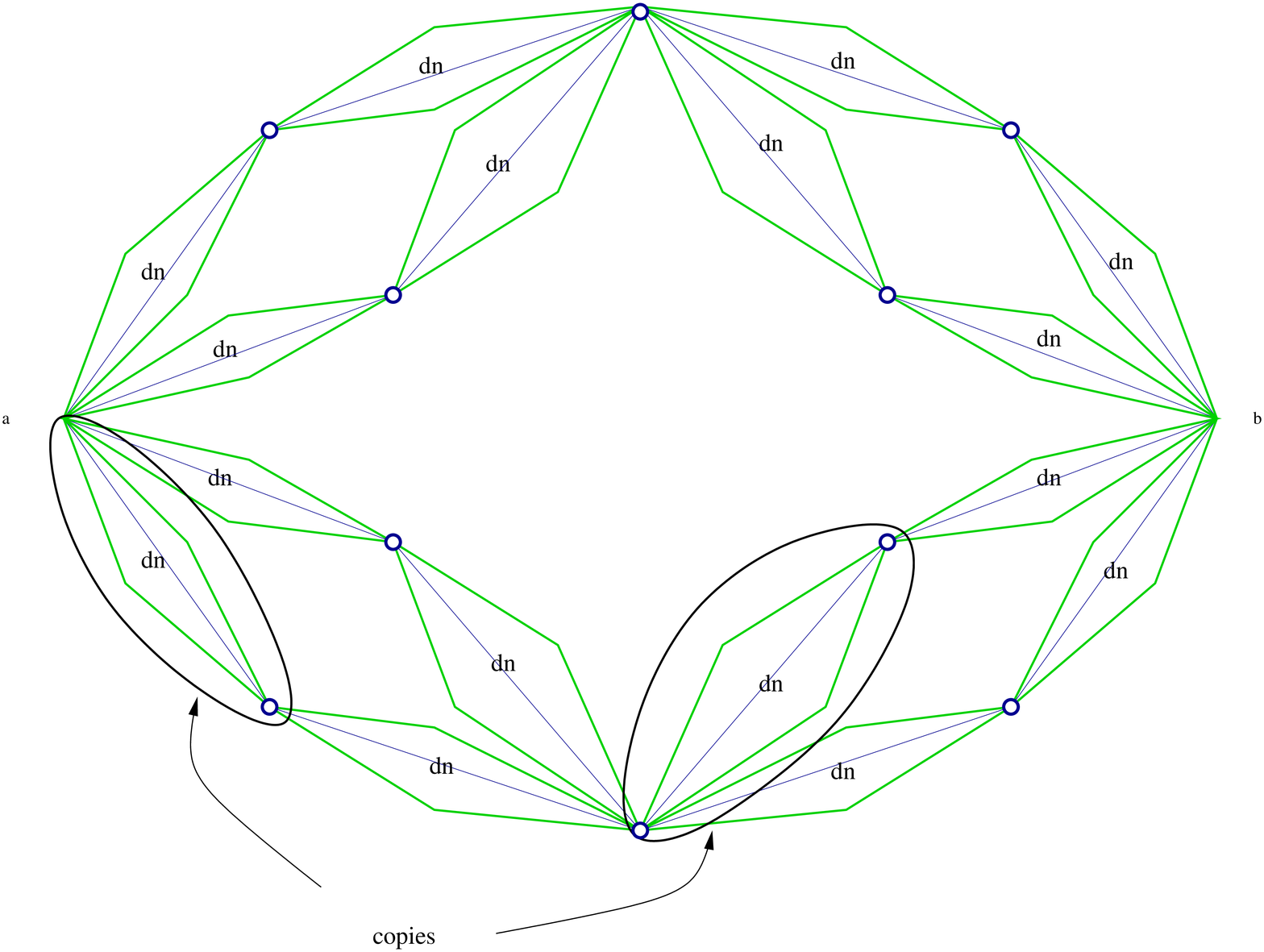}
\end{center}
\caption{\label{perco}On this figure, we scheme how $D_{n+m}$ with its random environment can be seen as independent copies of $D_n$ arrayed as $D_m$.
Here, we have $b=s=2$ $m=2$, each diamond corresponds to a copy of $D_n$ (we can identify it with an edge and get the underlying graph $D_2$). Note that we also have to take into account the environment present on the vertices denoted by circles.}
\end{figure}

\begin{figure}[h]
\begin{center}
\leavevmode
\epsfysize =6.5 cm
\psfragscanon
\psfrag{a}[c]{A}
\psfrag{b}[c]{B}
\psfrag{perco}[c]{\tiny{an open percolation path}}
\epsfbox{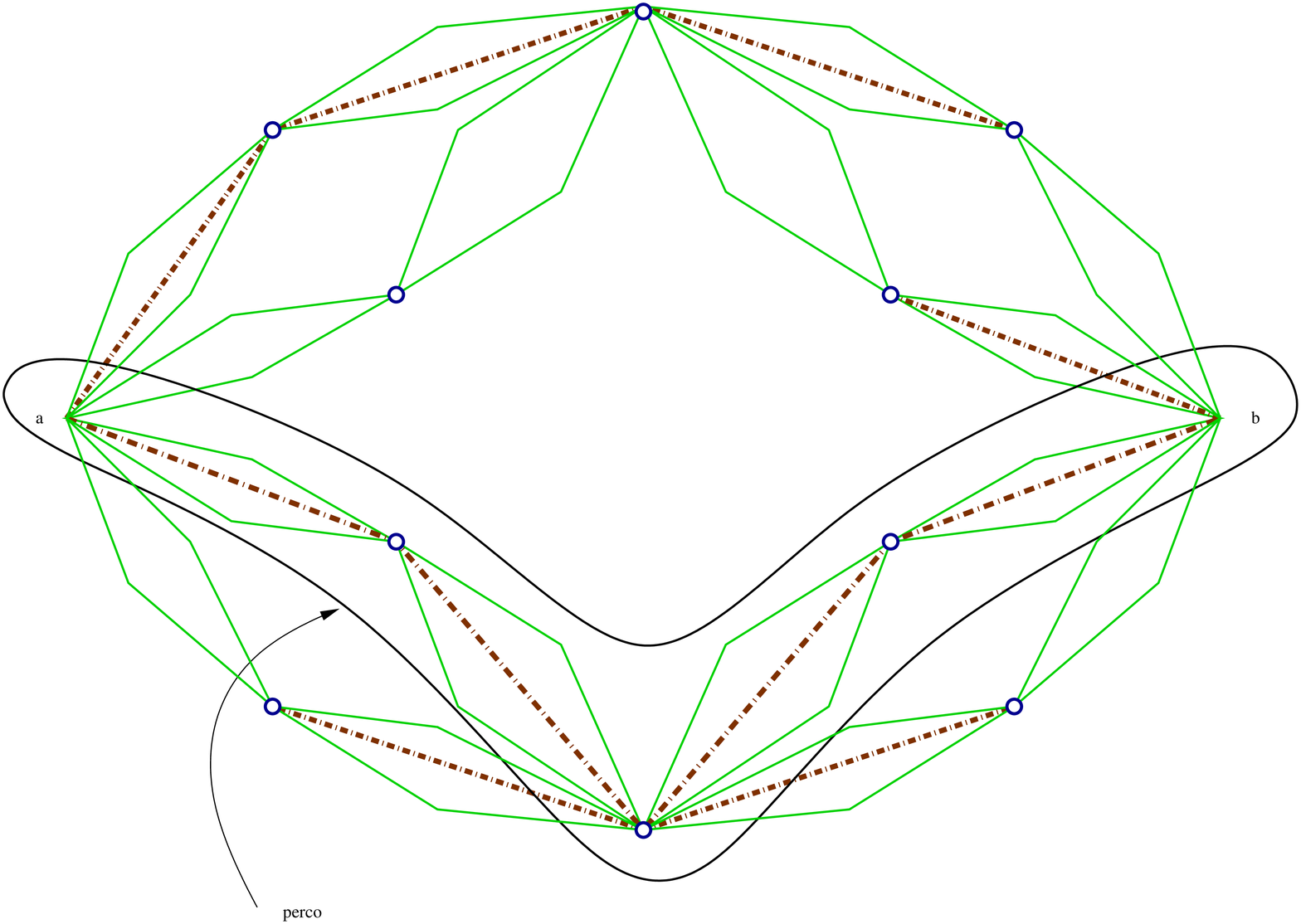}
\end{center}
\caption{\label{perco2}We represent here the percolation argument we use. In the previous figure, we have replaced by an open edge any of the copies of $D_n$ for which satisfies $W_n\ge 1/2$. As it happens with probability larger than $p_c$, it is likely that we can find an open path linking A to B in $D_{n+m}$, especially if $m$ is large.}
\end{figure}

When such paths exist, let $\gamma_0$ be one of them (chosen in a deterministic manner, e.g.\ the lowest such path for some geometric representation of $D_n$).
We look at the contribution of these family of paths in $D_m$ to the partition function.
We have
\begin{align*}
W_m\ge (1/2)^{s^{m-n}}\exp\left(\sum_{z\in \gamma_0} \gb\go_z-\gl(\gb)\right)
\end{align*}
Again, with positive probability (say larger than $1/3$), we have $\sum_{z\in \gamma_0} \go_z\ge 0$ (this can be achieved the the central limit theorem). Therefore with positive probability we have
\begin{align*}
\frac{1}{s^m}\log W_m\ge -\frac{1}{s^{n}}(\log 2+\gl(\gb)).
\end{align*}
As $1/s^m \log W_m$ converges in probability to the free energy this proves the result.

\end{proof}

\begin{proof}[Proof of the right-inequality in Theorems \ref{th:bs} and \ref{th:ss}].

The results now follow by combining Lemma \ref{th:bss} or \ref{th:sss} for $\gb$ small enough, with Proposition
\ref{th:perco}.

\end{proof}




\section{Fractional moment method, upper bounds and strong disorder}\label{fm}

In this section we develop a way to find an upper bound for $\lambda(\beta)-p(\beta)$, or just to find out if strong disorder hold. The main tool we use are fractional moment estimates and measure changes.

\subsection{Fractional moment estimate}
In the sequel we will use the following notation. Given a fixed parameter $\theta\in(0,1)$, define
\begin{eqnarray}
 u_n&:=&Q W_n^{\theta},\\ \label{un}
 a_\theta&:=&Q A^{\theta}=\exp(\gl(\theta\gb)-\theta\gl(\gb)) \label{aomega}.
\end{eqnarray}

\begin{proposition}\label{th:fracmom}
The sequence $(f_n)_n$ defined by 
$$f_n:=\theta^{-1}s^{-n}\log \left(a_\theta b^{\frac{1-\theta}{s-1}}u_n\right)$$
\noindent is decreasing and we have
\begin{align*}
\lim_{n\rightarrow\infty} f_n\ge p(\beta)-\lambda(\beta).
\end{align*}
\vspace{1ex}

\noindent $(i)$ In particular, if for some $n\in\N$, $u_n< a_{\theta}^{-1}b^{\frac{\theta-1}{s-1}}$, strong disorder holds.
\vspace{1ex}

\noindent $(ii)$ Strong disorder holds in particular if $a_\theta< b^{\frac{\theta-1}{s-1}}$.

\end{proposition}

\begin{proof}
The inequality $\left(\sum a_i\right)^{\theta}\le \sum a_i^{\theta}$ (which holds for any $\theta\in (0,1)$ and any collection of positive numbers $a_i$) applied to
\eqref{eq:recab} and averaging with respect to $Q$ gives

\begin{align*}
u_{n+1}\le b^{1-\theta}u_n^{s}a_{\theta}^{s-1}
\end{align*}

\noindent From this we deduce that the sequence
\begin{align*}
s^{-n}\log\left( a_{\theta} b^{\frac{1-\theta}{s-1}}u_n\right)
\end{align*}

\noindent is decreasing. Moreover we have
\begin{align*}
p(\beta)-\lambda(\beta)=\lim_{n\rightarrow\infty}\frac{1}{s^n}Q \log W_n\le \lim_{n\rightarrow \infty} \frac{1}{\theta s^n}\log Q W_n^{\theta}=\lim_{n\to\infty}f_n.
\end{align*}

As a consequence very strong disorder holds if $f_n<0$ for any $f_n$. As a consequence, strong disorder and very strong disorder are equivalent.
\end{proof}

\subsection{Change of measure and environment tilting}

The result of the previous section assures that we can estimate the free energy if we can bound accurately some non integer moment of $W_n$.
Now we present a method to estimate non-integer moment via measure change, it has been introduced to show disorder relevance in the case of wetting on non hierarchical lattice \cite{cf:GLT} and used since in several different contexts since, in particular for directed polymer models on $\Z^d$, \cite{cf:Lac}.
Yet, for the directed polymer on hierarchical lattice, the method is remarkably simple to apply, and it seems to be the ideal context to present it.
\\
Let $\tilde Q$ be any probability measure such that $Q$ and $\tilde Q$ are mutually absolutely continuous.
Using H\"older inequality we observe that 
\begin{equation}
Q W_n^{\theta}
   = \tilde Q \frac{\dd Q}{\dd \tilde Q} W_n^{\theta}
   \le \left[\tilde Q\left(\frac{\dd Q}
     {\dd \tilde Q}\right)^{\frac{1}{1-\theta}}\right]^{(1-\theta)}\left(\tilde Q W_n\right)^{\theta}.
     \label{holder}
\end{equation}
Our aim is to find a measure $\tilde Q$ such that the term $\left[\tilde Q\left(\frac{\dd Q}
     {\dd \tilde Q}\right)^{\frac{1}{1-\theta}}\right]^{(1-\theta)}$ is not very large (i.e.\ of order $1$), and which significantly lowers the expected value of $W_n$. To do so we look for $\tilde Q$ which lowers the value of the environment on each site, by exponential tilting. For $b<s$ it sufficient to lower the value for the environment uniformly of every site of $D_n\setminus \{A,B\}$ to get a satisfactory result, whereas for the $b=s$ case, on has to do an inhomogeneous change of measure. We present the change of measure in a united framework before going to the details with two separate cases.

Recall that $V_i$ denotes the sites of $D_i\setminus D_{i+1}$, and that the number of sites in $D_n$ is

\begin{align}\label{Dnn}
|D_n\setminus \{A,B\}|=\sum_{i=1}^n |V_i|=\sum_{i=1}^n (s-1)b^is^{i-1}=\frac{(s-1)b((sb)^n-1)}{sb-1}
\end{align}

We define $\tilde Q= \tilde Q_{n,s,b}$ to be the measure under which the environment on the site of the $i$-th generation for $i\in\{1,\dots,n\}$ are standard gaussians with mean  $-\delta_i=\delta_{i,n}$, where $\delta_{i,n}$ is to be defined.
The density of $\tilde Q$ with respect to $Q$ is given by

\begin{align*}
\frac{\dd \tilde Q}{\dd Q}(\go)=\exp\left(-\sum^n_{i=1}\sum_{z\in V_i} (\delta_{i,n}\omega_z+\frac{\delta^2_{i,n}}{2})\right).
\end{align*}

\noindent  As each path in $D_n$ intersects $V_i$ on $s^{i-1}(s-1)$ sites, this change of measure lowers the value of the Hamiltonian \eqref{eq: Hn} by $\sum_{i=1}^n s^{i-1}(s-1)\delta_{i,n}$ on any path. Therefore, 
both terms can be easily computed,

\begin{eqnarray}
	\tilde Q\left(\frac{\dd Q}
     {\dd \tilde Q}\right)^{\frac{1}{1-\theta}} =
     \exp \left\lbrace \frac{\theta}{2(1-\theta)} \sum^n_{i=1}|V_i| \delta^2_{i,n} \right\rbrace.
	\label{cost}
\end{eqnarray}

\begin{eqnarray}
	\left(\tilde Q W_n\right)^{\theta} = \exp \left\lbrace -\beta \theta \sum^n_{i=1}  s^{i-1}(s-1) \delta_{i,n}\right\rbrace.
	\label{gain}
\end{eqnarray}

\noindent Replacing \eqref{gain} and \eqref{cost} back into \eqref{holder} gives

\begin{eqnarray}
	u_n \leq \exp \left\lbrace \theta \sum^n_{i=1} \left(  \frac{|V_i|\delta^2_{i,n}}{2(1-\theta)}- \beta s^{i-1}(s-1) \delta_{i,n} \right) \right\rbrace.
	\label{eq:ineq0}
\end{eqnarray}

\vspace{2ex}

\noindent When $\delta_{i,n}= \delta_n$ (i.e. when the change of measure is homogeneous on every site) the last expression becomes simply

\begin{eqnarray}
	u_n \leq \exp \left\lbrace \theta \left(  \frac{|D_n\setminus\{A,B\}|\delta^2_{n}}{2(1-\theta)}- (s^n-1)\beta \delta_{n} \right) \right\rbrace.
	\label{ineq1}
\end{eqnarray}

\noindent In either case, the rest of the proof then consists in finding convenient values for $\delta_{i,n}$ and
$n$ large enough to insure that $(i)$ from Proposition \ref{th:fracmom} holds.




\subsection{Homogeneous shift method: $s>b$}

\begin{proof}[Proof of the left inequality in Theorem \ref{th:bs}, in the gaussian case]

Let $0<\theta<1$ be fixed (say $\theta=1/2$) and $\delta_{i,n}= \delta_n :=(sb)^{-n/2}$.

\noindent Observe from \eqref{Dnn} that $|D_n\setminus\{A,B\}| \delta^2_n \leq 1$, so that \eqref{ineq1} implies

\begin{align*}
u_n\le \exp\left(\frac{\theta}{2(1-\theta)}-\theta \gb (s/b)^{n/2}\frac{s-1}{s}\right).
\end{align*}

\noindent Taking $n=\frac{2(|\log \gb|+\log c_3)}{\log s-\log b}$, we get

\begin{align*}
u_n\le \exp\left(\frac{\theta}{2(1-\theta)}-\frac{\theta c_5s}{s-1}\right).
\end{align*}

\noindent Choosing $\theta=1/2$ and $c_3$ sufficiently large, we have
\begin{align}
f_n=s^{-n}\log a_{\theta}b^{\frac{1-\theta}{s-1}}u_n\le -s^{-n},
\end{align}
so that Proposition \ref{th:fracmom} gives us the conclusion

\begin{align*}
p(\beta)-\lambda(\beta) \le -s^{-n}=-(\gb/c_3)^{\frac{2\log s}{\log s -\log b}}.
\end{align*}
\end{proof}




\subsection{Inhomogeneous shift method: $s=b$} \label{inshm}
One can check that the previous method does not give good enough results for the marginal case $b=s$.
One has to do a change of measure which is a bit more refined and for which the intensity of the tilt in proportional to the Green Function on each site. This idea was used first for the marginal case in pinning model on hierarchical lattice (see \cite{cf:Hubert}).

\begin{proof}[Proof of the left inequality in Theorem \ref{th:ss}, the gaussian case]
This time, we set $\gd_{i,n}:=n^{-1/2}s^{-i}$.
Then (recall \eqref{Dnn}), \eqref{eq:ineq0} becomes

\begin{align*}
u_n\le \exp\left(\frac{\theta}{2(1-\theta)}\frac{s-1}{s}-\theta\gb n^{-1/2}\frac{s-1}{s}\right).
\end{align*}

\noindent Taking $\theta=1/2$ and $n=(c_4/\gb)^2$ for a large enough constant $c_4$, we get that $f_n\le -s^n$  and applying Proposition \ref{th:fracmom}, we obtain

\begin{align*}
p(\beta) - \lambda(\beta) \le -s^{-n}=-s^{-(c_4/\gb)^2}=\exp\left(-\frac{c_4^2\log s}{\gb^2}\right).
\end{align*}
\end{proof}



\subsection{Bounds for the critical temperature}

From Proposition \ref{th:fracmom}, we have that strong disorder holds if
$a_{\theta}< b^{(1-\theta)/(s-1)}$. Taking logarithms, this condition
reads

\begin{eqnarray*}
\lambda(\theta \beta)-\theta \lambda(\beta)< (1-\theta)\frac{ \log b}{s-1}.
\end{eqnarray*}

\noindent We now divide both sides by $1-\theta$ and let $\theta \to 1$. This
proves part $(iii)$ of Proposition \ref{misc}.

\noindent For the case $b>s$, this condition can be improved by the inhomogeneous shifting method; here, we perform it just in the gaussian case. Recall that

\begin{eqnarray}
	u_n \leq \exp \left\lbrace \theta \sum^n_{i=1} \left( \frac{|V_i|\delta^2_{i,n}}{2(1-\theta)}- \beta s^{i-1}(s-1)\delta_{i,n} \right) \right\rbrace.
\end{eqnarray}

\noindent We optimize each summand in this expression taking $\delta_{i,n}=\delta_i = (1-\theta) \beta / b^i$. Recalling
that $|V_i|= (bs)^{i-1}b(s-1)$, this yields

\begin{eqnarray}
\nonumber
	u_n &\leq& \exp \left\lbrace -\theta (1-\theta) \frac{\beta^2}{2} \frac{s-1}{s}  \sum^n_{i=1} \left( \frac{s}{b}\right)^i \right\rbrace\\
	\nonumber
	&\leq& \exp \left\lbrace -\theta (1-\theta) \frac{\beta^2}{2} \frac{s-1}{s}  \frac{s/b-(s/b)^{n+1}}{1-s/b} \right\rbrace.
\end{eqnarray}

\noindent Because $n$ is arbitrary, in order to guaranty strong disorder it is enough to have ( cf. first condition in Proposition \ref{th:fracmom}) for some $\theta\in(0,1)$

\begin{eqnarray}\nonumber
	\theta (1-\theta) \frac{\beta^2}{2} \frac{s-1}{s}  \frac{s/b}{1-s/b}
	> (1-\theta)\frac{\log b}{s-1}+\log a_\theta.
\end{eqnarray}

\noindent In the case of gaussian variables $\log a_\theta=\theta(\theta-1)\gb^2/2$. This is equivalent to

\begin{eqnarray*}
\frac{\beta^2}{2} > \frac{(b-s)\log b}{(b-1)(s-1)}.
\end{eqnarray*}

\noindent This last condition is an improvement of the bound in part $(iii)$ of
Proposition \ref{misc}.







\subsection{Adaptation of the proofs for non-gaussian variables}
\begin{proof}[Proof of the left inequality in Theorem \ref{th:bs} and \ref{th:ss}, the general case]
To adapt the preceding proofs to non-gaussian variables, we have to investigate the consequence of exponential tilting on non-gaussian variables. We sketch the proof in the inhomogeneous case $b=s$, we keep $\gd_{i,n}:=s^{-i}n^{-1/2}$.

Consider $\tilde Q$ with density
\begin{align*}
\frac{\dd \tilde Q}{\dd Q}(\go):=\exp\left(-\sum_{i=1}^n\sum_{z\in V_i}\left(\gd_{i,n}\go_z+\gl(-\gd_{i,n})\right)\right),
\end{align*}
(recall that $\gl(x):=\log Q \exp(x\go)$).
The term giving cost of the change of measure is, in this case,
\begin{eqnarray*}
\left[\tilde Q\left(\frac{\dd Q}{\dd \tilde Q}\right)^{\frac{1}{1-\theta}}\right]^{(1-\theta)}&=&\exp\left((1-\theta)\sum_{i=1}^{n}|V_i|\left[\gl\left(\frac{\theta\gd_{i,n}}{1-\theta}\right)+\frac{\theta}{1-\theta}\gl(-\gd_{i,n})\right]\right)\\
&\leq& \exp\left(\frac{\theta}{(1-\theta)}\sum_{i=1}^n|V_i|\gd_{i,n}^2\right)\, \le \, \exp\left(\frac{\theta}{(1-\theta)}\right)
\end{eqnarray*}
Where the inequality is obtained by using the fact the $\gl(x)\sim_0 x^2/2$ (this is a consequence of the fact that $\go$ has unit variance) so that if $\gb$ is small enough, one can bound every $\gl(x)$ in the formula by $x^2$.

We must be careful when we estimate $\tilde Q W_n$. We have
\begin{align*}
\tilde Q W_n=\exp \left(\sum_{i=1}^n (s-1)s^{i-1}\gl (\gb-\gd_{i,n})-\gl(\gb)-\gl(-\gd_{i,n})\right) Q W_n
\end{align*}
By the mean value theorem
\begin{align*}
\gl (\gb-\gd_{i,n})-\gl(\gb)-\gl(-\gd_{i,n})+\gl(0)=-\gd_{i,n}\left(\gl'(\gb-t_0)-\gl'(-t_0)\right)=-\gd_{i,n}\gb\gl''(t_1),
\end{align*}
for some $t_0\in(0,\gd_{i,n})$ and some $t_1\in (\gb,-\gd_{i,n})$. As we know that $\lim_{\gb\to 0}\gl''(\gb)=1$,
when $\gd_i$ and $\gb$ are small enough, the right-hand side is less than $-\gb\gd_{i,n}/2$.
Hence,
\begin{align*}
\tilde Q W_n\le \exp \left(\sum_{i=1}^n(s-1)s^{i-1}\frac{\gb\gd_{i,n}}{2}\right).
\end{align*}
We get the same inequalities that in the case of gaussian environment, with different constants, which do not affect the proof. The case $b<s$ is similar.
\end{proof}

\section{Fluctuation and localisation results}\label{flucloc}

In this section we use the shift method we have developed earlier to prove fluctuation results

\subsection{Proof of Proposition \ref{fluctuations}}

The statement on the variance is only a consequence of the second statement. Recall that the random variable $\go_z$ here are i.i.d. centered standard gaussians, and that the product law is denoted by $Q$.
We have to prove
\begin{equation}
Q\left\{ \log Z_n \in [a, a+\gb\gep(s/b)^{n/2}]\right\}\le 4\gep \quad \forall \gep>0, n\ge 0, a\in \bbR \label{fluc}
\end{equation}
Assume there exist real numbers $a$ and $\gep$, and an integer $n$ such that \eqref{fluc} does not hold, i.e.

\begin{equation}
Q\left\{ \log \bar Z_n \in [a, a+\gb\gep(s/b)^{n/2})\right\}> 4\gep.
\end{equation}
Then one of the following holds

\begin{equation}\begin{split}\label{abcd}
Q\left\{ \log Z_n \in [a, a+\gb\gep(s/b)^{n/2})\right\}\cap\left\{\sum_{z\in D_n} \go_z\ge 0\right\}&> 2\gep,\\
Q\left\{ \log  Z_n \in [a, a+\gb\gep(s/b)^{n/2})\right\}\cap\left\{\sum_{z\in D_n} \go_z\le 0\right\}&> 2\gep.
\end{split}\end{equation}
We assume that the first line is true. We consider the events related to $Q$ as sets of environments $(\go_z)_{z\in D_n\setminus \{A,B\}}$. We define

\begin{equation}
A_\gep=\left\{ \log Z_n \in [a, a+\gb\gep b^{-n/2})\right\}\cap\left\{\sum_{z\in D_n} \go_z\ge 0\right\},
\end{equation}
and

\begin{equation}
A^{(i)}_\gep=Q\left\{ \log Z_n \in [a-i\gb\gep(s/b)^{n/2}, a-(i-1)\gb\gep(s/b)^{n/2})\right\}.
\end{equation}
Define $\gd=\frac{s^{n/2}}{(s^n-1)b^{n/2}}$.  We define the measure $\tilde Q_{i,\gep}$ with its density:
\begin{equation}
\frac{\dd \tilde Q_{i,\gep}}{\dd Q}(\go):=\exp\left(\left[i\gep\gd^2 \sum_{z\in D_n} \go_z\right]-\frac{i^2\gep^2\gd^2|D_n\setminus\{A,B\}|}{2}\right).
\end{equation}
If the environment $(\go_z)_{z\in D_n}$ has law $Q$ then $(\hat \go_z^{(i)})_{z\in D_n}$ defined by
\begin{equation}
 \hat \go_z^{(i)}:= \go_z+\gep i \gd,
\end{equation}
 has law $ \tilde Q_{i,\gep}$. Going from $\go$ to $\hat \go^{(i)}$, one increases the value of the Hamiltonian by $\gep i(s/b)^{n/2}$ (each path cross $s^n-1$ sites).
Therefore if $(\hat\go^{(i)}_z)_{z\in D_n}\in A_{\gep}$, then $(\go_z)_{z\in D_n}\in A^{(i)}_{\gep}$.
From this we have $\tilde Q_{i,\gep} A_\gep \le Q A^{(i)}_{\gep}$, and therefore
\begin{equation}
Q A^{(i)}_{\gep}\, \ge \int_{A_\gep} \frac{\dd \tilde Q_{i,\gep}}{\dd Q}Q(\dd \go)\ge \exp(-(\gep i)^2/2)Q(A_\gep).
\end{equation}
The last inequality is due to the fact that the density is always larger than $\exp(-(\gep i)^2/2)$ on the set $A_\gep$ (recall its definition and the fact that $|D_n\setminus\{A,B\}|\gd^2\le 1$).
Therefore, in our setup, we have 
\begin{equation}
Q A^{(i)}_{\gep}> \gep, \quad  \forall i\in[0,\gep^{-1}].
\end{equation}
As the $A^{(i)}_{\gep}$ are disjoints, this is impossible.
If we are in the second case of \eqref{abcd}, we get the same result by shifting the variables in the other direction. \qed

\subsection{Proof of Proposition \ref{fluc2}}

Let us suppose that there exist $n$, $\gep$ and $a$ such that
\begin{equation}
Q\left\{ \log Z_n \in [a, a+\gb\gep\sqrt{n})\right\}> 8\gep.
\end{equation}
We define $\delta_{i,n}=\delta_i:=\gep s^{1-i}(s-1)^{-1}n^{-1/2}$.
Then one of the following inequality holds (recall the definition of $V_i$)

\begin{equation}\begin{split} \label{fghi}
Q\left\{ \log Z_n \in [a, a+\gb\gep\sqrt{n})\right\}\cap\left\{\sum_{i=1}^{n}\gd_i\sum_{z\in V_i}\go_z\ge 0\right\}&> 4\gep,\\
Q\left\{ \log Z_n \in [a, a+\gb\gep\sqrt{n})\right\}\cap\left\{\sum_{i=1}^{n}\gd_i\sum_{z\in V_i}\go_z\le 0\right\}&> 4\gep.
\end{split}\end{equation}
We assume that the first line holds and define

\begin{equation}
A_{\gep}=\left\{ \log Z_n \in [a, a+\gb\gep\sqrt{n})\right\}\cap\left\{\sum_{i=1}^{n}\gd_i\sum_{z\in V_i}\go_z\ge 0\right\}
\end{equation}
And
\begin{equation}
A_{\gep}^{(j)}=\left\{ \log Z_n \in [a-j\gb\gep\sqrt{n}, a-(j-1)\gb\gep\sqrt{n})\right\}
\end{equation}

\begin{equation}
j\gb\sum_{i=1}^n\gd_i (s-1)s^{i-1}=j\gb\gep \sqrt{n}.
\end{equation}
Therefore, an environment $\go\in A_{\gep}$ will be transformed in an environment in $A^{(j)}_{\gep}$.

We define $\tilde Q_{j,\gep}$ the measure whose Radon-Nicodyn derivative with respect to $Q$ is
\begin{equation}
\frac{\dd \tilde Q_{i,\gep}}{\dd Q}(\go):=\exp\left(\left[j\sum_{i=1}^n\gd_i \sum_{z\in V_i} \go_z\right]-\sum_{i=1}^{n}\frac{j^2\gd_i^2|V_i|}{2}\right).
\end{equation}
We can bound the deterministic term.
\begin{equation}
\sum_{i=1}^{n}\frac{j^2\gd_i^2|V_i|}{2}=j^2\gep^2 \sum_{i=1}^n \frac{s}{2(s-1)}\le j^2\gep^2.
\end{equation}
For an environment $(\go_z)_{z\in D_n\setminus\{A,B\}}$, define $(\hat \go_z^{(j)})_{z\in D_n\setminus\{A,B\}}$ by
\begin{equation}
 \hat \go_z^{(j)}:= \go_z+j\gep\gd_i, \quad   \forall z \in V_i.
\end{equation}
If $(\go_z)_{z\in D_n\setminus\{A,B\}}$ has $Q$, then $(\hat \go_z^{(j)})_{z\in D_n\setminus\{A,B\}}$ has law $\tilde Q_{j,\gep}$.
When one goes from $\go$ to $\hat\go^{(j)}$, the value of the Hamiltonian is increased by
\begin{equation*}
 \sum_{i=1}^n j\gd_i s^{i-1}(s-1)=\gep \sqrt{n}.
\end{equation*}
Therefore, if $\hat \go^{(j)}\in A_{\gep}$, then $\go\in A^{(j)}_{\gep}$, so that 
\begin{equation*}
 Q A^{(j)}_{\gep}\ge \tilde Q_{j,\gep} A_{\gep}.
\end{equation*}
Because of the preceding remarks
\begin{equation}
Q A^{(j)}_{\gep}\ge \tilde Q_{j,\gep} A_{\gep}=\int_{A_{\gep}}\frac{\dd \tilde Q_{i,\gep}}{\dd Q}Q(\dd\go) \ge \exp\left(-j^2\gep^2\right)Q A_{\gep}.
\end{equation}
The last inequality comes from the definition of $A_{\gep}$ which gives an easy lower bound on the Radon-Nicodyn derivative.
For $j\in [0,(\gep/2)^{-1}]$, this implies that $Q A^{(j)}_{\gep}> 2\gep$. As they are disjoint events this is impossible. The second case of \eqref{fghi} can be dealt analogously. \qed

\subsection{Proof of Corollary \ref{locloc}}

Let $g\in \gG_n$ be a fixed path. For $m\ge n$, define
\begin{equation}
 Z_m^{(g)}:=\sum_{\{ g'\in \gG_m: g|_n=g\}}\exp\left(\gb H_m(g')\right).
\end{equation}
With this definition we have
\begin{equation}
 \mu_m(\gamma|_n=g)=\frac{Z_m^{(g)}}{Z_m}.
\end{equation}
To show our result, it is sufficient to show that for any constant $K$ and any distinct $g,g'\in \gG_n$
\begin{equation}
 \lim_{m\to\infty}Q\left( \frac{\mu_m(\gamma|_n=g)}{\mu_m(\gamma|_n=g')}\in[K^{-1},K]\right)=0.
\end{equation}
For $g$ and $g'$ distinct, it is not hard to see that
\begin{equation}
 \log\left( \frac{\mu_m(\gamma|_n=g)}{\mu_m(\gamma|_n=g')}\right)=\log Z_m^{(g)}-\log Z_m^{(g')}=: \log Z^{(0)}_{m-n}+X,
\end{equation}
where $Z^{(0)}_{m-n}$ is a random variable whose distribution is the same as the one of $Z_{m-n}$, and $X$ is independent of $Z^{(0)}_{m-n}$.
We have
\begin{multline}
 Q\left(\log \left(\frac{\mu_m(\gamma|_n=g)}{\mu_m(\gamma|_n=g')}\right)\in [-\log K, \log K]\right)\\
=Q \left[ Q\left(\log Z^{(0)}_{m-n}\in [-\log K -X, \log K -X] \, \big| \, X\right)\right]
 \\
\le \max_{a\in \R} Q \left(\log Z_{m-n}\in [a,a+2\log K]\right).
\end{multline}
Proposition \ref{fluctuations} and \ref{fluc2} show that the right--hand side tends to zero. \qed





\section{The weak disorder polymer measure}\label{weakpol}

Comets and Yoshida introduced in \cite{cf:CY} an infinite
volume Markov chain at weak disorder that corresponds in some sense to the limit of the polymers
measures $\mu_n$ when $n$ goes to infinity. We perform the same construction here.
The notation is more cumbersome in our setting.

Recall that $\Gamma_n$ is the space of directed paths from $A$ to $B$ in $D_n$. Denote by $P_n$ the uniform law on $\gG_n$. For $g\in \Gamma_n$,
$0\leq t \leq s^n-1$, define $W_{\infty}(g_t,g_{t+1})$ by performing the same construction that leads to $W_{\infty}$, but taking $g_t$ and $g_{t+1}$ instead of $A$ and $B$ respectively. On the classical directed polymers on $\Z^d$, this would be equivalent to take the $(t,g_t)$ as the initial point of the polymer.

We can now define the weak disorder polymer measure for $\beta < \beta_0$. We define $\gG$ as the projective limit of $\gG_n$ (with its natural topology), the set of path on $D:=\bigcup_{n\ge 1} D_n$. As for finite path, we can define, for $\bar g\in \gG$, its projection onto $\gG_n$, $\bar g|_n$.
We define

\begin{eqnarray}\label{wdpol}
\mu_{\infty}(\bar \gamma|_n =g):= \frac{1}{W_{\infty}} \exp \{ \beta H_n(g) - (s^n-1) \lambda(\beta) \} \prod^{s^n-1}_{i=0} W_{\infty}({g}_i,{g}_{i+1}) \, P_n({\bar \gamma|_n = g}).
\end{eqnarray}

Let us stress the following:

\begin{itemize}
      
\item Note that the projection on the different $\Gamma_n$ are consistent (so that our definition makes sense)
  
      \begin{eqnarray*}
      \mu_{\infty}(\bar \gamma|_n = g)= \mu_{\infty}\left((\bar \gamma|_{n+1})|_n = g\right).
      \end{eqnarray*}

\item Thanks to the martingale convergence for both the numerator and the
      denominator, for any ${\bf s} \in \Gamma_n$,
      
      \begin{eqnarray*}
      \lim_{k\to +\infty} \mu_{k+n}(\gamma|_n = g) = \mu_{\infty}(\bar \gamma|_n = g).
      \end{eqnarray*}
Therefore, $\mu_{\infty}$ is the only reasonable definition for the limit of $\mu_n$.
      


\end{itemize}

\vspace{3ex}

It is an easy task to prove the law of large numbers for the time-averaged quenches mean of the energy.
This follows as a simple consequence of the convexity of $p(\beta)$.

\begin{proposition}
At each point where $p$ admits a derivative,
 \begin{eqnarray*}
   \lim_{n\to +\infty} \frac{1}{s^n} \mu_{n}(H_n(\gamma)) \to  p'(\beta), \quad Q-a.s..
 \end{eqnarray*}
\end{proposition}
\begin{proof}
 It is enough to observe that 
 \begin{eqnarray*}
  \frac{d}{d \beta} \log Z_n = \mu_n(H_n(\gamma)),
 \end{eqnarray*}
 
 \noindent then use the convexity to pass to the limit.
\end{proof}

\vspace{2ex}

We can also prove a quenched law of large numbers under our infinite volume
measure $\mu_{\infty}$, for almost every environment. The proof is very easy,
as it involves just a second moment computation. 

\begin{proposition}
 At weak disorder,
 \begin{eqnarray*}
   \lim_{n\to +\infty}\frac{1}{s^n}H_n(\bar \gamma|_n)=\lambda'(\beta),\quad \mu_{\infty}-a.s., \, Q-a.s..
 \end{eqnarray*}
\end{proposition}

\begin{proof}
We will consider the following auxiliary measure (size biased measure) on the environment

\begin{eqnarray*}
\overline{Q}(f(\go)) = Q(f(\go) W_{+\infty}).
\end{eqnarray*}

\noindent So, $Q$-a.s. convergence will follow from $\overline{Q}$-a.s. convergence.
This will be done by a direct computation of second moments. Let us write 
$\Delta = Q(\omega^2 e^{\beta \omega - \lambda(\beta)})$.

\begin{eqnarray*}
&\ &\overline{Q} \left( \mu_{\infty}(|H_n(\bar \gamma|_n)|^2) \right)\\
&=& Q \left[ P_n\left( |H_n(\gamma)|^2 \exp \{ \beta H_n(\gamma)- (s^n-1) \lambda(\beta)\} \prod^{s^n-1}_{i=0} W_{\infty}(\gamma_i,\gamma_{i+1})\right)\right]\\
&=& Q \left[ P_n( |H_n(\gamma)|^2 \exp \{ \beta H_n(\gamma)- (s^n-1) \lambda(\beta)\})\right]\\
&=& Q \left[ P_n( |\sum^{s^n}_{t=1}\omega(\gamma_t)|^2 \exp \{ \beta H_n(\gamma)- (s^n-1) \lambda(\beta)\})\right]\\
&=& Q\left[ \sum^{s^n-1}_{t=1}P_n\left(|\omega(\gamma_t)|^2\exp \{ \beta H_n(\gamma)- (s^n-1) \lambda(\beta)\}\right)\right]\\
&\ &+ Q \left[\sum_{1\le t_1\neq t_2\le s^{n}-1}P_n\left(\omega(\gamma_{t_1})\omega(\gamma_{t_2})\exp \{ \beta H_n(\gamma)- (s^n-1) \lambda(\beta)\}\right)\right]\\
&=& (s^n-1) \Delta+(s^n-1)(s^n-2)(\lambda'(\beta))^2,
\end{eqnarray*}

\noindent where we used independence to pass from line two to line three. So, recalling that $\overline{Q} ( \mu_{\infty}(H_n(\bar \gamma_n))= (s^n-1) \lambda'(\beta)$, we
have

\begin{eqnarray*}
&\ &\overline{Q} \left( \mu_{\infty}(|H_n(\gamma^{(n)}) - (s^n-1) \lambda'(\beta)|^2) \right)\\
&=& (s^n-1) \Delta+(s^n-1)(s^n-2)(\lambda'(\beta))^2 
- 2 (s^n-1) \lambda'(\beta)\overline{Q} \left( \mu_{\infty}(H_n(\bar \gamma_n)) \right)\\
&\ & \quad + \, (s^{n}-1)^2(\lambda'(\beta))^2\\
&=& (s^n-1) \left( \Delta - (\lambda'(\beta)^2)\right).
\end{eqnarray*}

\noindent Then

\begin{eqnarray*}
\overline{Q}  \mu_{\infty}\left( \left|\frac{H_n(\bar\gamma_n) - (s^n-1) \lambda'(\beta)}{s^n} \right|^2\right) 
\leq \frac{1}{s^n}\left( \Delta - (\lambda'(\beta)^2)\right),
\end{eqnarray*}

\noindent so the result follows by Borel-Cantelli.
\end{proof}

\section{Some remarks on the bond--disorder model}\label{bddis}

In this section, we shortly discuss, without going through the details, how the methods we used in this paper could be used (or could not be used) for the model of directed polymer on the same lattice with disorder located on the bonds.
\\

In this model to each bond $e$ of $D_n$ we associate i.i.d.\ random variables $\go_e$.
We consider each set $g\in \gG_n$ as a set of bonds and define the Hamiltonian as
\begin{equation}
H^{\omega}_n(g) = \sum_{e\in g} \go_e,
\end{equation}

\noindent  The partition function $Z_n$ is defined as
\begin{equation}
 Z_n:= \sum_{g\in \gG_n} \exp(\gb H_n(g)).
\end{equation}
One can check that is satisfies the following recursion 

\begin{equation}\begin{split}
Z_0& \, \stackrel{ \mathcal L}{=}\, \exp(\gb\go)\\
Z_{n+1}&\, \stackrel{\mathcal L)}{=}\, \sum_{i=1}^{b}Z_n^{(i,1)}Z_n^{(i,2)}\dots Z_n^{(i,s)}.
\end{split}
\end{equation}

\noindent where equalities hold in distribution and and $Z_n^{i,j}$ are i.i.d.\ distributed copies of $Z_n$.
Because of the loss of the martingale structure and the homogeneity of the Green function in this model
(which is equal to $b^{-n}$ on each edge), Lemma \ref{martingale} does not hold, and we cannot prove part $(iv)$ in Proposition \ref{misc}, Theorem \ref{th:ss} and Proposition \ref{fluc2} for this model. Moreover we have to change $b\le s$ by $b<s$ in $(v)$ of Proposition \ref{misc}.
Moreover, the method of the control of the variance would give us a result similar to \ref{th:ss} in this case
\begin{proposition}\label{ssbd}
When $b$ is equal to $s$, on can find constants $c$ and $\gb_0$ such that for all $\gb\le \gb_0$
\begin{equation}
0\le \lambda(\beta)-p(\beta) \le \exp\left(-\frac{c}{\gb^2}\right).
\end{equation}
\end{proposition}
However, we would not be able to prove that annealed and free energy differs at high temperature for $s=b$ using our method. The techniques used in \cite{cf:GLT_marg} or \cite{cf:Lac} for dimension $2$ should be able to tackle this problem, and show marginal disorder relevance in this case as well. 
\medskip

\chapter[Directed polymers in dimension $1+1$ and $1+2$]{New bounds for the free energy of directed polymers in dimension $1+1$ and $1+2$}\label{DirPOL12}

\section{Introduction}
\subsection{The model}
We study a directed polymer model introduced by Huse and Henley (in dimension $1+1$) \cite{cf:HH} with the purpose of investigating  impurity-induced domain-wall roughening in the 2D-Ising model. The first mathematical study of directed polymers in random environment was made by Imbrie and Spencer \cite{cf:IS}, and was followed by numerous authors \cite{cf:IS, cf:B, cf:AZ, cf:SZ, cf:CH_ptrf, cf:CSY, cf:CY, cf:CH_al, cf:CV, cf:V} (for a review on the subject see \cite{cf:CSY_rev}). 
Directed polymers in random environment model, in particular, polymer chains in a solution with impurities.

In our set--up the polymer chain is the graph $\{(i,S_i)\}_{1\le i\le N}$ of a nearest--neighbor path in $\Z^d$, $S$ starting from zero.
 The equilibrium behavior of this chain is described by a measure on the set of paths:
the impurities enter the definition of the measure as {\sl disordered potentials}, given by a typical realization of a field of i.i.d.\ random variables $\go=\{\go_{(i,z)}\ ;\ i\in \N, z\in \Z^d\}$ (with associated law $Q$). The polymer chain will tend to be attracted by larger values of the environment and repelled by smaller ones.
More precisely, we define the Hamiltonian 
\begin{equation}
 H_N(S):=\sum_{i=1}^N \go_{i,S_i}.
\end{equation}
We denote by $P$ the law of the simple symmetric random walk on $\Z^d$ starting at $0$ (in the sequel $P f(S)$, respectively $Q g(\go)$, will denote the expectation with respect to $P$, respectively Q).
One defines the polymer measure of order $N$ at inverse temperature $\gb$ as
\begin{equation}
 \mu^{(\gb)}_N(S)=\mu_N(S):=\frac{1}{Z_N}\exp\left(\gb H_N(S)\right)P(S),
\end{equation}
where $Z_N$ is the normalization factor which makes $\mu_N$ a probability measure
\begin{equation}
 Z_N:=P \exp\left(\gb H_N(S)\right).
\end{equation}
We call $Z_N$ the {\sl partition function} of the system. In the sequel, we will consider the case of $\go_{(i,z)}$ with zero mean and unit variance and such that there exists $B \in(0,\infty]$ such that
\begin{equation}\label{expm}
 \gl(\gb)=\log Q \exp(\gb\go_{(1,0)})<\infty, \quad \text{ for } 0\le \gb \le B. 
\end{equation}
Finite exponential moments are required to guarantee that $Q Z_N<\infty$.
The model can be defined and it is of interest also with environments with heavier tails (see e.g.\ \cite{cf:V}) but we will not consider these cases here.

\subsection{Weak, strong and very strong disorder}

In order to understand the role of disorder in the behavior of $\mu_N$, as $N$ becomes large, let us observe that, when $\gb=0$, $\mu_N$ is the law of the simple random walk, so that we know that, properly rescaled, the polymer chain will look like the graph of a $d$-dimensional Brownian motion. The main questions that arise for our model for $\gb>0$ are whether or not the presence of disorder breaks the diffusive behavior of the chain for large $N$, and how the polymer measure looks like when diffusivity does not hold.
\medskip

Many authors have studied diffusivity in polymer models: in \cite{cf:B}, Bolthausen remarked that the renormalized partition function $W_N:=Z_N/(Q Z_N)$ has a martingale property and proved the following zero-one law:
\begin{equation}
 Q\left\{\lim_{N\to\infty} W_N=0\right\}\in\{0,1\}.
\end{equation}
A series of paper  \cite{cf:IS,cf:B,cf:AZ,cf:SZ,cf:CY} lead to
\begin{equation}
  Q\left\{\lim_{N\to\infty} W_N=0\right\}=0 \Rightarrow  \text{ diffusivity },
\end{equation}
and a consensus in saying that this implication is an equivalence.
For this reason, it is natural and it has become customary to say that {\sl weak disorder} holds when $W_N$ converges to some non-degenerate limit and that {\sl strong disorder} holds when $W_N$ tends to zero.
\medskip

Carmona and Hu \cite{cf:CH_ptrf} and Comets, Shiga and Yoshida \cite{cf:CSY} \ proved that strong disorder holds for all $\gb$ in dimension $1$ and $2$. The result was completed by Comets and Yoshida \cite{cf:CY}: we summarize it here

\medskip
\begin{theorem}\label{strdis}
 There exists a critical value $\gb_c=\gb_c(d)\in[0,\infty]$ (depending of the law of the environment) such that
\begin{itemize}
 \item Weak disorder holds when $\gb<\gb_c$.
 \item Strong disorder holds when $\gb>\gb_c$.
\end{itemize}
Moreover:
\begin{equation}\begin{split}
 \gb_c(d)&=0 \text{ for } d=1,2\\
 \gb_c(d)&\in (0,\infty] \text{ for } d\ge 3. 
\end{split}
\end{equation}
\end{theorem}
We mention also that the case $\gb_c(d)=\infty$ can only occur when the random variable $\go_{(0,1)}$ is bounded.
\medskip

In \cite{cf:CH_ptrf} and \cite{cf:CSY} a characterization of strong disorder has been obtained in term of localization of the polymer chain: we cite the following result \cite[Theorem 2.1]{cf:CSY}
\medskip

\begin{theorem}
If $S^{(1)}$ and $S^{(2)}$ are two i.i.d.\ polymer chains, we have
\begin{equation}
 Q\left\{ \lim_{N\rightarrow\infty} W_N=0 \right\}=Q\left\{\sum_{N\ge 1} \mu_{N-1}^{\otimes 2}(S_N^{(1)}=S_N^{(2)})=\infty \right\}
\end{equation}
Moreover if   $Q\{ \lim_{N\rightarrow\infty} W_N=0 \}=1$ there exists a constant $c$ (depending on $\gb$ and the law of the environment) such that for 
\begin{equation}\label{localization}
-c\log W_N \le \sum_{n= 1}^N \mu_{n-1}^{\otimes 2}(S_n^{(1)}=S_n^{(2)})\le -\frac{1}{c} \log W_N.
\end{equation}
\end{theorem}
\medskip
One can notice that \eqref{localization} has a very strong meaning in term of trajectory localization when $W_N$ decays exponentially:
it implies that two independent polymer chains tend to share the same endpoint with positive probability.
For this reason we introduce now the notion of free energy, we refer to \cite[Proposition 2.5]{cf:CSY} and \cite[Theorem 
3.2]{cf:CY} for the following result:
\medskip
\begin{proposition}
 The quantity
\begin{equation}
 p(\gb):=\lim_{N\to\infty}\frac{1}{N}\log W_N \ ,
\end{equation}
	exists $Q$-a.s., it is non-positive and non-random. We call it the {\sl free energy} of the model,
 and we have
\begin{equation}\label{Pn}
 p(\gb)=\lim_{N\to\infty}\frac{1}{N}Q\log W_N=:\lim_{N\to\infty} p_N(\gb).
\end{equation}
Moreover $p(\gb)$ is non-increasing in $\gb$.
\end{proposition}
\medskip

We stress that the inequality $p(\gb)\le 0$ is the standard {\sl annealing}  bound. In view on \eqref{localization}, it is natural to say that {\sl very strong disorder} holds whenever $p(\gb)<0$.
One can moreover define $\bar \gb_c(d)$ the critical value of $\gb$ for the free energy i.e.\ :
\begin{equation}
 p(\gb)<0 \Leftrightarrow \gb>\bar\gb_c(d).
\end{equation}
Let us stress that, from the physicists' viewpoint, $\bar\gb_c(d)$ is the natural critical point because it is a point of non-analyticity of the free energy (at least if $\bar \gb_c(d)>0$). In view of this definition, we obviously have $\bar\gb_c(d)\ge \gb_c(d)$. 
It is widely believed that $\bar\gb_c(d)=\gb_c(d)$, i.e.\ that there exists no intermediate phase where we have {\sl strong disorder} but not {\sl very strong disorder}. However, this is a challenging question:
Comets and Vargas \cite{cf:CV} answered it in dimension $1+1$ by proving that $\bar\gb_c(1)=0$. In this paper, we make their result more precise. Moreover we prove that $\bar \gb_c(2)=0$.

\subsection{Presentation of the results}

The first aim of this paper is to sharpen the result of Comets and Vargas on the $1+1$-dimensional case. In fact, we are going
to give a precise statement on the behavior of $p(\gb)$ for small $\gb$. Our result is the following
\medskip
\begin{theorem}\label{PASGAUSSI}
When $d=1$ and the environment satisfies \eqref{expm}, there exist constants $c$ and $\gb_0< B$ (depending on the distribution of the environment) such that for all $0\le \gb\le \gb_0$ we have
\begin{equation}
-\frac{1}{c} \gb^4[1+(\log \gb)^2] \le p(\gb)\le -c \gb^4.
\end{equation}
\end{theorem}
\medskip
We believe that the logarithmic factor in the lower bound is an artifact of the method. In fact, by using replica-coupling, we have been able to get rid of it in the Gaussian case.
\medskip
\begin{theorem}\label{GAUSSI}
When $d=1$ and the environment is Gaussian, there exists a constant $c$ such that for all $\gb \le 1$.
\begin{equation}
-\frac{1}{c}\gb^4 \le p(\gb)\le  -c \gb^4.
\end{equation}
\end{theorem}
\medskip
These estimates concerning the free energy give us some idea of the behavior of $\mu_N$ for small $\gb$. Indeed, Carmona and Hu in \cite[Section 7]{cf:CH_ptrf} proved a relation between $p(\gb)$ and the overlap (although their notation differs from ours).
This relation together with our estimates for $p(\gb)$ suggests that, for low $\gb$, the asymptotic contact fraction between independent polymers
\begin{equation}
 \lim_{N\to\infty} \frac{1}{N}\mu_N^{\otimes 2} \sum_{n=1}^N \ind_{\{S_n^{(1)}=S_n^{(2)}\}},
\end{equation}
behaves like $\gb^2$.
\medskip

The second result we present is that $\bar \gb_c(2)=0$. As for the $1+1$-dimensional case, our approach yields an explicit bound on $p(\gb)$ for $\gb$ close to zero.
\medskip

\begin{theorem}\label{1+2UPLD}
 When $d=2$, there exist constants $c$ and $\gb_0$ such that for all $\gb\le \gb_0$,
\begin{equation}
-\exp\left(-\frac{1}{c\gb^2}\right) \le p(\gb)\le -\exp\left(-\frac{c}{\gb^4}\right),
\end{equation}
so that
\begin{equation}
\bar \gb_c(2)=0,
\end{equation}
and $0$ is a point of non-analyticity for $p(\gb)$.
\end{theorem}
\medskip

\begin{rem}\rm After the appearance of this paper as a preprint, the proof of the above result has been adapted by Bertin \cite{cf:Bertin} to prove the exponential decay of the partition function for {\sl Linear Stochastic Evolution} in dimension $2$, a model that is a slight generalisation of directed polymer in random environment.
\end{rem}
\medskip

\begin{rem}\rm
 Unlike in the one dimensional case, the two bounds on the free energy provided by our methods do not match. We believe that the second moment method, that gives the lower bound is quite sharp and gives the right order of magnitude for $\log p(\gb)$.
The method developped in \cite{cf:kbodies} to sharpen the estimate on the critical point shift for pinning models at marginality adapted to the context of directed polymer should be able to improve the result, getting $p(\gb)\le -\exp(-c_{\gep} \gb^{-(2+\gep)})$ for all $\gb\le 1$ for any $\gep$. 
\end{rem}

\subsection{Organization of the paper}
The various techniques we use have been inspired by ideas used successfully for another polymer model, namely the polymer pinning on a defect line (see \cite{cf:T_cmp,cf:GLT,cf:DGLT,cf:T_cg,cf:GLT_marg}). 

However the ideas we use to establish lower bounds differ sensibly from the ones leading to the upper bounds. For this reason, we present first the proofs of the upper bound results in Section \ref{rough}, \ref{onedim} and \ref{twodim}. The lower bound results are proven in Section \ref{lb11}, \ref{dim11} and \ref{dim12}.
\medskip

To prove the lower bound results, we use a technique that combines the so-called {\sl fractional moment method} and change of measure. This approach has been first used for pinning model in \cite{cf:DGLT} and it has been
refined since in \cite{cf:T_cg,cf:GLT_marg}. In Section \ref{rough}, we prove a non-optimal upper bound for the free energy in the case of Gaussian environment in dimension $1+1$ to introduce the reader to this method. In Section \ref{onedim} we prove the optimal upper bound for arbitrary environment in dimension $1+1$, and in Section \ref{twodim} we prove our upper bound for the free energy in dimension $1+2$ which implies that {\sl very strong disorder} holds for all $\gb$. These sections are placed in increasing order of technical complexity, and therefore, should be read in that order.
\medskip

Concerning the lower--bounds proofs: Section \ref{lb11} presents a proof of the lower bound of Theorem \ref{PASGAUSSI}. The proof combines the second moment method and a directed percolation argument. In Section \ref{dim11} the optimal bound is proven for Gaussian environment, with a specific Gaussian approach similar to what is done in \cite{cf:T_cmp}. In Section \ref{dim12} we prove the lower bound for arbitrary environment in dimension $1+2$.
These three parts are completely independent of each other.

\section{Some warm up computations}\label{rough}

\subsection{Fractional moment}

Before going into the core of the proof, we want to present here the starting step that will be used repeatedly thourough Sections \ref{rough}, \ref{onedim} and \ref{twodim}.
We want to find an upper--bound for the quantity
\begin{equation}
 p(\gb)=\lim_{N\to \infty}\frac{1}{N} Q \log W_N. 
\end{equation}
However, it is not easy to handle the expectation of a $\log$, for this reason we will use the following trick . Let $\theta\in(0,1)$, we have (by Jensen inequality)

\begin{equation}
Q \log W_N=\frac{1}{\theta}Q\log W_N^{\theta}\le \frac{1}{\theta}\log Q W_N^{\theta}.
\end{equation}
Hence

\begin{equation}\label{momentfrac}
 p(\gb)\le \liminf_{N\to\infty} \frac{1}{\theta N}\log Q W_N^{\theta}.
\end{equation}
We are left with showing that the fractional moment $Q W_N^{\theta}$ decays exponentially which is a problem that is easier to handle.

\subsection{A non optimal upper--bound in dimension $1+1$}

To introduce the reader to the general method used in this paper, combining fractional moment and change of measure, we start by proving a non--optimal result for the free--energy, using a finite volume criterion. As a more complete result is to be proved in the next section, we restrict to the Gaussian case here. The method used here is based on the one  of \cite{cf:CV}, marorizing the free energy of the directed polymer by the one of multiplicative cascades. Let us mention that is has bee shown recently by Liu and Watbled \cite{cf:LWat} that this majoration is in a sense optimal, they obtained this result by improving the concentration inequality for the free energy.
\medskip

The idea of combining fractional moment with change of measure and finite volume criterion has been used with success for the pinning model in \cite{cf:DGLT}.

 \begin{proposition}\label{lb}
 There exists a constant $c$ such that for all $\gb\le 1$ 

\begin{equation}
 p(\gb)\le -\frac{ c\gb^4}{(|\log \gb|+1)^2}
\end{equation}
 \end{proposition}

\begin{proof}[Proof of Proposition \ref{lb} in the case of Gaussian environment]
For $\beta$ sufficiently small, we choose $n$ to be equal to $\left\lceil \frac{C_1 |\log \gb|^{2}}{\gb^4}\right\rceil$ for a fixed constant $C_1$ (here and thourough the paper for $x\in\R$, $\lceil x\rceil$, respectively $\lfloor x\rfloor$ will denote the upper, respectively the lower integer part of $x$) and define $\theta:=1-(\log n)^{-1}$.
For $x\in \Z$ we define
\begin{equation}
 W_n(x):= P  \exp\left(\sum_{i=1}^n [\gb\go_{(i,S_i)}-\gb^2/2]\right)\ind_{\{S_n=x\}}.
\end{equation}
Note that $\sum_{x\in \Z} W_n(x)=W_n$.
We use a statement which can be found in the proof of Theorem 3.3. in \cite{cf:CV}:

\begin{equation}
\log Q W_{nm}^{\theta}\le m \log Q \sum_{x\in \Z} [W_n(x)]^{\theta} \quad \forall m\in \N.
\end{equation}
This combined with \eqref{momentfrac} implies that

\begin{equation}
p(\gb)\le \frac{1}{\theta n}\log Q \sum_{x\in \Z} [W_n(x)]^{\theta}.
\end{equation}

Hence, to prove the result, it is sufficient to show that

\begin{equation}
 Q \sum_{x\in \Z} [W_n(x)]^{\theta}\le e^{-1}, \label{eq:aprouver}
\end{equation}
for our choice of $\theta$ and $n$.

In order to estimate $Q [W_n(x)]^{\theta}$ we use an auxiliary measure $\tilde Q$.
The region where the walk $(S_i)_{0\le i\le n}$ is likely to go is
$J_n=\left([1,n]\times[-C_2\sqrt{n},C_2\sqrt{n}]\right)\cap \N\times \Z$ where $C_2$ is a big constant.

We define $\tilde Q$ as the measure under which the $\go_{i,x}$ are still independent Gaussian variables with variance $1$, but
such that $\tilde Q \go_{i,x}=-\delta_n\ind_{(i,x)\in J_n}$ where $\gd_n=1/(n^{3/4}\sqrt{2C_2\log n})$. This measure is absolutely continuous with respect to $Q$ and

\begin{equation} \label{eq:shift}
 \frac{\dd \tilde Q}{\dd Q}=\exp\left(-\sum_{(i,x)\in J_n}\left[\delta_n\go_{i,x}+\frac{\gd_n^2}{2}\right]\right).
\end{equation}
Then we have for any $x\in \Z$, using the H\"older inequality we obtain,

\begin{equation}\label{boun}
 Q \left[W_n(x)^{\theta}\right]=\tilde Q \left[\frac{\dd Q}{\dd \tilde Q}\left(W_n(x)\right)^\theta\right]\le \left(\tilde Q \left[\left(\frac{\dd Q}{\dd \tilde Q}\right)^{\frac{1}{1-\theta}}\right]\right)^{1-\theta}\left( \tilde Q W_n(x)\right)^{\theta}.
\end{equation}
The first term on the right-hand side can be computed explicitly and is equal to

\begin{equation}\label{bbboun}
\left(Q \left(\frac{\dd Q}{\dd \tilde Q}\right)^{\frac{\theta}{1-\theta}}\right)^{1-\theta}= \exp\left(\frac{\theta\gd_n^2}{2(1-\theta)}\# J_n \right)\le e,
\end{equation}
where the last inequality is obtained by replacing $\gd_n$ and $\theta$ by their values (recall $\theta=1-(\log n)^{-1}$). Therefore combining \eqref{boun} and \eqref{bbboun} we get that

\begin{equation}
  Q \sum_{x\in \Z} \left(W_n(x)\right)^{\theta}\le e \sum_{|x|\le n} \left(\tilde Q W_n(x)\right)^{\theta}.
\end{equation}

To bound the right--hand side, we first get rid of the exponent $\theta$ in the following way:

\begin{multline}
 \sum_{|x|\le n} n^{-3 \theta}\left(\tilde Q W_n(x)\right)^{\theta}\le n^{-3\theta}\#\{x\in\Z,\ |x|\le n \text{ such that } \tilde Q W_n(x)\le n^{-3}\}\\
\quad +\sum_{|x|\le n}\ind_{\{\tilde Q W_n(x)> n^{-3}\}}  \tilde Q W_n(x) n^{3(1-\theta)}.
\end{multline}
If $n$ is sufficiently large ( i.e., $\gb$ sufficiently small) the first term on the right-hand side is smaller than $1/n$ so that

\begin{equation}
 \sum_{|x|\le n} \left(\tilde Q W_n(x)\right)^{\theta}\le \exp(3)\tilde Q W_n+\frac{1}{n}.
\end{equation}
We are left with showing that the expectation of $W_n$ with respect to the measure $\tilde Q$ is small.
It follows from the definition of $\tilde Q$ that

\begin{equation}
 \tilde Q W_n= P\exp\left(-\gb\delta_n\#\{i\ |\ (i,S_i)\in J_n\}\right),
\end{equation}
and therefore

\begin{equation}
 \tilde Q W_n\le P\{\text{the trajectory $S$ goes out of $J_n$}\}+\exp(-n\gb\gd_n).
\end{equation}

One can choose $C_2$ such that the first term is small, and the second term is equal to
$\exp(-\gb n^{1/4}/\sqrt{2C_2\log n})\le \exp(-C_1^{1/4}/4\sqrt{C_2})$ that can be arbitrarily small by
choosing $C_1$ large compared to $(C_2)^{1/2}$.
In that case \eqref{eq:aprouver} is satisfied and we have
\begin{equation}
 p(\gb)\le \frac{1}{\theta n}{\log e^{-1}}\le -\frac{\gb^4}{2 C_1|\log \gb|^2}
\end{equation}
for small enough $\gb$. 
\end{proof}

\section{Proof of the upper bound of Theorem \ref{PASGAUSSI} and \ref{GAUSSI}}\label{onedim}

The upper bound we found in the previous section is not optimal, and can be improved by replacing the finite volume criterion \eqref{eq:aprouver} by a more sophisticated coarse graining method. The technical advantage of the coarse graining we use, is that we will not have to choose the $\theta$ of the fractional moment close to $1$ as we did in the previous section and this is the way we get rid of the extra $\log$ factor we had.
The idea of using this type of coarse graining for the copolymer model appeared in \cite{cf:T_cg} and this has been a substantial source of inspiration for this proof.

We will prove the following result first in the case of Gaussian environment, and then adapt the proof to general environment.


\begin{proof}[Proof in the case of Gaussian environment]

Let $n$ be the smallest squared integer bigger than $C_3\gb^{-4}$ (if $\gb$ is small we are sure that $n\le 2C_3\gb^{-4}$).  The number $n$ will be used in the sequel of the proof as a scaling factor. Let $\theta<1$ be fixed (say $\theta=1/2$). We consider a system of size $N=n m$ (where $m$ is meant to tend to infinity).

Let $I_k$ denote the interval $I_k=[k\sqrt{n},(k+1)\sqrt n )$. In order to estimate $Q W_N^{\theta}$ we decompose $W_N$ according to the contribution of different families path:

\begin{equation}\label{eq:decompo}
W_{N}=\sum_{y_1, y_2,\dots, y_m\in \Z} \check{W}_{(y_1,y_2,\dots,y_m)}
\end{equation}
where 

\begin{equation}
\check{W}_{(y_1,y_2,\dots,y_m)}=P \exp \left[\sum_{i=1}^N \left(\gb\go_{i,S_i}-\frac{\gb^2}{2}\right)\ind_{\left\{S_{in}\in I_{y_i},\forall i= 1,\dots,m \right\}}\right].
\end{equation}

\begin{figure}[h]
\begin{center}
\leavevmode
\epsfysize =6.5 cm
\psfragscanon
\psfrag{O}[c]{\tiny{O}}
\psfrag{n}[c]{\tiny{$n$}}
\psfrag{2n}[c]{\tiny{$2n$}}
\psfrag{3n}[c]{\tiny{$3n$}}
\psfrag{4n}[c]{\tiny{$4n$}}
\psfrag{5n}[c]{\tiny{$5n$}}
\psfrag{6n}[c]{\tiny{$6n$}}
\psfrag{7n}[c]{\tiny{$7n$}}
\psfrag{8n}[c]{\tiny{$8n$}}
\psfrag{rn}[c]{\tiny{$+\sqrt{n}$}}
\psfrag{r2n}[c]{\tiny{$+2\sqrt{n}$}}
\psfrag{r3n}[c]{\tiny{$+3\sqrt{n}$}}
\psfrag{r4n}[c]{\tiny{$+4\sqrt{n}$}}
\psfrag{m1n}[c]{\tiny{$-\sqrt{n}$}}
\psfrag{m2n}[c]{\tiny{$-2\sqrt n$}}
\psfrag{m3n}[c]{\tiny{$-3\sqrt n$}}
\psfrag{1}[c]{\tiny{$1$}}
\psfrag{2}[c]{\tiny{$2$}}
\psfrag{3}[c]{\tiny{$3$}}
\psfrag{4}[c]{\tiny{$4$}}
\psfrag{5}[c]{\tiny{$5$}}
\psfrag{6}[c]{\tiny{$6$}}
\psfrag{7}[c]{\tiny{$7$}}
\psfrag{8}[c]{\tiny{$8$}}
\epsfbox{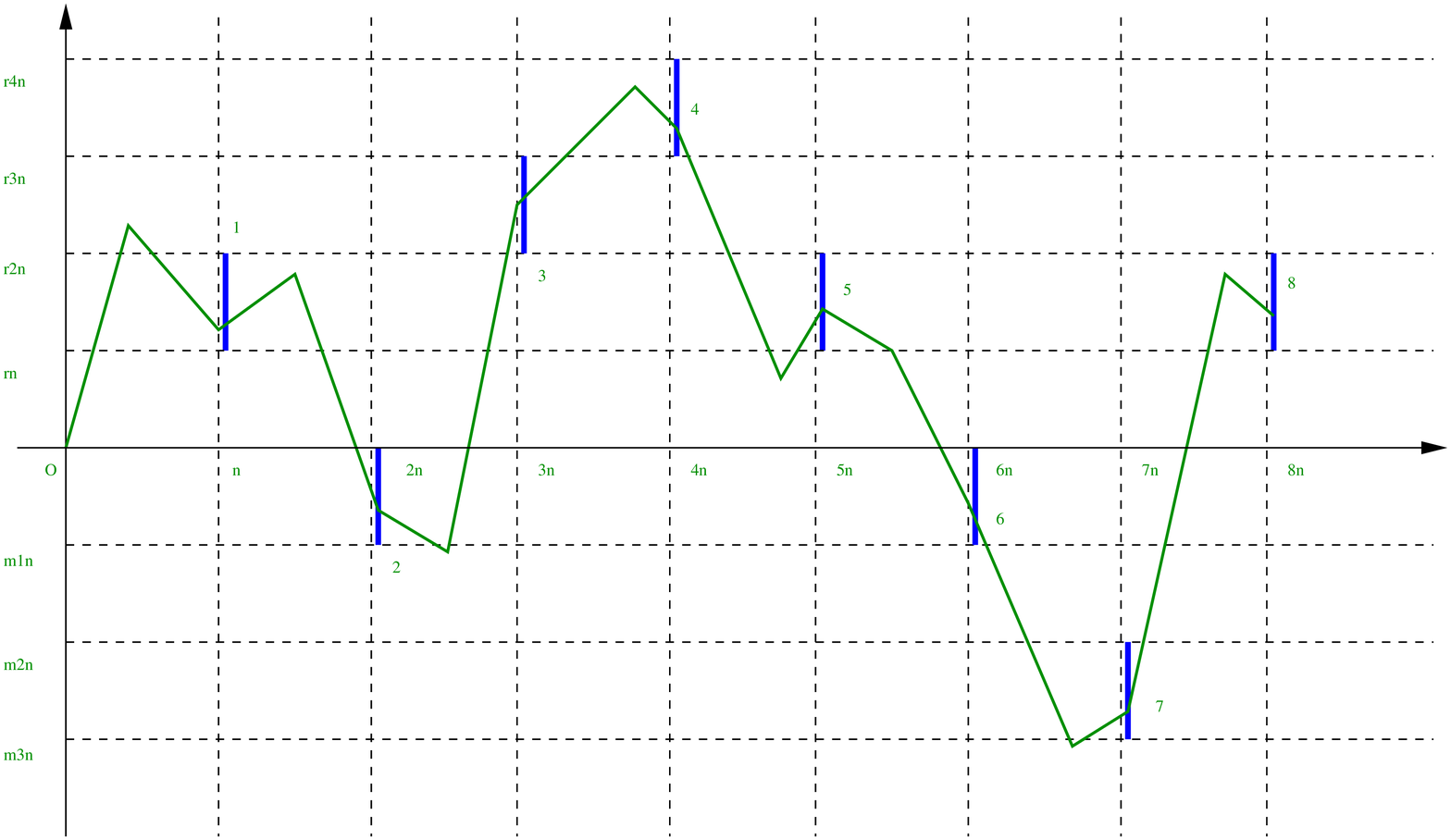}
\end{center}
\caption{\label{fig:wnn} The partition of $W_{nm}$ into $\check W^{(y_1,\dots,y_m)}$ is to be viewed as a coarse graining. For $m=8$, $(y_1,\dots,y_8)=(1,-1,2,3,1,-1,-3,1)$, $\check W_n^{(y_1,\dots,y_m)}$ corresponds to the contribution to $W_N$ of the path going through the thick barriers on the figure.}
\end{figure}

Then, we apply the inequality $\left(\sum a_i\right)^\theta\le \sum a_i^{\theta}$ (which holds for any finite or countable collection of positive real numbers) to this decomposition and average with respect to $Q$ to get,

\begin{equation}\label{eq:decompo2}
Q W_{nm}^{\theta}\le \sum_{y_1, y_2,\dots, y_m\in \Z} Q \check{W}_{(y_1,y_2,\dots,y_m)}^{\theta}.
\end{equation}
In order to estimate $Q\check{W}_{(y_1,y_2,\dots,y_m)}^{\theta}$, we use an auxiliary measure as in the previous section. The additional idea is to make the measure change depend on $y_1,\dots,y_m$.

For every $Y=(y_1,\dots,y_m)$ we define the set $J_Y$ as 
\begin{equation}
 J_Y:=\left\{ (km+i,y_k\sqrt{n}+z),\ k=0,\dots,m-1,\  i=1,\dots,n,\  |z|\le C_4\sqrt{n}\right\},
\end{equation}
where $y_0$ is equal to zero. Note that for big values of $n$ and $m$

\begin{equation} \label{cardj}
 \# J_Y\sim 2 C_4 m n^{3/2}
\end{equation}
We define the measure $\tilde Q_Y$ the measure under which the $\go_{(i,x)}$ are independent Gaussian variables with variance $1$ and mean $\tilde Q_Y \go_{(i,x)}=-\gd_n\ind_{\{(i,x)\in J_Y\}}$ where $\gd_n=n^{-3/4}C_4^{-1/2}$.
\begin{figure}[h]
\begin{center}
\leavevmode
\epsfysize =6.5 cm
\psfragscanon
\psfrag{O}[c]{\tiny{O}}
\psfrag{n}[c]{\tiny{$n$}}
\psfrag{2n}[c]{\tiny{$2n$}}
\psfrag{3n}[c]{\tiny{$3n$}}
\psfrag{4n}[c]{\tiny{$4n$}}
\psfrag{5n}[c]{\tiny{$5n$}}
\psfrag{6n}[c]{\tiny{$6n$}}
\psfrag{7n}[c]{\tiny{$7n$}}
\psfrag{8n}[c]{\tiny{$8n$}}
\psfrag{rn}[c]{\tiny{$+\sqrt{n}$}}
\psfrag{r2n}[c]{\tiny{$+2\sqrt{n}$}}
\psfrag{r3n}[c]{\tiny{$+3\sqrt{n}$}}
\psfrag{r4n}[c]{\tiny{$+4\sqrt{n}$}}
\psfrag{m1n}[c]{\tiny{$-\sqrt{n}$}}
\psfrag{m2n}[c]{\tiny{$-2\sqrt n$}}
\psfrag{m3n}[c]{\tiny{$-3\sqrt n$}}
\psfrag{1}[c]{\tiny{$1$}}
\psfrag{2}[c]{\tiny{$2$}}
\psfrag{3}[c]{\tiny{$3$}}
\psfrag{4}[c]{\tiny{$4$}}
\psfrag{5}[c]{\tiny{$5$}}
\psfrag{6}[c]{\tiny{$6$}}
\psfrag{7}[c]{\tiny{$7$}}
\psfrag{8}[c]{\tiny{$8$}}
\psfrag{ZMO}[c]{\small{Region where the environment is modified}}
\epsfbox{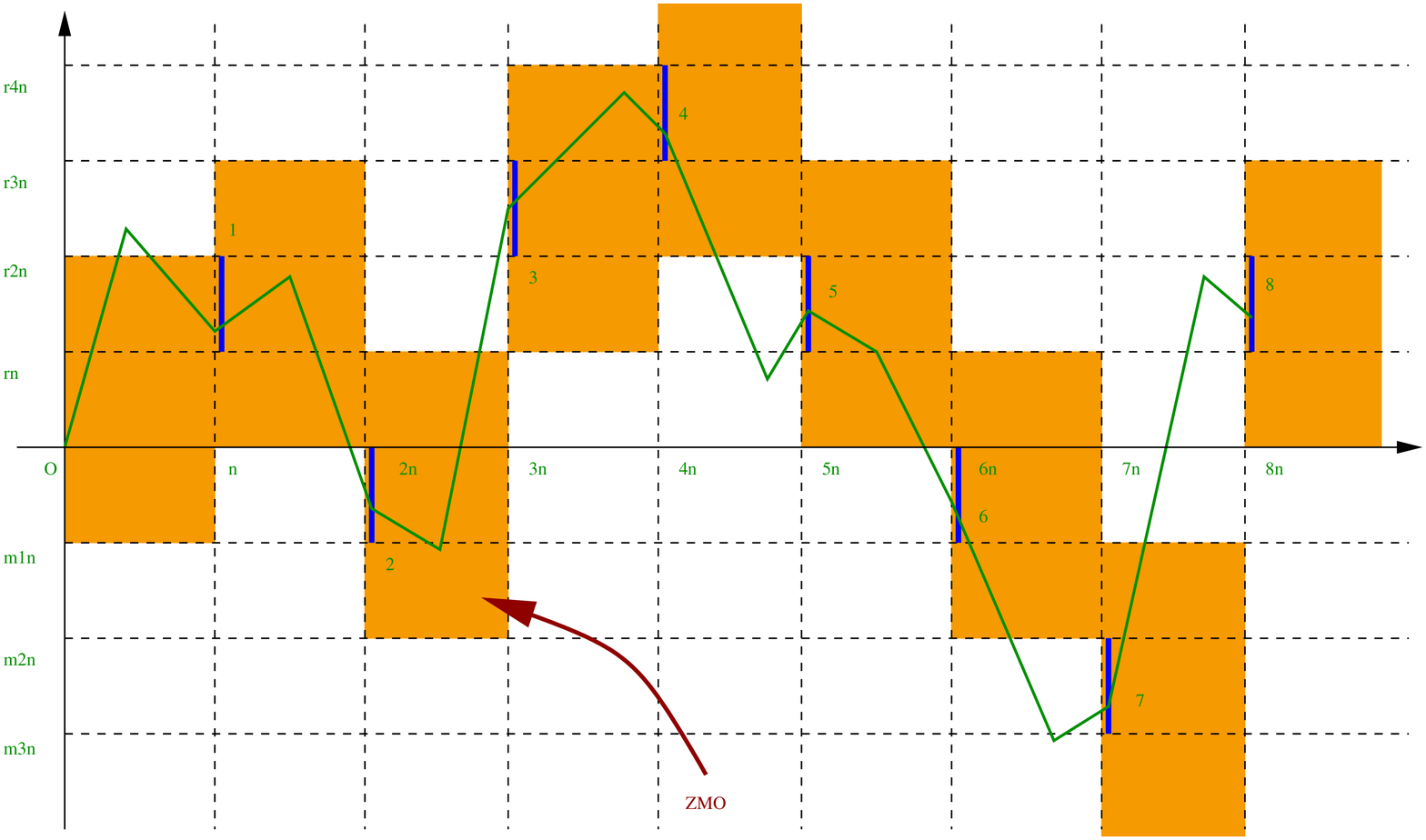}
\end{center}
\caption{\label{fig:wnn2} This figure represent in a rough way the change of measure $Q_Y$. The region where the mean of $\go_{(i,x)}$ is lowered (the shadow region on the figure) corresponds to the region where the simple random walk is likely to go, given that it goes through the thick barriers.}
\end{figure}
The law $\tilde Q_Y$ is absolutely continuous with respect to $Q$ and its density is equal to

\begin{equation}\label{qy1}
\frac{\dd\tilde Q_Y}{\dd Q}(\go)=\exp\left(-\sum_{(i,x)\in J_Y}\left[\gd_n\go_{(i,x)}+\gd_n^2/2\right]\right).
\end{equation}
Using H\"older inequality with this measure as we did in the previous section, we obtain

\begin{multline}
 Q \left[\check W_{(y_1,y_2,\dots,y_m)}^{\theta}\right]=\tilde Q_{Y} \left[\frac{\dd Q}{\dd \tilde Q_{Y}}\check W_{(y_1,y_2,\dots,y_m)}^{\theta}\right]\\
 \le \tilde Q_{Y}\left( \left[\left(\frac{\dd Q}{\dd \tilde Q_{Y}}\right)^{\frac{1}{1-\theta}}\right]\right)^{1-\theta}\left( \tilde Q_{Y} \check W_{(y_1,\dots,y_m)}\right)^{\theta} \label{eq:hhol}.
\end{multline}
The value of the first term can be computed explicitly

\begin{equation}
\left(Q \left[\left(\frac{\dd Q}{\dd \tilde Q_{Y}}\right)^{\frac{\theta}{1-\theta}}\right]\right)^{1-\theta}=\exp\left(\frac{\# J_Y\theta\gd_n^2}{2(1-\theta)}\right)\le \exp(3m), \label{eq:dens}
\end{equation}
where the upper bound is obtained by using the definition of $\gd_n$, \eqref{cardj} and the fact that $\theta=1/2$.

Now we compute the second term

\begin{equation}\label{ccek}
\tilde Q_{Y} \check W_{(y_1,\dots,y_m)}=P \exp\left(-\gb\gd_n\#\left\{i|(i,S_i)\in J_Y\right\}\right)\ind_{\{S_{kn}\in I_{y_k},\ \forall k \in [1,m]\}}.
\end{equation}
We define 
\begin{equation}\label{jayjay}\begin{split}
 J&:= \{(i,x),\ i=1,\dots,n,\ |x|\le C_4\sqrt{n}\}\\
 \bar J&:= \{(i,x),\ i=1,\dots,n,\ |x|\le (C_4-1)\sqrt{n}\}.
\end{split}
\end{equation}
Equation \eqref{ccek} implies that (recall that $P_x$ is the law of the simple random walk starting from $x$, and that we set $y_0=0)$
\begin{equation}\label{maxstart}
 \tilde Q_{Y} \check W_{(y_1,\dots,y_m)}\le \prod_{k=1}^m \max_{x\in I_0} P_x \exp\left(-\gb\gd_n\#\left\{i\ :\ (i,S_i)\in J\right\}\right)\ind_{\{S_{n}\in I_{y_k-y_{k-1}}\}}.
\end{equation}
Combining this with \eqref{eq:decompo}, \eqref{eq:hhol} and \eqref{eq:dens} we have

\begin{equation}
\log Q W_{N}^{\theta}\le m\left[3+\log \sum_{y\in\Z}\left(\max_{x\in I_0} P_x \exp\left(-\gb\gd_n\#\left\{i\ :\ (i,S_i)\in J\right\}\right)\ind_{\{S_{n}\in I_{y}\}}\right)^{\theta}\right].
\end{equation}
If the quantity in the square brackets is smaller than $-1$, by equation \eqref{momentfrac} we have  $p(\gb)\le -1/n$.
Therefore, to complete the proof it is sufficient to show that 
\begin{equation}
\sum_{y\in\Z}\left(\max_{x\in I_0} P_x \exp\left(-\gb\gd_n\#\left\{i\ :\ (i,S_i)\in J\right\}\right)\ind_{\{S_{n}\in I_{y}\}}\right)^{\theta} 
\end{equation}
is small. To reduce the problem to the study of a finite sum, we observe  (using some well known result on the asymptotic behavior of random walk) that given $\gep>0$ we can find $R$ such that

\begin{equation}\label{iop}
 \sum_{|y|\ge R}\left(\max_{x\in I_0} P_x \exp\left(-\gb\gd_n\#\left\{i\ :\ (i,S_i)\in J\right\}\right)\ind_{\{S_{n}\in I_{y}\}}\right)^{\theta}\le \sum_{|y|\ge R}\max_{x\in I_0} \left(P_x\{S_{n}\in I_{y}\}\right)^{\theta}\le \gep.
\end{equation}
To estimate the remainder of the sum we use the following trivial bound

\begin{multline}
 \sum_{|y|< R}\left(\max_{x\in I_0} P_x \exp\left(-\gb\gd_n\#\left\{i\ :\ (i,S_i)\in J\right\}\right)\ind_{\{S_{n}\in I_{y}\}}\right)^{\theta}\\
\le 2 R \left(\max_{x\in I_0} P_x \exp\left(-\gb\gd_n\#\left\{i\ :\ (i,S_i)\in J\right\}\right)\right)^{\theta}.
\end{multline}
Then we get rid of the $\max$ in the sum by observing that if a walk starting from $x$ makes a step in $J$, the walk with the same increments starting from $0$ will make the same step in $\bar J$ (recall \eqref{jayjay}).
\begin{equation} \label{fgh}
\max_{x\in I_0} P_x \exp\left(-\gb\gd_n\#\left\{i\ :\ (i,S_i)\in J\right\}\right)\le P \exp\left(-\gb\gd_n\#\left\{i|(i,S_i)\in \bar J\right\}\right).
\end{equation}
Now we are left with something similar to what we encountered in the previous section

\begin{equation} \label{jkl}
P \exp\left(-\gb\gd_n\#\left\{i\ :\ (i,S_i)\in \bar J\right\}\right)\le P\{\text{ the random walk goes out of $\bar J$ }\} + \exp(-n \gb \gd_n). 
\end{equation}
If $C_4$ is chosen large enough, the first term can be made arbitrarily small by choosing $C_4$ large, and the second is equal to $\exp(-C_3^{-1/4}/\sqrt{C_4})$ and can be made also arbitrarily small if $C_3$ is chosen large enough once $C_4$ is fixed.
An appropriate choice of constant and the use of \eqref{fgh} and \eqref{jkl} can  leads then to

\begin{equation} \label{klm}
 2 R \left(\max_{x\in I_0} P_x \exp\left(-\gb\gd_n\#\left\{i\ :\ (i,S_i)\in J\right\}\right)\right)^{\theta}\le \gep.
\end{equation}
This combined with \eqref{iop} completes the proof.
\end{proof}
\bigskip




\begin{proof}[Proof of the general case]
In the case of a general environment, some modifications have to be made in the proof above, but the general idea remains the same. In the change of measure one has to change the shift of the environment in $J_Y$ \eqref{qy1} by an exponential tilt of the measure as follow

\begin{equation}
\frac{\dd\tilde Q_{Y}}{\dd Q}(\gb)=\exp\left(-\sum_{(i,z)\in J_Y}\left[\gd_n\go_{(i,z)}+\gl(-\gd_n) \right]\right).
\end{equation}
The formula estimating the cost of the change of measure \eqref{eq:dens} becomes

\begin{equation}
\left(Q \left(\frac{\dd Q}{\dd \tilde Q_{Y}}\right)^{\frac{\theta}{1-\theta}}\right)^{1-\theta}=\exp\left(\# J_Y \left[(1-\theta)\gl\left(\frac{\theta\gd_n}{1-\theta}\right)+\theta\gl(-\gd_n)\right]\right)\le \exp(2 m),
\end{equation}
where the last inequality is true if $\gb_n$ is small enough if we consider that $\theta=1/2$ and use the fact
that $\gl(x)\stackrel{x\to 0}{\sim}x^2/2$ ($\go$ has $0$ mean and unit variance). The next thing we have to do is to compute the effect of this change of measure in this general case, i.e.\ find an equivalent for \eqref{ccek}.
When computing $\tilde Q_Y \check W_{(y_1,\dots,y_m)}$, the quantity 
\begin{equation}
\tilde Q_Y \exp(\gb\go_{1,0}-\gl(\gb))=\exp\left[\gl(\gb-\gd_n)-\gl(-\gd_n)-\gl(\gb)\right]
\end{equation}
appears instead of $\exp(-\gb\gd_n)$.
Using twice the mean value theorem, one gets that there exists $h$ and $h'$ in $(0,1)$ such that

\begin{equation}
 \gl(\gb-\gd_n)-\gl(-\gd_n)-\gl(\gb)=\gd_n\left[\gl'(-h\gd_n)-\gl'(\gb-h\gd_n)\right]=-\gb\gd_n\gl''(-h\gd_n+h'\gb).
\end{equation}
And as $\go$ has unit variance $\lim_{x\to 0} \gl''(x)=1$. Therefore if $\gb$ and $\gd_n$ are chosen small enough, the right-hand side of the above is less than $-\gb\gd_n/2$. So that \eqref{ccek} can be replaced by

\begin{equation}
\tilde Q_{Y} \check W_{(y_1,\dots,y_m)}\le P \exp\left(-\frac{\gb\gd_n}{2}\#\left\{i|(i,S_i)\in J_Y\right\}\right)\ind_{\{S_{kn}\in I_{y_k},\ \forall k \in [1,m]\}}.
\end{equation}

The remaining steps follow closely the argument exposed for the Gaussian case.
\end{proof}

\section{Proof of the upper bound in Theorem \ref{1+2UPLD}}\label{twodim}

In this section, we prove the main result of the paper: very strong disorder holds at all temperature in dimension $2$. 


\medskip

The proof is technically quite involved. It combines the tools of the two previous sections with a new idea for the change a measure: changing the covariance structure of the environment. We mention that this idea was introduced recently in \cite{cf:GLT_marg} to deal with the {\sl marginal disorder} case in pinning model. We choose to present first a proof for the Gaussian case, where the idea of the change a measure is easier to grasp.

\medskip

Before starting, we sketch the proof and how it should be decomposed in different steps:
\begin{itemize}
 \item [(a)] We reduce the problem by showing that it is sufficient to show that
 for some real number $\theta<1$, $Q W_N^{\theta}$ decays exponentially with $N$.
 \item [(b)] We use a coarse graining decomposition of the partition function by splitting it into different contributions that corresponds to trajectories that stays in a large corridor. This decomposition is similar to the one used in Section \ref{onedim}.
 \item [(c)] To estimate the fractional moment terms appearing in the decomposition, we change the law of the environment around the corridors corresponding to each contribution. More precisely, we introduce negative correlations into the Gaussian field of the environment. We do this change of measure in such a way that 
the new measure is not very different from the original one.
 \item [(d)] We use some basic properties of the random walk in $\Z^2$ to compute the expectation under the new measure.
\end{itemize}

\medskip

\begin{proof}[Proof for Gaussian environment]
 We fix $n$ to be the smallest squared integer bigger than $\exp(C_5/\gb^4)$ for some large constant $C_5$ to be defined later, for small $\gb$ we have $n\le \exp(2 C_5/\gb^4)$. The number $n$ will be used in the sequel of the proof as a scaling factor. 
For $y=(a,b)\in \Z^2$ we define 
$I_y=[a\sqrt{n},(a+1)\sqrt n-1]\times[b\sqrt{n},(b+1)\sqrt{n}-1]$ so that $I_y$ are disjoint and cover $\Z^2$.
For $N=nm$, we decompose the normalized partition function $W_N$ into different contributions, very similarly to what is done in dimension one (i.e.\ decomposition \eqref{eq:decompo2}), and we refer to the figure \ref{fig:wnn2} to illustrate how the decomposition looks like:

\begin{equation}
W_N=\sum_{y_1,\dots,y_m\in \Z^2}\check W_{(y_1,\dots,y_m)}
\end{equation}
where
\begin{equation}
\check{W}_{(y_1,\dots,y_m)}=P \exp \left(\sum_{i=1}^N \left[\gb\go_{i,S_i}-\gb^2/2\right] \right)\ind_{\left\{S_{in}\in I_{y_i},\forall i= 1,\dots,m \right\}}.
\end{equation}
We fix $\theta<1$ and apply the inequality $(\sum a_i)^{\theta}\le \sum a_i^{\theta}$ (which holds for any finite or countable collection of positive real numbers) to get

\begin{equation}\label{bobbob}
 Q W_N^{\theta}\le \sum_{y_1,\dots,y_m\in \Z^2}Q\check W_{(y_1,\dots,y_m)}^{\theta}.
\end{equation}

In order to estimate the different terms in the sum of the right--hand side in \eqref{bobbob}, we define some auxiliary measures $\tilde Q_Y$ on the the environment for every $Y=(y_0,y_1,\dots,y_m)\in \Z^{d+1}$ with $y_0=0$. We will choose the measures $Q_Y$ absolutely continuous with respect to $Q$. We use H\"older inequality to get the following upper bound:

\begin{equation}\label{hhold}
Q \check W_{(y_1,\dots,y_m)}^{\theta}\le  \left(Q \left(\frac{\dd Q}{\dd\tilde Q_{Y}}\right)^{\frac{\theta}{1-\theta}}\right)^{1-\theta} \left( \tilde Q_{Y} \check W_{(y_1,\dots,y_m)}\right)^{\theta}.
\end{equation}

Now, we describe the change of measure we will use. Recall that for the $1$-dimensional case we used a shift of the environment along the corridor corresponding to $Y$. The reader can check that this method would not give the exponential decay of $W_N$ in this case. Instead we change the covariance function of the environment along the corridor on which the walk is likely to go by introducing some negative correlation.

We introduce the change of measure that we use for this case. Given $Y=(y_0,y_1,\dots,y_m)$ we define $m$ blocks $(B_k)_{k\in[1,m]}$ and $J_Y$ their union (here and in the sequel, $|z|$ denotes the $l^{\infty}$ norm on $\Z^2$):

\begin{equation}\begin{split}\label{blocdef}
 B_k&:=\left\{(i,z)\in \N\times \Z^2 \, : \, \lceil i/n \rceil =k \text{ and } |z-\sqrt{n} y_{k-1}|\le C_6 \sqrt{n}\right\},\\
 J_Y&:=\bigcup_{k=1}^m B_k.
\end{split}
\end{equation}
We fix the covariance the field $\go$ under the law $\tilde Q_{Y}$ to be equal to

\begin{equation}\begin{split}\label{matvar}
&\tilde Q_Y \left(\go_{i,z}\go_{i,z'}\right)=\mathcal C^Y_{(i,z),(j,z')}\\
 &:=\begin{cases}
   \ind_{\{(i,z)=(j,z')\}}-V_{(i,z),(j,z')} & \text{ if } \exists\ k\in[1,m] \text{ such that } (i,z) \text{ and } (j,z')\in B_k
   \\
    \ind_{\{(i,z)=(j,z')\}} & \text{ otherwise,}
  \end{cases}\end{split}
\end{equation}
where

\begin{equation}\label{vdef}
 V_{(i,z),(j,z')}:=\begin{cases}
  0 & \text{ if } (i,z)=(j,z')
   \\
 \frac{\ind_{\{|z-z'|\le C_7 \sqrt{|j-i|}\}}}{100 C_6C_7n\sqrt{\log n}|j-i|} & \text{ otherwise.}
\end{cases}
\end{equation}
We define

\begin{equation}
\hat V:=(V_{(i,z),(j,z')})_{(i,z),(j,z')\in B_1}.
\end{equation}

One remarks that the so-defined covariance matrix  $\mathcal C^Y$ is block diagonal with $m$ identical blocks which are copies of $I-\hat V$ corresponding to the $B_k$, $k\in [1,m]$, and just ones on the diagonal elsewhere.
Therefore, the change of measure we describe here exists if and only if  $I-\hat V$ is definite positive. 

The largest eigenvalue for $\hat V$ is associated to a  positive vector and therefore is smaller than 
\begin{equation}\label{specray}
 \max_{(i,z)\in B_1}\sum_{(j,z')\in{B_1}} \left|V_{(i,z),(j,z')}\right|\le \frac{C_7}{C_6 \sqrt{\log n}}.
\end{equation}
For the sequel we choose $n$ such that the spectral radius of $\hat V$ is less than $(1-\theta)/2$ so 
that $I-\hat V$ is positive definite. 
With this setup, $\tilde Q_Y$ is well defined.

The density of the modified measure $\tilde Q_{Y}$ with respect to $Q$ is
given by
\begin{equation}
 \frac{\dd \tilde Q_Y}{\dd Q}(\go)=\frac{1}{\sqrt{\det \mathcal C^Y}}\exp\left(-\frac{1}{2}^t\go((\mathcal C^Y)^{-1}-I)\go\right),
\end{equation}
where 
\begin{equation}
^t\go M \go= \sum_{(i,z),(j,z')\in \N\times \Z^2}\go_{(i,z)}M_{(i,z),(j,z')}\go_{(j,z')},
\end{equation}
for any matrix $M$ of $(\N\times \Z^2)^2$ with finite support.

Then we can compute explicitly the value of the second term in the right-hand side of \eqref{hhold}
\begin{equation}\label{detei}
\left(Q\left(\frac{\dd Q}{\dd\tilde Q_Y}\right)^{\frac{\theta}{1-\theta}}\right)^{1-\theta}=\sqrt{\frac{\det \mathcal C^Y}{\det \left(\frac{\mathcal C^Y}{1-\theta}-\frac{\theta I}{1-\theta}\right)^{1-\theta}}}.
\end{equation}
Note that the above computation is right if and only if $\mathcal C^{Y} -\theta I$ is a definite positive matrix.
Since its eigenvalues are the same of those of $(1-\theta)I-\hat{V}$, this holds for large $n$ thanks to \eqref{specray}. 
Using again the fact that $\mathcal C^Y$ is composed of $m$ blocks identical to $I-\hat V$, we get from \eqref{detei}
\begin{equation}\label{detbl}
\left(Q\left(\frac{\dd Q}{\dd\tilde Q}\right)^{\frac{\theta}{1-\theta}}\right)^{1-\theta}=\left(\frac{\det(I
-\hat V)}{\det(I-\hat V/(1-\theta))^{1-\theta}}\right)^{m/2}.
\end{equation}
In order to estimate the determinant in the denominator, we compute the Hilbert-Schmidt norm of $\hat V$. One can check that for all $n$

\begin{equation} \label{hsnorm}
 \| \hat V \|^2= \sum_{(i,z),(j,z')\in B_1} V_{(i,z),(j,z')}^2\le 1.
\end{equation}
We use the inequality $\log( 1+x )\ge x-x^2$ for all $x\ge -1/2$ and the fact that the spectral radius of $\hat V/(1-\theta)$ is bounded by $1/2$ (cf. \eqref{specray}) to get that

\begin{equation}\label{truv}\begin{split}
 \det\left[I-\frac {\hat V}{1-\theta}\right]&=\exp\left(\tr\left(\log\left(I-\frac{\hat V}{1-\theta}\right)\right)\right)\ge \exp\left(-\frac{\| \hat V\|^2}{(1-\theta)^2}\right)\\
&\ge \exp\left(-\frac{1}{(1-\theta)^2}\right).
\end{split}\end{equation}
For the numerator, $\tr\ \hat V=0$ implies that that $\det(I-\hat V)\le 1$.
Combining this with \eqref{detbl} and \eqref{truv} we get

\begin{equation}\label{densityy}
\left(Q\left(\frac{\dd Q}{\dd\tilde Q_Y}\right)^{\frac{\theta}{1-\theta}}\right)^{1-\theta}\le \exp\left(\frac{m}{2(1-\theta)}\right).
\end{equation}
Now that we have computed the term corresponding to the change of measure, we estimate $\check W_{(y_1,\dots,y_m)}$ under the modified measure (just by computing the variance of the Gaussian variables in the exponential, using \eqref{matvar})\ :

\begin{multline}\label{modest}
 \tilde Q_{Y} \check W_{(y_1,\dots,y_m)}=P\, \tilde Q_Y \exp\left( \sum_{i=1}^N \left(\gb \go_{i,S_i}-\frac{\gb^2}{2}\right)\right)\ind_{\left\{S_{kn}\in I_{y_k},\forall k= 1,\dots,m \right\}}\\
=P \exp\left(\frac{\gb^2}{2}\sumtwo{1\le i,\ j\le N}{z,z'\in \Z^2} \left(\mathcal C^Y_{(i,z),(j,z')}-\ind_{\{(i,z)=(j,z')\}} \right)\ind_{\{S_i=z,S_j=z'\}}\right)\ind_{\left\{S_{kn}\in I_{y_k},\forall k= 1,\dots,m \right\}}.
\end{multline}
Replacing $\mathcal C^Y$ by its value we get that

\begin{multline}
 \tilde Q_{Y} \check W_{(y_1,\dots,y_m)}=P \exp  \left(-\frac{\gb^2}{2}\sumtwo{1\le i\neq j\le N}{1\le k\le m}\frac{\ind_{\{\left((i,S_i),(j,S_j)\right)\in B_k^2,\ |S_i-S_j|\le C_7\sqrt{|i-j|}\}}}{100C_6C_7 n \sqrt{\log n}|j-i|}\right)
\\
\ind_{\left\{S_{kn}\in I_{y_k},\forall k= 1,\dots,m \right\}}.
\end{multline}
Now we do something similar to $\eqref{maxstart}$: for each ``slice'' of the trajectory $(S_i)_{i\in[(m-1)k,mk]}$, we bound the contribution of the above expectation by maximizing over the starting point (recall that $P_x$ denotes the probability distribution of a random walk starting at $x$). Thanks to the conditioning, the starting point has to be in $I_{y_k}$. Using the translation invariance of the random walk, this gives us the following ($\vee$ stands for maximum):

\begin{multline} \label{rrrr} 
 \tilde Q_{Y} \check W^{(y_1,\dots,y_m)}\le \prod_{i=k}^m \max_{x\in I_0}\ P_x\bigg[\\
\exp\left(-\frac{\gb^2}{2}\sum_{1\le i\neq j\le n}\frac{\ind_{\{|S_i|\vee|S_j|\le C_6\sqrt{n}, \ |S_i-S_j|\le C_7\sqrt{|i-j|}\}}}{100 C_6 C_7 n \sqrt{\log n} |j-i|}\right)
\ind_{\left\{S_{n}\in I_{y_k-y_{k-1}}\right\}}\bigg].
\end{multline}
For trajectories $S$ of a directed random-walk of $n$ steps, we define the quantity 

\begin{equation}
G(S):=\sum_{1\le i\neq j\le n}\frac{\ind_{\{|S_i|\vee |S_j|\le C_6\sqrt{n}, \ |S_i-S_j|\le C_7\sqrt{|i-j|}\}}}{100 C_6 C_7 n \sqrt{\log n} |j-i|}.
\end{equation}
Combining \eqref{rrrr} with \eqref{densityy}, \eqref{hhold} and \eqref{bobbob}, we finally get

\begin{equation}\label{aaaaargh}
Q W_N^{\theta}\le \exp\left(\frac{m}{2(1-\theta)}\right)\left[\sum_{y\in \Z}\max_{x\in I_0} \left(P_x \exp\left(-\frac{\gb^2}{2}G(S)\right)\ind_{\{S_n\in I_{y}\}}\right)^{\theta}\right]^m.
\end{equation}
 The exponential decay of $Q W_N^{\theta}$ (with rate $n$) is guaranteed if we can prove that

\begin{equation}
 \sum_{y\in \Z}\max_{x\in I_0}\left(P_x \exp\left(-\frac{\gb^2}{2}G(S)\right)\ind_{\{S_n\in I_{y}\}}\right)^{\theta}
\end{equation}
is small. The rest of the proof is devoted to that aim.

We fix some $\gep>0$. Asymptotic properties of the simple random walk, guarantees that we can find $R=R_{\gep}$ such that
\begin{equation}\label{tr0}
  \sum_{|y|\ge R}\max_{x\in I_0}\left(P_x \exp\left(-\frac{\gb^2}{2}G(S)\right)\ind_{\{S_n\in I_{y}\}}\right)^{\theta}
\le \sum_{|y|\ge R} \max_{x\in I_0}\ \left(P_x\{S_n\in I_y\}\right)^{\theta}\le \gep.
\end{equation}
To estimate the rest of the sum, we use the following trivial and rough bound

\begin{equation}\label{tr10}
 \sum_{|y|< R}\max_{x\in I_0}\left[ P_x \exp\left(-\frac{\gb^2}{2}G(S)\right)\ind_{\{S_n\in I_{y}\}}\right]^{\theta}
\le R^2 \left[\max_{x\in I_0} P_x \exp\left(-\frac{\gb^2}{2}G(S)\right)\right]^{\theta}
\end{equation}
Then we use the definition of $G(S)$ to get rid of the $\max$ by reducing the width of the zone where we have negative correlation:

\begin{equation}
 \max_{x\in I_0} P_x \exp\left(-\frac{\gb^2}{2}G(S)\right)\le P \exp\left(-\frac{\gb^2}{2}\sum_{1\le i\neq j\le n}\frac{\ind_{\{|S_i|\vee |S_j|\le (C_6-1)\sqrt{n}, \ |S_i-S_j|\le C_7\sqrt{|i-j|}\}}}{100 C_6 C_7 n \sqrt{\log n} |j-i|}\right).
\end{equation}
We define $\bar B:=\{(i,z)\in \N\times \Z^2\ : \ i\le m, |z|\le (C_6-1)\sqrt{n}\}$. We get from the above that

\begin{multline}\label{tr11}
  \max_{x\in I_0} P_x \exp\left(-\frac{\gb^2}{2}G(S)\right)\le  P\{\text{the RW goes out of } \bar B \}\\
+
P \exp\left(-\frac{\gb^2}{2}\sum_{1\le i\neq j\le n}\frac{\ind_{\{ |S_i-S_j|\le C_7\sqrt{|i-j|}\}}}{100 C_6 C_7 n \sqrt{\log n} |j-i|}\right)
\end{multline}
One can make the first term of the right-hand side arbitrarily small by choosing $C_6$ large, in particular on can choose $C_6$ such that

 \begin{equation}\label{tr1}
 P\left\{ \max_{i\in[0,n]} |S_n|\ge (C_6-1)\sqrt{n}\right\}\le (\gep/R^2)^{\frac{1}{\theta}}.
\end{equation}
To bound the other term, we introduce the quantity

\begin{equation}
D(n):=\sum_{1\le i\neq j\le n}\frac{1}{n \sqrt{\log n} |j-i|},
\end{equation}
and the random variable $X$,

\begin{equation}\label{xdef}
 X:=\sum_{1\le i\neq j\le n}\frac{\ind_{\{|S_i-S_j|\le C_7\sqrt{|i-j|}\}}}{n \sqrt{\log n} |j-i|}.
\end{equation}
For any $\delta>0$, we can find $C_7$ such that $P(X)\ge( 1-\gd )D(n)$.
We fix $C_7$ such that this holds for some good $\gd$ (to be chosen soon), and by remarking that $0\le X \le D(n)$ almost surely, we obtain
(using Markov inequality)

\begin{equation}\label{xest}
 P\{ X>D(n)/2 \}\ge 1-2\delta.
\end{equation}
Moreover we can estimate $D(n)$ getting that for $n$ large enough

\begin{equation}\label{dest}
 D(n)\ge \sqrt{\log n}.
\end{equation}
Using \eqref{xest} and \eqref{dest} we get

\begin{equation}\begin{split}
 P \exp\left(-\frac{\gb^2}{2}\sum_{1\le i\neq j\le n}\frac{\ind_{\{ |S_i-S_j|\le C_7\sqrt{|i-j|}\}}}{100C_6 C_7 n \sqrt{\log n} |j-i|}\right)&=P \exp\left(-\frac{\gb^2}{200C_6C_7}X\right)\\
&\le 2\gd + \exp\left(-\frac{\gb^2\sqrt{\log n}}{200C_6C_7}\right).
\end{split}\end{equation}
Due to the choice of $n$ we have made (recall $n\ge \exp(C_5/\gb^4)$), the second term is less than
$\exp\left(-\gb^2 C_5^{1/2}/(200 C_6C_7)\right)$. We can choose $\gd$, $C_7$ and $C_5$ such that, the right-hand side is less that $(\gep/R^2)^{\frac{1}{\theta}}$.
This combined with \eqref{tr1}, \eqref{tr11}, \eqref{tr10} and \eqref{tr0} allow us to conclude that

\begin{equation}
 \sum_{y\in \Z}\max_{x\in I_0}\left(P_x \exp\left(-\frac{\gb^2}{2}G(S)\right)\ind_{\{S_n\in I_{y}\}}\right)^{\theta}\le 3\gep
\end{equation}
So that with a right choice for $\gep$, \eqref{aaaaargh} implies

\begin{equation}
Q W_N^{\theta}\le \exp(-m).
\end{equation}
Then \eqref{momentfrac} allows us to conclude that $p(\gb)\le -1/n$.

\end{proof}

\bigskip

\begin{proof}[Proof for general environment]
The case of general environment does not differ very much from the Gaussian case, but one has to a different approach for the change of measure in \eqref{hhold}. In this proof, we will largely refer to what has been done in the Gaussian case, whose proof should be read first.
\medskip

Let $K$ be a large constant.
One defines the function $f_K$ on $\R$ as to be
\begin{equation*}
 f_K(x)=-K\ind_{\{x>\exp(K^2)\}}.
\end{equation*}
Recall the definitions \eqref{blocdef} and \eqref{vdef}, and define $g_Y$ function of the environment as

\begin{equation*}
 g_Y(\go)=\exp\left(\sum_{k=1}^m f_K\left(\sum_{(i,z),(j,z')\in B_k} V_{(i,z),(j,z')}\go_{i,z}\go_{j,z'}\right)\right).
\end{equation*}
Multiplying by $g_Y$ penalizes by a factor $\exp(-K)$ the environment for which there is to much correlation in one block. This is a way of producing negative correlation in the environment.
For the rest of the proof we use the notation
\begin{equation}
 U_k:=\sum_{(i,z),(j,z')\in B_k} V_{(i,z),(j,z')}\go_{i,z}\go_{j,z'}
\end{equation}

We do a computation similar to \eqref{hhold} to get

\begin{equation}\label{hhold2}
 Q\left[\check W_{(y_1,\dots,y_m)}^{\theta}\right]\le \left(Q\left[ g_Y(\go)^{-\frac{\theta}{1-\theta}}\right]\right)^{1-\theta}\left( Q\left[ g_Y(\go)\check W_{(y_1,\dots,y_n)}\right]\right)^{\theta}.
\end{equation}
The block structure of $g_Y$ allows to express the first term as a power of $m$.

\begin{equation}
Q\left[ g_Y(\go)^{-\frac{\theta}{1-\theta}}\right]=\left(Q\left[\exp\left(-\frac{\theta}{1-\theta}f_K\left(U_1\right)\right)\right]\right)^m.
\end{equation}
Equation \eqref{hsnorm} says that
\begin{equation}
 \var_Q\left(U_1\right)\le 1.
\end{equation}
So that 
\begin{equation}
 P\left\{U_1\ge \exp(K^2)\right\}\le \exp(- 2 K^2),
\end{equation}
and hence
\begin{multline}\label{fsterm}
Q\left[\exp\left(-\frac{\theta}{1-\theta}f_K\left(U_1\right)\right)\right]\\
\le 1+\exp\left(-2K^2+\frac{\theta}{1-\theta}K\right)\le 2,
\end{multline}
if $K$ is large enough.
We are left with estimating the second term
\begin{equation}\label{secterm}
  Q\left[g_Y(\go)\check W_{(y_1,\dots,y_n)}\right]=P Q g_Y(\go)\exp\left(\sum_{i=1}^{nm}[\gb \go_{i,S_i}-\gl(\gb)] \right)\ind_{\{S_{kn}\in I_{y_k},\forall k=1\dots m\}}.
\end{equation}
For a fixed trajectory of the random walk $S$, we consider $\bar Q_S$ the modified measure on the environment with density

\begin{equation}
\frac{\dd \bar Q_S}{\dd Q}:=\exp\left(\sum_{i=1}^{nm}[\gb \go_{i,S_i}-\gl(\gb)] \right).
\end{equation}
Under this measure

\begin{equation}
 \bar Q_S \go_{i,z}=\begin{cases}
  0 & \text{ if } z\neq S_i
   \\
 Q \go e^{\gb\go_{0,1}-\gl(\gb)}:=m(\gb) & \text{ if } z=S_i.
\end{cases}
\end{equation}
As $\go_{1,0}$ has zero-mean and unit variance under $Q$, \eqref{expm} implies $m(\gb)=\gb+o(\gb)$ around zero and that $\var_{\bar Q _S} \go_{i,z}\le 2$ for all $(i,z)$ if $\gb$ is small enough. Moreover $\bar Q_S$ is a product measure, i.e.\ the $\go_{i,z}$ are independent variables under $Q_S$.
With this notation \eqref{secterm} becomes

\begin{equation}
 P \bar Q_S \left[g_Y(\go)\right]\ind_{\{S_{kn}\in I_{y_k},\forall k=1,\dots,m \}}.
\end{equation}
As in the Gaussian case, one wants to bound this by a product using the block structure. Similarly to \eqref{rrrr}, we use translation invariance to get the following upper bound

\begin{equation}
 \prod_{k=1}^m \max_{x\in I_0} P_x \bar Q_S \exp\left(f_K\left(U_1\right)\right)\ind_{\{S_n\in I_{y_k-y_{k-1}}\}}.
\end{equation}
Using this in \eqref{hhold2} with the bound \eqref{fsterm} we get the inequality

\begin{equation}
Q W_N^{\theta}\le 2^{m(1-\theta)}\!\! \left(\sum_{y\in \Z^2}\!\!  \left[\max_{x\in I_0} P_x \bar Q_S \exp\left(f_K\left(U_1\right)\right)\ind_{\{S_n\in I_{y}\}}\right]^{\theta}\right)^m\!\!\!\!.
\end{equation}
Therefore to prove exponential decay of $Q W_N^{\theta}$, it is sufficient to show that 

\begin{equation}
\sum_{y\in \Z^2} \left[\max_{x\in I_0} P_x \bar Q_S \exp\left(f_K\left(U_1\right)\right)\ind_{\{S_n\in I_{y}\}}\right]^{\theta}
\end{equation}
is small. As seen in the Gaussian case ( cf.\ \eqref{tr0},\eqref{tr10}), the contribution of $y$ far from zero can be controlled and therefore it is sufficient for our purpose to check
\begin{equation}
 \max_{x\in I_0} P_x \bar Q_S \exp\left(f_K\left(U_1\right)\right)\le \gd,
\end{equation}
for some small $\gd$. Similarly to \eqref{tr11}, we force the walk to stay in the zone where the environment is modified by writing
\begin{multline}\label{foor}
\max_{i\in I_0} P_x \bar Q_S \exp\left(f_K\left(U_1\right)\right)\le 
P\{ \max_{i\in[0,n]}|S_i|\ge (C_6-1)\sqrt{n}\}\\
+\max_{x\in I_0} P_x \bar Q_S \exp\left(f_K\left(U_1\right)\right)\ind_{\{|S_n-S_0|\le (C_6-1)\sqrt{n}\}}.
\end{multline}
The first term is smaller than $\gd/6$ if $C_6$ is large enough. To control the second term, we will find an upper bound for

\begin{equation}
 P_x \bar Q_S \exp\left(f_K\left(U_1\right)\right)\ind_{\{\max_{i\in[0,n]}|S_i-S_0|\le (C_6-1)\sqrt{n}\}},
\end{equation}
which is uniform in $x\in I_0$.

What we do is the following: we show that for most trajectories $S$ the term in $f_K$ has a large mean and a small variance with respect to $Q_S$ so that $f_K(\ \dots \ )=-K$ with large $\bar Q_S$ probability. The rest will be easy to control as the term in the expectation is at most one.

The expectation of $U_1$ under $\bar Q_S$ is equal to
\begin{equation}
 m(\gb)^2\sum_{1\le i,j\le n} V_{(i,S_i),(j,S_j)}.
\end{equation}
When the walk stays in the block $B_1$ we have (using definition \eqref{xdef})
\begin{equation}\label{xvij}
 \sum_{1\le i,j\le n} V_{(i,S_i),(j,S_j)}=\frac{1}{100C_6C_7}X.
\end{equation}
The distribution of $X$ under $P_x$ is the same for all $x \in I_0$. It has been shown earlier (cf.\ \eqref{xest} and \eqref{dest}), that if $C_7$ is chosen large enough, 
\begin{equation}
 P\left\{ \frac{m(\gb)^2}{100C_6C_7} X\le \frac{\sqrt{\log n}}{200 C_6C_7}\right\}\le \frac{\gd}{6}.
\end{equation}
As $m(\gb)\ge \gb/2$ if $\gb$ is small, if $C_5$ is large enough (recall $n\ge \exp(C_5/\gb^4)$), this together with \eqref{xvij} gives.
\begin{equation}\label{del2}
 P_x\bigg\{ m(\gb)^2 \bar Q_S\left(U_1\right)\le 2\exp(K^2);
 \max_{i\in[0,n]}|S_i-S_0|\le (C_6-1)\sqrt{n}\bigg\}\le \frac{\gd}{6}.
\end{equation}
To bound the variance of $U_1$ under $\bar Q_S$, we decompose the sum

\begin{multline}
 U_1=\sum_{(i,z),(j,z')\in B_1} V_{(i,z),(j,z')} \go_{i,z}\go_{j,z}=m(\gb)^2\sum_{1\le i,j\le n} V_{(i,S_i),(j,S_j)}\\
+
2m(\gb)\sumtwo{1\le i\le n}{(j,z')\in B_1} V_{(i,S_i),(j,z')}(\go_{j,z'}-m(\gb)\ind_{\{z'=S_j\}})\\
+\sum_{(i,z),(j,z')\in B_1} V_{(i,z),(j,z')}(\go_{i,z}-m(\gb)\ind_{\{z=S_i\}})(\go_{j,z}-m(\gb)\ind_{\{z'=S_j\}}).
\end{multline}
And hence (using the fact that $(x+y)^2\le 2x^2+2y^2$). 
\begin{equation}\label{varcontrol}
 \var_{\bar Q_S} U_1\le 16 m(\gb)^2\sum_{(j,z')\in B_1}\left(\sum_{1\le i\le n}V_{(i,S_i),(j,z')}\right)^2+8\!\!\!\! \sum_{((i,z),(j,z')\in B_1}\!\!\!\! V_{(i,z),(j,z')}^2,
\end{equation}
where we used that $\var_{\bar Q_S} \go_{i,z}\le 2$ (which is true for $\gb$ small enough). The last term is less than $8$ thanks to \eqref{hsnorm}, so that we just have to control the first one.
Independently of the choice of $(j,z')$ we have the bound

\begin{equation}\label{trivbo}
\sum_{1\le i\le n}V_{(i,S_i),(j,z')}\le \frac{\sqrt{\log n}}{C_6 C_7 n}.
\end{equation}
Moreover it is also easy to check that

\begin{equation}
\sum_{(j,z')\in B_1}\sum_{1\le i\le n}V_{(i,S_i),(j,z')}\le \frac{C_7n}{C_6\sqrt{\log n}},
\end{equation}
(these two bounds follow from the definition of $V_{(i,z),(j,z')}$:\ \eqref{vdef}).
Therefore
\begin{equation}
 \sum_{(j,z')\in B_1}\left(\sum_{1\le i\le n}V_{(i,S_i),(j,z')}\right)^2\le \left[\sum_{(j,z')\in B_1}\sum_{1\le i\le n}V_{(i,S_i),(j,z')}\right]\max_{(j,z)\in B_1}\sum_{1\le i\le n}V_{(i,S_i),(j,z')}\le 1.
\end{equation}
Injecting this into \eqref{varcontrol} guaranties that for $\gb$ small enough

\begin{equation}
 \var_{\bar Q_S} U_1 \le 10.
\end{equation}
With Chebyshev inequality, if $K$ has been chosen large enough and 
\begin{equation}
 \bar Q_S U_1\ge 2\exp(K^2),
\end{equation}
we have

\begin{equation}\label{del3}
 \bar Q_S  \left\{ U_1 \le \exp (K^2)\right\}\le \gd/6.
\end{equation}
Hence combining  \eqref{del3} with \eqref{del2} gives
\begin{equation}
 P_x\bar Q_S  \left\{ U_1 \le \exp (K^2);\ \max_{i\in[0,n]}|S_i-S_0|\le (C_6-1)\sqrt{n}\right\}\le \gd/3.
\end{equation}
We use this in \eqref{foor} to get 

\begin{equation}
  \max_{x\in I_0} P_x \bar Q_S \exp\left(f_K\left(U_1\right)\right)\le \frac{\gd}{2}+e^{-K}.
\end{equation}
So that our result is proved provided that $K$ has been chosen large enough.




\end{proof}

\section{Proof of the lower bound in Theorem \ref{PASGAUSSI}}\label{lb11}

In this section we prove the lower bound for the free-energy in dimension $1$ in arbitrary environment.
To do so we apply the second moment method to some quantity related to the partition function, and combine it
with a percolation argument. The idea of the proof was inspired by a study of a polymer model on hierarchical lattice \cite{cf:LM} where this type of coarse-graining appears naturally.
\medskip
\begin{proposition}\label{th:prop}
There exists a constant $C$ such that for all $\gb\le 1$ we have
\begin{equation}
p(\gb)\ge -C\gb^4 ((\log \gb)^2+1).
\end{equation}
\end{proposition}
\medskip
We use two technical lemmas to prove the result. The first is just a statement about scaling of the random walk, the second is more specific to our problem.

\medskip
\begin{lemma}\label{th:lem1}
There exists an a constant $c_{RW}$ such that for large even squared integers $n$,
\begin{equation}
P\{S_n=\sqrt{n},0<S_i<\sqrt{n} \text{ for } 0<i<n \}= c_{RW} n^{-3/2}+o(n^{-3/2}).
\end{equation}

\end{lemma}
\medskip

\begin{lemma}\label{th:lem2}
For any $\gep>0$ we can find a constant $c_{\gep}$ and $\gb_0$ such that for all $\gb\le \gb_{0}$,
for every even squared integer $n\le c_{\gep}/(\gb^4|\log \gb|)$ we have
\begin{equation}
 \var_Q\left[P \left(\exp\left ( \sum_{i=1}^{n-1}\left(\gb\go_{i,S_i}-\gl(\gb)\right)\right)\ \bigg| \ S_n=\sqrt{n}, 0<S_i<\sqrt{n} \text{ for } 0<i <n \right) \right]<\gep.
\end{equation}
\end{lemma}

\begin{proof}[Proof of Proposition \ref{th:prop} from Lemma \ref{th:lem1} and \ref{th:lem2}]
 
Let $n$ be some fixed integer and define

\begin{equation}
\bar W:= P \exp\left (\sum_{i=1}^{n-1}\left(\gb\go_{i,S_i}-\gl(\gb)\right)\right)\ind_{\{S_n=\sqrt{n}, 0<S_i<\sqrt{n} \text{ for } 0<i <n \}},
\end{equation}
which corresponds to the contribution to the partition function $W_n$ of paths with fixed end point $\sqrt{n}$ staying within a cell of width $\sqrt{n}$, with the specification the environment on the last site is not taken in to account. $\bar W$ depends only of the value of the environment $\go$ in this cell (see figure \ref{fig:restr}).

\begin{figure}[h]
\begin{center}
\leavevmode
\epsfysize =4.5 cm
\psfragscanon
\psfrag{O}[c]{O}
\psfrag{n1}[c]{$n$}
\psfrag{n2}[c]{$\sqrt{n}$}
\epsfbox{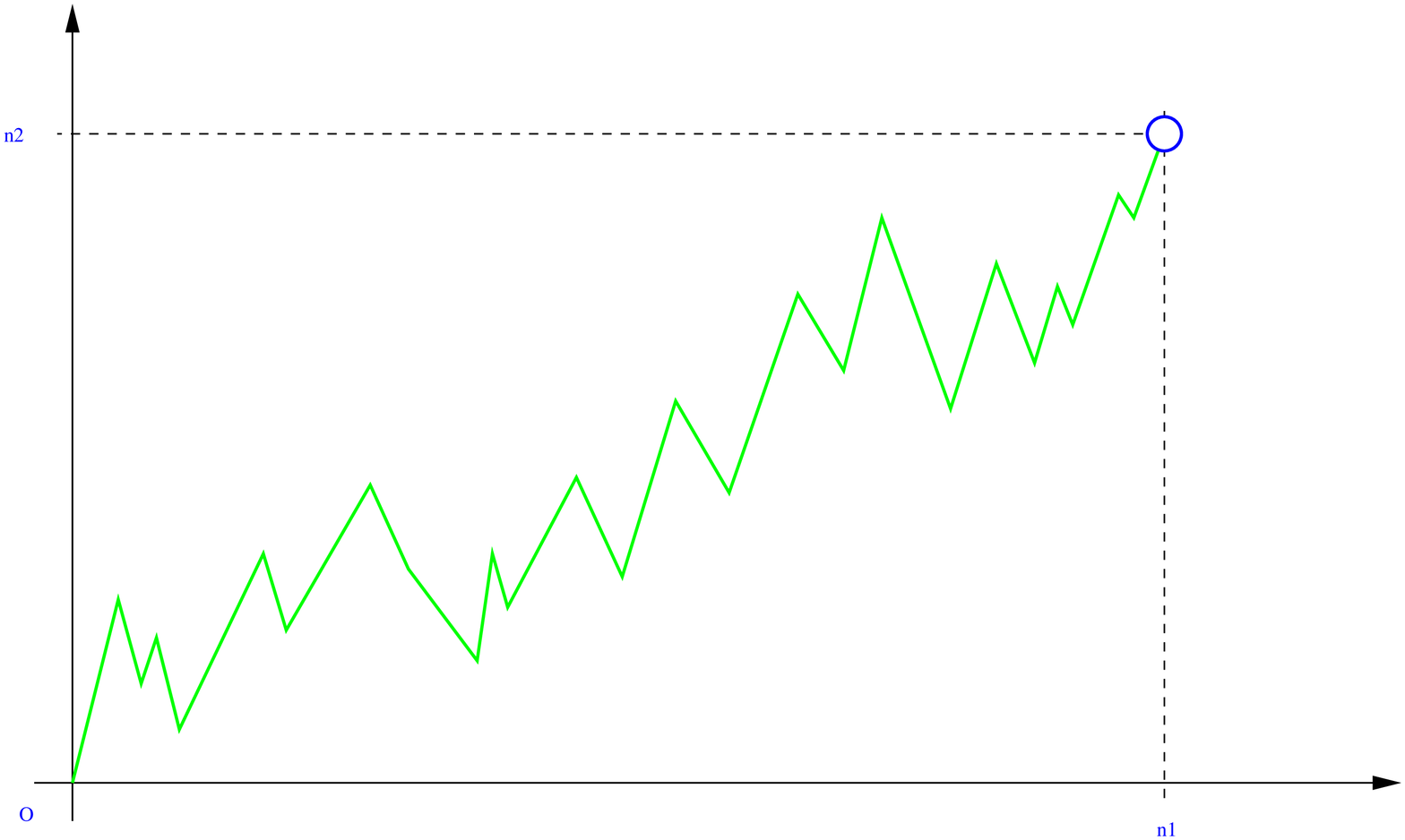}
\caption{\label{fig:restr} We consider a resticted partition function $\bar W$ by considering only paths going from one to the other corner of the cell, without going out. This restriction will give us the independence of random variable corresponding to different cells which will be crucial to make the proof works.}
\end{center}
\end{figure}
One also defines the following quantities for $(i,y)\in \N\times \Z$:

\begin{equation}\begin{split}
 \bar W_i^{(y,y+1)}&:=P_{\sqrt{n}y}\left[ e^{\sum_{k=1}^{n-1}\left[\gb \go_{in+k,S_k}-\gl(\gb)\right]}\ind_{\{S_n=\sqrt{(y+1)n}, 0<S_i-y \sqrt n<\sqrt{n} \text{ for } 0<i <n\}}\right],\\
 \bar W_i^{(y,y-1)}&:=P_{\sqrt{n}y}\left[\ e^{\sum_{k=1}^{n-1}\left[\gb \go_{in+k,S_k}-\gl(\gb)\right]}\ind_{\{S_n=(y-1)\sqrt{n}, -\sqrt{n}<S_i-y\sqrt{n}<0 \text{ for } 0<i <n\}}\right].
\end{split}\end{equation}
which are random variables that have the same law as $\bar W$. Moreover because of independence of the environment in different cells, one can see that
 \begin{equation*}
 \left(\bar W_i^{(y,y\pm 1)};\ (i,y)\in \N\times \Z \text{ such that } i-y \text { is even} \right),
\end{equation*}
is a family of independent variables.

Let $N=nm$ be a large integer. We define $\Omega=\Omega_N$ as the set of path
\begin{equation}
\Omega:=\{ S \ : \ \forall i\in [1,m],\ |S_{in}-S_{(i-1)n}|=\sqrt{n},\ \forall j\in [1,n-1],\  S_{(i-1)n+j}\in \left( S_{(i-1)n},S_{in}\right)\},
\end{equation}
where the interval $\left( S_{i(n-1)},S_{in}\right)$ is to be seen as $\left( S_{in},S_{i(n-1)}\right)$ if $S_{in}<S_{i(n-1)}$, and 
\begin{equation}
 \mathcal S:=\left\{s=(s_0,s_1,\dots,s_m)\in \Z^{m+1}\ :\  s_0=0 \text{ and } |s_i-s_{i-1}|=1, \ \forall i \in[1,m] \right\} .
\end{equation}
We use the trivial bound

\begin{equation}
W_N\ge P \left[\exp\big(\sum_{i=1}^{nm}(\gb\go_{i,S_i}-\gl(\gb))\big)\ind_{\{S\in\Omega\}}\right],
\end{equation}
to get that 

\begin{equation}\label{trtrtr}
 W_N\ge \sum_{s\in \mathcal S} \prod_{i=0}^{m-1} \bar W_{i}^{(s_{i},s_{i+1})}\exp\left(\gb\go_{(i+1)n,s_{i+1}\sqrt{n}}-\gl(\gb)\right).
\end{equation}
(the exponential term is due to the fact the $\bar W$ does not take into account to site in the top corner of each cell).

The idea is of the proof is to find a value of $n$ (depending on $\gb$) such that we are sure that for any value of $m$ we can find a path $s$ such that along the path the values of $(\bar W_{i}^{(s_{i},s_{i+1})})$ are not to low (i.e. close to the expectation of $\bar W$) and to do so, it seems natural to seek for a percolation argument.

Let $p_c$ be the critical exponent for directed percolation in dimension $1+1$ (for an account on directed percolation see \cite[Section 12.8]{cf:perc} and references therein). From Lemma \ref{th:lem2} and Chebyshev inequality, one can find a constant $C_8$ and $\gb_0$ such that for all $n\le\frac{C_8}{\gb^4|\log\gb|}$ and $\gb\le \gb_0$.
\begin{equation}
 Q \{\bar W \ge  Q \bar W /2\}\ge \frac{p_c+1}{2}. \label{eq:perco}
\end{equation}
We choose $n$ to be the biggest squared even integer that is less than $\frac{C_8}{\gb^4|\log\gb|}$.
(in particular have $n\ge \frac{C_8}{2\gb^4|\log\gb|}$ if $\gb$ small enough).

As shown in figure \ref{fig:wnnn}, we associate to our system the following directed percolation picture. For all $(i,y)\in \N\times \Z \text{ such that } i-y$ is even:
\begin{itemize}
 \item If $\bar W_i^{(y,y\pm 1)}\ge (1/2)Q \bar W$, we say that the edge linking the opposite corners of the corresponding cell is open.
 \item If $\bar W_i^{(y,y\pm 1)}< (1/2)Q \bar W$, we say that the same edge is closed.
\end{itemize}
Equation \eqref{eq:perco} and the fact the considered random variables are independent assures that with positive probability there exists an infinite directed path starting from zero.

\begin{figure}[h]
\begin{center}
\leavevmode
\epsfysize =6.5 cm
\psfragscanon
\psfrag{O}[c]{\tiny{O}}
\psfrag{n}[c]{\tiny{$n$}}
\psfrag{2n}[c]{\tiny{$2n$}}
\psfrag{3n}[c]{\tiny{$3n$}}
\psfrag{4n}[c]{\tiny{$4n$}}
\psfrag{5n}[c]{\tiny{$5n$}}
\psfrag{6n}[c]{\tiny{$6n$}}
\psfrag{7n}[c]{\tiny{$7n$}}
\psfrag{8n}[c]{\tiny{$8n$}}
\psfrag{rn}[c]{\tiny{$+\sqrt{n}$}}
\psfrag{r2n}[c]{\tiny{$+2\sqrt{n}$}}
\psfrag{r3n}[c]{\tiny{$+3\sqrt{n}$}}
\psfrag{r4n}[c]{\tiny{$+4\sqrt{n}$}}
\psfrag{m1n}[c]{\tiny{$-\sqrt{n}$}}
\psfrag{m2n}[c]{\tiny{$-2\sqrt n$}}
\psfrag{m3n}[c]{\tiny{$-3\sqrt n$}}
\psfrag{1}[c]{\tiny{$1$}}
\psfrag{2}[c]{\tiny{$2$}}
\psfrag{3}[c]{\tiny{$3$}}
\psfrag{4}[c]{\tiny{$4$}}
\psfrag{5}[c]{\tiny{$5$}}
\psfrag{6}[c]{\tiny{$6$}}
\psfrag{7}[c]{\tiny{$7$}}
\psfrag{8}[c]{\tiny{$8$}}
\epsfbox{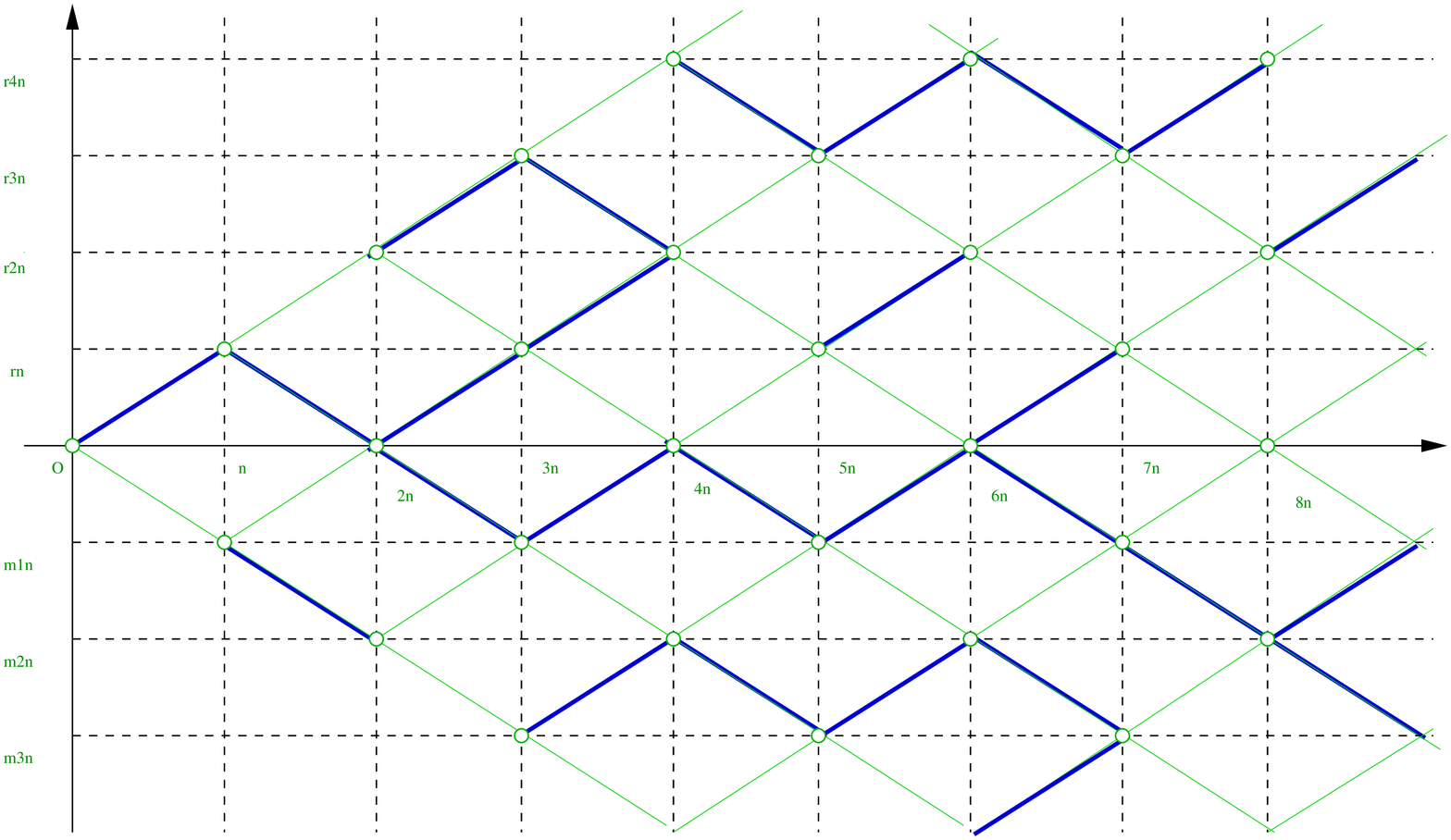}
\end{center}
\caption{\label{fig:wnnn} This figure illustrates the percolation argument used in the proof. To each cell is naturally associated a random variable $\bar W_i^{y,y\pm 1}$, and these  random variables are i.i.d. When $\bar W_i^{y,y\pm 1}\ge 1/2 Q \bar W$ we open the edge in the corresponding cell (thick edges on the picture). As this happens with a probability strictly superior to $p_c$, we have a positive probability to have an infinite path linking $0$ to infinity.}
\end{figure}

When there exists an infinite open path is linking zero to infinity exists, we can define the highest open path in an obvious way. Let $(s_i)_{i=1}^m$ denotes this highest path. If $m$ is large enough, by law of large numbers we have that with a probability close to one,
\begin{equation}
\sum_{i=1}^m \left[\gb\go_{ni,\sqrt{n}s_i}-\gl(\gb)\right] \ge -2m\gl(\gb). 
\end{equation}

Using this and and the percolation argument with \eqref{trtrtr} we finally get that with a positive probability which does not depend on $m$ we have

\begin{equation}
 W_{nm}\ge \left[(1/2)e^{-2\gl(\gb)}Q\bar W\right]^m.
\end{equation}
Taking the $\log$ and making $m$ tend to infinity this implies that

\begin{equation}
 p(\gb)\ge \frac{1}{n}\left[-2 \gl(\gb)-\log 2+ \log Q \bar W\right]\ge -\frac{c}{n}\log n.
\end{equation}
For some constant $c$, if $n$ is large enough (we used Lemma \ref{th:lem1} to get the last inequality.
The result follows by replacing $n$ by its value.




\end{proof}
\medskip

\begin{proof}[Proof of Lemma \ref{th:lem1}]

Let $n$ be square and even. $T_k$, $k\in \Z$ denote the first hitting time of $k$ by the random walk $S$ (when $k=0$ it denotes the return time to zero). 
We have
\begin{multline}
P\{S_n=\sqrt{n},0<S_i<\sqrt{n}, \text{ for all } 1<i<n\}\\
=\sum_{k=1}^{n-1} P\{T_{\sqrt{n}/2}=k,\ S_j>0 \text{ for all } j<n \text{ and } T_{\sqrt{n}}=n\}\\
= P\{T_{\sqrt{n}/2}<n,\ S_j<\sqrt{n} \text{ for all } j<n \text{ and } T_{0}=n)\},
\end{multline}
where the second equality is obtained with the strong Markov property used for $T=T_{\sqrt{n}/2}$, and the reflexion principle for the random walk.
The last line is equal to
\begin{equation}\label{eq:hop}
P\{\max_{k\in [0,n]} S_k\in [\sqrt{n}/2,\sqrt{n})|T_0=n\}P\{T_0=n\}.
\end{equation}
We use here a variant of Donsker's Theorem, for a proof see \cite[Theorem 2.6]{cf:K}.
\medskip

\begin{lemma}
The process
\begin{equation}
t\mapsto \left\{\frac {S_{\lceil nt \rceil }}{\sqrt{n}}\ \bigg| \ T_0=n\right\}, \quad t\in[0,1]
\end{equation}
converges in distribution to the normalized Brownian excursion in the space $D([0,1],\R)$.
\end{lemma}
\medskip
We also know that (see for example \cite[Proposition A.10]{cf:Book}) for $n$ even $P(T_0=n)=\sqrt{2/\pi}n^{-3/2}+o(n^{-3/2})$.
Therefore, from  \eqref{eq:hop} we have
\begin{multline}
P\{S_n=\sqrt{n},0<S_i<\sqrt{n}, \text{ for all } 1<i<n\}\\
= \sqrt{2/\pi}n^{-3/2}\bbP\left[\max_{t\in[0,1]} {\bf e}_t\in (1/2,1)\right]+o(n^{-3/2}).
\end{multline}
Where $\bf e$ denotes the  normalized Brownian excursion, and $\bbP$ its law.
\end{proof}
\medskip

\begin{proof}[Proof of Lemma \ref{th:lem2}]

Let $\gb$ be fixed and  small enough, and $n$ be some squared even integer which is less than $c_\gep/(\gb^4|\log \gb|)$. We will fix the value $c_{\gep}$ independently of $\gb$ later in the proof, and always consider that $\gb$ is sufficiently small.
By a direct computation the variance of

\begin{equation}
P\left[\exp\left( \sum_{i=1}^{n-1}[\gb\go_{i,S_i}-\gl(\gb)]\right)\bigg| \ S_n=\sqrt{n}, 0<S_i<\sqrt{n} \text{ for } 0<i <n \right] 
\end{equation}
is equal to

\begin{equation}\label{eq:abcd}
 P^{\otimes 2}\left[\exp \left(\sum_{i=1}^{n-1}\gga(\gb)\ind_{\{S^{(1)}_i=S^{(2)}_i\}}\right)\ \bigg| \ A_n\right]-1.
\end{equation}
where

\begin{equation}
 A_n=\left\{S_n^{(j)}=\sqrt{n}, 0<S^{(j)}_i<\sqrt{n} \text{ for } 0<i<n,\ j=1,2\right\},
\end{equation}
and $\gga(\gb)=\gl(2\gb)-2\gl(\gb)$ (recall that $\gl(\gb)=\log Q \exp(\gb\go_{(1,0)})$), and $S_n^{(j)}$, $j=1,2$ denotes two independent random walk with law denoted by $P^{\otimes 2}$.
From this it follows that if $n$ is small the result is quite straight--forward. We will therefore only be interested in the case of large $n$ (i.e.\ bounded away from zero by a large fixed constant).

We define $\tau=\left(\tau_k\right)_{k\ge 0}=\{S^{(1)}_i=S^{(2)}_i,i\ge 0\}$  the set where the walks meet (it can be written as an increasing sequence of integers). By the Markov property, the random variables $\tau_{k+1}-\tau_k$ are i.i.d.\ , we say that $\tau$ is a {\sl renewal sequence}.
\medskip

We want to bound the probability that the renewal sequence $\tau$ has too many returns before times $n-1$, in order to estimate \eqref{eq:abcd}. To do so, we make the usual computations with Laplace transform.

From \cite[p. 577]{cf:Feller2} , we know that
\begin{equation}\label{Lapl}
1- P^{\otimes 2} \exp(- x \tau_1)= \frac{1}{\sum_{n\in \N} \exp(-xn)P\{S_n^{(1)}=S_n^{(2)}\}}.
\end{equation}
Thanks to the the local central limit theorem for the simple random walk, we know that for large $n$
\begin{equation}
P\{S_n^{(1)}=S_n^{(2)}\}= \frac{1}{\sqrt{\pi n}}+o(n^{-1/2}). 
\end{equation}
So that we can get from \eqref{Lapl} that when $x$ is close to zero
\begin{equation}
 \log P^{\otimes 2} \exp(- x \tau_1)=  -\sqrt{x} +o(\sqrt{x}).
\end{equation}
We fix $x_0$ such that $\log P \exp(- x \tau_1)\le \sqrt x /2$ for all $x\le x_0$.
For any $k\le n$ we have

\begin{equation} \begin{split}
 P^{\otimes 2}\{|\tau\cap [1,n-1]|\ge k\}=P^{\otimes 2}\{\tau_k\le n-1\}&\le \exp((n-1)x) P^{\otimes 2} \exp (-\tau_k x)\\
&\le \exp\left[nx+k \log P^{\otimes 2}\exp(-x\tau_1)\right]. \end{split}
\end{equation}
For any $k\le \left\lfloor 4 n\sqrt{x_0}\right\rfloor=k_0$ one can choose $x=(k/4 n)^2\le x_0$ in the above and use the definition of $x_0$ to get that
 
\begin{equation}
  P^{\otimes 2}\{|\tau\cap [1,n-1]|\ge k\}\le \exp\left(-k^2/(32  n)\right).
\end{equation}
In the case where $k> k_0$ we simply bound the quantity by

\begin{equation}
   P^{\otimes 2}\{|\tau\cap [1,n-1]|\ge k\}\le \exp\left(k_0^2/(32 n)\right)\le \exp\left(-n x_0/4\right).
\end{equation}
By Lemma \eqref{th:lem1}, if $n$ is large enough
\begin{equation}
 P^{\otimes}A_n\ge 1/2 c_{RW}^2n^{-3}.
\end{equation}
A trivial bound on the conditioning gives us

\begin{equation}\label{bounee}\begin{split}
P^{\otimes 2}\left(|\tau\cap [1,n-1]|\ge k \ \big| \ A_n \right)&\le\min\left(1,2c_{RW}^{-2} n^3\exp\left(-k^2/(32 n)\right)\right) \text{ if } k\le k_0,\\
 P^{\otimes 2}\left(|\tau\cap [1,n-1]|\ge k \ \big| \ A_n \right)&\le 2c_{RW}^{-2}n^{3} \exp\left(-n x_0/4\right)
\text{ otherwise}.
\end{split}\end{equation}
We define $k_1:=\lceil 16\pi\sqrt{n\log(2c_{RW}^{-2} n^3)}\rceil$. 
The above implies that for $n$ large enough we have

\begin{equation}\begin{split}
P^{\otimes 2}\left(|\tau\cap [1,n-1]|\ge k \ \big| \ A_n \right)&\le 1 \text{ if } k\le k_1, \\
P^{\otimes 2}\left(|\tau\cap [1,n-1]|\ge k \ \big| \ A_n \right)&\le \exp\left(-k^2/(64 n)\right) \text{ if } k_1\le k\le k_0,\\
 P^{\otimes 2}\left(|\tau\cap [1,n-1]|\ge k \ \big| \ A_n \right)&\le\exp\left(-n x_0/8\right)
\text{ otherwise}.
\end{split}\end{equation}
Now we are ready to bound \eqref{eq:abcd}.
Integration by part gives,

\begin{equation}\begin{split}
 &P^{\otimes 2}\left[\exp\left(\gamma \gb |\tau\cap [1,n-1]|\right)\ \big| \ A_n\right]-1\\
&\quad\quad\quad\quad\quad\quad\quad\quad=\gga(\gb)\int_{0}^{\infty}\exp(\gga(\gb) x)P^{\otimes 2}\left(|\tau\cap [1,n-1]|\ge x \ \big| \ A_n\right)\dd x.
\end{split}\end{equation}
We split the right-hand side in three part corresponding to the three different bounds we have in \eqref{bounee}:
$x\in[0,k_1]$, $x\in[k_1,k_0]$ and $x\in[k_0,n]$. It suffices to show that each part is less than $\gep/3$ to finish the proof. The first part is

\begin{equation}
\gga(\gb)\int_{0}^{k_1}\exp(\gga(\gb) x)P^{\otimes 2}\left(|\tau\cap [1,n-1]|\ge x\ \big| \ A_n\right)\dd x \le \gga(\gb)k_1\exp(\gga(\gb)k_1).
\end{equation}
One uses that $n\le \frac{c_{\gep}}{\gb^4|\log \gb|}$ and $\gga(\gb)=\gb^2+o(\gb^2)$ to get that for $\gb$ small enough and $n$ large enough if $c_\gep$ is well chosen we have
\begin{equation}
 k_1\gga(\gb)\le 100 \gb^2 \sqrt{n\log n}\le \gep/4,
\end{equation}
so that $\gga(\gb)k_1\exp(\gga(\gb)k_1)\le \gep/3$.\\
We use our bound for the second part of the integral to get

\begin{equation}\begin{split}
&\gga(\gb)\int_{k_1}^{k_0}\exp(\gga(\gb) x)P^{\otimes 2}\left(|\tau\cap [1,n-1]|\ge x\ \big| \ A_n\right)\dd x\\
&\quad\le \gga(\gb)\int_{0}^{\infty}\exp\left(\gga(\gb)x-x^2/(64 n)\right)\dd x =\int_{0}^{\infty}\exp\left(x-\frac{x^2}{64 n \gga(\gb)^2}\right)\dd x.
\end{split}\end{equation}
Replacing $n$ by its value, we see that the term that goes with $x^2$ in the exponential can be made arbitrarily large, provided that $c_\gep$ is small enough. In particular we can make the left-hand side less than $\gep/3$.\\
Finally, we estimate the last part

\begin{equation}\begin{split}
&\gga(\gb)\int_{k_0}^{n}\exp(\gga(\gb)^2 x)P^{\otimes 2}\left(|\tau\cap [1,n-1]|\ge x \ \big| \ A_n\right)\dd x\\
&\quad\quad\quad\quad\quad\quad\quad\quad \le \gga(\gb)\int_{0}^{n}\exp(\gga(\gb)x-n x_0/8)\dd x=n\exp(-[\gga(\gb)-x_0/8]n).
\end{split}\end{equation}
This is clearly less than $\gep/3$ if $n$ is large and $\gb$ is small.

\end{proof}

\section{Proof of the lower bound of Theorem \ref{GAUSSI}}\label{dim11}

In this section we use the method of replica-coupling that is used for the disordered pinning model in \cite{cf:T_cmp} to derive a lower bound on the free energy. The proof here is an adaptation of the argument used there to prove disorder irrelevance.

The main idea is the following: Let $W_N(\gb)$ denotes the renormalized partition function for inverse temperature $\gb$. A simple Gaussian computation gives

\begin{equation}
 \frac{\dd Q \log W_N(\sqrt{t})}{\dd t}\bigg|_{t=0}=-\frac{1}{2}P^{\otimes 2}\sum_{i=1}^N \ind_{\{S_i^{(1)}=S_i^{(2)}\}}.
\end{equation}
Where $S^{(1)}$ and $S^{(2)}$ are two independent random walk under the law $P^{\otimes 2}$.
This implies that for small values of $\gb$ (by the equality of derivative at $t=0$),

\begin{equation}\label{eqqq}
Q \log W_N(\gb) \approx - \log P^{\otimes 2} \exp\left(\gb^2/2\sum_{i=1}^N \ind_{\{S_N^{(1)}=S_N^{(2)}\}}\right).
\end{equation}
This tends to make us believe that

\begin{equation}
 p(\gb)=-\lim_{N\rightarrow \infty} \log P^{\otimes 2} \exp\left(\gb^2/2\sum_{i=1}^N \ind_{\{S_N^{(1)}=S_N^{(2)}\}}\right).
\end{equation}
However, things are not that simple because \eqref{eqqq} is only valid for fixed $N$, and one needs some more work to get something valid when $N$ tends to infinity.
The proofs aims to use convexity argument and simple inequalities to be able to get the inequality

\begin{equation}
 p(\gb)\ge-\lim_{N\rightarrow \infty} \log P^{\otimes 2} \exp\left(2 \gb^2\sum_{i=1}^N \ind_{\{S_N^{(1)}=S_N^{(2)}\}}\right).
\end{equation}
The fact that convexity is used in a crucial way make it quite hopeless to get the other inequality using this method.

\begin{proof}
Let use define for $\gb$ fixed and $t\in[0,1]$

\begin{equation}
\Phi_N(t,\gb):=\frac{1}{N}Q\log P \exp\left(\sum_{i=1}^N \left[\sqrt{t}\gb \go_{i,S_i}-\frac{t\gb^2}{2}\right]\right),
 \end{equation}
and for $\gl\ge 0$
\begin{equation}
\Psi_N(t,\gl,\gb):=\frac{1}{2N}Q\log P^{\otimes 2} \exp\left(\sum_{i=1}^N\left[\sqrt{t}\gb(\go_{i,S_i^{(1)}}+\go_{i,S_i^{(2)}})-t\gb^2+\gl\gb^2\ind_{\{S_i^{(1)}=S_i^{(2)}\}}\right]\right).
\end{equation}
One can notice that $\Phi_N(0,\gb)=0$ and $\Phi_N(1,\gb)=p_N(\gb)$ (recall the definition of $p_N$ \eqref{Pn}), so that  $\Phi_N$ is an interpolation function.
Via the Gaussian integration by par formula

\begin{equation}
 Q \go f(\go)=Q f'(\go),
\end{equation}
valid (if $\go$ is a centered standard Gaussian variable) for every differentiable functions such that $\lim_{|x|\to\infty}\exp(-x^2/2)f(x)=0$, one finds

\begin{equation}\begin{split}\label{fifi}
 \frac{\dd }{\dd t}\Phi_N(t,\gb)&=-\frac{\gb^2}{2N}\sum_{j=1}^{N}\sum_{z\in \Z} Q \left(\frac{ P \exp\left(\sum_{i=1}^N \left[\sqrt{t}\gb\go_{i,S_i}-\frac{t\gb^2}{2}\right]\right)\ind_{\left\{S_j=z\right\}}}{P\exp\left(\sum_{i=1}^N \left[\sqrt{t}\gb\go_{i,S_i}-\frac{t\gb^2}{2}\right]\right)}\right)^2 \\
&=-\frac{\gb^2}{2N}Q \left(\mu^{(\sqrt{t}\gb)}_n\right)^{\otimes 2} \left[\sum_{i=1}^N \ind_{\{S_i^{(1)}=S_i^{(2)}\}}\right].
\end{split}
\end{equation}
This is (up to the negative multiplicative constant $-\gb^2/2$) the expected overlap fraction of two independent replicas of the random--walk under the polymer measure for the inverse temperature $\sqrt{t}\gb$. This result has been using It\^o formula in \cite[Section 7]{cf:CH_ptrf}.\\
For notational convenience, we define

\begin{equation}
H_N(t,\gl,S^{(1)},S^{(2)})=\sum_{i=1}^N\left[\sqrt{t}\gb(\go_{i,S_i^{(1)}}+\go_{i,S_i^{(2)}})-t\gb^2+\gl\gb^2\ind_     {\left\{S_i^{(1)}=S_i^{(2)}\right\}}\right].
\end{equation}
We use Gaussian integration by part again, for $\Psi_N$: 

\begin{multline}\label{psipsi}
\frac{\dd}{\dd t}\Psi_{N}(t,\gl,\gb)= \frac{\gb^2}{2N}\sum_{j=1}^N Q \frac{ P^{\otimes 2} \exp\left(H_N(t,\gl,S^{(1)},S^{(2)})\right)\ind_{\{S_j^{(1)}=S_j^{(2)}\}}}{P^{\otimes 2} \exp\left(H_N(t,\gl,S^{(1)},S^{(2)})\right)}\\
- \frac{\gb^2}{4N}\sum_{j=1}^N\sum_{z\in \Z} Q \left(\frac{ P^{\otimes 2} \left(\ind_{\{S_j^{(1)}=z\}}+\ind_{\{S_j^{(2)}=z\}}\right)\exp\left(H_N(t,\gl,S^{(1)},S^{(2)})\right)}{ P^{\otimes 2} \exp\left(H_N(t,\gl,S^{(1)},S^{(2)})\right)}\right)^2 \\
\le \frac{\gb^2}{2N}\sum_{j=1}^N Q \frac{ P^{\otimes 2} \exp\left(H_N(t,\gl,S^{(1)},S^{(2)})\right)\ind_{\{S_j^{(1)}=S_j^{(2)}\}}}{P^{\otimes 2} \exp\left(H_N(t,\gl,S^{(1)},S^{(2)})\right)}=\frac{\dd}{\dd\gl} \Psi_N(t,\gl,\gb).
\end{multline}
The above implies that for every $t\in[0,1]$ and $\gl\ge 0$

\begin{equation}\label{hmhm}
 \Psi_N(t,\gl,\gb)\le \Psi_N(0,\gl+t,\gb).
\end{equation}
Comparing \eqref{fifi} and \eqref{psipsi}, and using convexity and monotonicity of $\Psi_N(t,\gl,\gb)$ with respect to $\gl$, and the fact that $\Psi_N(t,0,\gb)=\Phi_N(t,\gb)$ one gets

\begin{multline}
                  -\frac{\dd}{\dd t}\phi_N(t,\gb)=\frac{\dd}{\dd \gl}\Psi_N(t,\gl,\gb)\bigg|_{\gl=0}\\
\le \frac{\Psi_N(t,2-t,\gb)-\Phi_N(t,\gb)}{2-t}\le \Psi_N(0,2,\gb)-\Phi_N(t,\gb),
                \end{multline}
where in the last inequality we used $(2-t)\ge 1$ and \eqref{hmhm}.
Integrating this inequality between $0$ and $1$ and recalling $\Phi_N(1,\gb)=p_N(\gb)$ we get

\begin{equation}
 p_N(\gb)\ge (1-e)\Psi_N(0,2,\gb).
\end{equation}

On the right-hand side of the above we recognize something related to pinning models. More precisely

\begin{equation}\label{alm}
 \Psi_N(0,2,\gb)=\frac{1}{2N}\log Y_N,
\end{equation}
where 
\begin{equation}
 Y_N=P^{\otimes 2} \exp\left(2\gb^2 \sum_{i=1}^N \ind_{\left\{S_N^{(1)}=S_N^{(2)}\right\}}\right)
\end{equation}
is the partition function of a homogeneous pinning system of size $N$ and parameter $2\gb^2$ with underlying renewal process the sets of zero of the random walk $S^{(2)}-S^{(1)}$. This is a well known result in the study of pinning model ( we refer to \cite[Section 1.2]{cf:Book} for an overview and the results we cite here) that 

\begin{equation}
 \lim_{N\rightarrow\infty} \frac{1}{N}\log Y_N=\tf(2\gb^2),
\end{equation}
where $\tf$ denotes the free energy of the pinning model. Moreover, it is also stated 
\begin{equation}
 \tf(h)\stackrel{h\to 0+}{\sim} h^2/2.
\end{equation}
Then passing to the limit in \eqref{alm} ends the proof of the result for any constant strictly bigger that $4$. 
\end{proof}

\section{Proof the lower bound in Theorem \ref{1+2UPLD}}\label{dim12}

The technique used in the two previous sections could be adapted here to prove the results but in fact it is not necessary.
Because of the nature of the bound we want to prove in dimension $2$ (we do not really track the best possible constant in the exponential), it will be sufficient here to control the variance of $W_n$ up to some value, and then the concentration properties of $\log W_n$ to get the result. The reader can check than using the same method in dimension $1$ does not give the right power of $\gb$.

First we prove a technical result to control the variance of $W_n$ which is the analog of \eqref{th:lem2} in dimension $1$. Recall that $\gga(\gb):=\gl(2\gb)-2\gl(\gb)$ with $\gl(\gb):=\log Q \exp (\gb\go_{(1,0)})$.

\begin{lemma}\label{th:lemmvar}
For any $\gep<0$, one can find a constant $c_\gep>0$ and $\gb_0>0$ such that for any $\gb\le \gb_0$, for any 
$n\le \exp\left(c_{\gep}/\gb^2\right)$ we have
\begin{equation}
\var_Q W_n\le \gep.
\end{equation}
\end{lemma}
\begin{proof}
A straight--forward computation shows that the the variance of $W_n$ is given by

\begin{equation}\label{var2d}
\var_Q W_n= P^{\otimes 2} \exp\left(\gga(\gb)\sum_{i=1}^n \ind_{\{S_i^{(1)}=S_i^{(2)}\}}\right)-1.
\end{equation}
where $S^{(i)}$, $i=1,2$ are two independent $2$--dimensional random walks.

As the above quantity is increasing in $n$, it will be enough to prove the result for $n$ large. For technical convenience we choose to prove the result for $n= \rfloor \exp(-c_{\gep}/\gga(\gb))\rfloor$ (recall $\gga(\gb)=\gl(2\gb)-2\gl(\gb)$) which does not change the result since $\gga(\gb)=\gb^2+o(\gb^2)$.

The result we want to prove seems natural since we know that $(\sum_{i=1}^n \ind_{\{S_i^{(1)}=S_i^{(2)}\}})/\log n$ converges to an exponential variable (see e.g.\ \cite{cf:GZ}), and $\gga(\gb)\sim c_{\gep}\log n$. However, convergence of the right--hand side of \eqref{var2d} requires the use of the dominated convergence Theorem, and the proof of the domination hypothesis is not straightforward. It could be extracted from the proof of the large deviation result in \cite{cf:GZ}, however
we include a full proof of convergence here for the sake of completeness.

We define $\tau=\left(\tau_k\right)_{k\ge 0}=\{S^{(1)}_i=S^{(2)}_i,i\ge 0\}$  the set where the walks meet (it can be written as an increasing sequence). By the Markov property, the random variables $\tau_{k+1}-\tau_k$ are i.i.d.\ . 

\medskip

To prove the result, we compute bounds on the probability of having to may point before $n$ in the renewal $\tau$. As in the $1$ dimensional case, we use Laplace transform to do so.
From \cite[p. 577]{cf:Feller2} , we know that

\begin{equation}\label{Lapl1}
1- P^{\otimes 2} \exp(- x \tau_1)= \frac{1}{\sum_{n\in \N} \exp(-xn)P\{S_n^{(1)}=S_n^{(2)}\}}.
\end{equation}
The local central limit theorem says that for large $n$

\begin{equation}
  P^{\otimes 2}\{S_n^{(1)}=S_n^{(2)}\}\sim \frac{1}{\pi n}.
\end{equation}
Using this into \eqref{Lapl1} we get that when $x$ is close to zero

\begin{equation}
\log P^{\otimes 2} \exp(- x \tau_1)\sim -\frac{\pi}{|\log x|}.
\end{equation}
We use the following estimate 

\begin{equation} \label{coucou}\begin{split}
 P^{\otimes 2}\{|\tau\cap [1,n]|\ge k\}=P^{\otimes 2}\{\tau_k\le n\}&\le \exp(nx) P^{\otimes 2} \exp (-\tau_k x)\\
&= \exp\left[nx+k\log  P^{\otimes 2} \exp(-x\tau_1)\right]. \end{split}
\end{equation}
Let $x_0$ be such that for any $x\le x_0$, $\log P^{\otimes 2} \exp(- x \tau_1)\ge -3/ |\log x|$.
For $k$ such that $k/(n\log(n/k))\le x_0$, we replace $x$ by $k/(n\log(n/k))$ in \eqref{coucou}
to get 
\begin{equation}
  P^{\otimes 2}\{|\tau\cap [1,n]|\ge k\}\le \exp \left(\frac{k}{\log(n/k)}-\frac{3 k}{\log\left[k/(n\log n/k)\right]}\right)\le \exp\left(-\frac{k}{\log(n/k)}\right),
\end{equation}
where the last inequality holds if $k/n$ is small enough. We fix $k_0=\gd n$ for some small $\gd$.
We get that
\begin{equation}\label{2db}\begin{split}
   P^{\otimes 2}\{|\tau\cap [1,n]|\ge k\}&\le \exp\left(-\frac{k}{\log(n/k)}\right) \quad \text{if } k\le k_0\\
   P^{\otimes 2}\{|\tau\cap [1,n]|\ge k\}&\le \exp\left(-\frac{k_0}{\log(n/k_0)}\right)=\exp\left(-\frac{\gd n}{\log (1/\gd)}\right) \quad \text{if } k\ge k_0.
\end{split}
\end{equation}
We are ready to bound \eqref{var2d}. We remark that using integration by part we obtain

\begin{equation}
P \exp\left(\gga(\gb) |\tau\cap [1,n]|\right)-1= \int_{0}^n \gga(\gb)\exp(\gga(\gb)x)P^{\otimes 2}(\tau\cap [1,n]|\ge x)\dd x.
\end{equation}
To bound the right--hand side, we use the bounds we have concerning $\tau$: \eqref{2db}. We have to split the integral in three parts. 
\medskip

The integral between $0$ and $1$ can easily be made less than $\gep/3$ by choosing $\gb$ small.

\medskip
Using $n\le \exp(c_{\gep}/\gga(\gb))$, we get that

\begin{multline}
 \int_{1}^{\gd n} \gga(\gb)\exp(\gga(\gb)x)P^{\otimes 2}(\tau\cap [1,n]|\ge x)\dd x\le \int_{1}^{\gd n}\gga(\gb)\exp\left(\gga(\gb)x-\frac{x}{\log(n/x)}\right)\dd x\\
\le \int_{1}^{\gd n}\gga(\gb)\exp\left(\gga(\gb)x-\frac{\gga(\gb)\gb x}{c_{\gep}}\right)\le \frac{c_{\gep}}{1-c_{\gep}}.
\end{multline}
This is less that $\gep/3$ if $c_{\gep}$ is chosen appropriately.
The last part to bound is
\begin{equation}
 \int_{\gd n}^n \gga(\gb)\exp(\gga(\gb)x)P^{\otimes 2}(\tau\cap [1,n]|\ge x)\le n\gga(\gb)\exp\left(\gga(\gb)n-\frac{\gd n}{\log 1/\gd}\right)\le \gep/3,
\end{equation}
where the last inequality holds if $n$ is large enough, and $\gb$ is small enough.
 
\end{proof}

\begin{proof}[Proof of the lower bound in Theorem \ref{1+2UPLD}]

By a martingale method that one can find a constant $c_9$ such that
\begin{equation}
\var_Q{\log W_n}\le C_9 n, \quad \quad \forall n\ge 0, \forall \gb\le 1.
\end{equation}
(See \cite[Proposition 2.5]{cf:CSY} and its proof for more details).\\
 Therefore Chebyshev inequality gives
\begin{equation}\label{eq:ineq}
Q\left\{\left|\frac 1 n \log W_n - \frac 1 n Q\log W_n\right|\ge n^{-1/4}\right\}\le C_9 n^{-1/2}.
\end{equation}
Using  Lemma \ref{th:lemmvar} and Chebyshev inequality again, we can find a constant $C_{10}$ such that for small $\gb$ and $n=\lceil\exp(C_{10}/\gb^2)\rceil$ we have
\begin{equation}
Q\left\{ W_n<1/2\right\}\le 1/2.
\end{equation}
This combined with \eqref{eq:ineq} implies that

\begin{equation}
\frac{-\log 2}{n}\le  n^{-1/4} + Q \frac 1 n \log W_n\le n^{-1/4}+p(\gb).
\end{equation}
Replacing $n$ by its value we get

\begin{equation}
 p(\gb)\ge  -n^{-1/4}-\frac{\log 2}{n}\ge -\exp (-C_{10}/5\gb^2).
\end{equation}

\end{proof}

\chapter[Brownian polymer models]{Pinning, superdiffusivity, and strong disorder for Brownian polymer models} \label{polBB}

\section{Introduction}
A lot of progress has been made lately in the understanding of localization/delocalization phenomena for random polymer models, and especially for two class of model: random pinning models, and directed polymer in a random environment (see the recent monographs \cite{cf:Book} and \cite{cf:denH} and references therein). These model where first introduced in a discrete setup, where the polymer is modeled by the graph of a random-walk in $\Z^d$ and the polymer-measure is a modification of the law of a simple random walk on $\Z^d$; but recently,  a lot of interest have been brought to 
the corresponding continuous models, involving Brownian Motion rather than simple random walk (see \cite{cf:CY_cmp, cf:RT, cf:BTV, cf:CKMV}). In this paper, we study two models which are a priori unrelated:
\begin{itemize}
 \item A Brownian homogeneous pinning model in a polynomial potential.
 \item A model of Brownian directed polymer in Brownian environment with spatial power-law correlations.
\end{itemize}
However, as it will appear in our development, homogeneous pinning models happen to be a natural tool to study directed polymer in random environment as soon as second moment methods are used (as in \cite{cf:Lac}).
\medskip

Let $P$ be the Wiener measure, and $(B_{t})_{t\in\R}$ the associated $d$-dimensional Brownian Motion starting from zero.
For a fixed $t$, to each trajectory we associate an energy defined by the Hamiltonian $H_t(B)$ which is a function of $(B_s)_{s\in[0,t]}$ and for an inverse temperature $\gb$ we define the polymer measure $\mu_t$ by its Radon-Nicodym derivative
\begin{equation}
\frac{\dd \mu_t}{\dd P}(B):=\frac{1}{Z_t}\exp(\gb H_t(B)).
\end{equation}
where 
\begin{equation}
 Z_t:=P\left[\exp\left(\gb H_t(B)\right)\right],
\end{equation}
is a renormalization factor called partition function of the system ($P[f(B)]$ denotes the expectation with respect to $P$, we will keep that notation for all the probability law we introduce in the paper). For fixed $(B_s)_{s\in[0,t]}$, $H_t(B)$ can be deterministic or random (depending on a random environment). Our aim is to describe asymptotic properties of the measure $\mu_t$ as $t\to\infty$, and to discuss the existence of a phase transition in the parameter $\gb$.
\medskip

We present the results obtained for the two models separately. To avoid confusion we will keep the above notation for the random directed polymer, and we will introduce a different one for the pinning model.

\subsection{Brownian homogeneous pinning in a polynomial tailed potential}

We consider the following homogeneous pinning model, that was studied in \cite{cf:CKMV}. We consider two cases:
$V$ is either a bounded continuous non--negative function of $\R^d$ such that $V(x)$ tends to zero when $x$ goes to infinity or  $V(x)=\ind_{\{\|x\|\le 1\}}$.

We define the energy of a continuous trajectory up to time $t$, $(B_{s})_{s\in\bbR}$ to be the Hamiltonian
\begin{equation}
 G_t(B):= \int_{0}^t V(B_s)\dd s.
\end{equation}
We define $\nu_t^{(h)}$ the Gibbs measure associated to that Hamiltonian and pinning parameter (or inverse temperature) $h\in \R$ to be
\begin{equation}
\dd \nu_t^{(h)}(B):=\frac{\exp(h G_t(B))}{Y^{(h)}_t}\dd P( B),
\end{equation}
where $Y_t$ denotes the partition function
\begin{equation}
 Y^{(h)}_t:=P \left[\exp(h G_t(B))\right].
\end{equation}

The essential question for this model is to know whether or not the pinning force $h$ is sufficient to keep the trajectory of the polymer near the origin, where $V$ takes larger values. It is intuitively clear that for large $h$, the potential localizes the polymer near the origin (the distance remains $O(1)$ as $t$ grows), and that $h\le 0$ has no chance to localize the polymer. Therefore, one has to find if the polymer is in the localized phase for all positive $h>0$ or if the phase transition occurs for some critical value of $h_c>0$. For This question was studied in \cite{cf:CKMV} in the case of compactly supported smooth $V$, where it was shown that localization holds for all $h>0$ if and only if $d\le 2$ (in fact this issue is strongly related to reccurence/transcience of the Brownian Motion). Our aim is to describe how this result is modified if the potential is not compactly supported and has polynomial decay at infinity.
Answering to this question relies on studying the free energy.
\medskip
\begin{proposition}\label{th:freeee}
 The limit
\begin{equation}
 \tf(h):=\lim_{t\to\infty} \frac{1}{t}\log Y^{(h)}_t.
\end{equation}
exists is and it is non-negative. $h\mapsto \tf(h)$ is a non-decreasing, convex function.
\end{proposition}
We call $\tf(h)$ the {\sl free energy} of the model. We define

\begin{equation}
 h_c=h_c(V):=\inf\{h: \tf(h)>0\}\ge 0.
\end{equation}

\medskip
The existence of the limit is not straight forward. Cranston et al.\ proved it in \cite[Section 7]{cf:CKMV} in the case of $C^{\infty}$ compactly supported potential, we adapt their proof to our case.
To understand why $\tf(h)>0$ corresponds to the localized phase, we remark the following:
convexity allows us to interchange limit and derivative, therefore at points where $\tf$ has a derivative
\begin{equation}
 \frac{\dd \tf}{\dd h}(h)=\lim_{t\to\infty} \frac{1}{t}\nu_t\left[\int_0^t V(B_s)\dd s\right].
\end{equation}
We want to be able to settle what are the condition on $V$ so that $h_c=0$ and to describe the behavior of $\tf$. Cranston et al.\ answered the question quite completely in \cite[Theorem 6.1]{cf:CKMV} for compactly supported potential using a functional analysis approach.
We present a part of their results, adapted for the ``step'' potential $V(x)=\ind_{\{\|x\|\le 1\}}$, that we will use as a tool in our paper.
Adapting the result from \cite[Theorem 6.1]{cf:CKMV} is straight  forward, using monotonicity in $V$ of the free energy.

\begin{theorem}\label{comp}
Let $V: \R^d\rightarrow \R_+$ defined as
\begin{equation}
 V(x)=\ind_{\{\|x\|\le 1\}}.
\end{equation}
Then for $d=1$, $2$ we have $h_c=0$ and as $h$ tends to zero,
\begin{equation}\label{freeen}\begin{split}
 \tf(h)&=2h^2(1+o(1)) \text{ for } d=1,\\
 \tf(h)&= \exp\left(-h^{-1}(1+o(1))\right) \text{ for } d=2.
\end{split}
\end{equation}
For $d\ge3$, we have $h_c>0$.
\end{theorem}

\begin{rem} \rm This model (especially in the compactly supported case) is the continuous equivalent of the discrete model for homogeneous pinning on a defect line (see \cite[Chapter 2]{cf:Book} for an account on the subject), and Theorem \ref{comp} is completely analogous to results obtained in the discrete. The discrete model has the particularity of being exactly soluble (see \cite{cf:Fisher}). The techniques used for the discrete model, based on renewal processes, do not seem to be adaptable in our setup.
\end{rem}

In the present study, we pay a particular attention to potentials $V$ with power law decay at infinity
i.e.\ such that there exists $\theta>0$ with 
\begin{equation}\label{poldecay}
 V(x)\asymp \|x\|^{-\theta} \text{ as } x\to \infty,
\end{equation}
where $f(x)\asymp g(x)$ as $x\to \infty$ means that there exist constants $R$ and $c$ such that
\begin{equation}
 c^{-1} f(x)\le g(x) \le c f(x) \quad \forall x,\  \|x\|\ge R.
\end{equation}
In this particular setup, the phase transition is different from the one observed 
for the case of compactly supported case.
We prove that in dimension $d \ge 3$, whether $h_c$ is equal to zero depends only on the exponent $\theta$. Furthermore, when the values of $\theta$ varies, so does the critical exponent which can take any value in $(1,\infty)$.

\begin{theorem}\label{pinning}
 For $d\ge 3$ we have
\begin{equation}\label{asymp0}
\theta> 2\Rightarrow h_c(V)>0.
\end{equation}
And when $h>0$ is small enough
\begin{equation}
 \lim_{t\to\infty} Y_t<\infty.
\end{equation}

Moreover when $\theta<2$
\begin{equation}\label{asymp1}
\tf(h)\asymp h^{\frac{2}{2-\theta}}. 
\end{equation}
\end{theorem}

For the lower dimensional cases, it follows from \cite[Theorem 6.1]{cf:CKMV} and monotony of $\tf$ in $V$ that $h_c=0$ for all $\theta$.
However, some of the features of Theorem \ref{pinning} still hold.
\begin{theorem}
 For $d=2$, $\theta<2$ and $d=1$, $\theta<1$, we have
\begin{equation}\label{asymp2}
 \tf(h)\asymp h^{\frac{2}{2-\theta}}.
\end{equation}
\end{theorem}
For the sake of completeness, we also present  the result for the case $d=1$, $\theta>1$. The following is the generalization of a result proved for compactly supported smooth function in \cite[Theorem 6.1]{cf:CKMV}.
\begin{proposition}\label{L1case}
For $d=1$, and $V\in \bbL_1(\bbR)$, $V$ continuous non--negative we have
\begin{equation}\label{L11}
 \tf(h)\sim \frac{\|V\|_{\bbL_1(\R)}^2}{2}h^2.
\end{equation}
\end{proposition}
 In the case $d=2$, $\theta\ge 2$, it can be checked that monotonicity of the free energy  and \eqref{freeen} implies that $h_c>0$ and that $\tf(h)$ decays faster that any polynomial around $h=0$. In the case $d=1$, $\theta=1$, monotonicity again implies that $\tf(h)$ behaves like $h^{2+o(1)}$.

\begin{rem} \rm
 In \cite{cf:CKMV}, the critical behavior of the free--energy is computed (in a sharp way) for every dimension, even in the case where $d\ge 3$ as it can be done in the discrete case. As it will be seen in the proof, the method for getting critical exponents in the present paper fails to give any result when $h_c>0$. However, it would be natural to think that for some value of $\theta>2$, the critical exponent for the free energy is independent of $d$ when $d$ is large enough.
\end{rem}

\subsection{Brownian polymers in brownian environment}

\subsubsection{Weak and strong disorder}
We describe now the model of Brownian polymer introduced by Rovira and Tindel in \cite{cf:RT}, and we introduce the results we found for that model.
Let $\left(\go(t,x)\right)_{(t\in \R_+,x\in \R^d)}$ be a real Gaussian-Field with covariance function
\begin{equation}
\bP \left[\go(t,x)\go(s,y)\right]=:(t\wedge s) Q(x-y). 
\end{equation}
Where $Q$ is a continuous non-negative function, going to zero at infinity.
Informally, the field can be seen as an addition in time of independent, infinitesimal translation invariant fields $\go(\dd t,x)$ with covariance function $Q(x-y)\dd t$. To avoid normalization, we assume $Q(0)=1$. We define the random Hamiltonian formally as
\begin{equation}
 H_{\go,t}(B)=H_t(B):=\int_0^{t}\go(\dd s,B_s).
\end{equation}
For a straight definition of $H_t$ we refer to \cite[Section 2]{cf:BTV}, where a rigorous meaning of the above formula is given. However, to have a meaningful and cheap definition it is sufficient to consider that $\left(H_t(B)\right)$ is a centered Gaussian vector indexed by the continuous function $B\in C [0,t]$, with covariance matrix
\begin{equation}
\bP\left[H_t(B^{(1)})H_t(B^{(2)})\right]:=\int_{0}^{t}Q(B^{(1)}_s-B^{(2)}_s)\dd s.
\end{equation}

One defines the (random) polymer measure for inverse temperature $\gb$ as a transformation of the Wiener measure $P$ as follows
\begin{equation}
 \dd \mu^{\gb,\go}_t(B):=\frac{1}{Z^{\gb,\go}_t}\exp(\gb H_t(B))\dd P(B),
\end{equation}
where $Z^{\gb,\go}_t$ is the partition function of the model
\begin{equation}
 Z^{\gb,\go}_t:=P \left[\exp(\gb H_t)\right].
\end{equation}
It is  more convenient to work with the renormalized partition function
\begin{equation}
 W^{\gb,\go}_t=W_t:=P\left[\exp\left(\int_0^{t} \gb\go(\dd s,B_s)-\gb^2/2\dd s\right)\right].
\end{equation}
It can be checked without much effort that $W_t$ is a positive martingale with respect to 
\begin{equation}
 \mathcal F_t:=\sigma\left\{ \go_s, s\le t\right\}.
\end{equation}
Therefore it converges to 
a limit $W_{\infty}$. An argument used by Bolthausen in a discrete setup \cite{cf:B}, can be adapted in our setup to show that
\begin{equation}
\bP\left\{ W_{\infty}:=0\right\}\in\left\{0,1\right\}.
\end{equation}
To know wether or not $W_\infty$ is degenerate is a crucial issue for the asymptotic behavior of $\mu_{\gb,\go}$. 
It is known (in the discrete setup see for example \cite{cf:CY}, the argument can  probably be adapted for the present model) that when $W_{\infty}$ is non--trivial, the behavior of $B$ under $\mu_t$ when $t$ is large is essentially diffusive, this situation is referred to as {\sl weak disorder}. The other situation is referred to as {\sl strong disorder}. In the Gaussian setup, a partial annealing argument can show that increasing $\gb$ increases the influence of disorder, so that there exists a critical value $\gb_c$ separating the two phases, i.e., that there exists $\gb_c\in[0,\infty)$ such that
\begin{equation} \begin{split}
\gb\in(0,\gb_c)\ &\Rightarrow \text{ weak disorder holds}, \\
\gb >\gb_c\ &\Rightarrow \text{ strong disorder holds}.
                 \end{split}
\end{equation}

\medskip

From physicists point of view, it is however more natural to have a definition of strong disorder in free energy term.
The quantity of interest is the difference between quenched and annealed free energy. 
\medskip

\begin{proposition}
 The a.s. limit
\begin{equation}
 p(\gb):=\lim_{t\rightarrow \infty}\frac{1}{t}\log W_t=\lim_{t\rightarrow \infty}\frac{1}{t}\bP\left[\log W_t\right]=:\lim_{t\rightarrow \infty} p_t(\gb).
\end{equation}
exists and is almost surely constant. The function $\gb\mapsto(\gb)$ is non positive and non increasing.
\end{proposition}

\medskip

We can define 
\begin{equation}
\bar{\gb}_c:=\sup\left\{ \gb>0 \text{ such that } p(\gb)=0\right\}. 
\end{equation}
It is obvious from the definitions that $\gb_c\le \bar\gb_c$.
\medskip

For a proof of existence of the limits above and the equality, we refer to \cite[Lemma 2.4, Proposition 2.6]{cf:RT}.
The non-positivity follows from Jensen inequality
\begin{equation}
 \bP\left[\log W_t\right]\le \log \bP\left[W_t\right]=0.
\end{equation}

It can be shown (for result in the discrete setup see \cite{cf:CH_ptrf, cf:CSY}) that an exponential decay of $W_t$ corresponds to a significant localization property of the trajectories. More precisely, under this condition, it can be shown that two paths chosen independently with law $\mu^{\gb,\go}_{t}$ tend to spend a time proportional to $t$ in the same neighborhood.
For example whenever the right hand side exists (i.e. everywhere except maybe for countably many $\gb$ as $p(\gb)+\gb^2/2$ is convex) we have
\begin{equation}
 \frac{\dd p}{\dd \gb}(\gb):=\lim_{t\to\infty}\frac{1}{t\gb}\bP\left[\mu_{\gb,t}^{\otimes 2}\left(\int_{0}^t\left(Q(B^{(1)}_s,Q(B^{(2)}_s\right)\right)\right].
\end{equation}
It has become customary to refer to this situation as {\sl very strong disorder}.

It is widely expected that the two notion of {\sl strong disorder} coincide outside of the critical point and that we have $\gb_c=\bar \gb_c$. However it remains an unproved conjecture. In \cite{cf:CV} and \cite{cf:Lac} is has been show that for model in $\Z^d$ with i.i.d.\ site disorder, very strong disorder holds at all temperature, and it has been known well before that there is a non trivial phase transition as soon as $d\ge 3$. The same is expected to hold for in continuous time, if the correlation function $Q$ decays sufficiently fast at infinity. In this paper we place ourselves in the case where $Q$ has power law decay and investigate the behavior of the free energy for small $\gb$. Unless otherwises stated we will consider that there exists $\theta>0$ such that
\begin{equation}\label{Qpold}
 Q(x)\asymp \|x\|^{-\theta}.
\end{equation}
We have obtained the following conditions for strong/weak disorder to hold.
\begin{theorem}\label{ddd333}
For $d\ge3$ we have:\begin{itemize}
                   \item $\gb_c$ and $\bar \gb_c>0$ if $\theta>2$.
                   \item $\bar \gb_c=0$ if $\theta<2$.
                  \end{itemize}
For $d=2$, $\gb_c=\bar\gb_c=0$ if $\theta<2$. For $d=1$, $\gb_c=\bar \gb_c=0$ for any $\theta$.
\end{theorem}

\begin{rem} \rm In the cases we have left unanswered, namely $d=2,\ \theta\ge 2$ or $d=3$, $\theta=2$, the technique used for the two dimensional discrete case (see \cite{cf:Lac}) can be adapted to prove that $\bar \gb_c=0$. Since the method is relatively heavy, and very similar to what is done in the discrete case, we do not develop it here.
 \end{rem}

In the cases where $\bar \gb_c=0$, we obtained a sharp bound on the free energy
\begin{theorem}\label{dddd4444}
For $d\ge 2$, $\theta<2$, or $d=1$, $\theta<1$ we have
\begin{equation}
 p(\gb)\asymp -\gb^{\frac{2}{2-\theta}}.
\end{equation}
For $d=1$, $Q\in \bbL_1(\R)$ (with no other assumption on the decay), we have
\begin{equation}
p(\gb)\asymp -\gb^{4}. 
\end{equation}
\end{theorem}

\subsubsection{Superdiffusivity}

Another much studied issue on which a lot of study where lead for directed polymer is the superdiffusivity phenomenon \cite{cf:Piza,cf:P,cf:J,cf:M, cf:BTV}.
As mentioned earlier, in the weak disorder phase, the trajectory of the polymer conserves all the essential features of the non disordered model (i.e. standard Brownian motion).
Therefore, in one looks at a trajectory up to time $t$, the end position of the chain, the maximal distance to the origin, the typical distance of a point in the chain to the origin are of order $t^{1/2}$, this phenomenon is one of the features of  a {\sl diffusive} behavior.
It is believed that in the {\sl strong disorder} phase, this property is changed, and that the quantities mentioned earlier are greater than $t^{1/2}$. 
From the physicists point of view, it is in fact widely expected that their exist a positive real number $\xi>1/2$ such that under $\mu_t$
\begin{equation}
\max_{s\in[0,t]}\|B_s\|\asymp t^{\xi}.
\end{equation}
We refer to $\xi$ as the volume exponent. It is usually believed that in the strong disorder phase, $\xi$ does not depend on the temperature, and it is the equal to the exponent of the associated oriented last passage percolation model which corresponds to zero temperature (see \cite{cf:Piza}).
However, rigorously, the very existence of these exponents cannot be proved in general, and their is no consensus in the physics literature for their value when the transversal dimension is more than $2$.
In the one dimensional case, for the discrete model on $\Z^d$ and for the one presented here,  if the correlations are not too important ($Q\in\bbL_1(\bbR)$ seems to be the reasonable condition), the volume exponent is conjectured to be equal to $2/3$. This identity remains however a very challenging issue. Johansson proved \cite{cf:J} that it hold for a special case of last-passage oriented percolation, but the method he employed cannot be generalized to the case of finite temperature. To our knowledge this is the only case where the existence of $\xi$ is proved in a non-diffusive case. In the sequel we talk about bounds on $\xi$ to mention bound on the volume fluctuation.
\\
For undirected last-passage percolation Newman and Piza proved that $\xi\le 3/4$, and, in collaboration with Licea, that $\xi\ge 3/5$, using geometric arguments. 
In the case of directed polymer, some results have been obtained in dimension $1$ by Petermann \cite{cf:P} and Mejane \cite{cf:M} for a model of directed polymer in a semi-discrete setup. Petermann proved that with a high probability, a polymer chain of size $n$ would go out of a box of size $n^\alpha$ if $\alpha<3/5$, whereas Mejane proved a result going in the other direction saying that the chain stays in a box of size $n^\alpha$ if $\alpha>3/4$. The result of Mejane extends to every dimension.
Informally this says that
\begin{equation}
 3/5 \le \xi \le 3/4.
\end{equation}
The result of Petermann has recently been adapted for Brownian polymer in Brownian environment by Bezerra, Tindel and Viens \cite{cf:BTV}.
\medskip

In this paper, we present some new results in the case of polynomial decay of $Q$ (condition \eqref{Qpold}). In particular, we prove that under some assumption, superdiffusivity holds in every dimension and at every temperature

\begin{theorem}\label{supdiffff}
 When $d\ge 2$ and $\theta<2$ or $d= 1$ and $\theta<1$, we have
\begin{equation}
\lim_{\gep\to 0}\liminf_{t\to\infty} \bP \mu_t^{\gb,\go} \left\{\sup_{0\le s\le t} |B_s|\ge \gep t^{\frac{3}{4+\theta}}\right\}=1.
 \end{equation}
For $d=1$, $Q\in \bbL_1(\bbR)$ we have
\begin{equation}
 \lim_{\gep\to 0}\liminf_{t\to\infty} \bP \mu_t^{\gb,\go} \left\{\sup_{0\le s\le t} |B_s|\ge \gep t^{\frac{3}{5}}\right\}=1.
\end{equation}
\end{theorem}

 We will not go into the details for the proof of the case $Q\in \bbL_1(\bbR)$ as it is very similar to what we do in the other cases.

\begin{rem}\rm
 In terms of exponent, the above theorem, states that $\xi\ge 3/(4+\theta)$. We strongly believe that the argument of Mejane in \cite{cf:M} can be adapted to this case (maybe with another definition for $\xi$, using the position of the end point rather than the maximal distance), so that we can say at an informal level that
\begin{equation}
 \frac{3}{4+\theta}\le \xi\le \frac{3}{4}.
\end{equation}
The two bounds can get arbitrarily close to eachother when $\theta$ get close to zero, showing that when $\theta$ is small, our result is close to being optimal.
\end{rem}

\begin{rem}\rm
 There is an oriented-last-passage percolation model associated to this model. It corresponds to the case $\gb=\infty$. The result we prove here for finite temperature extends to this case that we describe here.
Given a realization of the environment $\go$ we consider the Hamiltonian acting on piecewise $C_1$ path of length $t$ in $\R^d$:
\begin{equation}
 H^{\infty}_t(f):=\int_0^t \left\| \nabla f (s) \right\|^2+\int_0^t \go(\dd s, f(s)).
\end{equation}
The oriented-last passage percolation problem consists in finding a path $f$ such that the Hamiltonian is maximized. If one admits the existence of such a path $f_{\max}$, it is not very difficult to adapt the proof of the directed polymer problem to prove that with a big $\bP$ probability
\begin{equation}
 \max_{s\in[0,t]}\| f_{\max} (s)\|\ge t^{\frac{3}{4+\theta}}.
\end{equation}

\end{rem}

\begin{rem}\rm
 The argument we use for our proof uses change of measure argument instead of purely Gaussian computation, in our views, this makes the computation much clearer. Besides, our proof is shorter and goes substancially farther.
\end{rem}

\begin{rem}\rm
 Another model of Brownian polymer in random environment in a Poissonian environment has been introduced and studied by Comets and Yoshida \cite{cf:CY_cmp}. We would like to stress that our proofs do not rely on the Gaussian nature of the environment, and that superdiffusivity with exponent $3/5$, as well as very strong disorder in dimension $1$ and $2$ (in dimension $2$ one needs to adapt the method used in \cite{cf:Lac}) could be proved for this model as well by using methods developed in the present paper. We focused on Brownian polymer mainly because it was the natural model to study the effect of long spatial memory.
\end{rem}

\section{Brownian pinning in polynomial potential}

\subsection{Basic results} In this subsection, we give the proof of minor results that are easy consequences of results in \cite{cf:CKMV}.

\begin{proof}[Proof of Proposition \ref{th:freeee}]

Given $V$, and $\gep$ and one can find a compactly supported $C^{\infty}$ function $\check V$ such that
\begin{equation}
   \check V \le V \le \check V+ \gep.
\end{equation}

We write $\check Y_t=\check Y_t^{(h)}$ for the partition function corresponding to $\check V$. 
Trivially we have for every $t$
\begin{equation}
  \frac{1}{t}\log \check Y_t\le  \frac{1}{t}\log Y_t \le \frac{1}{t}\log \check Y_t+h\gep.
\end{equation}
As proved in  \cite[Section 7]{cf:CKMV}, $\log \check Y_t$ converges as $t$ goes to infinity so that,
\begin{equation}
 \limsup \frac{1}{t}\log Y_t- \liminf \frac{1}{t}\log Y_t\le h\gep.
\end{equation}
The proof can also be adapted to prove the existence of the free--energy for the potential 
\begin{equation}
 V(x):=\ind_{\{\|x\|\le 1\}}.
\end{equation}
We omit the details here.
\end{proof}

\medskip

\begin{proof}[Proof of proposition \ref{L1case}]

First we prove the upper bound. By the occupation times formula (see e.g.\ \cite[pp. 224]{cf:RY} ) we have 
\begin{equation}
\int_0^t V(B_s)\dd s= \int_{\R} L_t^xV(x)\dd x,
\end{equation}
where $L_t^x$ is the local time of the Brownian motion in $x$ at time $t$.
By Jensen inequality we have then
\begin{equation}
 \exp\left(h \int_{\R} L_t(x)V(x)\dd x\right)\le \int_{\R}\frac{V(x)\dd x}{\|V\|_{\bbL_1(\R)}}\exp\left(h \|V\|_{\bbL_1(\R)}L_t^x\right).
\end{equation}
Moreover, under the Wiener measure with initial condition zero, $L_t^x$ is stochastically bounded (from above) by $L_t^0$ for all $x$, so that
\begin{multline}\label{boundL1}
 Y_t\le \int_{\R}\frac{V(x)\dd x}{\|V\|_{\bbL_1(\R)}} P\left[\exp\left(h \|V\|_{\bbL_1(\R)}L_t^x\right)\right]\\
\le P\left[\exp\left(h \|V\|_{\bbL_1(\R)}L^0_t\right)\right]\le 2 \exp\left(\frac{t h^2 \|V\|^2_{\bbL_1(\R)}}{2}\right).
\end{multline}
In the last inequality we have used the fact that $L_t^0\stackrel{\mathcal{L}}{=}\sqrt{t}|\mathcal N(0,1)|$.
Taking the limit as $t$ tends to infinity gives the upper bound.
For the lower bound, the assumption we have on $V$ guaranties that give $\gep>0$, we can find $\check V$ smooth and compactly supported such that 
\begin{equation}\begin{split}
\check V &\le V,\\ 
\|\check V\|_{\bbL_1(\R)}& \ge \| V \|_{\bbL_1(\R)}-\gep.
\end{split}\end{equation}
Let $\check \tf$ be the free energy associated to $\check V$.
For \cite[Theorem 6.1]{cf:CKMV} and monotonicity we have that for $h$ small enough (how small depending on $\gep$)
\begin{equation}
 \tf(h)\ge \check\tf(h)\ge \left(\frac{\|\check V\|_{\bbL_1(\R)}^2}{2}-\gep\right)h^2.
\end{equation}
As $\gep$ is chosen arbitrarily this gives the lower bound.
\end{proof}

\subsection{Proof of upper bound results on the free--energy}

In this subsection we prove the upper bounds corresponding to \eqref{asymp0}, \eqref{asymp1} and \eqref{asymp2}.
Thourough this section we will use the following tool.

\begin{lemma}\label{lemrelou}
 Let $(a_n)_{n\in\N}$, be a sequence of positive integers, and $(p_n)_{n\in\N}$, a sequence of strictly positive real numbers satisfying $\sum_{n\in \N} p_n^{-1}=1$.
Then we have
\begin{equation}\label{prood}
\prod_{n\ge 1} a_n\le \sum_{n\ge 1} p_n^{-1}a_n^{p_n},
\end{equation}
as soon as the left-hand side is defined. This formula is also valid for a  product with finitely many terms.
\end{lemma}
\begin{proof}
Let $X$ be the random variable whose distribution $P$ is defined
\begin{equation}
 P\{X=x\}=\sum_{\{n\,:\,p_n\log a_n=x\}}\frac{1}{p_n}.
\end{equation}
The formula considered is just Jensen inequality:
\begin{equation}
 \exp(P[X])\le P \exp(X).
\end{equation}
\end{proof}

Now we go the proof of delocalization at high temperature for $\theta>2$ and $d\ge 3$

\begin{proof}[Proof of \eqref{asymp0}]
Let $\theta>2$, and $\gep>0$ we define 
\begin{equation}
 \bar{V}(x):=\sum_{i=0}^{\infty} \ind_{\{\|x\| \le 2^n\}}2^{-n\theta}.
\end{equation}
It follows from assumption \eqref{poldecay} that there exists constants $c_1$ and $C_1$ such that
\begin{equation}
     c_1 \bar V(x) \le V(x)\le C_1  \bar V (x).
\end{equation}

Hence, for any $p\in(0,1)$ and $h>0$, we have
\begin{equation}\begin{split}
Y_t&\le \exp\left(hC_1 \int_{0}^t \bar{V}(B_s)\dd s\right)\\
&\le (1-p) \sum_{n=0}^{\infty}p^{n} P\left[ \exp\left((1-p)^{-1}p^{-n}hC_1\int_0^t\ind_{\{\|B_s\| 
\le 2^n\}}2^{-n\theta}\}\dd s \right)\right]\\
&= (1-p)\sum_{n=0}^\infty p^n P\left[\exp\left((1-p)^{-1}p^{-n}2^{n(2-\theta)}h C_1 \int_0^{2^{-2n}t}\ind_{\{\|B_s\| \le 1\}}\dd s\right)\right],
\end{split}\end{equation}
where the second inequality uses Lemma \ref{lemrelou} with $p_n:= (1-p)p^n$, and  the last equality is just a rescaling of the Brownian Motion.
We choose $p$ such that $p2^{\theta-2}=1$. We get
\begin{equation}
Y_t\le \exp\left((1-p)^{-1}h C_1\int_0^{t}\ind_{\{\|B_s\| \le 1\}}2^{-n\theta}\dd s\right).
\end{equation}
For $h$ small enough, Theorem \ref{comp} for $d\ge 3$ allow us to conclude. Moreover (see \cite[Proposition 4.2.1]{cf:CY_cmp}) in this case, we have 
\begin{equation}
 \lim_{t\rightarrow \infty} Y_t:=Y_{\infty}<\infty.
\end{equation}

\end{proof}

\begin{proof}[Proof of the upper bounds of \eqref{asymp1} and \eqref{asymp2}]
From assumption \eqref{poldecay}, there exist a constant $C_2$ such that for any $h\le 1$,
\begin{equation}\label{boundL12}\begin{split}
 V(x)&\le C_2 h^{\frac{\theta}{2-\theta}} \quad \forall x, \ |x|\ge h^{-\frac{1}{2-\theta}},\\
\int_{|x|\le  h^{-\frac{1}{2-\theta}}}V(x)\dd x &\le C_2 h^{-\frac{1-\theta}{2-\theta}}.
\end{split}\end{equation}
We write $\check V(x):= V(x)\ind_{\left\{|x|\le h^{-\frac{1}{2-\theta}}\right\}}$.
We have
\begin{equation}
 V(B_s)\le \check V (B_s)+C_2 h^{\frac{\theta}{2-\theta}},
\end{equation}
so that
\begin{equation}
 \log Y_t \le t C_2 h^{\frac{2}{2-\theta}}+\log P\left[\exp \left(h\int_0^t\check V(B_s)\dd s\right)\right]. 
\end{equation}
We know from \eqref{boundL1} and \eqref{boundL12}  that the second term is smaller than
\begin{equation}
 \log 2+ t\frac{\|\check V\|^2_{\bbL_1(\R)}}{2}\le \log 2+ t C_2^2 h^{\frac{2}{2-\theta}}.
\end{equation}
Which is the desired result.
\medskip
Now we consider the case $d\ge2$, $\theta<2$.
Define $\bar n=\bar n_{h}:=\lceil |\log h|/(2-\theta)\log 2 \rceil- K$ for some large integer $K$.
We have
\begin{equation}
 \sum_{n>\bar n} 2^{-n\theta} \ind_{\{\|B_s\|\le 2^{-n}\}}\le  \sum_{n>\bar n} 2^{-n\theta}=\frac{2^{-(\bar n+1)\theta}}{1-2^{-\theta}}.
\end{equation}
Therefore we can find a constant $C_3$ (depending on $K$ and $C_1$) such that
\begin{equation}\label{stepep}
Y_t\le P\left[\exp \left(C_1 h\int_{0}^{t} \bar V(B_s)\dd s\right)\right]\le e^{C_3 h^{\frac{2}{2-\theta}}}P\left[ \exp\left(C_1h\int_{0}^{t} \sum_{n=0}^{\bar n} 2^{-n\theta} \ind_{\{\|B_s\|\le 2^{-n}\}}\dd s\right)\right].
\end{equation}
We study the behavior of the second term on the right--hand side. For any $p$ (We choose $p=2^{\theta/2-1}$), as a result of Lemma \ref{lemrelou}\ , we have
\begin{equation}\begin{split}
 &P \left[\exp\left(C_1h\int_{0}^{t} \sum_{n=0}^{\bar n} 2^{-n\theta} \ind{\{\|B_s\|\le 2^{-n}\}}\dd s\right)\right]\\
&\quad \le
\sum_{n=0}^{\bar n}\frac{(1-p)p^n}{1-p^{\bar n+1}} P \left[\exp\left(C_1h \int_{0}^{t}\frac{p^{-n}(1-p^{\bar n+1})}{1-p} 2^{(n-\bar n) \theta} \ind_{\{\|B_s\|\le 2^{-(\bar n-n})\}}\dd s \right)\right]\\
& \quad =\sum_{n=0}^{\bar n}\frac{1-p}{1-p^{\bar n+1}}p^n P\left[ \exp\left(C_1h \int_{0}^{t2^{-2(\bar n-n)}}\frac{p^{-n}(1-p^{\bar n+1})}{1-p} 2^{(\bar n-n)(2- \theta)} \ind_{\{\|B_s\|\le 1\}}\dd s \right)\right],
\end{split}\end{equation}
where the last line is simply obtained by rescaling the Brownian motion in the expectation.
Now observe that for any $\gep>0$, one can find a value of $K$ such that
\begin{equation}
 C_1 h\le (1-p)\gep 2^{-\bar n(2-\theta)},
\end{equation}
so that
\begin{equation}
 C_1h\frac{p^{-n}(1-p^{\bar n+1})}{1-p} 2^{(\bar n-n)(2- \theta)}\le \gep 2^{n(\theta/2-1)}.
\end{equation}
Therefore we have
\begin{multline}
  P\left[ \exp\left(C_1h\int_{0}^{t} \sum_{n=0}^{\bar n} 2^{-n\theta} \ind{\{\|B_s\|\le 2^{-n}\}}\dd s\right)\right]\\
\le \max_{n\in\{0,\dots, \bar n\}} P\left[\exp \left(\int_{0}^{t2^{-2(\bar n-n)}}\gep 2^{n(\theta/2-1)}\ind_{\{\|B_s\|\le 1\}}\dd s \right)\right].
\end{multline}
For $d\ge 3$ the right--hand side is less than
\begin{equation}
P\left[\exp \left(\int_{0}^{t}\gep \ind_{\{\|B_s\|\le 1\}}\dd s \right)\right], 
\end{equation}
which stays bounded as $t$ goes to infinity.
For $d=2$, if $t$ is sufficiently large, and $\gep$ small enough Theorem \ref{comp} allows us to write
\begin{equation}
\log  P\left[\exp \left(\int_{0}^{t2^{-2(\bar n-n)}}\gep 2^{n(\theta/2-1)}\ind_{\{\|B_s\|\le 1\}}\dd s \right)\right]\le t 2^{-2\bar n} 2^{2n} \exp\left(-2\gep^{-1} 2^{n(1-\theta/2)}\right).
\end{equation}
The maximum over $n$ of the right--hand side is attained for $n=0$. Therefore
\begin{equation}
  \log  P \left[\exp\left(C_1h\int_{0}^{t} \sum_{n=0}^{\bar n} 2^{-n\theta} \ind{\{\|B_s\|\le 2^{-n}\}}\dd s\right)\right]\le 2^{-2\bar n}.
\end{equation}
Injecting this in \eqref{stepep} ends the proof.
\end{proof}

\subsection{Proof of lower bounds on the free-energy}

In this section we the lower bounds for the  \eqref{asymp1} and \eqref{asymp2}.
We start with the case $d=1$, $\theta<1$.

For any $n\in \N$,
\begin{equation}
 Y_t\ge P\left[ \exp\left( c_1 h \int_0^t \bar V(B_s)\dd s\right)\right]\ge P\left[\exp\left(c_1 h 2^{-n\theta}\int_{0}^t \ind_{\{ \|B_s\|\le 2^{-n}\}}\dd s\right)\right].
\end{equation}
Rescaling the Brownian motion we get
\begin{equation}
Y_t\ge P\left[\exp\left(c_1 h 2^{n(2-\theta)}\int_{0}^{t2^{-2n}} \ind_{\{ \|B_s\|\le 1\}}\dd s\right)\right].
\end{equation}
We can choose $n=n_{h}=\lceil |\log h|/(2-\theta)\log 2 \rceil+ K$, for some integer $K$. Let $C_2>0$ be such that
\begin{equation}
 \lim_{t\rightarrow \infty} \frac{1}{t}\log P \left[\exp\left(\int_{0}^tC_2 \ind_{\{\|B_s\|\le 1\}}\dd s\right)\right] \ge 1.
\end{equation}
By choosing $K$ large enough, we can get 
\begin{equation}
 \log Y_t\ge  \log P \left[\exp\left(C_2 \int_{0}^{t2^{-2n_{h}}} \ind_{\{ \|B_s\|\le 1\}}\dd s\right)\right] \ge  t 2^{-2n_{h}}.
\end{equation}
From this
we get that $h_c=0$, and that

\begin{equation}
 \tf(h)\ge 2^{-2(K+1)} h^{\frac{2}{2-\theta}}.
\end{equation}
\qed

\section{The Brownian Polymer}

\subsection{Lower bounds on the free energy, the second moment method}
In this section, will make use of result obtained for homogeneous pinning model to get some lowerbound on the free energy and prove the corresponding half of Theorem \ref{ddd333} and \ref{dddd4444}. A simple way to relate the two models is to remark that the second moment of $W_n$ is the partition function of an homogeneous pinning model. 
We start by a short proof of the fact that weak disorder holds for small $\gb$ if $d\ge 3$, $\theta>2$,

\begin{proof}
It is sufficient to show that $W_t$ converges in $L_2$ for $\gb$ sufficiently small.
We have
\begin{equation}\begin{split}
\bP\left[ W_t^2\right]&=\bP\left[ P^{\otimes 2} \exp\left(\int_0^{t}\left[ \gb\go(\dd s,B^{(1)}_s)+ \gb\go(\dd s,B^{(2)}_s)\right]-\gb^2\dd s\right)\right]\\
          &=P^{\otimes 2} \left[\exp\left(\gb^2\int_{0}^t Q(B_s^{(1)}-B_s^{(2)})\dd s\right)\right].
\end{split}
\end{equation}
The left--hand side is the partition function of the homogeneous pinning model described in the first section. Therefore, the result is a simple consequence of Proposition \ref{pinning}.
\end{proof}
\medskip

We prove now the lower--bound on the free energy corresponding to Theorem \ref{dddd4444}. We use a quite involved method called replica-coupling. The idea to use such a method for directed polymers came in \cite{cf:Lac} and was inspired by a work on pinning model of Toninelli \cite{cf:T_cmp}.
\medskip

\begin{proof}
Define for $\gb>0$, $r\in[0,1]$
\begin{equation}
 \Phi_t\left(r,\gb\right):=\frac{1}{t}\bP \left[\log P \exp\left(\int_0^{t}\sqrt{r}\gb \go(\dd s,B_s)-r\gb^2/2\dd s\right)\right].
\end{equation}
and for $\gb>0$, $r\in[0,1]$, $\gl>0$
\begin{multline}
 \Psi_t\left(r,\gl,\gb\right):=\\
\frac{1}{2t}\bP\left[ \log P^{\otimes 2} \exp\left(\int_0^{t}\sqrt{r}\gb\left[\go( \dd s,B^{(1)}_s)+ \go(\dd s,B^{(2)}_s)\right]-\gb^2\left[r+\gl Q(B_s^{(1)}-B_s^{(2)})\right]\dd s\right)\right]\\
=:\frac{1}{2t}\bP\left[ \log P^{\otimes 2} \exp\left(\hat H_t(B^{(1)},B^{(2)},r,\gl)\right)\right].
\end{multline}
The function $r\mapsto \Phi_t\left(r,\gb\right)$, satisfies
\begin{equation}
 \Phi_t(0,\gb)=0 \quad \text{and} \quad \Phi_t(\gb)=p_t(\gb).
\end{equation}
In the sequel we use the following version of the Gaussian integration by part formula. The proof is straight--forward.
\begin{lemma}
Let $(\go_1,\go_2)$ be a two--dimensional Gaussian vector. We have
\begin{equation}
 \bP \left[\go_1f(\go_2)\right]:=\bP[\go_1\go_2]\bP\left[f'(\go_2)\right].
\end{equation}
\end{lemma}
Using this formula we get that
\begin{equation}\label{phhhi}
 \frac{\dd}{\dd r}\Phi_t(r,\gb)=-\frac{\gb^2}{2t}\bP \left[\left(\mu^{(\sqrt{r}\gb)}_t\right)^{\otimes 2}\left[\int_0^t Q(B_s^{(1)}-B_s^{(2)})\dd s\right]\right].
\end{equation}
Doing the same for $\Psi_t$ we get
\begin{multline}\label{psssi}
 \frac{\dd }{\dd r} \Psi_t(r,\gl,\gb)=\frac{\gb^2}{2t}\bP\left[\frac{P^{\otimes 2}\exp\left(\hat H_t(B^{(1)},B^{(2)},r,\gl)\right) \int_0^t Q(B_s^{(1)}-B_s^{(2)})\dd s}{P^{\otimes 2}\exp\left(\hat H_t(B^{(1)},B^{(2)},r,\gl)\right)}\right]\\
-\frac{\gb^2}{t}\bP\left[\frac{P^{\otimes 4}\exp\left(\hat H_t(B^{(1)},B^{(2)},r,\gl)+\hat H_t(B^{(3)}, B^{(4)},r,\gl)\right)\int_{0}^{t}Q(B_s^{(1)}-B_s^{(3)})\dd s}{P^{\otimes 4}\exp\left(\hat H_t(B^{(1)},B^{(2)},r,\gl)+\hat H_t(B^{(3)}, B^{(4)},r,\gl)\right)}\right]\\
\le \frac{\gb^2}{2t}\bP\left[\frac{P^{\otimes 2}\exp\left(\hat H_t(B^{(1)},B^{(2)},r,\gl)\right)\int_0^t Q(B_s^{(1)}-B_s^{(2)})\dd s}{P^{\otimes 2}\exp\left(\hat H_t(B^{(1)},B^{(2)},r,\gl)\right)}\right]=\frac{\dd }{\dd \gl} \Psi_t(r,\gl,\gb).
\end{multline} 
This implies that for every $r\in[0,1]$ we have
\begin{equation}
 \Psi_t(r,\gl,\gb)\le \Psi(0,\gl+t,\gb).
\end{equation}
In view of \eqref{phhhi} and \eqref{psssi}, using convexity and monotonicity of $\Psi_t$ with respect to $\gl$ and $\Psi_t(r,0,\gb)=\Phi_t(r,\gb)$ we have
\begin{multline}
 -\frac{\dd }{\dd r}\Phi_t(r,\gb)=\left.\frac{\dd}{\dd \gl}\Psi_t(r,\gl,\gb)\right|_{\gl=0}\\ \le
\frac{\Psi_t(r,2-r,\gb)-\Phi_t(r,\gb)}{2-r}\le \Psi_t(0,2,\gb)-\Psi_t(r,\gb),
\end{multline}
where the last inequalities uses that $r\le 1$.
Integrating this inequality between zero and one, we get
\begin{equation}
 p_t(\gb)\ge (1-e)\Psi_t(0,2,\gb).
\end{equation}
The right hand side can be interpreted as
\begin{equation}
 \Psi_t(0,2,\gb)=\frac{1}{2t}\log Y_t,
\end{equation}
where $Y_t$ is the partition function of an homogeneous pinning model in polynomial tailed potential with pinning parameter $2\gb^2$.
Therefore we know from Theorem \ref{pinning} and Proposition \ref{L1case} that
\begin{equation}
 \lim_{t\rightarrow\infty}\frac{1}{t}\log Y_t\asymp \gb^{\frac{4}{2-\theta}} \text{ or } \lim_{t\rightarrow\infty}\frac{1}{t}\log Y_t\asymp \gb^4,
\end{equation}
(where the case to be considered depend on the assumption we have on $Q$).
This ends the proof.
\end{proof}

\subsection{Proof of upper bounds on the free energy (Theorem \ref{dddd4444}}

The technique of the proof is mainly based on change of measure argument. This method developed and used first for pinning models \cite{cf:GLT} and and adapted for directed polymer in \cite{cf:Lac}. Here, we have to adapt it to the continuous case, and to make benefit of the existence of space memory. We briefly sketch the proof:
\begin{itemize}
 \item First we use Jensen inequality to reduce the proof to estimating a fractional moment (a non-integer moment) of $W_t$.
 \item We decompose $W_t$ into different contribution corresponding to path along a corridor.
 \item For each corridor, we change slightly the measure via a tilting procedure which lowers the value of $\go$ in the corridor.
 \item We use the change of measure to estimate the fractional moment of each contributions.
\end{itemize}
\medskip
We start by stating a trivial Lemma, which will be of good use for our proof, and for the next section
\medskip

\begin{lemma}\label{gaussest}
 Let $(\go_x)_{x\in X}$ be a Gaussian vector indexed by $X$ defined on the probability space
$(\Omega,\bbP,\mathcal F)$ closed by linear combination. Define the measure $\tilde \bbP$ as
\begin{equation}
 \frac{\dd \tilde \bbP}{\dd \bbP}=\exp(\go_{x_0}-\var \go_{x_0}/2).
\end{equation}
Then under $\bbP$,  $(\go_x)_{x\in X}$ are still Gaussian variable, their covariance remain unchanged and their expectation is equal to
\begin{equation}
 \tilde \bbP [\go_x]= \bbP[\go_x\go_{x_0}].
\end{equation}
\end{lemma}

We go now for the proof. Let $\gamma<1$ be fixed. We remark that
\begin{equation}
 \bP [\log W_t]= \frac{1}{\gga}\bP [\log W_t^{\gga}]\le \frac{1}{\gga}\log \bP  [W_t^{\gga}].
\end{equation}
For this reason we have
\begin{equation}\label{momentfrac2222}
 p(\gb)\le \frac{1}{\gga}\liminf_{t\to \infty} \log \bP W_t^{\gga}.
\end{equation}
Therefore, our aim is to prove that $\bP[ W_t^{\gga}]$ decays exponentially.
Fix $T:=C_1 \gb^\frac{-4}{2-\theta}$, and $\gamma<1$ (say $\gga:=1/2$).
Choose $t:=Tn$.
For $y=(y^1,\dots,y^d)\in \Z^d$, define $I_y:=\prod_{i=1}^d [y^i\sqrt{T},(y^i+1)\sqrt{T})$.

We decompose the partition function $W_t$ into different contributions corresponding to different families of path. We have
\begin{equation}
 W_t:=\sum_{y_1,\dots,y_n\in \Z} W_{(y_1,\dots,y_n)},
\end{equation}
where
\begin{equation}
 W_{(y_1,\dots,y_n)}:= P \left[\exp\left(\int_0^{t} \gb\go(\dd s,B_s)-\gb^2/2\dd s\right) \ind_{\left\{ B_{kT}\in I_{y_k}, \forall k=1\dots n\right\}}\right].
\end{equation}

We use the inequality $(\sum a_i)^{\gga}\le \sum a_i^{\gga}$ which holds for an arbitrary collection of positive numbers to get
\begin{equation}\label{dec}
 \bP [W_t^{\gga}]\le \sum_{y_1,\dots,y_n\in \Z} \bP [W^{\gamma}_{(y_1,\dots,y_n)}].
\end{equation}
Given $Y=(y_1,\dots,y_n)$, and $\tilde \bP_Y$ a probability measure on $\go$ we have

\begin{equation}\label{hollder}
 \bP W^{\gga}_{(y_1,\dots,y_n)}=\tilde \bP_Y \left[\frac{\dd \bP}{\dd \tilde\bP_Y}  W^{\gga}_{(y_1,\dots,y_n)}\right]\le \left(\bP\left[\left(\frac{\dd \bP}{\dd \tilde\bP_Y}\right)^{\frac{\gga}{1-\gga}}\right]\right)^{1-\gga}\left(\tilde \bP_Y[ W_{(y_1,\dots,y_n)}]\right)^{\gga}.
\end{equation}
Let $C_2$ a fixed constant. Define the blocks $A_k$ by

\begin{equation}\begin{split}
A_k&:=[(k-1)T,T]\times \prod_{i=1}^d[y^{i}_{k-1}-C_2\sqrt{T},y^{i}_{k-1}+C_2\sqrt{T}]=:[kT,(k+1)T]\times \bar A_k,\\
J_Y&:=\bigcup_{k=1}^n B_k,
\end{split}
\end{equation}
with the convention $y_0=0$.
Moreover we define the random variable
\begin{equation}\begin{split}
\gO_k&:=\frac{\int_{A_k}    \go(\dd t,x)\dd x}{\sqrt{T\int_{\bar A^2_k}Q(x-y)\dd x\dd y}},\\
\Omega_Y&:=\sum_{k=1}^n \gO_k.
\end{split}\end{equation}
Note that with this definition $(\Omega_k)_{k\in\{1,\dots,n\}}$ are standard centered independent Gaussian variables.
Define $\tilde \bP_Y$ by
\begin{equation}
 \frac{\dd \tilde \bP_Y}{\dd \bP}:=\exp\left(-\Omega_Y-n/2\right).
\end{equation}
From this definition and using the fact that $\gga=1/2$, we have
\begin{equation}\label{dete}
 \left(\bP\left[\left(\frac{\dd \bP}{\dd \tilde\bP_Y}\right)^{\frac{\gga}{1-\gga}}\right]\right)^{1-\gga}=\exp(n/2).
\end{equation}
We define also the measure $\tilde \bP_1$ by
\begin{equation}
 \frac{\dd \tilde \bP_1}{\dd \bP}:=\exp(-\gO_1-1/2).
\end{equation}
We consider now the expectation of $W_{(y_1,\dots,y_n)}$ with respect to $\tilde \bP_Y$. As the covariance of the Gaussian field remains the same after the measure change we have
\begin{equation}\begin{split}
 \tilde \bP_Y[ W_{(y_1,\dots,y_n)}]&= P\exp\left(\gb \tilde \bP_Y\left[\int_0^t \go(\dd s,B_s)\right]\right)\ind_{\left\{ B_{kT}\in I_{y_k}, \forall k=1\dots n\right\}}\\
&\le \prod_{k=0}^{n-1}\max_{x\in I_O} P_x \left[\exp\left(\gb  \tilde \bP_1\left[\int_0^T  \go(\dd s,B_s)\right]\right)\ind_{\left\{ B_{T}\in I_{y_k-y_{k-1}}, \forall k=1\dots n\right\}}\right],
\end{split}
\end{equation}
Where $P_x$ denotes the Wiener measure with initial condition $x$ (the law of a standard Brownian Motion starting at $x$).
 The second inequality is obtained by cutting the Brownian trajectory into $n$ slices and using the fact that the environments in the slices are independent under $\tilde \bP_Y$ and taking the maximum over the possible intermediate point.
Coming back to \eqref{dec} and using \eqref{hollder} and \eqref{dete} we get 
\begin{equation}\label{machin}
 \bP W_t^{\gga}\le \exp(n/2)\left[\sum_{y\in\Z^d} \left(\max_{x\in I_O} P_x\left[\exp\left(\gb  \tilde \bP_1\left[\int_0^T  \go(\dd s,B_s)\right]\right)\ind_{\left\{ B_{T}\in I_y\right\}}\right)^{\gga}\right]\right]^n.
\end{equation}
We are able to prove that the right--hand side decays exponentially with rate $T$ if we are able to show that
\begin{equation}\label{troidiz}
 \sum_{y\in\Z^d} \left(\max_{x\in I_O} P_x\left[\exp\left(-\gb  \tilde \bP_1\left[\int_0^T \dd \go(s,B_s)\right]\right)\ind_{\left\{ B_{T}\in I_y\right\}}\right]\right)^{\gga},
\end{equation}
is small. To do so, we have to estimate the expectation of the Hamiltonian under $\tilde \bP_1$. We use Lemma \ref{gaussest} and get
\begin{equation}\label{modifexp}
 -\tilde \bP_1\left[\int_0^T  \go(\dd s,B_s)\right]=\bP\left[\gO_1\int_0^T \go( \dd s,B_s)\right]=\frac{\int_{A_1} Q(x-B_s)\dd x \dd t}{\sqrt{T\int_{\bar A_1^2} Q(x-y)\dd x \dd y}}.
\end{equation}
The above quantity is always positive. However, it depends on the trajectory $B$. One can check that the assumption of polynomial decay for $Q$ implies that
\begin{equation}
\int_{\bar A_1^2} Q(x-y)\dd x \dd y\asymp T^{d-\theta/2}.
\end{equation}
To control the numerator of in the right-hand side of \eqref{modifexp} we need an assumption on the trajectory. What we do is we control the value for trajectories that stays within $A_1$. 
Uniformly over any trajectory $(s,B_s)_{s\in[0,T]}$ that stays in $A_1$ we have
\begin{equation}
 \int_O^T\int_{[-C_2T,C_2T]^d} Q(x-B_s)\dd x \dd s\asymp T^{(d+1-\theta)/2}.
\end{equation}
Alltogether we get that there exists a constant $C_3$ such that uniformly on trajectories staying in $A_1$.
\begin{equation}\label{esta1}
 \tilde \bP_1\left[\int_0^T  \go(\dd s,B_s)\right]\le -C_3T^{\frac{2-\theta}{4}}.
\end{equation}
The distribution of the Brownian Motion allow us, given $\gep>0$, to find $R=R_{\gep}$ such that
\begin{multline}
 \sum_{\|y\|\ge R} \left(\max_{x\in I_O} P_x\left[\exp\left(\gb  \tilde \bP_1\left[\int_0^T  \go(\dd s,B_s)\right]\right)\ind_{\left\{ B_{T}\in I_y\right\}}\right]\right)^{\gga}\\
\le \sum_{\|y\|\ge R}\left(\max_{x\in I_O} P_x\left\{B_T\in I_y\right\}\right)^{\gga} \le \gep.
\end{multline}
Where the first inequality simply uses the fact that $\bP_1(\dots)$ is negative.
For the terms corresponding to $\|y\|< R$ we use the rough bound

\begin{multline}
 \sum_{\|y\|<R} \left(\max_{x\in I_O} P_x\left[\exp\left(\gb  \tilde \bP_1\left[\int_0^T \go(\dd s,B_s)\right]\right)\ind_{\left\{ B_{T}\in I_y\right\}}\right]\right)^{\gga}\\
\le (2R)^d \left(\max_{x\in I_O} P_x\left[\exp\left(\gb  \tilde \bP_1\left[\int_0^T  \go(\dd s,B_s)\right]\right)\right)^{\gga}\right].
\end{multline}
Fix $\delta:=(\gep/(2R)^d)^{\frac{1}{\gga}}$. Our remaining task is to show 
\begin{equation}
 \max_{x\in I_O} P_x\exp\left(\gb  \tilde \bP_1\left[\int_0^T  \go(\dd s,B_s)\right]\right)\le \delta.
\end{equation}
To get the above inequality, we separate the right--hand side into two contribution: trajectories that stays within $A_1$ and trajectories that go out of $A_1$. Bounding these contribution gives
\begin{multline}
 \max_{x\in I_O} P_x\left[\exp\left(\gb  \tilde \bP_1\left[\int_0^T \go(\dd s,B_s)\right]\right)\right] \le P\left\{\max_{s\in[0,T]} |B_s|
\le |C_2-1|\sqrt{T}\right\}\\ 
+ \max_{x\in I_O} P_x\left[\exp\left(\gb  \tilde \bP_1\left[\int_0^T \go(\dd s,B_s)\right]\right)\ind_{\{(s,B_s):s\in[0,T]\subset A_1\}}\right].
\end{multline}
We can fix $C_2$ so that the first term is less that $\gd/2$. Equation \eqref{esta1} guaranties that the second term is less than
\begin{equation}
\exp\left(-\gb C_3  T^{\frac{2-\theta}{4}}\right)=\exp\left(-\gb C_3C_1^{\frac{2-\theta}{4}}\right)\le \gd/2.
\end{equation}
where the last inequality is obtained by choosing $C_1$ sufficiently large.
We have shown alltogether that \eqref{troidiz} is less than $2\gep$. Combining this result with \eqref{machin} and \eqref{momentfrac2222} implies (for $\gep$ small enough)
\begin{equation}
 p(\gb)\le -\frac{1}{T},
\end{equation}
which is the result we want.
\qed

\subsection{Proof of Theorem \ref{supdiffff}}

The proof can be decomposed in two steps. The first step is inspired by the work of Peterman \cite{cf:P}, is the proof of a lemma which bounds from above the entropic cost of going in a region which is of order $t^{\alpha}$ away from the origin.
The second step bounds from below the typical energy fluctuations that we may encounter in the same region.
After this analysis all that remains is to fit together all the element of the proof to get the result.
\medskip
Let $N$ be some  even integer, $\alpha:=\frac{3}{4+\theta}$.
For $k\in \N$ define
\begin{equation}
 A_k:=[t/2,t]\times\left[\frac{(2k-1)t^{\alpha}}{N^2},\frac{(2k+1)t^{\alpha}}{N^2}\right]
\times\left[-\frac{t^{\alpha}}{N^2},\frac{t^{\alpha}}{N^2}\right]^{d-1}.
\end{equation}
Define
\begin{equation}
 Z_t^{(k)}:= P\left[\exp\left(\gb\int_0^t  \go(\dd s,B_s)\right)\ind_{\{(s,B_s)\in A_k, \forall s\in[t/2,t]\}}\right].
\end{equation}

\begin{lemma}\label{th:Q}
 With probability greater than $1-1/N$, we have
\begin{equation}
\sum_{k\in\{-N,\dots,N\}\setminus\{0\}} Z_t^{(k)}\ge \exp\left(-\frac{8}{N^2}t^{2\alpha-1}\right)Z_t^{(0)}.
\end{equation}
\end{lemma}
\begin{proof}
This Lemma and its proof proof is strongly inspired from \cite[Lemma 2]{cf:P}.
We use the following transformation $h_k$ which transform a path contributing to $W_t^{(k)}$ into a path contributing to 
$Z_t^{(0)}$.

\begin{equation}
h_k:\quad (s,x)\mapsto x+(2s/t \wedge 1)\frac{2k}{N^2}t^{\alpha}, 
\end{equation}
and define 

\begin{equation}
 \bar Z_t^{(k)}:= P\exp\left(\gb\int_0^t \go(\dd s, h_k(s,B_s))\right)\ind_{\{(s,B_s)\in A_0, \forall s\in[t/2,t]\}}.
\end{equation}
On can check that $\left( \bar Z_t^{(k)} \right)_{k\in\{-N/2,\dots,N/2\}}$ is a family of identically distributed random variable.
Moreover, elementary reasoning gives us that there exists an integer $k_0$, with $|k_0|\le N/2$  such that
\begin{equation}
 Q\left\{\bar Z_t^{(k_0)}\ge \max_{k\in\{-N/2,k_0-1\}\cup \{k_0+1,\dots,N/2\}} \bar Z_t^{(k)}\right\}<\frac{1}{N}.
\end{equation}
From this it is easy to deduce by translation invariance that
\begin{equation}
 Q\left\{\bar Z_t^{(0)}\ge  \max_{k\in\{-N/2-k_0,-1\}\cup \{1,\dots,N/2+k_0\}}\bar Z_t^{(k)}\right\}<\frac{1}{N}.
\end{equation}
Using Girsanov transform, one can check that

\begin{equation}\begin{split}
  Z_t^{(k)}&=P\exp\left(-\frac{4k}{N^2}t^{\alpha-1}B^{1}_{t/2}-\frac{4k^2}{N^4}t^{2\alpha-1}+\gb\int_0^t \go(\dd s ,h_k(s,B_s))\right)\ind_{\{(s,B_s)\in A_0, \forall s\in[t/2,t]\}}\\
&\ge \exp\left(-\frac{4(k^2+|k|)}{N^4}t^{2\alpha-1}\right)\bar Z_t^{(k)}\ge \exp\left(-\frac{8}{N^2}t^{2\alpha-1}\right)\bar Z_t^{(k)},
\end{split}
\end{equation}
where $B^1$ is the first coordinate of the Brownian motion.
\end{proof}
For the rest of the proof, the idea which is used differs substantially from the one used in \cite{cf:P} and adapted then in \cite{cf:BTV}.
Instead of using purely Gaussian tool and working with covariance matrix, we use changes of measure that are similar to the one of the previous section. This shortens the proof considerably, and to makes it less technical and more intuitive. Moreover, it enlightens the fact that the proof could be adapted to a non Gaussian context.
We set $T:=t^{\alpha}N^{-3}$ and define
\begin{equation}
  \Omega:=\frac{\int_{[-T,T]^d}\int_{t/2}^{t}\go(\dd  s,x)\dd x}{\sqrt{t/2 \int_{[-T,T]^d\times[-T,T]^d}Q(x-y)\dd x\dd y}},
\end{equation}
which is standard centered Gaussian variable.
We define the probability measure $\bP_0$ by
\begin{equation}
 \frac{\dd \bP_0}{\dd \bP}(\go):=\exp\left(-\gO-1/2\right).
\end{equation}

\begin{lemma}\label{petitlem}
Let $A$ be any event, we have
\begin{equation}
 \bP(A)\le \sqrt{e\bP_0(A)}.
\end{equation}
 \end{lemma}
\begin{proof}
 This is a simple application of H\"older inequality.
\begin{equation}
 \bP(A)=\bP_0 \left[\frac{\dd \bP}{\dd \bP_0} \ind_{A}\right]\le \sqrt{\bP \left[\frac{\dd \bP}{\dd \bP_0}\right]}\sqrt{\bP_0 (A)}.
\end{equation}
\end{proof}
Now, our aim is to show that under $\bP_0$, the probability that the walk stays in $[-t^{\alpha}/N^3,t^{\alpha}/N^3]^d$ is small and to use the above lemma to conclude.
We use Lemma \ref{gaussest} to do a measure coupling.
Indeed if $\go$ has distribution $\bP$ then $\hat \go$ defined by
\begin{equation}\begin{split}
\hat\go(0,x)&:=0, \quad \quad \forall x\in \R^d,\\
\hat\go( \dd  s,x)&:= \go(\dd s,x)-\bP \left[\gO\dd \go(s,x)\right]\\
&= \go(\dd s,x)-\frac{\int_{[-T,T]^d}Q(x-y)\dd y}{\sqrt{t/2 \int_{[-T,T]^d\times[-T,T]^d}Q(x-y)\dd x\dd y}}\dd t.
                \end{split}
\end{equation}
has distribution $\bP_0$. 
We define
\begin{equation}\begin{split}
 X_t:=P\left[ \exp\left(\gb\int_0^t \go(\dd s,B_s)\right)\ind_{\{B_s\in [-T,T]^d,\ \forall s\in[t/2,t]\}}\right]\\
\hat X_t:= P \left[\exp\left(\gb\int_{0}^{t} \hat \go(\dd s,B_s)\right)\ind_{\{B_s\in [-T,T]^d,\ \forall s\in[t/2,t]\}}\right],\\
\hat Z_t^{(k)}:= P\left[\exp\left(\gb\int_0^t  \hat \go(\dd s,B_s)\right)\ind_{\{(s,B_s)\in A_k,\ \forall s\in[t/2,t]\}}\right]. 
\end{split}
\end{equation}
From this definition
\begin{equation}\label{truc}
 \mu_t\left(\max(B_s\le t^{\alpha}/N^3\right\}\le \frac {X_t}{\sum_{k\in\{-N,\dots,N\}\setminus\{0\}} Z_t^{(k)}},
\end{equation}
And 
\begin{equation}\label{der22}
 \bP_0\left[  \mu_t\left\{\max_{0\le s\le t} |B_s|\le t^{\alpha}/N^3)\right)\le x \right]\le \bP\left[\frac {\hat X_t}{\sum_{k\in\{-N,\dots,N\}\setminus\{0\}} \hat Z_t^{(k)}}\le x \right],
\end{equation}
for any $x$.
\begin{lemma}
For $N$ large enough, and $t$ large enough 
\begin{equation}
 \frac {\hat X_t}{\sum_{k\in\{-N,\dots,N\}\setminus\{0\}} \hat Z_t^{(k)}}\le \frac {X_t}{\sum_{k\in\{-N,\dots,N\}\setminus\{0\}} Z_t^{(k)}}\exp(-C_4 t^{1/2}T^{-\theta/2}).
\end{equation}
\end{lemma}
\begin{proof}
 It follows from definitions that
\begin{multline}
\log \frac{\hat X_t\sum_{k\in\{-N,\dots,N\}\setminus\{0\}} Z_t^{(k)}}{X_t \sum_{k\in\{-N,\dots,N\}\setminus\{0\}} \hat Z_t^{(k)}} \le \\ 
-\!\!\!\!\!\!\min_{\{B:\|B_s\| \le T,\ \forall s \in [t/2,t]\}}\!\!\!\!\!\!\bP\left[\int_{t/2}^t \go(\dd s,B_s)\gO\right]+\!\!\!\!\!\! \max_{\{B\ : \|B_s\|\ge t^{\alpha}/N^2,\  \forall s \in [t/2,t]\}}\!\!\!\!\!\!\bP \left[\int_{t/2}^t \go(\dd s,B_s)\gO\right]
\end{multline}
Now, we bound the right-hand side of the above. We have

\begin{equation}
 \bP\left[\int_{t/2}^t\dd \go(s,B_s)\gO\right]=\frac{\int_{t/2}{t}\int_{[-T,T]^d}Q(B_s-x)\dd x}{\sqrt{t/2\int_{[-T,T]^d\times [-T,T]^dQ(x-y)\dd x \dd y}}}.
\end{equation}
First remark that
\begin{equation}
 \int_{[-T,T]^d\times [-T,T]^d} Q(x-y)\dd x \dd y\asymp T^{2d-\theta}.
\end{equation}
If the trajectory $B$ is constrained to stay in $[-T,T]^d$ when $s\in[t/2,t]$, then there exist a constant $C_5$ only depending on $Q$ such that
\begin{equation}
 \int_{[-T,T]^d}Q(B_s-x)\dd x\ge C_5 T^{d-\theta} \forall  s \in [t/2,t].
\end{equation}
If the trajectory $B$ is such that $d(B_s,[-T,T]^d)\ge NT/2$ for all $s\in [t/2,t]$, then 
\begin{equation}
 \int_{[-T,T]^d}Q(B_s-x)\dd x\le C_6 \int_{[-T,T]^d} (NT)^{-\theta} \dd x= C_6 T^{d-\theta}N^{-\theta}, \quad \forall s\in[t/2,t].
\end{equation}
This is sufficient to prove the result for $N$ large enough.
\end{proof}
Now the preceding result together with  Lemma \ref{th:Q} assures that with propability larger than $1-1/N$ we have
\begin{multline}\label{der11}
 \frac {\hat X_t}{\sum_{k\in\{-N,\dots,N\}\setminus\{0\}} \hat Z_t^{(k)}}\le \frac {Z_t^{(0)}}{\sum_{k\in\{-N,\dots,N\}\setminus\{0\}} Z_t^{(k)}}\exp(-t^{1/2}T^{-\theta/2})\\
\le \exp\left(\frac{8}{N^2}t^{2\alpha-1}-C_4t^{1/2}T^{-\theta/2}\right).
\end{multline}
 Now we estimate what's in the exponential of the right--hand side 

\begin{equation}\label{der12}
 \frac{8}{N^2}t^{2\alpha-1}-C_4t^{1/2}T^{-\theta/2}=\left[\frac{8}{N^2}-C_4 N^{3\theta}\right]t^{\frac{2-\theta}{4+\theta}}\le -t^{\frac{2-\theta}{4+\theta}}.
\end{equation}
Now we put together all the bricks of our reasoning. Equations \eqref{der11}, \eqref{der12} combined with \eqref{der22} gives us
\begin{equation}
 \bP_0\left\{\mu_t\{\max_{0\le s\le t} \|B_s\| \le t^{\alpha}/N^3\}\ge \exp\left(-t^{\frac{2-\theta}{4+\theta}}\right)\right\}\le \frac{1}{N}.
\end{equation}
Lemma \ref{petitlem} allows us get from this
\begin{equation}
 \bP_0\left\{\mu_t\{\max_{0\le s\le t} \|B_s\| \le t^{\alpha}/N^3\}\ge \exp\left(-t^{\frac{2-\theta}{4+\theta}}\right)\right\}\le \frac{1}{\sqrt{N}}
\end{equation}
So that
\begin{equation}
 \bP \mu_t \{\max \|B_s\| \le t^{\alpha}/N^3\}\le \exp\left(-t^{\frac{2-\theta}{4+\theta}}\right)+\frac{1}{\sqrt{N}}.
\end{equation}
We get our result by choosing $N$ large enough.\qed

\newpage

\addtolength{\topmargin}{-0.02 \textheight}

\addtolength{\textheight}{30mm}

\thispagestyle{empty}

{\advance\baselineskip -2.2pt

\vspace*{\fill}

\begin{small}
\noindent {\bf R\sc{\textbf{ésumé}}}: Cette thèse est consacrée à l'étude de différents modèles aléatoires de polymères.
On s'intéresse en particulier à l'influence du désordre sur la localisation des trajectoires pour les modèles d'accrochage et pour les polymères dirigés en milieu aléatoire. En plus des modèles classiques dans $\Z^d$, nous abordons l'étude de modèles dit hiérarchiques, construits sur une suite de réseaux auto-similaires, très présent dans la littérature physique.  Les résultats que nous avons obtenus concernent principalement l'énergie libre et le phénomène de surdiffusivité. En particulier, nous prouvons: (1) la pertinence du désordre à toute température  dans pour le modèle d'accrochage désordonné en dimension $1+1$, (2) l'occurence d'un désordre très fort à toute température en dimension $1+2$ pour les polymères dirigés en milieu aléatoire.

\vspace{2mm}

\noindent {\bf M\sc{\textbf{ots-cl\'es}}} Polymère dirigé, Modèle d'accrochage, Très fort désordre, Surdiffusivité, Modèles Hiérarchiques, Pertinence du  désordre.

\vspace{1cm}

\noindent {\bf A\sc{\textbf{bstract}}}: This thesis studies models of random directed polymers. 
We focus on the influence of disorder on localization of the trajectories for pinning model and directed polymers in random environment. In addition to the classical $\Z^d$ models, we pay a particular attention to so-called hierachical models, buildt on a sequence of self-similar lattices, that are frequently studied in the physics literature. The results we obtain concern mainly free energy and superdiffusivity properties. In particular we present the proof that: (1) disorder is relevant at arbitrary high temperature for pinning models in dimension $1+1$, (2) very strong disorder holds at all temperature in dimension $1+2$ for directed polymers in random environment.

\vspace{3mm}

\noindent {\bf K\sc{\textbf{ey words}}}: Directed polymer, Pinning models, Very Strong Disorder, Superdiffusivity, Hierarchical Models, Disorder relevance.
\vspace{1cm}

\begin{center}
{\bf Université Paris Diderot et Laboratoire de Probabilité et Modèles Aleatoires (CNRS
 U.M.R. 7599),\\ U.F.R. Mathématiques, Case 7012 (Site Chevaleret),\\
 75205 Paris cedex 13, France}
\end{center}

\end{small}
 \end{document}